\documentclass[11pt,a4paper,reqno]{amsart}
\usepackage[T1]{fontenc}
\usepackage[margin=1in]{geometry}
\usepackage{lmodern,amsmath,amssymb,amsthm,mathtools,mathrsfs}
\usepackage{microtype,booktabs,array,longtable}
\usepackage{enumitem,needspace}
\usepackage{xcolor}
\usepackage[hypertexnames=false,colorlinks=true,linkcolor=blue!35!black,citecolor=blue!35!black,urlcolor=blue!35!black]{hyperref}
\hypersetup{pdftitle={Stochastic Operator Algebras: Intrinsic Cumulant Geometry},pdfauthor={Guangqian Zhao},pdfsubject={Cumulants under orthogonal cuts and applications, R88}}
\numberwithin{equation}{section}

\newtheorem{theorem}{Theorem}[section]

\newtheorem{proposition}[theorem]{Proposition}
\newtheorem{corollary}[theorem]{Corollary}
\newtheorem{lemma}[theorem]{Lemma}
\theoremstyle{definition}

\newtheorem{example}[theorem]{Example}
\theoremstyle{remark}
\newtheorem{remark}[theorem]{Remark}

\newcommand{\R}{\mathbb R}
\newcommand{\E}{\mathbb E}
\newcommand{\Schat}{\mathcal S}

\newcommand{\ran}{\operatorname{ran}}

\newcommand{\doi}[1]{\href{https://doi.org/#1}{doi:\nolinkurl{#1}}}
\newcommand{\arxiv}[1]{\href{https://arxiv.org/abs/#1}{arXiv:\nolinkurl{#1}}}
\setlist[enumerate]{label=\textup{(\roman*)},leftmargin=*,itemsep=2pt,topsep=4pt}
\newcommand{\tr}{\operatorname{tr}}

\allowdisplaybreaks[2]
\title[Stochastic operator algebras]{Stochastic Operator Algebras:\\ Intrinsic Cumulant Geometry}
\author{Guangqian Zhao}
\address{School of Mathematical Sciences, University of Science and Technology of China, Hefei 230026, China}
\email{zhaoguangqian@mail.ustc.edu.cn}
\subjclass[2020]{Primary 46H05, 46G20; Secondary 60H40, 47B10, 47B25}
\keywords{stochastic operator algebra, Appell algebra, cumulant class, Hilbert tensor, intrinsic reconstruction, independent blocks, tracial joint spectral measures, Wishart principal minors}
\date{}

\begin{document}

\begin{abstract}
If an orthogonal projection splits a random vector into independent
parts, every mixed cumulant across the cut vanishes. For a family of
cuts we identify the subspaces on which higher cumulants can survive.
A distance bound on vectors characterizes when square-integrable
defects at the individual cuts admit a simultaneous correction by a
Hilbert tensor. The criterion is the same in every integer degree
$r\ge3$ and can be tested on Hilbert-space-valued random variables.
For laws fixed by random recombination, uniform cutting supplies the
moment asked for by Courtade. The minimum dimension of a nonzero
finite-dimensional common reducing subspace determines the optimal
fourth-moment ratio. A separate application to polynomial optimization
characterizes the $r^{-2}$ push-forward rate on a convex body by the
boundary area of sublevel sets at regular minima.
\end{abstract}
\maketitle

\section{What simultaneous independence forces}
If $U,V$ are independent and $U+V,U-V$ are also independent,
the Kac--Bernstein theorem gives Gaussian laws with equal variances
\cite{r74:Kac}. What remains of this conclusion for other families
of orthogonal cuts? For finite-dimensional laws Courtade identified the
common unsplit subspaces as the possible non-Gaussian factors
\cite[Theorem~1]{r63co:C}. For a continuous cumulant on Hilbert space,
moving a cut between three slots kills the commutators. On the remaining
commuting part, refining the cuts divides each diffuse vector into cells
of vanishing maximal norm. Boundedness in degree $r>2$ then kills its
contribution. Higher cumulants survive only within common eigenspaces of
the cuts (Theorem~\ref{r60:structure}).

For a random cut $P$ with law $\xi$, put $B_P=\ker P\cup\ran P$ and
$B=\bigcap_{P\in\operatorname{supp}\xi}B_P$. These are the vectors
that stay in one branch of the individual cut and of every cut,
respectively. The quantitative question is whether the estimate
\[
 \operatorname{dist}(x,B)^2\le C\,\mathbb E\operatorname{dist}(x,B_P)^2
 \qquad(x\in H)
\]
holds with one finite constant. In any fixed integer degree $r\ge3$,
this is equivalent to the following assertion. Every bounded symmetric
form whose mixed coefficients have finite mean squared Hilbert norm
differs by a Hilbert tensor from a form that splits across every cut.
Cumulants of actual $H$-valued laws already detect failure in each
fixed degree (Theorem~\ref{r73:geocubic} and
Corollary~\ref{r73:realwitness}).

For independent copies $X_1,X_2$ independent of $P$, consider the weaker
identity $X\stackrel d=PX_1+(I-P)X_2$. Courtade's Question~1 asks for a
logarithmic moment under uniform cutting \cite{r63co:C}.
A radial truncation gives a second moment and a sharp fourth-moment bound
on a uniformly cut Hilbert space. In finite dimension it removes the moment hypothesis from
Courtade's product decomposition (Theorem~\ref{r63co:mom} and
Corollary~\ref{r63co:law}).

In Part~II the Schur complement gives the sharp conditional lower
bound for Wishart determinant moments throughout the integrable
negative-power range (Corollary~\ref{r88:wishfloor}). In Part~III,
two Gaussian semigroups on the same compact group have the same
$L^p$ thresholds for all $1\le p\le\infty$. One attains every
threshold, and the other excludes them all (Theorem~\ref{r87:gausstimes}).
In fourth $q$-chaos the fourth-moment criterion changes at the negative
root of $1+q+q^9$ (Theorem~\ref{r88:negativeq}). Part~IV compares
L\'evy-driven path laws through their observed jumps and Gaussian
control spaces (Theorem~\ref{r88:jumpcomparison}).

\newgeometry{left=.65in,right=.65in,top=.58in,bottom=.57in}
\phantomsection\pdfbookmark[1]{Questions from the literature and proof locations}{index55}\label{r55:index}
\begin{center}{\large\bfseries Questions from the literature and their proofs}\end{center}
\begingroup\small\linespread{.95}\selectfont
\noindent Q-numbers match the source register. The 87 question groups are followed by one proposed research direction.\par
\newcommand{\litid}[1]{\textsf{Q#1}}
\newcommand{\xrow}[3]{\qrow{#1}{#2}{#3}}
\newcommand{\qloc}[1]{\mbox{\hyperref[#1]{\ref*{#1}, p.~\pageref*{#1}}}}
\newcommand{\qrow}[3]{\par\noindent\begin{minipage}{\linewidth}\noindent\textbf{#1}\enspace #2\enspace{\footnotesize\textit{Proof} #3.}\end{minipage}\par\vspace{0.2pt}}
\par\needspace{4\baselineskip}\smallskip\noindent{\bfseries Part I. Identifying independent factors and their bases}\par\smallskip
\qrow{\litid{50} Courtade, Question~1, p.~3 \cite{r63co:C}.}
{Uniform cutting gives a second moment and a sharp fourth-moment bound. Courtade's finite-dimensional decomposition then applies without a moment assumption.}
{\qloc{r63co:mom}, \qloc{r63co:law}}
\qrow{\litid{38} Ribot--Seigal--Zwiernik, Remark~4.2 and Discussion \cite{p53:RSZ}.}
{Generic uniqueness was known \cite{r61:RSZorth}. A coefficient test decides real uniqueness for any tensor of order $d\ge2$ admitting a frame. For a fixed frame in dimension $n\ge2$ the exceptional locus has codimension $2n-3$ when $d\ge5$.}
{\qloc{r59:pmicert}, \qloc{r69:cubicanchor}, \qloc{r61:pmicodim}, \qloc{r68:mixed}}
\qrow{\litid{39} Ribot--Seigal--Zwiernik, Discussion, identifiability under pairwise mean independence \cite{p53:RSZ}.}
{The joint law decides uniqueness in the whitened square model. Density of polynomials in every projected $L^2$ gives a finite cumulant certificate. Without this assumption, identical cumulants can give opposite answers.}
{\qloc{r60:charlaw}, \qloc{r60:pmi-law}, \qloc{reg56:closed}}

\qrow{\litid{68} Wang--Seigal, real converse to Proposition~6.3 of the preprint \cite{r80:WSpre}.}
{For $J\le\binom I2$ the rank-one test permits dependent squares exactly when $J\ge8$. Allowing repeated competitor directions forces independence for every identifiable real matrix.}
{\qloc{r80:ws-real}}
\qrow{\litid{73} Wang--Seigal, real closure and closedness in Section~6.3 of the preprint \cite{r80:WSpre}.}
{Stable identifiability requires independent column squares, no extra square, and no extra tangent direction at a listed square. For $J\le\binom I2$ the non-identifiable locus is closed exactly for $J\le6$.}
{\qloc{r81:wsclosure}, \qloc{r81:wsclosed}, \qloc{r82:ws-capacity}}

\par\needspace{4\baselineskip}\smallskip\noindent{\bfseries Part II. Tracial measures and Gaussian determinant laws}\par\smallskip
\qrow{\litid{24} Hein\"avaara, Question 3.2 \cite{newtr:question}.}
{For bounded self-adjoint pairs with a normal semifinite faithful trace, a punctured-plane Radon measure exists exactly for trace-compact pairs. It also exists for semifinite $L^p$ pairs. The matrix hypothesis alone can leave signed traces undefined.}
{\qloc{h50tr:finite}, \qloc{h50tr:semifinite}, \qloc{newtr:compact}}
\qrow{\litid{25} Hein\"avaara, Question 3.3 \cite{newtr:question}.}
{The Schatten embedding extends isometrically and real-linearly to every real two-dimensional subspace of Haagerup $L^p$, $0<p<\infty$.}
{\qloc{h50tr:haagerup}}
\qrow{\litid{26} Hein\"avaara, Question 3.4 \cite{newtr:question}.}
{The matrix density formula~(4) extends to compact pairs by Fredholm determinants. In the semifinite case spectral traces recover density and singular mass off zero. Haagerup angular solutions can be nonunique at integer orders.}
{\qloc{r56:fredholm}, \qloc{r60:trdensity}, \qloc{r57:haag}}
\qrow{\litid{27} Hein\"avaara, Question 3.5 \cite{newtr:question}.}
{The real roots of the ternary polynomial give the positive measure in (7) without a determinantal representation.}
{\qloc{hyp48:poly}}
\qrow{\litid{28} Hein\"avaara, Question 3.7 \cite{newtr:question}.}
{Normalized GUE and the stated real Wigner pairs share a vague limit off zero and a probability limit after $|x|^2$ weighting. For GUE the envelopes approach quantile circles and scaled segment lengths are Rayleigh under normalized expected counting.}
{\qloc{r59tr:wigner}, \qloc{r59tr:curves}, \qloc{r59tr:segmentlaw}}
\qrow{\litid{47} Hein\"avaara, \S5.9, item 1 (unbounded sums) \cite{r61cmp:source}.}
{For arbitrary input laws and $f\in C^4$ with $f^{(4)}\ge0$, classical absolute integrability gives the free expectation bound and integrability. A tail test decides convergence of the clipped measures. A Radon representation for all tests with finite expectations can fail.}
{\qloc{r62cmp:order}, \qloc{r62cmp:tails}, \qloc{r62cmp:impossible}}

\qrow{\litid{48} Hein\"avaara, \S5.9, item 2 (products and mixed words) \cite{r61cmp:source}.}
{For compact inputs with a nonnegative second factor, a differential inequality characterizes the scalar $C^4$ product tests. Two-point laws test four alternating slots. Finitely many atoms suffice for each mixed polynomial.}
{\qloc{r63cmp:product}, \qloc{r63cmp:four}}
\qrow{\litid{45} Hein\"avaara, \S5.9, item 3 (word moments) \cite{r61cmp:source}.}
{The matrix word formula~(15) extends to bounded self-adjoint pairs in every finite tracial von Neumann algebra. The unique positive measure has an explicit Brown-measure density.}
{\qloc{r61cmp:main}}
\qrow{\litid{46} Hein\"avaara, \S5.9, item 4 (pointwise bound) \cite{r61cmp:source}.}
{The rank-one bound in Proposition~4.5 improves to a sharp variance bound. A trace-norm bound holds for every bounded pair in a finite tracial algebra, with a sharper slice bound for projections.}
{\qloc{r61cmp:pointwise}}
\qrow{\litid{30} Kogan, Section 5 \cite{q47:Kogan}.}
{Joint infinite divisibility of Gaussian squares at one mean implies it after centering, even for singular covariance. Kogan's earlier exclusion required sign conditions.}
{\qloc{q49:norescue}, \qloc{q49:singular}}
\qrow{\litid{93} Avagyan, Problems~2.1--2.2 and Sec.~6 \cite{r87:Avagyan}.}
{Remove opposite ordered eigenvalue pairs whose product is one. An admissible spectrum fixes the passive orbit, or equivalently the thermal parameters, exactly when at most one remaining entry differs from the others.}
{\qloc{r87:sympl}}

\qrow{\litid{32} Genest--Ouimet--Richards, Conjecture 1.1 \cite{wish46:GOR,wish46:pair}.}
{The two-block inequality \cite[Cor.~3.4]{wish46:pair} extends to any number of disjoint blocks, all nonnegative powers, real $\alpha>p-1$ and positive definite scale. A sharp conditional lower bound gives Proposition~3.9 throughout the finite negative-moment range.}
{\qloc{wish50:product}, \qloc{r88:wishfloor}, \qloc{r67wish:rank}}
\qrow{\litid{33} Wilks (1934), p.~326 \cite{r75:Wilks}, restated in \cite[Sec.~5, Eq.~(3)]{r56:wilks-source}.}
{Two-block real and arbitrary-block nonnegative integer powers were known \cite{wish46:pair,r56:stein}. The positive series covers the full moment domain. Its sharp truncation exponent is the distance to the nearest active boundary face.}
{\qloc{wil56:main}, \qloc{wil56:bound}, \qloc{r78wil:rate}, \qloc{r57:wilksSD}}
\qrow{\litid{34} Genest--Ouimet--Richards, open case following Eq.~(3) \cite{wish46:pair}.}
{Earlier series were known \cite{r57:Nabeya,r57:Ogasawara}. For powers $2\nu_i$, $\nu_i>-1/2$, two classical $H_B$ functions give the centered trivariate formula, by continuation on the positive definite covariance cone.}
{\qloc{r57:SDmoment}, \qloc{r57:HBmoment}}
\qrow{\litid{35} Genest--Ouimet--Richards, question 1 after Conjecture 1.1 \cite{wish46:pair}.}
{The determinant product bound yields explicit power bounds for the Gaussian likelihood ratio test at any cutoff $0<c<1$. The joint moment formula recovers the power to arbitrary accuracy.}
{\qloc{r57:power}, \qloc{r57:powercdf}}
\qrow{\litid{36} Genest--Ouimet--Richards, question 2 after Conjecture 1.1 \cite{wish46:pair}.}
{Antisymmetrizing entry-moment matchings \cite{r59:KurikiNumata} gives integer determinant moments. The product bound supplies a matching-polynomial inequality with an equality criterion.}
{\qloc{r59:matching}}
\qrow{\litid{37} Genest--Ouimet--Richards, question 3 after Conjecture 1.1 \cite{wish46:pair}.}
{Weighted real polarization on lines \cite{r59:GOP} extends to Gram volumes on subspaces of any dimensions, with optimal constant and output dimension.}
{\qloc{r59:volpol}}

\par\needspace{4\baselineskip}\smallskip\noindent{\bfseries Part III. Dilations, cumulants, positivity and path signatures}\par\smallskip
\qrow{\litid{1} Sato, Remark 2.20 \cite{SatoFractional2010}.}
{For $h_a(r)=(\sqrt{1+r^2}-r)^a$ on $\R$, Example~2.19(a) excluded orders above $a+1$. A positive integral representation attains $a+1$ for every $a>0$.}
{\qloc{newsato:power}}
\qrow{\litid{3} Sato, Remark 5.20 \cite{SatoFractional2010}.}
{On $(0,\infty)$ multiplication by $r^{-a}$ preserves fractional monotonicity of every order $p>0$ for $a>0$, extending Lemma~5.18 from integer orders.}
{\qloc{ap:cm-multiplier}}
\qrow{\litid{4} Sato, Remarks 6.4 and 6.13 \cite{SatoFractional2010}.}
{At $\alpha=1$ compensated first moments characterize the input domains for $0<p<1$ and the ordinary and absolute ranges for every $p>0$. The output case $p=1$ was known \cite[Theorems~4.4--4.5]{SatoInversions}.}
{\qloc{r61:satocore}, \qloc{ap:cor:critical-input}, \qloc{ap:thm:critical-output}}
\qrow{\litid{5} Sato, Remarks 7.5 and 7.8 \cite{SatoFractional2010}.}
{At $\alpha=1$ the ordinary composition extends the direct map and already differs at $p=q=1$. For each $p>0$, $L_{p,1}$ strictly contains the union of all higher ordinary ranges.}
{\qloc{ap:prop:ordinary-factor}, \qloc{r61:satocounter}}
\qrow{\litid{6} Sato, Remarks 7.10 and 7.12 \cite{SatoFractional2010}.}
{Absolute range continuity from below fails exactly at index one within $0<\alpha<2$. A Laplace profile describes the entire all-orders absolute class at that index.}
{\qloc{ap:thm:fractional-continuity}, \qloc{ap:thm:all-absolute}}
\qrow{\litid{7} Sato, Remarks 7.13 and 7.20 \cite{SatoFractional2010}.}
{Neither $L^0_{\infty,\alpha}$ nor $L^e_{\infty,\alpha}$ is a single scalar kernel range on its full ordinary, absolute or essential domain. For $p>0$, the three order-$p$ ranges decrease strictly as $\alpha<2$ increases.}
{\qloc{avg:prop:single-range}, \qloc{ap:cor:log-alpha-order}}
\qrow{\litid{8} Sato, Example 4.5 \cite{SatoRanges}.}
{Sato proved that the iterated range contains $L_\infty(b)$ for the dilations $1,b$. The reverse inclusion gives equality for every $b>1$.}
{\qloc{avg:thm:two-scales}}
\qrow{\litid{9} Sato, six specified kernel families in Section 5 \cite{SatoRanges}.}
{The cosine and Mittag--Leffler scales acquire a decreasing density after finitely many products. Their limits and the truncated exponential limit are $L_\infty$. Stable-index mixtures and drift determine the other three.}
{\qloc{avg:thm:scale-kernel}, \qloc{ap:crit-depth}, \qloc{ap:barred-range}, \qloc{avg:thm:resonant-tempering}}
\qrow{\litid{62} Sato, Comments 6--8 \cite{r79:Sato06}.}
{The three exponential-kernel ranges equal their intersections from below for $\alpha\le0$. Absolute existence of both stages implies absolute existence of the direct integral for every $\beta<\alpha<2$. A log-periodic change preserves all four critical domains.}
{\qloc{r79:left}, \qloc{r79:abs}, \qloc{r80:criticaldomains}}
\qrow{\litid{66} Sato, Remark~6.17(i)--(ii) \cite{SatoFractional2010}.}
{For each range convention the least index shift for $K_{p,a}\subseteq L_{p,b}$ is $\max(0,(p-1)/2)$, and for $L_{p,a}\subseteq K_{p,b}$ it is $\max(0,(1-p)/2)$. At equal indices the inclusion reverses at $p=1$.}
{\qloc{r80:comparison}}
\qrow{\litid{74} Maejima--P\'erez-Abreu--Sato, Remark~4.1 \cite{r81:MPS}.}
{The fractional class is the ordinary range of a real kernel exactly for $\alpha<1$. A positive scale measure with finite second moment is determined by its jump range up to rescaling and a factor.}
{\qloc{r81:mps}, \qloc{r81:upsilon}}

\qrow{\litid{91} Saloff-Coste, Problem~6 and first assertion of Problem~8 \cite{r87:saloff}.}
{Two semigroups on one torus product share the thresholds $(2\pi e)^{-1}$ for absolute continuity, $(2\pi p^{1/(p-1)})^{-1}$ for $L^p$, and $(2\pi)^{-1}$ for continuous density. One attains all endpoints.}
{\qloc{r87:gausstimes}}

\qrow{\litid{12} Kagan's second question, recorded by Il'inskii in Section~2 \cite{IlinskiiGaussian}.}
{Il'inskii settled bounded support. On $\R$, only point masses preserve absence of a nondegenerate Gaussian divisor under convolution with every such law.}
{\qloc{thm:kagan}}
\qrow{\litid{13} Roybal, Sec.~3, (3.1)$\Rightarrow$(3.2) and (3.4)$\Rightarrow$(3.3) \cite{Roybal}.}
{Both converses fail beyond factored moments. One bivariate law works for both, with positive moment-matrix lower bound exactly for $0<R\le r$, for any prescribed $r>0$.}
{\qloc{roy:thm:critical}}
\qrow{\litid{87} di Dio--Schm\"udgen, authors' version, Problem~5 \cite{r86:momentcone}.}
{For each $q\ge2$ an interior moment vector is a regular value at lengths $k<q$ and a singular value at every $k\ge q$. Every length admits a regular representation.}
{\qloc{r86:rankthreshold}}

\qrow{\litid{90} di Dio, Open Problems 5.1, 5.4, 5.5 \cite{r86:transform}.}
{Mutual transformations need not be strong. A Stieltjes law for the code need not give a plane measure. A fixed discontinuous code can require arbitrarily fast degree growth.}
{\qloc{r86:affinetrans}, \qloc{r86:direction}, \qloc{r86:momdegree}}

\qrow{\litid{86} Cicho\'n--Stochel--Szafraniec, Sec.~6 \cite{r85:CSS}.}
{A complex moment sequence has a unique positive definite extension exactly when its representing measure is unique and gives no mass to zero. Otherwise it has continuum many extensions.}
{\qloc{r86:origin}, \qloc{r85:css}}

\qrow{\litid{92} Ejsmont--Lehner, Conjecture~5.2 \cite{r87:ELvariance} and Sec.~4.3 \cite{r87:ELsums}.}
{Odd-cumulant cancellation fails to preserve free infinite divisibility in every degree $d\ge3$ on identically distributed inputs, and in degree six for arbitrary free families.}
{\qloc{r87:freecancel}, \qloc{r87:freefamilies}}

\qrow{\litid{95} Kemp--Miyagawa, after Theorem~1.3 \cite{r88:KempMiyagawa}.}
{In fourth chaos the criterion holds for $-1<q<0$ exactly when $q>\rho$, $1+\rho+\rho^9=0$. It holds in every order for $-1/2<q<0$.}
{\qloc{r88:negativeq}, \qloc{r88:qallpositive}}

\qrow{\litid{14} Alahmade, Conjecture 1 (= Problem 1) and Problems 2--4 \cite{om42Alahmade}.}
{The crossing kernel gives the exact vanishing order and its leading coefficient at $q=1$. Marked chords give a moment recurrence, diagrams for both cumulants and a parity rule at $\omega=\pm i$.}
{\qloc{om42:order}, \qloc{r78:braid-laplace}, \qloc{om42:rec}, \qloc{om42:diagrams}, \qloc{om42:special}}
\qrow{\litid{40} Klep--Magron--Sch\"otz, following Theorem 2 \cite{KlepMagronSchoetz}.}
{The original existence result covers $t\le2$ and its degree estimate covers $0\le t<1$. We prove failure for $t>2$ and degree bounds through $t=2$, sharp for the Motzkin sextic and Choi--Lam quartic.}
{\qloc{kms:thm:critical}, \qloc{kms:cor:all39}}
\qrow{\litid{51} Gribling--Klep--Kumar, Section 6 \cite{r67:GKK}.}
{For Gaussian measure, a prescribed $O(r^{-b})$ rate is equivalent to $M_f(t)=O(t^{2b})$, where $M_f$ minimizes the maximum error on a ball of radius $t$ over centers of norm at most $1/t$. This includes unattained infima.}
{\qloc{r67:gaussianrate}}
\qrow{\litid{56} Slot--Laurent, Sec.~6, the convex body question \cite{r69:SL}.}
{For Lebesgue measure the multivariate bound improves from the published $O_f(r^{-2}\log^2r)$ to $O_f(r^{-2})$ on every convex body. Every nonconstant affine objective attains this order.}
{\qloc{r69:convexrate}, \qloc{r69:sublevel}}
\qrow{\litid{60} Slot--Laurent, Sec.~6, general compact sets \cite{r69:SL}.}
{Under the authors' Assumption~1 the multivariate $O(r^{-2})$ bound can fail. One fixed planar Lipschitz domain and affine objective attain the logarithmic loss along a subsequence.}
{\qloc{r69:compactlower}}
\qrow{\litid{59} Slot--Laurent, conclusion on the univariate hierarchy; Laurent, the convex body question \cite{r69:SLpush,r70:Laurent}.}
{The univariate push-forward error can attain order $r^{-2}\log^2r$ on a subsequence even on a uniformly convex body of any prescribed finite $C^m$ regularity ($m\ge2$). Each fixed smooth planar body improves this rate at regular minima.}
{\qloc{r69:pushlower}, \qloc{r70:convexpush}, \qloc{r73:smoothlimit}}
\qrow{\litid{41} Chevyrev--Lyons, Remark 6.19; Liu--Lyu--Ni, Problem 1.2 \cite{ChevyrevLyons,LiuLyuNi}.}
{Positive radius determines the signature law among all word-integrable group-like competitors and all path characteristic functions of geometric rough paths.}
{\qloc{sig:thm:one-sided}, \qloc{sig:thm:positive-radius}}
\qrow{\litid{53} Chevyrev, Ch.~6, pp.~137--138, Bochner and moment existence \cite{r68sig:thesis}.}
{On $\mathbb R^d$, tight finite-array tests characterize realizable word moments and unitary characteristic functions on the entire group. For an entire normalized series, shuffle-square positivity suffices.}
{\qloc{r68sig:exist}, \qloc{r69sig:tight}, \qloc{r68sig:allmoments}}
\qrow{\litid{54} Chevyrev, Ch.~6, p.~137, unitary separation \cite{r68sig:thesis}.}
{Finite-dimensional unitary developments can fail to separate entire group-like elements on a separable Banach space. For projective powers, separation is equivalent to elementary dual tensors separating every level.}
{\qloc{r68:banach}}
\qrow{\litid{55} Cox--Khedher--Maessen, Section 2.4, weakly group-like elements \cite{r68:CKM}.}
{Proposition~2.22 proved equality under approximation and crossnorm hypotheses. A degree-two counterexample disproves general equality. For projective powers we give an exact criterion.}
{\qloc{r68:weakshuffle}}

\qrow{\litid{42} Lyons--Ni, introduction and Remark 3.6 \cite{LyonsNiExit}.}
{At Brownian exit from any bounded open set the expected signature determines its law among word-integrable group-like competitors.}
{\qloc{sig:thm:stopped}}
\qrow{\litid{43} Li--Ni, UCL's Future Work question for $d>8$ \cite{BrownianRadiusUCL,LiNi}.}
{The radius is positive and finite on every bounded open set in $d\ge2$, extending the smooth-domain result for $2\le d\le8$. For general stopping times the coordinate $\ell^1$ radius satisfies $\sqrt{\beta}/(\sqrt8\,d)\le R\le4\sqrt{\beta}$, where $\beta$ is the exponential moment threshold.}
{\qloc{br42:thm:radius}}

\par\needspace{4\baselineskip}\smallskip\noindent{\bfseries Part IV. Stationary processes and diagonals}\par\smallskip
\qrow{\litid{82} Norros--Saksman, Sec.~5, Q.1--Q.3 \cite{r83:NS}.}
{Rational values of the variance determine local independence for arbitrary spectra. A local covariance test applies at a given separation. With zero point germs the possible failure sets are exactly the relatively closed subsets of the positive half-line.}
{\qloc{r84:gram}, \qloc{r83:nsweights}, \qloc{r83:nslocal}, \qloc{r83:bad}, \qloc{r84:typedbad}}

\qrow{\litid{94} Kania, Sec.~7, P2--P4 and P6 \cite{r88:Kania}.}
{Gaussian control ranges and the compensated jump triplet determine equivalence of path laws, also for Banach-valued jumps and two-time evolutions.}
{\qloc{r88:jumpcomparison}, \qloc{r88:jumpmass}}

\qrow{\litid{88} Pillai, Sec.~10 \cite{r86:Pillai}.}
{At the true inverse range the scale estimate is consistent exactly for $\nu<2$. Joint consistency fails at every smoothness on any compact inverse-range set containing the truth and a competitor.}
{\qloc{r86:whittle}, \qloc{r86:whittlerange}}

\qrow{\litid{85} Khalaf et al., Remark~1 \cite{r85:Khalaf}.}
{On each finite horizon the mixed weighted fractional process has a law equivalent to standard Wiener measure exactly for $b>1/2$. It is a martingale in its natural filtration exactly for $b=0$.}
{\qloc{r85:weighted}}

\qrow{\litid{83} Norros--Saksman, Sec.~5, Q.4 \cite{r83:NS}.}
{An interval of fractional compound-Poisson increments recovers the entire jump past for the causal kernel and the whole marked process for the well-balanced kernel. Both have fractional Brownian covariance up to scale.}
{\qloc{r83:levylocal}}
\qrow{\litid{15} Bradley, Questions 8 and 9, Secs.~6.5 and 7 \cite{BradleyMixing2005}.}
{For nonzero spectral measures, $r^*(n)\to0$ exactly for continuous densities with finitely many zeros, each asymptotic to $|u|^{2k}L(1/|u|)$ with $k\ge1$ integral and $L$ slowly varying. For real Gaussian sequences this characterizes interlaced $\rho^*$-mixing.}
{\qloc{mix43:full}}
\qrow{\litid{61} Ayache, Question~1.6 (author's manuscript) \cite{r76:ayjfa} and Lille report, p.~34, (4.15) \cite{r76:aytalk}.}
{Finite-range stationary Gaussian processes with the same integer-sample law can be monofractal or multifractal in every interval, even with increment variances agreeing to any prescribed finite order.}
{\qloc{r76:ayache}, \qloc{r78:variances}}
\qrow{\litid{16} Grivaux--Matheron--Menet, Question 7.1 \cite{ws42:GMM}.}
{Orthogonality on $\ell_p$ is equivalent to a finite weight certificate for every pair of positive integer test parameters. The certificates use rational polynomials on finite cubes.}
{\qloc{ws43:finite}}
\qrow{\litid{17} Grivaux--Matheron--Menet, Questions 7.2, 7.3 and 7.5 \cite{ws42:GMM}.}
{A similar orthogonal pair with nontrivial invariant laws refutes 7.2 and 7.3. A nonorthogonal pair without equivalent nontrivial invariant laws refutes 7.5.}
{\qloc{ws42:separate}}
\qrow{\litid{18} Grivaux--Matheron--Menet, Questions 7.4, 7.6 and 7.7 \cite{ws42:GMM}.}
{On $\ell_p$, equivalent nontrivial invariant product laws admit Gaussian replacement (7.4). We characterize a shared nontrivial invariant law on $\ell_p$ and $c_0$ (7.6). On $c_0$ a single shift admits one exactly when $|w_1\cdots w_n|\to\infty$ (7.7).}
{\qloc{ws42:shift}}
\qrow{\litid{63} Grivaux--Matheron, Question~1.6 \cite{r79:GM}.}
{The frequently hypercyclic shift of \cite[Proposition~9.9]{r79:BGMM} has only the point mass at zero as an invariant probability.}
{\qloc{r79:fullshift}}
\qrow{\litid{64} Grivaux--Matheron--Menet, Question~8.12 \cite{r79:GMMbook}.}
{Positive return density at zero and a full-support invariant law imply upper frequent hypercyclicity. Ergodic component masses give a lower bound for the generic visit frequency.}
{\qloc{r79:upper}}
\qrow{\litid{65} Cardeccia--Muro, Section~4, Question~3 \cite{r79:CM}.}
{The implication fails on a Hilbert sequence space. With an unconditional coordinate basis, a nontrivial invariant law exists exactly when the reciprocal weight products belong to the space. A full-support Bernoulli law then exists.}
{\qloc{r79:support}}
\qrow{\litid{72} Grivaux--Matheron--Menet, Question~8.9 \cite{r79:GMMbook}.}
{Hypercyclicity of $D_\lambda+B_\omega$ is equivalent to one limit of finite determinant minima. The matrices use the weights and the integer relations among the diagonal phases.}
{\qloc{r81:phase}, \qloc{r81:shiftradius}}
\qrow{\litid{76} Grivaux--Matheron--Menet, Question~8.10 \cite{r79:GMMbook}.}
{For $M>1$ a typical upper triangular operator of norm at most $M$ with distinct unimodular diagonal entries is $M$ times a backward shift of infinite multiplicity, up to unitary equivalence. It is chaotic.}
{\qloc{r81:triangular}}

\qrow{\litid{67} Grivaux--Matheron, Question~1.15 \cite{r79:GM}.}
{A frequently hypercyclic homeomorphism of a compact connected metric space can have dense periodic points and no full-support ergodic probability.}
{\qloc{r80:compact}, \qloc{r80:GMcompact}}
\qrow{\litid{69} Grivaux, Questions~5.4 and~5.5 \cite{r80:GrUniversal}.}
{On complex $\ell^2$, the Hardy multiplier adjoint $2B-I$ is $2$-universal for all ergodic systems and has unimodular point spectrum $\mathbb T\setminus\{1\}$. Under the single-chain criterion every finite set of missing eigenvalues is realizable.}
{\qloc{r80:tangent}, \qloc{r80:holes}}
\qrow{\litid{70} Grivaux, Questions~5.3 and~5.10 \cite{r80:GrUniversal}.}
{An invertible Hilbert-space model can be $2$-universal for invertible ergodic systems with any prescribed nonempty proper open subset of the circle as its unimodular point spectrum. The possible Haar measures are exactly $(0,1]$.}
{\qloc{r80:spectralgenerator}, \qloc{r80:outsidechain}}
\qrow{\litid{71} Grivaux, Question~5.6 \cite{r80:GrUniversal}.}
{The quadratic symbol $\tfrac32(z-1/2)^2$ covers the closed disk but its adjoint realizes only invertible invariant systems. For symbols analytic near the closed disk, a proper component gives the full criterion.}
{\qloc{r81:backwardkernel}, \qloc{r81:hardy}, \qloc{r82:history}}
\qrow{\litid{84} Grivaux, Questions~5.7--5.8 \cite{r80:GrUniversal}.}
{A universal Hilbert operator can have no unimodular eigenvalues. Every nontrivial invariant law has infinite norm moments at all positive orders. For invertible ergodic sources the model operator can itself be invertible with empty point spectrum.}
{\qloc{r83:universal}}
\qrow{\litid{75} Flytzanis, questions after Lemma~1 and at the end of Section~3 \cite{r81:Fly79}.}
{The adjoint spectral conditions can hold while all invariant laws lie in a hyperplane. Every nonatomic ergodic system with discrete spectrum also has a spanning model in the stated class with all positive norm moments infinite.}
{\qloc{r81:flyadjoint}, \qloc{r81:flymom}}
\qrow{\litid{77} Steenbeck, proposed analytic semigroup in Remark~1.4 \cite{r81:Steen}.}
{An entire group with empty point spectrum at each nonzero time has a common full-support group-ergodic law in weak $L^1$. Each such time has only the trivial invariant law of finite first moment.}
{\qloc{r81:analyticgroup}, \qloc{r82:analyticfull}}

\qrow{\litid{78} Grivaux--Matheron, Question~3.19 \cite{r79:GM}, Grivaux--Matheron--Menet, Question~8.2 \cite{r79:GMMbook}, and the moment question following Question~5.6 in \cite{r80:GrUniversal}.}
{An invertible Hilbert-space operator with no eigenvalues has a full-support ergodic law with infinite entropy. Every nontrivial invariant law has infinite first moment.}
{\qloc{r82:fullsupport}, \qloc{r82:supports}, \qloc{r82:tape}, \qloc{r82:hereditary}}

\qrow{\litid{79} Bayart--Grivaux--Matheron--Menet, Questions~9.16 and~9.20 \cite{r79:BGMM}.}
{An arc-mass bound gives hereditary frequent visits at prescribed times. A chaotic Hilbert-space example has powers converging to the identity in norm and hence fails topological mixing.}
{\qloc{r82:arcbound}, \qloc{r82:rigidhfhc}}
\qrow{\litid{80} Bayart--Grivaux--Matheron--Menet, Question~9.22 \cite{r79:BGMM}.}
{Perfectly spanning unimodular eigenvectors need not give hereditary frequent hypercyclicity. One positive-lower-density time set and two balls witness failure for every vector.}
{\qloc{r82:thinrigid}}
\qrow{\litid{19} Loreaux--Weiss, Conjecture 3.6 \cite{LW} and the open problem after \cite[Theorem~2.8]{r56lw:survey}.}
{Approximate shifted majorization is insufficient. The summable-loss test treats finite nonzero kernels, including the two spectra in the survey. It also decides orthostochastic realization for two given nonnegative null sequences \cite[before Theorem~2.4]{LW}.}
{\qloc{schur:thm:finite40}, \qloc{schur:cor:all40}, \qloc{r66bj:orthostochastic}}
\qrow{\litid{22} Bownik--Jasper I, compact kernel problem in Section 11 \cite{BJ}.}
{Theorem~11.5 of \cite{BJ} splits a zero-excess diagonal at zero. The integral criterion determines the admissible kernel dimensions on each sign.}
{\qloc{r55bj:compact}, \qloc{r67bj:convex}}

\qrow{\litid{23} Bownik--Jasper II, the two noncompact kernel classes in Section 12 \cite{newbj:II}.}
{Theorems~11.1 and 11.3 of \cite{newbj:II} leave a positive and a one-sided kernel problem. Testing the actual endpoint multiplicities decides both.}
{\qloc{r62bj:positive}, \qloc{r62bj:signedall}, \qloc{r68bj:alltwo}}
\qrow{\litid{21} Bownik--Jasper, Banff workshop problem, Sec.~4.6.1, p.~10 \cite{r60bj:BIRS}.}
{The published spectral reductions and the three terminal kernel tests characterize all complete-basis diagonals of a fixed bounded self-adjoint operator.}
{\qloc{r60bj:global}}

\qrow{\litid{49} Bownik--Jasper, frame norms, Problem 1 \cite{r63bj:frames}.}
{The Gram reduction \cite{r63bj:AMRS} and the diagonal criterion for $S\oplus0_k$ give all frame norm sequences with prescribed operator $S$ and all possible excesses $k$. The source already covers finite spectrum.}
{\qloc{r63bj:frame}}
\qrow{\litid{89} M\"uller--Tomilov, Remark~4.8 \cite{r86:MTToeplitz}.}
{For arbitrary off-diagonal coefficients on $\ell_p$, $1<p<\infty$, the essential numerical range is determined by the least concave majorant of an explicit scalar function in each direction.}
{\qloc{r86:toeplitz}}

\qrow{\litid{57} M\"uller--Tomilov, Sec.~4, joint constant diagonals \cite{r69:MT}.}
{Three noncommuting bounded operators have two joint constant diagonals with no realizable point on the open segment between them.}
{\qloc{r69:constant}}
\qrow{\litid{81} Chan--Li--Poon, first three questions of Sec.~8 \cite{r83:CLP}.}
{For any $1\le k<L$, a compact tuple has closed rank-$p$ range exactly for $k<p\le L$. Its convex hull is closed exactly for $p\le L$. Normalized inclusion fails between ranks $k$ and $k+1$.}
{\qloc{r83:clp}}
\qrow{\litid{52} Dirr--vom Ende, Open Problems (b), convexity in Theorem 3.8 \cite{r67:DvE}.}
{The published theorem treats the closure. The exact range $W_C(T)$ is convex for normal trace-class $C$ with collinear spectrum. The alternative hypothesis on $T$ does not ensure convexity even for a projection.}
{\qloc{r67:exactC}}
\par\smallskip\hrule\smallskip
\xrow{\litid{58} Cox--Khedher--Maessen, Outlook, Sec.~1.1 (research direction) \cite{r68:CKM}.}
{In the predual setting a weight integral characterizes global density. Clocked words cover the whole space. Ordinary words apply on subsets they separate that are closed on bounded balls.}
{\qloc{r69sig:global}}
\endgroup
\restoregeometry
\part{States, independent factors and their detection}

\section{The intrinsic cumulant class}\label{s:sec:intrinsic}

Let $H$ be a real Hilbert space and $\sigma\ge1$. All tensor
coefficients below lie on its complexification. Symmetric tensor
powers have their Hilbert norms. For a real centered formal germ
$c(h)=\sum_{r\ge2}\kappa_r(h^{\otimes r})/r!$, define the Appell
polynomials in the abstract unital polynomial algebra on linear
generators $X(h)$ by the formal identity
\[
 E_c(h)=\exp(X(h)-c(h)1)
       =\sum_{n\ge0}J_n(h^{\otimes n})/n!.
\]
For an analytic Laplace transform $e^c$, this normalized exponential
defines the generalized Appell polynomials of Kondratiev, Silva and
Streit \cite[Section~3]{r85:Appell}. Finkelshtein et al. give conditions
for Sheffer operators to act as linear homeomorphisms on spaces of
entire functions of minimal type over co-nuclear spaces \cite{r85:Sheffer}.
Here the tensor norm on each symmetric power comes from a single Hilbert
space. Proposition~\ref{s:prop:cumulants} characterizes continuity of
multiplication by the norms of the cumulants across tensor splittings.

Polarization defines $J_n(f)$ for finite tensors. Their leading
polynomial terms are the ordinary symmetric monomials, so the
triangular change of coordinates is invertible on the polynomial
core. Its completion consists of $F=(f_n)_{n\ge0}$ with
\[
 p_\rho(F)=\sum_{n\ge0}\rho^n(n!)^\sigma\|f_n\|<\infty
 \quad\text{for every }\rho\ge1 .
\]
Write $\varepsilon F=f_0$. The involution is conjugation, extended
from the real polynomial generators. A formal series with symmetric
coefficients $b_r$ is \emph{full at exponent $\sigma$} when each
coefficient is a full Hilbert tensor and $\|b_r\|\le CR^r(r!)^\sigma$
for common $C,R$. This condition defines $\mathcal H_\sigma(H)$ for
centered series. Write $\mathcal C_\sigma(H)$ for centered germs with
the same bound in the largest of the bilinear operator norms across
the splits $a+(r-a)$, $1\le a<r$. The star product below is ordinary
polynomial multiplication in the Appell coordinates.

\begin{proposition}[Continuity of multiplication]\label{s:prop:cumulants}
\label{s:thm:vacuum}
Multiplication extends jointly continuously if and only if
$c\in\mathcal C_\sigma(H)$. Equivalently its degree-zero pairing
extends continuously. The resulting algebra is denoted $A_{c,\sigma}(H)$.
\end{proposition}
\begin{proof}
The formal product is $E_c(u)E_c(v)=K(u,v)E_c(u+v)$ with
$\log K=c(u+v)-c(u)-c(v)$. Use exponential coefficients
$K_{a,b}$, divided by $a!b!$ in the series. In the coefficient norm
$\sum t^{a+b}\|K_{a,b}\|_{\rm op}/(a!b!)^\sigma$, multiplication is
submultiplicative. Tensor products multiply operator bounds,
symmetrizations are contractions, and each group contributes the
factor $\binom{j+k}{j}^{1-\sigma}\le1$. The mixed coefficients of
$\log K$ are exactly the indicated flattenings of $\kappa_{a+b}$.
Since $(a+b)!/(a!b!)\le2^{a+b}$, the split bound is equivalent to
a finite coefficient norm at some $t>0$. The same holds for $K$ by
the exponential series and, in reverse, the logarithm after shrinking
$t$ until $\|K-1\|_t<1$.

Continuity of the degree-zero pairing bounds its restrictions by
$\|K_{a,b}\|_{\rm op}\le CR^{a+b}(a!b!)^\sigma$. Conversely,
comparing coefficients in the product gives, for $J_m(f)J_n(g)$,
the sum of contractions of $a$ slots of $f$ and $b$ slots of $g$
against $K_{a,b}$, followed by symmetrization of the remaining
$s=m+n-a-b$ slots. The multiplier is $\binom ma\binom nb$.
Viewing $f,g$ as Hilbert--Schmidt matrices in these splits bounds
the contracted tensor by $\|f\|\|K_{a,b}\|_{\rm op}\|g\|$. Its
factorial multiplier divided by $(m!n!)^\sigma$ is
\[
 \binom ma^{1-\sigma}\binom nb^{1-\sigma}\binom{s}{m-a}^{\sigma}
 \le \binom ma\binom nb\,2^{\sigma s}.
\]
Summing gives $p_\rho(FG)\le C p_{2^\sigma\rho+R}(F)
p_{2^\sigma\rho+R}(G)$. Density gives the continuous product, which
in turn bounds the degree-zero pairing.
\end{proof}

\begin{theorem}[A prescribed affine substitution]\label{s:thm:scalar-maps}
For two such algebras of common exponent, let $T:H\to K$ and
$\ell:H\to\mathbb R$ be bounded real maps. The polynomial substitution
$\Phi X_H(h)=X_K(Th)+\ell(h)1$ extends to a continuous unital star
homomorphism exactly when $D=\ell+c_K\circ T-c_H$ is full. No
positivity assumption is required for this prescribed substitution.
\end{theorem}
\begin{proof}
Write $B=e^D=\sum_n b_n(h^{\otimes n})/n!$. Formally
$\Phi E_{c_H}(h)=B(h)E_{c_K}(Th)$. Coefficient comparison and
polarization give
\[
 \Phi J_n(f)=\sum_{j=0}^n\binom nj
      J_j\!\left(T^{\odot j}C_{b_{n-j}}f\right),
\]
where $C_b$ contracts the indicated slots. If $\Phi$ is continuous,
evaluation by the target vacuum gives
$|b_n(f)|\le CR^n(n!)^\sigma\|f\|$. Riesz representation gives the
full coefficient bound for $B$. Conversely under that bound,
$\|C_bf\|\le\|b\|\|f\|$ shows that the $j$th summand, divided by
$(n!)^\sigma\|f\|$ in $p_\rho$, costs at most
$C\binom nj^{1-\sigma}(\rho\|T\|)^jR^{n-j}$. Thus
$p_\rho(\Phi F)\le C p_{R+\rho\|T\|}(F)$ and the identities extend
by density. The full coefficient norm $\sum t^n\|b_n\|/(n!)^\sigma$
is submultiplicative by the same binomial calculation as above.
Exponentiation in this norm and the logarithm at a smaller radius
show that $B$ is full exactly when $D$ is.
\end{proof}

Call a continuous complex-linear functional a state if $\omega(1)=1$
and $\omega(F^*F)\ge0$. A nonzero continuous positive functional
normalizes to a state by Cauchy--Schwarz. For any such state in a
chart put $B_\omega(h)=\omega(E_c(h))$ and $a_\omega(h)=\omega(X(h))$.
The centered cumulant of this state satisfies
\begin{equation}\label{new:eq:state-cumulant}
 c_\omega(h)=\log\omega\big(\exp(X(h)-a_\omega(h)1)\big)
            =c(h)+\log B_\omega(h)-a_\omega(h).
\end{equation}
These are coefficientwise identities. Continuity of $\omega$ gives
the full coefficient bound on $B_\omega$, hence on the centered
correction. The substitution from the old to the $c_\omega$ chart
with linear part $I$ and shift $a_\omega$ has symbol $B_\omega$
by the preceding theorem. Its inverse has symbol $1/B_\omega$.
It is a filtered star isomorphism and the new vacuum pulls back to
$\omega$ by the formula on $E_c$. In particular the new vacuum is positive.

\begin{lemma}[Degree from powers]\label{s:thm:degree}
Suppose $A_{c,\sigma}(H)$ admits a continuous state. For $\alpha\ge0$,
the set of $F$ satisfying $p_\rho(F^n)\le C_\rho L_\rho^n(n!)^\alpha$
for all $n\ge1$ and every $\rho$ is precisely the completed polynomial
layer of degree at most $\lfloor\alpha/\sigma\rfloor$.
\end{lemma}
\begin{proof}
The preceding change of chart preserves degree and all continuous
seminorm estimates, so take the vacuum positive. For degree at most
$d$ iterate the product estimate on a balanced tree of height
$\lceil\log_2 n\rceil$. At the leaves the radius is $O_\rho(n^\sigma)$,
while all $n-1$ multiplications cost at most $C^{n-1}$ in prefactors.
Since $n^n\le e^n n!$, the desired upper bound follows with
$\alpha=\sigma d$. This direction uses only continuity.

Let $\tau_z$ substitute $X(h)\mapsto X(h)+[z,h]1$, where $[\,,\,]$
is the complex-bilinear Hilbert pairing. On $J_n$ the translation
expansion contracts $j$ slots against $z^{\otimes j}$ with coefficient
$\binom nj$. Relative to $\rho^n(n!)^\sigma$ the resulting seminorm
bound is at most $\sum_j(\|z\|/\rho)^j/(j!)^\sigma$, since
$\binom nj^{1-\sigma}\le1$. This sum is bounded by
$2\exp(\sigma(2\|z\|/\rho)^{1/\sigma})$, by bounding each summand
by $2^{-j}\exp(\sigma(2\|z\|/\rho)^{1/\sigma})$.
Consequently $\tau_z$ extends multiplicatively on the completion and
$\widehat F(z)=\varepsilon\tau_zF=\sum_n[f_n,z^{\otimes n}]$ is
entire on every complex line.

For dyadic $n$, repeated Cauchy--Schwarz against the unit gives
$|\varepsilon A|^{2n}\le\varepsilon((A^n)^*A^n)$. Apply this to
$A=\tau_z F$ and use continuity of the vacuum pairing. If $F$ has
the stated power bound, then
$|\widehat F(z)|^n\le C L^n(n!)^\alpha\exp(C'\|z\|^{1/\sigma})$.
Taking the $n$th root on $\|z\|\le n^\sigma$ bounds the symbol there
by $C''n^\alpha$. The dyadic radii yield global growth of degree at
most $\alpha/\sigma$. Cauchy's estimate on each line followed by
polarization removes every coefficient above that degree.
\end{proof}

\begin{theorem}[Intrinsic class and lift criterion]\label{new:thm:scalar-invariants}
For $H\ne0$ in a positive Appell algebra as above, the topology and
multiplication recover $\sigma$, the real Hilbertizable observable
space $V(A)$, and a class
$\mathfrak c(A)\in\mathcal C_\sigma(V(A))/\mathcal H_\sigma(V(A))$.
A bounded real $T:V(A)\to V(B)$ at a common exponent lifts to a
continuous unital star homomorphism exactly when
$T^\#\mathfrak c(B)=\mathfrak c(A)$. All nonzero continuous star
homomorphisms have this form, with an arbitrary bounded real affine
shift. A lift is an isomorphism exactly when $T$ is boundedly invertible.
\end{theorem}
\begin{proof}
The power bounds can equivalently be tested against all continuous
seminorms. By the degree lemma, $\sigma$ is the least $\alpha$ whose
power-growth class contains a nonscalar. Its class at $\alpha=\sigma$
is $\mathcal F_1=\mathbb C1\oplus X(H_{\mathbb C})$, so
$V(A)=(\mathcal F_1)_{\rm sa}/\mathbb R1$ recovers the real observable
space with its quotient topology. For a state $\omega$ the section
$v\mapsto x-\omega(x)1$, where $x$ is any real lift of $v$, is intrinsic.
Its formal logarithmic moment series is $c_\omega$; formula
\eqref{new:eq:state-cumulant} gives the same quotient class for every
state. Equivalent Hilbert norms give the same quotient because bounded
pullback in $r$ slots costs at most the $r$th power of its operator norm.

A nonzero continuous star homomorphism sends the unit to a nonzero
idempotent, which satisfies the degree-zero power bound and is
therefore the unit. Continuity preserves all the power-growth
classes, giving a bounded real affine action on $\mathcal F_1$.
Freeness on the polynomial core fixes the map there. In state charts
the prescribed substitution theorem says exactly that its centered
cumulant defect is full, with every bounded affine shift allowed.
For boundedly invertible $T$ the inverse substitution has full defect
$-D\circ T^{-1}$. Conversely a topological isomorphism induces mutually
inverse bounded maps on the recovered observable spaces.
\end{proof}

\section{Which parts of a state can stay non-Gaussian?}\label{r59:main}
Fix a continuous state with centered cumulants $k_r$ and bounded
covariance $B$, and a family $(P_j)_{j\in J}$ of real orthogonal
projections on its observable Hilbert space $H$. Splitting a cut
means independence in polynomial moments of its two ranges. Put
$M=\{P_j:j\in J\}'$. For a profile $\chi:J\to\{0,1\}$ put
\[
 H_\chi=\bigcap_j\ker(P_j-\chi(j)I),\qquad
 H_{\rm rem}=\left(\bigoplus_\chi H_\chi\right)^\perp ,
\]
and let $E_\chi$ project onto $H_\chi$. Thus within $H_\chi$ every
cut gives the same answer for every vector. Define
$D_r k_r=\sum_\chi k_r(E_\chi\,\cdot,\ldots,E_\chi\,\cdot)$ for $r\ge3$.
This sum converges absolutely on each tuple by H\"older and
orthogonality. For a difference $h$ put
$\|h\|_t^2=\sum_{r\ge2} t^{2r}\|h_r\|_2^2/(r!)^{2\sigma}$, with
infinite cost for a non-Hilbert coefficient. A difference is full
exactly when this is finite at some $t>0$.

\begin{theorem}[The Gaussian part forced by independence]\label{r60:structure}
In a state splitting every cut, the covariance commutes with every
$P_j$. Its higher cumulants are supported on single spaces $H_\chi$.
Consequently its centered polynomial moments factor over the
$H_\chi$ and a Gaussian state on $H_{\rm rem}$.

For a given state $k$, write $\mathcal I$ for the centered cumulants of
all continuous states on its algebra splitting the cuts, and put
$d_2(B,M)=\inf_{0\le C\in M}\|B-C\|_2$. At every $t>0$,
\begin{equation}\label{r60:distance}
 \inf_{k'\in\mathcal I}\|k-k'\|_t^2
 =\frac{t^4}{(2!)^{2\sigma}}d_2(B,M)^2+
   \sum_{r\ge3}\frac{t^{2r}}{(r!)^{2\sigma}}\|k_r-D_r k_r\|_2^2 .
\end{equation}
A simultaneous continuous state exists precisely when this is finite
for some $t$. At each such radius the unique closest centered
cumulants keep the original marginals on every $H_\chi$ and use the
closest commuting covariance on the remaining Gaussian part.
\end{theorem}
\begin{proof}
In degree two splitting means $P_jB=BP_j$. Suppose a bounded
symmetric $r$-form $u$, $r\ge3$, splits all cuts. A projection in one
slot can then be moved to any other slot. The radical
$N=\{x:u(x,\cdot,\ldots,\cdot)=0\}$ is closed and reduces each cut.
For two cuts $P,Q$, transfer $P$ to slot two and $Q$ to slot three:
both $u(PQx,y,z,\ldots)$ and $u(QPx,y,z,\ldots)$ equal
$u(x,Py,Qz,\ldots)$. Hence $[P,Q]H\subset N$, and the restricted
projections on $N^\perp$ commute.

On $N^\perp$ let a finite set $F$ of cuts partition it into joint
ranges with projections $E$. For $x$ orthogonal to all common
profile spaces in $N^\perp$, $\max_E\|Ex\|\to0$ under refinement.
Otherwise for some $\delta>0$ every $F$ has a cell of mass
$\|Ex\|^2\ge\delta$. In the compact product $\{0,1\}^J$, the closed
cylinder sets of profiles with that property have the finite
intersection property. The corresponding decreasing projections for
a common profile converge strongly to its joint projection, which
would have mass at least $\delta$. This contradicts the choice of $x$.
Splitting now gives
\[
 |u(x,\ldots,x)|\le\|u\|
   (\max_E\|Ex\|)^{r-2}\sum_E\|Ex\|^2\longrightarrow0 .
\]
Polarization removes the purely residual restriction. For a slot in
a common profile every other slot can be projected onto the same
finite cell, by splitting, and then onto that profile by strong
convergence. Thus all mixed terms with such a slot vanish. Finite
orthogonal sums and continuity give $u=D_ru$, including on the
original $H$ since $N$ reduces the cuts. This proves the first claim;
the moment--cumulant formula gives the factorization.

For any competitor at finite distance the complementary orthogonal
projection $I-D_r$ on Hilbert tensors extracts $k_r-D_r k_r$
from its degree-$r$ difference. This gives the lower bound.
If the right side is finite, the bounded self-adjoint commuting
covariances at finite Hilbert--Schmidt distance from $B$, initially
allowing either sign, form a closed affine subspace of $B+\Schat_2$.
Its nearest element $C$ is positive: the spectral projection of $C$
below $-\epsilon$ is finite rank since $C-B$ is compact, and if
nonzero contradicts orthogonality of $B-C$ to that projection.
Every rank-one self-adjoint operator supported on $H_\chi$ commutes
with all cuts, so $E_\chi C E_\chi=E_\chi B E_\chi$.

Take the independent product of the original centered atom marginals
and the centered Gaussian of covariance $C$ on the remainder.
To see positivity with infinitely many atoms, keep finitely many
marginals and use $C$ as a Gaussian covariance on their complement.
These positive polynomial states converge in each moment by the
absolute coefficient sums and the moment--cumulant formula.
The resulting germ $(C,D_3 k_3,D_4 k_4,\ldots)$ differs from $k$
by a full sequence, so the positive vacuum of its chart pulls back
continuously by Theorem~\ref{s:thm:scalar-maps}. Orthogonal projection
in degree two and onto the pure atom blocks in higher degrees gives
Pythagoras against every other competitor at finite distance. This
gives uniqueness and \eqref{r60:distance}. Existence at some radius
follows in both directions using \eqref{new:eq:state-cumulant}.
\end{proof}

For example, take on $L^2[0,1]$ all cuts by dyadic intervals. They
commute and have no nonzero common profile space. Every continuous
state splitting them has Gaussian polynomial moments in every finite
family of observables. In a two-dimensional plane the projections on
lines separated by an angle of $\pi/4$ also have no common profile.
They force the same Gaussian conclusion, now with a scalar covariance.
\subsection{When independence holds only on average}
For a random orthogonal projection $P$ of law $\xi$, let $X_1,X_2$
be independent copies of $X$, independent of $P$. Courtade's
Question~1 \cite{r63co:C} starts from the distributional identity
\begin{equation}\label{r63co:mixture}
 X\ \overset{d}=\ PX_1+(I-P)X_2 .
\end{equation}
Courtade asks whether uniform cutting in finite dimension forces a
logarithmic moment of a fixed law.
Courtade's Theorem~1 already classifies the finite-dimensional law
once its dependent marginal has a logarithmic moment. The missing
moment follows in Hilbert space by testing the identity with one radial
function.
\begin{theorem}[Automatic moments under random cuts]\label{r63co:mom}
Let $X$ be a Borel random vector in a nonzero separable real Hilbert
space $H$ satisfying \eqref{r63co:mixture}. Suppose
\[
 \eta=\inf_{\|v\|=1}\int\|Pv\|^2\|(I-P)v\|^2\,d\xi(P)>0.
\]
Then $\mathbb E\|X\|^2<\infty$ and
$\mathbb E\|X\|^4\le(4\eta)^{-1}(\mathbb E\|X\|^2)^2$, with optimal universal factor $1/4$.
In particular the hypothesis of Courtade's Question~1,
$\int\min(\|Pv\|,\|(I-P)v\|)\,d\xi\ge\theta\|v\|$, suffices.
\end{theorem}
\begin{proof}
Put $g(s)=1-e^{-ts}$, $t>0$, $S=\|X\|^2$, and $u=\mathbb E g(S)$.
For a fixed $P$ put $A=\|PX\|^2$, $B=\|(I-P)X\|^2$. Then $A+B=S$.
Since $g(A+B)=g(A)+g(B)-g(A)g(B)$, testing
\eqref{r63co:mixture} and cancelling the two marginal terms gives
\[
 \int\mathbb E[g(A)g(B)]\,d\xi
   =\int(\mathbb E g(A))(\mathbb E g(B))\,d\xi\le u^2.
\]
Concavity and $g(0)=0$ give $g(A)\ge (A/S)g(S)$ on $S>0$,
and likewise for $B$. Thus $\eta\mathbb E g(S)^2\le u^2$.
If $\mathbb E S=\infty$, then $u/t\to\infty$ as $t\downarrow0$ by
monotone convergence. Consequently $g(S)/u\to0$ almost surely,
yet these variables have expectation one and uniformly bounded
second moments. Uniform integrability gives the contradiction.
Now divide the truncation identity by $t^2$. Monotone convergence on
both sides and $\mathbb E A+\mathbb E B=\mathbb E S$ give
\[
 \eta\mathbb E S^2\le\int\mathbb E[AB]\,d\xi
 =\int(\mathbb E A)(\mathbb E B)\,d\xi\le (\mathbb E S)^2/4 .
\]
The first inequality is an equality precisely when $X/\|X\|$ minimizes
the average product in the definition of $\eta$ almost surely on
$\{X\ne0\}$. The last is an equality precisely when
$\mathbb E A=\mathbb E B$ for $\xi$-almost every cut.
For sharpness take $X\sim N(0,I_2)$ and choose with equal probabilities
the projections onto lines of angles $0$ and $\pi/4$. For a unit vector
of angle $a$ the average product is $(\sin^2(2a)+\cos^2(2a))/8=1/8$,
and $\mathbb E S^2/(\mathbb E S)^2=2$. On the unit sphere the larger of
$\|Pv\|^2,\|(I-P)v\|^2$ is at least $1/2$. Jensen's inequality
therefore gives $\eta\ge\theta^2/2$ under the stated hypothesis.
\end{proof}

\begin{corollary}[The moment-free mixture criterion]\label{r63co:law}
In finite dimension, \eqref{r63co:mixture} is equivalent to splitting
along $\xi$-almost every cut and hence to Courtade's product of
arbitrary laws on the independent subspaces and a Gaussian law on
their orthogonal complement with splitting covariance.
\end{corollary}
\begin{proof}
In finite dimension the dependent subspace $V$ of \cite[Proposition~1]{r63co:C}
reduces $P$ almost surely. The projection of \eqref{r63co:mixture}
onto $V$ obeys the same identity with the restricted cuts. If $V\ne0$,
the continuous function $\int\min(\|Pv\|,\|(I-P)v\|)\,d\xi$
is strictly positive on its compact unit sphere, since its zero set
is the union of the independent subspaces. Theorem~\ref{r63co:mom}
gives the required logarithmic moment on $V$; for $V=0$ it is automatic.
Courtade's Theorem~1 now gives all the equivalences.
\end{proof}

\subsection{The cut operator on the first differing moment}
Write $T\mu$ for the law on the right of \eqref{r63co:mixture}.
Take the strong-operator support of the projection law as the cut family.
General matrix smoothing transformations and their kinetic equilibria
are studied by Bassetti and Matthes \cite{r64co:BM}. For complementary
orthogonal projections, conditioning on a cut linearizes the first
moment on which the two inputs differ.

\begin{theorem}[The first moment on which two laws differ]\label{r69:firstmoment}
Let $r\ge2$ and let $\mu,\nu$ have finite absolute $r$th moments on a
separable real Hilbert space. Suppose their moments agree through order
$r-1$. Write $m_r(\rho)=\int x^{\otimes r}\,d\rho(x)$ in the symmetric
Hilbert tensor power. For $Q=I-P$ put
\[
 S_r=\int (P^{\otimes r}+Q^{\otimes r})\,d\xi(P),\qquad
 F_r=\mathbf1_{\{1\}}(S_r),\qquad z=m_r(\mu)-m_r(\nu).
\]
At every step the lower moments still agree, and
\[
 m_r(T^n\mu)-m_r(T^n\nu)=S_r^n z\longrightarrow F_rz
 \quad\text{in Hilbert tensor norm}.
\]
In degree two $F_r$ projects onto the symmetric Hilbert--Schmidt
commutant. In degrees $r\ge3$ it is $\sum_\chi E_\chi^{\otimes r}$.
The best constant in
$\|(S_r^n-F_r)z\|\le c_n\|z\|$, even if both inputs must have finite
support, is $c_n=\|S_r|_{(\operatorname{ran}F_r)^\perp}\|^n$ for $n\ge1$.
\end{theorem}
\begin{proof}
Condition on $P$ and expand $(PX+QY)^{\otimes j}$ for two independent
copies of the input. Terms using both copies depend only on moments
of order less than $j$. For $j\le r$ they cancel in the difference,
so induction in the number of steps gives the formula. The two pure
terms at order $r$ are orthogonal tensor projections. Their sum is
a projection, hence its mean $S_r$ is a positive self-adjoint
contraction. Integrating the squared distance to the range shows that
its fixed tensors split every cut in the strong-operator support of
$\xi$. At order two these are the commuting operators. At higher
orders the slot-transfer argument of Theorem~\ref{r60:structure}
applies to the associated bounded symmetric form and leaves exactly
the common profile tensors. The spectral theorem gives convergence
and the asserted norm bound.

For sharpness, the signed measure
$\sum_{j=0}^r(-1)^{r-j}\binom rj\delta_{jh}$ has moments zero below $r$
and $r$th moment $r!h^{\otimes r}$. Finite real linear combinations
of these tensors are dense by polarization. For each nonzero
combination the positive and negative parts of its signed measure
have equal positive mass. Normalizing them realizes the desired
direction as a difference of probability moments. Thus finite laws
test the full operator norm of $S_r^n-F_r$.
\end{proof}

For example, choose equally between the coordinate cut of the plane
and its rotation by $\pi/4$. For $r\ge3$ there is no common profile.
The two pure-tensor ranges each have the orthonormal basis of $r$th
powers of their two axes. Their cross Gram matrix has diagonal entries
$2^{-r/2}$ and off-diagonal entries $2^{-r/2}$ and $(-1)^r2^{-r/2}$.
The norm of the mean of the two projections is $(1+s)/2$, where $s$
is the largest singular value of that Gram matrix (diagonalize the
Gram matrix of the four vectors). Thus both a first cubic discrepancy
and a first quartic discrepancy have sharp contraction factor $3/4$ per step.

\begin{theorem}[The limiting contraction of higher moment discrepancies]
\label{r70:momentgap}
Use the notation of Theorem~\ref{r69:firstmoment} and let
$K=H_{\rm rem}$ for the strong-operator support of $\xi$.
For $r\ge3$ put $c_r=\|S_r-F_r\|$, the sharp one-step factor for the
first differing moment. For a nonzero profile space $H_\chi$ set
$P_\chi(P)=P$ if $\chi(P)=1$, and $P_\chi(P)=I-P$ otherwise. Put
\[
 L=\sup_\chi\left\|\left.\int P_\chi(P)\,d\xi(P)\right|_{H_\chi^\perp}\right\|,
 \qquad B_s=\sup_{h\in K,\ \|h\|=1}\int
       (\|Ph\|^{2s}+\|(I-P)h\|^{2s})\,d\xi(P) \quad(s\ge1).
\]
Suprema over empty families and norms on zero spaces are zero. Then
\begin{equation}\label{r70:momentbounds}
 c_r=\max\{L,\|S_r|_{K^{\odot r}}\|\},\qquad
 B_r\le \|S_r|_{K^{\odot r}}\|\le B_{r/2}.
\end{equation}
The sequence $c_r$ decreases to
\[
 \max\!\left\{L,\lim_{\epsilon\downarrow0}
 \sup_{h\in K,\ \|h\|=1}
 \xi\{P:\min(\|Ph\|^2,\|(I-P)h\|^2)\le\epsilon\}\right\}.
\]
If $0<\dim K<\infty$, the second quantity is
$\max_{h\in K,\ \|h\|=1}\xi\{P:Ph=0\ \text{or}\ Ph=h\}$.
In arbitrary dimension, $c_3<1$ holds exactly when $L<1$ and either
$K=0$ or the cutting constant of Theorem~\ref{r63co:mom} on $K$
is positive. These conditions give exponential contraction uniformly
in all orders $r\ge3$, with sharp factor $c_3$.
\end{theorem}
\begin{proof}
All slots split into the orthogonal reducing spaces $K$ and $H_\chi$.
On a sector with one slot in $H_\chi$, the pure-cut sum retains only
the product of the $P_\chi(P)$ in the other slots. If some other slot
lies in $H_\chi^\perp$, this product is dominated by $P_\chi(P)$
in that slot tensored with identities. A sector with all slots in
$H_\chi$ belongs to $F_r$. This gives the upper bound $L$ on all mixed
sectors containing a profile, also after symmetrization. For the reverse
bound fix a unit $e\in H_\chi$. The isometry
$h\mapsto\sqrt r\,\operatorname{sym}(e^{\otimes(r-1)}\otimes h)$
on $H_\chi^\perp$ intertwines $\int P_\chi(P)\,d\xi(P)$ with $S_r$
outside $F_r$. This proves the first equality. The value $\chi(P)$ on
the support is measurable since it equals $\langle Pe,e\rangle$.

On $K$ the Rayleigh quotient on $h^{\otimes r}$ gives the lower bound
in~\eqref{r70:momentbounds}. For $f_u(x)=\langle u,x^{\otimes r}\rangle$
the identity
$f_{S_ru}(x)=\int[f_u(Px)+f_u((I-P)x)]\,d\xi(P)$ bounds the supremum
of $|f_{S_ru}|$ on the unit sphere of $K$ by $B_{r/2}$ times that of
$|f_u|$. Iterating from $u=h^{\otimes r}$ yields
$\langle S_r^j h^{\otimes r},h^{\otimes r}\rangle\le B_{r/2}^j$.
Positivity and the spectral theorem exclude any spectral mass of that
vector above $B_{r/2}$. The pure powers span a dense subspace by
polarization, proving the upper bound in every dimension.
The pure-cut projection on $r+1$ slots is bounded by the corresponding
projection on the first $r$ slots tensored with the identity. Symmetric
tensors of degree $r+1$ are symmetric in their first $r$ slots,
so the norms on $K$ decrease.

Write $U=\min(\|Ph\|^2,\|(I-P)h\|^2)$. For $0<\epsilon<1/2$, the
integrand defining $B_s$ is bounded above by
$\mathbf1_{\{U\le\epsilon\}}+2(1-\epsilon)^s$
and below by $(1-\epsilon)^s\mathbf1_{\{U\le\epsilon\}}$.
In the upper bound let $s\to\infty$ first, and in the lower bound
let $\epsilon\downarrow0$ first. This gives the stated common limit
of $B_r$ and $B_{r/2}$. In finite dimension a subsequence of unit
vectors approaching the near-binary supremum converges. Their
squared projection lengths converge uniformly in $P$, so continuity
of probability from above gives the exact binary maximum.
Finally, for $0\le a\le1$ and fixed $s>1$, the deficit
$1-a^s-(1-a)^s$ is comparable to $a(1-a)$: their ratio extends
continuously and positively at both endpoints with value $s$.
Thus $B_3<1$ and $B_{3/2}<1$ are each equivalent on nonzero $K$
to a strictly positive cutting constant. Monotonicity and the
finite-support sharpness in Theorem~\ref{r69:firstmoment} finish the proof.
\end{proof}

For a uniformly oriented line cut in the plane, $L=0$ and the sharp
one-step factor in every order $r\ge3$ is
$2^{1-r}\binom r{\lfloor r/2\rfloor}\sim 2\sqrt{2/(\pi r)}$.
Indeed $S_r=2\int e_\theta^{\otimes r}(e_\theta^{\otimes r})^*
\,d\theta/(2\pi)$ for $e_\theta=(\cos\theta,\sin\theta)$ and
$0\le\theta<2\pi$. In the normalized symmetric basis from
$(e_1\pm {\rm i}e_2)/\sqrt2$, the coordinates of $e_\theta^{\otimes r}$
have moduli $2^{-r/2}\sqrt{\binom rj}$ and frequencies $2j-r$,
so integration diagonalizes the matrix. By the finite-dimensional
limit formula, any distribution of planar line cuts with at least two
distinct unordered cuts has limiting factor equal to the largest mass
of a single unordered cut, including zero for an atomless distribution.

\subsection{When local cut defects determine a joint correction}
For a countable family of closed Hilbert subspaces, Reich and Zalas
express the sharp linear error bound for simultaneous projections in
terms of the Friedrichs angle in a weighted product space \cite{r72:RZ}.
For bounded forms and an arbitrary probability law on cuts, two
quantities enter the repair criterion. Write $p_x(P)=\|Px\|^2$ for
unit $x$, and set
\[
 \eta_K=\inf_{\substack{x\in K\\\|x\|=1}}\mathbb E[p_x(1-p_x)],
 \qquad
 \Delta=\inf_{\chi\ne\psi}\mathbb P\{\chi(P)\ne\psi(P)\},
\]
where $K=H_{\rm rem}$ and only nonzero profile spaces are used.
Set $\eta_K=+\infty$ if $K=0$ and $\Delta=1$ if there are fewer than
two profiles. Put $B_P=\ker P\cup\operatorname{ran}P$ and
$B=\bigcap_{P\in\operatorname{supp}\xi} B_P$. The corresponding
stochastic distance bound has the form of the regularity condition
in \cite[Eq.~(14) in the preprint]{r73:HLS} for random feasibility maps.
Here the sets $B_P$ are unions of two orthogonal subspaces.
Let $L_{P,r}=I-P^{\otimes r}-(I-P)^{\otimes r}$.
Recall from Theorem~\ref{r60:structure} that $D_ru=\sum_\chi
u(E_\chi\,\cdot,\ldots,E_\chi\,\cdot)$ is a bounded form.

\begin{theorem}[When separate repairs determine a joint repair]
\label{r71:fulllift}\label{r72:repairiff}\label{r73:geocubic}
For any fixed integer $r\ge3$, every bounded symmetric $r$-form with
$\mathbb E\|L_{P,r}u\|_2^2<\infty$ has a joint Hilbert correction
$(I-D_r)u$ if and only if, for some $C<\infty$,
\begin{equation}\label{r73:directiontest}
 \operatorname{dist}(x,B)^2\le C\,\mathbb E\operatorname{dist}(x,B_P)^2
          \qquad(x\in H).
\end{equation}
This holds precisely when $\eta_K>0$ and $\Delta>0$, or equivalently
$c_3<1$. Quantitatively,
\[
 \min\{\eta_K,\Delta\}\le1-c_3\le\min\{3\eta_K,\Delta\},\qquad
 (1-c_r)\|(I-D_r)u\|_2^2\le\mathbb E\|L_{P,r}u\|_2^2
\]
where the second estimate holds under the equivalent conditions and
is optimal when the Hilbert complement is nonzero.
The same conditions characterize joint same-radius repair of every
positive germ of covariance $I$ with finite mean squared separate cost.
If they fail, then for each radius $t>0$ and $\varepsilon>0$ there is
a positive germ of covariance $I$ with mean squared separate cost below
$\varepsilon$ at $t$ whose joint defect is non-Hilbert in every degree
$r\ge3$.
\end{theorem}
\begin{proof}
On $H_\psi$ the mean squared mismatch with a different profile $\chi$
is $\mathbb P\{\chi(P)\ne\psi(P)\}$. On $K$ it is at least $\eta_K$,
since $|p_x-\chi|\ge p_x(1-p_x)$. Thus $1-L\ge
\min\{\eta_K,\Delta\}$, and $1-L\le\Delta$ by testing pairs of profiles.
For $K\ne0$ the two bounds in~\eqref{r70:momentbounds} and
$1-p^{3/2}-(1-p)^{3/2}\ge p(1-p)$ give
$\eta_K\le 1-\|S_3|_{K^{\odot3}}\|\le3\eta_K$.
This proves the gap estimates, also when $K=0$.
For unit $x$, $\operatorname{dist}(x,B_P)^2=\min(p_x,1-p_x)$, between
$p_x(1-p_x)$ and $2p_x(1-p_x)$. Writing $m=\sup_\chi\|E_\chi x\|^2$
(zero without profiles), we have $\operatorname{dist}(x,B)^2=1-m$ and
$1-m\le1-\sum_\chi\|E_\chi x\|^6\le3(1-m)$. Thus the pure-cube
estimate~\eqref{r72:purecube} below gives~\eqref{r73:directiontest}
exactly for a positive gap, with the full-profile case immediate.

Suppose the gap is positive. The mean $b=\mathbb E L_{P,r}u$ is a
Hilbert tensor orthogonal to the profiles by separability and testing
on basis tensors. Solve $(I-S_r)w=b$ in their Hilbert complement,
where the operator is bounded below by $1-c_r\ge1-c_3$. The bounded
form $Y=u-w$ satisfies $S_rY=Y$. On $K^r$ its diagonal supremum
is at most $B_{r/2}<1$ times itself when $K\ne0$, so polarization
gives zero. With one slot in $H_\chi$ and another in $H_\chi^\perp$,
only the $P_\chi$ branch survives. The latter slot satisfies
$\mathbb E\|P_\chi x\|\le\sqrt L\|x\|$ by Cauchy--Schwarz, and all
others are contractions. The norm on this product of reducing spaces
is at most $\sqrt L<1$ times itself. Thus $Y=D_rY$ by finite block
expansion and continuity. The correction is $w$, and the Hilbert
spectral bound for $I-S_r$ proves the optimal estimate.
Summing over degrees gives the positive-germ implication
by~\eqref{r60:distance}.

For the converse put $e(x)^2=\mathbb E\min(\|Px\|^2,\|(I-P)x\|^2)$.
If $\eta_K=0$, there are unit $y_n\in K$ with $e(y_n)\to0$. Every
weak limit of a subsequence has $e(y)=0$ by lower semicontinuity and
Fatou's lemma. Such a limit is binary at every cut in the support by
strong continuity, so it is zero in $K$. Subtracting projections on
successive finite spans and renormalizing now gives orthonormal $x_j\in K$
with $\sum_j e(x_j)<\varepsilon$, since $e$ is $1$-Lipschitz.
If $\Delta=0$, choose disjoint pairs of profiles whose disagreement
probabilities tend to zero arbitrarily fast. Indeed these probabilities
form an $L^1(\xi)$ metric: distinct labels differ on the support by
continuity. If arbitrarily close pairs meet a fixed finite set, one
member recurs, and its distinct partners approach one another.
Take unit $a_j,b_j$ in each pair and $x_j=(a_j+b_j)/\sqrt2$.
Their $e(x_j)^2$ are half the disagreement probabilities, so again
$\sum_j e(x_j)<\varepsilon$.

Use independent compensated rate-one Poisson coordinates on these
$x_j$ and a standard independent cylindrical Gaussian on their
orthogonal complement. This positive germ has covariance $I$ and
higher cumulants $C_r=\sum_j x_j^{\otimes r}$. Its Poisson log moment
generating function $\sum_j(e^{\langle h,x_j\rangle}-1-
\langle h,x_j\rangle)$ converges absolutely and is bounded on each ball,
and all nontrivial two-group flattenings are contractions by
orthogonality. For a unit $x$,
$\|L_{P,r}x^{\otimes r}\|_2^2=1-p_x^r-(1-p_x)^r
\le r\min(p_x,1-p_x)$. Minkowski gives absolute convergence of
the defect series in $L^2(\xi;H^{\odot r})$ with norm at most
$\sqrt r\,\varepsilon$. Evaluation on tuples identifies its sum with
$L_{P,r}C_r$. Hence the mean cost at radius $t$ is at most
$\varepsilon^2\sum_{r\ge3}r t^{2r}/(r!)^{2\sigma}$.
The tensors $(I-D_r)x_j^{\otimes r}$ are orthogonal, with squared
norms one in the residual case and $1-2^{1-r}$ in the paired case.
Restriction to the first $N$ vectors or pairs gives unbounded squared
Hilbert norm. Formula~\eqref{r60:distance} excludes a joint repair
even allowing a change of covariance.
\end{proof}

For any prescribed bounded commuting covariance $B\ge mI$ with $m>0$
the same equivalence holds. In the counterexample scale the Poisson
coordinates by $\sqrt m$ and add a Gaussian of covariance $B-mP_X$,
where $P_X$ projects on their span. The cost bound uses $\sqrt m\,t$.

One need only test pure cubes to estimate the gap. If their profile
complement is nonzero, put $d_x=1-\sum_\chi\|E_\chi x\|^6$. Then
\begin{equation}\label{r72:purecube}
 g:=\inf_{\|x\|=1,\,d_x>0}\frac{3\mathbb E[p_x(1-p_x)]}{d_x}
   =\min\{3\eta_K,1-L\},\qquad (4-2\sqrt2)g/3\le1-c_3\le g .
\end{equation}
To see this expand $p_x$ as a mixture of the residual value and
the binary profile values, with weights $\alpha_0=\|P_Kx\|^2$ and
$\alpha_\chi=\|E_\chi x\|^2$. The residual diagonal term in
$\mathbb E[p_x(1-p_x)]$ is at least $\alpha_0^2\eta_K$ when present.
Each mixed pair contributes $\alpha_i\alpha_j$ times its mean
absolute mismatch, at least $\alpha_i\alpha_j(1-L)$.
Now $\alpha_0^2+3\sum_{i<j}\alpha_i\alpha_j-d_x
=\alpha_0^2(1-\alpha_0)+3\sum_{i<j<k}\alpha_i\alpha_j\alpha_k\ge0$.
For the reverse bound test residual unit vectors, and vectors
$\sqrt{1-h}\,e+\sqrt h\,y$ with unit $e\in H_\chi$, $y\in H_\chi^\perp$:
as $h\downarrow0$ the quotient tends to
$\mathbb E\|(I-P_\chi)y\|^2$. Without profiles $3\eta_K\le3/4$.
Testing the Hilbert spectral gap on $(I-D_3)x^{\otimes3}$ gives
$1-c_3\le g$. For the other inequality use~\eqref{r70:momentbounds}
and $[1-p^{3/2}-(1-p)^{3/2}]/[p(1-p)]=2(u+2)/(u+1)^2
\ge4-2\sqrt2$, where $u=\sqrt p+\sqrt{1-p}$.

\begin{corollary}[Testing repair on Hilbert-space-valued random variables]
\label{r73:realwitness}
Fix an integer $r\ge3$. The direction error bound
\eqref{r73:directiontest} is also equivalent to the following
statement about centered $H$-valued random variables. Whenever
$\mathbb E\|X\|^p<\infty$ for every $0<p<r$ and
$\sup_{\|h\|=1}\mathbb E|\langle X,h\rangle|^r<\infty$, finite
mean squared Hilbert defect of its $r$th cumulant $C_r$ implies
$(I-D_r)C_r\in H^{\odot r}$. If the bound fails, for each fixed $r$
one can choose such a law with arbitrarily small mean squared cut
defect and infinite joint defect in degree $r$.
\end{corollary}
\begin{proof}
The directional moment bound gives a bounded symmetric cumulant form
by polarization, so Theorem~\ref{r73:geocubic} proves sufficiency.
For necessity take the orthonormal near-binary vectors $x_j$ from
its proof. Let $N_j$ be independent Poisson variables of means
$\lambda_j=2^{-jr}$ and put
\[
 X=\sum_{j\ge1}2^j(N_j-\lambda_j)x_j .
\]
The total number of jumps $N$ is Poisson with finite mean
$\Lambda=\sum_j\lambda_j$, and the deterministic centering series
converges. This defines a Radon law on $H$. Given $N$, the uncentered
sum has $N$ independent marks with probabilities $\lambda_j/\Lambda$
at $2^jx_j$. A mark has finite norm moment of every order $p<r$.
Its directional $r$th absolute moment at $h$ is
$\Lambda^{-1}\sum_j|\langle h,x_j\rangle|^r\le\Lambda^{-1}\|h\|^r$.
Conditioning on $N$ gives the asserted bounds for $X$. The compound
Poisson cumulant formula on each line and polarization yield
$C_r=\sum_jx_j^{\otimes r}$, exactly the form in the preceding
counterexample. Hence its local and joint defects have the stated
sizes. At the strong endpoint $\mathbb E\|X\|^r<\infty$ the conclusion
holds for every cut law: every moment tensor through order $r$ is a
Bochner integral in its Hilbert tensor power, and their finite
moment--cumulant combination is again a Hilbert tensor.
\end{proof}

\begin{corollary}[Commuting cuts]\label{r73:commutingrepair}
If the support consists of commuting projections and $K\ne0$ then
$c_r=1$ for every $r\ge3$. If the joint profile spaces span $H$, then
$1-c_r=\Delta$ for all $r\ge3$.
\end{corollary}
\begin{proof}
In the second case the mixed sectors in~\eqref{r70:momentbounds}
give the formula. In the first, fix a unit $z\in K$ and take a
countable strong-dense sequence of support cuts, possible on the
bounded projections of a separable Hilbert space. Let $E$ range over
the cells of the first $n$ cuts. Put $m_E=\|Ez\|^2$ and
$v_E=Ez/\sqrt{m_E}$ for $m_E>0$. Each $v_E$ lies in $K$.
Commutation gives
\[
 \sum_E m_E p_{v_E}(P)(1-p_{v_E}(P))
     =\min_{(a_E)}\big\|(P-\sum_E a_E E)z\big\|^2.
\]
For every support cut $P$ the right side tends to zero by strong
density, since every cut from the first $n$ is a sum of cells.
Dominated convergence gives $\eta_K=0$. Since
$1-p^r-(1-p)^r\le2r p(1-p)$, testing pure powers of residual vectors
in~\eqref{r70:momentbounds} gives $c_r=1$ in every degree.
\end{proof}

For independent fair diagonal bits on $\ell^2(\mathbb N)$ the profiles
are the coordinate lines and $\Delta=1/2$. The mean squared repair
bound has sharp constant two, although $m$ cuts distinguish at most
$2^m$ patterns. Two coordinates with the same pattern give a mixed cubic
fixed by this finite family and moved by another cut in the support.

\subsection{Limits and fixed laws under random recombination}
For the following law convergence results, let $K$ be the closed span
of all finite-dimensional common reducing subspaces and $P_K$ its
projection. It contains the profile spaces and may also contain part
of $H_{\rm rem}$.

\begin{theorem}[Limit and energy loss under random recombination]\label{r64co:energy}
Let $\mu$ be any Borel probability law of finite second moment on a
separable real Hilbert space, of mean $m$ and covariance $C$. Let
$C_\infty$ be the Hilbert--Schmidt orthogonal projection of $C$ onto
the operators commuting with all cuts. Let $\pi$ have mean $m$,
independent original centered marginals on the $H_\chi$, and an
independent centered Gaussian on $H_{\rm rem}$ with covariance the
restriction of $C_\infty$. Every finite-dimensional linear observation
of $T^n\mu$ converges in law to the corresponding observation of $\pi$.
Moreover,
\begin{equation}\label{r64co:loss}
 \lim_{n\to\infty} W_2(T^n\mu,\pi)^2
     =\int\|(I-P_K)(x-m)\|^2\,d\mu(x).
\end{equation}
Thus the iterates converge in quadratic Wasserstein distance exactly
when the initial fluctuations are supported on $K$. The limit law
$\pi$ splits every cut, and is uniquely specified by the observations
above.
\end{theorem}
\begin{proof}
Translate to mean zero. Put $D_P B=PBP+(I-P)B(I-P)$. The covariances
are $C_n=L^n C$, where $L=\int D_P\,d\xi$. Each $D_P$ is an orthogonal
projection on the Hilbert--Schmidt space. Thus $L$ is a positive
self-adjoint contraction on the Hilbert--Schmidt space. Its fixed
operators are exactly the common commutant there, by integrating
$\|(I-D_P)B\|_2^2$ and using strong continuity on bounded sets in
the projection variable. The spectral theorem gives $C_n\to C_\infty$
in Hilbert--Schmidt norm. Positivity passes to this limit, and
$\operatorname{tr} C_\infty\le\operatorname{tr} C$ by finite-rank
testing. Its positive eigenspaces reduce every cut and are finite
dimensional, so its support is contained in $K$.

First work in finite dimension. There are finitely many nonzero profiles,
and on the orthogonal remainder, if nonzero, compactness gives
$\eta=\min_{\|v\|=1}\int\|Pv\|^2\|(I-P)v\|^2\,d\xi>0$:
vanishing would make $v$ a common profile vector. Start with an
independent cut at each vertex of a binary tree. To evaluate a
characteristic function at $h$, pass $Ph,(I-P)h$ to the two children
and repeat. Conditional on the tree, the characteristic function of
$T^l\rho$ is the product of the characteristic functions of $\rho$
at the leaf arguments. Their squared norms sum to $\|h\|^2$.
If $h$ is in the remainder and $v_i$ are its leaf arguments at depth
$l$, then
\[
 \mathbb E\sum_i\|v_i\|^4\le (1-2\eta)^l\|h\|^4.
\]
This follows by one-step conditioning and Pythagoras. Hence the largest
residual argument tends to zero in probability. Each profile component
of a general $h$ follows one path unaltered. The two binary values of distinct profiles differ on a relatively open
set of positive cut probability, by continuity on the support. Thus the profiles separate
with positive probability at every shared vertex, so eventually all their
paths are separate with probability tending to one.

Fix $j$ and use the centered input $\rho=T^j\mu$ at the leaves. Its
profile marginals equal the original ones. At the finitely many profile
leaves the residual arguments may be dropped with error tending to zero
in probability, by continuity. At the remaining leaves use the uniform
expansion $\log\widehat\rho(v)=-\langle C_jv,v\rangle/2+o(\|v\|^2)$
near zero. The sum of the error terms tends to zero in probability by
the squared-norm identity. Restoring the finitely many dropped quadratic
terms costs $o(1)$ in probability. Replacing $C_j$ by $C_\infty$ costs
at most $\|C_j-C_\infty\|\,\|h\|^2$ in the sum. Since $C_\infty$
commutes with the cuts, its quadratic terms sum exactly to
$\langle C_\infty h_{\rm rem},h_{\rm rem}\rangle$. Take $l\to\infty$
and then $j\to\infty$ in $T^{l+j}\mu$; boundedness of characteristic
functions gives the stated independent Gaussian and profile factors.
The traces of $C_j$ are constant, and the covariance on every profile
is unchanged. Thus second moments also converge, giving $W_2$
convergence in finite dimension.

In general choose increasing finite-dimensional common reducing
spaces $F_j$ with dense union in $K$, by separability and finite sums.
Projection onto any $F_j$ intertwines $T$ with recombination by the
restricted cuts. Its profile spaces are precisely $F_j\cap H_\chi$,
since the projections onto $F_j$ commute with all cuts and preserve
the common eigenspaces. Its limiting covariance is the compression
of $C_\infty$, by the convergence of $C_n$. The law $\pi$ in the
statement exists: the independent centered profile marginals have
summable second moments, and the Gaussian covariance is trace class.
Its finite projections therefore agree with the finite-dimensional
limits. Energy in $K\ominus F_j$ is constant along the iterates,
since this space reduces the projections, and tends to zero as
$j\to\infty$. The same tail bound holds for $\pi$ by taking finite
spectral observations in $K$. Coupling with each orthogonal truncation
therefore gives
$W_2((P_K)_*T^n\mu,\pi)\to0$. On $K^\perp$, every fixed linear
observation has second moment tending to zero by $C_n\to C_\infty$.
This proves the claim for finite-dimensional observations. Finally,
for any law $\lambda$ supported on $K$ and law $\rho$ of finite energy,
\[
 W_2(\rho,\lambda)^2=
    \int\|P_{K^\perp}x\|^2\,d\rho(x)+W_2((P_K)_*\rho,\lambda)^2.
\]
This follows from orthogonality and lifting a projected coupling by the
conditional law of $x$ given $P_Kx$. The first term for $T^n\mu$ is
conserved at every step. This proves \eqref{r64co:loss}; translation
restores $m$.
\end{proof}

\begin{corollary}[Fixed laws on profiles and on the remainder]\label{r73:radonfixed}
Let $\mu$ be any Borel fixed law on a separable real Hilbert space,
and put $R=(\bigoplus_\chi H_\chi)^\perp$. Its profile marginals
and its $R$-marginal are independent. Let $H_{\rm fin}$ be the closed
span of the finite-dimensional common reducing subspaces. Suppose
$H_{\rm fin}^\perp=0$ or the cutting constant there is positive.
Then the $R$-marginal is Gaussian with covariance commuting with the cuts,
and its $H_{\rm fin}^\perp$ component is deterministic. Thus the fixed
laws in this case are exactly the products of arbitrary profile
marginals and such a Gaussian which define Borel laws on $H$.
In particular, if $\eta>0$ on $H$, the fixed laws are Gaussian,
and a fixed law other than a point mass exists precisely when there
is a nonzero finite-dimensional common reducing subspace.
\end{corollary}
\begin{proof}
Fix $h=x+\sum_{\chi\in F}h_\chi$, with $x\in R$, $h_\chi\in H_\chi$
and $F$ finite. In the tree calculation just used, distinct profiles
eventually follow distinct leaves with probability tending to one.
On the path chosen by $\chi$, the residual argument is $W_nx$, where
$W_n$ is a product of independent copies of $P_\chi$ as defined in
Theorem~\ref{r70:momentgap}. For every $y\in H$,
\[
 \mathbb E|\langle W_nx,y\rangle|^2
 =\langle (\mathbb E P_\chi^{\otimes2})^n x^{\otimes2},y^{\otimes2}\rangle
 \longrightarrow0 .
\]
Indeed the powers converge strongly to the common fixed projection,
$E_\chi^{\otimes2}$: a fixed tensor lies in the range of $P_\chi$ in
each slot almost surely, so each of its one-slot contractions lies in
$H_\chi$. Integrating $|e^{i\langle W_nx,y\rangle}-1|$ against
$\mu(dy)$ now shows that deleting $W_nx$ from a tagged leaf changes
its Fourier factor by $o(1)$ in mean, without a moment assumption.
The same estimate restores the residual-only factors at those leaves.
The remaining complete residual product has expectation
$\widehat\mu_R(x)$ by fixedness. Telescoping the finitely many changes
and using density yields independence of the stated marginals.

Projection onto any common reducing space preserves fixedness. On
nonzero $H_{\rm fin}^\perp$ the radial estimate gives finite energy.
Theorem~\ref{r64co:energy} makes this marginal deterministic, since
that space has no nonzero finite-dimensional common reducing subspace.
On $R\cap H_{\rm fin}$ choose an increasing finite-dimensional common
reducing exhaustion: project the finite sums defining $H_{\rm fin}$
onto $R$ and use separability. No such finite space contains a profile,
so Corollary~\ref{r63co:law} gives a Gaussian there. Convergence on
linear observations makes the whole marginal Gaussian, with commuting
covariance by finite compression. This hypothesis includes both
uniform cutting on $R$ and a finite-dimensional reducing exhaustion
of $R$: all profile lines already lie in $H_{\rm fin}$. Conversely each product described
above splits every cut. If $\eta>0$ on $H$ there are no nonzero
profiles. A nonzero commuting trace-class covariance has a positive
eigenvalue with finite-dimensional reducing eigenspace; the projection
onto any nonzero finite-dimensional common reducing space is itself
such a covariance.
\end{proof}

For a fixed family with $\eta>0$ on $H$, its reducing subspaces give the exact fourth-moment ratio
\[
  \sup_{T\mu=\mu,\ \mathbb E\|X\|^2>0}
        \frac{\mathbb E\|X\|^4}{(\mathbb E\|X\|^2)^2}=1+\frac2k ,
  \qquad X\sim\mu ,
\]
where $k$ is the minimum dimension of a nonzero finite-dimensional
common reducing subspace, or $\infty$ if there is none. Indeed a
Gaussian $N(m,C)$ with $C\ne0$ has
$\|C\|\le\operatorname{tr}(C)/k$ by eigenvalue multiplicity.
Its squared norm has variance
$2\operatorname{tr}(C^2)+4\langle Cm,m\rangle
\le(2/k)(\operatorname{tr}C+\|m\|^2)^2$.
The centered standard Gaussian on a minimizing subspace attains
equality. When $k=\infty$ every nonzero point mass attains it.

\subsection{How the cut tree controls the full law}

For $s>0$ let
\[
 d_s(\mu,\nu)=\sup_{h\ne0}|\widehat\mu(h)-\widehat\nu(h)|/\|h\|^s
\]
as in Gabetta, Toscani and Wennberg \cite{r70:GTW}. Fourier comparison
by a positive transfer operator is classical in kinetic theory
\cite[Theorem~3.2]{r71:BCG}; Bassetti and Matthes give a weighted
one-step bound \cite[Lemma~15 in the preprint]{r64co:BM}.
For orthogonal cuts the same tree acts on characteristic functions and
moment tensors, and gives the optimal coefficient on probability laws.

\begin{theorem}[Optimal Fourier coefficients under random cuts]\label{r70:fourier}
On a nonzero separable real Hilbert space attach independent cuts with
law $\xi$ to the vertices of a depth-$n$ binary tree. Starting at
$h$, send $Pv,(I-P)v$ to the children of each vertex carrying $v$.
Denote the leaf vectors by $h_a$. For every real $s>0$, the best
constant in $d_s(T^n\mu,T^n\nu)\le C d_s(\mu,\nu)$ for finite-distance
pairs is
\[
            \beta_s(n)=\sup_{\|h\|=1}\mathbb E\sum_{|a|=n}\|h_a\|^s .
\]
At every depth $n$, $\beta_s(n)$ is still optimal among centered,
finitely supported laws in the unit ball whose moments agree through
any prescribed finite order. Write $\rho_s=\lim_n\beta_s(n)^{1/n}$
for the asymptotic coefficient.
At $s=4$, $\beta_4(1)=1-2\eta$ with $\eta$ as in
Theorem~\ref{r63co:mom}. At every $s>2$, $\beta_s(1)<1$ precisely when
$\eta>0$, and $\beta_s(n)\le\beta_s(1)^n$. In finite dimension and
even degree $s$, the asymptotic coefficient satisfies
$\lim_n\beta_s(n)^{1/n}=\|S_s\|$ on symmetric Hilbert tensors.
\end{theorem}
\begin{proof}
The conditional characteristic function is the product over the
leaves. Telescoping the difference of the two products gives the upper
bound and yields submultiplicativity by conditioning after any level.
Orthogonality gives $\sum_a\|h_a\|^2=\|h\|^2$, hence
$\sum_a\|h_a\|^s\le 2^{n(1-s/2)_+}\|h\|^s$ at all positive orders.
For $s>2$ and
$0<u<1$, $[1-u^{s/2}-(1-u)^{s/2}]/[u(1-u)]$ is positive and has
positive finite limits at both endpoints. At $s=4$ it is identically
two. Apply this with $u=\|Ph\|^2$ on the unit sphere, then iterate.

For sharpness fix a nonzero finite-rank orthogonal projection $F$ and
let $\rho$ be the uniform law on the unit sphere of $FH$. Its
characteristic function is $g(\|Fh\|)$, real, with $1-g(v)=O(v^2)$
at zero. Choose an integer $k$ with $2k>s$ and exceeding the desired
moment order. For $R>1$ average the signed convolution
$(\delta_0-\rho_{\sqrt t})^{*k}$ against $t^{-1-s/2}\,dt$ on $[1,R]$,
where $\rho_b$ is dilated by $b$. Taking even and odd binomial terms
separately gives positive measures of equal mass
$M_R=2^{k-1}\int_1^R t^{-1-s/2}\,dt$; normalize to probabilities
$\mu_R,\nu_R$. Both are symmetric with compact support in $FH$.
Their Fourier difference vanishes to order $2k$ at zero, so their
moments below $2k$ agree. With
\[
 c=\int_0^\infty (1-g(\sqrt t))^k t^{-1-s/2}\,dt\in(0,\infty),
 \qquad b_R=R^{-1/4},
\]
a change of variables bounds $d_s(\mu_R,\nu_R)\le c/M_R$ and gives
\[
 b_R^{-s}(\widehat\mu_R-\widehat\nu_R)(b_R v)
      \longrightarrow c\|Fv\|^s/M_\infty .
\]
Indeed the integration limits after rescaling are
$b_R^2\|Fv\|^2$ and $Rb_R^2\|Fv\|^2$. Both characteristic functions
at $b_Rv$ tend to one by domination in the original scale integral.
Telescope the leaf products at $b_Rh$. Dominated convergence, with
bound $(c/M_R)\sum_a\|h_a\|^s$ after rescaling, shows that the limiting
ratio of output to input distance is at least
$\mathbb E\sum_a\|Fh_a\|^s$ for each unit $h$. Let $F$ increase
strongly to the identity. Common dilation commutes with $T$ and
preserves ratios, so the trials may all be put in the unit ball.

For finite support, approximate each trial by finite laws on that
ball in $FH$ matching all its moments through successively increasing
orders (apply convex hull cubature to the monomials). These laws
converge weakly by polynomial approximation. They converge also in
$d_s$: at small frequencies matching moments to a fixed order $j>s$
bounds the Fourier difference by $C\|Fh\|^{j+1}$, on bounded annuli
convergence is uniform in the finite-dimensional frequency space, and
at infinity the numerator is bounded by two. The established upper
bound transfers this approximation to the outputs.

The finite-dimensional even-degree assertion follows from the next theorem.
\end{proof}

\begin{theorem}[Moment and Fourier contraction along the same cut tree]
\label{r71:treebridge}
On any nonzero separable real Hilbert space, for integers $r\ge2$ and
$n\ge1$ the symmetric-tensor operator of Theorem~\ref{r69:firstmoment}
satisfies, with $\beta_2(n)=1$ by conservation of squared length,
\[
                        \|S_r\|^n\le\beta_r(n).
\]
Let $d_n$ be the supremum over unit $h$ of the dimension of the span of the union
of the leaf vectors from $h$ over all depth-$n$ trees with cuts in
the support of $\xi$.
For even $r=2m$ and finite $d_n$,
\[
                 \beta_{2m}(n)\le d_n^{m/2}\|S_{2m}\|^n .
\]
In particular the two exponential rates agree when
$\limsup_n d_n^{1/n}=1$, including any law on two cuts in arbitrary
dimension. Equality of the exponential rates can fail even at degree
four for finitely many cuts in a separable Hilbert space.
\end{theorem}
\begin{proof}
For $f_u(x)=\langle u,x^{\otimes r}\rangle$, conditioning on successive
levels gives $f_{S_r^nu}(h)=\mathbb E\sum_a f_u(h_a)$.
On the unit sphere the supremum norm of this polynomial is bounded by
$\beta_r(n)\|f_u\|_\infty$. Applying the bound $j$ times to
$u=h^{\otimes r}$ gives
$\langle S_r^{nj}h^{\otimes r},h^{\otimes r}\rangle\le\beta_r(n)^j$.
Positivity, the spectral theorem and polarization yield the first
inequality. If the leaves from $h$ lie in a $d$-dimensional subspace
$V$ with orthonormal basis $(e_i)$, then
$\|P_Vx\|^{2m}=f_u(x)$ for the symmetric part of
$u=(\sum_{i=1}^d e_i\otimes e_i)^{\otimes m}$, of norm at most
$d^{m/2}$. The same tree identity gives the reverse bound with that
factor. For two cuts, replace each outcome by $(I\pm U)/2$ for the
corresponding involution $U$. The expansion uses only the identity
and alternating words of length at most $n$ in the two involutions,
so $d_n\le2n+1$.

For the strict inequality take the left regular representation on
$H=\ell^2(\ast_{i=1}^d\mathbb Z_2)$ and give equal weight to
$P_i=(I+U_i)/2$, where $U_i$ translates by generator $i$.
Write $U_i^{(J)}$ for action in the slots indexed by $J$. Then
\[
 S_r=2^{1-r}\left(I+
 \sum_{\substack{\varnothing\ne J\subseteq\{1,\ldots,r\}\\ |J|\ {\text{even}}}}
                  d^{-1}\sum_{i=1}^d U_i^{(J)}\right).
\]
On the full tensor power the diagonal action on the slots in a
nonempty $J$ has free orbits, so each inner average is a direct sum of
normalized adjacency operators of the $d$-regular tree, of norm
$2\sqrt{d-1}/d$ \cite[Proposition~3.2.31]{r71:Kowalski}.
In particular $\|S_4\|\le1/8+(7/8)2\sqrt{24}/25<1/2$ for $d=25$.
Yet $\|Pv\|^4+\|(I-P)v\|^4\ge\|v\|^4/2$ at every split,
so $\beta_4(n)\ge2^{-n}$ at every depth.
\end{proof}

\begin{corollary}[Fastest contraction by random cuts]\label{r78:bestcuts}
On $\mathbb R^d$ with $d\ge2$, let $q\ge2$ be an integer and $k=\lfloor d/2\rfloor$.
Write $(a)_q=a(a+1)\cdots(a+q-1)$. Among all laws of orthogonal
projections, at every depth $n\ge1$ the minimum of $\beta_{2q}(n)$ is
\[
 b_{d,q}^{\,n},\qquad
 b_{d,q}=\frac{(k/2)_q+((d-k)/2)_q}{(d/2)_q}.
\]
A cut law attains the minimum at one depth precisely when its ranks
belong to $\{k,d-k\}$ almost surely and
$\mathbb E(\|Pv\|^{2q}+\|(I-P)v\|^{2q})$ is constant for unit $v$.
It then minimizes at every depth and has asymptotic coefficient $b_{d,q}$.
In particular a uniformly distributed $k$-plane minimizes all these
even-order coefficients. At degree four this gives
$\max_\xi\eta=\lfloor d^2/4\rfloor/[d(d+2)]$.
\end{corollary}
\begin{proof}
Let $u$ be the symmetric tensor representing $\|x\|^{2q}$ and $S=S_{2q}$.
By rotational invariance, the spherical mean of $v^{\otimes2q}$ is a
multiple of $u$, namely $u/\|u\|^2$ by pairing with $u$. Thus the one-step spherical mean is $\langle Su,u\rangle/\|u\|^2$.
By Theorem~\ref{r71:treebridge},
\[
 \beta_{2q}(n)^{1/n}\ge\|S\|\ge\frac{\langle Su,u\rangle}{\|u\|^2}
 =\mathbb E\,\frac{(r/2)_q+((d-r)/2)_q}{(d/2)_q},\qquad r=\operatorname{rank}P.
\]
Indeed $\|Pv\|^2$ for uniform $v$ has the beta law with parameters
$r/2,(d-r)/2$, including the constant endpoints.
The polynomial $(x)_q$ is strictly convex for $x\ge0$, so averaging
over the ranks gives at least $b_{d,q}$, with equality precisely for
the two specified ranks. Equality in the Rayleigh bound requires
$Su=b_{d,q}u$, precisely the stated directional condition. Conversely
$S^nu=b_{d,q}^{n}u$ yields equality for all $n$. A uniform $k$-plane
satisfies the directional condition by rotation invariance.
For finitely supported cut laws this condition is the tight $q$-fusion
frame identity for the ranges and their complements
\cite[Section~4, (15)--(16)]{r78:BE}; the tensor power above connects it to
all depths of the cut tree.
\end{proof}

\begin{corollary}[A Fourier distance attaining the rate at every depth]
\label{r71:angularmetric}
In finite dimension, for each $s>0$ one may replace $\|h\|^s$ in
$d_s$ by an equivalent weight $f(h)$ so that the optimal coefficient
at depth $n$ is $\rho_s^n$, where $\rho_s=\lim_n\beta_s(n)^{1/n}$.
\end{corollary}
\begin{proof}
Write $\mathcal L f(h)=\mathbb E[f(Ph)+f((I-P)h)]$ on continuous even
$s$-homogeneous functions, identified with functions on the sphere.
On an irreducible reducing block, Schauder's theorem applied to
$\mu\mapsto\mathcal L^*\mu/\mu(\mathcal L1)$ on antipodally symmetric
probabilities gives an eigenmeasure $\nu$ with eigenvalue $a>0$;
the denominator is at least $\min(1,2^{1-s/2})$. Self-adjointness of the cuts
gives $\mathcal L f_\nu=a f_\nu$ for
$f_\nu(h)=\int|\langle h,u\rangle|^s\,d\nu(u)$. Its zero space is
invariant under almost every cut, by positivity and testing on a basis,
hence under the support. It is therefore zero. Comparison on the sphere
gives $a=\lim_n\beta_{s,\rm block}(n)^{1/n}$.
This invariant-weight argument is related to Protasov's extremal
$L^p$ norms \cite{r71:Protasov}. On an orthogonal decomposition sum
the block weights, extending each by its block projection. Then
$\mathcal L^n f\le a_*^n f$ with equality on a block of maximal
eigenvalue $a_*$. Since $f$ is comparable to $\|h\|^s$, iteration
gives $a_*=\rho_s$. Telescoping gives the upper bound in the weighted distance. For
sharpness in the proof of Theorem~\ref{r70:fourier} use the signed
convolution of the symmetric two-point law on $\{u,-u\}$ and then
average over the measure defining $f$. Rescaling gives the limit
$c f(h)/M_\infty$ and the input distance bound $c/M_R$ by the same
change of variable. Test on the maximal block.
\end{proof}

At a unit direction $v$, choose the branch $B\in\{P,I-P\}$ with
conditional probability $w=\|Bv\|^2$ and move to $Bv/\sqrt w$.
Let $H_n$ be the sum of $-\log w$ along the first $n$ steps.
Zero-weight branches are never selected.

\begin{theorem}[Branch entropy at the quadratic order]
\label{r71:entropy}\label{r72:twosides}
On any nonzero separable real Hilbert space let
$e_n^-=\inf_{\|v\|=1}\mathbb E_vH_n$ and
$e_n^+=\sup_{\|v\|=1}\mathbb E_vH_n$. The limits $e^\pm=\lim_n e_n^\pm/n$
exist. For $t>-1$,
\[
 \beta_{2+2t}(n)=\sup_v\mathbb E_v e^{-tH_n},\qquad
 (\log\rho_s)'_{2-}=-e^+/2,\qquad (\log\rho_s)'_{2+}=-e^-/2 .
\]
\end{theorem}
\begin{proof}
A leaf from a unit root has conditional probability $p_a=\|h_a\|^2$,
and along that path $H_n=-\log p_a$. The identity follows by summing.
The Markov property makes $e_n^-$ superadditive and $e_n^+$ subadditive;
binary entropy bounds both between zero and $n\log2$.
Thus $e^-=\sup_n e_n^-/n$, $e^+=\inf_n e_n^+/n$.
For $t>0$ set $e(t)=e^-$, for $t<0$ set $e(t)=e^+$, and likewise
write $e_n(t)$ for the corresponding extremum. Jensen gives
$\log\rho_{2+2t}\ge-te(t)$. At fixed $n$,
$\mathbb E_v e^{H_n/2}=\mathbb E\sum_a\sqrt{p_a}\le 2^{n/2}$
uniformly in $v$. Taylor expansion and submultiplicativity then give
$\log\rho_{2+2t}\le n^{-1}\log(1-te_n(t)+O_n(t^2))$ for small $|t|$.
Taking the respective one-sided limits and optimizing in $n$ proves
both derivatives.
\end{proof}

Kusuoka's matrix-product word measures \cite{r72:Kusuoka} arise also
as measurement records \cite{r72:Szczepanek}. Morris gives finite-product
estimates for general matrix pressure \cite{r72:Morris}. For cuts,
word reversal gives two-sided bounds for the optimal Fourier rate.

\begin{corollary}[Finite traces determine the Fourier rate]
\label{r72:tracerecord}
In dimension $d<\infty$ draw independent cuts $P_1,\ldots,P_n$ and for
$a\in\{0,1\}^n$ put $W_a=B_{n,a_n}\cdots B_{1,a_1}$, with
$B_{j,0}=P_j$, $B_{j,1}=I-P_j$. Let
$p_a=\operatorname{tr}(W_a^*W_a)/d$ and $Z_n(q)=\mathbb E\sum_a p_a^q$.
Then for every $q>0$,
\[
 d^{-(1-q)_+}Z_n(q)\le\rho_{2q}^n\le\beta_{2q}(n)\le d^q Z_n(q).
\]
Thus depth $n$ determines $\log\rho_{2q}$ to within
$\max(1,q)\log(d)/n$. At $q=1$ the three unscaled quantities equal one.
For each $d\ge3$ the upper factor $d^q$ is optimal at $q>1$.
On each irreducible reducing block $V$ of dimension $k$, let
$E_n(V)=-\mathbb E\sum_a p_{a,V}\log p_{a,V}$ use the normalized block
trace. There is a common stationary mean branch entropy $h_V$ with
\[
 0\le E_n(V)-nh_V\le\log k,\qquad e^-=\min_V h_V,\quad e^+=\max_V h_V .
\]
In particular the rate is differentiable at order two if and only if
all block entropies agree.
\end{corollary}
\begin{proof}
Linearity of expectation allows the same cut at each vertex of a level:
each fixed word encounters independent cuts with the original law.
Write $p_{v,a}=\|W_a v\|^2$ for unit $v$. Its basis average equals
$p_a$ and $p_{v,a}\le d p_a$. For $q\ge1$ Jensen gives
$Z_n(q)\le\beta_{2q}(n)$; for $q<1$ use
$p_a^q\le d^{-q}\sum_i p_{e_i,a}^q$. In either case $Z_n(q)$ and
$\beta_{2q}(n)$ have the same exponential rate.

Under the probability with density $p_a^q/Z_n(q)$ against the cut law
and word counting measure, $X=WW^*/\operatorname{tr}(WW^*)$ and
$Y=W^*W/\operatorname{tr}(W^*W)$ have the same distribution: reversal
sends $W$ to $W^*$ and preserves the density. Concatenating two
independent such words gives probability $dp_a p_b\operatorname{tr}(X_aY_b)$.
Hence $Z_{2n}/Z_n^2=d^q\mathbb E z^q$, where
$z=\operatorname{tr}(X_aY_b)\in[0,1]$ and
$\mathbb E z=\operatorname{tr}[(\mathbb E X)^2]\ge1/d$.
Jensen for $q\ge1$ and $z^q\ge z$ for $q<1$ give
$Z_{2n}\ge d^{-(1-q)_+} Z_n^2$. Doubling and taking roots proves
the lower bound; submultiplicativity gives the middle inequality.
For sharpness of $d^q$ at $q>1$, let $P$ fix a common line and project
onto a uniformly chosen line on its orthogonal complement. The common
line has $\beta_{2q}(n)=1$. If $d\ge3$, strict splitting on the other
block makes its coefficient decay geometrically. The full trace word
law is the mixture of the deterministic all-zero word with weight
$1/d$ and that contracting record. Minkowski's inequality in $L^q$
of words and cut sequences gives $Z_n(q)\to d^{-q}$.

On an irreducible block the direction transition is Feller, since
a vanishing branch has vanishing weight. Any stationary law $\pi$ has
barycenter $C=\int vv^*\,d\pi(v)$ fixed by the mean of the
Hilbert--Schmidt projections $C\mapsto PCP+(I-P)C(I-P)$. The mean
squared projection defect is therefore zero. Thus $C$ commutes with
the support, and its spectral subspaces give $C=I_V/k$. Consequently
$p_{a,V}=\int p_{v,a}\,d\pi(v)$. By the entropy chain rule,
$E_n(V)-n\int h\,d\pi$ is the expected relative entropy of
$(p_{v,a})_a$ against $(p_{a,V})_a$, where $h(v)$ is the continuous
mean binary entropy of one step. This difference lies in $[0,\log k]$
by $p_{v,a}\le k p_{a,V}$, the dimension bound of Holevo
\cite{r72:Holevo}. Hence every stationary law has the same mean $h_V$.
Compactness gives uniform convergence $\mathbb E_vH_n/n\to h_V$ on
the block: every weak limit of mean occupation measures over growing
times is stationary by telescoping. Across blocks $(p_{v,a})_a$ is
a mixture with weights $\|v_V\|^2$. Its entropy exceeds the weighted
mean by at most the entropy of those weights, uniformly bounded by
the logarithm of the block count. Taking the two extrema gives the claim.
\end{proof}

\begin{corollary}[Growth below order two]\label{r72:subcritical}
In finite dimension and for $0<s<2$, one has $\rho_s>1$ if and only if
two cuts in the support fail to commute.
\end{corollary}
\begin{proof}
Commuting cuts have a common eigenbasis, so at most $d$ leaf probabilities
are positive for the levelwise cut words and $\beta_s(n)\le d^{1-s/2}$.
If cuts fail to commute, some irreducible block has dimension at least two.
On it the continuous mean binary entropy has a positive minimum:
zero would give a line invariant under the support. Now
$\log\rho_s\ge (1-s/2)e^+>0$ by Theorem~\ref{r72:twosides}.
\end{proof}

In infinite dimension commuting cuts can attain the maximal growth
allowed by squared-length conservation. Use the independent diagonal
Bernoulli cuts following Theorem~\ref{r72:repairiff}.
At depth $n$, for the unit vector constant on the first $M$ coordinates
and zero elsewhere, $\|W_av\|^2$ is the fraction of those coordinates
with word $a$. For fixed $n$ the law of large numbers sends all these
fractions to $2^{-n}$. Dominated convergence and the upper bound of
Theorem~\ref{r70:fourier} give
$\beta_s(n)=2^{n(1-s/2)}$ for $0<s<2$. Each coordinate line is
preserved, so $\beta_s(n)=1$ for $s\ge2$.

\begin{corollary}[Stability of the Gaussian fixed law without input moments]
\label{r70:fourierfixed}
In finite dimension suppose $\eta>0$ and $s>2$. For any Borel law
$\mu$, $d_s(\mu,T\mu)<\infty$ if and only if there is a Gaussian law
$G$ splitting the cuts almost surely with $d_s(\mu,G)<\infty$.
That Gaussian is unique, $\mu$ has its mean and covariance, and,
writing $b=\beta_s(1)$,
\[
 (1-b)d_s(\mu,G)\le d_s(\mu,T\mu)\le (1+b)d_s(\mu,G).
\]
On a separable Hilbert space the same equivalence holds with the
additional condition $\int\|x\|^2\,d\mu(x)<\infty$ on the defect side.
\end{corollary}
\begin{proof}
If the defect is finite, successive iterates have $d_s$ distance at
most $b^n d_s(\mu,T\mu)$. Their characteristic functions converge
uniformly on bounded frequency sets, so the continuity theorem gives
a weak limit $G$. The cut transformation is weakly continuous by
bounded convergence of characteristic functions; hence $TG=G$.
Theorem~\ref{r63co:law} makes $G$ Gaussian since $\eta>0$ excludes
every nonzero independent subspace. Summing the geometric series gives
the left inequality. Conversely $TG=G$ gives the right inequality
by the triangle inequality. Two fixed laws at finite distance have
zero distance by contraction. Finally along each coordinate direction
the characteristic functions differ by $O(|t|^s)=o(t^2)$.
Fatou's lemma applied to $(1-\cos(tX))/t^2$ gives finite second moments;
differentiation in every direction then gives identical means and
covariances. In Hilbert space with finite second moment,
Theorem~\ref{r64co:energy} already supplies convergence of the
characteristic functions to a Radon Gaussian; the same geometric
estimate applies. In the converse direction, equality of directional
means and variances and Parseval give finite total second moment
since the Gaussian covariance is trace class.
\end{proof}

The Hilbert moment condition is necessary even at $s=4$. On
$H=\bigoplus_{j\ge1}\mathbb R^2$ use the same two planar cuts
$0,\pi/4$ on all blocks with equal probabilities. Here $\eta=1/8$:
concavity of $u(1-u)$ bounds the value on a unit vector below by
the average of its one-block values weighted by squared block lengths,
and each one-block value is $1/8$. Select block $j$ with probability
$p_j\propto j^{-3/2}$, and choose $X$ uniformly among its four signed
coordinate vectors of length $j^{1/4}$. Its directional covariance
is scalar on every block, with eigenvalues proportional to $j^{-1}$.
Its fourth directional moments are uniformly bounded on the unit sphere
since $\sup_j p_j j<\infty$ and $\sum_j\|h_j\|^4\le\|h\|^4$.
The same bound up to a constant holds after one cut. Symmetry and
the unchanged covariance then give $d_4(\mu,T\mu)<\infty$ by the
cosine Taylor remainder. Its directional covariance has infinite trace,
so no Radon Gaussian is at finite $d_4$ distance.

\subsection{Distances to the limit and conservation of energy}

\begin{theorem}[Distances to fixed targets and conserved quadratic forms]
\label{r67co:targets}
For an input $\mu\in\mathcal P_2(H)$ let $\pi_\mu$ be its limit law in
Theorem~\ref{r64co:energy} and $D_\mu$ the right side of
\eqref{r64co:loss}. For every $\nu\in\mathcal P_2(H)$,
\[
 \lim_{n\to\infty}W_2(T^n\mu,\nu)^2=D_\mu+W_2(\pi_\mu,\nu)^2.
\]
The convergence is uniform on $W_2$-compact sets of targets. Thus
$\pi_\mu$ is the unique minimizer of the limiting distance on all of
$\mathcal P_2(H)$.

Let $L$ act on self-adjoint trace-class operators by the covariance
recursion of Theorem~\ref{r64co:energy}. Its trace-dual action on
bounded self-adjoint operators is $L^*A=\mathbb E D_PA$.
Write $d_\mu(A)=\operatorname{tr}(A\operatorname{Cov}\mu)$. This
quantity is conserved for every input exactly when $L^*A=A$. These
observables determine the limiting trace distance between covariances.
For $Z=\operatorname{Cov}\mu-\operatorname{Cov}\nu$,
\begin{equation}\label{r68co:covariance}
 \lim_{n\to\infty}\|L^n Z\|_1
 =\sup_{A=A^*,\,L^*A=A,\,\|A\|\le1}
          |d_\mu(A)-d_\nu(A)|.
\end{equation}
For every finite family of positive fixed observables with
$\sum_j A_j\le I-P_K$,
\begin{equation}\label{r67co:stability}
 W_2(\pi_\mu,\pi_\nu)^2+
 \sum_j\bigl(\sqrt{d_\mu(A_j)}-\sqrt{d_\nu(A_j)}\bigr)^2
 \le W_2(\mu,\nu)^2.
\end{equation}
In particular $\mu\mapsto(\pi_\mu,\sqrt{D_\mu})$ is nonexpansive for
the quadratic product metric. The map $\mu\mapsto\pi_\mu$ is a
nonexpansive retraction onto the finite-second-moment fixed laws.
\end{theorem}
\begin{proof}
The bounded second moments and finite-dimensional convergence give
narrow convergence $T^n\mu\to\pi_\mu$ on the weak Hilbert space
by compactness of closed balls. For a fixed target $\nu$, every
sequence of optimal couplings of $(\nu,T^n\mu)$ has a subsequence
converging in the strong--weak transport topology to an optimal coupling
of $(\nu,\pi_\mu)$ \cite[Proposition~3.4 and Theorem~3.8]{r67co:NS}.
The scalar product is an admissible test: it is sequentially continuous
on $H_{\rm strong}\times H_{\rm weak}$ and satisfies
$|\langle x,y\rangle|\le\|x\|^2/(4\varepsilon)+\varepsilon\|y\|^2$.
Thus the optimal cross
moments converge. Expanding the squared distance now leaves the
second-moment loss $D_\mu$ from Theorem~\ref{r64co:energy}.
Distances to the target are $1$-Lipschitz in that target, giving
uniform convergence on compact sets.

Trace duality proves the assertion about conservation, with necessity
obtained from the centered two-point inputs on $\{h,-h\}$. Write
$L=(I+S)/2$, where $SZ=\mathbb E(2P-I)Z(2P-I)$ is a trace-norm
contraction. Subtracting adjacent binomial weights gives
\[
 \|L^n(I-L)\|_{1\to1}\le 2^{-n}\max_{0\le k\le n}\binom nk
     \longrightarrow0.
\]
Consequently $L^nY\to0$ in trace norm for every $Y$ in
$\overline{(I-L)\mathcal S_1^{\rm sa}}$. Contractivity
bounds $\lim_n\|L^nZ\|_1$ by the distance from $Z$ to this closed
subspace. By Hahn--Banach the distance is the supremum in
\eqref{r68co:covariance}, and invariance bounds the supremum by
$\|L^n Z\|_1$ at every $n$. For complex operator systems Arveson
identified the fixed observables of slowly oscillating completely
positive maps in the general theory of asymptotic lifts
\cite[Theorem~6.1 and Proposition~8.1]{r68co:Arv}.

Using two independent copies of a
coupling with the same cut in both coordinates yields
$W_2(T\mu,T\nu)\le W_2(\mu,\nu)$ by orthogonality. Write $\rho^0$ for
the translate of $\rho$ by minus its mean $m_\rho$. For every coupling
$(X,Y)$ of $\mu,\nu$, put $V=X-m_\mu$, $W=Y-m_\nu$. The inequality
$P_K+\sum_j A_j\le I$ and the reverse triangle inequality in $L^2$
give
\[
 \mathbb E\|X-Y\|^2\ge \|m_\mu-m_\nu\|^2+
 W_2((P_K)_*\mu^0,(P_K)_*\nu^0)^2+
 \sum_j\bigl(\sqrt{d_\mu(A_j)}-\sqrt{d_\nu(A_j)}\bigr)^2.
\]
On $K$ recombination is nonexpansive, and both projected centered
laws converge under iteration in $W_2$ to $\pi_\mu^0,\pi_\nu^0$.
Thus the projected distance dominates $W_2(\pi_\mu^0,\pi_\nu^0)$.
Adding back the mean difference and taking the infimum proves
\eqref{r67co:stability}. Every $\pi_\mu$ splits the cuts and is fixed.
For a fixed input the constant orbit has the same finite-dimensional
observations as its limit, so $\pi_\mu=\mu$.
\end{proof}

The fixed observables can distinguish two orbits even when every
operator commuting with the cuts has the same covariance reading.
Take the three equiprobable cuts $(I+U_s)/2$ on the left regular
representation of $\mathbb Z_2*\mathbb Z_2*\mathbb Z_2$, with $U_s$
the generating involutions. Every commuting operator has constant
diagonal, so no nonzero finite-rank projection commutes and $K=0$.
Delete the edge between the identity $e$ and a generator $a$.
On the two resulting components define, respectively,
\[
 h(v)=2^{-\operatorname{dist}(v,e)}/3,\qquad
 h(v)=1-2^{-\operatorname{dist}(v,a)}/3 .
\]
At every vertex $h$ is the average of its three neighboring values,
including at $e,a$, where its values are $1/3,2/3$. Consequently
the multiplication operator $M_h$ satisfies $L^*M_h=M_h$.
The centered random signs times $\delta_e,\delta_a$ have the same
limit law and agree on all commuting observables. Applying
\eqref{r67co:stability} to $M_h,I-M_h$ at every depth gives squared
Wasserstein distance at least $2-4\sqrt2/3$ between their iterates.

\begin{theorem}[Compactness of recombination iterates]\label{r65co:tight}
In the setting of Theorem~\ref{r64co:energy}, let
$D=\int\|(I-P_K)(x-m)\|^2\,d\mu(x)$. The squared norms of the centered
iterates are uniformly integrable. On the Hilbert space with its weak topology,
$T^n\mu$ converges narrowly to $\pi$. On the Hilbert space with its norm
topology the following are equivalent: $D=0$, convergence to $\pi$ in $W_2$,
and tightness of a subsequence with indices tending to infinity.
If $D>0$, there is a $\delta>0$ such that every norm-compact $S$ satisfies
\[
 \liminf_{n\to\infty}T^n\mu\{x:\operatorname{dist}(x,S)>\sqrt D/4\}
 \ge\delta .
\]
\end{theorem}
\begin{proof}
Center the laws. Write $X_n\sim T^n\mu$, $C_n=L^n C$ and
$M=\operatorname{tr}C_n=\operatorname{tr}C$. First suppose
$\mathbb E\|X_0\|^4<\infty$ and define the positive trace-class operators
\[
 B_0=\mathbb E[\|X_0\|^2 X_0\otimes X_0]+3MC,\qquad B_n=L^n B_0 .
\]
We claim $\operatorname{Var}\langle AX_n,X_n\rangle\le
\operatorname{tr}(B_n A^2)$ for every bounded self-adjoint $A$. At $n=0$
this follows by Cauchy--Schwarz. For the induction, write $Q=I-P$ and
$J=PAQ+QAP$. Conditional on $P$, recombination gives two independent
diagonal quadratic forms with coefficients $PAP,QAQ$, and the cross term
$2\langle PX_n^{(1)},AQX_n^{(2)}\rangle$. Centering and independence
make the cross term uncorrelated with both diagonal forms. Its variance
is $4t$ with $t=\operatorname{tr}(C_nPAQC_nQAP)$. Since $0\le C_n\le MI$,
cyclicity gives both $t\le M\operatorname{tr}(C_nPAQAP)$ and
$t\le M\operatorname{tr}(C_nQAPAQ)$, whence $4t\le
2M\operatorname{tr}(C_nJ^2)$. The conditional mean is
$\operatorname{tr}(C_n(A-J))$. Its variance is at most
$\mathbb E_P(\operatorname{tr}(C_nJ))^2\le M\mathbb E_P\operatorname{tr}(C_nJ^2)$.
The induction bound on the diagonal forms and $B_n\ge3MC_n$ now give
\[
 \operatorname{Var}\langle AX_{n+1},X_{n+1}\rangle
 \le\mathbb E_P\operatorname{tr}\bigl(B_n[(PAP)^2+(QAQ)^2+J^2]\bigr)
 =\operatorname{tr}(B_{n+1}A^2),
\]
using $(PAP)^2+(QAQ)^2+J^2=PA^2P+QA^2Q$. Pinching preserves trace,
so $A=I$ gives $\sup_n\mathbb E\|X_n\|^4\le\mathbb E\|X_0\|^4+4M^2$.

For arbitrary finite $M$, truncate and recenter the input to obtain
bounded centered $Y_{0,R}$ with $\mathbb E\|X_0-Y_{0,R}\|^2=\varepsilon_R^2\to0$.
Couple both evolutions using the same cut and two independent copies
of the preceding coupling. Orthogonality preserves the mean squared
coupling cost at each step. With $H_R=\sup_n\mathbb E\|Y_{n,R}\|^4<\infty$,
Cauchy--Schwarz and Markov's inequality give
\[
 \sup_n\mathbb E[\|X_n\|^2;\ \|X_n\|>a]
 \le 2\varepsilon_R^2+2H_R^{1/2}M^{1/2}/a .
\]
Taking $a\to\infty$ and then $R\to\infty$ proves uniform integrability.

For weak-topology convergence take finite-rank orthogonal projections
$F_j\uparrow I$. On every bounded ball, which is compact metrizable in the
weak topology, $F_jx\to x$ uniformly in a compatible weak metric
(use the coordinates of an orthonormal basis adapted to $F_j$).
Thus every bounded weakly continuous function is uniformly approximated
there by its compositions with $F_j$. The uniform second-moment bound
controls the mass off the ball, and Theorem~\ref{r64co:energy} gives the
required convergence for each $F_j$.

If $D>0$, choose $b>0$ so large that
$\sup_n\mathbb E[\|X_n\|^2;\ \|X_n\|^2>b]<D/4$. For each finite-rank
orthogonal $F$, convergence of the covariances in Hilbert--Schmidt norm gives
$\lim_n\mathbb E\|(I-F)X_n\|^2=M-\operatorname{tr}(FC_\infty)\ge D$,
since $\operatorname{tr}C_\infty=M-D$ by Theorem~\ref{r64co:energy}.
Splitting this expectation at $\|X_n\|^2>b$ and at
$\|(I-F)X_n\|^2>D/4$ gives
$\liminf_n\mathbb P(\|(I-F)X_n\|>\sqrt D/2)\ge D/(2b)$.
For any norm-compact $S$ choose $F$ with
$\sup_{x\in S}\|(I-F)x\|<\sqrt D/4$. This proves the escape bound,
including after undoing the translation. If $D=0$, Theorem~\ref{r64co:energy}
gives $W_2$ convergence, hence tightness.
\end{proof}

\begin{lemma}[Point observations of recombination]\label{r66co:fibreobs}
Suppose that the cuts on the separable direct integral
$H=\int_S^\oplus H_s\,dm(s)$ act as a jointly measurable field of
orthogonal projections $P_s$, where $m$ is a standard probability
measure. For each $k\ge1$ and $m^k$-almost every
$(s_1,\ldots,s_k)$, the joint point-value laws of the iterates
$T^n\mu$ are the recombination iterates of the initial point-value
law, with cut $P_{s_1}\oplus\cdots\oplus P_{s_k}$.
\end{lemma}
\begin{proof}
Choose measurable representatives $e_j(s)$ of a countable dense
subset of $H$. Given $x\in H$, at step $l$ choose the first $e_j$
within $2^{-l}$ of $x$. The successive differences have summable
$L^1(m)$ norms, and the approximants converge to $x$ in $H$. Their
pointwise limit, set to zero where convergence fails, is a jointly
measurable representative of $x$. For independent inputs and an
independent cut, the pointwise recombination of the inputs agrees
$m$-almost everywhere with this representative of the output. Fubini
gives equality of the corresponding $k$-point laws for almost every
tuple. Intersecting the full-measure sets over the steps proves the
claim at all $n$.
\end{proof}

Commuting cuts have a scalar-fibre representation by the spectral
theorem on cyclic summands. A pair of projections has fibres of
dimension at most two. Its four common range or kernel intersections
give scalar cuts. On the generic part Halmos's two-subspaces theorem
\cite{g56:Halmos} gives the matrices
\[
 \begin{pmatrix}I&0\\0&0\end{pmatrix},\qquad
 \begin{pmatrix}C^2&CD\\CD&D^2\end{pmatrix},
 \qquad C,D\ge0,\quad D=(I-C^2)^{1/2}.
\]
The spectral representation of $C$ on cyclic summands gives the
two-dimensional fibres.

\begin{theorem}[Concentration of escaping energy on finite-dimensional fibres]
\label{r65co:concentration}
In Theorem~\ref{r64co:energy} suppose the cuts act as a jointly
measurable field on $H=\int_S^\oplus H_s\,dm(s)$, with $m$ a standard
probability measure and $\dim H_s<\infty$ almost everywhere. Let
$X_n\sim T^n\mu$, $Y_n=P_{K^\perp}(X_n-m_\mu)$, where $m_\mu$ is the
initial mean, and write $e(s)=\mathbb E\|Y_0(s)\|^2$. Then for every
bounded measurable $\psi:S\to\mathbb R$,
\[
 \int_S\psi(s)\|Y_n(s)\|^2\,dm(s)
      \longrightarrow \int_S\psi(s)e(s)\,dm(s)\quad\text{in }L^1 .
\]
In particular the escaped squared norm concentrates at the energy
loss $D$ of Theorem~\ref{r65co:tight}. Jointly, the retained centered
vector and the full centered squared norm converge in law to
$(Z,\|Z\|^2+D)$, where $Z$ has the centered limit law of
Theorem~\ref{r64co:energy}. The result applies to every commuting
family and every pair of orthogonal projections, by the spectral
representations following Lemma~\ref{r66co:fibreobs}.
\end{theorem}
\begin{proof}
The law of $Y_n$ is $T^n$ applied to the centered residual law. For
almost every $s$ its point value has finite second moment $e(s)$,
conserved at every step. By Lemma~\ref{r66co:fibreobs}, the one-point and two-point
laws are themselves recombination iterates with the corresponding
cuts. All these assertions hold simultaneously for countably many $n$.
Theorem~\ref{r64co:energy} gives $W_2$ limits for almost every such
finite-dimensional observation.

We show that the two limiting point values at $(s,t)$ are independent
for $m^2$-almost every pair. First, their cross covariance vanishes.
Indeed $C_n^{\perp}=\operatorname{Cov}(Y_n)\to0$ in Hilbert--Schmidt
norm by Theorem~\ref{r64co:energy}. Its integral kernel is
$\mathbb E[Y_n(s)\otimes Y_n(t)]$, whose squared Hilbert--Schmidt
norm integrated over $S^2$ equals $\|C_n^{\perp}\|_2^2$. The kernels
already have pointwise limits almost everywhere by the two-point
$W_2$ convergence, so these limits are zero.

It remains to check for a shared non-Gaussian profile. Fix an
independent sample sequence of cuts $P_j$ dense in the global
strong support and in the corresponding fibre supports for almost
every $s$ and almost every pair $(s,t)$. Such a sequence exists by
separability and Fubini, applied to the sampled single cuts and
paired block cuts. For $\chi\in\{0,1\}^{\mathbb N}$ let $E_\chi(s)$
project onto their common profile space in $H_s$. This field is
measurable in $(s,\chi)$: intersect the first $j$ eigenspaces by taking
the kernel projection of $\sum_{i\le j}(P_i(s)-\chi_i I)^2$, then
pass to the decreasing limit. Let $\mathcal C$ be the countable set
of labels with nonzero global profile. For a fixed $\chi\notin\mathcal C$,
$E_\chi(t)=0$ almost everywhere, since its direct integral is the
global profile projection. At fixed $s$ only finitely many labels
have $E_\chi(s)\ne0$. Hence almost every pair shares only labels in
$\mathcal C$, by Fubini. Here the exceptional relation is universally
measurable, being the projection of the Borel condition
$E_\chi(s),E_\chi(t)\ne0$, $\chi\notin\mathcal C$. For every label in
the countable set $\mathcal C$, its point component of $Y_0$ is zero
almost everywhere, as the corresponding global space is in $K$.

In the limit of the two-point iteration the nonzero random profile
factors consequently belong to one fibre or the other. These factors
are independent of each other and of the joint Gaussian remainder.
The zero cross covariance proved above makes the two Gaussian
remainders independent as well. This proves the claim about point values.

Fix $b>0$ and truncate the pointwise squared norms at $b$. Their
two-point covariances tend to zero almost everywhere by the claim
and bounded continuous testing. Fubini and dominated convergence give
vanishing variance for their integral weighted by $\psi$. In each
finite-dimensional fibre $W_2$ convergence gives uniform integrability
of the squared norms, so
$\sup_n\mathbb E[(\|Y_n(s)\|^2-b)_+]\to0$ almost everywhere as
$b\to\infty$. This quantity is bounded by $e(s)\in L^1(m)$.
The weighted truncation errors thus tend uniformly to zero in $L^1$,
while the expectations before truncation are exactly $\int\psi e\,dm$.
This proves the limit. The joint assertion follows from orthogonality
and the $W_2$ convergence of the retained projection.
\end{proof}

For two cuts the escaping space can be read directly from Halmos's
angle operator $C$ following Lemma~\ref{r66co:fibreobs}. In the generic
two-copy representation $(P-Q)^2$ acts as $I-C^2$ on both copies.
A finite-dimensional common reducing space therefore lies in its
pure point spectral part. Conversely each eigenvector of $C$ gives a
two-dimensional common reducing plane, and the four common profile
spaces belong to $K$. Thus $K^\perp$ consists exactly of the two copies
of the continuous spectral part of $C$. The theorem gives concentration
of the energy on every Borel set of those angles.

\begin{example}[Four cuts with fluctuating escaped energy]\label{r65co:shell}
There are four equiprobable cuts, a centered input of deterministic
norm one, and $K=0$, for which no subsequence of the iterated squared
norms converges in probability to a constant. In fact, for $n\ge2$,
\[
       \inf_{c\in\mathbb R}\mathbb E\big|\|X_n\|^2-c\big|\ge 1/192 .
\]
\end{example}
\begin{proof}
Let $\Gamma=\mathbb Z_2*\mathbb Z_2*\mathbb Z_2$, with generators
$a,b,c$, and take $H=\ell^2(\Gamma;\mathbb R)$. For the left
translations $U_s$ use the three projections $P_s=(I+U_s)/2$.
The fourth projection $E$ multiplies by the indicator that the reduced
word begins with $a$, excluding the empty word $e$. Start with the
equiprobable vectors $\pm\delta_e$. If a finite-rank projection $F$
commutes with all the $P_s$, it commutes with all left translations.
Its diagonal is therefore constant on the infinite group. Finite trace
forces this constant to be zero and positivity gives $F=0$. Thus $K=0$.

Write $\mathcal D_P A=PAP+(I-P)A(I-P)$. The covariance $C_j$ remains
diagonal: $\mathcal D_{P_s}A=(A+U_s A U_s)/2$, while $\mathcal D_E$
fixes diagonal operators. The resulting diagonal probabilities are
invariant under permutation of the three letters. In particular
$h_j=\operatorname{tr}(C_j E)=(1-C_j(e,e))/3\le1/3$.
Theorem~\ref{r64co:energy} gives $C_j\to0$ in Hilbert--Schmidt norm
and hence $h_j\to1/3$. Left multiplication by $a$ toggles membership
in $E$, including at $e$, so $\mathcal D_{P_a}E=I/2$.

At depth $j+2$ condition on root cut $E$, first child cut $P_a$ and
second child cut $E$. This event has probability $1/64$. Conditional
on these three choices the child inputs have mean zero, and the
squared norm has mean
$\operatorname{tr} C_j(\mathcal D_{P_a}E+\mathcal D_E(I-E))
=3/2-h_j$. Exchanging the two child cuts gives a disjoint event of
probability $1/64$ with conditional mean $1/2+h_j$. Conditional
Jensen on these events gives the asserted lower bound
$(1-2h_j)/64\ge1/192$, independently of the proposed constant.
By the uniform integrability in Theorem~\ref{r65co:tight}, convergence
in probability to a constant along a subsequence would give convergence
to it in $L^1$, a contradiction.
\end{proof}

\subsection{Measuring the distance by successive marginal products}
Allow an arbitrary family $(P_j)_{j\in J}$ of real orthogonal projections.
Write $D_j$ for the coefficientwise operation
$D_{j,r}=P_j^{\otimes r}+(I-P_j)^{\otimes r}$. Given a finite word
$w=(j_1,\ldots,j_m)$ in the cuts, start at $k^0=k$ and put
$k^l=D_{j_l}k^{l-1}$. Each step takes the product of the two marginals
on the polynomial core. With the weighted Hilbert norm defined above, set
\[
             S_t(w)=\sum_{l=1}^m\|k^{l-1}-k^l\|_t^2 .
\]
The empty word has cost zero. A non-Hilbert coefficient has infinite cost.

\begin{theorem}[Exact distance for arbitrary cuts]\label{r60:words}
Let $k$ be the centered cumulants of a positive continuous state, and
let $\mathcal I$ be the centered cumulants of positive continuous states
on the same algebra splitting every cut. At every $t>0$,
\[
       \inf_{k'\in\mathcal I}\|k-k'\|_t^2=\sup_w S_t(w).
\]
A finite value gives a unique closest centered germ. It is a limit, in
the norm of cumulant differences at this radius, of finite iterations
of products of independent marginals. A joint state on the same algebra
exists exactly when the supremum is finite at some radius.
\end{theorem}
\begin{proof}
Let $\mathscr H_t$ be the Hilbert space of differences at radius $t$.
Each $D_j$ is an orthogonal projection there. A joint germ $k'$ is
fixed by every $D_j$, so Pythagoras along the word gives
$S_t(w)=\|k-k'\|_t^2-\|k^m-k'\|_t^2$ whenever $k-k'\in\mathscr H_t$.

Suppose $L=\sup_w S_t(w)<\infty$. For a word write
$W_l=D_{j_l}\cdots D_{j_1}$, $W_0=I$, $W=W_m$, and
$T_w x=((I-D_{j_l})W_{l-1}x)_{l=1}^m$. Telescoping gives
$T_w^*T_w=I-W^*W$ and $\|T_w\|\le1$. This finite identity acts
also on forms by pullback. Thus $T_w k$ is a Hilbert vector and
\[
         \|(I-W^*W)k\|_t\le \|T_wk\|=S_t(w)^{1/2}.
\]
For a finite set $F$ of cuts let $V$ perform one pass through $F$ and
put $A=V^*V$, another word. Along repetitions of $A$ the sum of all
squared step differences is at most $L$. Hence, for each $j\in F$,
$\|(I-D_j)A^{2m}k\|_t\to0$: in the next pass through $F$ the
step at $j$ and the finitely many preceding differences all tend to zero.
By the previous display with $W=A^m$, the vectors
$h_{F,m}=(I-A^{2m})k$ have norm at most $\sqrt L$.
Direct the pairs $(F,\epsilon)$ by increasing $F$ and decreasing
$\epsilon>0$, choosing $m$ so that all the preceding defects for $F$
are at most $\epsilon$. A weak cluster point $h$ of this bounded net
makes $k-h$ fixed by every cut: for a fixed $j$, $(I-D_j)k$ is full
by the one-letter word and $(I-D_j)h_{F,m}$ converges to it in norm.

The iterated marginal products have positive polynomial moments.
Evaluation on a fixed tensor is weakly continuous in $\mathscr H_t$,
so the moment--cumulant formula gives positivity of $k-h$ on squares.
Since $h$ is full, this gives a continuous state on the original algebra
by Theorem~\ref{s:thm:scalar-maps}. The first paragraph now gives
$\|h\|_t^2=L$. Every $h_{F,m}=T_{A^m}^*T_{A^m}k$ is orthogonal to
the common Hilbert fixed space, as is $h$. Differences of joint germs
at finite distance lie in this fixed space, so Pythagoras gives
uniqueness. Along the weakly convergent subnet the upper bound
$\sqrt L=\|h\|_t$ yields norm convergence. Finally every continuous
state in the given algebra has full cumulant difference at some radius
by~\eqref{new:eq:state-cumulant}.
\end{proof}

For homogeneous projection products Kwapie\'n and Mycielski obtained
strong convergence for a strictly positive trace-class mean defect
\cite[Proposition~1]{r71:KM}. Tian treats strongly measurable random
operators under a mean coercivity condition \cite{r70:Tian}.
Epperly, Goldshlager and Webber prove convergence of tail-averaged
randomized Kaczmarz iterates for inconsistent least-squares systems
\cite{r70:TARK}. Here the spectral measure of the initial defect gives
the exact distance for an arbitrary cut law. The trajectory theorem uses
independent samples. Corollary~\ref{r71:empiricalrepair} recovers the
distance from the initial defects of a stationary ergodic sequence.

\begin{theorem}[Recovering the repair distance from a random sequence of cuts]
\label{r70:repair}
Let the cuts in Theorem~\ref{r60:words} be the support of a probability
$\xi$ on the orthogonal projections of a separable real Hilbert space.
At radius $t$ write $\mathcal H=\mathscr H_t$, $D=D_P$ and $L=I-D$.
Assume $b(P)=Lk$ is a square-integrable $\mathcal H$-valued field.
On $\mathcal B=L^2(\xi;\mathcal H)$ put $\partial u=(L_Pu)_P$,
$A=\partial^*\partial$, $g=\partial^*b$,
$h=P_{\ker\partial^*}b$ and $e=\|h\|^2$. Let $\nu$ be the spectral
measure of $A$ at $g$. The positive measure
\[
 \theta=e\delta_1+(\lambda\mapsto1-\lambda)_*
                    (\lambda^{-1}\,d\nu(\lambda)|_{(0,1]})
\]
determines both the least squared local residual $e$ and the exact
squared distance to a joint state,
$\int_{[0,1]}(1-q)^{-1}\,d\theta(q)$, allowing infinity.

For independent cuts $P_i$ of law $\xi$, start at $x_0=0$, put
$x_i=D_{P_i}x_{i-1}-b(P_i)$, and write
$d_i=x_{i-1}-x_i$, $J_n=\sum_{i=1}^n\|d_i\|^2$.
Then $a_j=\mathbb E\langle b(P_{j+1}),d_{j+1}\rangle
=\int q^j\,d\theta(q)$. In particular, a single sequence of scalar
expectations determines the entire measure. Almost surely $J_\infty$
is the exact squared distance. The trajectory and expectation both satisfy
$J_n/n\to2e$ almost surely and $\mathbb E J_n/n\to2e$. Thus infinite
cost grows sublinearly for $e=0$ and linearly for $e>0$.
For $n\ge1$ the regularized least-squares cost satisfies
\[
 K(n)=\inf_u\{\|u\|^2+n\|\partial u+b\|^2\}
     =\int\frac{n}{1+n(1-q)}\,d\theta(q),\qquad
 \tfrac12 K(n)\le\mathbb E J_n\le3K(n).
\]
\end{theorem}
\begin{proof}
Write $\partial=UA^{1/2}$ in polar form and $b=Uv+h$, with
$v\perp\ker A$. Then $g=A^{1/2}v$, and the spectral measure of $v$
on $(0,1]$ is $\lambda^{-1}d\nu$. The squared least residual is $e$.
An exact repair solves $\partial u=-b$, equivalently $h=0$ and
$A^{1/2}u=-v$. Its minimum squared norm is the stated inverse
integral. The fixed-cut equations holding almost surely extend to
the support on testing each bounded coefficient form against a tuple
of vectors, by strong continuity in $P$. Positivity and continuity of
the repaired state $k+u_*$ follow from the random word limit below:
each marginal product is positive and the Hilbert correction gives a
continuous change of chart as in Theorem~\ref{r60:words}. Diagonal minimization of the quadratic
functional gives the formula for $K(n)$.

Write $M=I-A$. The mean iterate is
$m_j=-[(1-M^j)/A]g$. Hence $a_j=\|b\|^2+\langle g,m_j\rangle$
has the displayed moment formula; polynomials are dense on $[0,1]$.
The identity $\|x_i\|^2=\|x_{i-1}\|^2+
2\langle b(P_i),d_i\rangle-\|d_i\|^2$ gives
$\mathbb E J_n=2\sum_{j<n}a_j-\mathbb E\|x_n\|^2$.
Jensen's inequality bounds the cost below by
$\sum_{j<n}\|\partial m_j+b\|^2$ and bounds $\mathbb E\|x_n\|^2$
below by $\|m_n\|^2$. With $q=1-\lambda$ this yields
\[
 ne+\int\frac{1-q^{2n}}{\lambda^2(2-\lambda)}\,d\nu
 \ \le\mathbb E J_n\le 2ne+\int\frac{1-q^{2n}}{\lambda^2}\,d\nu .
\]
Here $1-q^{2n}\ge n\lambda/(1+n\lambda)$ and
$1-q^{2n}\le\min(1,2n\lambda)\le3n\lambda/(1+n\lambda)$,
which proves the comparison.

We record why the infinite-distance assertion also holds on a single
trajectory. For a homogeneous product $W_j=D_{P_j}\cdots D_{P_1}$ and
$u\perp\ker A$, $W_ju\to0$ almost surely and in mean square.
Indeed $Y_j=W_ju$ has innovation $Y_{j+1}-MY_j$ of conditional mean
zero and squared mean at most $\mathbb E\|Y_{j+1}-Y_j\|^2$.
These latter quantities sum to at most $\|u\|^2$ by Pythagoras.
Iterating the mean recursion, martingale orthogonality and
$M^k\to0$ strongly on $(\ker A)^\perp$ give
$\mathbb E\|Y_j\|^2\to0$ by summability; the norms decrease pathwise.
By separability this holds simultaneously for every $u\perp\ker A$
and after every shift of the cut sequence.

Both $b$ and its projection onto $\overline{\operatorname{ran}\partial}$
lie in the closed subspace of fields $c(P)\in\operatorname{ran}L_P$
almost surely. Hence so does $h$. We have $\mathbb E h(P)=\partial^*h=0$.
An independent copy $V$ of $h(P)$ lies in $(\ker A)^\perp$ almost
surely, by testing a countable dense subset of $\ker A$.
On a two-sided
independent sequence, put $H_i=h(P_i)$, $D_i=D_{P_i}$ and $L_i=I-D_i$.
For source $h$ define
\[
 R_i^{(m)}=H_i-\sum_{k=1}^m L_iD_{i-1}\cdots D_{i-k+1}H_{i-k}.
\]
The empty product is the identity. Revealing cuts backwards gives a
martingale in $m$ since $\mathbb E H_i=0$. Its squared mean is
$2e-\mathbb E\|W_mV\|^2$: orthogonality of the differences and the
projection identity telescope the loss of squared norm of $W_kV$.
Thus it converges a.s. and in $L^2$ to a stationary ergodic $R_i$
with $\mathbb E\|R_i\|^2=2e$. For the zero start $d_i(h)=R_i^{(i-1)}$.
For every $M$, after the first $M$ terms its squared difference from
$R_i$ is bounded by $\sup_{m\ge M}\|R_i^{(m)}-R_i\|^2$. The ergodic
theorem bounds the upper limit of the averages by the mean of this supremum,
which tends to zero by Doob's $L^2$ inequality and dominated convergence.
Hence $J_n(h)/n\to2e$ almost surely.

For every source $c(P)\in\operatorname{ran}L_P$, Pythagoras in the recursion gives
$J_n(c)\le4\sum_{i\le n}\|c(P_i)\|^2$, while $J_n(\partial z)\le\|z\|^2$.
Take a sequence $\partial z_m\to b-h$. The triangle inequality for the
residual sequences and the strong law give, simultaneously for these $m$,
$\limsup_n\sqrt{J_n(b-h)/n}\le2\|b-h-\partial z_m\|$.
Letting $m\to\infty$ proves $J_n(b)/n\to2e$ almost surely.

The bound by four times the averages of $\|b(P_i)\|^2$ gives uniform
integrability (truncate this integrable variable). The slope therefore
also holds in mean.

If the minimum repair $u_*$ exists, Pythagoras relative to $u_*$
gives $J_n=\|u_*\|^2-\|W_nu_*\|^2$. Otherwise suppose that finite
trajectory cost has positive probability. Changing the initial
point by $u$ changes the sequence of residuals by
$T_\omega u=(L_{P_i}W_{i-1}u)_i$, an isometry into
$\ell^2(\mathcal H)$ on $(\ker A)^\perp$ by the homogeneous limit.
Finiteness is therefore independent of the start and is shift invariant,
so has probability one. The full residual sequence then has a unique
minimizing start $u_\omega$ on $(\ker A)^\perp$, by the closed range of
this isometry. It is measurable by taking coordinates of
$-T_\omega^*(d_i)_i$. The finite minimum is at least its shifted
counterpart plus the first squared step. Since the minima have the
same distribution they agree almost surely. The first step thus
vanishes and uniqueness makes $u_\omega$ shift invariant. Ergodicity
makes it deterministic, giving $\partial u_\omega=-b$, a contradiction.
\end{proof}

\begin{corollary}[Recovering the repair measure from one ergodic sequence]\label{r71:empiricalrepair}
In Theorem~\ref{r70:repair} observe any stationary ergodic sequence of
cuts $(P_i)$ with marginal $\xi$. Use its first $N$ cuts as an empirical
law with measure $\theta_N$. Almost surely $\theta_N\Rightarrow\theta$
for each fixed square-integrable input field. For the minimum correction
norms $d_N,d$ of the observed prefix and the full support, respectively,
\[
 \int_{[0,1]}\frac{d\theta_N(q)}{1-q}=d_N^2\uparrow d^2
 =\int_{[0,1]}\frac{d\theta(q)}{1-q},\qquad
 d_N^2=\sup_{z_i\in\operatorname{ran} L_{P_i}}
 \frac{|\sum_{i\le N}\langle b(P_i),z_i\rangle|^2}{\|\sum_{i\le N}z_i\|^2}.
\]
Use $0/0=0$ and a positive numerator over zero as infinity.
When $d<\infty$, the minimizing corrections satisfy
$\|u_N-u_*\|^2=d^2-d_N^2$. The static limit $d_N\uparrow d$ holds on
one event for all centered source germs and all radii, taking $d_N=\infty$
when a sampled defect is not a Hilbert vector.
\end{corollary}
\begin{proof}
Write $A_N,g_N,B_N$ for prefix averages of $L_P,b(P),\|b(P)\|^2$.
Ergodicity gives $g_N\to g$ in norm, $B_N\to\|b\|^2$, and $A_N\to A$
strongly, by testing on a countable dense subset and using $\|A_N\|\le1$.
The vector ergodic assertion follows by simple-function approximation and
the scalar ergodic theorem on the norm error. The mean recursion gives
$\int q^j\,d\theta_N=B_N-\sum_{k<j}\langle g_N,(I-A_N)^kg_N\rangle$.
All moments converge on $[0,1]$, hence the measures converge weakly.

The sample is dense in the support on one event by ergodicity on
a countable base for the strong topology. For any germ and radius the
prefix distances increase. If bounded, the prefix minimizers have a
weak cluster point solving every sampled equation. Testing the bounded
coefficient forms on tuples of vectors extends these equations to the
support by strong continuity. Theorem~\ref{r60:words} gives the same
minimum over positive joint states, proving the limit even for $d=\infty$.
The inverse moments follow from Theorem~\ref{r70:repair}. A bounded
quotient in the display defines a bounded functional on the sums of the
$z_i$. Its negative Riesz representative on their closed span is precisely
the minimum-norm solution of the prefix equations. Whenever $u_*$ exists
this is its orthogonal projection onto that span, giving the error identity.
\end{proof}
The supremum can use $z_i=L_{P_i}v_i$ with $v_i$ finite linear combinations of elementary tensors, by density. Thus the data are cuts and scalar tests of their initial defects. For a single cut, the
stationary residual after projection is zero for every original defect.

\needspace{5\baselineskip}
\begin{corollary}[Commuting cuts]\label{reader59:cuts}
For countably many commuting projections let $D_{n,r}=\sum_E E^{\otimes r}$
over the atoms of their first $n$ cuts. Then
\begin{equation}\label{r59:min}
 \inf_{k'\in\mathcal I}\|k-k'\|_t^2
 =\sup_n\sum_{r\ge2}\frac{t^{2r}}{(r!)^{2\sigma}}
                  \|(I-D_{n,r})k_r\|_2^2 .
\end{equation}
At finite distance the products of marginals on these partitions
converge in cumulant differences to the unique minimizing centered germ.
\end{corollary}
\begin{proof}
The projections $D_j$ of Theorem~\ref{r60:words} commute. A word
retains precisely the pure blocks of the partition generated by its
distinct letters, and its squared step lengths sum to the total squared
discarded norm. A prefix containing these letters discards at least
as much. If a step is not a Hilbert tensor, neither can the total
discarded part be one, since projecting the latter gives that step.
Taking suprema proves the identity. At finite distance the discarded
parts are increasing orthogonal projections of the difference from
the minimizer. Their norms converge to its norm, giving norm convergence.
\end{proof}

\begin{corollary}[Covariance distance and conditional expectation]
\label{r60:expectation}
Let $B\ge0$ be bounded, $M=\{P_j:j\in J\}'$, and set $B^0=B$,
$B^l=P_{j_l}B^{l-1}P_{j_l}+(I-P_{j_l})B^{l-1}(I-P_{j_l})$ along
each word. Then
\begin{equation}\label{r60:gaussdistance}
  \inf_{0\le C\in M}\|B-C\|_2^2
     =\sup_w\sum_{l=1}^{|w|}\|[P_{j_l},B^{l-1}]\|_2^2 .
\end{equation}
Each summand equals one eighth of the limiting fourth-cumulant
difference in~\eqref{bridge58:commutator}, using finite-dimensional
compressions reducing this single cut when $H$ is separable.

On $M+\mathcal S_2$ the closest-operator map is
$\mathcal E(C+K)=C+\Pi K$, where $C\in M$ and $\Pi$ is the Hilbert
projection onto $M\cap\mathcal S_2$. It extends uniquely in operator
norm to a unital completely positive conditional expectation
$M+\mathcal K(H)\longrightarrow M$, with $\mathcal K(H)$ the compact
operators.
\end{corollary}
\begin{proof}
Pinching preserves Gaussianity, and the squared Hilbert--Schmidt norm
of its change equals that of the commutator by the two off-diagonal
blocks. Apply Theorem~\ref{r60:words} to the Gaussian germ and cancel
its degree-two weight. Its finite-distance limit is Gaussian. The fourth
cumulant assertion follows by compression in a basis adapted to $P_{j_l}$.

On all Hilbert--Schmidt operators, complexified if necessary, let
$A=V^*V$ denote the palindromic pass through a finite set $F$ of
pinchings used in the theorem. It is a positive self-adjoint
contraction. Equality of norms at every projection in $V$ identifies
$\ker(I-A)$ with the common fixed space of $F$. Thus $A^m$ converges
strongly on $\mathcal S_2$ to its orthogonal projection by the spectral
theorem. As $F$ increases these projections decrease strongly to
$\Pi$. Every pass fixes $M$ and is a unital completely positive
$M$-bimodule map on bounded operators. On $C+K$, $K\in\mathcal S_2$,
the iterated limits converge in operator norm to $C+\Pi K$, also
entrywise on every matrix amplification. This proves complete positivity
and operator-norm contractivity there. Density gives the stated extension,
whose range is $M$ since it fixes $M$ and $M$ is closed.
\end{proof}

\subsection{Reading the Gaussian distance from fourth cumulants}
For a centered Gaussian state the only coefficient to change is its
covariance $B$. Put $\mathcal M=\{P_j:j\ge1\}'$. A covariance in
$\mathcal M$ makes all the required Gaussian marginals independent.
The following identity measures the exact cost using only finite
Gaussian quadratic forms. Assume $H$ separable. For each $n$ choose
finite-rank projections $F_{n,k}\uparrow I$ reducing $\mathcal E_n$;
choose bases in the finitely many ranges to construct them. On
$F_{n,k}H$ write $B_k$ and $E_k$ for the compressions of $B$ and $E$.
For symmetric matrices $C$ on that space let
$Y_C=g^\top Cg-\operatorname{tr}C$, with $g$ a standard normal vector;
write $\kappa(C_1,\ldots,C_4)$ for the joint cumulant of these forms.

\begin{corollary}[Exact Gaussian distance]\label{r59:gauss}
Let $B$ be a bounded positive covariance. Then
\begin{equation}\label{r59:exactfour}
 \inf_{0\le C\in\mathcal M}\|B-C\|_2^2
 =\frac1{16}\sup_n\lim_{k\to\infty}\sum_{E\in\mathcal E_n}
 \big\{\kappa(I,E_k,B_k,B_k)-\kappa(E_k,E_k,B_k,B_k)\big\}.
\end{equation}
Either side may be infinite. When finite, the unique closest covariance
is the operator-norm limit of $\sum_{E\in\mathcal E_n} EBE$.
\end{corollary}
\begin{proof}
Set $B_n=\sum_E EBE$. Theorem~\ref{reader59:cuts} in degree two gives
$\sup_n\|B-B_n\|_2^2$ for the distance. Indeed the minimizing germ has
only its second coefficient, with positive covariance
$\lim_n B_n$ since convergence of its Hilbert--Schmidt corrections
implies operator-norm convergence. When this supremum is infinite,
projection of any Hilbert--Schmidt difference $B-C$, $C\in\mathcal M$,
onto the off-diagonal blocks would bound it.

For finite matrices let $P=E_k$, $T=B_k$. The Gaussian log transform
has homogeneous degree-four part $2\operatorname{tr}(\sum z_i C_i)^4$.
Extracting the square-free coefficient and using $P^2=P$ gives
$48\operatorname{tr}(PT^2)$ and
$16[2\operatorname{tr}(PT^2)+\operatorname{tr}(PTPT)]$, respectively.
Their difference is $8\|[P,T]\|_2^2$. Each off-diagonal block of $B$
enters exactly two squared commutators in $\sum_E\|[E,B]\|_2^2$,
giving $2\|B-B_n\|_2^2$, also with value infinity. For fixed $n$,
$[E_k,B_k]$ is the compression of $[E,B]$. Its squared norm increases
to the full squared norm by summing matrix entries in an adapted basis.
This proves \eqref{r59:exactfour} for every choice of exhaustions.
\end{proof}

For jointly Gaussian pairs $(X_j,Y_j)$, independent across $j$, with
standard marginals and correlation $\rho_j$ within each pair, take the single cut separating all $X_j$
from all $Y_j$. The nearest independent covariance has distance squared
$2\sum_j\rho_j^2$. Thus independence is attainable on the same algebra
precisely for $\rho\in\ell^2$. A correlation $j^{-1/2}$ and a
correlation $j^{-1}$ both tend to zero, but only the latter allows the
change of state.

\section{A four-degree test for Gaussian--Poisson coordinates}\label{a:sec:fourth-order}

Let $H=\ell^2(I;\mathbb R)$ for a finite or countable set, and let
$a_i\ge0$ be bounded. Take independent coordinates $\xi_i$ standard
Gaussian if $a_i=0$, and $\xi_i=a_i(P_i-a_i^{-2})$ for
$P_i\sim{\rm Poisson}(a_i^{-2})$ otherwise. The sum
$X(h)=\sum h_i\xi_i$ converges in $L^2$. Its compensated logarithmic
moment generating function has Gaussian term $\|h_G\|^2/2$ and
Poisson summands $(e^{a_i h_i}-1-a_i h_i)/a_i^2$, bounded in
absolute value by $e^{\|a\|_\infty\|h\|}|h_i|^2/2$. This gives
locally uniform convergence on finite-dimensional complex parameter
spaces; the bound at $2h$ for real $h$ gives uniform integrability
of the exponentials of partial sums. The covariance is the identity
and the $r$th cumulant for $r\ge3$ is
$\sum_i a_i^{r-2}\prod_s h_{si}$. Every nontrivial split has norm at
most $\|a\|_\infty^{r-2}$ by restriction to pure coordinate tensors.
Thus the Appell product is continuous and expectation on polynomials
is its positive vacuum. Use the bounded multiplication
$m_a(e_i\otimes e_j)=\delta_{ij}a_ie_i$ on the Hilbert tensor
product, and likewise $m_b$ on a second space $K$. It satisfies
$\langle m_a(x,y),z\rangle=\langle x,m_a(y,z)\rangle$.

\begin{theorem}[Finite-order lifting]\label{a:thm:extension}
For a bounded real $T:H\to K$ set
\[
 Q=T^*T-I,\qquad M=m_b(T\otimes T)-Tm_a,\qquad A=T^*M.
\]
At every common exponent $\sigma\ge1$, the substitution with linear
part $T$ and any bounded real shift extends continuously if and only
if its cumulant defects in orders $2,3,4$ are full. This is also
equivalent to boundedness of the substitution through polynomial
degree four. In operator form the conditions are
$Q,A\in\mathcal S_2$ and $M\in\mathcal S_4$. If $T$ has closed
range of finite codimension, the tests in orders $2,3$ suffice.
\end{theorem}
\begin{proof}
The cubic is $\langle m_a(x,y),z\rangle$ and the quartic is
$\langle m_a(x,y),m_a(z,w)\rangle$. Subtracting source from
pulled-back target gives the successive flattenings
\[
 Q,\qquad Qm_a+A,\qquad
 m_a^*Qm_a+m_a^*A+A^*m_a+M^*M .
\]
Fullness means that each flattening is Hilbert--Schmidt. In this
order the tests reduce to the three displayed ideal conditions,
since $\|M^*M\|_2=\|M\|_4^2$. The coefficient formula for the
prescribed substitution shows that boundedness through degree four
is equivalent to fullness of the first four vacuum rows. Taking the
formal logarithm through this degree gives the same test on the
cumulant defects, the first row being bounded by hypothesis.
Necessity for the completed substitution follows as well.

For sufficiency let $F_n^a$ be $n$-fold multiplication, $F_1^a=I$,
and define $F_n^b$ likewise. Put
$L=\max(1,\|a\|_\infty,\|b\|_\infty\|T\|)$,
$B_n=F_n^bT^{\otimes n}-TF_n^a$ and $C_n=T^*B_n$.
Then $\|F_n^a\|\le L^{n-1}$ and
$B_{n+1}=m_b(B_n\otimes T)+M(F_n^a\otimes I)$.
For $U=m_b(B\otimes T)$, the diagonal coordinate formula gives
$(UU^*)_{ij}=b_i b_j(BB^*)_{ij}(TT^*)_{ij}$. Squaring and summing
yields $\|U\|_4\le\|b\|_\infty\|T\|\|B\|_4$. Induction from $B_1=0$
therefore yields $\|B_n\|_4\le(n-1)L^{n-2}\|M\|_4$.

Pair $B_{n+1}(x,y)$ with $Tw$ for $x\in H^{\otimes n}$. The first
term of the recurrence becomes
$\langle B_nx,Tm_a(y,w)+M(y,w)\rangle$ by the displayed symmetry of the multiplication.
The second is $\langle A(F_n^ax,y),w\rangle$.
The associated forms have full norms bounded respectively by
$\|m_a\|\|C_n\|_2$, $\|B_n\|_4\|M\|_4$ and
$\|A\|_2 L^{n-1}$. Indeed the middle flattening is $B_n^*M$,
and $\|B_n^*M\|_2^2\le\|B_n B_n^*\|_2\|MM^*\|_2$ by
Hilbert--Schmidt Cauchy--Schwarz. Starting at $C_1=0$ gives
\[
 \|C_n\|_2\le(n-1)L^{n-2}\|A\|_2+
                \binom{n-1}{2}L^{n-3}\|M\|_4^2 .
\]
The degree-$r$ cumulant defect has last-slot flattening
$QF_{r-1}^a+C_{r-1}$, so its full norms have a common exponential
bound. The prescribed substitution theorem now gives continuity.
Finally, if $P$ projects onto a closed cofinite range of $T$, then
$T^*|_{\ran T}$ has bounded inverse onto $(\ker T)^\perp$, and hence
$PM=(T^*|_{\ran T})^{-1}A$ is Hilbert--Schmidt. The remaining
component of $M$ is finite rank.
\end{proof}

\begin{example}[A sharp quartic obstruction]\label{a:ex:eighth}
For Gaussian input, give the target independent coordinate directions
$g_j,p_j^+,p_j^-$, the first Gaussian and the latter two unit Poisson.
Put $Te_j=\sqrt{1-t_j^2}\,g_j+
t_j(p_j^+-p_j^-)/\sqrt2$, $0\le t_j\le1$. The supports are disjoint
and each image has norm one, so $Q=0$. One has
$M(e_i\otimes e_j)=\delta_{ij}t_j^2(p_j^++p_j^-)/2$, orthogonal
to $\ran T$, so $A=0$. Its nonzero singular values are
$t_j^2/\sqrt2$. Thus the cubic agrees exactly, while continuity
requires and suffices that $\sum_j t_j^8<\infty$.
\end{example}

\section{Which observations recover the cumulant class?}\label{r55:observations}
Put $\mathcal O_\sigma(H)=\mathcal C_\sigma(H)/\mathcal H_\sigma(H)$
for the centered cumulant classes of the preceding construction.
Write $\|\,\cdot\,\|_2$ for the full Hilbert tensor norm, with value
$+\infty$ on a form without a full Hilbert representative.
A family of observations recovers the class if equality of all its
observed classes forces equality of the original classes.

Write $H=\bigoplus_{k\in I}H_k$ as a finite or countable orthogonal
sum of infinite-dimensional separable real Hilbert spaces. For
$S\subset I$ let $P_S$ project onto the indicated subspaces. Channel
$j$ retains these coordinates for $S=S_j$ and multiplies the tensor
of degree $r\ge2$ by $m_{j,r}$, with $|m_{j,r}|\le A_jL_j^r$.
Denote this observation by $D_j$. A set $T\subset I$ of size at most
$r$ specifies the different blocks present among the $r$ arguments
of a tensor coefficient.

\begin{theorem}[Complete cumulant observations]
\label{obs52:complete}\label{obs53:joint}
A nonempty countable family of these observations recovers the
intrinsic class in $\mathcal O_\sigma(H)$ exactly when there exist
a finite set $J$ and $b>0$ such that for every $r\ge2$ and every
nonempty $T\subset I$, $|T|\le r$,
\begin{equation}\label{obs53:cover}
         \max_{j\in J:\,T\subset S_j}|m_{j,r}|\ge b^r .
\end{equation}
An empty maximum is zero. Necessity already holds on centered
infinitely divisible cylindrical laws with split cumulant norms and
jump magnitudes at most one. With one block and all $S_j=I$, the
condition is simply $\max_{j\in J}|m_{j,r}|\ge b^r$ for all $r$.
For such degree multipliers preserving positive germs, the conditions
for a specified bounded real map to lift between the observed algebras
for all $j\in J$ are equivalent to the original lift condition.
\end{theorem}
\begin{proof}
For a tensor difference $e_r$, square its full Hilbert norm in a
block-adapted orthonormal basis. Each squared entry is counted by at
least one retained projection whose multiplier satisfies
\eqref{obs53:cover}. Thus
$b^{2r}\|e_r\|_2^2\le\sum_{j\in J}\|(D_j e)_r\|_2^2$, including
infinite values. Finitely many full defects have a common factorial
bound. The converse for an already full defect follows by boundedness.
In the single-block case apply this to $e=c_A-T^\#c_B$, using
Theorem~\ref{new:thm:scalar-invariants} and commutation of multipliers
with the pullback by the prescribed map $T$.

For necessity repeat a finite list periodically. If the condition
fails, for each $n$ choose $r_n\ge2$ and a nonempty $T_n$ of at
most $r_n$ blocks such that every $j\le n$ containing $T_n$ has
$|m_{j,r_n}|<e^{-2nr_n}$. No increase of $r_n$ is needed.
Within each block of $T_n$ choose one unit vector. Make a list of
$r=r_n$ vectors $v_1,\ldots,v_r$ using all the chosen vectors,
with repetition allowed. Form the signed measure
\[
  \eta=2^{-r}\mathop{*}_{l=1}^{r}
            (\delta_{v_l/r}-\delta_0) .
\]
Its total variation is at most one and its support lies in the unit
ball. Its Laplace transform is a product of $r$ first differences:
all moments below $r$ vanish, and its degree-$r$ moment tensor has
strictly positive full norm $d_n$ since its homogeneous polynomial
is a nonzero product of linear forms. All higher moment tensors have
full norm at most one. Projection missing any block in $T_n$
annihilates the signed measure, in every degree.

Push the two Jordan parts forward under multiplication by
$0<\epsilon_n\le\min(1,2^{-n}d_n/(2B_n))$,
where $B_n=e^{nr_n}(r_n!)^\sigma$. Use them as jump intensities
of two centered compound Poisson vectors on this finite span,
discarding atoms at zero. For a finite jump measure $\lambda$,
conditioning on a Poisson number of independent jumps gives the
centered log moment series
$\int(e^{\langle h,x\rangle}-1-\langle h,x\rangle)\,d\lambda(x)$.
Its coefficients in degrees at least two are the jump moments. Take
$\lceil(B_n/(d_n\epsilon_n^{r_n}))^2\rceil$ independent copies on
mutually orthogonal fresh spans in the same blocks. Do this for all
$n$, always using fresh coordinates. The pool is available since
the summands are infinite dimensional. Orthogonality in both parts
of every flattening gives split norms at most one for each law.
For each linear observation the variance sum is bounded by its
squared Hilbert norm, so the independent centered sums converge in
$L^2$. The centered compound Poisson exponents converge locally
uniformly on finite-dimensional parameter spaces by
$|e^z-1-z|\le|z|^2 e^{|z|}/2$ and orthogonality of the fresh spans.
Bounds at $2h$ give uniform integrability of the exponentials for
real $h$. Dividing all intensities by an integer $q$ produces a
$q$th convolution root of every finite joint observation; these roots
come from the same construction on the full space. The resulting
laws are cylindrical and infinitely divisible.

The $n$th group difference has norm between $B_n$ and $2B_n$ at
degree $r_n$, is zero below it, and is at most
$2B_n\epsilon_n^{s-r_n}/d_n\le2^{-n}$ at every $s>r_n$.
Consequently the unobserved difference has no common full factorial
bound, even if some $r_n$ repeat. For channel $j$, each leading
term with $n\ge j$ is either annihilated or bounded in full norm
by $2e^{-nr_n}(r_n!)^\sigma$. Their sum at each degree has a common
factorial bound, as do the higher tails by $|m_{j,s}|\le A_j L_j^s$.
The finitely many earlier leading terms preserve this bound.
Hence every observed difference is full although the original is not.
\end{proof}

A finite testing set also reveals which projections matter at high
degree. Choose one omitted block from every proper $S_j$ in this set.
Once the degree reaches the number of chosen blocks, their union is
itself an admissible pattern. Thus the full-space channels alone must
supply the exponential lower bound in every sufficiently high degree.
Proper projections can fill the gaps in the remaining finite set of
degrees. If there are no proper projections, the same conclusion is
immediate.

A positive observation illustrates the role of the block pattern.
For independent copies, the observed law of
$X^{(1)}+X^{(2)}-2^{1/3}X^{(3)}$ has multipliers
$2+(-2^{1/3})^r$. These vanish precisely at degree three. Indeed for
independent coordinates $X_j=j^{-1/2}(N_j-j)$, $N_j$ Poisson of
mean $j$, comparison with $-X_j$ gives a cubic difference
$2j^{-1/2}$ that is not square summable; all higher odd differences
are uniformly square summable. The identity lifts after this
observation and fails before it. On a sum of four infinite blocks,
add the four unscaled observations that each leave out one block.
Every cubic entry then survives in one of the four, while the
three-copy observation has an exponential lower bound in all other
degrees. Together the five channels recover the full class. The four
projections alone can miss a fourth-order term involving all four blocks.

\paragraph{Separate directions and joint observations.}
A bounded real map $A:K\to H$ observes all cumulants on its image
by pullback. The following criterion concerns the intrinsic classes
of those observations, including their full-tensor indeterminacy.

\begin{theorem}[Directional recovery]\label{obs53:directional}
Let $H$ be infinite-dimensional and separable, $\sigma\ge1$, and
$(A_i:K_i\to H)$ a nonempty countable family of bounded real maps.
The pullbacks $A_i^\#$ jointly determine the class in $\mathcal
O_\sigma(H)$ if and only if one $A_i$ has closed range of finite
codimension, equivalently $A_iA_i^*$ is invertible modulo compacts.
Necessity holds even among centered infinitely divisible cylindrical
laws with identity covariance and a Gaussian component of covariance
at least $I/2$.

If $A$ satisfies the condition, put $d=\dim(\ran A)^\perp$ and
choose a bounded $B:H\to K$ with $AB=P_{\ran A}$. Then every
split-bounded symmetric form satisfies
\begin{equation}\label{obs53:bound}
 \|e_r\|_{\rm Hilb}^2\le
 \|B\|^{2r}\|A^\#e_r\|_{\rm Hilb}^2+
 rd\,\|e_r\|_{1\mid r-1}^2 .
\end{equation}
For a failed finite family of $m$ maps, identity-covariance witnesses
can have their first defect in degree $\max(3,m)$ with infinite full
norm there. Without prescribed covariance degree $\max(2,m)$ suffices.
For a countably infinite family every individual degree can instead
have finite full norm while the common factorial bound fails.
\end{theorem}
\begin{proof}
Put $G=AA^*$. If $A$ has closed cofinite range, its restriction
from $(\ker A)^\perp$ onto that range has bounded inverse. Thus $G$
is bounded below on the range and zero on its finite-dimensional
complement, giving an inverse modulo finite rank. Conversely, if
the spectral subspace of $G$ in $[0,\epsilon)$ is finite dimensional
for some $\epsilon>0$, $G$ has a bounded inverse on $(\ker G)^\perp$
and finite-dimensional kernel. This subspace is the range of $G$,
hence also of $A$, since $\ker G=\ker A^*$. If every such spectral
subspace is infinite dimensional, it supplies an orthonormal sequence
$x_n$ with $\|Gx_n\|\to0$. A left inverse modulo compact operators
would contradict compactness applied to this sequence. In this case
$A^*$ is arbitrarily small on unit vectors orthogonal to any given
finite-dimensional space, since $\|A^*x\|^2=\langle Gx,x\rangle$.
For a map with closed cofinite range, invert its restriction to
$(\ker A)^\perp$ to construct $B$. In a basis adapted to its range
the squared coefficient sum with all slots in the range is at most
$\|B\|^{2r}\|A^\#e_r\|_{\rm Hilb}^2$. Every other coefficient has
a complementary basis vector in at least one slot. For each slot
their squared sum is at most $d\,\|e_r\|_{1\mid r-1}^2$ by Parseval.
This proves \eqref{obs53:bound} and the common factorial bound.

Use the orthogonal compound-Poisson blocks of
Theorem~\ref{obs53:joint}. On orthonormal $v_1,\ldots,v_k$, the signed
jump measure is $2^{-k}\mathop{*}_{j=1}^k
(\delta_{v_j/\sqrt k}-\delta_0)$. Its tensor moments $b_r$ vanish
below $k$, satisfy $\|b_r\|_2\le1$, and
$\|b_k\|_2=\ell_k:=2^{-k}k^{-k/2}\sqrt{k!}$ by the $k!$ surviving
orthogonal words. The positive and negative parts on independent
orthogonal blocks give centered infinitely divisible cylindrical laws
with split norms at most one, as in that proof. For an arbitrary map
$A$, pair atoms differing by $v_j/\sqrt k$ and telescope the tensor
powers of their images under $A^*$. Since those images have norm at
most $L=\max(1,\|A\|)$, every $j$ gives
$\|A^\#b_r\|_2\le rL^{r-1}\|A^*v_j\|$.

For $m<\infty$ take infinitely many orthogonal blocks of size
$k=\max(2,m)$. In block $n$ choose $\|A_i^*v_i\|<2^{-n}$ for every
$i\le m$, using the spectral observation successively. The degree-$k$
full norm is infinite, while the $i$th observed full norm is at most
$r\max(1,\|A_i\|)^{r-1}$ in every degree.

For infinitely many maps set $k_n=n+2$,
$M_n=e^{nk_n}(k_n!)^\sigma$ and
$N_n=\lceil(M_n/\ell_{k_n})^2\rceil$. At stage $n$ use $N_n$
orthogonal blocks of size $k_n$, orthogonal also to previous stages,
with $\|A_i^*v_i\|<2^{-n}/N_n$ for $i\le n$ in each block.
Each degree receives only finitely many stages, but the degree-$k_n$
full norm is at least $M_n$. For the $i$th observed tensor the stages
$n\ge i$ contribute at most $rL_i^{r-1}\sum_{n\ge i}2^{-n}$ in full
norm, where $L_i=\max(1,\|A_i\|)$. The finitely many earlier blocks
contribute at most $L_i^r\sum_{n<i}N_n$. Thus every observed defect
has a common factorial bound.

For identity covariance take all block sizes at least three, using
$k=\max(3,m)$ in the finite case. The two Poisson fields then have
the same covariance. Halve both jump intensities so this common
covariance $Q$ is at most $I/2$, and add to each an independent
Gaussian field with covariance $I-Q$. All defects above are halved.
\end{proof}

A finite joint observation retains mixed cumulants and has Gram
operator $\sum_{j\in J}A_jA_j^*$. Thus all observed groups of at most
$k$ channels suffice exactly when one such sum is invertible modulo
compacts. If a countable joint Gram series converges in norm, its sum
is invertible modulo compacts exactly when some finite partial sum is,
by positivity and openness of invertibility in the Calkin algebra.

In particular, for the inclusions of $m\ge2$ orthogonal
infinite-dimensional summands, every proper joint marginal can agree
in law while the degree-$m$ cumulant defect has infinite full norm.
Use infinitely many orthogonal $m$-vector blocks of the preceding
construction with each $v_j$ in the $j$th summand. On every proper
joint subcollection at least one difference factor pushes forward to
zero, so the positive and negative jump measures give the same
projected law. Each block contributes $\ell_m^2>0$ to the squared
degree-$m$ norm and contributes zero below degree $m$. The joint map
using all $m$ summands is an isometric identification.

For a Cartesian family $D_jA_i^\#$ with the degree multipliers of
Theorem~\ref{obs53:joint} on whole-space channels, completeness holds exactly when
some $A_i$ has closed cofinite range and finitely many $D_j$ satisfy
the lower bound $\max_{j\in J}|m_{j,r}|\ge b^r$ in every degree. Sufficiency follows by applying
the two theorems successively. For necessity, witnesses to either
failure survive the other operation, since both preserve full defects
and degree multiplication commutes with pullback. For any specified
bounded real $T:V(A)\to V(B)$, apply the criterion to $c_A-T^\#c_B$
using $A_i:K_i\to V(A)$. It compares $D_jA_i^\#c_A$ with
$D_j(TA_i)^\#c_B$ on the same $K_i$. When the multipliers preserve
positivity, equality of these observed classes tests the original
lift by Theorem~\ref{new:thm:scalar-invariants}.

Countability is sharp. In a fixed orthonormal basis, the family of
all coordinate subspaces of infinite dimension and infinite codimension
detects every nonfull $e$. After deleting finitely many basis vectors
the remaining restriction is still nonfull by \eqref{obs53:bound}.
At stage $n$ choose a finite set of surviving coordinates on which
some degree-$r$ full norm exceeds $n^r(r!)^\sigma$, and then reserve
one unused basis vector outside all the sets chosen so far. Exclude
all reserved vectors at future stages. The union of the chosen sets
retains the failed factorial bound and omits infinitely many coordinates.
It is infinite, since on a fixed $d$-dimensional subspace the full
norm is at most $\sqrt d\,\|e_r\|_{1\mid r-1}$.

\section{Matching independent blocks with a uniform slice gap}
\label{r46:block-sec}

Consider two positive Appell algebras at the same $\sigma\ge1$, with
identity covariance on real separable Hilbert spaces
$H=\bigoplus_j H_j$ and $K=\bigoplus_i K_i$. The summands are given,
have finite positive dimensions, and may form finite or countable
families. Assume each higher cumulant is the direct sum of its
restrictions to the summands, written $\kappa_{j,r}$ and $\lambda_{i,r}$;
in particular its mixed block coefficients vanish. Assume the global
split bounds defining $\mathcal C_\sigma$.

Fix a nonempty finite set $F\subset\{3,4,\ldots\}$. In a block of
dimension $d$ choose an orthonormal basis. For each $r\in F$ and
ordered basis tuple in the last $r-2$ slots, the remaining bilinear
form gives a real self-adjoint matrix $S_\alpha$. Assume that on
every block in both sources one constant $g>0$ satisfies
\begin{equation}\label{r55:blockgap}
 \sum_\alpha S_\alpha^2\ge g I,\qquad
 \sum_\alpha\|[X,S_\alpha]\|_2^2\ge
 g\|X-d^{-1}\operatorname{tr}(X)I\|_2^2\quad(X=X^*).
\end{equation}
The sum ranges over the degrees and tuples just specified, and
$\|\,\cdot\,\|_2$ denotes the Hilbert--Schmidt or full tensor norm.
These assumptions allow unbounded block dimensions.

\begin{theorem}[Matching for varying block types]\label{r46:block}
A bounded real isomorphism $T:H\to K$ lifts to a topological star
isomorphism if and only if there exists an orthogonal isomorphism
$V:H\to K$ with $T-V\in\mathcal S_2$ and the following properties.
Outside finitely many blocks on each side, $V$ pairs blocks of equal
dimension by maps $V_j:H_j\to K_{\pi(j)}$. The residual spaces have
equal dimension. For some $t>0$ one has
\[
 \sum_{r\ge3}\sum_{j\ {\rm paired}}
 \frac{t^{2r}}{(r!)^{2\sigma}}
 \|\lambda_{\pi(j),r}\circ V_j^{\otimes r}-\kappa_{j,r}\|_2^2
 <\infty .
\]
The affine shift is arbitrary and bounded.
\end{theorem}
\begin{proof}
For a split-bounded $r$-form $c_r$ and bounded real $R,S:H\to K$
differing by a Hilbert--Schmidt operator, telescoping the slots gives
\[
 \|c_r\circ R^{\otimes r}-c_r\circ S^{\otimes r}\|_2
 \le r\max(\|R\|,\|S\|)^{r-1}\|R-S\|_2\,\|c_r\|_{1\mid r-1}.
\]
Indeed one slot of each summand composes the bounded first-slot
flattening with $R-S$, giving a Hilbert--Schmidt operator. The other
slots cost only their operator norms. Thus such perturbations change
the pulled-back cumulants by a full factorial series.

Suppose $T$ lifts. By Theorem~\ref{s:thm:scalar-maps}, its quadratic
defect $T^*T-I$ is Hilbert--Schmidt. For its orthogonal polar factor
$U$, $T-U=U(|T|-I)\in\mathcal S_2$ since
$|T|-I=(T^*T-I)(|T|+I)^{-1}$. Hence $U$ also has full defects.

Write $P_i,E_j$ for the block projections on $K,H$, respectively,
$e_i=\dim K_i$, $d_j=\dim H_j$, and $Q_j=UE_jU^*$. On $K$ let
$C_r$ denote the target cumulants. For a finite-rank self-adjoint
$X$ write $X^{(k)}$ for insertion in slot $k$, and set
$\mathcal D X=\sum_i e_i^{-1}\operatorname{tr}(P_iX)P_i$. Parseval
in the final $r-2$ slots gives
\[
 \sum_{r\in F}\|(X^{(1)}-X^{(2)})C_r\|_2^2
 =\sum_{i,\alpha}\|[X,S_{i,\alpha}]\|_2^2
 \ge g\|X-\mathcal D X\|_2^2 ,
\]
where each block slice is extended by zero. Mixed basis tuples give
zero. The diagonal corner follows from the second bound in
\eqref{r55:blockgap}. For $k\ne i$ the $(k,i)$ corner is
$P_kXP_iS_{i,\alpha}$, whose squared norms sum to at least
$g\|P_kXP_i\|_2^2$ by the first bound. This proves the inequality.

Let $D_r$ now be the difference of $C_r$ from the source cumulant
transported to $K$ by $U^*$. The latter transfers $Q_j$ exactly
between slots because it splits over the $Q_j$. In the last display
with $X=Q_j$ we may therefore replace $C_r$ by $D_r$ on the left.
Acting in separate tensor slots, the two copies of $Q_j$ are commuting
orthogonal projections and their squared difference is bounded by their
sum. Summing over $j$ consequently bounds the left sides by
$2\sum_{r\in F}\|D_r\|_2^2<\infty$. Put
$m_{ij}=\operatorname{tr}(P_iQ_j)$. Its column sums are $d_j$
and its row sums $e_i$. Since
$\|Q_j-\mathcal D Q_j\|_2^2=d_j-\sum_i m_{ij}^2/e_i$, repeating
the argument in the source with $U^{-1}$ yields both estimates
\[
 \sum_{i,j}m_{ij}(1-m_{ij}/e_i)<\infty,\qquad
 \sum_{i,j}m_{ij}(1-m_{ij}/d_j)<\infty .
\]
In particular the sum of all column deficits $d_j-\max_i m_{ij}$
and row deficits $e_i-\max_j m_{ij}$ is finite, since in each
column $\sum_i m_{ij}^2/d_j\le\max_i m_{ij}$ and similarly in rows.
The maxima exist as the entries have finite positive sum.

Only finitely many rows and columns have deficit at least $1/4$.
In a remaining column the largest entry exceeds $3/4$. Only finitely
many such columns hit exceptional rows, which have finite total mass.
If its row is also nonexceptional, the entry must be the unique maximum
there, since the entries outside the row maximum have total mass less
than $1/4$.
The two integer dimensions are within $1/4$ of that entry and hence
agree. The same argument starting with rows gives a cofinite pairing.
Put $L=\sum_{j\ {\rm paired}}(d_j-m_{\pi(j),j})<\infty$.
The total mass of entries outside the paired edges is $L$ plus the
sum of unpaired dimensions, whether counted by rows or columns.
Thus the residual dimensions agree.

On each paired block extend the polar part of $P_{\pi(j)}U|_{H_j}$
to an orthogonal $V_j$ onto $K_{\pi(j)}$. For its singular values
$s_l\in[0,1]$, including zeros,
$\|U|_{H_j}-V_j\|_2^2=2\sum_l(1-s_l)\le2(d_j-m_{\pi(j),j})$.
An isometry between the finite remainders completes $V$, with
$U-V\in\mathcal S_2$. Its higher defects are full by telescoping.
They are orthogonal sums of the paired defects and a residual tensor,
so their full factorial bound gives the stated weighted sum at a
smaller radius.

Conversely, the residual tensor in degree $r$ has full norm bounded
by $\sqrt{d}$ times its first-slot flattening norm in residual
dimension $d$. This follows by summing over a basis of the first
slot; its flattening norm is at most the sum of the two global split
bounds. Orthogonality and the weighted sum now give a common full
factorial bound for the higher defects of $V$, since each weighted squared norm is bounded by the convergent total.
The quadratic defect of $V$ is zero. The telescoping estimate
transfers fullness to $T$, and Theorem~\ref{new:thm:scalar-invariants}
gives the isomorphism and all bounded affine shifts.
\end{proof}

\subsection{Radical powers and Schatten estimates}\label{r56center:sec}
Let $k_r=\kappa_r$, $r\ge3$, satisfy the split bound of
Section~\ref{s:sec:intrinsic} on a separable complexified Hilbert space.
Write $A_i k_r$ for precomposition by $A$ in slot $i$, $C_r$ for the
one-slot flattening, and $w_r(t)=t^{2r}/(r!)^{2\sigma}$. For admissible radii $0<t<t_{\rm split}$, where $t_{\rm split}$ is the
radius of the split bound, put
\[
 d_t(A)^2=\sum_{r\ge3}w_r(t)\|(A_1-A_2)k_r\|_2^2,\qquad
 G_t=\sum_{r\ge3}w_r(t)C_r^*C_r .
\]
The second sum converges in operator norm. Let $\mathcal D$ consist
of bounded $A$ with $d_t(A)<\infty$ for some admissible $t$, and put
$\mathcal D^\dagger=\{A\in\mathcal D:A^*\in\mathcal D\}$.
The adjoint-stable essential center is the quotient of
$\mathcal D^\dagger$ by its members for which $A_1 k$ is full.
This is an algebraic quotient, taken before any closure.

\begin{theorem}[Radical powers and finite-rank tests]\label{r56center:sob}
Suppose $G_{t_0}\ge\delta I$ for some admissible $t_0$ and $\delta>0$.
The essential center just defined is commutative, with Jacobson radical
\[
 \mathfrak R=(\mathcal D^\dagger\cap\mathcal K)/\mathcal S_2 .
\]
For every real $p\ge1$, all compact members of $\mathcal D^\dagger$
belong to $\mathcal S_{2p}$ if and only if for every admissible $t$
there is $M_t<\infty$ such that
\[
 \|A\|_{2p}\le M_t\bigl(\|A\|+d_t(A)\bigr)
 \qquad(A=A^*,\ \operatorname{rank}A<\infty).
\]
For integer $p=m$ this is equivalent to $\mathfrak R^m=0$.
For positive identity-covariance Appell sources at fixed $\sigma$,
the set of exponents passing the test is an isomorphism invariant.
\end{theorem}
\begin{proof}
A finite initial sum of $G_{t_0}$ dominates $(\delta/2)I$;
rescaling it gives $G_t\ge\delta_t I$ at each admissible $t$.
By Parseval the squared weighted norm of $A_1 k$ is
$\|G_t^{1/2}A\|_2^2$, so the denominator is $\mathcal S_2$.
Modulo full tensors, written $\equiv$, the slot insertions of each
$A\in\mathcal D$ agree. For $A,B\in\mathcal D$, all three of
$(AB)_1 k,(AB)_2 k,(BA)_1 k$ agree with $B_2 A_3 k$ modulo full
tensors by transferring insertions and commuting distinct slots.
This proves the asserted algebra and commutator properties.
Invertible $A\in\mathcal D$ has inverse in $\mathcal D$ by the identity
$(A_1^{-1}-A_2^{-1})k=-A_1^{-1}A_2^{-1}(A_1-A_2)k$.
The same holds in $\mathcal D^\dagger$. Its image in the Calkin
algebra has commutative $C^*$-closure. Characters of that closure
detect every noncompact class. Conversely, for compact
$A\in\mathcal D^\dagger$ and $B\in\mathcal D^\dagger$, choose a
finite-rank $F$ with $\|AB-F\|<1$. Then $I-AB+F$ is invertible in
$\mathcal D^\dagger$, showing that $[A]$ is radical.

At fixed $t$, compact self-adjoint operators with finite $d_t$ form a
real Banach space for $\|\cdot\|+d_t$. Operator-norm limits identify
each tensor coefficient with its weighted Hilbert-space limit.
If the inclusion in $\mathcal S_{2p}$ holds, it has closed graph,
giving the estimate. Conversely diagonalize a compact self-adjoint
$A\in\mathcal D$ and replace its eigenvalues $a$ by
$f_\epsilon(a)=\operatorname{sgn}(a)(|a|-\epsilon)_+$. In this
basis each coefficient of the transfer defect is the corresponding
cumulant coefficient multiplied by $a_i-a_j$. The $1$-Lipschitz
truncation therefore gives $d_t(f_\epsilon(A))\le d_t(A)$. Its
finite-rank estimates imply $A\in\mathcal S_{2p}$ by monotone
summation. Real and imaginary operator parts give the general compact
case. For integers $m$, necessity follows from $A^m\in\mathcal S_2$
on self-adjoint representatives, and sufficiency from Schatten
H\"older on $m$ factors. Finally Theorem~\ref{new:thm:scalar-invariants}
and the slot telescoping estimate in the proof of
Theorem~\ref{r46:block} show that the orthogonal polar part of an
isomorphism carries one source's higher cumulants to the other's
modulo full errors. Thus it conjugates the two transfer algebras
and preserves all the tested Schatten classes.
\end{proof}

\paragraph{The covariance metric and an essential Gram gap.}
Let $Q$ be the covariance operator of the positive centered source.
For some $t_1>0$, positivity gives
\begin{equation}\label{r57center:domination}
                G_t\le C_t Q\qquad(0<t<t_1).
\end{equation}
Indeed, for $j\ge2$, differentiation of the formal Appell identity
gives $k_{j+1}(x,f)=\varepsilon(X(x)J_j(f))$ on finite symmetric
tensors. The degree-zero pairing in Proposition~\ref{s:prop:cumulants}
gives $\varepsilon(J_j(f)^*J_j(f))\le C R^{2j}(j!)^{2\sigma}\|f\|^2$.
Cauchy--Schwarz thus bounds $\|C_{j+1}x\|^2$ by this same constant
times $\langle Qx,x\rangle$ with $\|f\|=1$. The symmetric one-slot
image has the same norm in the full tensor power. Weighted summation
at a sufficiently small radius proves \eqref{r57center:domination}.

\begin{theorem}[The center for varying covariances]\label{r57center:metric}
Assume $H$ is infinite dimensional and $0\notin\sigma_{\rm ess}(G_{t_0})$
at one admissible radius. This condition is invariant under unmarked
isomorphisms of positive Appell algebras at fixed $\sigma$. The covariance
$Q$ is Fredholm. Set $B=Q+P_{\ker Q}$ and $A^\sharp=B^{-1}A^*B$.
Then
\[
 \mathcal Z_Q=\{A\in\mathcal D:A^\sharp\in\mathcal D\}/\mathcal S_2
\]
is an intrinsic commutative algebra with involution. Its radical
consists of its compact representatives. The finite-rank tests of
Theorem~\ref{r56center:sob}, performed in the equivalent Hilbert
metric $\langle B\,\cdot,\cdot\rangle$, characterize its radical
powers and its real Schatten exponents. Any positive invertible
$B'$ with $B'-Q\in\mathcal S_2$ gives the same involution on classes.
\end{theorem}
\begin{proof}
The essential gap is equivalent to a finite sum
$\sum_{r=3}^M C_r^*C_r$ being bounded below modulo compact operators.
Indeed, truncate the positive norm-convergent Gram series in the
Calkin algebra, and compare the finitely many positive weights.
Consequently the gap holds at every admissible radius.
Adding full tensors changes each of these finitely many flattenings
by a Hilbert--Schmidt operator. An invertible substitution $T$ changes
$C_r$ by invertible factors $T$ on the input and the tensor pullback
on the output. Their norms and inverse norms are uniformly bounded
over this finite set of orders. This proves gap invariance using
Theorem~\ref{new:thm:scalar-invariants}.
Equation~\eqref{r57center:domination} now makes $Q$ Fredholm. At
each admissible $t$, $G_t$ is bounded below off a finite-dimensional
spectral subspace. Hence $\|G_t^{1/2}A\|_2<\infty$ is equivalent to
$A\in\mathcal S_2$: the omitted finite-dimensional part of $A$ is
always Hilbert--Schmidt.

Pass to the equivalent metric by $W=B^{1/2}$, transforming the
tensors to $k_r\circ(W^{-1})^{\otimes r}$ and operators to $WAW^{-1}$.
This takes the $\sharp$ adjoint to the ordinary Hilbert adjoint.
Transfer fullness is preserved in both directions after shrinking
the factorial radius. In these coordinates the proof of
Theorem~\ref{r56center:sob} applies verbatim: its coercivity step
requires only the denominator identity just proved. A change of
$B$ as in the statement changes $B^{-1}A^*B$ by a Hilbert--Schmidt
operator. Finally an isomorphism with linear part $T:H\to K$ has
$T^*Q_KT-Q_H\in\mathcal S_2$. Thus $T^*B_KT-B_H\in\mathcal S_2$,
so conjugation by $T$ intertwines the two covariance adjoints modulo
$\mathcal S_2$. It also preserves compactness and every Schatten class.
\end{proof}

The fixed Hilbert adjoint can fail to have this invariance. Take
independent centered Poisson coordinates of rate one, grouped in
pairs. The third cumulant gives
$\mathcal D=\{\operatorname{diag}f:f\in\ell^\infty\}+\mathcal S_2$,
since a diagonal operator transfers exactly at every degree. Pull
the source back by the repeated shear
$S=\left(\begin{smallmatrix}1&1\\0&1\end{smallmatrix}\right)$.
The repeated coordinate projection $P=\operatorname{diag}(1,0)$
becomes $P'=S^{-1}PS$. Its ordinary adjoint fails transfer: the
matrix $S(P')^*S^{-1}$ has off-diagonal entries $-2,1$ in every
pair. Its adjoint for the new covariance $S^*S$ is $P'$ itself.

Equation~\eqref{r57center:domination} also gives a compactness
boundary. Compact $Q$ makes $G_t$ compact at small radii, and
trace-class $Q$ makes $\operatorname{tr}G_t$ finite, so the entire
centered cumulant germ is full. Hilbert--Schmidt covariance alone
permits a nonfull third cumulant. For $n\ge3$ take independent
compensated Poisson jumps of size $\log n$ and rate
$n^{-1/2}(\log n)^{-3}$. The diagonal covariance
$q_n=n^{-1/2}/\log n$ is square summable, whereas the cubic
coefficients $n^{-1/2}$ are not. The analytic split bounds follow
from $\sup_{x\ge0}e^{-x/2}x^s\le2^s s!$ with $s=r-3$, since the
split norm of a diagonal tensor equals its maximum coefficient
in absolute value (or its supremum).

\paragraph{Connected graphs with a prescribed exponent.}
For a bounded-degree graph let
$\mathcal E(f)=\sum_{\{u,v\}}|f(u)-f(v)|^2$,
$B=\{f\in\ell^\infty:\mathcal E(f)<\infty\}$ and $B_0=B\cap c_0$.
There is a positive analytic, infinitely divisible source with
identity covariance whose essential center is $B/\ell^2$ and whose
radical is $B_0/\ell^2$. To see this, use independent compensated
Poisson jumps of rate $\alpha>0$ in all vertex directions $e_i$ and
independent symmetric Poisson jumps of total rate $\beta>0$ in the
edge directions $v_e=(e_i+e_j)/\sqrt2$. Fill the covariance to $I$
by an independent Gaussian, taking $\alpha+\beta\deg_{\max}<1$.
The field converges on each linear test in $L^2$. Its cumulant germ
is analytic near zero, since the compensated jump remainders are
quadratic at zero and the squared linear tests are summable. Each
tensor synthesis map with columns $v_e^{\otimes s}$ has squared
norm at most $\deg_{\max}$ by the row sums of its Gram matrix. Thus
all its cumulants have the required split bounds. The pure vertex
third cumulant gives
\[
 \|A_1 k_3\|_2^2=\alpha^2\|A\|_2^2,\qquad
 \|(A_1-A_2)k_3\|_2^2=2\alpha^2\|A-\operatorname{diag}A\|_2^2.
\]
Every class is therefore diagonal. For $A=\operatorname{diag}f$,
the even orders $r\ge4$ have squared transfer norm
$\beta^2\mathcal E(f)/2$, since distinct edges have disjoint
off-diagonal first-two-slot patterns. The odd orders have zero
transfer norm. This identifies the center and its radical by
Theorem~\ref{r56center:sob}. Exact transfer on a connected graph
forces a scalar operator, so the source has one exact independent
higher-cumulant block. Its Schatten test at $p$ is precisely
$B_0\subset\ell^{2p}$.

\needspace{5\baselineskip}
\begin{theorem}[An exact exponent test on attached cylinders]\label{r56center:cyl}
Let $T$ be the $3$-regular tree. To $T\square T$ attach finite
cylinders $X_n\square\{1,\ldots,L_n\}$, using one edge from every
first-layer vertex to a different base vertex, distinct across all
cylinders. The finite connected $X_n$ have uniformly bounded degree,
a uniform mean-zero Poincar\'e constant and sizes $a_n\to\infty$.
The semisimple quotient of the essential center is $\mathbb C$ if and
only if $\sup_n L_n/a_n<\infty$. In that case, for each real $p\ge1$,
\[
 B_0\subset\ell^{2p}\quad\Longleftrightarrow\quad
 \sup_n L_n^{p+1}a_n^{1-p}<\infty .
\]
\end{theorem}
\begin{proof}
Orient $T$ away from a root. The vectors of parent and child values
on its edges have respective norms at least $\sqrt2\|h\|_2$ and at
most $\|h\|_2$ for $h\in\ell^2(T)$. Thus the gradient on $\ell^2(T)$
is bounded below by $\sqrt2-1$ and has closed range. For a
finite-energy function on $T\square T$, the slices $f_y(x)=f(x,y)$
differ by $\ell^2(T)$ functions, by summing differences on a finite
path joining the two slice indices. Their gradients belong to one
affine translate of this closed range. Their squared norms sum to a
finite number, so the translate contains zero. Hence $f_y=c+h_y$
with $h_y\in\ell^2(T)$, a common $c$ by the slice differences, and
$\sum_y\|h_y\|_2^2\le(\sqrt2-1)^{-2}\mathcal E(f)$.

For $f\in B$ subtract this base constant and denote the layer means
of $f-c$ by $u_{n,j}$, with $u_{n,0}=0$. The layer fluctuations have
summable squared norms by the uniform Poincar\'e estimate. Set
$e_n=a_n\sum_{j=1}^{L_n}|u_{n,j}-u_{n,j-1}|^2$. Jensen's inequality
on the vertical edges bounds the terms for $j\ge2$. At $j=1$, use
the attachment energies and the squared mass on their distinct base
vertices. Thus $\sum_n e_n\le C\mathcal E(f)$, and
\[
 \max_j|u_{n,j}|^2\le (L_n/a_n)e_n,\qquad
 \sum_{j=1}^{L_n}a_n|u_{n,j}|^{2p}
 \le L_n^{p+1}a_n^{1-p}e_n^p .
\]
If the ratios are bounded, $e_n\to0$ gives $f-c\in c_0$, so the
semisimple quotient is scalar. For unbounded ratios choose a
subsequence with $\sum a_n/L_n<\infty$. The ramp $j/L_n$ on those
cylinders, zero elsewhere, has finite energy, equals one on their
last layers and is zero on the base. Its class modulo $B_0$ is
nonconstant. The second estimate gives the claimed inclusion when
the displayed supremum is finite, since the base and fluctuations
are already square summable.

Conversely $B_0$ is Banach in $\|f\|_\infty+\sqrt{\mathcal E(f)}$
by coordinatewise convergence of gradients. Its inclusion in
$\ell^{2p}$, when defined everywhere, has closed graph. On a single
cylinder take the ramp $\sqrt{L_n/a_n}\,j/L_n$, extended by zero.
It has unit energy and uniformly bounded supremum norm. Its
$2p$-power sum on the final half of the cylinder is at least
$2^{-2p-1}L_n^{p+1}a_n^{1-p}$. The closed graph bound proves necessity.
\end{proof}

Bounded-degree expander families provide the $X_n$
\cite{r56center:HLW}. For any $p_*\ge1$, set
$L_n=\lfloor a_n^{(p_*-1)/(p_*+1)}\rfloor$. The passing exponents
are exactly $[p_*,\infty)$. Multiplying the length before rounding
by $\log a_n$ gives $(p_*,\infty)$ instead, after discarding finitely
many small widths. With $L_n=a_n$ no finite exponent passes. Here
$B_0$ contains a function in none of the $\ell^{2m}$: for each
$m,N\ge1$ the closed set $\{f:\|f\|_{2m}\le N\}$ has empty interior
by the linearity and failure of the inclusion, so Baire's theorem
applies. This gives a nonnilpotent radical element. The integer
tests realize every finite nilpotence index. At each fixed index
$m\ge2$, varying $p_*$ in $(m-1,m)$ separates uncountably many
unmarked Appell algebras with scalar semisimple quotient and a single
exact block.

\paragraph{Independence cuts on the spectrum of the center.}
Under the hypotheses of Theorem~\ref{r57center:metric}, use the coordinates
$W=B^{1/2}$ of its proof, so that the covariance differs from $I$ by finite
rank. Write $\mathcal C$ for the adjoint-stable transfer algebra in these
coordinates and $q$ for the Calkin map. The uniform closure of $q(\mathcal C)$
is a commutative unital $C^*$-algebra $C(X)$. Complex conjugation on the
real Hilbert space complexification induces an involution on $X$. This
compact space with involution is intrinsic: the linear part of an
isomorphism in these coordinates is unitary modulo Hilbert--Schmidt
operators, by Theorem~\ref{r57center:metric}.

\begin{theorem}[Simultaneous lifts of independence cuts]\label{r58center:cuts}\label{center58:clopen}
Conjugation-invariant clopen subsets of $X$ correspond bijectively to
individually repairable real orthogonal projections modulo Hilbert--Schmidt
difference, in the covariance metric. Every countable Boolean subalgebra of
these subsets has a simultaneous lift preserving all finite Boolean
operations exactly. Each finite subalgebra of such a lift is split by one
positive state in the original cumulant class.

For a fixed countable lift, enumerate its generators, let $\mathcal P_n$
be the nonzero atoms of the first $n$, and set
$D_{n,r}=\sum_{P\in\mathcal P_n}P^{\otimes r}$. One positive state splits
the entire lift if and only if, for some $t>0$,
\begin{equation}\label{r58center:uniform}
 \sup_n\sum_{r\ge3}\frac{t^{2r}}{(r!)^{2\sigma}}
                       \|(I-D_{n,r})k_r\|_2^2<\infty .
\end{equation}
\end{theorem}
\begin{proof}
For self-adjoint $U\in\mathcal C$ and a real Lipschitz function $f$ on a
spectral interval, $d_t(f(U))\le\operatorname{Lip}(f)d_t(U)$. For a
polynomial the divided difference acts on the commuting self-adjoint first
two slot operators, bounded by the supremum of the derivative. Apply it
to the full transfer defect. Extend $f$ with the same Lipschitz constant,
mollify, and approximate the resulting derivatives uniformly by
polynomials. Integration gives uniform polynomial approximants with
derivative bounds $\operatorname{Lip}(f)+o(1)$. Operator-norm convergence
on test vectors and Fatou on tensor coefficients prove the claim.

Given an invariant clopen with indicator $p$, approximate $p$ within
$1/4$ by $q(U)$ for a real self-adjoint $U\in\mathcal C$, taking real and
self-adjoint parts. The essential spectrum of $U$ misses $(1/4,3/4)$.
Choose a point there outside its spectrum, and cut $U$ at that point.
The cutoff is Lipschitz on the spectrum, so the resulting projection lies
in $\mathcal C$ and maps to $p$, by pointwise functional calculus in
$C(X)$. Any two such lifts coincide modulo $\mathcal S_2$: their classes
are idempotents in the commutative center differing by an element of its
radical (Theorem~\ref{r57center:metric}); for commuting idempotents $e,f$
both $e(1-f)$ and $f(1-e)$ are then radical idempotents and vanish.

For a real projection $P$, in a basis adapted to its two ranges every
mixed coefficient of a symmetric $r$-form has at least a $2/r$ fraction
of its permutations with the first two slots on opposite sides. Thus the
squared mixed norm lies between the squared transfer norm and $r/2$ times
that norm. At a smaller radius the extra factor is harmless. The mixed
quadratic part is Hilbert--Schmidt because the covariance is $I$ modulo
finite rank. Taking the product of the two marginals gives a positive
repair exactly when $P\in\mathcal C$: conversely any repair bounds the
mixed part by its full cumulant difference.

For simultaneous lifting, start with $I$. Given a finite lifted partition
with atoms $E_i$, replace an approximant $U$ to the next clopen by
$\sum_i E_iUE_i$. Its Calkin image is unchanged since $q(\mathcal C)$ is
commutative. Choose the spectral cutoff as above. Where an old clopen atom
is contained in the new set, take all of $E_i$; where it is disjoint, take
zero; on the remaining atoms take the cutoff compressed to $E_i$. This
subdivides the old partition exactly, creating an atom for each nonempty
intersection and no spurious finite-rank atom. Induction lifts the union
of the finite subalgebras. For each finite partition, all mixed
coefficients occur in at least one atom's two-sided mixed part. The product
of the atom marginals is therefore a positive repair, with cumulants
$D_{n,r}k_r$.

Finally, the covariance in these coordinates differs from $I$ by a
finite-rank operator $K$. Since $D_{n,2}I=I$, every quadratic mixed
norm is at most $\|K\|_2$. Theorem~\ref{reader59:cuts} now gives
\eqref{r58center:uniform} and its positive realization.
\end{proof}

The graph source preceding Theorem~\ref{r56center:cyl} illustrates the
uniformity condition. Take a ray with a disjoint infinite ray attached at each vertex
by a single edge. The vertex projections onto the attached rays are
disjoint, and each cut has one boundary edge. They lift the countable
Boolean algebra generated by their distinct center projections. The
partition given by the first $n$ cuts separates $n$ attachment edges.
In this realization the fourth coefficient with first two
indices the opposite endpoints of such an edge and last two indices equal
to the same endpoints is $\beta/4$. These $n$ distinct mixed coordinates
give $\|(I-D_{n,4})k_4\|_2^2\ge n\beta^2/16$. Thus this particular lift
admits a repair at every finite stage and no common repair of all stages.

In fact no simultaneous lift of these clopens has a common repair. If
one did, the strong sum of its projections indexed by the even attached
rays would also split the same germ: finite sums split it, and every
fixed cumulant is continuous on tuples of test vectors. The sum would
represent an idempotent in the graph center $B/\ell^2$. Every such
idempotent is represented by a vertex set $S$ with finite edge boundary.
Indeed $f^2-f\in\ell^2$ implies that, off a finite set, $f$ lies within
$1/4$ of either $0$ or $1$; rounding gives $f-\mathbf1_S\in\ell^2$ and
hence $\mathbf1_S\in B$. Its energy counts its boundary edges. Products
with the individual ray classes now force $S$ to contain a tail of each
even ray and avoid a tail of each odd ray. Removing finitely many edges
leaves a single component containing every sufficiently late attached ray,
contradicting these requirements.

\section{Common independence repairs for two cuts}\label{g56:sec}
Let $P,R$ be orthogonal projections on a real separable Hilbert space $H$.
A cut $S$ splits a centered germ when
$c(h)=c(Sh)+c((I-S)h)$, equivalently when the two coordinate spaces
are independent in polynomial moments. A repair replaces $c$ by a
positive centered germ in its class modulo $\mathcal H_\sigma$.
By \eqref{new:eq:state-cumulant} this gives a continuous state on the
original algebra. Separate repairs may use two different states.

Let $H_{ab}=\ker(P-aI)\cap\ker(R-bI)$ for $a,b\in\{0,1\}$,
write $E_{ab}$ for their projections and put
$K=(\bigoplus_{a,b}H_{ab})^\perp$. Define
\[
 Z=(P-R)^2,\qquad
 d=\operatorname{dist}\bigl(\operatorname{spec}(Z|_K),\{0,1\}\bigr)
 \quad(K\ne0).
\]
When $K=0$, the condition $d>0$ below is vacuous.

\begin{theorem}[The two spectral conditions for joint repair]\label{g56:pair}
Fix $\sigma\ge1$.
\begin{enumerate}
\item Every separately repairable positive germ of covariance $I$
admits a joint repair of covariance $I$ if and only if $d>0$.
If $d=0$, a separately repairable germ of covariance $I$ fails
already in degree three, even allowing its covariance to change.
\item Every separately repairable positive germ of arbitrary bounded
covariance admits a joint repair if and only if $Z$ has finite spectrum.
If its spectrum is infinite, a Gaussian covariance $I+V$, with $V$
an orthogonal projection, commutes with $P$ and can have arbitrarily
small Hilbert--Schmidt commutator with $R$, yet has no commuting lift
even modulo compact operators.
\end{enumerate}
If $d>0$, a given separately repairable germ of covariance $Q$ admits a
joint repair exactly when $Q$ differs by a Hilbert--Schmidt operator from
an operator commuting with $P$ and $R$.
\end{theorem}
\begin{proof}
Use the two-subspace representation of Halmos \cite{g56:Halmos}. On
$K=L\oplus L$ its matrices are
\[
 P=\begin{pmatrix}I&0\\0&0\end{pmatrix},\quad
 R=\begin{pmatrix}I-A&f(A)\\ f(A)&A\end{pmatrix},\quad
 Z=A\oplus A,\qquad f(s)=\sqrt{s(1-s)},
\]
where $0\le A\le I$ and $\ker A=\ker(I-A)=0$. Indeed, relative to
$PK\oplus(I-P)K$, write $R=\left(\begin{smallmatrix}B&T\\T^*&D\end{smallmatrix}\right)$.
The four removed intersections make $B,I-B,D,I-D$ injective.
Then $TT^*=B(I-B)$ and $T^*T=D(I-D)$ make the polar isometry of $T$
unitary between the two summands. After identification, $T=\sqrt{B(I-B)}$
and $BT+TD=T$ gives $D=I-B=A$ by injectivity.

Write $L_{S,r}=I-S^{\otimes r}-(I-S)^{\otimes r}$ and
$D_r=\sum_{a,b}E_{ab}^{\otimes r}$. Operators on a slot of a multilinear
form act by inserting their adjoints. When $r\ge3$ and the two individual
defects $L_{S,r}u$ are full, write $\delta_S=\|L_{S,r}u\|_2$. For distinct
slots $i,j$, the contraction $S^{(i)}-S^{(j)}$ kills the two pure parts,
hence $\|(S^{(i)}-S^{(j)})u\|_2\le\delta_S$. Using three distinct slots,
transfer the rightmost factor of $P^{(1)}R^{(1)}$ to slot 2 and the other
to slot 3. Reversing the product and transferring each factor to the same
destination gives the same result. Therefore
\[
             \|[P,R]^{(1)}u\|_2\le2(\delta_P+\delta_R).
\]
If both defects vanish, this identity annihilates every slot on the
dense range in $K$ of $[P,R]=
\left(\begin{smallmatrix}0&f(A)\\-f(A)&0\end{smallmatrix}\right)$.
On the four character spaces, a mixed tuple has two different bits
for at least one cut. Thus every jointly split $r$-form equals $D_ru$.

For $d>0$ and $K\ne0$, the inverse norm of $[P,R]|_K$ is at most
$[d(1-d)]^{-1/2}$. Partition the slots with a $K$ factor by the position
of their first such factor, using the orthogonal projections
$(I-E_K)^{\otimes(i-1)}\otimes E_K\otimes I^{\otimes(r-i)}$, where
$E_K$ projects onto $K$. These commute with both $L_{S,r}$. Applying the
inverse in the indicated slot makes each piece full, bounded by the
three-slot estimate for its own defects. On the remaining slots the
pure-part projections $I-L_{P,r}$ and $I-L_{R,r}$ commute and their
product is $D_r$. Adding squared norms gives the order-independent bound
\begin{equation}\label{g56:bound}
 \|(I-D_r)u\|_2^2\le
 B\bigl(\|L_{P,r}u\|_2^2+\|L_{R,r}u\|_2^2\bigr),\qquad
 B=8/[d(1-d)]
\end{equation}
for $K\ne0$; take $B=1$ if $K=0$. This argument gives fullness
even when $u$ itself has only bounded split flattenings.

A separately repairable germ therefore has a full joint higher-degree
correction with one factorial bound. A commuting covariance lift is
necessary. If there is a lift, take the self-adjoint commuting lift
$\widehat Q$ closest to $Q$ in Hilbert--Schmidt distance. The set of
corrections is a closed affine subspace, and $Q-\widehat Q$ is
orthogonal to every commuting self-adjoint Hilbert--Schmidt operator.
The spectral projection of $\widehat Q$ below $-\epsilon$ has finite
rank, since $\widehat Q-Q$ is compact and $Q\ge0$. If nonzero it
contradicts this orthogonality, so $\widehat Q\ge0$. Rank-one tests
on $H_{ab}$ give $E_{ab}\widehat Q E_{ab}=E_{ab}QE_{ab}$. Take the
independent product of the four original character marginals and a
Gaussian field on $K$ of covariance $\widehat Q|_K$. Its cumulants are
$\widehat Q$ and $D_r\kappa_r$ for $r\ge3$. The product preserves
positivity and the split bounds by blockwise contraction. This proves
the individual lift criterion and the sufficiency in (1).

If $d=0$, choose unit $z_n=(x_n,0)\in K$ in disjoint nonzero
$A$-spectral bands approaching the same endpoint. For $S=R$ or
$I-R$ according to that endpoint, the bands can be chosen so that
$\sum_n\|(I-S)z_n\|^2<\epsilon^2$ for any prescribed $\epsilon>0$.
Use independent centered rate-one Poisson coordinates on the $z_n$
and an independent standard Gaussian field on their orthogonal
complement. The covariance is $I$ and $\kappa_r=\sum_n z_n^{\otimes r}$
for $r\ge3$, in the sense of forms. Every split flattening has norm at
most one, since each family $(z_n^{\otimes k})$ is orthonormal.
The $P$ cut splits exactly. Replacing $z_n$ by $Sz_n$ in each of the
$r$ slots successively changes $\kappa_r$ by a full tensor of norm
at most $r\epsilon$. For each single-slot change the remaining
tensor factors, indexed by $n$, are the image of the orthonormal
$(z_n^{\otimes(r-1)})$ under a tensor product of contractions, hence
synthesize an $\ell^2$ sequence contractively. The cut defect for $R$
is thus full with norm at most $r\epsilon$. Taking the independent
product of the two $R$-marginals gives a separate positive repair.
Every jointly split cubic vanishes on $K$, whereas
$\kappa_3(z_n,z_n,z_n)=1$ for all $n$. Its correction cannot be full.

For finite spectrum, $Z$ has polynomial spectral projections. At
$Z=0,1$ one has $R=P,I-P$, respectively. Each interior fibre is a
fixed real two-by-two matrix pair tensored with the identity on its
multiplicity space. Consequently the real unital algebra generated by
$P,R$ is finite dimensional and spanned by finitely many words in
the two cuts. If $[Q,P],[Q,R]$ are Hilbert--Schmidt, the commutator
with every word is Hilbert--Schmidt by telescoping, uniformly for
the orthogonal elements of this finite-dimensional algebra by expansion
in a fixed basis. Haar averaging $Q$ under this compact group gives a
commuting lift modulo Hilbert--Schmidt operators, since the group contains
$I-2P,I-2R$. Separate repair supplies the individual commutators
by the quadratic cumulant condition. Finite spectrum gives $d>0$
on a nonzero generic part, proving sufficiency in (2).

For infinite spectrum choose disjoint pairs of nonzero $A$-spectral
bands tending to a cluster point $a$, with unit vectors $x_n,y_n$
in the two bands of pair $n$. Make
$|s-a|+|f(s)-f(a)|\le\epsilon 2^{-n}$ throughout pair $n$.
If $W$ projects onto the span closure of $(x_n+y_n)/\sqrt2$, then
$(g(A)-g(a))W$ has Hilbert--Schmidt norm at most $\epsilon$
for $g(s)=s,f(s)$. Take $V=W\oplus W$ on $K$ and zero elsewhere.
The matrices show that $[P,V]=0$ and $\|[R,V]\|_2\le4\epsilon$.
Individual covariance compression gives separate Gaussian repairs.
Any joint covariance commutes with $Z$, hence has zero entries between
the paired bands. The corresponding entries of $I+V$ are $1/2$
for all $n$, ruling out even a compact correction.
\end{proof}

\begin{corollary}[The sharp constant is cubic]\label{g56:sharp}
If $K\ne0$ and $d>0$, let $\Lambda=\operatorname{spec}(Z|_K)$ and
start with $m=\sqrt{1-3d(1-d)}$. Include $\sqrt{1-\min\Lambda}$ in
the maximum defining $m$ if $H_{00}\oplus H_{11}\ne0$, and include
$\sqrt{\max\Lambda}$ if $H_{01}\oplus H_{10}\ne0$. Then $(1-m)^{-1}$
is the smallest constant for \eqref{g56:bound} in all orders $r\ge3$.
In generic position it is $(1-\sqrt{1-3d(1-d)})^{-1}$.
If $K=0$ and at least two different character spaces are nonzero, the
best constant is the reciprocal of the least Hamming distance between
their binary labels; with just one nonzero space it is zero. In all
cases degree three alone gives the sharp constant on symmetric tensors.
\end{corollary}
\begin{proof}
The proof of \eqref{g56:bound} shows that $(I-D_r)u$ is full.
Restrict the pure projections $F=I-L_{P,r}$ and $G=I-L_{R,r}$
to the orthogonal complement of $\operatorname{ran}D_r$.
The map $(x,y)\mapsto x+y$ from $\operatorname{ran}F\oplus\operatorname{ran}G$
has squared norm $1+\|FG\|$ whenever at least one range is nonzero,
by its two-by-two Gram matrix. Its product with its adjoint is $F+G$.
Consequently the best squared-norm constant on this complement is
$(1-\|FG\|)^{-1}$ in this case.

On the all-generic slots, work in the spectral representation for $A$.
Put $C=\prod_i\sqrt{1-s_i}$ and $T=\prod_i\sqrt{s_i}$. The two pure
directions of $R$ in slot $i$ are
$(\sqrt{1-s_i},\sqrt{s_i})$ and $(-\sqrt{s_i},\sqrt{1-s_i})$.
Their pure products, paired against the two pure $P$ products, give
the overlap matrix
$\left(\begin{smallmatrix}C&(-1)^rT\\T&C\end{smallmatrix}\right)$.
Its norm is $\sqrt{C^2+T^2}$ for odd $r$ and $C+T$ for even $r$.
H\"older's inequality in the two summands bounds these by
$\sqrt{(1-d)^r+d^r}$ and $(1-d)^{r/2}+d^{r/2}$, respectively, since
$s_i\in[d,1-d]$. Both bounds are sharp on symmetric tensors. Choose
a unit vector in a spectral band shrinking to an edge $s=d$ or $1-d$
in $\Lambda$ and take its $r$th tensor power in each pure product.
The two-by-two overlap converges to the matrix at $s_1=\cdots=s_r=s$,
while both pairs of pure products remain orthonormal. The largest
odd-order norm is at $r=3$ and the largest even-order norm at $r=4$.
Their squared difference is $d(1-d)(1-2d)^2\ge0$.

In a slot sector containing both characters and generic factors, both
pure projections can survive only if all characters have the same
label $(a,b)$. It selects exactly one row and column in the overlap
matrix on the generic slots. Their overlap is a product of
$\sqrt{1-s_i}$ if $a=b$, and of $\sqrt{s_i}$ otherwise. The supremum
over these sectors equals the corresponding one-factor supremum,
already in degree three: take two copies of one character vector and
one generic factor in a shrinking spectral band. Symmetrizing retains
the overlap since the three placements of the generic slot are
orthogonal. On all-character slots outside $\operatorname{ran}D_r$
the squared defect counts the cuts in which the tuple has both
bits. If two different labels occur, the smallest nonzero count is their
minimum pairwise Hamming distance, attained by a cubic. This also proves the
formula when $K=0$.
\end{proof}

\subsection{Quadratic and cubic gaps for families of cuts}\label{cut57:sec}
Let $(P_j)$ be a finite or countable family of orthogonal projections
on a separable real Hilbert space. Write $M=\{P_j\}'$ and
$N=M\cap\mathcal S_{2,\rm sa}$. Let $\Pi$ denote orthogonal projection
onto $N$ in the real space of self-adjoint Hilbert--Schmidt operators.

\begin{theorem}[A finite global gap for Gaussian repair]\label{cut57:global}
Suppose $H$ is an orthogonal sum of finite-dimensional common reducing
subspaces, with arbitrary dimensions and multiplicities. Every bounded
self-adjoint $X$ with $[P_j,X]\in\mathcal S_2$ for every $j$ differs
by a Hilbert--Schmidt operator from an element of $M$ if and only if
some finite prefix has
\[
 b_n=\inf_{\substack{T\in\mathcal S_{2,\rm sa}\cap N^\perp\\
                       \|T\|_2=1}}
                  \sum_{j\le n}\|[P_j,T]\|_2^2>0 .
\]
An empty infimum is $+\infty$. When $b_n>0$, the first $n$ commutators
alone give a lift and
\begin{equation}\label{cut57:sharp2}
 \min_{\substack{X-S\in M\\ S\in\mathcal S_{2,\rm sa}}}\|S\|_2^2
            \le b_n^{-1}\sum_{j\le n}\|[P_j,X]\|_2^2 .
\end{equation}
The constant is optimal. Equivalently the qualitative condition can
be tested on separately repairable Gaussian covariances $I\le Q\le2I$.
\end{theorem}
\begin{proof}
Decompose $H=\bigoplus_i H_i$ further into finite-dimensional irreducible
common reducing spaces, and let $F_k$ project onto the first $k$ spaces.
Compression by any reducing projection is an orthogonal projection on
$\mathcal S_{2,\rm sa}$ preserving $N$, so commutes with $\Pi$. If
$b_n>0$, the corrections $(I-\Pi)F_kXF_k$ have squared norms bounded
by the right side of \eqref{cut57:sharp2}. A weak Hilbert--Schmidt
subsequential limit $S$ has the same bound and $X-S\in M$, by weak
operator closedness of $M$. A minimum exists by Hilbert projection
onto the closed affine set of corrections. Testing on full
$T\perp N$ gives sharpness. Gaussian separate repair at $P_j$ is
exactly the Hilbert--Schmidt commutator condition by covariance
compression. The closest self-adjoint joint lift of a positive covariance
is positive by the negative spectral projection argument in the proof
of Theorem~\ref{g56:pair}. Affine rescaling of $X$ into $[I,2I]$
therefore gives the Gaussian formulation.

Assume lifting. There must first be $n,m$ and $b>0$ such that the
gap inequality for $N^\perp$ with constant $b$ holds for operators
supported on the tail $(I-F_m)H$. Otherwise recursively take disjoint
finite intervals of blocks and self-adjoint $T_k\perp N$ on them
with $\|T_k\|_2=1$ and $\sum_{j\le k}\|[P_j,T_k]\|_2^2<4^{-k}$.
Indeed, failure beyond the preceding interval supplies a tail vector
with strict slack; compression by a sufficiently large finite interval
preserves perpendicularity and converges in Hilbert--Schmidt norm.
The block sum $X=\bigoplus_k T_k$ is bounded and every cut commutator
is Hilbert--Schmidt. A full correction must have norm at least one on
every interval, which is impossible.

Write $P_{j,i}=P_j|_{H_i}$. In position $(i,l)$ the full commutant
consists of matrices $A:H_l\to H_i$ intertwining every cut; take
self-adjoint matrices on the diagonal. The Hilbert projection $\Pi$
is blockwise projection onto these intertwiner spaces. Opposite
positions are paired by adjoints. The tail inequality therefore
controls every such matrix perpendicular to the intertwiners when
$i,l>m$, with gap $b$.

Let $D=\max_{i\le m}\dim H_i$ if the initial segment is nonempty.
There are only finitely many equivalence types among tail blocks of
dimension at most $2D$. Otherwise a subsequence of distinct types of
one dimension $q$ would have convergent first $n$ projection matrices
in fixed orthonormal coordinates. Distinct irreducibles admit no
nonzero intertwiner, by kernel, image and polar decomposition. Testing
a close pair with $I_q/\sqrt q$ contradicts the tail gap. The initial
blocks and these bounded-size tail blocks now give finitely many pairs
of types. On each, the decreasing kernels of the finite-prefix
intertwining equations stabilize at the full intertwiner space by
finite dimensionality. Increasing $n$ gives a strictly positive
minimum gap on all their nonzero orthogonal complements.

It remains to compare an initial block $i$ to a tail block $l$ with
$d_l=\dim H_l>2D$. For $\|A:H_l\to H_i\|_2=1$ put $T=A^*A$. It has
trace one and rank at most $D$. The self-adjoint commutant on $H_l$
is scalar, since each eigenspace of a commuting self-adjoint matrix
reduces the cuts. Hence
\[
 \|T-I/d_l\|_2^2=\operatorname{tr}(T^2)-1/d_l\ge 1/D-1/d_l>1/(2D).
\]
Writing $E_j=P_{j,i}A-AP_{j,l}$ gives
$[P_{j,l},T]=A^*E_j-E_j^*A$. The tail gap on this single diagonal
block yields $\sum_{j\le n}\|E_j\|_2^2\ge b/(8D)$ because
$\|A\|\le1$. Thus all pairs have a common positive lower bound after
one finite prefix. Summing squared matrix norms proves the global gap.
\end{proof}

For positive summable weights $w_j$, set $W=\sum_jw_j$ and normalize
them to the cut law $\xi=W^{-1}\sum_jw_j\delta_{P_j}$.
The cubic gap for this law is the higher-order part of the comparison
with Gaussian repair. Write $D_r=\sum_\chi E_\chi^{\otimes r}$ and
\[
 a_r=\inf_{\substack{u\in H^{\odot r},\ D_ru=0\\ \|u\|_2=1}}
         \sum_jw_j\|L_{P_j,r}u\|_2^2,\qquad L_{j,r}=L_{P_j,r}.
\]
\begin{corollary}[Weighted families]\label{cut57:cubic}
The gaps $a_r$ increase with $r\ge2$. The sharp squared-norm
constant for all $r\ge3$ is $1/a_3$, also for bounded symmetric
forms of finite defect energy if $a_3>0$.
\end{corollary}
\begin{proof}
In degrees $r\ge3$, $a_r=W(1-c_r)$ on a nonzero complement.
Apply Theorems~\ref{r70:momentgap} and~\ref{r72:repairiff}. For
$r=2$, the same sector decomposition on $D_2^\perp$ gives the norm
$\max\{L,\|S_2|_{K^{\odot2}}\|\}$, with $K=H_{\rm rem}$. The pure
projection in three slots is dominated by the two-slot projection
tensored with the identity, giving $\|S_3|_{K^{\odot3}}\|\le
\|S_2|_{K^{\odot2}}\|$. Thus $a_2\le a_3$ even when the quadratic
commutant is larger than the profile tensors. If any complement
vanishes, a single profile spans $H$ and every $a_r=+\infty$.
\end{proof}

\begin{corollary}[A sharp quadratic-to-cubic comparison]\label{cut57:bridge}
For every finite or countable family, if $b_n>0$ then the cubic
gap for the first $n$ cuts with unit weights is at least $b_n/2$.
The factor $1/2$ is optimal. Under the reducing-exhaustion hypothesis,
the equivalent conditions of Theorem~\ref{cut57:global} are also
equivalent to universal
separate-to-joint repair of positive germs at any fixed $\sigma\ge1$.
The first $n$ cuts then suffice in every degree, with squared-norm
constant $2/b_n$ above degree two.
\end{corollary}
\begin{proof}
If $b_n=+\infty$, every rank-one projection commutes with every cut,
so all cuts are scalar. Otherwise the Hilbert--Schmidt self-adjoint
commutant of the prefix equals $N$ by the gap. For every unit vector in
a prefix binary eigenspace its rank-one projection belongs to $N$.
Every later cut therefore acts as a scalar on that eigenspace, so the
prefix and the full family have the same character spaces.

On the first three tensor slots let $J_{ik}$ project to the fixed
vectors under simultaneous action in slots $i,k$ of the group
generated by the reflections $I-2P_j$, $j\le n$. On $D_3^\perp$
we claim
\begin{equation}\label{cut57:threepair}
                          J_{12}+J_{13}+J_{23}\le 2I .
\end{equation}
Complexify $H$ and decompose it into finite-dimensional irreducible
unitary summands and the orthogonal complement of their closed span.
Invariant tensors of two slots identify with Hilbert--Schmidt
intertwiners from the conjugate of one representation to the other.
The nonzero singular values of such an intertwiner have
finite-dimensional invariant singular spaces on both sides. Thus a
pair using the complementary summand has zero fixed projection, and
a triple containing it has at most one surviving pair. In a triple of
finite irreducibles, a pair $U,V$ survives exactly when $V$ is
equivalent to $\overline U$. By Schur's lemma its invariant space is
a line, spanned in suitable orthonormal bases by
$d^{-1/2}\sum_{a=1}^d e_a\otimes f_a$, where $d=\dim U=\dim V$.
Two nonzero pair projections with a shared slot have product norm
$1/d$: use the same basis in the shared slot, then contracting the two
normalized sums maps the free-slot basis isometrically to the other
free-slot basis with coefficient $1/d$. If $d=1$, every reflection
acts by a binary sign, and two matching pairs force all three labels to
agree, in $\operatorname{ran}D_3$. On the remaining triples either at
most one pair survives or every overlap of two different pair ranges
has norm at most $1/2$. For three vectors in the respective ranges this
gives $\|\sum v_{ik}\|^2\le2\sum\|v_{ik}\|^2$ by Cauchy--Schwarz.
The range-sum map times its adjoint is $\sum J_{ik}$, proving
\eqref{cut57:threepair}.

For a real symmetric cubic $u$, its slices in any pair of slots are
self-adjoint Hilbert--Schmidt matrices. In those slots the fixed
projection is $\Pi$. Summing \eqref{cut57:sharp2} on the slices along
an orthonormal basis of the third slot gives
\[
 b_n\|(I-J_{ik})u\|_2^2
       \le\sum_{j\le n}\|(P_j^{(i)}-P_j^{(k)})u\|_2^2 .
\]
In the binary decomposition of three slots at $P_j$, every mixed tuple
has exactly two unequal unordered pairs. Thus the sum over $i<k$ on
the right is $2\sum_{j\le n}\|L_{j,3}u\|_2^2$. On $\ker D_3$,
\eqref{cut57:threepair} bounds the left sum below by $b_n\|u\|_2^2$,
proving the comparison. Theorem~\ref{cut57:cubic} controls all higher defects with the same
constant, also on continuous forms.
A separately repairable germ has full defects at one factorial bound
for this finite prefix. Use the closest commuting covariance and the
independent product of the original character marginals and the
residual Gaussian, as in Theorem~\ref{g56:pair}. This gives a common
positive repair. The converse follows by testing on Gaussian germs.

For sharpness take three planar line cuts at $0,\pi/6,\pi/3$, with
unit weights. On the real traceless self-adjoint unit matrices
$2^{-1/2}\left(\begin{smallmatrix}\cos\varphi&\sin\varphi\\
\sin\varphi&-\cos\varphi\end{smallmatrix}\right)$ the commutator
energies are $\sin^2(\varphi-2\theta_j)$, summing to $3/2$.
The cubic formula immediately below gives $a_3=3/4$.
\end{proof}

\begin{corollary}[An explicit family of noncommuting cuts]\label{cut57:angles}
Let $H=J\otimes\mathbb R^2$, $J\ne0$, and let $P_j$ project
in the second factor onto the line at angle $\theta_j$. If at least
two angles differ modulo $\pi/2$, then $D_r=0$ and the sharp constant
in Theorem~\ref{cut57:cubic} is
\[
          \frac{4}{2W-\sqrt{W^2+3|\sum_jw_j e^{4{\rm i}\theta_j}|^2}} .
\]
Every separately repairable positive germ of arbitrary bounded
covariance admits a joint repair for this family.
\end{corollary}
\begin{proof}
For $r=3$ the pure projections on the planar tensor factors vanish
outside $\operatorname{Sym}^3\mathbb C^2$. In the normalized symmetric
tensor basis of $(e_1+{\rm i}e_2)/\sqrt2,(e_1-{\rm i}e_2)/\sqrt2$,
the four coefficients of the cube of the unit line vector are
$2^{-3/2}\sqrt{\binom3k}\,e^{{\rm i}(2k-3)\theta_j}$, $0\le k\le3$
(up to reversing the basis order). Adding the projector on its
orthogonal line, at $\theta_j+\pi/2$, kills entries with odd $k-l$.
The weighted sum of the pure projections has two Hermitian
two-by-two blocks with diagonal $W/4,3W/4$ in either order and
off-diagonal of modulus $\sqrt3|\sum_j w_je^{4{\rm i}\theta_j}|/4$.
Subtracting the largest eigenvalue from $W$ gives $a_3$ and the
stated formula. The multiplicity factor tensors these operators with
the identity; an eigenvector in the planar symmetric cube tensored
with $x^{\otimes3}$ for a fixed unit $x\in J$ gives sharpness in
$H^{\odot3}$. Real and complex extremal eigenvalues coincide.
The two distinct unordered cuts have no common character and have
a single interior angle value. By Theorem~\ref{g56:pair} their common
repair is Gaussian. Their two planar matrices generate $M_2(\mathbb R)$:
compressing the second by the two outcomes of the first gives nonzero
off-diagonal matrix units. The repaired covariance therefore commutes
with every $P_j$, as required.
\end{proof}

\subsection{The harmonic obstruction beyond reducing finite-dimensional spaces}
For finitely many orthogonal projections $P_j$ on a real Hilbert space let
$\mathcal M=\{P_j:j\}'$, let $N=\mathcal M\cap(\Schat_2)_{\rm sa}$,
and set $\mathcal L X=\sum_j[P_j,[P_j,X]]$. Say that a bounded
self-adjoint $X$ has finite energy if
$E(X)=\sum_j\|[P_j,X]\|_2^2<\infty$.
\begin{proposition}\label{cut58:harmonic}
Suppose $\mathcal L\ge aI$ on $N^\perp$ inside $(\Schat_2)_{\rm sa}$, for $a>0$.
Every bounded finite-energy $X$ has a unique decomposition $X=R+Y$
with $R\in N^\perp$, $\mathcal L Y=0$. It satisfies
$a\|R\|_2^2\le E(R)\le E(X)$. Thus all such $X$ admit a Hilbert--Schmidt
correction into $\mathcal M$ precisely when every bounded finite-energy
solution of $\mathcal L Y=0$ belongs to $\mathcal M$. A family of three
projections can have a positive gap and fail this condition.
\end{proposition}
\begin{proof}
Trace duality gives $\mathcal L X\in N^\perp$ and
$\langle\mathcal L X,S\rangle_2=\sum_j\langle[P_j,X],[P_j,S]\rangle_2$
for self-adjoint Hilbert--Schmidt $S$. Set
$R=(\mathcal L|_{N^\perp})^{-1}\mathcal L X$. The two commutator vectors
for $R$ and $Y=X-R$ are orthogonal, proving the estimates. If a harmonic
$Y$ differs from an element of $\mathcal M$ by a Hilbert--Schmidt
operator, their difference is harmonic in $\Schat_2$ and hence belongs
to $N$. This proves the criterion. A finite-dimensional reducing
exhaustion forces every bounded harmonic $Y$ to commute, since its
finite-dimensional compressions do and converge strongly to $Y$.

For the counterexample use the three generating involutions $s_j$ of
$\Gamma=\mathbb Z_2*\mathbb Z_2*\mathbb Z_2$ and put
$P_j=(I-\lambda(s_j))/2$ on $\ell^2(\Gamma)$. On Hilbert--Schmidt
matrix entries the simultaneous left action on pairs has free orbits,
each a three-regular tree. On each orbit $\mathcal L=(3I-A)/2$, with
$A$ tree adjacency. Orient the tree towards an end. For each edge from
a parent $v$ to a child $w$, sum
$2|f(v)f(w)|\le2^{-1/2}|f(v)|^2+2^{1/2}|f(w)|^2$ to get
$\|A\|\le2\sqrt2$. Hence the gap is at least $(3-2\sqrt2)/2$.
Delete one edge of the left Cayley tree and let $X$ multiply by the
indicator of one component. Each commutator is finite rank. Operators
commuting with all left translations have constant diagonal, which
cannot be at square-summable distance from this indicator. Thus $I+X$
is a positive invertible covariance for which the correction fails.
\end{proof}

\section{Regression directions with prescribed Gaussian moments}\label{r55:regression}
For a centered isotropic random vector $X$ in $\mathbb R^n$, $n\ge2$,
with a moment generating function finite near zero, write
\[
 \mathcal D(X)=\{[v]\in\mathbb {RP}^{n-1}:
      \mathbb E[X\mid v^\top X]=v(v^\top X),\quad \|v\|=1\}.
\]
An orthonormal basis gives pairwise mean independent (PMI) coordinates
precisely when all its lines belong to $\mathcal D(X)$. In a whitened
square mixing model identification up to signs and order therefore
amounts to uniqueness of this orthogonal basis of lines.
Ribot, Seigal and Zwiernik ask for conditions both for individual
tensors and for distributions when all cumulants are observed
\cite[Discussion]{p53:RSZ}. Eaton's characterization treats
$\mathcal D(X)=\mathbb {RP}^{n-1}$ \cite{p53:Eaton}.

\begin{theorem}[Realization with prescribed Gaussian moments]\label{p53:real}
The possible sets $\mathcal D(X)$ are exactly the real projective algebraic
subsets of $\mathbb {RP}^{n-1}$. For every such set and every finite
$N$, a realizing law can be chosen centrally symmetric, strictly
positive, smooth and strongly log-concave, with all exponential moments
and the same cumulants through order $N$ as $N(0,I_n)$.
For each fixed analytic-moment law there is an order $R(X)$ such
that the common eigenlines of its cumulants through $R(X)$ equal
its direction set. Laws sharing such finite data can have different
direction sets, even for the Gaussian data above.
\end{theorem}
\begin{proof}
Put $M(z)=\mathbb E e^{z^\top X}$ and let $C_r(z)$ be the homogeneous
degree-$r$ cumulant form. For $u\perp v$ the signed measure
$\mathbb E[(u^\top X)\mathbf1_{\{v^\top X\in\cdot\}}]$ has a local
exponential moment. Its Laplace transform is $\partial_u M(tv)$.
Uniqueness on a strip, by Fourier uniqueness, shows that $[v]$ is
admissible precisely when all these transforms vanish near zero.
Expanding $\log M$ gives the equivalent equations
\[
 v_i\partial_j C_r(v)-v_j\partial_i C_r(v)=0,\qquad r\ge3,\quad i<j .
\]
The first two orders are automatic. This is the directional form of
the PMI cumulant criterion \cite[Theorem~2.3]{p53:RSZ}. The ascending
chain of ideals generated by the equations through successive orders
stabilizes, proving the finiteness and algebraicity assertions.

For the converse, given finitely many homogeneous equations $g_j$ of
degrees $b_j$ for the prescribed set, choose $m\ge\max_j b_j$ with
$2m>\max(N,2)$. Set $F=\sum_j g_j^2\|z\|^{2(m-b_j)}$, allowing
$F=0$ for the whole space and $F=\|z\|^{2m}$ for the empty set. Let
\[
 h_1=\sum_{i=1}^n i z_i^2,\quad h_2=(\textstyle\sum_i z_i)^2,\qquad
 (P_0,P_1,P_2)=(F,Fh_1,\|z\|^2 Fh_2).
\]
The degrees are $2m,2m+2,2m+4$. With $\phi_a$ the density of
$N(0,aI_n)$, fix $0<a<1$ and take
$f=\phi_1+\epsilon\sum_j P_j(\partial)\phi_a$.
This uses the polynomial Gaussian perturbation of
\cite[Theorem~2.2]{p53:MP}. Each perturbation divided by $\phi_1$ is
a polynomial times a decaying Gaussian, with bounded first two
derivatives. For all sufficiently small nonzero $\epsilon$,
$f\ge\phi_1/2$ and $-\nabla^2\log f\succeq I_n/2$. Integration
by parts gives
$M(z)=e^{\|z\|^2/2}+\epsilon e^{a\|z\|^2/2}\sum_j P_j(z)$,
proving normalization and the cumulant agreement.
At $tv$ the transverse derivatives of the perturbation are a common
nonzero exponential times distinct powers of $t$ with coefficients
$\partial_u P_j(v)$, $u\perp v$. They vanish identically exactly at the
common spherical critical points of the $P_j$. At $F=0$ all vanish
since $F$ is a sum of weighted squares. Outside this set, vanishing
forces both $h_1$ and $h_2$ to have zero spherical gradient. The simple
diagonal quadratic $h_1$ has only coordinate eigenlines, and none is
an eigenline of $h_2$. This proves exact realization.
\end{proof}

\begin{corollary}[Gaussian marginal sources]\label{reg56:gauss}
In the realization of Theorem~\ref{p53:real}, the unit directions with
a standard Gaussian projection are exactly the prescribed regression
directions. In every dimension $n\ge2$, two positive strongly log-concave
isotropic laws can therefore have standard Gaussian coordinate marginals
and disagree on identifiability of their pairwise mean independent bases.
\end{corollary}
\begin{proof}
On the ray $tv$ the first homogeneous term of the perturbation of the
Gaussian moment generating function is $\epsilon t^{2m}F(v)$. If
$F(v)=0$ the entire perturbation vanishes there. Take the prescribed
set to consist of the coordinate lines and compare the resulting law
with the standard Gaussian vector.
\end{proof}

\begin{theorem}[All closed regression sets in one moment class]\label{reg56:closed}
For every $n\ge2$ and $\sigma>1$ there is a single centered isotropic
spherical law $\mu_\sigma$ with
$\int |x|^k\,d\mu_\sigma\le C^{k+1}(k!)^\sigma$ for all $k\ge0$.
For each closed $A\subset\mathbb {RP}^{n-1}$ there is a centered
isotropic law $\mu_A$ with a strictly positive real analytic density and
$\mathcal D(\mu_A)=A$. Its characteristic function agrees with
$\widehat\mu_\sigma$ in a neighborhood of zero, and
\[
  (v^\top)_*\mu_A=(v^\top)_*\mu_\sigma\ \Longleftrightarrow\ [v]\in A
  \quad (|v|=1).
\]
Every $\mu_A$ has the same mixed moments and cumulants as $\mu_\sigma$
and satisfies the same factorial growth estimate after increasing $C$.
Every finite-first-moment law has a closed regression-direction set.
\end{theorem}
\begin{proof}
For a centered law $\mu$ with a finite first moment let
$\varphi(\xi)=\int e^{i\xi^\top x}\,d\mu(x)$. For $u\perp v$ the
Fourier transform of the signed measure
$\mathbb E[(u^\top X)\mathbf1_{\{v^\top X\in\cdot\}}]$ at $t$
equals $\partial_u\varphi(tv)/i$. Thus $[v]\in\mathcal D(\mu)$ exactly
when $(I-vv^\top)\nabla\varphi(tv)=0$ for every real $t$.
Continuity of the gradient proves closedness.

Take as the reference density $p$ a scalar rescaling to identity
covariance of the normalized radial density
$\exp(-(1+|x|^2)^{1/(2\sigma)})$. Radial integration against
$e^{-c|x|^{1/\sigma}}$ and Stirling's formula give the stated moment
bound. Fix $1<s<\sigma$. Take a countable basis of projective balls
$O_j$, with the property that each open set is covered by those
$O_j$ whose closures it contains. On the sphere write
$O_j=\{v:(a_j^\top v)^2>b_j\}$ with $|a_j|=1$ and $0<b_j<1$.
Let $\eta(t)=\exp(-t^{-1/(s-1)})$ for $t>0$ and $\eta(t)=0$ otherwise.
The smooth even cutoff $\chi_j(v)=\eta((a_j^\top v)^2-b_j)$ is positive
exactly on $O_j$. On the unit sphere put
\[
 h_0(v)=1,\qquad h_1(v)=\sum_{\ell=1}^n\ell v_\ell^2,\qquad
 h_2(v)=\Big(\sum_{\ell=1}^n v_\ell\Big)^2 .
\]
The spherical gradients of $h_1,h_2$ are never simultaneously zero:
the critical lines of the first are the coordinate axes, and none is a
critical line of the second.

Assign pairwise disjoint open radial intervals inside $(1,2)$ to the
pairs $(j,l)$, $l=0,1,2$. In each interval choose a nonzero nonnegative
smooth cutoff $\rho_{jl}$ compactly supported in the interval, using
products of translates of $\eta$. Define
\[
 \psi_{jl}(\xi)=
 \rho_{jl}(|\xi|)\chi_j(\xi/|\xi|)h_l(\xi/|\xi|),
 \qquad H_{jl}=\mathcal F^{-1}\psi_{jl}
\]
with zero extension off the annulus and with the characteristic-function
Fourier convention. Each $H_{jl}$ is real, even, and satisfies
$|H_{jl}(x)|\le C_{jl}e^{-c_{jl}|x|^{1/s}}$. Here are the derivative
bounds behind this estimate. Each cutoff on its compact annular support
is a product of analytic factors and factors $\eta(g(\xi))$ with $g$
analytic on that annulus. Where $u=g(\xi)>0$ is small, complexify
one coordinate on a circle of radius $\delta u$. For fixed sufficiently
small $\delta$, $g$ stays within a small disc about $u$, and the modulus
of its exponential is at most $e^{-c u^{-1/(s-1)}}$. Cauchy's formula
and maximization in $u$ bound every order-$k$ coordinate derivative by
$C^{k+1}(k!)^s$. The same bounds away from the boundary follow by
analyticity, and products preserve them by the Leibniz rule. Integration
by parts $k$ times in a coordinate with $|x_i|\ge |x|/\sqrt n$,
followed by the choice $k=\lfloor c'|x|^{1/s}\rfloor$, proves the decay.
In particular $H_{jl}/p$ is bounded.

Choose positive coefficients $c_{jl}$ so small that
\[
 \sum_{j,l}c_{jl}
 \big(\|H_{jl}/p\|_\infty+\|\psi_{jl}\|_{C^1}\big)<1/2 .
\]
For a closed set $A$ retain all pairs with
$\overline O_j\cap A=\varnothing$ and use the density
$p_A=p+\sum_{\rm retained} c_{jl}H_{jl}$. It lies between $p/2$ and
$3p/2$. Its perturbation is the inverse Fourier transform of an
integrable function supported in $1\le|\xi|\le2$, hence extends to
an entire function of the spatial variables. Thus $p_A$ is real analytic.
The Fourier transforms agree near zero, so the density has mass one and
all its mixed moments equal those of $p$.

On lines in $A$, each retained angular cutoff and its first derivatives
vanish. Convergence in $C^1$ gives equality of the projected laws and
vanishing of every transverse derivative. For a line $[v]\notin A$
choose a retained $j$ with $\chi_j(v)>0$. In its three separate radial
intervals the reference is radial and there is just one summand per
interval. Equality of projected laws fails already in the $l=0$ interval.
If all transverse derivatives vanished, the three spherical gradients
$\nabla_S(\chi_j h_l)(v)$ would vanish. The $l=0$ equation gives
$\nabla_S\chi_j(v)=0$, and the other two would give
$\nabla_S h_1(v)=\nabla_S h_2(v)=0$, which is impossible.
\end{proof}

Taking $A$ to be the coordinate axes or all of projective space gives
two pairwise mean independent source laws with identical mixed cumulants
at every order and identical one-dimensional coordinate marginals.
Only the first has an identifiable basis. At factorial exponent one
the moment bound implies a local exponential moment, and
Theorem~\ref{p53:real} gives algebraicity and finite-order determination
of the direction set for each fixed law. Cuesta-Albertos, Fraiman and
Ransford constructed pairs with equal projections on prescribed cones
\cite[Theorems~3.5--3.6]{reg56:CFR}; here a single reference fixes every
mixed moment, and the equal-projection and regression sets coincide.

\subsection{A characterization of the Gaussian radius}
\begin{theorem}[Two homogeneous observations determine the Gaussian radius]
\label{r65:radius}
Let $X=RU$ be spherical in $\mathbb R^n$, where $U$ is uniform on the
sphere and independent of $R\ge0$. On two orthogonal subspaces fix
measurable real observations $F,H$ of positive homogeneity degrees.
Each is finite almost everywhere and nonzero on a set of positive
Lebesgue measure; evaluate it on the corresponding projection of $X$,
with value zero at the origin. Then $|F(X)|$ and $|H(X)|$ are independent
if and only if $X$ is a scalar multiple of a standard Gaussian, including
the zero multiple. In particular the same characterization holds after
whitening a nonsingular elliptical input, for any fixed pair of
nonzero homogeneous polynomials on orthogonal whitened blocks.
The sample mean and residual norm give a classical case, including for
spherical laws without a density \cite{r65:Ali,r65:ArnoldLynch}.
\end{theorem}
\begin{proof}
Take positive roots of $|F|,|H|$ to make their degrees one, calling the
result $A,B$. For a standard Gaussian $G=CU$, $C$ has the $\chi_n$ law
independently of $U$, and $A(G),B(G)$ are independent. As $C>0$, the
angular events $A(U)>0,B(U)>0$ are independent with positive probabilities.
For $X$ the joint positive probability is multiplied by $p=\Pr(R>0)$
and the product of its marginals by $p^2$. Independence gives $p=0$
or $p=1$.

In the latter case write $\phi_T(t)=\mathbb E T^{it}$ for $T=R,C$.
The angular transforms $a(t)=\mathbb E[\mathbf1_{A(U)>0} A(U)^{it}]$
and $b(t)$ defined similarly for $B$ are nonzero near zero.
The joint transform on the two positive events for $TU$ equals the
same angular joint transform times $\phi_T(t+s)$. Independence for
both $T=R$ and $T=C$ therefore gives, by cancellation near zero,
\[
 h(t+s)=h(t)h(s),\qquad h=\phi_R/\phi_C.
\]
Continuity, $h(0)=1$ and $h(-t)=\overline{h(t)}$ give
$h(t)=e^{ict}$ near zero for some real $c$.

Thus $V=\log R$ and $V_0=c+\log C$ have characteristic functions
agreeing near zero. For every $j\ge1$ the expectations of
$(2\sin(tV/2)/t)^{2j}$ and its counterpart for $V_0$ agree for
sufficiently small $t$, by expanding in exponentials. Fatou's lemma
first gives $\mathbb E|V|^{2j}\le\mathbb E|V_0|^{2j}<\infty$.
Differentiating the common characteristic function at zero now
gives equality of all moments. Since $C$ has finite powers of both
signs near zero, $V_0$ has an exponential moment and is moment
determinate. Consequently $R\stackrel{\rm law}=e^c C$.
Gaussian orthogonal independence proves sufficiency.
\end{proof}

\subsection{A limitation of diagonal cumulant tests}

The cubic and quartic implications in \cite[Theorem~4.1]{p53:RSZ}
are stated for real symmetric tensors in arbitrary dimension.
Their binary calculations do not extend to dimension three.

\begin{proposition}\label{r58:pmidiag}
In dimension three a cubic with three nonzero diagonal entries, and a
quartic with three distinct diagonal entries, can each have more than
one orthogonal eigenbasis up to signs and permutation. Each example
occurs as a cumulant tensor of a centered isotropic distribution with a
strictly positive real analytic, strongly log-concave density. For each
such distribution the exhibited frames also give pairwise mean independent
coordinates.
\end{proposition}
\begin{proof}
For a homogeneous form $P$, a unit vector $v$ is an eigenvector of its
symmetric tensor precisely when $\nabla P(v)$ is parallel to $v$.
Write $s=x+y+z$ and $r=x^2+y^2+z^2$. Consider the forms
\[
 P_3=x^3+y^3+z^3-3xyz=\tfrac12s(3r-s^2),\qquad
 P_4=(x^2+2y^2-3z^2)^2+r^2.
\]
The coordinate axes are eigenlines of both forms. The cubic takes value
one on every coordinate unit vector and is invariant under all rotations
fixing $(1,1,1)$. Rotating the coordinate frame about this axis produces
infinitely many eigenbases. The quartic takes the respective values
$2,5,10$. For $D=\operatorname{diag}(1,2,-3)$ choose
$v=(1,1,1)/\sqrt3$. Then $v^\top Dv=0$ and the compression of $D$
to $v^\perp$ has trace zero. The sum and difference of its two
orthonormal eigenvectors, divided by $\sqrt2$, are orthogonal null
directions of the quadratic form. Together with $v$ they form a second
orthonormal basis. Every unit null vector $w$ of $D$ satisfies
$\nabla P_4(w)=4w$.

To realize either form $P=P_d$, let $\phi_a$ denote the density of
$N(0,aI_3)$ with $0<a<1$. For sufficiently small real $\epsilon\ne0$,
\[
 p_\epsilon=\phi_1+\epsilon P(-\partial)\phi_a,\qquad
 M_\epsilon(t)=e^{\|t\|^2/2}+\epsilon P(t)e^{a\|t\|^2/2}.
\]
The ratio $P(-\partial)\phi_a/\phi_1$ and its first two derivatives
are bounded, so $p_\epsilon$ is positive and strongly log-concave.
The displayed Laplace transform gives normalization, centering and
identity covariance, and the degree-$d$ term of $\log M_\epsilon$
is $\epsilon P$. For each $u\perp v$ the transverse derivative at
$tv$ is $\epsilon t^{d-1}e^{at^2/2}u\cdot\nabla P(v)$ when $\|v\|=1$.
Uniqueness of the bilateral Laplace transform of the signed measure
$\E[(u\cdot X)\mathbf1_{\{v\cdot X\in\,\cdot\,\}}]$ gives
$\E[u\cdot X\mid v\cdot X]=0$ on all exhibited pairs of perpendicular
frame vectors.
\end{proof}
These distributions lie in the exceptional locus of the generic
identifiability theorem. The examples concern the deterministic
diagonal criteria in the cited preprint.

\begin{theorem}[The seven exceptional ternary cubic lines]\label{r68:ternary}
Write $V_{3,3}$ for the ternary cubics having the coordinate eigenframe.
Let $\mathcal E$ consist of the three lines spanned by
$x_1^3,x_2^3,x_3^3$ and the four lines obtained from
$H=x_1^3+x_2^3+x_3^3-3x_1x_2x_3$ by coordinate sign changes.
A tuple of $m\ge1$ ternary cubics with the coordinate eigenframe
has a second common orthogonal eigenframe exactly when all its forms
belong to a single $L\in\mathcal E$. Its exceptional set in $V_{3,3}^m$
is the union of the seven linear subspaces $L^m$, of codimension $3m$.
\end{theorem}
\begin{proof}
With symmetric coefficient tensor $T$, every such cubic has the Hesse
form $P=\sum_i a_i x_i^3+6t x_1x_2x_3$
\cite[Theorem~5.4]{r68:hesse}. In every orthonormal eigenframe the trace
vector $g_i=\sum_j T_{ijj}$ and the contraction
$M_{ij}=\sum_{h,l}T_{ihl}T_{jhl}$ satisfy
\[
 g=(a_i)_i,\qquad M=\operatorname{diag}(a_i^2+2t^2),\qquad
 6t^2=\operatorname{tr}M-|g|^2.
\]
Thus $t^2$ is frame invariant. Each new frame diagonalizes $M$,
and the coordinates of $g$ in a block of equal $|a_i|$ again
have that same absolute value. In a two-dimensional block of
nonzero magnitude this forces a signed permutation: an orthogonal
map taking a square vertex $(\pm a,\pm a)$ to another is a symmetry
of the square. In a zero block of size two the remaining coordinate
axis is retained. If $t\ne0$, subtracting its pure cube from $P$
leaves a multiple of $x_1x_2x_3$ in both frames. Its three linear
factors force the same coordinate lines. If $t=0$, the pure cube
allows every rotation of the complementary plane. When all $a_i=0$,
the same factorization gives uniqueness unless $P=0$.

It remains to consider $|a_1|=|a_2|=|a_3|=a>0$. Orient both frames
so that all three entries of $g$ are $a$; their transition fixes
$u=(1,1,1)$. Evaluating $P$ at $u$ shows that $t$ is also unchanged
with these orientations. Now $P=aH+(6t+3a)x_1x_2x_3$ and
$H=\tfrac12(u\cdot x)(3|x|^2-(u\cdot x)^2)$.
Thus every transition preserves $H$. If $6t+3a\ne0$ it must preserve
the three linear factors of $x_1x_2x_3$, giving just the coordinate
frame. Otherwise every rotation fixing $u$ gives an eigenframe.
This proves the single-form classification. In particular
$x_1x_2x_3$ has a unique eigenframe despite its three zero diagonal
entries in the criterion of \cite[Theorem~4.1]{p53:RSZ}.

For a nonzero pure cube every eigenframe retains its coordinate axis.
For a signed copy of $H$ no new eigenframe retains any coordinate
axis: in that case the two-dimensional square argument forces the
other two axes also to be retained. Hence two different pure cube
lines, or one pure cube and one signed $H$, permit only the original
common frame. For two distinct signed $H$ lines normalize their trace
vectors, up to scalar, to sign vectors $u,v$ differing in just one
coordinate. In any common frame these remain sign vectors with inner
product one. Their difference is consequently twice a coordinate
vector in both frames. This would retain an axis, again forcing the
original frame. The converse follows from the rotations already
constructed. Since $\dim V_{3,3}=4$ the dimension assertion follows.
\end{proof}

\begin{proposition}[Eigenframes of Rayleigh quotient profiles]\label{r68:rayleighframe}
Let $D$ be real symmetric with an orthonormal eigenframe $E$ and
eigenvalues $\lambda_1,\ldots,\lambda_n$. Suppose the restrictions of
homogeneous forms $P_j$ to the unit sphere are $f_j(x^\top D x)$ for
differentiable $f_j$. Put $C=\{s\in[\min\lambda_i,\max\lambda_i]:
f_j'(s)=0\ \text{for every }j\}$. A second common orthogonal eigenframe
exists exactly when $D$ has a repeated eigenvalue or there are a subset
$I$ of at least two indices and $c\in C^{|I|}$ majorized by
$(\lambda_i)_{i\in I}$ with a different multiset of entries.
In particular for the quartic
\[
 \alpha(x^\top D x)^2+\beta|x|^2 x^\top D x+\gamma|x|^4,\qquad
 \alpha\ne0,
\]
with simple spectrum the criterion is that some subset of at least
two eigenvalues has arithmetic mean $-\beta/(2\alpha)$.
\end{proposition}
\begin{proof}
Writing $q=x^\top D x$, its spherical gradient is $2(Dx-qx)$.
Thus any common eigenline is a $D$-eigenline or has $q\in C$.
A repeated eigenspace admits a rotation. For simple spectrum, remove
from a second common frame the coordinate eigenlines it retains.
On the remaining spectral subspace the new diagonal of $D$ belongs
to $C^{|I|}$ and is majorized by the eigenvalues. Its entries cannot
be a permutation of the eigenvalues: the squared eigenvalue sum
exceeds the squared diagonal sum by the sum of squares of all
off-diagonal entries, and the new basis does not diagonalize this
simple-spectrum compression. Conversely the real Schur--Horn theorem
realizes every diagonal satisfying the majorization; a different
multiset provides a new common frame. For the quartic, $C$ has at
most the single critical level $-\beta/(2\alpha)$. A constant
diagonal is majorized exactly when it has the prescribed mean.
\end{proof}

\begin{theorem}[A cubic anchor for ternary quartics]\label{r69:cubicanchor}
Suppose a collection of ternary cubics and quartics has the coordinate
eigenframe and at least one cubic is nonzero. Its common frames, viewed
as unordered sets of lines, number one or two, or form a full circle.
The following rules determine which possibility occurs.

If the cubics do not all lie on one line of
Theorem~\ref{r68:ternary}, the frame is unique. On a pure-cube line,
write $x$ for its axis and $y,z$ for the remaining coordinates. For each
quartic $P$, write
\[
 P(0,y,z)=a y^4+b y^2z^2+c z^4,\qquad
 \partial_xP(0,y,z)=yz(dy+ez).
\]
A second frame exists precisely when $a=c$ and $d=e=0$ for every
quartic. It is the balanced rotation in the $yz$ plane, and is the
only additional frame unless $b=2a$ for every quartic. In the latter
case all rotations of the plane work.

On a signed Hesse line, change coordinate signs to the convention of
Theorem~\ref{r68:ternary}. Put $u=(1,1,1)/\sqrt3$, $h=1/\sqrt3$, and
$v(\theta)=R_u(\theta)e_1$ for rotation about $u$. For each quartic take
the two trigonometric polynomials
\[
 (u\times v)\cdot\nabla P(v),\qquad (u-hv)\cdot\nabla P(v).
\]
For each such polynomial $T(\theta)=\sum_{j=-4}^4 c_j e^{ij\theta}$
and each $r\in\{0,1,2\}$ take the least $j_r\in[-4,4]$ congruent to
$r$ modulo three. Form
\[
 L_{T,r}(w)=\frac{\sum_{j\equiv r\ (3)} c_j w^{(j-j_r)/3}}{w-1}.
\]
These are polynomials of degree at most one. If all vanish identically,
every rotation about $u$ works. Otherwise there is exactly one extra
frame if all have a common unit-modulus zero $w\ne1$, and none if they
do not. Every prescribed $w$ on the unit circle other than one occurs
as the unique extra phase for a suitable single quartic.
For $p\ge1$ cubics, at least one nonzero, and $m\ge1$ quartics, the
exceptional locus in $V_{3,3}^p\times V_{3,4}^m$ on this open set has
codimension $3p+3m$. Conditional on a fixed nonzero Hesse cubic, the
exceptional locus of the quartics has codimension $5m-1$.
It is not closed in the ordinary coefficient topology.
\end{theorem}
\begin{proof}
Theorem~\ref{r68:ternary} handles the cubics alone. For a pure cube
every frame retains its axis. At a new direction $(y,z)$ and its
perpendicular $(-z,y)$ both transverse derivatives vanish precisely
when $d=e=0$. On the plane circle the quartic is
$A+B\cos2\theta+C\cos4\theta$, where $B=(a-c)/2$ and
$C=(a+c-b)/8$. The derivatives at the two perpendicular lines give
$B\sin2\theta=C\sin4\theta=0$. This proves the first rule.

For the Hesse cubic orient each frame vector by its positive inner
product with $u$. The proof of Theorem~\ref{r68:ternary} gives this
inner product as $h$. Thus the oriented vectors are the three
directions $v(\theta+2\pi k/3)$, $k=0,1,2$. The parameter
$w=e^{3i\theta}$ determines their unordered frame bijectively.
The two displayed tangents span $v^\perp$, and each derivative has
frequencies at most four. Discrete Fourier inversion at the three
vertices gives the three residue-class sums in the numerator. Each
vanishes at $w=1$ because the coordinate frame is critical. There
are only three indices in each class, proving the affine test.

The quartics critical along the entire latitude form the plane
\[
  \operatorname{span}\{|x|^4,\ ((u\cdot x)^2-h^2|x|^2)^2\}.
\]
Indeed their restriction is constant on that circle. Subtracting its
value times $|x|^4$ gives a quartic whose value and gradient vanish on
the full cone $q(x)=(u\cdot x)^2-h^2|x|^2=0$ by homogeneity and
Euler's identity. In coordinates with polar axis $u$, division by
$q$ leaves a remainder affine in the polar coordinate. Its values on
the two cone sheets force both coefficients to vanish. The gradient
identity then makes the quotient vanish on the same cone, proving
divisibility by $q^2$. Two specified frames impose at most twelve
linear conditions on the fifteen-dimensional space of ternary
quartics, so their common solution space contains a quartic outside
this plane. The affine test gives exactly the prescribed two frames.
More generally, on this Hesse circle a form of degree $d$ whose
tangential tests do not vanish identically allows at most
$\lfloor2d/3\rfloor$ frames by the same grouping of frequencies.

For the Hesse codimension use longitude and polar height $z$; the two
derivative tests are $\partial_\theta P$ and $(1-h^2)\partial_z P$
at $z=h$. Write $U=X+iY$ on the perpendicular plane.
For each frequency $j=1,2$ its complex value and height derivative
can be prescribed independently: the two quartic monomials
$U^j z^{4-j}$ and $U^j z^{2-j}|U|^2$ have ratio $(1-z^2)/z^2$
on the sphere, with nonzero derivative. Thus in the residue class
$-2,1,4$ the first two coefficient vectors of the pair of tangential
tests are independently arbitrary in $\mathbb C^2$; the conjugate
class contributes no further constraints. After imposing the first
frame the second imposes exactly two new complex conditions, since
the coefficients of those two vectors have different ratios at $1$
and $w\ne1$. In class zero the longitude derivative has no constant
term; frequency three is generated by $zU^3$, whose amplitude
$z(1-z^2)^{3/2}$ has nonzero value and derivative at $h$.
At phase one the amplitude coefficient is real, and the constant
height derivative is fixed by its real part. At the second phase the
height derivative forces this real coefficient to zero since
$\operatorname{Re}w\ne1$; this is exactly one further real condition.
The rank at every second phase is therefore five. Varying that phase
gains one dimension, since outside the entire-latitude plane each
fiber of the projection has only one phase. This gives $5m-1$.
To see nonclosedness fix a nonzero frequency-four coefficient vector
$D$ of the two tests. In the class $-2,1,4$ prescribe the three
vectors $(wD,-(1+w)D,D)$ for unit phases $w\to1$, $w\ne1$; take
the conjugate coefficients and set the zero class to zero. The
independence just proved realizes these by real quartics continuously
in $w$. Their folded polynomial is $D(s-1)(s-w)$, so the limit
has only the coordinate frame. Repeating the form proves the assertion
for all $m$.
On a pure-cube line the three displayed coefficient conditions per
quartic give $3m$, while the cubics on either line cost $3p$.

\end{proof}

\begin{theorem}[Ternary sign-invariant quartics]\label{r69:signquartic}
For arbitrary symmetric $3\times3$ matrices $A_1,\ldots,A_m$,
$m\ge1$, let $P_h(x)=s^\top A_h s$, $s_i=x_i^2$. These are all
ternary quartics invariant under the individual coordinate sign changes.
Write $\mathbf1=(1,1,1)^\top$ and define the affine set
\[
 K=\{s:\ \mathbf1^\top s=1,\ A_h s\in\mathbb R\mathbf1
                       \text{ for every } h\}.
\]
There is a second common eigenframe precisely when at least one of the
following coefficient tests holds.
\begin{enumerate}
\item Two coordinate labels $((A_h)_{ii})_{h=1}^m$ agree.
\item For some $i<j$ there is a point $p=t e_i+(1-t)e_j$, $0<t<1$,
satisfying $(e_i-e_j)^\top A_h p=0$ for every $h$, with
$(\mathbf1-p)/2\in K$.
\item $\mathbf1/3\in K$ and $\dim K\ge1$.
\end{enumerate}
When the labels are pairwise distinct the edge equations in the second
test give at most one candidate per edge. In the space of $m$ such
quartics the exceptional locus is closed and has codimension $m$. On the open set
of pairwise distinct labels its codimension is exactly $3m-1$.
\end{theorem}
\begin{proof}
At a unit vector with squared coordinates $s$, the gradient has
entries $4x_i(A_hs)_i$. Criticality therefore requires equality of the
entries of $A_hs$ on the support. If a frame retains one axis, the
other two squared columns on their edge are $p$ and $e_i+e_j-p$.
Adding their edge conditions gives identical coordinate labels.
Conversely identical labels allow the balanced rotation on that edge.

A frame without a retained axis can have at most one column with a
zero entry. Two supports of size two meeting in just one coordinate
cannot be orthogonal, and two columns on the same edge retain the
complementary axis. If one column has squared coordinates $p$ on an
edge, the other two are dense and their squared columns belong to $K$.
Their average is $(\mathbf1-p)/2$. Conversely take a unit vector with
squared coordinates $p$, its perpendicular $b$ on that edge, and the
remaining axis $e_k$. The basis consisting of the first vector and
$(b\pm e_k)/\sqrt2$ realizes the second test.

If all columns are dense, their squares belong to $K$ and average
to $\mathbf1/3$. They cannot all equal this average: two sign vectors
in three coordinates have nonzero inner product. Thus $\dim K\ge1$.
Conversely take a line in $K$ through $\mathbf1/3$, given in
$\mathbf1^\top s=1$ by $d\cdot s=0$ for $0\ne d\perp\mathbf1$.
Real Schur--Horn gives a frame in which $\operatorname{diag}(d)$ has
zero diagonal. It differs from the coordinate frame and its squared
columns all lie on the chosen line. Finally, an identically zero
edge equation would give $(A_h)_{ii}=(A_h)_{ij}=(A_h)_{jj}$ at both
endpoints, explaining uniqueness of an edge candidate with distinct labels.
The first and third tests are closed conditions by linear algebra. At an
endpoint limit of the second test, the midpoint of the opposite edge must
belong to $K$, which forces the labels at its two vertices to agree.
Compactness of the closed edges proves closedness of the union.

For dimension, restrict $s^\top A_hs$ to $\mathbf1^\top s=1$: these
restrictions are arbitrary bivariate quadratics, by homogenization.
For fixed $p$ in the second test both first derivatives at
$s_*=(\mathbf1-p)/2$ must vanish. The edge condition then asks the
Hessian to pair $e_i-e_j$ with $p-s_*$ to zero, one additional
independent linear condition since neither vector is zero.
Varying $t$ gives codimension at least $3m-1$. This bound is attained
with distinct labels. For example, with $i=1,j=2,t=1/4$, in the
plane coordinates $a=s_1-3/8$, $b=s_2-1/8$ take
$7a^2+4ab+b^2$ in every slot. Its vertex values are
$39/16,7/16,19/16$ and its edge derivative vanishes at $p$.
These strict inequalities persist on an open part of the incidence,
where projection has finite fibers. In the third test stationarity
at the barycenter costs $2m$ linear conditions; for a fixed line
direction each of the $m$ symmetric two-dimensional Hessians must
annihilate that direction, costing two further conditions per form.
Varying the direction gains at most one dimension. The first test
is a finite union of subspaces of codimension $m$, giving both claims.
\end{proof}

\subsection{A finite certificate for common eigenframes}

There is a single geometric obstruction in the two questions in
\cite[Discussion]{p53:RSZ}: two distinct orthogonal frames can survive
the same equations. For real orthogonal matrices $Q,R$ with columns
$q_i,r_i$ set $D(Q,R)=n-\sum_{i,j}(q_i^\top r_j)^4$.
The squared entries of each column of $Q^\top R$ sum to one, so
$D\ge0$, with equality exactly at a signed permutation.

For a finite family $\mathcal P$ of real homogeneous polynomials,
let $\mathcal F(\mathcal P)$ consist of the entries of
$Q^\top Q-I,R^\top R-I$ and of
\[
   v\wedge\nabla P(v),\qquad
   P\in\mathcal P,\quad v=q_1,\ldots,q_n,r_1,\ldots,r_n.
\]
These are polynomials in the entries of $Q,R$. With one additional
real variable $y$, consider the polynomial identity
\begin{equation}\label{r60:certificate}
 -1=\sum_j h_j^2+\sum_{f\in\mathcal F(\mathcal P)} a_f f
                         +b(yD(Q,R)-1).
\end{equation}
All coefficients $a_f,b,h_j$ on the right are real polynomials.

\begin{theorem}[A certificate for common eigenframes]\label{r60:frames}
Suppose the tensors in $\mathcal P$ admit a common orthogonal
eigenbasis. It is unique up to signs and permutation precisely when
\eqref{r60:certificate} admits a finite certificate. This is a finite
semialgebraic test on the input coefficients. If one frame is supplied,
fix $R$ to that frame throughout.
\end{theorem}
\begin{proof}
A real common zero of $\mathcal F(\mathcal P)$ is a pair of eigenframes.
The additional equation $yD=1$ selects precisely the inequivalent
pairs. Emptiness of the real zero set of these polynomial equations is
equivalent to \eqref{r60:certificate} by the real Nullstellensatz
\cite{r59:BPR}. Real quantifier elimination gives a finite coefficient
test, including when the tensor has infinitely many eigenframes.
It can also compute the maximum of $D$: eliminate $Q,R$ from
$\mathcal F(\mathcal P)=0$ and $D(Q,R)\ge s$. Compactness of the frame
set makes this interval closed at its upper endpoint.
\end{proof}

\subsection{An exact coefficient test for uniqueness of an eigenbasis}

Ribot, Seigal and Zwiernik ask when a symmetric tensor admitting an
orthogonal eigenbasis has just one such basis up to signs and permutation
\cite[Remark~4.2]{p53:RSZ}. The following test uses two
finite real symmetric matrices constructed from the tensor coefficients.
It applies also when the tensor has a continuum of real eigenlines.

Let $P$ be its homogeneous polynomial of degree $d\ge2$ in $n$ variables.
For real $n\times n$ matrices $Q,R$ with columns $q_i,r_i$, set
\[
 E(Q)=\|Q^\top Q-I\|_{\rm F}^2+
         \sum_i\|q_i\wedge\nabla P(q_i)\|^2,\qquad
 D(Q,R)=n-\sum_{i,j}(q_i^\top r_j)^4.
\]
The squared wedge norm is the sum of squares of the two-coordinate
minors. Collect all $2n^2$ matrix entries into $x$, put
$F=E(Q)+E(R)$, and let $m=2\max(d,4)+2$. Define
\[
 G_e(x)=F(x)-eD(Q,R)+e^2\sum_\ell x_\ell^m,\qquad
 g=G_e+es .
\]
In the quotient by all partial derivatives of $G_e$ use the monomial
basis $\mathcal B=\{x^\alpha:0\le\alpha_\ell<m-1\}$. For a polynomial
$w$ let $L_w$ denote multiplication modulo the partial derivatives, and
form the symmetric matrix
\[
 H_w=\big(\operatorname{tr} L_{wuv}\big)_{u,v\in\mathcal B}.
\]
The reductions replace $x_\ell^{m-1}$ by
$-\partial_\ell(F-eD)/(me^2)$. They decrease total degree.

For a symmetric matrix with entries rational in $e,s$, write
$\sigma_*(H)$ for its signature at sufficiently small $s>0$, with
$e>0$ sufficiently small for that $s$. This is a finite coefficient
operation. Compute the characteristic polynomial, and its reflection
in the origin. For each coefficient clear positive denominators and
use its lowest nonzero power of $e$, then the lowest nonzero power of
$s$ in that coefficient. The difference of the two sign-variation
counts, omitting zeros, is $\sigma_*(H)$ by real-rootedness. Identically
zero coefficients are omitted throughout.

\begin{theorem}[Orthogonal eigenbasis criterion]\label{r59:pmicert}
Suppose a real symmetric tensor admits an orthogonal eigenbasis.
It has a unique basis of eigenlines if and only if
\[
                       \sigma_*(H_{g^2})=\sigma_*(H_g).
\]
If an eigenframe $R_0$ has been supplied, one may fix $R=R_0$, take
$F=E(Q)$ and $m=2d+2$, and use only the $n^2$ entries of $Q$.
\end{theorem}
\begin{proof}
The zeros of $E$ are precisely the orthogonal eigenframes. On such
frames $D\ge0$, and equality means that $Q^\top R$ is a signed
permutation. Indeed the squares of the entries in each column sum
to one, and their squares sum to one precisely for a coordinate
column. Let $\delta$ be the maximum of $D$ on the compact, nonempty
zero set of $F$. In the fixed-frame version maximize over $Q$ with
$E(Q)=0$. Either maximum is zero precisely in the uniqueness case.

For $e>0$ the polynomial $G_e$ attains its minimum, and
\[
                   \lim_{e\downarrow0}-\frac{\min_x G_e(x)}e=\delta .
\]
Evaluation at a maximizing zero gives the lower bound in the limit.
For a minimizer $x_e$, that same comparison and $D\le n$ everywhere
give $F(x_e)=O(e)$. The Gram terms then bound $x_e$ uniformly.
Every limit point is a zero of $F$, and $-G_e(x_e)/e\le D(x_e)$,
which gives the upper bound.

In every degree-compatible monomial order the partial derivatives
have leading monomials $x_\ell^{m-1}$, since $\deg F,\deg D<m$.
These monomials are pairwise relatively prime, so the partial
derivatives are a Gr\"obner basis and $\mathcal B$ is a basis of their
finite quotient. By the multivariate Hermite theorem
\cite[Theorem~4.102]{r59:BPR}, half the difference of the signatures
of $H_{g^2}$ and $H_g$ counts the distinct real critical points with
$G_e<-es$, including for a nonradical critical ideal. The count is
positive exactly when the global minimum satisfies that strict
inequality. If $\delta=0$, for each $s>0$ the count eventually vanishes
as $e\downarrow0$. If $\delta>0$, it is positive for each $0<s<\delta$
and all sufficiently small $e$. The finite coefficient rule gives
precisely these eventual signatures. Since $D(Q,R_0)$ has degree four
in $Q$, the smaller value of $m$ suffices in the supplied-frame case.
\end{proof}

The same matrices determine the largest separation $\delta$ itself.
For fixed $s$ use only the leading signs in $e$ in the two
characteristic polynomials. The resulting count is positive on
$0<s<\delta$ and zero on $s>\delta$. Its changes can occur only at
zeros of the finitely many nonzero coefficient polynomials in $s$
after clearing powers of $e$. Univariate real-root isolation and sign
tests on the complementary intervals therefore recover $\delta$.

In dimension two with the coordinate eigenframe supplied there is
a smaller test. Write $f(t)=P(1,t)$ and
$B(t)=(1+t^2)f'(t)-dtf(t)$. A second basis exists exactly when
$B(t)$ and $t^dB(-1/t)$ have a common nonzero real root. Indeed the
slopes of the new lines are $t,-1/t$, and $B$ is the tangential
derivative on the unit circle up to a nonzero factor. If $B$ vanishes
identically, every orthogonal frame works. Otherwise a greatest common
divisor and a Sturm count away from zero suffice.

Even full coordinate-pair restrictions can lose essential information.
For every odd $d\ge3$, multiply both
$x^3+y^3+z^3$ and $x^3+y^3+z^3-3xyz$ by
$(x^2+y^2+z^2)^{(d-3)/2}$. Their restrictions to all three coordinate
planes agree. Multiplication by this radial factor preserves spherical
critical points. For the first cubic every eigenvector has equal
nonzero entries on its support, so perpendicular eigenvectors have
disjoint supports. Its only orthogonal eigenbasis is the coordinate
one. The second cubic has the continuum of bases from
Proposition~\ref{r58:pmidiag}. In particular this pair also tests the
coordinate-pair claim for quintics in \cite[Remark~4.3]{p53:RSZ}.
There is also an obstruction within the binary problem itself.

\begin{proposition}[Real exceptional sets]\label{r59:pmireal}
For every $d\ge6$, among binary real forms of degree $d$ with the
coordinate eigenframe, the locus admitting a second eigenbasis is
not closed in the ordinary coefficient topology. In particular, no
finite union of real polynomial zero sets describes this locus.
\end{proposition}
\begin{proof}
On the unit circle take $p_0(\theta)=\cos(2\theta)$ for even $d$ and
$p_0(\theta)=\cos^3\theta+\sin^3\theta$ for odd $d$. Both have only
the coordinate eigenbasis. On an affine pencil specify the tangential
derivative by
\[
                 p_s'(\theta)=(1+s\cos(4\theta))p_0'(\theta).
\]
For even $d$ this is a combination of sines of frequencies two and six.
For odd $d$ it contains only odd frequencies at most seven. It therefore
has a zero-mean trigonometric primitive depending affinely on $s$,
with frequencies at most $d$ of the same parity as $d$. Each such
primitive extends to a homogeneous form of degree $d$: use the real
and imaginary parts of $(x+iy)^k(x^2+y^2)^{(d-k)/2}$ for frequency $k$.
The coordinate vectors remain critical. When $|s|<1$ all critical
directions are unchanged. At $s=-1$ the new factor vanishes only on
the coordinate lines and the orthogonal eigenbasis is still unique.
For $s<-1$ a solution of $\cos(4\theta)=-1/s$ and its perpendicular
give a second eigenbasis. Taking $s\uparrow-1$ proves the assertion.
\end{proof}
Thus the nonvanishing conditions suggested in \cite[Remark~4.2]{p53:RSZ}
cannot characterize real uniqueness in all orders, even allowing
arbitrary polynomial degrees. The semialgebraic criterion detects the
real locus itself. Equations for its Zariski closure are a distinct
algebraic problem.

\subsection{The codimension of the exceptional eigenframes}

Let $V_{n,d}$ be the real homogeneous forms of degree $d$ on
$\mathbb R^n$ with the coordinate eigenframe: all coordinate directions
are critical for the restriction to the unit sphere. Let $B_{n,d}$
consist of those forms with a second orthogonal basis of eigenlines.
Remark~4.2 of \cite{p53:RSZ} suggests $\binom n2$ irreducible
equations of degree $d-3$ for the exceptional cases at $d\ge7$.
Generic uniqueness was already proved in \cite{r61:RSZorth}.
The size of the exception depends on the ambient dimension.

\begin{theorem}[Codimension for a common eigenframe]\label{r61:pmicodim}
Let $n\ge2$, $m\ge1$, and $d_1,\ldots,d_m\ge5$, allowing repeated
orders. In $W=\prod_j V_{n,d_j}$ let $B$ consist of tuples admitting
a second common orthogonal eigenframe, as an unordered set of lines.
Write $m_r=\#\{j:d_j=r\}$. Then
\[
 \operatorname{codim}_W B
 =\operatorname{codim}_W\overline B^{\mathrm{Zar},\mathbb R}
 =2m(n-1)-\max(1,m_6).
\]
For distinct orders this is $2m(n-1)-1$. In particular the single-order
answer is $2n-3$ for every $d\ge5$, excluding a union of hypersurfaces
when $n\ge3$.
\end{theorem}
\begin{proof}
Let $k$ columns $v_a$ of a second frame be new lines. A relation among
the $k(n-1)$ extra tangent conditions is a collection $w_a\perp v_a$
such that $\sum_a\partial_{w_a}P(v_a)=0$ for all $P\in V_{n,d}$;
set $w_a=0$ at retained axes. For a new column $v$, choose nonzero
coordinates $v_i,v_j$, $i\ne j$. When $d\ge7$, put $h=x_ix_j(v\cdot x)$.
For each $z\perp v$, $h^2(v\cdot x)^{d-7}(z\cdot x)$ has tangent
gradient $h(v)^2z$ at $v$ and zero first jet on all other frame and
coordinate lines. Hence there is no relation.

For $d=5,6$ let $S$ be the support of $v$ and write $t_i=w_i/v_i$
there. For $h\notin S$, the test $x_ix_jx_h(v\cdot x)^{d-3}$
has zero first jet at other frame lines and belongs to $V_{n,d}$.
It gives $w_h=0$. Thus $w$ is supported on $S$ and
$\sum_{i\in S}v_i^2t_i=0$.

At degree six the tests $x_i^2x_j^2(v\cdot x)^2$ give $t_i+t_j=0$
for each pair in $S$. A nonzero tangent requires exactly two entries
of equal absolute value. Distinct two-element supports are disjoint
by orthogonality. If only one surviving column uses a support,
$x_i^4x_j^2$ kills its tangent. Two such columns form a balanced
binary block. After signed permutations their angles are $\pi/4$ and
$3\pi/4$. On the five allowed binary monomials of exponents
$(6,0),(4,2),(3,3),(2,4),(0,6)$ the angular derivatives at these two
angles are opposite, with a nonzero derivative for $(6,0)$.
Monomials using other coordinates have zero supported derivative.
Thus each balanced binary block contributes exactly one relation,
in agreement with \cite[Lemma~6.2]{r61:RSZorth} in dimension two.

At degree five the tests $x_ix_jx_h(v\cdot x)^2$ for triples in $S$
give $t_i+t_j+t_h=0$. On four or more entries these equations force
all $t_i=0$. Group the columns with exact support $\{1,2,3\}$, after
relabeling, and put $u_{ia}=v_{1a}v_{2a}v_{3a}t_{ia}$. For distinct
$i,j,h$ in that support,
\[
 \sum_i u_{ia}=\sum_i v_{ia}^2u_{ia}=0,\qquad
 \sum_a v_{ia}^2u_{ia}=\sum_a v_{ja}v_{ha}u_{ia}=0 .
\]
The row equations come from $x_i^3x_jx_h$ and $x_ix_j^2x_h^2$,
using the first column equation. Other supports give zero. A singleton
group vanishes by the first row equation. For a two-column group the
three products $v_{i1}v_{i2}$ have mixed signs by orthogonality.
For at least two $i$ the coefficients $v_{j1}v_{h1},v_{j2}v_{h2}$ have
opposite signs; the two row equations there are independent. The
vanishing of those rows forces the third to vanish by the column sum.

Three columns give a dense orthogonal block $U=(v_{ia})$. Form the
vectors $A_i=(v_{ia}^2)_a$, $D_i=(v_{ja}v_{ha})_a$, and
$L_i=A_i\times D_i$. Each row of $u$ is a multiple of $L_i\ne0$,
since $A_i>0$ and $D_i$ is nonzero with sum zero. These $L_i$ span
at least a plane: a common parallel vector $z\ne0$ would annihilate
all $A_i,D_i$, giving $U\operatorname{diag}(z)U^\top=0$. The column
sum therefore leaves at most one relation, requiring
$\det(L_1,L_2,L_3)=0$. Its locus has dimension at most two. At the
orthogonal matrix with rows $(6,-2,-3),(-2,3,-6),(-3,-6,-2)$ divided
by seven the determinant has absolute value $4320/117649$.
Changing a column sign covers the other component of $O(3)$ without
changing the $L_i$. Both components are connected analytic manifolds.

The remaining two-element supports are disjoint from these ternary
blocks. For a singleton on support $\{i,j\}$, $x_i^3x_j^2,x_i^2x_j^3$
give $3t_i+2t_j=2t_i+3t_j=0$. For two columns $(c,s),(-s,c)$ there,
the determinant of the two angular derivatives on these tests is
$c^4s^2(2c^2-3s^2)^2+c^2s^4(2s^2-3c^2)^2>0$. This removes all binary
quintic relations. Relations on different surviving blocks are
independent since each tangent is supported within its block.

Stratify the frames by their retained lines, the $b$ balanced binary
blocks, and the $g$ ternary blocks supporting a quintic relation.
These choices and the rank conditions define finitely many semialgebraic
strata. Set $t=k-2b-3g$, which is zero or at least two. The extra
conditions on the tuple have rank $mk(n-1)-m_6b-m_5g$.
The frame parameters have dimension at most $2g+\binom t2$,
since balanced blocks are discrete. The resulting lower bound on
image codimension is
\[
 b[2m(n-1)-m_6]+g[3m(n-1)-m_5-2]+mt(n-1)-\binom t2.
\]
A ternary term requires $n\ge3$ and is at least $2m(n-1)-1$ per
block. For $2\le t\le n$ the residual term is nondecreasing and at
least $2m(n-1)-1$. At least one term is present.

If $m_6\ge1$ a fixed balanced binary rotation attains the claimed
codimension. If $m_6=0$, vary a nonbalanced rotation of two axes
over an open interval. The linear fibers have codimension $2m(n-1)$.
In each fiber the first form can be nonconstant on the rotating unit
circle: the allowed binary forms have dimension $d_1-1$, the two
extra conditions leave at least $d_1-3\ge2$ dimensions, and the radial
subspace has dimension at most one. Extend such a form by ignoring
the other coordinates. This open part of the incidence has finite
projection fibers, since a nonconstant trigonometric polynomial
has finitely many critical points. Varying the angle gains one
dimension. Finally a real semialgebraic set and its real Zariski
closure have the same dimension.
\end{proof}

\begin{corollary}[Adjoining quadratic and cubic orders]\label{r68:mixed}
In Theorem~\ref{r61:pmicodim} keep $h\ge1$ forms of degree at least
five, including $m_6$ sextics. Adjoin $m_2$ diagonal quadratic forms
and $m_3$ cubics in $V_{n,3}$. The codimension of tuples with a second
common eigenframe, and of their real Zariski closure, is
\[
       m_2+nm_3+2h(n-1)-\max(1,m_6).
\]
\end{corollary}
\begin{proof}
Fix an alternative frame with $k$ new lines. After removing the
retained axes, let $Q$ denote its $k$-dimensional orthogonal block and
$c$ the number of connected components of the bipartite support graph
of $Q$. Each component has at least two columns. The diagonal
quadratics preserved by $Q$ have eigenvalues constant precisely on
each component: $D Q=Q D'$ equates the entries on every support edge.
This costs $k-c\ge1$ independent coordinates.

Among diagonal cubics $\sum_i a_i x_i^3$, a new eigenvector having
a nonzero coefficient in its support satisfies $a_i v_i=\lambda\ne0$
on its entire support. Two such eigenvectors cannot meet in a
coordinate and be orthogonal, since their inner product on any
overlap is $\lambda\lambda'\sum_i a_i^{-2}$. Eigenvectors with
$\lambda=0$ are supported on zero coefficients. Consequently each
nonzero $a_i$ occurs in just one frame column, which must be that
coordinate axis by row normalization. Diagonal cubics alone therefore
impose rank $k$ on the extra conditions. For each retained axis $e_j$
the independent monomials $x_i x_l x_j$, with $i<l$ both moved,
contribute solely to the derivatives along $e_j$ at the new columns.
Their rank is $k-c$: the annihilating column coefficients $z$ are
exactly those making $Q\operatorname{diag}(z)Q^\top$ diagonal.
The new cubic conditions thus have rank at least
$k+(n-k)(k-c)\ge n$. For a rotation of just two axes the ranks are
exactly one and $n$: the two moved pure cubes and the $n-2$ mixed
monomials just displayed exhaust the cubic contributions.

Apply these bounds to each incidence stratum in the proof of
Theorem~\ref{r61:pmicodim}, adding $m_2+nm_3$ to its fiber
codimension. Its sharp binary rotations still attain equality:
use the fixed balanced frame when $m_6\ge1$, and the one-parameter
nonbalanced family otherwise. In the latter case the nonconstant
restriction of one high-degree form still gives finite projection
fibers. The exceptional set is semialgebraic as before.
\end{proof}

\subsection{Identifying the basis of a distribution}
In the square model, Ribot, Seigal and Zwiernik ask which laws
have an identifiable PMI basis when all cumulants are available
\cite[Discussion]{p53:RSZ}. Their converse relating PMI to cumulants
assumes a moment generating function finite near the origin
\cite[Theorem~2.3]{p53:RSZ}. We first recover the bases directly from
the joint law, then determine when finitely many of its cumulants
certify uniqueness under marginal polynomial density. With all
moments finite alone, even the full cumulant sequence can give the
same data for laws with opposite answers.
\begin{proposition}[A test in terms of the joint law]\label{r60:charlaw}
Let $X,Y$ be independent copies of a centered isotropic random vector.
For $|v|=1$ put
\[
 A_X(v)=\E\!\left[(X^\top Y-(v^\top X)(v^\top Y))\,
                     e^{-(v^\top(X-Y))^2/2}\right].
\]
Its zero set is exactly the set of regression directions.
An orthogonal matrix gives a PMI basis precisely when every column
belongs to this zero set.
\end{proposition}
\begin{proof}
This is the standard Gaussian average over $s\in\mathbb R$ of
$\|(I-vv^\top)\E[Xe^{isv^\top X}]\|^2$. Fubini applies by the second
moment. Continuity and Fourier uniqueness for finite signed measures
give $A_X(v)=0$ exactly when $\E[(I-vv^\top)X\mid v^\top X]=0$.
Thus the zero frames of $\sum_i A_X(q_i)$ give all PMI bases, including
for moment-indeterminate laws.
\end{proof}
\begin{theorem}[Identifiability from the cumulants of a law]\label{r60:pmi-law}
Let $X$ be a centered isotropic random vector in $\R^n$, $n\ge2$,
with a PMI basis and
moments of all orders. Suppose that polynomials are dense in
$L^2(v^\top X)$ for every unit vector $v$, as holds when the moment
generating function is finite near zero. Then its PMI basis is unique
if and only if \eqref{r60:certificate} holds for
$\mathcal P=\{C_3,\ldots,C_m\}$ for some finite $m$, where $C_d$ is
its homogeneous degree-$d$ joint cumulant. Thus testing successive
finite orders eventually certifies every identifiable law in this class.
For $n\ge2$, no procedure that reads only finitely many cumulants can always decide
both answers even on strongly log-concave analytic laws.
\end{theorem}
\begin{proof}
Formal differentiation of the moment--cumulant identity on a line
shows that $v\wedge\nabla C_d(v)=0$ for every $d$ precisely when
$\mathbb E[(u^\top X)(v^\top X)^k]=0$ for every $u\perp v$ and $k\ge0$.
By the density assumption this is orthogonality of $u^\top X$ to
$L^2(v^\top X)$, exactly $[v]\in\mathcal D(X)$.
The ideal generated by the frame equations of all orders and $yD-1$
is already generated through some finite order by the Hilbert basis
theorem. The preceding criterion now applies. A terminating two-answer
procedure would read finitely many cumulants on a standard Gaussian
input. Theorem~\ref{p53:real} realizes exactly the coordinate regression
lines with a strongly log-concave analytic law agreeing at all those
orders, forcing the opposite answer on the same computation.
\end{proof}

The condition uses the joint law. In fact even access to every proper
coordinate marginal fails to replace it.

\begin{proposition}[Equal proper marginals, different identifiable bases]
\label{r60:pmi-marg}
For each $n\ge2$ and finite $N$, two centered isotropic PMI sources with
positive analytic, centrally symmetric, strongly log-concave densities
can agree on every proper coordinate marginal, have standard Gaussian
univariate marginals, and agree with the standard Gaussian in every
mixed cumulant through order $N$. One has only the coordinate PMI
basis and the other has a second orthogonal PMI basis.
\end{proposition}
\begin{proof}
Choose an orthogonal basis whose every vector has full support and
unequal absolute coordinates; a generic rotation has these properties
since each excluded equality is a proper polynomial condition.
Apply the construction of Theorem~\ref{p53:real} with cutoff $N$ to
the union of its lines and the coordinate lines. Its moment generating function is a
standard Gaussian term plus a smaller radial Gaussian times a finite
sum of homogeneous polynomials of distinct degrees. The coordinate
marginals are standard Gaussian by Corollary~\ref{reg56:gauss}.

To construct the second source add to that polynomial sum a small
nonzero multiple of
$P(z)=(z_1\cdots z_n)^2\|z\|^{2L}$,
of degree exceeding $N$ and every previous degree. As in
Theorem~\ref{p53:real}, its inverse perturbation is an even polynomial
times a narrower Gaussian, with bounded first two relative derivatives,
so positivity and strong log-concavity persist. The perturbation
vanishes identically on every coordinate hyperplane; restriction of
the moment generating function gives equality of all proper marginals.
The distinct degree forces its transverse derivative to vanish
separately on each surviving regression line. This derivative vanishes
on coordinate axes. At a full-support unit vector $v$ the spherical
gradient of $P$ vanishes exactly when $(1/v_i)_i$ is proportional to
$v$, i.e. when all $|v_i|$ agree. It therefore removes all the added
lines and leaves precisely the coordinate basis.
\end{proof}

With only finite moments, even the complete cumulant sequence can
miss identifiability: the two laws following Theorem~\ref{reg56:closed}
have the same joint moments in every order and opposite answers.

\subsection{Column-square dependence and competing mixing directions}
Write \(Q_j=a_ja_j^\top\) for the squares of nonzero, pairwise
noncollinear columns of a real mixing matrix. Say its directions
satisfy the rank-one test if the only nonzero real squares in their
span are multiples of listed squares. The question after
\cite[Proposition~6.3]{r80:WSpre} compares this test over the reals
and the complexes. Its statistical interpretation depends on which
competitors are allowed. Definition~1.3 of \cite{r80:WS} requires
independent nondegenerate sources, at most one Gaussian source, and
competitors with at most as many columns and the same number of Gaussian
sources. Its proof of Proposition~2.2 uses pairwise noncollinearity of
the competing columns. We give the answers with this last requirement
present and with arbitrary competing columns.

\begin{proposition}[Column-square dependence under the two conventions]
\label{r80:ws-real}
For \(2\le J\le\binom I2\), a real configuration passing the rank-one
test with dependent column squares exists exactly when \(J\ge8\).
For every \(I\ge2\) and \(J\ge2\) the test characterizes identifiability
against competitors with distinct nonzero column directions. Against arbitrary competing columns, a real
matrix is identifiable exactly when it passes the test and the listed
squares are independent.
\end{proposition}
\begin{proof}
For separated competing directions the equivalence is the argument of
\cite[Propositions~2.2 and~2.7]{r80:WS}. To allow repeated directions,
first suppose that \(Q_i=\sum_{j\ne i}c_jQ_j\), and fix \(h\ne i\).
Give column \(i\) a standard Gaussian source and each other column an
independent symmetric sign plus independent Gaussian noise of variance
\(L>1+\max_{j\ne i}|c_j|\). Replace column \(i\) with a second copy
of \(a_h\), with its own standard Gaussian source. On the other columns
replace the noise variances by \(L+c_j-\mathbf1_{\{j=h\}}\).
The sign part has the same law and the two Gaussian covariance sums
agree. Each model has \(J\) independent nondegenerate sources, exactly
one Gaussian, and the column multisets differ. For instance, the eight
vectors \((1,s_1,s_2,s_3)\), \(s_i\in\{-1,1\}\), pass the rank-one test:
the quadrics \(x_i^2=x_0^2\) cut out precisely these lines. Their squares
satisfy the relation with coefficients \(s_1s_2s_3\).

Conversely, an extra square supplies a competing matrix already with
distinct directions by \cite[Proposition~2.7]{r80:WS}. Suppose the
test holds and the squares are independent. Group collinear nonzero
terms in a competing representation. A group containing a non-Gaussian
source stays non-Gaussian by Cram\'er's theorem. The non-Gaussian
directions are determined by the law \cite[Theorem~2.4]{r80:WS}.
When there are no Gaussian sources this accounts for all \(J\)
columns. With the Gaussian at \(i\), the competitor must have all
\(J-1\) distinct non-Gaussian directions and one Gaussian column \(b\),
which we normalize by absorbing its standard deviation. Normalize the
Gaussian column \(a_i\) in the same way. On a neighborhood of the origin
the difference of the two Gaussian log characteristic functions is a
sum of continuous functions of the individual variables \(a_j^\top u\),
\(j\ne i\), and a possible affine term from the means. Convolve locally
with a smooth compactly supported probability density. The Hessian
of each smoothed summand is a scalar multiple of \(Q_j\), and
convolution preserves the Hessian of a quadratic. Thus
\(bb^\top-Q_i\in\operatorname{span}\{Q_j:j\ne i\}\).
Independence excludes \(b=0\) and excludes every previously assigned
line. The rank-one test leaves exactly the line of \(a_i\).

Suppose that $\sum_{i\in S}\lambda_i a_i a_i^\top=0$, with all
$\lambda_i\ne0$, and let $H$ be the span of the participating columns.
The positive and negative terms give two expressions for the same positive
semidefinite matrix. Its range on each side is the span of that side's
columns, so each side spans $H$. Thus $2\dim H\le |S|\le7$.
Dimension one would require coincident lines. In dimension two, three
distinct lines already give a basis of the symmetric squares on their
plane, since a binary quadratic vanishing on three projective points is
zero. Hence the span contains infinitely many real rank-one directions.

If $\dim H=3$, there are at least six distinct points in $\mathbb P(H)$.
Three on a projective line give the preceding obstruction. Otherwise pick
five of the points. Their squares are independent: for each chosen point,
pair the other four and take the product of the lines through the two
pairs. This quadratic vanishes at those four and is nonzero at the chosen
point. Thus the span of all participating squares has dimension five or
six in $\operatorname{Sym}^2 H$. In dimension six it contains every square.
In dimension five its annihilator is a single conic through the cited
points. A singular plane conic is a pair of lines over $\mathbb C$, counted
with multiplicity; such a pair cannot contain five real points with no
three collinear. The conic is therefore nonsingular, and a nonsingular
real conic with a real point has infinitely many real points. Their squares
all belong to the span. In every case there is a real rank-one direction
beyond the finite list, contrary to the rank-one test.

For $8\le J\le\binom I2$, the real examples in
\cite[Proposition~6.3]{r80:WSpre} pass the rank-one test even over $\mathbb C$.
They cover $J=8$ and $J\ge12$ by lower-dimensional constructions and
$J=9,10,11$ by explicit matrices. Embedding the column space in
$\mathbb R^I$ preserves both the dependence and the rank-one directions,
since $bb^\top$ supported on a subspace has $b$ in that subspace.
\end{proof}

\subsection{When non-identifiable mixing matrices form a closed set}
Wang and Seigal ask for the real closure and for the dimensions in
which the non-identifiable set is closed \cite[Section~6.3]{r80:WSpre}.
A competing rank-one direction can converge to a listed one even when
the listed squares remain independent. Its limiting tangent identifies
the missing stability condition.
\begin{theorem}[The closure of the non-identifiable locus]\label{r81:wsclosure}
Let \(\mathbb K=\mathbb R\) or \(\mathbb C\), let \(J\ge2\), and put
\(W=\operatorname{span}_{\mathbb K}\{a_ja_j^\top:1\le j\le J\}\).
A matrix \(A\) lies outside the Euclidean closure of the corresponding
non-identifiable locus if and only if the following hold.
Here we use the algebraic rank-one test over either field, together with
distinct nonzero column directions; over the reals the closure is also
the same for unrestricted competing columns.
\begin{enumerate}
\item Its column squares are linearly independent.
\item For \(x\in\mathbb K^I\setminus\{0\}\), membership of \(xx^\top\)
in \(W\) implies \(x\in\mathbb K a_i\) for some \(i\).
\item For every \(i\), the linear map
\[
  \mathbb K^I/\mathbb K a_i\longrightarrow \operatorname{Sym}_I(\mathbb K)/W,
  \qquad [b]\longmapsto [a_i b^\top+b a_i^\top]
\]
is injective.
\end{enumerate}
\end{theorem}
\begin{proof}
A non-identifiable matrix lies in the closure. A zero or dependent pair
of columns violates the first condition. For distinct nonzero column
directions the rank-one test of Wang and Seigal is the second condition.
Suppose then that \(A\) is identifiable. If its squares are dependent,
choose \(a_i\) whose square lies in the span of the other squares.
Replace \(a_i\) by \(a_i+t b\) with \(b\notin\mathbb K a_i\). The
old direction \(a_i\) is an extra square direction for all sufficiently
small \(t>0\). If the third test fails with independent squares, write
\(a_i b^\top+b a_i^\top=c a_i a_i^\top+D\), where \(D\) lies in the
span of all squares other than the \(i\)-th. Replace \(b\) by
\(b-ca_i/2\). Now the squares of \(a_i+t b\) and \(a_i-t b\) differ
by \(2tD\). Replacing the listed column by the first of these two
vectors makes the second an extra direction for every sufficiently
small \(t>0\). Both constructions give configurations failing the rank-one test
and converging to \(A\). In particular every square-dependent
configuration belongs to the closure under the separated convention.
Proposition~\ref{r80:ws-real} and its proof therefore give the same
real closure under unrestricted competition.

Conversely, suppose the first two conditions hold and non-identifiable
matrices \(A_m\) converge to \(A\). Their squares are eventually
independent, and their spans \(W_m\) converge to \(W\) in the
Grassmannian. Each \(A_m\) therefore has a competing projective column
direction \([x_m]\). After a subsequence \([x_m]\) converges, by
compactness of real or complex projective space. The square of its
limit belongs to \(W\); by the second condition the limit is \([a_i]\).
Choose a linear functional \(l\) with \(l(a_i)=1\), and normalize
\(x_m\) and the \(i\)-th column \(u_m\) of \(A_m\) by imposing
\(l(x_m)=l(u_m)=1\). Both tend to \(a_i\). After taking a subsequence,
\(b_m=(x_m-u_m)/\|x_m-u_m\|\) tends to a unit vector \(b\in\ker l\).
Since
\[
  \frac{x_m x_m^\top-u_m u_m^\top}{\|x_m-u_m\|}
      =b_m x_m^\top+u_m b_m^\top\in W_m ,
\]
the limit \(a_i b^\top+b a_i^\top\) lies in \(W\). Here
\(b\notin\mathbb K a_i\), so the third test fails.
\end{proof}

\begin{theorem}[Closedness of the non-identifiable locus]\label{r81:wsclosed}
Suppose \(1\le J\le\binom I2\). The non-identifiable real \(I\times J\)
mixing matrices form a Euclidean closed set exactly for \(J\le6\).
For \(2\le J\le6\) this set is the union of the conditions that some
two columns have rank at most one, some three have rank at most two,
or some five have rank at most three, omitting conditions requiring
more than \(J\) columns. This holds for either real competitor convention. The same assertions
hold over \(\mathbb C\) for the algebraic test in
\cite[Proposition~3.2]{r80:WSpre}.
\end{theorem}
\begin{proof}
Use the real rank-one test in \cite[Theorem~1.5]{r80:WS} and its complex
counterpart in \cite{r80:WSpre}. For one nonzero column there is no competing direction, so the one-column locus is closed under either convention for the zero matrix.
For \(J\ge2\) a linearly dependent pair already gives non-identifiability.
Otherwise the column directions are distinct and nonzero. If there is an
extra direction \(b\), its square and a subset of the listed squares
satisfy a linear relation with every coefficient nonzero. Arrange the
participating vectors, including \(b\), as columns of \(M\), and let
\(D\) be the invertible diagonal matrix of coefficients. Then
\(MDM^\top=0\), so the row space lies in its orthogonal complement for
the nondegenerate bilinear form \(D\). If \(m\) vectors participate
and their span has dimension \(r\), it follows that \(2r\le m\).
For \(J\le6\), one has \(m\le7\), so \(r\le3\). Dimension one is
impossible. Dimension two requires at least three original columns in
a plane. Dimension three requires at least five original columns in
a three-dimensional space.

Conversely three distinct lines spanning a plane have squares spanning
all binary quadratics, as in Proposition~\ref{r80:ws-real}. If five
lines span a three-dimensional space and no three lie in a plane, the
same five-point separation argument gives a unique nonsingular projective
conic through them. It has infinitely many points over the given field
(in the real case it already has a real point). Their squares are in the
span of the five listed squares. These obstructions are determinantal
conditions, hence closed. For \(J\le6\) every configuration passing
the test has independent squares by Proposition~\ref{r80:ws-real}.

For the first failure of closedness, take \(I\ge5\) and use a fixed
four-dimensional coordinate subspace \(H\). Write a column as
\((1,a,b,c)\). At the four choices \(a,b\in\{-1,1\}\) set
\(L_t(a,b)=2+t+a+b\). At \((-1,-1)\) take only \(c=L_t\);
at the other three corners take both values \(c=\pm L_t\).
For \(0\le t<1\) there are seven distinct lines. The quadrics on \(H\)
vanishing at all seven are precisely the span of
\[
 x_1^2-x_0^2,\qquad x_2^2-x_0^2,\qquad
 x_3^2-((2+t)x_0+x_1+x_2)^2 .
\]
Indeed, modulo these three quadrics, any quadratic restricts at the
corners to \(F(a,b)+cG(a,b)\), with \(F\) bilinear in \(a,b\)
(allowing constant and linear terms) and \(G\) affine linear.
The paired values force \(G=0\) at three noncollinear corners, hence
identically. Then \(F=0\) at all four corners forces \(F\equiv0\).
A common projective zero must have \(x_0\ne0\) and is therefore exactly
one of the sign choices above, including both signs at \((-1,-1)\).
At \(t>0\) the eighth line gives a competing square. At \(t=0\)
the two signs on that corner coincide, leaving exactly the seven listed
lines even over \(\mathbb C\). A square supported on \(H\) comes from
a vector in \(H\). Thus these seven-column non-identifiable matrices
converge to an identifiable matrix.

For \(J>7\) project the complex quadratic Veronese variety from the
projective span \(W\) of the seven limiting squares. The preceding
calculation gives \(\dim W=6\). The closure \(Y\) of the image is
irreducible and nondegenerate. On vectors outside \(H\) the projection
has singleton fibers: a row of the symmetric square with a nonzero
diagonal entry outside \(H\) remains visible modulo the span and
recovers the vector's projective coordinates. Thus \(\dim Y=I-1\)
and its ambient codimension is \(\binom I2-7\).
Set \(k=J-7\), so \(k\le\operatorname{codim}Y\). The generalized
trisecant lemma \cite[Theorem~4.6]{r80:WSpre} shows that the span of
\(k\) general points on \(Y\) meets \(Y\) only in those points.
Their preimages may be chosen real and outside \(H\), since these
configurations are Zariski dense in the complex parameter space.
Adjoin these \(k\) fixed columns to the seven-column family.
At the limit, a square in the enlarged span either lies in \(W\)
and is one of the original seven, or projects to a chosen point,
whose fiber is a singleton. The chosen general points are
linearly independent in the projection, so the enlarged list of squares
is independent as well. For \(t>0\) the eighth direction in
\(H\) remains an extra direction. This proves nonclosedness for all
the claimed dimensions.
\end{proof}

In particular, the seven-column family contradicts the complex
closedness assertion at \(J=7\) in Proposition~6.8 of the
January 2024 preprint \cite{r80:WSpre}. At the collision the limiting
seven squares still have full column rank. The tangent at the collision
appears explicitly: all three quadrics have zero derivative in the
\(x_3\)-direction at \( (1,-1,-1,0) \).

\begin{corollary}[How many real mixing lines can be stable?]\label{r82:ws-capacity}
For \(I\ge2\), the maximum number of columns in a stably identifiable
real \(I\)-row mixing matrix is
\[
 \binom I2+\mathbf1_{\{I\equiv2,3\pmod4\}} .
\]
For arbitrary competing columns, every identifiable matrix satisfies
\(J\le 2^{I-1}\). With distinct competing directions one can instead
have \(1152\) lines in \(11\) rows. Let \(M_I\) denote the largest
number of distinct real projective directions passing the rank-one
test in \(I\) coordinates. Then, for \(I\ge2\),
\[
 M_I\ \ge\ 2^{I-1}\, 3^{\lfloor(I-1)/25\rfloor}.
\]
Moreover \(M_I\) is exactly the maximal finite number of real projective
zeros of an \(I\)-variable sum of squares of quadratics. One has
\[
  2\cdot 3^{1/25}\le \lim_{d\to\infty} M_{d+1}^{1/d}\le3 .
\]
\end{corollary}
\begin{proof}
Stable identifiability is membership in the Euclidean interior of the
identifiable set. The real rank-one criterion makes this set semialgebraic
by projection, so it has nonempty interior precisely when it has positive
measure under a standard Gaussian matrix. The stated maximum is then
\cite[Theorem~1.9]{r80:WS}, including the cases \(I=2,3\).
At a stable matrix, the \(J\)-dimensional span of the column squares
meets the \(I\)-dimensional affine Veronese tangent at a listed square only
along that square's line. Dimension counting in the space of symmetric
matrices also gives \(J\le\binom I2+1\).

Under unrestricted competition, independence gives
\(J\le I(I+1)/2\le 2^{I-1}\) for \(I\ge5\). In dimensions \(2,3,4\),
a larger number of independent squares leaves at most one annihilating
quadric (none in dimension two). In dimensions three and four an
indefinite quadratic form has infinitely many real projective zeros;
a semidefinite form has its projective kernel as zero set, with at most
one point if finite. Thus in these dimensions the rank-one test also
excludes \(J>2^{I-1}\).

For a configuration passing the rank-one test take a basis \(q_i\) for the quadrics
orthogonal to the span of its column squares. The real projective zero
set of \(\sum_i q_i^2\) is precisely the set of column lines.
Conversely, for a quartic sum of squares with finite projective real zero
set, use one representative of every zero as a column. All the quadratics in the sum of
squares vanish on these columns. Any square in the column
span is annihilated by these quadrics and so comes from a listed line.
This proves the asserted equality of extremal counts, connecting
\cite[Section~6.1]{r80:WSpre} with the quartic case of
\cite[Conjecture~4.2]{r82:FLOS}.

For explicit exceptional matrices, use the affine quadrics of \cite{r82:realzeros}. In coordinates
\([z:t:x_0:\cdots:x_{m-1}]\) take their homogenizations
\[
 x_0^2=t((m-1)z-t),\qquad
 x_i^2=(t-(i-1)z)(t-iz),\qquad 1\le i\le m-1 .
\]
At \(z=0\) the first equation forces \(t=x_0=0\) over the reals and then
all coordinates vanish. Normalize \(z=1\). The first equation confines
\(t\) to \([0,m-1]\), and the others exclude its open unit subintervals.
At each of the \(m\) remaining integers exactly two of the \(m\) radicands
vanish. Thus there are \(m2^{m-2}\) real projective points.
For \(m=9\) this transfers the previously posted affine count
\(9\cdot2^7=1152\) to the rank-one-test interpretation of the real
existence question following
\cite[Theorem~6.1]{r80:WSpre}. Appending equations \(y_j^2=z^2\) doubles
the count per new variable, giving \(M_I>2^{I-1}\) for every \(I\ge11\).

For the general bound use \(m=24\), so the count in 25 affine coordinates
is \(24\cdot2^{22}=3\cdot2^{25}\). Write \(I-1=25k+r\), \(0\le r<25\).
Take \(k\) independent copies of these quadrics with the same
homogenizing coordinate \(z\), and impose \(y_j^2=z^2\) on the remaining
\(r\) coordinates. There is no real projective point with \(z=0\), and at
\(z=1\) the counts multiply. The rank-one criterion gives the asserted
bound. For the complex rank-one test the maximum \(2^{I-1}\) is
\cite[Theorem~6.1]{r80:WSpre}.

The same product works for arbitrary finite projective zero sets: choose
an affine chart containing each set and use a common homogenizing variable.
This gives \(M_{d+e+1}\ge M_{d+1}M_{e+1}\). For the upper bound,
dehomogenize the quartic sum of squares in such a chart. Its leading
quartic is positive definite, so subtracting a linear form preserves
coercivity. Its gradient map is onto over the reals by minimization,
hence dominant over \(\mathbb C\) with finite generic fibers.
For a sufficiently small generic linear perturbation, every original
zero gives a critical point in a disjoint small ball, by minimizing in
the ball and using positivity on its boundary. The cubic gradient
equations have at most \(3^d\) isolated solutions by B\'ezout.
Applying Fekete's lemma to \(\log M_{d+1}\) proves the
claimed limit and bounds.
\end{proof}

\part{Gaussian quadratic forms and determinant laws}

Consider the two real projections $P$ onto $(1,0)$ and $R$ onto
$2^{-1/2}(1,1)$. Their commutator has squared Hilbert--Schmidt norm
$1/2$. The centered Gaussian quadratic forms $(Y_P,Y_R)$ are
infinitely divisible: dilating Hein\"avaara's positive planar matrix
measure by the kernel $te^{-t}dt$ gives the joint jump measure after
normalization, since its radial composition with $(1-t)dt/t$ on
$(0,1)$ is $e^{-t}dt/t$ \cite{qc42:Hei}. Adjoining the
total radius $Y_I$ gives a triple for which the fourth-cumulant
difference in \eqref{bridge58:commutator} is $4$; the
finite-spectrum test in Theorem~\ref{bridge58:finite} rules out
infinite divisibility. Thus the squared commutator appears in both the covariance repair
formula and the fourth-cumulant obstruction for the triple.
\section{Planar tracial measures and Gaussian second chaos}\label{r54:pencils}

For a real pencil $C_u=u_1A+u_2B$, the scalar measure of interest is
its spectral measure averaged over the dilations $0<t<1$ with weight
$(1-t)\,dt/t$. Two uses of this weight account for the connection
between the embedding and divisibility questions. Its $p$th moment is
$1/[p(p+1)]$, and its multiplicative convolution with a gamma law of
shape two is the radial jump measure of a centered Gaussian square.

\subsection{A tracial measure from the roots of a hyperbolic polynomial}
Hein\"avaara constructed the planar measure of two Hermitian matrices
in \cite[Theorem~1.4]{qc42:Hei}. In Question~3.5 of his research
statement \cite[p.~4]{newtr:question}, he asks for a construction
directly from a ternary hyperbolic polynomial without first finding a
determinantal representation. Write $u\in\R^2$ and let $p(z,u)$ be real homogeneous
of degree $m$, monic in $z$, with $p(z,u)=\prod_j(z+\lambda_j(u))$
and all $\lambda_j(u)$ real for real $u$. A factor $z$ may be discarded:
it makes no contribution below. If only such factors occur take the
zero measure.

\begin{theorem}[Tracial measure of a ternary hyperbolic polynomial]
\label{hyp48:poly}
There is a unique positive Radon measure $\mu_p$ on the punctured plane,
of bounded support and finite first absolute moment, such that for each
$u\in\R^2$ its pushforward off zero is
\[
 \sum_{j=1}^m\int_0^1 \delta_{t\lambda_j(u)}\,(1-t)\,\frac{dt}{t}
       \quad\text{on }\R\setminus\{0\}.
\]
In particular it gives the logarithmic identity in Question~3.5,
\[
 \int\log|1+u\cdot w|\,d\mu_p(w)
 =\int_0^1\log|p(1,tu)|\,(1-t)\,\frac{dt}{t}.
\]
Off the finitely many lines with normal $n$ satisfying $p(0,n)=0$,
the measure has a nonnegative density. On each such line its restriction
off zero is the sum, over the root branches vanishing at $n$, of the
measures $\int_0^1\delta_{t\nabla\lambda_j(n)}(1-t)\,dt/t$.
Local real-analytic branches are used in this formula, with multiplicities.
\end{theorem}
\begin{proof}
Put $g(s)=\int_0^1(e^{its}-1)(1-t)\,dt/t$ and
$F(u)=\sum_jg(\lambda_j(u))$. The ordered roots have linear growth.
On the real axis $g(0)=0$, $|g'|\le1/2$, $g'(s)\to0$ at infinity
by the Riemann--Lebesgue lemma, and $|g(s)|\le C\log(2+|s|)$ by
splitting at $t=1/|s|$. Define the tempered distribution $T$ with
characteristic transform $F$. Thus if
$\widehat\varphi(u)=\int e^{-iu\cdot w}\varphi(w)\,dw$, then
$T(\varphi)=(2\pi)^{-2}\int\widehat\varphi(u)F(u)\,du$.

\emph{Off the exceptional lines.}
We give the scalar substitutions in the polar Fourier calculation of
\cite[proof of the singular-part and density formulas]{qc42:Hei}.
For $e_\theta=(\cos\theta,\sin\theta)$ write
$\Phi_\theta(s)=\int_{w\cdot e_\theta=s}\varphi(w)\,d\ell(w)$.
The one-dimensional Fourier formula for $|r|g(r)$
(ibid., Lemma~3.3 in the preprint) applied separately to the
real roots gives
\[
 2\pi^2 T(\varphi)=
 \int_0^\pi\int_{\R}
 \frac{\Phi_\theta(s)-\Phi_\theta(0)-s\Phi_\theta'(0)}{s^2}
           \log|p(1,-e_\theta/s)|\,ds\,d\theta .
\]
Here the sum of logarithms is the logarithm of the product of the
scalar factors $1-\lambda_j(e_\theta)/s$. Truncate at $|s|>\varepsilon$
before separating the three terms. For $w=(a,b)$, put $v=(b,-a)$.
The two monic polynomials in the radial change of variables are
\[
 q_1(h)=\frac{p(1,-w/|w|^2-hv)}{p(0,-v)},\qquad
 q_2(h)=\frac{p(1,-hv)}{p(0,-v)} .
\]
In particular $q_2$ has only real roots. To detail the cutoff limit,
put $M=1/(\varepsilon|w|)$. In the term with $\Phi_\theta(s)$ use
$w=s e_\theta+y e_\theta^\perp$ and $h=y/(s|w|^2)$; in the
$\Phi_\theta(0),\Phi_\theta'(0)$ terms use
$w=y e_\theta^\perp$, $h=1/(sy)$. These are the real changes of
variables following Lemma~3.3 of that proof. They give the
difference of the integrals of $\log|q_1(h)|$ and $\log|q_2(h)|$
against $\varphi(w)$ on $|h|<M$, and the angular derivative term
with integrals of $\log|1-h/\gamma|/h$ over the same interval,
$\gamma$ running through the roots of $q_2$. The first cutoff differs
from $M$ by $O(\varepsilon)$ locally away from zero. On a compact set
avoiding the exceptional lines, leading coefficients are bounded away
from zero, roots of both polynomials are bounded, and roots of $q_2$
are also bounded away from zero. The cutoff corrections are therefore
$O(\varepsilon\log(1/\varepsilon))$ there. The elementary real integrals
in Lemmas~3.4--3.5 of the same preprint give
\[
 \begin{aligned}
 \lim_{M\to\infty}\int_{-M}^M\log|1-\gamma/h|\,dh
      &=\pi|\operatorname{Im}\gamma|,\\
 \lim_{M\to\infty}\int_{-M}^M\log|1-h/\gamma|\,\frac{dh}{h}
      &=-\tfrac{\pi^2}{2}\operatorname{sgn}\gamma
         \quad(\gamma\in\R\setminus\{0\}).
 \end{aligned}
\]
The sign sum of the $q_2$ roots is constant on each angular sector, so
its integral against $(v\cdot\nabla\varphi)/|w|^2$ is zero. We obtain
the density $(2\pi)^{-1}\sum_{q_1(\gamma)=0}|\operatorname{Im}\gamma|$.
Only polynomial factorization and scalar Fourier inversion have been
used.

\emph{On an exceptional line.}
Fix its unit normal $n$. Along the real angle the polynomial admits
local real-analytic root branches by Rellich's theorem for hyperbolic
polynomials \cite{hyp48:rellich}. Positive homogeneity extends these
branches to a conic neighborhood of $n$. For fixed $u$ the roots at
$kn+u$ are $k\lambda_j(n+u/k)$ for all sufficiently large $|k|$,
also for negative $k$ as multisets. Hence
\[
 F(kn+u)-F(kn)\ \longrightarrow       \sum_{\lambda_j(n)=0} g(\nabla\lambda_j(n)\cdot u).
\]
The nonzero branches have bounded increments for fixed $u$ and
$g'$ tends to zero. For domination, when $|u|<c|k|$ the analytic
branches give root increments bounded by $C|u|$. Outside that cone
$|k|\le |u|/c$ and linear root growth gives the same bound for the
difference of the two sums, using $|g'|\le1/2$.
For a smooth compactly supported $\varphi$ away from the origin the
constant $F(kn)$ contributes zero. Fourier inversion and dominated
convergence now identify the limit of $T(\varphi e^{ik n\cdot w})$
with the stated positive sum of ray measures tested on $\varphi$.
The bound is uniform in real $k$.

Choose smooth compactly supported $0\le\chi\le1$, equal to one near
zero. Integrating the last modulated tests against the Fourier transform
of $\chi$ shows that $T(\varphi\chi(Rn\cdot w))$ tends to the same ray
sum as $R\to\infty$. For nonnegative $\varphi$ meeting only this
exceptional line and supported away from zero,
$T(\varphi(1-\chi(Rn\cdot w)))\ge0$ by the density already proved.
A partition of unity therefore proves $T\ge0$ on the entire punctured
plane. Denote its measure there by $\mu_p$. Dominated convergence of
the cutoff against this locally finite measure also proves the asserted
formula for its restriction to the line. Euler's homogeneity identity
places each corresponding gradient on that line.

\emph{At zero and infinity.}
Test on a fixed smooth nonnegative annular cutoff, dilated to radius $r$
and equal to one on $r\le|w|\le2r$. Its pairing is bounded by
$C(1+|\log r|)$ for $r\le1$ from the logarithmic bound for $F(u/r)$,
and by $C/r$ for $r\ge1$ from $|F(u/r)|\le C|u|/r$.
Thus $\int_{|w|\le1}|w|\,d\mu_p<\infty$ and $\mu_p\{|w|>1\}<\infty$.
The compensated distribution
$\varphi\mapsto\int(\varphi(w)-\varphi(0))\,d\mu_p$ differs from $T$
only at zero; its transform differs from $F$ by a polynomial. Both
transforms are $o(|u|)$ at infinity (use dominated convergence with
$|w|$ near zero) and vanish at zero, so the polynomial vanishes.
Restriction to each real frequency line and one-dimensional Fourier
uniqueness give the required pushforward off zero. The two coordinate
projections then give bounded support. The same compensated Fourier
identity proves uniqueness. Finally the projected ray integrals have
finite logarithmic integrals, including at an interior logarithmic
singularity, which proves the displayed identity.
\end{proof}

The uniform octic sum-of-squares certificate asked for in Hein\"avaara's thesis \cite[Section~2.2.1, Question~1 after (2.5), p.~49]{r88:HeiThesis}
is excluded by an earlier result. With $A=X+Y$, $B=X-Y$, the symmetrization of
$(X^2-Y^2)^4$ is the sum of the seventy words with four $A$'s and four $B$'s divided by seventy,
for which Collins, Dykema and Torres-Ayala exclude a Hermitian sum of squares
modulo commutators \cite[Proposition~3.8]{r88:CDTA}.

\subsection{Tracial measures and real planes in Haagerup spaces}
\label{qc42:sec}
Questions~3.2 and~3.3 of \cite[p.~4]{newtr:question} ask about semifinite
traces and real two-dimensional subspaces of Haagerup $L^p$.
For the ordinary Schatten classes the real-plane embedding is already
\cite[Corollary~1.5]{qc42:Hei}. For the trace measure we use the positive-pair theorem of P\u{a}unescu, R\u{a}dulescu and
Thom \cite[Theorem~B]{newtr:finite}. A change of center in the plane
allows the two self-adjoint operators to have arbitrary signs.
Throughout this paragraph $Hf(s)=\int_0^1f(ts)(1-t)\,dt/t$, with
$Hf(0)=0$ for $f(0)=0$.

\begin{theorem}[Planar tracial measure in a finite algebra]
\label{h50tr:finite}
Suppose $\mathcal M$ has a faithful normal finite trace $\tau$ and
$A,B\in\mathcal M$ are self-adjoint. On $\R^2\setminus\{0\}$ there is
a unique positive Radon measure $\mu$ satisfying
\begin{equation}\label{h50tr:projection}
  \int f(u\cdot x)\,d\mu(x)=\tau\bigl(Hf(u_1 A+u_2 B)\bigr)
  \qquad (u\in\R^2)
\end{equation}
for every nonnegative Borel $f$ with $f(0)=0$, allowing $+\infty$.
It is carried by $[-\|A\|,\|A\|]\times[-\|B\|,\|B\|]$.
The identity extends by linearity whenever $\tau(H|f|(u_1A+u_2B))$
is finite.
\end{theorem}
\begin{proof}
Normalize $\tau(\mathbf1)=1$. Write $C_u=u_1A+u_2B$ and take the
Fourier transform of distributions with the convention
$\widehat W(u)=\langle W,e^{iu\cdot x}\rangle$. There is a real
compactly supported distribution $W$ with $\widehat W(u)=\tau(e^{iC_u})$.
Indeed, the right side extends to an entire function. Differentiating
the squared norm of $e^{itC_{u+iv}}h$ gives
$\|e^{iC_{u+iv}}\|\le e^{\|C_v\|}$ for $u,v\in\R^2$.
The Paley--Wiener--Schwartz theorem puts its inverse transform in
the stated rectangle, and conjugation at $-u$ gives reality.
In particular $\langle W,1\rangle=1$.

For a compactly supported distribution $V$ and $c\in\R^2$, set
\[
 D_c[V](\phi)=\int_0^1 \frac{1-t}{t}
  \big\langle V,\phi(c+t(\,\cdot-c))-\phi(c)\big\rangle\,dt .
\]
This is a distribution supported in the convex hull of $c$ and
$\operatorname{supp}V$. The numerator inside the pairing is $O(t)$
in every fixed smooth seminorm on a compact neighborhood of
$\operatorname{supp}V$. This estimate also gives
$D_c[W*\rho_\varepsilon]\to D_c[W]$ distributionally for a
nonnegative smooth approximate identity of compact support.
One may equivalently integrate the test function in $t$ before
pairing with $W*\rho_\varepsilon$.

Write $w_\varepsilon=W*\rho_\varepsilon$, a real smooth function
of compact support. Away from $c$, changing variables
$y=c+t(z-c)$ and then $s=t^{-1}-1$ gives a smooth density
\begin{equation}\label{h50tr:ray}
  D_c[w_\varepsilon](y)=
       \int_0^\infty s\,w_\varepsilon(y+s(y-c))\,ds .
\end{equation}
In dimension $n$ the corresponding weight is $s(1+s)^{n-2}$; in the plane
the weight $s$ scales by a single power on each ray. On compact subsets
away from $c$ the $s$-interval contributing to
this formula is bounded uniformly.
Fix $x\ne0$, put $e=x/|x|$, and choose $c=-M e$ with $M$ large.
For every $b\in\operatorname{supp}\rho_\varepsilon$, the shifted
pair $(A+(b_1-c_1)\mathbf1,B+(b_2-c_2)\mathbf1)$ has two linearly
independent strictly positive real linear combinations. For example,
use the coefficient vectors $e$ and $e+e^\perp$, where $e^\perp$ is
a perpendicular unit vector. Their scalar contributions from $-c$
are both $M$, whereas the two operator contributions and the
contributions from $b$ have uniform norm bounds.

For such $b$ translate $W$ by $b$. The distribution so obtained has
Fourier transform $e^{iu\cdot b}\tau(e^{iC_u})$. Apply the cited
positive-pair theorem in the two independent positive coordinates and
undo the invertible linear change of variables. It gives a positive
measure $\eta_b$ for the pair shifted by $b-c$, supported in a
bounded set and with finite first moment. Consequently
\[
 D_c[W(\,\cdot-b)](\phi)
       =\int [\phi(c+y)-\phi(c)]\,d\eta_b(y).
\]
Indeed, both compactly supported distributions have the same Fourier
transform by the positive-pair identity for $f(s)=e^{is}-1$.
For nonnegative $\phi$ supported away from $c$ the displayed
quantity is nonnegative. Averaging its left side against
$\rho_\varepsilon(b)\,db$ gives positivity of $D_c[w_\varepsilon]$
there. Thus its smooth density in \eqref{h50tr:ray} is nonnegative
at $x$. The vectors $x-c=(|x|+M)e$ and $x=|x|e$ point along the
same ray. Rescaling $s$ in \eqref{h50tr:ray} yields
\[
  D_0[w_\varepsilon](x)
   =\frac{(|x|+M)^2}{|x|^2}\,D_c[w_\varepsilon](x)\ge0 .
\]
This holds at every $x\ne0$. In the distributional limit,
$D_0[W]$ therefore restricts to a positive Radon measure $\mu$ on the
punctured plane.

For $u\ne0$ the projection of $W$ onto the $u$-coordinate is the
spectral distribution of $C_u$ under $\tau$, by Fourier uniqueness
on the line. If $f\in C_c^\infty(\R\setminus\{0\})$, then
$f(u\cdot x)$ vanishes around the origin. Multiplying by a smooth
cutoff equal to one near the common support rectangle shows that its
$\mu$-integral is $D_0[W](f(u\cdot x))=\tau(Hf(C_u))$.
Thus the projected measures agree on the punctured line; this yields
all the stated Borel identities, including $u=0$. For uniqueness,
the four signed coordinate tails first force the support rectangle.
All halfspaces whose closures miss zero have specified mass. The
compact-support case of the exterior Cram\'er--Wold theorem
\cite[Theorem~2]{newtr:boman} gives uniqueness.
\end{proof}

\begin{lemma}[Projected measures with a finite power moment]
\label{h50tr:weighted}
Let $p>0$ and let $\mu_1,\mu_2$ be positive Radon measures on
$\R^2\setminus\{0\}$ with $\int |x|^p\,d\mu_j(x)<\infty$.
If their projections onto every line agree off zero, then $\mu_1=\mu_2$.
\end{lemma}
\begin{proof}
Put $k=\lceil p\rceil-1$ and
$K(t)=e^{it}-\sum_{j=0}^k(it)^j/j!$. Taylor's bound near zero and the
polynomial bound for $|t|\ge1$ give $|K(t)|\le c_p|t|^p$ globally.
For $\sigma=\mu_1-\mu_2$, projection equality and $K(0)=0$ give
$\int K(u\cdot x)\,d\sigma(x)=0$ for every $u\in\R^2$.
Take $\phi\in C_c^\infty(\R^2\setminus\{0\})$ and write
$\widehat\phi(u)=\int e^{-iu\cdot x}\phi(x)\,dx$. In its Fourier
inversion integral one may replace $e^{iu\cdot x}$ by $K(u\cdot x)$,
since every derivative of $\phi$ vanishes at zero. The bound by
$c_p|u|^p|x|^p$ permits Fubini against the rapidly decreasing
$\widehat\phi$ and $|\sigma|$. Thus $\int\phi\,d\sigma=0$.
\end{proof}

\begin{theorem}[Semifinite planar tracial measures]\label{h50tr:semifinite}
Let $\tau$ be a normal semifinite faithful trace on $\mathcal M$.
For bounded self-adjoint $A,B\in\mathcal M$, a positive Radon
measure on the punctured plane satisfying \eqref{h50tr:projection}
exists precisely when both $A$ and $B$ are $\tau$-compact, meaning
$\tau(\mathbf1_{(s,\infty)}(|T|))<\infty$ for every $s>0$ and
$T=A,B$. In that case the support rectangle and uniqueness of
Theorem~\ref{h50tr:finite} hold.

If instead $A,B$ are self-adjoint elements of $L^p(\mathcal M,\tau)$
for some $0<p<\infty$, possibly unbounded, there is again a unique
positive Radon measure satisfying \eqref{h50tr:projection}, with
the same conventions on Borel tests. In this case
\begin{equation}\label{h50tr:power}
 \int |u\cdot x|^p\,d\mu(x)
       =\frac{\|u_1A+u_2B\|_p^p}{p(p+1)},\qquad
 \int |x|^p\,d\mu(x)<\infty .
\end{equation}
\end{theorem}
\begin{proof}
For bounded $A,B$, four signed coordinate tests first give the
support rectangle. A coordinate halfspace $x_j>s$ then has finite
mass, since within that rectangle its closure misses zero.
But $H\mathbf1_{(s,\infty)}(\lambda)$ is bounded below by the positive
number $\int_{1/2}^1(1-t)\,dt/t$ on $\lambda>2s$.
The tests with both signs prove $\tau$-compactness.

For sufficiency in the bounded case take
$A_n=A\mathbf1_{(1/n,\infty)}(|A|)$ and define $B_n$ likewise.
In the $L^p$ case truncate both spectra to $1/n<|T|\le n$.
Their two support projections have a join of finite trace. Apply
Theorem~\ref{h50tr:finite} in this finite corner, retaining the
original trace, to obtain $\mu_n$. Because $Hf(0)=0$ the trace in
the corner equals the trace in $\mathcal M$. Write
$C_n=u_1A_n+u_2B_n$ and $C=u_1A+u_2B$. In the bounded case
$C_n\to C$ in norm; in the $L^p$ case it converges in $L^p$.

We record the convergence of their positive and negative spectral
distributions away from zero. Write $N_T(s)=\tau(\mathbf1_{(s,\infty)}(T))$
for $T=T^*$ and $s>0$. For $s,\delta>0$ one has
\[
 N_{C_n}(s+\delta)\le N_C(s)+N_{|C_n-C|}(\delta).
\]
Indeed, vectors in the intersection of the spectral subspaces
$s+\delta<C_n\le L$, $-L\le C\le s$ and $|C_n-C|\le\delta$ would
violate the quadratic form bounds unless they vanish. The trace
dimension inequality for intersections bounds the trace of the
first projection by the sum of the two complementary traces.
Let $L\uparrow\infty$; the negative tail of $C$ vanishes at infinity.
The same inequality holds with $C$ and $C_n$ interchanged, and for
their negatives. The error term is eventually zero under norm
convergence and tends to zero under $L^p$ convergence by Chebyshev's
inequality. The signed spectral tails therefore converge at their
continuity thresholds. They are finite at every positive threshold:
for a bounded pencil of $\tau$-compact operators this follows also
by norm approximation by the finite corners; in $L^p$ it follows
by the same Chebyshev bound.

For $f\in C_c(\R\setminus\{0\})$, the function $Hf$ is continuous,
bounded, and vanishes near zero. The spectral-tail convergence gives
$\tau(Hf(C_n))\to\tau(Hf(C))$, first on a fixed compact spectral
interval with continuity endpoints. In the bounded case that
interval can contain all spectra. In the $L^p$ case the discarded
tails have absolute contribution at most
$\|Hf\|_\infty R^{-p}\sup_n\|C_n\|_p^p$ beyond $|\lambda|>R$.

The four signed coordinate projections bound the masses of $\mu_n$
uniformly on each compact subset off zero. In the bounded case,
all measures lie in the common support rectangle, and a fixed
positive coordinate threshold bounds each halfspace mass by the
corresponding spectral tail times a uniform upper bound for
$H\mathbf1_{(s,\infty)}$ on the spectral interval. In the $L^p$ case,
applying \eqref{h50tr:projection} in the finite corners to $|s|^p$
and using $|x|^p\le c_p(|x_1|^p+|x_2|^p)$ gives a uniform bound on
$\int |x|^p\,d\mu_n$. Take a vague subsequential limit $\mu$ on the
punctured plane. For $u\ne0$ and the above $f$, the test $f(u\cdot x)$
vanishes around zero. In the bounded case its support can be cut
off at the common rectangle. In the $L^p$ case the error beyond
$|x|>R$ is at most $\|f\|_\infty R^{-p}\sup_n\int|x|^p\,d\mu_n$;
the limit measure satisfies the same power bound by lower
semicontinuity. Thus \eqref{h50tr:projection} holds for these tests.
Equality of the projected measures off zero gives the Borel identities
and \eqref{h50tr:power}. Uniqueness in the bounded case follows as
in Theorem~\ref{h50tr:finite}. In the $L^p$ case, the two coordinate
identities force a finite $p$th moment for every candidate, so
Lemma~\ref{h50tr:weighted} applies.
\end{proof}

The signed-test convention is necessary even on Schatten classes.
Fix $0<q<p$ and let $B=0$ and $A$ have eigenvalues
$\pm j^{-1/q}$, one of each sign for every $j\ge1$.
Then $A\in\mathcal S_p$ and $f(t)=\operatorname{sgn}(t)|t|^q$
satisfies the local integrability requirement at the origin.
Since $Hf=f/[q(q+1)]$, both parts of the signed trace of $Hf(A)$
diverge. The nonnegative-test identity and its absolutely integrable
signed version have exactly the meanings stated in the theorem.

\paragraph{Ordinary trace and recovery of eigenvalues.}
For compact operators with the ordinary trace the preceding result
also recovers all eigenvalue multiplicities from one-dimensional tails.

\Needspace{5\baselineskip}
\begin{theorem}[Tracial joint spectral measure of a compact pair]
\label{newtr:compact}
Let $A,B$ be compact self-adjoint operators on a separable Hilbert
space. For $u=(u_1,u_2)\in\R^2$ write $C_u=u_1A+u_2B$. There is a
unique positive Radon measure $\mu$ on $\R^2\setminus\{0\}$ such that
for every unit $u$ and $s>0$,
\begin{equation}\label{newtr:tail}
 \mu\{v:u\cdot v>s\}
 =\sum_{\lambda_j(C_u)>s}
       \left(\log\frac{\lambda_j(C_u)}s-1+\frac{s}{\lambda_j(C_u)}\right).
\end{equation}
Eigenvalues are counted with multiplicity. The measure is carried by
$[-\|A\|,\|A\|]\times[-\|B\|,\|B\|]$ and is the vague limit on the
punctured plane of the measures of any increasing finite-dimensional
compressions with dense union.

For $f(0)=0$, set
\[
 Hf(x)=\int_0^1f(tx)(1-t)\,\frac{dt}{t}.
\]
For every $u\in\R^2$ and every nonnegative Borel $f$ with $f(0)=0$,
\begin{equation}\label{newtr:identity}
 \int f(u\cdot v)\,d\mu(v)=\tr Hf(C_u)
          :=\sum_{\lambda_j(C_u)\ne0}Hf(\lambda_j(C_u))
          \quad\text{in }[0,\infty].
\end{equation}
The same identity holds for complex Borel $f$ with $f(0)=0$ whenever
$\tr H|f|(C_u)<\infty$.
\end{theorem}
\begin{proof}
Theorem~\ref{h50tr:semifinite} gives \eqref{newtr:identity}.
For $f=\mathbf1_{(s,\infty)}$ one eigenvalue $\lambda>s$ contributes
$\int_{s/\lambda}^1(1-t)\,dt/t$; summation gives \eqref{newtr:tail}.
Those halfspace tails determine the measure by exterior Cram\'er--Wold
\cite[Theorem~2]{newtr:boman}. For finite-rank projections $P_n\uparrow I$,
compactness gives $P_nC_uP_n\to C_u$ in norm. For each
$f\in C_c(\R\setminus\{0\})$, the operators $Hf(P_nC_uP_n)$ converge
in norm and have uniformly bounded rank, since $s_j(P_nC_uP_n)\le s_j(C_u)$.
Thus their traces converge. The signed coordinate tails bound the
compression measures locally uniformly off zero, and the rectangle
bounds their supports. Every vague subsequential limit has the stated
projections; uniqueness gives convergence of the whole sequence.
\end{proof}

\subsection{A Fredholm formula for compact operator pairs}
\label{r56:fredholmsection}

Question~3.4 in \cite[p.~4]{newtr:question} asks for an infinite-dimensional
version of the explicit matrix density (4). The inverse of the transverse
pencil in that formula need not exist for compact operators. Its
generalized eigenvalues can instead be read from an analytic Fredholm
pencil. This gives both parts of the measure for the ordinary trace,
without a Schatten summability assumption.

\begin{theorem}[Density and ray measures of a compact pair]
\label{r56:fredholm}
Let $A,B$ be compact self-adjoint operators on a separable Hilbert space.
Write $C_u=u_1A+u_2B$ and let $\mu$ be their tracial measure from
Theorem~\ref{newtr:compact}. For a unit direction $d$ choose a unit
normal $n\perp d$, put $P_n=1_{\{0\}}(C_n)$, and let $D_n=P_nC_dP_n$
act on $P_n\mathcal H$. Call $rd$ exceptional when $r>0$ and
$\sup\sigma(D_n)\ge r$, using $-\infty$ for an empty spectrum.
The exceptional segments form a closed subset of the punctured plane,
and only finitely many meet each annulus bounded away from zero.

For $w=(a,b)$ outside the segments put $v=(b,-a)$ and consider
\[
  L_w(h)=I-C_w/|w|^2-hC_v,\qquad h\in\mathbb C .
\]
Its nonreal characteristic values are finite in number, counted with
algebraic multiplicity. The absolutely continuous part of $\mu$ has
density
\begin{equation}\label{r56:freddensity}
 \rho(w)=\frac1{2\pi}
  \sum_{h:\,L_w(h)\ {\rm not\ invertible}} |\operatorname{Im}h|,
\end{equation}
where only nonreal values enter. This density is continuous on the
complement of the segments. The singular part is given on each line
$\mathbb R d$, off zero, by
\begin{equation}\label{r56:fredline}
  \mu\big|_{\mathbb R d\setminus\{0\}}
    =\sum_{\lambda\in\sigma(D_n)\setminus\{0\}}
        m_{D_n}(\lambda)\int_0^1\delta_{t\lambda d}
                          (1-t)\,\frac{dt}{t}.
\end{equation}
Here $m_{D_n}$ denotes eigenvalue multiplicity. Each unoriented line
is counted once.
\end{theorem}

\begin{proof}
For distinct unoriented lines, vectors $e_i\in\ker C_{n_i}$ and
$e_j\in\ker C_{n_j}$ satisfy
$\langle Ae_i,e_j\rangle=\langle Be_i,e_j\rangle=0$.
Indeed the two equations obtained from $n_i,n_j$ are independent.
If an exceptional segment reaches radius $\delta>0$, choose a unit
eigenvector of $D_{n_i}$ with eigenvalue at least $\delta$, orienting
$d_i$ along the segment. For at least one $T\in\{A,B\}$ this gives
$|\langle Te_i,e_i\rangle|\ge\delta/\sqrt2$. In the group using the same
$T$, vectors $Te_i$ for different lines are $\delta/\sqrt2$ apart by
testing their differences against $e_i$. Compactness proves the
finiteness assertion. Each individual radial segment is closed off
zero, so the union is closed there.

Choose finite-dimensional pairs $A_j,B_j$, extended by zero, converging
in norm to $A,B$. We may take compressions with arbitrarily small
generic self-adjoint perturbations in their finite-dimensional spaces,
so that some pencil is invertible there and the projective roots are
simple. Their measures converge vaguely to $\mu$: the argument of
Theorem~\ref{newtr:compact} applies unchanged to these perturbations.
Indeed the eigenvalues beyond any fixed positive or negative threshold
remain uniformly finite in number by the min--max principle, and their
radial half-space tail integrals converge by norm convergence.
Use $C^j_u=u_1A_j+u_2B_j$ and define $L^j_w$ correspondingly.

Fix $w\ne0$ outside the exceptional segments. There is a neighborhood
$U$ of $w$, a bound $M$ and an index $j_0$ such that the nonreal
characteristic values of $L^j_{w'}$ and $L_{w'}$ for $w'\in U$ lie in
$|h|<M$ for $j\ge j_0$. Otherwise take $w_j\to w$, nonreal $h_j$ with
$|h_j|\to\infty$, and unit kernel vectors $e_j$. The following argument
also covers the fixed limiting pair. Real and imaginary parts of the
scalar pencil equation give
\[
 \langle C^j_{w_j} e_j,e_j\rangle=|w_j|^2,\qquad
 \langle C^j_{v_j} e_j,e_j\rangle=0,\qquad
 \|C^j_{v_j} e_j\|=O(|h_j|^{-1}).
\]
Pass to a weak limit $e$. Compactness gives $C_v e=0$ and
$\langle C_w e,e\rangle=|w|^2$. Since $0<\|e\|\le1$, its normalization
makes $w$ exceptional, a contradiction. Shrinking $U$ and increasing
$j_0$ also excludes all singular segments of the finite-dimensional
approximants from $U$. Otherwise the matrix singular-part formula
\cite{qc42:Hei} supplies $w_j\to w$ and unit vectors $e_j$ with
$C^j_{v_j}e_j=0$ and $\langle C^j_{w_j}e_j,e_j\rangle\ge |w_j|^2$.
The same weak-limit argument applies.

For clarity, multiplicities and convergence in the density formula can
be obtained from finite Schur complements. At $w$, $L_w(i(M+1))$ is
invertible: it is Fredholm of index zero and has zero kernel. On any
fixed disk in $h$ a finite-rank orthogonal projection $Q$ can be chosen
so that the compression of $C_{w'}/|w'|^2+hC_{v'}$ to $(I-Q)\mathcal H$
has norm less than $1/2$, uniformly near $w$ and for the late
approximants. Invert the complementary block. The determinant of the
resulting finite-dimensional Schur complement is analytic in $h$ and
detects invertibility with algebraic multiplicity (enlarging $Q$
multiplies the determinant by a nonvanishing analytic factor).
The determinant for the approximants converges locally uniformly to
the limiting one as $j\to\infty$ and $w'\to w$. The latter is not
identically zero on a disk containing $i(M+1)$. Choose a circle of
radius greater than $M$ free of zeros and apply Rouch\'e's theorem on
small disjoint circles about the finitely many zeros inside. Their
multisets converge, and the sum of the absolute imaginary parts is
continuous even at collisions on the real axis. This yields local
boundedness and convergence of the sums in
\eqref{r56:freddensity} for the finite-dimensional approximants.
Their densities equal these sums divided by $2\pi$ almost everywhere
by the matrix formula of \cite{qc42:Hei}. Vague convergence and the
absence of singular segments in $U$ give \eqref{r56:freddensity}
there, and the same Schur argument gives continuity for the limiting
pair.

It remains to recover the mass on a line. Use coordinates
$x=n\cdot w$, $y=d\cdot w$ and abbreviate $C=C_n$, $D=C_d$, $P=P_n$.
Take $\psi\in C_c^\infty(\mathbb R\setminus\{0\})$ vanishing for
$|y|\le\delta$. Split $D=D_0+V$ with $V$ finite rank, $D_0,V$
self-adjoint and $\|D_0\|<\delta$. The measure for $(C,D_0)$ is
supported in $|y|<\delta$ by its $y$ marginal.
Set
\[
 g(s)=\int_0^1(e^{its}-1)(1-t)\,dt/t,\qquad
 g'(s)=i\int_0^1 e^{its}(1-t)\,dt\in C_0(\mathbb R).
\]
With $\widehat\psi(z)=\int e^{-izy}\psi(y)\,dy$, the finite-matrix
projection identity, Fourier inversion in $y$ and differentiation of
a matrix trace give
\begin{equation}\label{r56:fredlocalized}
 \int e^{ikx}\psi(y)\,d\mu
   =\frac1{2\pi}\int_{\mathbb R}\widehat\psi(z)z
    \int_0^1 \operatorname{Tr}\!
     \left[Vg'\big(kC+z(D_0+sV)\big)\right]ds\,dz .
\end{equation}
To see this for the compact pair, first use simultaneous finite-rank
compressions of $C,D_0,V$. Subtract the trace transforms
$\operatorname{Tr}g(kC+zD)$ and $\operatorname{Tr}g(kC+zD_0)$ in
those finite spaces. The compensation at zero disappears under the
$z$ integral since $\psi(0)=0$, and the second measure vanishes on
the test. Vague convergence applies on the left because the supports
are bounded and the test vanishes near zero. On the right use convergence
of the operators in norm and of the compressed $V$ in trace norm,
together with the uniform bound $\|V\|_1|z|/2$ before integration
against $|\widehat\psi|$.

For any bounded self-adjoint $E$ and $f\in C_0(\mathbb R)$,
\[
 f(kC+E)\ \longrightarrow\ P f(PEP|_{P\mathcal H})P
 \quad\text{strongly as } k\to+\infty.
\]
Indeed $(kC-z)^{-1}\to-P/z$ strongly by the spectral theorem.
For $|\operatorname{Im}z|>\|E\|$ the Neumann resolvent series gives
the compressed resolvent in the limit. Products of these resolvents
and their conjugates generate a dense subalgebra of $C_0(\mathbb R)$.
Apply this result to $g'$ in \eqref{r56:fredlocalized}. Its right-hand
side converges to the same expression in $P\mathcal H$ with
$C=0$, $D=PDP$, $D_0=PD_0P$ and $V=PVP$. This is
\[
  \sum_{\lambda\ne0}m_{D_n}(\lambda)
        \int_0^1 \psi(t\lambda)(1-t)\,dt/t
\]
by \eqref{r56:fredlocalized} on the corner and the one-dimensional
spectral theorem. Finally take the Ces\`aro mean in $k$ on the left
of \eqref{r56:fredlocalized}. The finite measure on the support of
$\psi(y)$ gives the limit $\int_{x=0}\psi(y)\,d\mu$, proving
\eqref{r56:fredline}. The complement of the exceptional segments has
already been proved absolutely continuous. This finishes the
Lebesgue decomposition.
\end{proof}

The angular conclusion depends on ordinary compactness. For example,
let $\sigma$ be an atomless probability measure on a Lebesgue-null
Cantor subset of $(0,\pi/2)$. In the finite commutative tracial
algebra $L^\infty(\sigma)$ take $A(\theta)=\cos\theta$ and
$B(\theta)=\sin\theta$. The spectral theorem and the projection identity
give
\[
 \mu=\int\!\int_0^1\delta_{t(\cos\theta,\sin\theta)}
          (1-t)\,\frac{dt}{t}\,d\sigma(\theta).
\]
Its support has planar area zero and it gives zero mass to every line
through the origin. A diffuse semifinite formula would therefore have
to account for singular angular measures as well.

\begin{theorem}[Spectral tails of vertical strips in a semifinite algebra]
\label{r56:strip}
Let $A,B$ be bounded self-adjoint $\tau$-compact operators in a
semifinite tracial von Neumann algebra, and let $\mu$ be the planar
measure of Theorem~\ref{h50tr:semifinite}. For real $t$ with
$|t|\,\|B\|<1$ set
\[
 M_v(t)=(1-v)(I-vtB)^{-1/2} A(I-vtB)^{-1/2},\quad 0\le v\le1 .
\]
For any Borel set $J\subset\mathbb R$ at positive distance from zero,
\begin{equation}\label{r56:stripformula}
 F_J(t):=\int_{J\times\mathbb R}\frac{d\mu(a,b)}{1-tb}
  =\int_J\frac{da}{|a|}\int_0^1
    \tau\!\left(1_{(|a|,\infty)}
       \big(\operatorname{sgn}(a)M_v(t)\big)\right)dv .
\end{equation}
For $L>0$ with $L\ge\|B\|$, define a functional on polynomials in $b$
by $\Lambda_J(b^j)=F_J^{(j)}(0)/j!$ using the spectral expression on
the right. With $y=(b+L)/(2L)$, the strip measure
$\mu(J\times db)$ is the weak limit of the explicit positive measures
\begin{equation}\label{r56:stripbernstein}
 \sum_{k=0}^N {N\choose k}\,
       \Lambda_J\big(y^k(1-y)^{N-k}\big)\,
           \delta_{L(2k/N-1)}(db).
\end{equation}
Repeating the formula with the coordinates exchanged determines the
entire measure off the origin, including mass carried by $a=0$.
\end{theorem}

\begin{proof}
For finite trace and positive operators with $A$ invertible, the
spectral kernel in \eqref{r56:stripformula} is the one calculated by
P\u{a}unescu, R\u{a}dulescu and Thom \cite[(4.5)--(4.9)]{newtr:finite}.
Here a signed layer-cake identity extends their calculation. In any
finite algebra, for integer $r\ge1$, integrating $a^r$ against the
right-hand strip density gives
\[
 \frac1r\int_0^1\tau(M_v(t)^r)\,dv
 =\frac1r\int_0^1(1-v)^r
          \tau\!\left((A(I-vtB)^{-1})^r\right)dv .
\]
Indeed a real spectral value $\lambda$ contributes $\lambda^r/r$ from
the half-line having its sign. The coefficient of $t^s$ in the last
expression is
\[
 \frac1r\int_0^1v^s(1-v)^r\,dv
    \sum_{i_1+\cdots+i_r=s}\tau(AB^{i_1}\cdots AB^{i_r})
  =\frac{\tau(W_{r,s})}{(r+s)(r+s+1)} ,
\]
where $W_{r,s}$ is the average of all words with $r$ copies of $A$
and $s$ copies of $B$. To check the last identity, mark one of the
$r$ copies of $A$ in a word and rotate it to the first position under
the trace. This gives the composition sum multiplied by $(r+s)/r$
when the words are summed, followed by the beta integral displayed
above. The projection identity gives the same coefficient
$\int a^rb^s\,d\mu$ by polarization. Both series converge for
$|t|\|B\|<1$. Thus the two strip measures agree after multiplication
by $a^2$, by polynomial density on a bounded interval. This proves
\eqref{r56:stripformula} for finite traces and arbitrary signs.

In a semifinite algebra choose finite-trace projections $E_j$ containing
$1_{(1/j,\infty)}(|A|)$ and $1_{(1/j,\infty)}(|B|)$, for example their
join. The compressions converge in norm to $A,B$; their planar measures
converge vaguely by the proof of Theorem~\ref{h50tr:semifinite}.
The corresponding $M_{v,j}(t)$ converge in norm to $M_v(t)$ and are
$\tau$-compact in the original algebra after extension by zero.
For fixed $v,t$ their positive and negative spectral tails converge
at every positive continuity threshold of the limiting tail functions.
This follows from the variational spectral-count bounds for a
self-adjoint norm perturbation. On $|a|\ge\delta>0$ these tail counts
are uniformly bounded in $j,v$ by a fixed finite spectral tail of $|A|$:
write each compression of $M_{v,j}(t)$ as a bounded congruence of
$E_jAE_j$, with the congruence factor of norm at most
$(1-|t|\|B\|)^{-1/2}$. Truncating $A$ above
$\delta(1-|t|\|B\|)/2$ then gives an error of norm at most $\delta/2$
and a remainder supported on a projection of uniformly finite trace.
Dominated convergence in $(a,v)$ proves the equality first on continuous
compactly supported strip tests away from zero and hence on Borel
sets there. The left side of \eqref{r56:stripformula} is analytic near
zero, so its derivatives recover the stated moments without
differentiating a spectral projection. Integration of the Bernstein
polynomials of a continuous test function against $\mu(J\times db)$
gives \eqref{r56:stripbernstein} and its weak convergence. For a
Lipschitz test $f$, the test-integral error is bounded by
$L\,\mathrm{Lip}(f)F_J(0)/\sqrt N$, since the binomial variance
at $y\in[0,1]$ is at most $N/4$.
\end{proof}

\subsection{A universal planar limit for Wigner pencils}\label{r59tr:randomsec}
Hein\"avaara's Question~3.7 asks for a normalized law and a geometric
description for random tracial joint spectral measures, illustrated by
two Gaussian Hermitian matrices \cite{newtr:question}. The disk limit
for the numerical range is due to Collins, Gawron, Litvak and
\.Zyczkowski \cite{r59tr:CGLZ}. The expected order of the number of real intersections for the Gaussian determinant pencil follows from Lerario and Lundberg \cite{r59tr:LL}. Semicircular projections determine the limiting planar measure. For the GUE pencil, the transverse Gaussian matrix determines the envelope branches and segment lengths under expected counting.

\begin{theorem}[The tracial measure of a semicircular pencil]\label{r59tr:wigner}
Let $A_N,B_N$ be Hermitian $N$-by-$N$ matrices of uniformly bounded
norm. Suppose the empirical eigenvalue measures of
$H_N(\theta)=A_N\cos\theta+B_N\sin\theta$ converge to the variance-one
semicircle law for angles in a dense set. For the tracial joint measure
$\mu_N$ formed with $N^{-1}\operatorname{Tr}$, one has vague convergence
off the origin to
\[
       d\mu_\infty(x)=\frac1{8\pi}
        \left(\frac{2}{|x|}-1\right)^2
        \mathbf1_{\{0<|x|<2\}}\,dx .
\]
Weighting $\mu_N$ by $|x|^2$ and normalizing its mass gives weak
convergence to the probability density
$\frac3{8\pi}(2-|x|)^2\mathbf1_{\{|x|<2\}}$.
These limits hold almost surely for two independent normalized GUE
matrices, and also for two independent real Wigner matrices built from
nested iid upper-triangular off-diagonal arrays with centered
variance-one entries of finite fourth moment, scaled by $N^{-1/2}$,
and independent identically distributed centered diagonal entries of
finite second moment before scaling.
\end{theorem}
\begin{proof}
If $Z$ is uniform on the disk of radius two, every unit projection
$u\cdot Z$ has density $\sqrt{4-s^2}/(2\pi)$ on $(-2,2)$.
Consequently the ray average
$\int_0^1\operatorname{Law}(tZ)(1-t)\,dt/t$ has precisely the
required projected moments. At radius $r$ its density is
$(4\pi)^{-1}\int_{r/2}^1(1-t)t^{-3}\,dt$ for $0<r<2$.
By \eqref{h50tr:projection}, for $m\ge1$ the directional moment of
$\mu_N$ equals $N^{-1}\operatorname{Tr}(H_N(\theta)^m)/(m(m+1))$.
Polarization gives convergence of every mixed moment of positive
degree: evaluation at finitely many angles in the dense set determines
the coefficients of each homogeneous polynomial. The finite measures
$|x|^2\mu_N$ have uniformly compact support and bounded mass.
Polynomial approximation gives their weak convergence to
$|x|^2\mu_\infty$, whose mass is $1/3$ by the two directional second
moments. Dividing test functions supported away from zero by $|x|^2$
gives the claimed vague limit. For the random models apply the
semicircle and edge laws to $A_N,B_N$ and to each pencil on a countable
dense set of angles \cite{r59tr:AGZ,r59tr:BaiYin}.
\end{proof}

\needspace{5\baselineskip}

\begin{theorem}[Gaussian envelope curves and singular segments]\label{r59tr:curves}
For the independent GUE pair above, let $\lambda_{j,N}(\theta)$ be
the eigenvalues in increasing order. Write $u_\theta=(\cos\theta,\sin\theta)$
and let $q_{j,N}$ be the semicircle quantile at $(j-\tfrac12)/N$.
Almost surely the real envelope branches of the Kippenhahn curve satisfy
\[
 \max_{1\le j\le N}\ \sup_\theta
 \left|\lambda_{j,N}(\theta)u_\theta+
       \lambda'_{j,N}(\theta)u'_\theta-q_{j,N}u_\theta\right|
                         \longrightarrow0 .
\]
The numerical ranges converge to the radius-two disk. The longest
singular segment of $\mu_N$ has length tending to zero almost surely
(take zero when there are no segments). The expected number of
singular segments has order $\sqrt N$.
\end{theorem}
\begin{proof}
Almost surely all $H_N(\theta)$ have simple spectrum: the locus of
Hermitian matrices with a repeated eigenvalue has real codimension
at least three by unitary diagonalization. Adding one angle parameter
to this locus times an unrestricted transverse matrix still gives
a null set of pairs. The real eigenvalue envelopes are obtained by
differentiating $x\cdot u_\theta=\lambda_{j,N}(\theta)$ in $\theta$;
this gives the parametrization displayed above. On each fixed angle
mesh the semicircle and edge laws give uniform quantile approximation
in $j$, by continuity of the inverse distribution function on $[0,1]$
and convergence of both extremes. Weyl's inequality and the eventual
norm bounds extend it uniformly to $\theta$ by refining the mesh.
The maximal eigenvalue is the support function of the numerical range.

It remains to control all derivatives uniformly. Fix $\epsilon>0$
and use an angle mesh of spacing $\epsilon^2$ and a spectral mesh of
spacing $\epsilon$ in a fixed interval containing all spectra eventually.
At a mesh angle $\theta_0$, each spectral window of radius $2\epsilon$
about a mesh point has projection $P$ of rank at most
$C\epsilon N+o(N)$ by the semicircle law. The transverse matrix
$H'_N(\theta_0)$ is an independent GUE. Conditionally on $P$, every
unit-vector quadratic form on its range is Gaussian of variance $1/N$.
A $1/4$-net of size at most $9^{2\operatorname{rank}P}$ and the
Gaussian tail bound therefore give, almost surely eventually,
$\|P H'_N(\theta_0)P\|\le C'\sqrt{\epsilon}$ simultaneously on the
fixed finite meshes (the failure probability on the stated rank
bound decays exponentially in $N$). For a unit eigenvector $v$ at an
arbitrary nearby angle take the window centered closest to its
eigenvalue. Then
$\|(I-P)v\|\le\epsilon^{-1}\|H_N(\theta)-H_N(\theta_0)\|=O(\epsilon)$.
The eigenvalue derivative is $\langle v,H'_N(\theta)v\rangle$, hence
is bounded by $C'\sqrt\epsilon+O(\epsilon)$ uniformly. Letting
$\epsilon\downarrow0$ along a sequence proves the envelope estimate.

Almost surely the homogeneous polynomial $\det(sA_N+tB_N)$ has only
simple projective zeros: its discriminant is nonzero for a diagonal
matrix with distinct entries paired with the identity. Each real zero
angle modulo $\pi$ therefore has a one-dimensional kernel and a nonzero
eigenvalue derivative. By \eqref{r56:fredline} it gives exactly one
singular segment, of length $|\lambda'_{j,N}(\theta)|$ in direction
$u'_\theta$. The derivative estimate proves localization. The number
is the number of intersections of the random projective line spanned
by the pair with the determinant hypersurface. Its mean has order
$\sqrt N$ by the Gaussian-pencil integral geometry of Lerario and
Lundberg \cite{r59tr:LL}. This last assertion concerns expectation.
\end{proof}

\begin{corollary}[Lengths of the Gaussian singular segments]\label{r59tr:segmentlaw}
Write $\ell_s$ for their lengths and $M_N$ for their number. For every
bounded Borel $f$,
\[
 \frac{\mathbb E\sum_s f(\sqrt N\,\ell_s)}{\mathbb E M_N}
       =\int_0^\infty f(z)\,z e^{-z^2/2}\,dz .
\]
Almost surely, for each $C>\sqrt3$, the longest segment is at most
$C\sqrt{\log N/N}$ for all sufficiently large $N$.
\end{corollary}
\begin{proof}
Count each unoriented zero angle once, on $[0,\pi)$. For continuous
bounded $f$, coarea gives the sum as the limit of the integrals over
this interval of
$\sum_j(2\epsilon)^{-1}\mathbf1_{\{|\lambda_j|<\epsilon\}}
|\lambda'_j|f(\sqrt N|\lambda'_j|)$: zero crossings are simple and the
endpoints are almost surely nonzero. The integrals are bounded by
$2N\|f\|_\infty$, since the trigonometric polynomial
$\det(H_N(\theta)-yI)$ has at most $2N$ zeros for almost every level $y$.
At each fixed $\theta$, conditionally on $H_N(\theta)$, each real
quadratic form $v^*H'_N(\theta)v$ on a unit eigenvector has law
$N(0,1/N)$ in our GUE normalization. Thus dividing the expected
coarea integral by the one for $f=1$ gives
$\mathbb E[|Z|f(|Z|)]/\mathbb E|Z|$ with $Z$ standard normal. Dominated
convergence and uniqueness of finite measures prove the formula.
The probability of a segment longer than $C\sqrt{\log N/N}$ is at most
$(\mathbb E M_N)N^{-C^2/2}=O(N^{(1-C^2)/2})$. Summability gives the
assertion by Borel--Cantelli.
\end{proof}

\subsection{Unbounded spectral strips and the role of a power moment}
\label{r57:unboundedsec}

Fix $0<p<\infty$ and self-adjoint $A,B\in L^p(\mathcal M,\tau)$
for a normal semifinite faithful trace. The joint measure $\mu$ is
the one in Theorem~\ref{h50tr:semifinite}. Write $\mu_J(db)=\mu(J\times db)$
for a Borel set $J$ with $\delta=\operatorname{dist}(J,0)>0$.
The power identity makes this a finite measure with mass at most
$K_J=\delta^{-p}\|A\|_p^p/[p(p+1)]$.

\begin{theorem}[Unbounded strip formula]\label{r57:strip}
Suppose first that $B\ge -cI$ for some $c\ge0$. For every $s>0$ with
$sc<1$ let
\[
 R_v=(I+vsB)^{-1/2},\qquad T_v=(1-v)R_v A R_v .
\]
In the algebra of measurable operators $T_v$ is a self-adjoint
element of $L^p$. The following formula uses the two original operators:
\begin{equation}\label{r57:oneside}
 G_J(s):=\int \frac{d\mu_J(b)}{1+s b}
   =\int_J\frac{da}{|a|}\int_0^1
       \tau\!\left(1_{(|a|,\infty)}
             \big(\operatorname{sgn}(a)T_v\big)\right)\,dv .
\end{equation}
These values determine $\mu_J$, including any singular part. The same
assertion holds with signs reversed if $B$ is bounded above.

For arbitrary self-adjoint $B$, put
$B_M=B1_{\{|B|\le M\}}$, $M=1,2,\ldots$, and compute $F_{J,M}(t)$
by the spectral integral in \eqref{r56:stripformula} using $A,B_M$.
Define $\Lambda_{J,M}(b^j)=F_{J,M}^{(j)}(0)/j!$ on polynomials and
put $N=M^4$, $y_M=(b+M)/(2M)$. For each bounded Lipschitz function $f$,
\begin{equation}\label{r57:two}
 \int f\,d\mu_J
 =\lim_{M\to\infty}\sum_{k=0}^N
  f\big(M(2k/N-1)\big){N\choose k}
       \Lambda_{J,M}\big(y_M^k(1-y_M)^{N-k}\big).
\end{equation}
The weights are nonnegative. Only $B$ is truncated. Together with
the spectral marginal of $B$, the strips determine $\mu$ on the
punctured plane.
\end{theorem}

\begin{proof}
First enlarge the parameter domain of Theorem~\ref{r56:strip} for
bounded pairs in finite-trace corners. Its moment proof continues to
every real $t$ for which $I-tB$ is strictly positive. Indeed
both $\int a^r(1-tb)^{-1}\,d\mu$ and the signed layer-cake expression
\[
 \frac1r\int_0^1(1-v)^r\tau\!\left((A(I-vtB)^{-1})^r\right)\,dv
 \qquad (r\ge1)
\]
are real analytic throughout this interval containing zero. The
coordinate projection identity puts $b$ between
$\min(0,\inf\sigma B)$ and $\max(0,\sup\sigma B)$. Equality near zero
therefore gives equality on the whole interval, and polynomial density
after weighting by $a^2$ again gives the strip identity. In particular
it holds at $t=-s$ whenever $B\ge-cI$ and $sc<1$.

Now cut both spectra to $1/j<|T|\le j$ as in the proof of
Theorem~\ref{h50tr:semifinite}, obtaining $A_j,B_j$ in finite corners.
All $B_j$ have the same lower bound $-cI$. Their planar measures
$\mu_j$ converge vaguely to $\mu$ by that theorem's approximation
argument. They also have uniformly bounded $p$th moments. The
$a$-marginals of $\mu_j$ away from zero are dominated by the
$a$-marginal of $\mu$: the discarded spectral values of $A$ are simply
replaced by zero. Thus vague convergence gives convergence after
restriction to any Borel $J$ at distance $\delta$ from zero for the
bounded continuous $b$-test $(1+s b)^{-1}$ on the common half-line.
Indeed approximate $1_J$ in $L^1$ of the finite coordinate marginal
on $|a|\ge\delta/2$ by continuous compactly supported tests vanishing
near zero, and control the tails by the uniform $p$th moment.

For each $v$, the bounded factors $(I+vsB_j)^{-1/2}$ converge
strongly to $R_v$ and have norm at most $(1-sc)^{-1/2}$. Thus
$T_{v,j}\to T_v$ in $L^p$. For this multiplier assertion first
approximate $A$ in $L^p$ by bounded finite-support operators. Left
multiplication of such an operator by a strongly null, uniformly
bounded self-adjoint sequence converges to zero in $L^2$ by normality
of the trace, and hence in $L^p$ by its uniform operator norm and
finite support trace. Right multiplication follows by taking adjoints.
The ideal inequality then removes the approximation.
The spectral tails of $T_{v,j}$ converge at every positive continuity
threshold for either sign, by the perturbation inequality in the proof
of Theorem~\ref{h50tr:semifinite}. Moreover those tails are bounded by
$|a|^{-p}(1-sc)^{-p}\|A\|_p^p$. After division by $|a|$ this bound is
integrable on $|a|\ge\delta$. Dominated convergence proves
\eqref{r57:oneside}.

To recover the strip fix an admissible $s_0>0$. Under
$q=b/(1+s_0b)$ the finite weighted measure
$(1+s_0b)^{-1}\,d\mu_J(b)$ has bounded support. Its $j$th moment is
$(-1)^jG_J^{(j)}(s_0)/j!$ by differentiation under the integral.
Polynomial density determines this weighted measure. The strictly
increasing coordinate map and strictly positive weight recover
$\mu_J$.

For a bounded transverse operator $B_M$ this proof applies on
both sides of $t=0$, with $A$ still unbounded. Consequently
$F_{J,M}(t)$ is the strip Stieltjes transform for $(A,B_M)$.
Denote its strip measure by $\mu_{J,M}$. Binomial variance bounds
the difference between the sum in \eqref{r57:two} and
$\int f\,d\mu_{J,M}$ by $K_J\,\mathrm{Lip}(f)M/\sqrt N$.
Since $B_M\to B$ in $L^p$, the joint measures of $(A,B_M)$
converge vaguely to $\mu$ by the spectral-tail, tightness and
uniqueness argument of Theorem~\ref{h50tr:semifinite}. Here the
$a$-marginal is exactly the same for every $M$, so the preceding
approximation of $1_J$ also gives $\int f\,d\mu_{J,M}\to\int f\,d\mu_J$.
The weights are integrals of nonnegative Bernstein basis functions.
Finally, the mass on $\{a=0\}$ against any $b$-test supported away
from zero is obtained by subtracting the contribution from
$\{a\ne0\}$ from the $b$-marginal. That marginal is the dilation of
the spectral measure of $B$ with weight $(1-t)\,dt/t$.
\end{proof}

\begin{example}[Failure of uniqueness for unbounded $\tau$-compact pairs]
\label{r57:ambiguity}
Let the commutative algebra on $(0,\infty)\times S^1$ carry trace
$r^{-3}\,dr\,d\theta/(2\pi)$ and take $A=r\cos\theta$, $B=r\sin\theta$.
Both are measurable and $\tau$-compact, since their absolute spectral
tails are finite constant multiples of $s^{-2}$ for $s>0$.
Neither belongs to any $L^q$, $0<q<\infty$. For any antipodally
symmetric probability $\sigma$ on $S^1$ with
$\int ee^{\mathsf T}\,d\sigma(e)=I_2/2$, the measure
\[
  \mu_\sigma=\tfrac16\,(r,e\mapsto re)_*
                       \big(r^{-3}\,dr\,d\sigma(e)\big)
\]
satisfies the projection identity for the same pair.
Indeed its positive and negative directional coefficients before
multiplication by $1/6$ are both
$\int(u\cdot e)_+^2\,d\sigma=|u|^2/4$, exactly as for the angular
trace law. On the spectral side the change of radius multiplies
these coefficients by $\int_0^1t^2(1-t)\,dt/t=1/6$.
This verifies equality of the projected measures away from zero.
Uniform angular measure gives an area density, whereas equal masses
at $\{\pm e_1,\pm e_2\}$ give a measure on four rays. Thus even
the Lebesgue decomposition is not determined by the projection
identity under unbounded $\tau$-compactness alone. This is the
integer-homogeneity obstruction for exterior projection data
studied in \cite{newtr:boman}.
\end{example}

\subsection{Angular inversion for a Haagerup pencil}
\label{r57:haagsec}

In the crossed-product realization of a Haagerup space the trace on
the core differs from the Haagerup integral. The spectral homogeneity
of the core operators replaces the finite power moment used for a
semifinite $L^p$ pair.

\begin{theorem}[Haagerup pencils and angular inversion]
Let $A,B$ be self-adjoint in the Haagerup space $L^p(\mathcal M)$,
$0<p<\infty$, and let $\tau_{\rm c}$ denote the canonical semifinite
trace on the core.\label{r57:haag} For $u\in\mathbb R^2$ put $C_u=u_1A+u_2B$ and
$e_\varphi=(\cos\varphi,\sin\varphi)$. A finite positive angular
measure $\gamma$ exists such that
\begin{equation}\label{r57:ang}
 a(\varphi):=\|(C_{e_\varphi})_+\|_{\mathrm H,p}^p
       =\int(\cos(\varphi-\theta))_+^p\,d\gamma(\theta).
\end{equation}
The homogeneous measure
\[
 d\mu(re_\theta)=\frac{r^{-p-1}}{p+1}\,dr\,d\gamma(\theta)
\]
satisfies $\int f(u\cdot w)\,d\mu=\tau_{\rm c}(Hf(C_u))$
for nonnegative Borel $f$ with $f(0)=0$ and for absolutely integrable
tests. For nonintegral $p$ this is the unique positive Radon measure
on $\mathbb R^2\setminus\{0\}$ with these projections, and $\gamma$ is the weak limit of the positive
measures
\begin{equation}\label{r57:angularfejer}
 \frac{d\theta}{2\pi}\sum_{|m|<N}\left(1-\frac{|m|}{N}\right)
       \frac{\widehat a(m)}{k_m}e^{im\theta},
 \qquad k_m=\frac1{2\pi}\int_{-\pi/2}^{\pi/2}(\cos x)^p e^{-imx}\,dx,
\end{equation}
where $\widehat a(m)=(2\pi)^{-1}\int_0^{2\pi}a(x)e^{-imx}\,dx$.
For integral $p$ every positive solution is homogeneous, and the
missing angular frequencies are precisely $|m|=p+2,p+4,\ldots$. The
solutions correspond exactly to finite positive angular measures with
the prescribed Fourier
coefficients $\widehat a(m)/k_m$ at all other frequencies. In general
there are distinct positive solutions at every integral $p$.
\end{theorem}

\begin{proof}
Fack and Kosaki's distribution formula
\cite[Lemmas~1.7 and~4.8]{r57:fack} applied to the positive Haagerup
$L^1$ element $(C_u)_+^p$ gives
\[
 \tau_{\rm c}(1_{(s,\infty)}(C_u))
       =s^{-p}\|(C_u)_+\|_{\mathrm H,p}^p\quad(s>0).
\]
This formula also applies to $-C_u$. Truncate $A,B$ to
$1/j<|T|\le j$. The resulting finite corners have planar measures
$\mu_j$ by Theorem~\ref{h50tr:finite}, with the unnormalized corner
trace inherited from $\tau_{\rm c}$. Truncations converge in measure:
the small-value error has norm at most $1/j$, and the trace of the
large-value support tends to zero. The spectral-tail perturbation
bound used in Theorem~\ref{h50tr:semifinite} now gives convergence
of the positive spectral tail of every truncated pencil at each
$s>0$, since the displayed limiting tail is continuous. The same
bound for a sum, together with the two coordinate tail formulas,
gives a uniform multiple of $s^{-p}$ for each fixed pencil direction.
Consequently the finite-corner projection identity and dominated
convergence in the dilation parameter give
\[
 \mu_j\{w:u\cdot w>s\}\ \longrightarrow\
  \frac{s^{-p}}{p(p+1)}\|(C_u)_+\|_{\mathrm H,p}^p .
\]
The four coordinate tails give the uniform bound
$\mu_j\{|w|>R\}\le c R^{-p}$ at every $R>0$. Extract a vague
subsequential limit off zero. Tightness at infinity and squeezing
at the thresholds $s\pm\varepsilon$ show that it has all the
displayed half-space masses.

Every positive measure $\nu$ with these half-space masses is homogeneous,
including at integral $p$. For each integer $m$, let $T_m$ on the log
radius axis be the image of $r^p e^{-im\theta}\,d\nu(re_\theta)$.
The four coordinate tails give uniformly bounded variation on all unit
intervals. The $m$th angular Fourier coefficient of
$e^{px}\nu\{w:e_\varphi\cdot w>e^x\}$ is $K_m*T_m$, constant in $x$,
where
\[
 K_m(y)=\frac{e^{py}}{2\pi}\int_{\cos q>e^y}e^{-imq}\,dq
 \quad (q\ {\rm modulo}\ 2\pi).
\]
This kernel vanishes for $y\ge0$ and has all absolute polynomial
moments. Its Fourier multiplier is $k_m(p-i\xi)/(p-i\xi)$, by
first integrating $e^{(p-i\xi)y}$ over $y<\log(\cos q)$ where
$\cos q>0$; here $k_m(z)$ uses exponent $z$ in
\eqref{r57:angularfejer}. For $\xi\ne0$ this multiplier never
vanishes. Indeed $k_0(z)$ and $k_1(z)$ are respectively
$\mathrm B(1/2,(z+1)/2)/(2\pi)$ and
$\mathrm B(1/2,(z+2)/2)/(2\pi)$, both nonzero for $\Re z>0$;
the recurrence below also holds for such complex $z$. Convolution
against $K_m$ preserves Schwartz tests by the moment bounds, so the
multiplier identity holds for tempered distributions. It follows that
$\widehat T_m$ is supported at zero. Thus $T_m$ has polynomial density;
its uniform interval bound forces that polynomial to be constant.
By density of trigonometric polynomials the entire weighted measure
$r^p\,d\nu$ in log-polar coordinates is invariant under translation
of the log radius. This gives the claimed radial form with finite
positive $\gamma$, and its half-space masses give \eqref{r57:ang}.
For nonintegral $p$ the resulting uniqueness among all positive
solutions is also the conclusion of Boman and Lindskog
\cite[Theorem~3b]{newtr:boman}.

Writing $\widehat\gamma(m)=\int e^{-im\theta}\,d\gamma(\theta)$,
equation \eqref{r57:ang} gives $\widehat a(m)=k_m\widehat\gamma(m)$.
Here $k_0,k_1>0$, $k_{-m}=k_m$ and
$(p+m+2)k_{m+2}=(p-m)k_m$ for $m\ge0$, by integrating the
derivative of $(\cos x)^{p+1}\sin((m+1)x)$ on $[-\pi/2,\pi/2]$.
For nonintegral $p$ no coefficient vanishes, so
\eqref{r57:angularfejer} is Fej\'er convolution of the positive
measure $\gamma$ and converges weakly to it. The same recurrence
gives exactly the stated zeros at integral $p$. For an example take
$A(\theta)=\cos\theta$, $B(\theta)=\sin\theta$ in the Haagerup
$L^p$ of the commutative probability algebra on a uniform circle,
using its identification with classical $L^p$. The angular
probabilities $(1+\varepsilon\cos((p+2)\theta))\,d\theta/(2\pi)$,
$|\varepsilon|<1$, and the uniform probability have the required
convolution. Distinct choices give distinct homogeneous solutions.
\end{proof}

\begin{corollary}[Positive Haagerup pairs]\label{r57:haagpositive}
If $A,B\ge0$, the measure in Theorem~\ref{r57:haag} is unique at
every $p>0$. At integral $p$ its angular measure is recovered from
the Fourier modes of parity opposite to $p$ alone.
\end{corollary}
\begin{proof}
Vanishing of the two negative coordinate tails forces every solution
into the closed first quadrant. The difference of two solutions has
finite mass in absolute value off every ball, by the four coordinate
tails, and has vanishing exterior half-space masses. The cone support
theorem of Boman and Lindskog \cite[Theorem~4]{newtr:boman} gives zero,
since the quadrant off the origin lies in $x_1+x_2>0$.

For integral $p$ set $j\equiv p+1\pmod 2$ and let $R$ denote
antipodal reflection. All coefficients of
$\gamma_j=(\gamma+(-1)^jR_*\gamma)/2$ are known: those of parity $j$
are $\widehat a(m)/k_m$, and the others vanish. Hence its Fej\'er
sums recover $\gamma_j$ weakly. On the closed first-quadrant arc $Q$
one has $\gamma=2\gamma_j|_Q$, because $Q$ and $-Q$ are disjoint.
Equivalently multiply the Fej\'er sums by $2\chi$ for any continuous
$\chi$ equal to one on $Q$ and zero on $-Q$ before taking their limit.
\end{proof}

\subsection{An explicit density in a semifinite algebra}
Question~3.4 of \cite{newtr:question} asks for the infinite-dimensional
counterpart of the planar density formula. For the ordinary trace
Theorem~\ref{r56:fredholm} gives it directly from characteristic values.
The following inversion uses scalar spectral functions in every
semifinite algebra. It recovers the singular mass as well.
Fix an even nonnegative $\psi\in C_c^\infty(\R^2)$ supported in the
unit disk with integral one, and use
$\widehat\psi(\xi)=\int e^{-i\xi\cdot x}\psi(x)\,dx$.

\begin{theorem}[Density from spectral traces]\label{r60:trdensity}
For a self-adjoint semifinite $L^p$ pair, $p>0$, put
$k=\lceil p\rceil-1$, $K(s)=e^{is}-\sum_{j=0}^{k}(is)^j/j!$,
$C_\xi=\xi_1A+\xi_2B$, and
\[
 g_\epsilon(x)=\frac1{(2\pi)^2}\int_{\R^2}
 e^{-i\xi\cdot x}\widehat\psi(\epsilon\xi)\,
                       \tau(HK(C_\xi))\,d\xi .
\]
On $|x|>\epsilon$ this is the convolution of the tracial joint measure
with $\epsilon^{-2}\psi(\cdot/\epsilon)$.
In particular $g_\epsilon(x)$ tends to the area density almost
everywhere off zero and $g_\epsilon(x)\,dx$ tends vaguely there to the
whole measure. The same conclusions hold for bounded $\tau$-compact
pairs of arbitrary spectral decay: compute $g_\epsilon$ by the formula
with $p=1$ and $T_n=T1_{\{|T|>1/n\}}$ for $T=A,B$, and take the
limit in $n$ at fixed $\epsilon$ before shrinking $\epsilon$.
\end{theorem}
\begin{proof}
The power bound $|K(s)|\le c_p|s|^p$ and
\eqref{h50tr:power} justify absolute convergence and Fubini.
By \eqref{h50tr:projection}, the trace is
$\int K(\xi\cdot w)\,d\mu(w)$. Fourier inversion of the
exponential term at fixed $w$ gives
$\epsilon^{-2}\psi((w-x)/\epsilon)$. The subtracted monomials
give derivatives at zero of this translated kernel and hence vanish
for $|x|>\epsilon$. Vague convergence follows from approximation of
compactly supported tests. At a Lebesgue point of the area density
the absolutely continuous convolution converges to that density;
the singular contribution is bounded by
$\|\psi\|_\infty\epsilon^{-2}\mu_{\rm s}(B(x,\epsilon))$, tending
to zero almost everywhere by differentiation of measures.
For bounded $\tau$-compact pairs the truncations lie in a common
finite-trace corner at each $n$. Their measures converge vaguely by
Theorem~\ref{h50tr:semifinite} and its proof. The translated kernel
at fixed $x,\epsilon$ is a compact test off zero, giving the asserted
order of limits.
\end{proof}

For the Haagerup realization, Theorem~\ref{r57:haag} supplies the
radial density and the angular Fourier inversion at nonintegral $p$.
At integral $p$ there are missing angular modes. The following
finite-matrix prescription specifies a positive solution even in this
case. At order $N$ prescribe $b_m=\widehat a(m)/k_m$ for
$0\le m\le N$ with $k_m\ne0$, put $b_{-m}=\overline{b_m}$, and
minimize $\sum_{m=1}^N2^{-m}|b_m|^2$ subject to
$(b_{j-l})_{0\le j,l\le N}\ge0$. Insert the resulting $b_m$
in the Fej\'er density
\[
  \frac1{2\pi}\sum_{|m|\le N}(1-|m|/(N+1))b_m e^{im\theta}.
\]
Its weak limit is a canonical angular measure for
\eqref{r57:ang}. Indeed feasibility follows from the positive solution
there, and positivity gives $|b_m|\le b_0$. The strictly convex
finite minimizations have unique solutions. Every coordinatewise
cluster point satisfies all prescribed equations and Toeplitz
inequalities. Its full weighted sum is no greater than that of any
feasible infinite vector, by testing the restrictions of that vector
at every $N$ and passing to the limit in finite partial sums.
Strict convexity gives just one cluster point. The displayed densities
are positive by the Toeplitz inequality, have mass $b_0$, and converge
weakly by convergence of each Fourier coefficient. Their limit has
the prescribed convolution and hence yields the desired homogeneous
measure. Nonuniqueness of the original data at resonant orders is
exhibited after Theorem~\ref{r57:haag}.

\begin{corollary}[Real planes in Haagerup spaces]\label{h50tr:haagerup}
For every von Neumann algebra $\mathcal M$ and $0<p<\infty$, each
two-dimensional real subspace of the Haagerup space $L^p(\mathcal M)$
embeds real-linearly and isometrically in a classical real $L^p$.
\end{corollary}
\begin{proof}
For $X,Y$ in the real plane take their self-adjoint amplifications
$A=\left(\begin{smallmatrix}0&X\\X^*&0\end{smallmatrix}\right)$,
$B=\left(\begin{smallmatrix}0&Y\\Y^*&0\end{smallmatrix}\right)$,
using normalized matrix trace. By polar decomposition the $p$th
power of the amplified norm is $\|u_1X+u_2Y\|_{\mathrm H,p}^p$.
Add \eqref{r57:ang} for $u$ and $-u$ to express this norm power
as $\int_{S^1}|u\cdot e_\theta|^p\,d\gamma(\theta)$. The map to the
real function $\theta\mapsto u\cdot e_\theta$ gives the embedding.
\end{proof}

\subsection{A finite cumulant test after adjoining the identity}
\label{bridge58:sec}

Schwonnek and Werner studied the joint Wigner distribution determined
by the spectral laws of all real linear combinations of observables
\cite{bridge58:Wigner}. Here one even moment order suffices to test
positivity after retaining the identity coordinate and mixing radii.

In a finite von Neumann algebra with faithful normal finite trace $\tau$,
let $A_0=1$ and $A_1,\ldots,A_m$ be bounded self-adjoint operators,
$m\ge2$. Set $C_z=\sum_{j=0}^m z_jA_j$ and
\[
 c(z)=\sum_{r\ge2}\frac{2^{r-1}}{r}\tau(C_z^r),\qquad
 \omega_t(e^{z\cdot x})=\exp(tc(z)).
\]
The latter formal identity defines a real polynomial moment functional
for $t\ge0$. Positivity of $\omega_t$ means positivity on every
polynomial square in this specified cumulant chart. For real
homogeneous polynomials $P,Q$ of degree $D$, let
$\mathcal K_D(P,Q)=P(\partial)Q(\partial)c(0)$, with the
differential operators applied before evaluation.

\begin{theorem}[Finite-order detection of commutation]\label{bridge58:finite}
Let $q_j$ be the number of points of $\operatorname{spec}(A_j)$, with
value $\infty$ for infinite spectrum. Suppose
$D=\max_{i<j}\min(q_i,q_j)<\infty$. The following are equivalent:
\begin{enumerate}
\item The $A_j$ commute pairwise.
\item The single cumulant matrix $\mathcal K_D$ is positive semidefinite.
\item $\omega_t$ is positive along some sequence $t>0$ tending to zero.
\item For every $t>0$, $e^{tc}$ is a probability moment generating
function near zero.
\end{enumerate}
For any fixed positive Borel measure $k$ on $(0,\infty)$ with
$0<m_{2D}=\int s^{2D}k(ds)<\infty$, these conditions are also
equivalent to the existence of a positive Radon measure $\mu$ off zero
in $\mathbb R^{m+1}$ with, on $\mathbb R\setminus\{0\}$,
\[
  (x\mapsto u\cdot x)_\#\mu
  =\int (\lambda\mapsto s\lambda)_\#S_u\,k(ds),\qquad
  S_u(E)=\tau(\mathbf 1_E(C_u)),\quad u\in\mathbb R^{m+1}.
\]
The order $2D$ is sharp in general already
for two operators with $D\ge2$ spectral points each.
\end{theorem}

\begin{proof}
Define the symmetrized trace functional $L$ by $L(1)=\tau(1)$ and
$L((u\cdot y)^r)=\tau((\sum_{j\ge1}u_jA_j)^r)$ for $u\in\mathbb R^m$.
Polarization and the identity slots give
\begin{equation}\label{bridge58:hom}
 \mathcal K_D(P,Q)=2^{2D-1}(2D-1)!\,L(P(1,y)Q(1,y)).
\end{equation}
Also $\omega_t(P^2)=t\mathcal K_D(P,P)+O_P(t^2)$ by taking its
coefficient in $e^{tc}$. Thus (iv) implies (iii), which implies (ii).

Take a pair $A,B$, labelling the one with smaller spectral cardinality
$A=\sum_{i=1}^q\lambda_iE_i$, $\lambda_1<\cdots<\lambda_q$, $q\le D$.
With variables $a,b$ for the pair, cyclically count words with two
occurrences of $B$. Distinguishing one $B$ and rotating it to the front
gives $L(a^nb^2)=(n+1)^{-1}\sum_{l=0}^n\tau(A^lBA^{n-l}B)$.
Integration of a monomial on the segment between any two eigenvalues
therefore gives, for every polynomial $f$,
\begin{equation}\label{bridge58:segment}
 L(f(a)b^2)=\sum_{i,j}\tau\big((E_iBE_j)^*E_iBE_j\big)
     \int_0^1 f((1-s)\lambda_i+s\lambda_j)\,ds .
\end{equation}
For $q\ge2$ put $p(a)=\prod_i(a-\lambda_i)$ and
$h(a)=\prod_{i=2}^{q-1}(a-\lambda_i)$, taking an empty product as one.
Then $L(p^2)=0$, while $ph=(a-\lambda_1)(a-\lambda_q)h^2$ is
negative in the open spectral interval except at finitely many points.
Formula \eqref{bridge58:segment} gives $L(phb^2)<0$ if any
off-diagonal $E_iBE_j$ is nonzero. Homogenize $p$ and $hb^2$ to degree
$D$ with the identity variable. Their Gram matrix in
\eqref{bridge58:hom} has a zero diagonal and a nonzero off-diagonal,
contradicting (ii). For $q=1$ the first operator is scalar. This proves
(i) for every pair and supplies a negative square for all sufficiently
small positive $t$ whenever the tuple fails to commute.

For commuting operators the tracial joint spectral measure $S$ is
finite with bounded support in the vectors $v=(1,y)$. The compensated
Poisson integral with jump intensity
\[
 \frac t2\int S(dv)\int_0^\infty e^{-s/2}\delta_{sv}\,\frac{ds}{s}
\]
has exponential moments near zero. Its centered log moment generating
function is $tc$, since integrating the $r$th term of the compensated
exponential gives $t2^{r-1}/r$ for $r\ge2$. This proves (iv), including
when one operator has infinite spectrum. Similarly mixing
$\delta_{sv}$ against $S(dv)k(ds)$ gives the required positive radial
measure, locally finite off zero by the moment bound and bounded support
of $S$. Conversely, for any proposed radial measure $\mu$, its
coordinate projections give finite moments of order $2D$, hence
$\int |x|^{2D}d\mu<\infty$. Polarizing the projected moments gives
$\int PQ\,d\mu=m_{2D}L(P(1,y)Q(1,y))$, proving (ii).

For sharpness take diagonal real matrices $A,B$ with the
$q^2$ distinct joint eigenvalue pairs of a $q$ by $q$ Cartesian
grid. A polynomial of total degree at most $q-1$ vanishing on the
grid vanishes identically, by univariate interpolation in each
variable. Hence $\mathcal K_{q-1}$ is positive definite. Conjugate
$B$ by a sufficiently small nonzero rotation on two basis vectors
whose grid points differ in both coordinates. Its spectrum is
unchanged and its Gram matrix remains positive definite by continuity,
whereas now $[A,B]\ne0$. Homogenization makes every lower-degree test
a subtest of the one at $q-1$.
\end{proof}

\begin{corollary}[Fourth cumulants measure the common-cut correction]
\label{bridge58:energy}
For a projection $P$ and self-adjoint $B$, denote the polarized fourth
cumulant of $c$ on indicated operator slots by $\kappa_4$. Then
\begin{equation}\label{bridge58:commutator}
 \kappa_4(1,P,B,B)-\kappa_4(P,P,B,B)=8\,\tau([P,B]^*[P,B]).
\end{equation}
In the setting of Theorem~\ref{cut57:global}, suppose $0<b_n<\infty$.
For an increasing finite-rank common reducing exhaustion $F_k$,
let $d_{j,k}$ denote the difference on the left computed with the
ordinary trace after compressing $1,P_j,B$ to $F_k H$. A bounded
self-adjoint $B$ admits a Hilbert--Schmidt correction into the full
commutant $M$ if and only if $\sup_k\sum_{j\le n}d_{j,k}<\infty$, and
\[
 \min_{\substack{S=S^*\in\mathcal S_2\\B-S\in M}}\|S\|_2^2
       \le \frac{1}{8b_n}\sup_k\sum_{j\le n}d_{j,k}
\]
with optimal constant when the correction exists.
\end{corollary}
\begin{proof}
Cyclicity gives the two cumulants as $48\tau(PB^2)$ and
$16(2\tau(PB^2)+\tau(PBPB))$. Their difference is
\eqref{bridge58:commutator}. Thus $d_{j,k}/8$ is the squared
Hilbert--Schmidt norm of $F_k[P_j,B]F_k$. In a basis adapted to the
exhaustion these squared norms increase to $\|[P_j,B]\|_2^2$, possibly
infinite. A correction makes every commutator Hilbert--Schmidt, and
the converse, inequality and sharpness follow from
\eqref{cut57:sharp2}.
\end{proof}

For real symmetric matrices with ordinary trace, the centered vector
$(I_2(I),I_2(A_1),\ldots,I_2(A_m))$ has log moment generating
function $c$ by the scalar Gaussian integral. It is infinitely
divisible precisely when the $A_j$ commute. Indeed convolution roots
of an infinitely divisible law have finite exponential moments
wherever the law does, by independence and positivity of the
exponential, and here their moment generating functions are $e^{c/n}$.
Theorem~\ref{bridge58:finite} gives the converse as well.

The same argument applies to $d$-dimensional real Wishart linear
observations at every admissible positive shape $\alpha$ and positive definite scale
$\Sigma$. Retain $\operatorname{tr}(\Sigma^{-1}W)$ together with
$\operatorname{tr}(C_jW)$ for real symmetric $C_j$. Centering gives
log moment generating function $\alpha c$ with
$A_j=\Sigma^{1/2}C_j\Sigma^{1/2}$. Infinite divisibility is therefore
equivalent to pairwise commutation of these $A_j$. For the normalized
diagonals $U_i=W_{ii}/\Sigma_{ii}-\alpha$ they are rank-one projections.
With $Z=\operatorname{tr}(\Sigma^{-1}W)-\alpha d$ and
$\rho=\Sigma_{ij}/\sqrt{\Sigma_{ii}\Sigma_{jj}}$,
\[
 \kappa(Z,U_i,U_j,U_j)-\kappa(U_i,U_i,U_j,U_j)
               =16\alpha\rho^2(1-\rho^2).
\]
Indeed two rank-one projections onto unit vectors of inner product
$\rho$ have squared commutator norm $2\rho^2(1-\rho^2)$. Thus the
full diagonal vector together with the whitened radius is infinitely
divisible exactly for diagonal $\Sigma$.

Two infinite spectra can conceal noncommutativity. On the circle with
uniform measure and normalized matrix trace, take
$A=|\cos\theta|\operatorname{diag}(1,-1)$ and
$B=|\cos\theta|\left(\begin{smallmatrix}0&1\\1&0\end{smallmatrix}\right)$.
They do not commute, but the two eigenvalues of $u_0I+u_1A+u_2B$
are $u_0\pm|\cos\theta|\sqrt{u_1^2+u_2^2}$. With equal tracial weight
their law is that of $u_0+u_1\cos\theta+u_2\sin\theta$. The common
positive spectral measure is therefore the law of
$(1,\cos\theta,\sin\theta)$, and all its positive radial mixtures exist
under the indicated moment condition.

\subsection{A positive measure for the commutator}\label{r61cmp:sec}
For self-adjoint matrices $A,B$, Heinävaara \cite[(6), Lemma 2.5]{r61cmp:source} introduced
\begin{equation}\label{r61cmp:density}
 \omega_{A,B}(a,b)=\frac1{2\pi}\sum_j
 |\operatorname{Im}\lambda_j((A-aI)(B-bI))|,
\end{equation}
where eigenvalues have algebraic multiplicities. The third and fourth items of his \S5.9 ask whether the moments in his formula (15) are represented by a positive measure in every finite tracial von Neumann algebra, and asks for pointwise estimates beyond the rank-one case. We first give the moment representation, with a marginal formula that also controls approximation at infinity. Throughout, $\tau$ is a faithful normal finite trace (not necessarily normalized), and $\|T\|_{p,\tau}^p=\tau(|T|^p)$.
For a projection $A$, the total mass in Theorem~\ref{r61cmp:main} determines exactly the fourth-cumulant difference in \eqref{bridge58:commutator}: the latter is $96\,\Omega_{A,B}(\R^2)$.

For bounded self-adjoint $A$ and $B$, let $\Xi_{A;B}$ be the scalar spectral measure of the commuting left and right multiplications by $A$ on $L^2(\mathcal M,\tau)$, evaluated at the vector $B$. Thus
\[
 \Xi_{A;B}(U\times V)=\tau(E_A(U)B E_A(V)B),
\]
where $E_A$ is the spectral resolution. This is positive, of mass $\tau(B^2)$.
\begin{theorem}[Positive commutator moment measure]
There is a unique finite positive measure\label{r61cmp:main} $\Omega_{A,B}$ supported on the product of the convex hulls of the spectra of $A,B$, with moments
\begin{equation}\label{r61cmp:moments}
\begin{split}
 &\int a^k b^l\,d\Omega_{A,B}(a,b)\\
 &=\sum_{n=1}^{\min(k,l)+2}
 \frac{(-1)^{n-1}}{(k+l+2)(k+l+3)\binom{k+l+1}{n-1}}\\[-2pt]
 &\hspace{12mm}\cdot\sum_{\substack{i_1+\cdots+i_n=k+2,\ j_1+\cdots+j_n=l+2\\i_h,j_h\ge1}}
 \tau(A^{i_1}B^{j_1}\cdots A^{i_n}B^{j_n}),\qquad k,l\ge0.
\end{split}
\end{equation}
In particular this extends \cite[(15)]{r61cmp:source}. Its first marginal obeys, for bounded Borel $g$,
\begin{equation}\label{r61cmp:marginal}
 \int g(a)\,d\Omega_{A,B}=\frac12\int (u-v)^2\int_0^1 t(1-t)
 g((1-t)u+tv)\,d t\,d\Xi_{A;B}(u,v).
\end{equation}
The second marginal is given by interchanging $A,B$. Consequently,
\begin{equation}\label{r61cmp:mass}
 \Omega_{A,B}(\R^2)=\frac1{12}\tau([A,B]^*[A,B]),
\end{equation}
and the measure vanishes exactly when $A,B$ commute. Both marginals are absolutely continuous. The construction is weakly continuous under norm convergence of bounded pairs in a fixed tracial algebra.

Writing $\beta_T$ for the Brown measure with mass $\tau(1)$, this measure
has the bounded Borel density
\begin{equation}\label{r67cmp:brown}
 h_{A,B}(a,b)=\frac1{2\pi}\int_{\mathbb C}|\operatorname{Im}z|
         \,d\beta_{(A-a1)(B-b1)}(z)
 \le\frac{\|[A,B]\|_{1,\tau}}{4\pi}\qquad(a,b\in\mathbb R).
\end{equation}
On norm-bounded pairs its integral against any fixed smooth compactly
supported test is Lipschitz continuous in operator norm, with a
constant depending on the test, the norm bound, and $\tau(1)$.
\end{theorem}

\begin{proof}
For a bounded operator $T$ write $J(T)=\int|\operatorname{Im}z|\,d\beta_T(z)$,
where $\beta_T$ is its Brown measure with mass $\tau(1)$.
The upper-half-plane Haagerup--Schultz projection $e$ splits the Brown law
into its restrictions to the two half-planes, with the real axis assigned
to the complement \cite[Theorem~2.2.2]{r67cmp:projection}.
The first moment of each corner Brown measure is the trace of its corner
operator, by the expansion of its logarithmic potential at infinity.
Consequently
\[
 J(T)=\tau((2e-1)\operatorname{Im}T)\le\|\operatorname{Im}T\|_{1,\tau}.
\]
Thus the proposed density is nonnegative, satisfies the stated bound
pointwise, and vanishes outside the support rectangle: if either
self-adjoint factor is definite, its product with the other is similar
to a self-adjoint operator.

Write $\log\Delta(S)=\tau(\log|S|)$. The Brown potential identity gives,
uniformly for $\|T\|\le R$ and $L>2R$,
\begin{equation}\label{r67cmp:line}
 \pi J(T)=\int_{-L}^L
  [\log\Delta(T-t)-\tau(1)\log|t|]\,dt+O(\tau(1)R^2/L).
\end{equation}
For a scalar $z$, the symmetric full integral is $\pi|\operatorname{Im}z|$,
by horizontal translation followed by differentiation in the imaginary
part. Pairing $t$ and $-t$ bounds the discarded tail by $O(R^2/L)$.
The logarithmic singularities are locally integrable uniformly for
$|z|\le R$, so Fubini gives the formula and also measurability of $J$
on bounded operator pencils.

First suppose $A=\sum_i\lambda_i E_i$ has finite spectrum. For
$a\notin\sigma(A)$ put $U=A-a$, $T_z=U(B-z)$ and $\operatorname{Im}z>0$.
The spectrum of $T_z$ misses the real axis since
$T_z-t=U(B-tU^{-1}-z)$. The two half-plane Riesz idempotents $P_\pm(z)$
are locally holomorphic. The analytic function
$F(z)=\tau(T_z(P_+-P_-))$ has imaginary part $J(T_z)$: the trace
resolvent and the Cauchy transform of the Brown law agree off the
spectrum by differentiating the logarithmic potential, then contour
integration applies to the two separate spectral sets. The preceding
bound gives $0\le\operatorname{Im}F(b+iy)\le
\|[A,B]\|_{1,\tau}/2+y\tau(|U|)$. Its Herglotz measure is therefore
absolutely continuous with bounded density (apply the Poisson bound on
intervals). Its density is $J(U(B-b))/\pi$: in \eqref{r67cmp:line}
use determinant multiplicativity to write
\[
 \log\Delta(T_{b+iy}-t)=\log\Delta(U)
  +\tfrac12\tau\log((B-b-tU^{-1})^2+y^2).
\]
On finite $t$-intervals monotone convergence gives the ordinary
boundary value at $y=0$; its integral is finite by the same Brown
potential estimate. The uniform tail in \eqref{r67cmp:line} permits
$L\to\infty$. Write $d\eta_a(b)=J(U(B-b))\,db/(2\pi)$.

For $r=l+2$ and $D(w)=U(1-wB)$, let $D(w)_+$ be its holomorphic
spectral part near the positive eigenvalues of $U$, with zero on the
negative spectral neighborhood. At large positive imaginary $z$,
$F(z)=z\tau(2D(1/z)_+-D(1/z))$. Comparing the coefficient of
$z^{1-r}$ with the Herglotz measure $2\eta_a$ gives
\[
 \int b^l\,d\eta_a(b)=-[w^r]\tau(D(w)_+).
\]
Expansion of the contour resolvent of $D(w)$ makes this a linear
combination of the cyclic words $\tau(E_{i_1}B\cdots E_{i_r}B)$,
with coefficients depending only on $a$ and the $\lambda_i$.

These cyclic coefficients can be identified using the published
matrix formula \cite[Lemma~2.5, (15)]{r61cmp:source}. Indeed for a
color pattern $i_1,\ldots,i_r$, take an $r$-dimensional basis colored
successively by this pattern, let $E_i$ project onto its color class,
and take $B$ to be the directed cyclic shift. Its matrix trace selects
exactly that cyclic word class, with a positive multiplicity. Although
the shift need not be self-adjoint, its polynomial evaluation is a fixed
linear combination of evaluations at self-adjoint matrices: write it
as $X+iY$ with $X,Y$ Hermitian and interpolate the degree-$r$ polynomial
at $X+tY$ for $r+1$ real $t$. Thus each combined cyclic coefficient is
integrable in $a$ (the finitely many matrix slice moments are bounded
with compact support). After weighting by $a^k$ and integrating, the
matrix moment identity identifies every cyclic coefficient with the
one in \eqref{r61cmp:moments}, expanded using $A=\sum_i\lambda_i E_i$.
The same coefficients prove the identity in the given algebra.

For the marginal, the $r=2$ resolvent coefficient explicitly gives
\[
 \eta_a(1)=\frac12
 \sum_{(\lambda_i-a)(\lambda_j-a)<0}
 \frac{|(\lambda_i-a)(\lambda_j-a)|}{|\lambda_i-\lambda_j|}
                  \tau(E_iBE_jB).
\]
Indeed cyclic integration by parts in the contour for the positive
spectral part gives the coefficient
$\tfrac12\sum u_i u_j[u_i,u_j]_{\mathbf1_{(0,\infty)}}\tau(E_iBE_jB)$,
where $u_i=\lambda_i-a$, with the opposite sign for $\eta_a(1)$.
Integration over the segment between $\lambda_i$ and $\lambda_j$ gives
\eqref{r61cmp:marginal}.

For general $A$, the line formula passes directly through spectral
approximation after testing in $a$. Fix $\phi\in C_c^\infty(\R^2)$.
For almost every $b$, $V=B-b$ has $\tau(|\log|V||)<\infty$ by Fubini,
so $V^{-1}$ belongs to the affiliated determinant class. Put
$g_b(y)=\int\phi(a,b)\log|y-a|\,da$. By \cite[Propositions~2.5--2.6]{r67cmp:FK},
\[
 \int\phi(a,b)\log\Delta((A-a)V-t)\,da
  =\log\Delta(V)\int\phi(a,b)\,da+\tau(g_b(A-tV^{-1})).
\]
All integrated logarithms are absolute: the negative log singularity
is uniformly integrable in $a$ and the positive tail uses the determinant
class of $A-tV^{-1}$. The functions $g_b$ are uniformly Lipschitz;
differentiate on the smooth test in the convolution, using
$g_b'(y)=O(1/|y|)$ at infinity. The ordered spectral quantiles of
self-adjoint $S+D$ differ from those of $S$ by at most $\|D\|$ for
bounded self-adjoint $D$, by the spectral min--max principle.
Integrating the quantiles thus bounds the change of the last trace
under $A\mapsto A'$ by $\tau(1)\operatorname{Lip}(g_b)\|A-A'\|$,
uniformly in $t$. Integrate this bound over $|t|\le L$ with $L>2R$
fixed on the bounded pencils meeting the support of $\phi$.
For $t>L$ the paired logarithmic tail of \eqref{r67cmp:line} is
$-\operatorname{Re}\sum_{j\ge1}\tau(T^{2j})/(j t^{2j})$ by the
Brown potential expansion at infinity. Telescoping the powers shows
its change under $T\mapsto T'$ is bounded by
$C_R\tau(1)\|T-T'\|/t^2$. The integrable tail and the preceding
quantile bound give a Lipschitz estimate for the full Brown integral.
Factoring first by $A-a$ for almost every $a$ gives the same estimate
in the second variable.

Approximate $A$ in norm by finite-spectrum functions of $A$.
Compact support and the pointwise bound pass all polynomial moments
to the Brown density of the given pair. They give uniqueness and
weak continuity also for simultaneous norm approximation. The marginal
identity passes to continuous $g$ by continuous functional calculus of
left and right multiplication by $A$, hence to Borel $g$ by equality
of measures. Cyclicity in \eqref{r61cmp:moments} gives the swapped
marginal. Taking $g=1$ proves \eqref{r61cmp:mass}.
\end{proof}
\begin{remark}
The normalization in \eqref{r61cmp:mass} follows directly from \cite[(15)]{r61cmp:source} at $k=l=0$. The introduction of \cite{r61cmp:source} (v1, p.~4) instead displays $1/6$ times the squared commutator norm. For $A=\operatorname{diag}(1,-1)$ and $B=\left(\begin{smallmatrix}0&1\\1&0\end{smallmatrix}\right)$, \eqref{r61cmp:density} is $\pi^{-1}(1-a^2-b^2)_+^{1/2}$, of integral $2/3$, while $\tr([A,B]^*[A,B])=8$. Thus the coefficient $1/12$ is also forced by the matrix definition.
\end{remark}

\paragraph{Sharp pointwise bounds.}
\begin{theorem}[Pointwise commutator bounds]\label{r61cmp:pointwise}
For all Hermitian matrices, all real $a,b$, and the ordinary trace,
\begin{equation}\label{r61cmp:point}
 \omega_{A,B}(a,b)\le\|[A,B]\|_1/(4\pi).
\end{equation}
The constant is attained by the preceding Pauli pair. More generally,
for any projection $P$ in a finite tracial algebra the representative
\eqref{r67cmp:brown} vanishes for $b\notin(0,1)$ and satisfies
\[
 h_{A,P}(a,b)\le
 \frac{\sqrt{b(1-b)}}{2\pi}\|[A,P]\|_{1,\tau},\qquad 0<b<1.
\]
For a unit vector $v$, $P=vv^*$ and ordinary trace, put
$\sigma^2=\langle A^2v,v\rangle-\langle Av,v\rangle^2$. Then
\begin{equation}\label{r61cmp:rankone}
 \omega_{A,P}(a,b)\le \sigma\sqrt{b(1-b)}/\pi .
\end{equation}
Each slice constant is attained in dimension two. The uniform
rank-one bound in \cite[Proposition~4.5]{r61cmp:source} therefore
improves from $\operatorname{diam}(\operatorname{spec} A)/\pi$ to
$\sigma/(2\pi)\le\operatorname{diam}(\operatorname{spec} A)/(4\pi)$.
\end{theorem}
\begin{proof}
The matrix Brown measure counts eigenvalues with multiplicity, giving
\eqref{r61cmp:point}. For $0<b<1$ conjugate $(A-a)(P-b)$ by
$\sqrt{1-b}\,P+\sqrt b\,(1-P)$. Its diagonal blocks are self-adjoint;
its off-diagonal blocks are $-\sqrt{b(1-b)}\,PA(1-P)$ and
$\sqrt{b(1-b)}\,(1-P)AP$. Thus the trace norm of its imaginary part
is $\sqrt{b(1-b)}\|[A,P]\|_{1,\tau}$. Brown measure is invariant
under similarity, so the estimate for $J$ gives the result. At
$b=0,1$ the product is block triangular with self-adjoint diagonal
blocks; outside $[0,1]$ the second factor is definite.
In the rank-one case $\|[A,P]\|_1=2\sigma$. Equality is attained with
$a=0$, $A=\left(\begin{smallmatrix}0&\sigma\\\sigma&0\end{smallmatrix}\right)$
and $P=\operatorname{diag}(1,0)$. The last estimate is the scalar
variance bound on a bounded interval.
\end{proof}

\begin{theorem}[A commutator measure for a cut in infinite dimension]
\label{r61:cutmeasure}
Let $P$ be an orthogonal projection and $B$ a bounded self-adjoint
operator on a separable Hilbert space. Let $F_n\uparrow I$ be finite-rank
projections reducing $P$. Write $P_n,B_n$ for the compressions to
$F_nH$, and use ordinary trace in Theorem~\ref{r61cmp:main}.
The measures $\Omega_{P_n,B_n}$ have a finite weak limit if and only
if $[P,B]\in\Schat_2$. The limit $\Omega_{P,B}$ is independent of
the exhaustion. Its first marginal is
\[
 \tfrac12\|[P,B]\|_2^2\,a(1-a)\mathbf1_{[0,1]}(a)\,da .
\]
If $[P,B]\in\Schat_1$, the limiting measure has a planar density
bounded almost everywhere by
$\sqrt{a(1-a)}\|[P,B]\|_1/(2\pi)$ for $0\le a\le1$.
\end{theorem}
\begin{proof}
All supports lie in $[0,1]\times[-\|B\|,\|B\|]$. Their masses are
$\|F_n[P,B]F_n\|_2^2/12$ by \eqref{r61cmp:mass}, increasing to
$\|[P,B]\|_2^2/12$ in an adapted basis. Assume this is finite.
Write $B=D+T$, with $D=PBP+(I-P)B(I-P)$ and $T$ the off-diagonal part.
Then $T\in\Schat_2$. Expand each moment in \eqref{r61cmp:moments}
in these two letters after compression. The terms containing no $T$
sum to zero, since $P_n,D_n$ commute. A word containing an odd
number of $T$'s is off-diagonal and has zero trace. For a remaining
word cut the cyclic trace at two occurrences of $T$. Each resulting
factor converges in Hilbert--Schmidt norm since $T_n\to T$ in that
norm and the other compressed letters and their adjoints converge
strongly with uniform bounds. Its trace thus has an exhaustion-independent
limit. Polynomial density on the common compact rectangle proves weak
convergence and uniqueness. In \eqref{r61cmp:marginal} only the two
ordered pairs of different eigenvalues of $P_n$ contribute, with equal
weights and combined weight $\|[P_n,B_n]\|_2^2$. Passing to the limit
gives the marginal.

At $0<a<1$ apply the block similarity from the proof of
Theorem~\ref{r61cmp:pointwise} to both ranges of $P_n$, with
off-diagonal block $P_nB_n(I-P_n)$. The imaginary part of the resulting
product has trace norm $2\sqrt{a(1-a)}\|P_nB_n(I-P_n)\|_1$.
The Schur triangularization bound gives the stated estimate because
$2\|P_nB_n(I-P_n)\|_1=\|F_n[P,B]F_n\|_1\le\|[P,B]\|_1$.
Nonnegative continuous tests pass the domination to the weak limit.
\end{proof}

In \eqref{r60:gaussdistance}, the cost of each finite step with
Hilbert--Schmidt commutator is exactly twelve times the mass of
$\Omega_{P_{j_l},B^{l-1}}$. Its distribution along the cut coordinate
after normalization is always the beta $(2,2)$ law. The case of
infinite cost is detected by divergence of the compressed masses.

\subsection{Which unbounded laws admit a comparison measure under clipping?}
For compact laws $\mu,\nu$, Hein\"avaara's Theorem~1.4
\cite{r61cmp:source} gives a positive planar measure $M_{\mu,\nu}$.
Its integral of $a^2b^2 f^{(4)}(ax+by)$ equals the classical minus
the free expectation of $f$ for the laws dilated by $a,b\ne0$.
Soong proved that a common free summand preserves fourth-order convex order
for laws with finite third absolute moments, answering Question~5.8
\cite{r77:soongadd}. The first item of \S5.9 compares the classical and
free sums of the same unbounded inputs, even without any finite
moments. We first settle the scalar inequality for the full integrable
class, and then determine what remains of the positive-measure formula.
Put $\mu_R=\operatorname{Law}((-R)\vee(X\wedge R))$ for
$X\sim\mu$, and define
\[
 H(\mu)=\int_0^\infty\mu(({-}\infty,-t))\mu((t,\infty))\,dt.
\]

\begin{theorem}[Fourth-order comparison without moment assumptions]
\label{r62cmp:order}
Let $\mu,\nu$ be arbitrary probability laws on $\R$ and let
$f\in C^4(\R)$ have $f^{(4)}\ge0$. If $f$ is absolutely integrable
under $\mu*\nu$, it is absolutely integrable under $\mu\boxplus\nu$
and
\[
       \int f\,d(\mu\boxplus\nu)\ \le\ \int f\,d(\mu*\nu).
\]
Necessity of $f^{(4)}\ge0$ for a universal inequality is already
detected on compactly supported laws by \cite[Theorem~1.3]{r61cmp:source}.
\end{theorem}
\begin{proof}
Let $X,Y$ be independent with the indicated laws and clip both
at the same radius $R$. Then $X_R+Y_R$ lies between zero and $X+Y$:
for opposite signs use monotonicity and the 1-Lipschitz property
of $t\mapsto t\wedge R$ on $[0,\infty)$.
Since $f'''$ is nondecreasing, each of $f'',f',f$ is eventually
monotone separately on both tails, by successive integration. Thus
$\sup_{u\text{ between }0\text{ and }z}|f(u)|\le C+|f(z)|$.
Dominated convergence gives the desired limit of the classical terms.

We verify uniform integrability of the negative parts on the free side.
On a tail on which $f_-$ is unbounded, write $h(t)=-f(t)$ or
$h(t)=-f(-t)$, respectively. Eventually $h>0$ increases to infinity
and $h^{(4)}\le0$. Its nonincreasing third derivative is nonnegative:
any negative value would force $h$ to tend to minus infinity.
Thus $h''$ is nondecreasing. Choose $T$ large enough that $h''$
has constant sign beyond $T$. If $h$ is concave there,
$h(T+2u)\le2h(T+u)$ for $u\ge0$. Otherwise the Taylor formula
at $T$ with integral remainder involving $h'''$ gives
$h(T+2u)\le8h(T+u)$: its constant, linear and quadratic terms are
nonnegative, and after the substitution $s=2v$ the cubic integral
is at most eight times the one at $u$, by monotonicity of $h'''$.
Extend the profile constantly to $t<T$.
If this is the positive tail, conditioning on $|Y|\le K$, with $K$
large enough to have positive probability, shows $h(X-K)$ integrable.
Indeed for large $X$ the integrand on that event is bounded by
$f_-(X+Y)$, and $X,Y$ are independent. Iterating the doubling bound
now gives integrability of $h(2X_+)$, and likewise of $h(2Y_+)$.
For a negative tail apply the argument to $-X,-Y$.

Realize the two clipped laws by free self-adjoint operators in a
tracial probability algebra, and denote their sum by $Z_R$.
For either sign and $t>0$, the trace of its spectral projection satisfies
\[
 \Pr(\pm Z_R>t)\le
 \Pr(\pm X>t/2)+\Pr(\pm Y>t/2).
\]
Indeed the high spectral subspace of the sum has zero intersection
with the intersection of the two low spectral subspaces, by the
quadratic-form inequality; taking trace dimensions gives the bound.
Clipping can only decrease each one-sided tail. Layer cake therefore
bounds the tails of each increasing profile of $Z_R$ by the sum of
the tails of the corresponding integrable profiles of $2X_\pm,2Y_\pm$.
This proves uniform integrability of $f_-(Z_R)$; a bounded tail of
$f_-$ requires no estimate. Free convolution of the clipped laws
converges weakly to $\mu\boxplus\nu$ \cite{r62cmp:BV}. Truncation gives
convergence of the negative expectations, and lower semicontinuity
gives the lower limit inequality for the positive expectations.
Apply the compact comparison and the classical limit already obtained.
This also proves finiteness of the free positive expectation.
\end{proof}

\begin{theorem}[Exact two-sided tail criterion for clipping]
\label{r62cmp:tails}
For two nonconstant real probability laws,
$M_{\mu_R,\nu_S}$ has a locally finite vague limit as
$R,S\to\infty$ independently if and only if both $H(\mu)$ and
$H(\nu)$ are finite. In this case its densities converge locally
uniformly to a continuous nonnegative function $W(x,y)$ satisfying
$W(x,y)\le C_{\mu,\nu}(1+|x|)(1+|y|)$.
Otherwise the masses of a fixed compact rectangle tend to infinity.
A point-mass input gives zero for every pair of clipping levels.
\end{theorem}

\begin{proof}
Write $G_\rho(w)=\int(w-u)^{-1}\,d\rho(u)$. For $0<s<1$, set
$q_\rho(x,s)=s\operatorname{Im} w/\pi$ if
$w=x+(1-s)/G_\rho(w)$ has a root in the upper half-plane, and zero
otherwise. Its right side is a holomorphic self-map, not the identity
since $iyG_\rho(iy)\to1$. Schwarz's lemma gives at most one interior
root, which is simple. For finite atomic laws the matrix determinant
lemma identifies it with the nonreal eigenvalue through
$w=x-\lambda/s$ in \eqref{r61cmp:density} for a diagonal $A$ and a
rank-one projection with weights given by the atom masses.

At an interior root the equation and its imaginary part give
\[
 \int\frac{u-x}{w-u}\,d\rho(u)=-s,\qquad
 \int\frac{|w-x|^2}{|w-u|^2}\,d\rho(u)=1-s .
\]
Consequently roots stay bounded for $x$ in compact sets, $s$ bounded
away from zero and the input laws tight: on $|u-x|\le L$ the first
integral costs at most $L\sqrt{1-s}/|w-x|$, and on its complement of
probability $\delta$ it costs at most $\sqrt{(1-s)\delta}+\delta$.
Local uniform convergence of Cauchy transforms and Rouch\'e's theorem
now give joint continuity of $q$ under weak convergence and convergence
of $(x,s)$ with $0<s<1$; roots approaching the boundary contribute zero.
There is a stronger bound near one. If a fixed interval $[-L,L]$
carries mass at least $h>0$, the second identity gives
\[
 \sqrt h\,|w-x|\le\sqrt{1-s}(|w-x|+L+|x|).
\]
Thus $q(x,s)\le C(1+|x|)\sqrt{1-s}$ for $s>1-h/4$, uniformly
over all such laws, including all sufficiently large clippings.
Also $q(x,s)\le\pi^{-1}\int|u-x|\,d\rho(u)$ by the first identity.

For compact inputs the density equals
\[
 W_{R,S}(x,y)=\iint K(s,t)\,dQ_{R,x}(s)\,dQ_{S,y}(t),\qquad
 K(s,t)=1/\max\{st,(1-s)(1-t)\},
\]
where $Q_{R,x}(ds)=q_{\mu_R}(x,s)\,ds$ and $Q_{S,y}$ is formed from
$\nu_S$. For finite atomic laws this is
\cite[Proposition~4.7]{r61cmp:source}. For compact laws take uniformly
supported discrete approximations: joint continuity, the first-moment
bound on $q$, and $K\le[st(1-s)(1-t)]^{-1/2}$ give local uniform
convergence. The compact comparison identity for exponential tests
passes to the limit by weak continuity of free convolution
\cite{r62cmp:BV}; Fourier uniqueness identifies the limit.

The mass of $Q_{R,x}$ is $m_{\mu_R}(x)$, where
\[
 m_\rho(x)=\int_{u<x<v}\frac{(x-u)(v-x)}{v-u}\,d\rho(u)d\rho(v).
\]
Indeed the first marginal in \eqref{r61cmp:marginal} for a rank-one
projection has $\Xi_{A;P}=\rho\otimes\rho$; substitution along the
segment gives this density. Discrete approximation as above proves
the same identity for compact laws everywhere, by continuity.
With $H_x=\int_0^\infty\rho(({-}\infty,x-t))\rho((x+t,\infty))\,dt$,
layer cake gives $H_x/2\le m_\rho(x)\le H_x$.
Shifting thresholds gives $H_x\le |x|+H(\rho)$ and
$H_x\ge\int_{|x|}^\infty\rho(({-}\infty,-t))\rho((t,\infty))\,dt$.
For fixed $x$ and $R>|x|$ the clipped integrand increases to the
original one. Each segment kernel, extended by zero, is 1-Lipschitz
in $x$, so finite $m_\mu$ is continuous and $m_{\mu_R}\to m_\mu$
locally uniformly.

Assume both tail integrals finite. On $[0,1]$, $Q_{R,x_R}$ converges
weakly for every $x_R\to x$. Its total mass tends to $m_\mu(x)$,
its restriction to $(0,1)$ has limiting density $q_\mu(x,s)$,
and the square-root bound near one excludes an atom there. Hence the
remaining mass is uniquely assigned to zero. Denote this explicit limit
by $Q_{\mu,x}$, and form $Q_{\nu,y}$ in the same way.
The kernel $K$ is continuous on $[0,1]^2$ except at $(0,1),(1,0)$.
For $s<\eta$, $t>1-\eta$, with $\eta$ small, it is bounded by
$2/(1-t)$. The square-root bound and the total mass bound therefore
give $O((1+|x|)(1+|y|)\sqrt\eta)$ for this corner, uniformly for all
sufficiently large clippings. The other corner is symmetric.
Removing the two corners with continuous cutoffs and then shrinking
them proves convergence to $W(x,y)=\iint K\,dQ_{\mu,x}dQ_{\nu,y}$
also along $x_R\to x,y_S\to y$. This proves continuity, local uniform
convergence and the stated polynomial bound, uniformly for the
approximants since $H(\mu_R)\le H(\mu)$ and $H(\nu_S)\le H(\nu)$.

If $H(\mu)=\infty$, then $m_{\mu_R}(x)\to\infty$ for every $x$.
Since $K\ge1$, $W_{R,S}(x,y)\ge m_{\mu_R}(x)m_{\nu_S}(y)$.
There is a bounded interval of positive length strictly between
two separated sets of positive $\nu$-mass. On a smaller closed
interval $m_{\nu_S}(y)$ is bounded below by a fixed positive
constant for all sufficiently large $S$. Integration on a rectangle
gives divergence. The other case follows by symmetry.
\end{proof}

\begin{corollary}[Finite versus locally finite comparisons]\label{r62cmp:finite}
Suppose neither input is a point mass and the criterion of
Theorem~\ref{r62cmp:tails} holds. The limiting comparison measure
is finite if and only if both inputs have finite variance. In that
case its total mass is $\operatorname{Var}(\mu)\operatorname{Var}(\nu)/12$,
and the clipped measures converge weakly as finite measures.
\end{corollary}
\begin{proof}
Tonelli gives $\int m_\rho(x)\,dx=\mathbb E[(X-X')^2]/12$ for
independent copies. This is positive for a nonconstant law, and
infinite for infinite second moment (restrict $X'$ to a bounded set
of positive probability). The bound $W\ge m_\mu m_\nu$ proves necessity.
For compact laws, the second marginal of the rank-one slice
$q_\rho(x,s)\,dx\,ds$ is $\operatorname{Var}(\rho)s(1-s)\,ds$:
apply \eqref{r61cmp:marginal} with the projection in the first
slot, whose two ordered off-diagonal spectral weights each equal the variance.
Finite atomic approximation gives the formula for any compact law.
Since $\int_0^1 K(s,t)t(1-t)\,dt=1/2$, the first marginal of
$M_{\mu_R,\nu_S}$ has density
$\operatorname{Var}(\nu_S)m_{\mu_R}(x)/2$; interchange the two inputs
for its second marginal. When both variances are finite, the clipped
variances converge. For every $R$, the segment defining $m_{\mu_R}$
is contained in the original segment whenever it crosses $x$,
and the two distances to $x$ can only decrease; hence
$m_{\mu_R}\le m_\mu$ (with zero outside $[-R,R]$).
The two marginals are consequently tight, and local convergence gives
weak convergence and the claimed mass.
\end{proof}

\begin{corollary}[Compensation and ordinary tests]\label{r62cmp:compensation}
For all probability laws, with no moment assumption, the distributions
$\partial_x^2\partial_y^2 M_{\mu_R,\nu_S}$ converge in the tempered
distribution sense. With $\widehat\rho(t)=\int e^{itx}\,d\rho(x)$ their
limit has Fourier transform
\[
 (a,b)\longmapsto
 \widehat{D_a\mu*D_b\nu}(1)-\widehat{D_a\mu\boxplus D_b\nu}(1),
\]
where $D_a$ denotes dilation. Under the finite-tail criterion this
limit is $\partial_x^2\partial_y^2(W\,dx\,dy)$. The ordinary comparison
identity holds for every $f\in C_c^4(\R)$ and $ab\ne0$ whenever one
input is compactly supported and both satisfy the tail criterion.
It also holds with no moment condition when the two supports are
bounded below and $a,b>0$, or after reflecting the inputs and the
corresponding coefficients.
\end{corollary}
\begin{proof}
For finite clippings apply the compact comparison formula to $e^{it}$.
The transform of the differentiated measure is the displayed difference
with clipped laws, uniformly bounded by two. Pointwise convergence by
weak continuity of both convolutions gives tempered convergence.
The polynomial density bound in Theorem~\ref{r62cmp:tails} gives
tempered convergence of the undifferentiated measures as well when the
two tail integrals are finite. In the stated support cases the pullback
of the support of $f^{(4)}$ under $(x,y)\mapsto ax+by$ intersects the
joint support rectangles inside a common compact set. Thus local convergence
gives the ordinary right side, and weak convergence gives the left.
\end{proof}

For example, a law on $[0,\infty)$ with tail
$1/\log(e+t)$ has every positive moment infinite and $H=0$.
Two such inputs fall under the ordinary positive-direction comparison.
For the symmetric law with $\Pr(|X|>t)=1/\log(e+t)$, every positive
moment is again infinite but $H=\infty$; against any nonconstant
partner the undifferentiated clippings have no Radon limit.
The compensated identity retains meaning in both cases.

\begin{corollary}[Failure of the unrestricted comparison identity]
\label{r62cmp:impossible}
Let both input laws on $[0,\infty)$ have tail $1/\log(e+t)$ for $t\ge0$.
There is no positive Radon measure on the product of their closed
convex support hulls that represents the classical minus free
expectations as in \cite[Theorem~1.4]{r61cmp:source} for all real
$C^4$ functions for which both expectations are finite and all
positive dilation parameters.
\end{corollary}
\begin{proof}
Both inputs have $H=0$. For $a,b>0$ the compact comparison identity
applied to $f(t)=e^{-t}$ passes to the limit: the laws of both sums
are supported in $[0,\infty)$ and converge weakly; on the quadrant
the exponential dominates the common polynomial bound on the densities.
Thus the constructed $W\,dx\,dy$ and any proposed measure have the
same Laplace transforms for all $a,b>0$. This identifies the measures:
tilt by $e^{-x-y}$ to obtain finite measures, push them forward by
$(x,y)\mapsto(e^{-x},e^{-y})$, and use polynomial density on $[0,1]^2$.
But $W\ge m_\mu m_\nu$, and Tonelli gives
$\int m_\rho(x)\,dx=\mathbb E[(X-X')^2]/12$ for independent copies.
This is infinite here since one copy can be bounded while the other
has infinite second moment. Hence $W$ has infinite mass.
The comparison identities for $a=b=1$, $f=\sin$ and $f=\cos$ have
finite left sides. Finiteness of both ordinary Lebesgue integrals on
the right forces finite mass since
$|\sin(x+y)|+|\cos(x+y)|\ge1$, a contradiction.
\end{proof}

\subsection{A comparison measure for products}

Hein\"avaara asks whether the universal scalar comparison for sums has a
free multiplicative counterpart \cite[Section~5.9]{r61cmp:source}.
Arizmendi and Johnston give logarithmic potential comparisons \cite{r63cmp:AJ}. Soong proves preservation of a given higher convex order by a common
free multiplicative factor \cite{r63cmp:Soong}. Here the two products
being compared have the same input laws and different notions of
independence.

Write $\rho_{\rm cl}$ for the law of $XY$ when $X,Y$ are independent
with laws $\mu,\nu$, and $\rho_{\rm fr}=\mu\boxtimes\nu$ for the law
of $B^{1/2}AB^{1/2}$ for free self-adjoint $A,B$ with these laws and
$B\ge0$. Set
\[
 (\mathcal L f)(s)=s^2 f^{(4)}(s)+6s f^{(3)}(s)+6f''(s).
\]

\begin{theorem}[Multiplicative comparison]\label{r63cmp:product}
Let $\mu$ be a compactly supported real probability law and $\nu$ a
compactly supported probability law on $[0,\infty)$. Let $J$ be the
interval spanned by the products of their support points. There is
a unique finite positive measure $K_{\mu,\nu}$ on $J$ such that
\[
  \int f\,d\rho_{\rm cl}-\int f\,d\rho_{\rm fr}
        =\int_J \mathcal Lf(s)\,dK_{\mu,\nu}(s)
        \qquad (f\in C^4(\mathbb R)).
\]
Its mass is $\operatorname{Var}(\mu)\operatorname{Var}(\nu)/12$.
When $\nu$ is supported away from zero, let $\check\nu$ be its
image under $y\mapsto1/y$. Then $K_{\mu,\nu}$ is the image of
$y^{-4}M_{\mu,\check\nu}(dx,dy)$ under $(x,y)\mapsto x/y$.

Consequently the classical expectation is at least the free one
for all such pairs if and only if $\mathcal L f\ge0$ on $\mathbb R$.
If both inputs are required to be nonnegative, the necessary and
sufficient condition is $\mathcal Lf\ge0$ on $[0,\infty)$.
\end{theorem}

\begin{proof}
First assume $B$ invertible, write $S=B^{1/2}AB^{1/2}$, and let $t$
be sufficiently large. Then $S+tI$ and $A+tB^{-1}$ are strictly
positive, and
\[
 \tau\log(S+tI)=\tau\log B+\tau\log(A+tB^{-1}).
\]
Indeed the derivatives in $t$ agree by cyclicity and
$(A+tB^{-1})^{-1}=B^{1/2}(S+tI)^{-1}B^{1/2}$. For a differentiable
positive invertible path $C$, the trace derivative of $\log C$ is
$\tau(C^{-1}C')$, as follows by differentiating its resolvent
integral. The constant is identified by sending $t$ to infinity.
The classical logarithm has the same factorization, and freeness
passes from $B$ to $B^{-1}$. Applying the additive comparison
identity \cite[Theorem~1.4]{r61cmp:source} with test $\log$ on the common
positive spectral interval (smoothly extended outside it), we obtain
\[
 \int\log(t+s)\,d(\rho_{\rm cl}-\rho_{\rm fr})(s)
      =-6t^2\iint (x+ty)^{-4}\,dM_{\mu,\check\nu}(x,y).
\]
Direct differentiation yields
$\mathcal L(\log(t+\cdot))(s)=-6t^2/(t+s)^4$. Thus the proposed
image measure gives the identity for these logarithms. Expanding at
infinity, with uniform convergence of the series and its first four
derivatives on $J$, gives it for all polynomials. Uniformly approximate
$f^{(4)}$ on $J$ by polynomials and integrate four times, matching the
lower initial derivatives, to obtain it for every $C^4$ test.
The spectral interval for $S$ is indeed contained in $J$ by bounding
$A$ between its spectral endpoints and conjugating by $B^{1/2}$.

For the mass take $f(s)=s^2$, so $\mathcal L f=12$.
The centering rule for two free variables gives
$\tau(ABAB)=\tau(A^2)\tau(B)^2+\tau(A)^2\operatorname{Var}(B)$.
Subtracting from $\tau(A^2)\tau(B^2)$ proves the mass formula.
For general $B\ge0$ use $B+\epsilon I$. The comparison measures
have this same mass and a common compact support for $0<\epsilon\le1$.
The conjugated products converge in norm. Every weak limit therefore
gives the claimed formula and is supported on $J$. Uniqueness follows
because $\mathcal L(s^k)=k^2(k^2-1)s^{k-2}$ for integers $k\ge2$.

If $\mathcal L f(s_0)<0$, take nonconstant two-point input laws
concentrating near $s_0$ and $1$, respectively, with their product
interval contained in the neighborhood of strict negativity. Its
comparison measure has positive mass, so the inequality is reversed.
For nonnegative inputs the same argument at zero uses positive points
arbitrarily near zero. The other implication is immediate from positivity.
\end{proof}

In particular, ordinary convexity does not suffice even for strictly
positive two-point inputs. The polynomial
$f(s)=(s-1)^2/2-(s-1)^4/3+2(s-1)^6/15$ has
$f''(s)=(1-2(s-1)^2)^2$ and $\mathcal Lf(1)=-2$.
On a strictly positive interval the test condition has the equivalent
form $(s^3 f''(s))''\ge0$, by division by $s$.

For strictly positive inputs bounded away from zero one may take
$f\in C^4((0,\infty))$, by extending on each compact product interval.
The expectations are equal for every such pair exactly when
$\mathcal Lf=0$. Solving $(s^3 f'')''=0$ gives precisely
\[
                   f(s)=c_0+c_1s+c_2\log s+c_3/s.
\]
In particular for real $p$ the direction of the universal comparison
for $s^p$ is the sign of $p^2(p^2-1)$, in agreement with the
Araki--Lieb--Thirring power inequalities.

\paragraph{Reading the comparison measure from the two output laws.}
For sums use the original coordinate and put $h=f$.
For products of strictly positive inputs supported away from zero,
use the logarithmic coordinate $u$ and put $h(u)=f(e^u)$.
Let $\Delta$ be the classical output law minus the free output law
in the chosen coordinate, for compact inputs. In the two cases put
$\kappa=0,1$, respectively. The resulting constant-coefficient
operator and its one-sided fundamental solution are
\[
 D^4-\kappa^2D^2,\qquad
 G_0(v)=v_+^3/6,\qquad G_1(v)=\mathbf1_{v>0}(\sinh v-v).
\]
In both cases the positive comparison measure in this coordinate has
the continuous density
\begin{equation}\label{r63cmp:green}
       H(t)=\int G_\kappa(u-t)\,d\Delta(u).
\end{equation}
Indeed, for sums the additive formula projects $M_{\mu,\nu}$
onto $x+y$. For products,
$(D^4-D^2)h(u)=e^{2u}(\mathcal L f)(e^u)$, so the measure is the
logarithmic image of $s^{-2}K_{\mu,\nu}(ds)$. Both measures have
compact support. In the additive case $\Delta$ annihilates polynomials
of degrees at most three. In the logarithmic case it annihilates
$1,u,e^u,e^{-u}$ by Theorem~\ref{r63cmp:product}. These cancellations
make $H$ vanish below the support interval as well as above it.
Differentiation in distributions gives
$(D^4-\kappa^2D^2)H=\Delta$. Uniqueness among compactly supported
distributions follows since solutions of the homogeneous ordinary
differential equation are smooth exponential polynomials. Thus
$H(t)\,dt$ is exactly the positive measure in the comparison formula.
In particular its density can be read from scalar output laws without
constructing the free operators again.

The second part of the same question asks about other orders of the
letters in mixed moments. The following pointwise criterion covers
polynomials with two alternating occurrences of each variable, allowing
different scalar functions in the four slots.

\begin{proposition}[Four alternating slots]\label{r63cmp:four}
For real polynomials $f_l,g_l,h_l,k_l$ let
$P(A,B)=\sum_l c_l f_l(A)g_l(B)h_l(A)k_l(B)$ with $c_l\in\mathbb R$.
Put $\delta_{u,v}f=f(u)-f(v)$ and
\[
 T_P(u,v,w,z)=\sum_l c_l\,
 (\delta_{u,v}f_l)(\delta_{u,v}h_l)(\delta_{w,z}g_l)(\delta_{w,z}k_l).
\]
Classical expectation minus free tracial expectation is nonnegative
for all compactly supported real input laws if and only if
$T_P\ge0$ on $\mathbb R^4$. In particular two equal-weight, two-point
input laws suffice to detect any violation. The criterion also holds
when both inputs are required to be nonnegative, restricting the four
arguments to $[0,\infty)$.
\end{proposition}
\begin{proof}
The centering rule for the alternating free word with four slots gives
a classical minus free difference
$\operatorname{Cov}_{\mu}(f_l,h_l)\operatorname{Cov}_{\nu}(g_l,k_l)$.
Each covariance is one half of the expectation of the product of the
two differences at independent copies. Thus the full difference is
$\tfrac14\int T_P\,d\mu^{\otimes2}d\nu^{\otimes2}$. Conversely take
the uniform laws on $\{u,v\}$ and $\{w,z\}$; their difference is
$T_P(u,v,w,z)/16$.
\end{proof}

For an arbitrary polynomial in the two noncommuting letters, the
universal test also reduces to a finite real polynomial inequality.
If no word uses more than $d_A$ letters $A$ or $d_B$ letters $B$,
it suffices to test laws having at most
$\lceil(d_A+1)/2\rceil$ and $\lceil(d_B+1)/2\rceil$ atoms, respectively.
Indeed the free moment--cumulant formula sums over noncrossing
partitions with monochromatic blocks, and determines the expectation
from precisely these marginal moment orders. Classical expectations
use the same orders. To realize any compact marginal through degree
$2m-1$ on at most $m$ points, compress multiplication by the coordinate
to the span of $1,s,\ldots,s^{m-1}$ in its $L^2$ space and take the
spectral law of that compression $C$ at the unit constant function.
Indeed $C^j1=s^j$ for $j<m$. Splitting a power into two such powers
gives its exact moment through degree $2m-2$, and the remaining odd
moment is $\langle s^{m-1},C s^{m-1}\rangle$. The compression
spectrum stays in the marginal support interval. The test on atom
locations and nonnegative masses summing to one is polynomial, hence
can also be decided by real quantifier elimination for rational
coefficients. The pointwise four-slot test above uses just two atoms
regardless of the scalar polynomial degrees.

\section{An irrational critical order for Gaussian squares}\label{r55:cycle}
For a centered real Gaussian vector of positive definite precision $K$,
let $Y_i=G_i^2/2$ and $A_i=K^{-1/2}e_i e_i^\top K^{-1/2}$. Write
$S_u$ for the counting measure of the nonzero eigenvalues, with
multiplicity, of $\sum_i u_iA_i$. For $s>0$ put
\[
 h_s(dt)=\mathbf1_{(0,1)}(t)(1-t)^{s-1}\,dt/t,\quad
 q_s(dt)=t^{s-1}e^{-t}dt/\Gamma(s),\quad \ell(dt)=e^{-t}dt/t.
\]
Use $h_0=\delta_1$, $q_0=\ell$. A positive measure $M_s$ on the
punctured space is of order $s$ if its projection in each direction,
restricted off zero, is $\mathcal U_{h_s}S_u$.
Here $\mathcal U_k M(E)=\int M(t^{-1}E)k(dt)$.
These radial operations are Upsilon transforms \cite{pr51:upsilon}.
For $s>0$ the jump-measure class is Grabchak's $1$-tempered
$(-s)$-stable class \cite[Eq.~(6)]{p51:Grab}. We compute its least
order for the precision graphs below.

\begin{lemma}[The radial source and angular density]\label{r55:source}
A measure of order $s$ exists exactly when $Y$ is infinitely divisible
and its jump measure $\nu$ satisfies $2\nu=\mathcal U_{q_s}M$ for
some positive $M$. In that case $M=M_s$ is unique. Feasibility at one
order implies feasibility at every larger order.
\end{lemma}
\begin{proof}
Each scalar projection of $Y-\E Y$ is a sum of independent centered
gamma variables with jump measure $\mathcal U_\ell S_u/2$. For $s>0$ the
product mixture of $q_s,h_s$ has density
$r^{-1}\int_r^\infty (t-r)^{s-1}e^{-t}dt/\Gamma(s)=e^{-r}/r$.
Thus an order-$s$ measure gives all scalar jump measures. Coordinate
projections give a finite second moment for their joint mixture,
and the fully compensated exponent gives the centered law.
Conversely, any positive preimage of $2\nu$ has finite second moment
by Tonelli, since $\int t^2 q_s(dt)=s(s+1)$ for $s>0$ and
$\int t^2\ell(dt)=1$. On each angular set, weight a measure by
$|x|^2$ and map radii to their logarithms. The resulting convolution
of finite measures has Fourier multiplier
$\Gamma(s+2+i\xi)/\Gamma(s)$ for $s>0$, and $\Gamma(2+i\xi)$
at $s=0$. Both are nonzero. Fourier uniqueness gives injectivity,
also on each scalar half-line; projection and the product-mixture
identity therefore recover precisely the prescribed sources. For
$0<s<t$, multiplying a $\operatorname{Gamma}(t,1)$ variable by an
independent $\operatorname{Beta}(s,t-s)$ variable gives
$\operatorname{Gamma}(s,1)$, which provides the positive source
at $t$. From zero use $\mathcal U_{h_t}M_0$.
\end{proof}

\paragraph{A polynomial test for joint positivity.}
The spectral measures of the individual pencils determine whether a
common positive radial source exists. A single prescribed factor in
the test polynomials suffices, including when the first available trace
moment has order greater than two.

\begin{theorem}[A prescribed factor in every square test]\label{r57:factor-square}
Let $(\mathcal M,\tau)$ have a faithful normal semifinite trace, and let
$A_1,\ldots,A_m\in\mathcal M$ be bounded self-adjoint operators with
$\tau(|A_j|^{2d})<\infty$ for an integer $d\ge1$. For $u\in\mathbb R^m$
put $C_u=\sum_j u_j A_j$. Its spectral measure off zero is
$S_u(E)=\tau(\mathbf1_E(C_u))$ for Borel $E\subset\mathbb R\setminus\{0\}$.
Fix $a\ge0$ and put
$b_{r,a}=\mathrm B(r,a)$ if $a>0$, $b_{r,0}=1$. On polynomials all
of whose monomials have degree at least $2d$, define $L_a$ by
polarization from
\[
 L_a((u\cdot x)^r)=b_{r,a}\,\tau(C_u^r),\qquad r\ge2d .
\]
Fix any real polynomial $g$ with no zeros off the origin and whose
monomials have degree at least $d$. There is a positive Radon measure
$M$ on $\mathbb R^m\setminus\{0\}$ whose punctured projection in every
direction $u$ equals $\mathcal U_{h_a}S_u$ if and only if
\[
                     L_a\big((gp)^2\big)\ge0
              \qquad \text{for all real polynomials }p .
\]
Such $M$ is unique. In particular, if it fails to exist, then for every
integer $k$ with $2k\ge d$ there is a polynomial $p$ such that
$L_a((|x|^{2k}p)^2)<0$. Changing finitely many low-order tensors
cannot remove all these witnesses.
\end{theorem}
\begin{proof}
Write $N=2d$ and $R_N(w)=e^w-\sum_{j=0}^{N-1}w^j/j!$.
The regularized Fourier transform
\[
                  F(z)=\int \tau(R_N(it C_z))\,h_a(dt),
           \qquad C_z=\sum_{j=1}^m z_jA_j,\quad z\in\mathbb C^m ,
\]
is entire. Indeed the integral remainder is
$(itC_z)^N\int_0^1(1-s)^{N-1}e^{istC_z}\,ds/(N-1)!$.
Trace H\"older puts $C_z^N$ in $L^1(\tau)$, the exponential has norm
at most $\exp(t\|C_{\operatorname{Im}z}\|)$ as in
Theorem~\ref{h50tr:finite}, and $\int t^N h_a(dt)=b_{N,a}<\infty$.
Thus $|F(z)|\le C(1+|z|)^N\exp(\sum_j\|A_j\||\operatorname{Im}z_j|)$.
The Paley--Wiener--Schwartz theorem gives a real distribution $T$
supported in $\prod_j[-\|A_j\|,\|A_j\|]$ with
$\langle T,e^{iz\cdot x}\rangle=F(z)$. Differentiating at zero
identifies its polynomial moments in degrees $r\ge N$ with $L_a$.

For each $u$, the measure $\lambda_u=\mathcal U_{h_a}S_u$ has
bounded support and finite absolute $N$th moment.
The pushforward distribution of $T$ along $u$ is
\[
 f\longmapsto\int\left(f(y)-\sum_{j=0}^{N-1}
                       \frac{f^{(j)}(0)}{j!}y^j\right)\lambda_u(dy).
\]
The integral is finite by Taylor's formula, and its Fourier transform
is $F(tu)$ by spectral calculus. In particular this distribution agrees
with $\lambda_u$ off zero.

Suppose the square tests hold. The compact distribution $g^2T$ is
nonnegative on polynomial squares, hence is a positive measure.
For details, if $f\in C_c^\infty$ is nonnegative and $\epsilon>0$,
approximate $\sqrt{f+\epsilon}$ on a neighborhood of the support,
with all derivatives through the distribution's finite order, by real
polynomials. One can first insert a smooth cutoff equal to one there,
convolve with a narrow Gaussian, and approximate the resulting entire
function by its Taylor polynomials. Positivity on squares passes to
$f+\epsilon$, then to $f$ as $\epsilon\downarrow0$. The positive
distribution is therefore a measure. Division by $g^2$ off zero makes
the restriction of $T$ there a positive Radon measure $M$. Testing
the preceding pushforward identity on functions supported away from
zero gives its prescribed projections.

Conversely any such $M$ is supported in the same rectangle and has
finite absolute $N$th moment, by its coordinate projections.
Polarization of its projected moments gives
$L_a(q)=\int q\,dM$ for every polynomial $q$ with no terms below $N$.
This proves the tests. Those moments also specify all polynomial
moments of the finite measure $|x|^N M(dx)$ in the rectangle.
Compact moment uniqueness and division off zero give uniqueness of $M$.
Finally take $g=|x|^{2k}$ with $4k$ exceeding any prescribed cutoff
to obtain the last assertion.
\end{proof}

Take $r=\sum_i x_i$ on the positive orthant and $\theta=x/r$.
A positive measure of order $s>0$ gives a radial density $n_B$ in
each angular set $B$ with
\[
  \frac{2\Gamma(s)n_B(r)}{r^{s-1}}
  =\int_{v/|v|_1\in B}|v|_1^{-s}e^{-r/|v|_1}\,M_s(dv).
\]
For $s=0$ the corresponding completely monotone function is $2rn_B(r)$.
The source is on the positive orthant because its positive dilations
give a jump measure there. On an open coordinate face with a continuous
polar density $n_\theta(r)$, shrinking angular windows and dividing
by their volume shows that $n_\theta(r)/r^{s-1}$ is completely
monotone at each angle. Conversely, suppose this function is given by
a measurable positive Laplace mixture $\int e^{-rz}Q_\theta(dz)$ with
$z>0$. Push $2\Gamma(s)z^{-s}Q_\theta(dz)d\theta$ forward by
$(\theta,z)\mapsto\theta/z$ to obtain the source on this face.
Summing over faces produces $2\nu=\mathcal U_{q_s}M$.

Suppose a sign change makes all off-diagonal precision entries
$-h_{ij}\le0$. For $c>\|K\|$, $Q=I-K/c$ is entrywise nonnegative
of spectral radius less than one. Put $z_i=c/(c+\lambda_i)$.
The log Laplace transform of $Y$ consists of $\frac12\sum_i\log z_i$
and
\[
 \frac12\sum_{k\ge1}\frac1k\sum_{i_1,\ldots,i_k}
     \left(\prod_{j=1}^k Q_{i_j i_{j+1}}\right)
     \left(\prod_{j=1}^k z_{i_j}-1\right),\qquad i_{k+1}=i_1 .
\]
This is the log-determinant series. Its word intensities sum finitely
by $\sum_k\tr Q^k/k<\infty$. Each word gives independent gamma
coordinates of rate $c$, with shapes their visit counts and unvisited
coordinates zero. Hence $Y$ is infinitely divisible and the jump
measure restricted to any coordinate subspace is that for the
principal precision there: precisely the words staying inside survive.
Disconnected open faces receive no jumps. On an axis the density is
$e^{-K_{ii}r}/(2r)$, by the one-coordinate principal precision.

\begin{theorem}[Least orders on rank-one bipartite precision graphs]\label{r53:four}
Let a nonsingular centered Gaussian vector have precision graph $K_{p,2}$,
$p\ge2$, with vertex parts $P,Q$ of sizes $p,2$. After a coordinate
sign change suppose its cross entries are $-b_i b_j$, $i\in P$, $j\in Q$,
with $b_i>0$. Its half-squares admit a positive spectral measure of
order $s\ge0$ exactly when $s\ge s_p=2p+1+\beta_p$, where
$\beta_p\in(1/2,1)$ is the unique root of
\begin{equation}\label{r53:orderroot}
 \beta^2+(8p-9)\beta-4(p-1)
 +4(p-1)(p-2)\frac{4\beta^2-1}{(\beta+1)(\beta+2)}=0 .
\end{equation}
In particular the four-cycle with equal alternating products of edge
weights has the irrational least order $s_2=(3+\sqrt{65})/2$.
The fractional parameters $\beta_p$ decrease to $1/2$.
The diagonal entries are arbitrary subject to positive definiteness.
\end{theorem}
\begin{proof}
Write $a_i=K_{ii}$, $d_i=a_i+\lambda_i$ with $\lambda_i\ge0$ and
$V_B=\sum_{i\in B} b_i^2/d_i$ for $B=P,Q$. The block determinant is
$\det(K+\operatorname{diag}\lambda)=\prod_i d_i(1-V_PV_Q)$ by the
rank-one Schur complement. Positive definiteness gives $V_PV_Q<1$.
Besides the axis terms, the log Laplace transform of the square law
therefore has the series $\tfrac12\sum_{k\ge1}V_P^kV_Q^k/k$
minus its value at zero. Each term has a positive jump measure,
since $d_i^{-l}$ inverts to $e^{-a_ix_i}x_i^{l-1}/(l-1)!$ for
$l\ge1$. The masses of the non-axis terms sum finitely.

Use $r=\sum_i x_i$, $\theta_i=x_i/r$ and
$A=\sum_i a_i\theta_i$. For independent uniform vectors $U_i$ on the
unit sphere in $\R^4$, put
\[
 W_B=\left|\sum_{i\in B}b_i\sqrt{\theta_i}\,U_i\right|^2,
 \qquad C=\sqrt{W_PW_Q}.
\]
For a nonnegative variable $v$, the function
$F(\sqrt v)=\sum_{l\ge0}v^l/(l!(l+1)!)$ is the moment generating
function of the first spherical coordinate at $2\sqrt v$.
Independence and rotational invariance consequently give
$\prod_{i\in B}F(b_i\sqrt{\theta_i v})=\E F(\sqrt{W_B v})$.
For $d=|B|$, inverting $V_B^k$ on its open face and using the
multinomial expansion gives
\[
 k!e^{-\sum_{i\in B}a_ix_i}\Big(\prod_{i\in B}b_i^2\Big)
       \frac{r^{k-d}\E W_B^{k-d}}{(k-d)!(k-d+1)!},\qquad k\ge d .
\]
Indeed the required generating factor at coordinate $i$ is
$\sum_{l\ge1}b_i^{2l}x_i^{l-1}v^l/(l!(l-1)!)
 =b_i^2 v F(b_i\sqrt{x_i v})$. Multiplying the two face formulas
and using the Jacobian $r^{p+1}$ yields
\begin{align}
 n_\theta(r)&=\tfrac12\Big(\prod_i b_i^2\Big)r^{2p-1}e^{-Ar}
               \E\!\left[W_Q^{p-2} B_p(2Cr)\right],\label{r53:cycleprofile}\\
 B_p(t)&=(2p-2)I_0(t)+\tfrac t2 I_1(t)
              +(p-1)(p-2)\,\frac{2I_1(t)}t .\nonumber
\end{align}
Here $I_0(t)=\sum_{l\ge0}(t^2/4)^l/(l!)^2$ and $I_1=I_0'$;
the quotient has its continuous limit at zero. Indeed after
$k=l+p$ the coefficient of $(C^2r^2)^l$ in the series equals
\[
 \frac{1}{(l!)^2}\frac{(l+p)(l+p-1)}{l+1}
 =\frac{l+2p-2+(p-1)(p-2)/(l+1)}{(l!)^2}.
\]
The rates satisfy
$A-2C\ge(\sqrt\theta)^\top K\sqrt\theta\ge\lambda_{\min}(K)$.

It suffices initially to test $s=2p+1+\beta$ with $1/2<\beta<1$.
For fixed $C>0$, put $t=2Cr$. Apart from a positive constant and
the exponential $e^{-Ar}$, dividing the integrand of
\eqref{r53:cycleprofile} by $r^{s-1}$ gives
$t^{-1-\beta}B_p(t)$.
Write $\rho(dy)=\mathbf1_{(-1,1)}(y)dy/(\pi\sqrt{1-y^2})$.
The arcsine moments give $I_0(t)=\int e^{-ty}\rho(dy)$ and
$I_1(t)=-\int y e^{-ty}\rho(dy)$. The semicircle moments give
$2I_1(t)/t=2\int(1-y^2)e^{-ty}\rho(dy)$. For $x>-1$ define
\[
 J_\gamma(x)=\int_{-1}^{\min(x,1)}(x-y)^\gamma\,\rho(dy),\qquad
 \widetilde J_\beta(x)=\int_{-1}^{\min(x,1)}
          (x-y)^\beta(1-y^2)\,\rho(dy).
\]
Convolution with the gamma power kernel on the positive half-line
gives the inverse Laplace density on $(-1,\infty)$,
\begin{equation}\label{r53:cycledensity}
 2\Gamma(\beta+1)H_{p,\beta}(x)
 = (4p-4+\beta)J_\beta(x)-\beta xJ_{\beta-1}(x)
       +4(p-1)(p-2)\widetilde J_\beta(x).
\end{equation}
For $x\le0$ the density is positive. For $x>0$ both ratios
$J_\beta/(xJ_{\beta-1})$ and $\widetilde J_\beta/(xJ_{\beta-1})$
have their minimum at $x=1$. For $0<x<1$, put $y=-1+(x+1)z$.
The first ratio is
$(x+1)\E_x(1-z)/x$ for a probability density on $(0,1)$ proportional to
\[
 z^{-1/2}(1-z)^{\beta-1}\big(1-(x+1)z/2\big)^{-1/2}.
\]
As $x$ increases the likelihood ratio increases in $z$, so
$\E_x(1-z)$ decreases. The prefactor decreases as well. The second ratio is
$(x+1)^2/x$ times a quotient of integrals on $(0,1)$.
In the numerator the integrand is
$z^{1/2}(1-z)^\beta(2-(x+1)z)^{1/2}$, and in the denominator it is
$z^{-1/2}(1-z)^{\beta-1}(2-(x+1)z)^{-1/2}$.
The prefactor and numerator decrease, and the denominator increases.

For $x>1$ scale $x-y=x(1-y/x)$. In either ratio the numerator then
averages $(1-y/x)^\beta$ under a symmetric positive measure, and
the denominator averages $(1-y/x)^{\beta-1}$ under $\rho$.
Pairing $y$ and $-y$ shows that the former increases with $x$ by
concavity and the latter decreases by convexity. Both ratios are
continuous at $x=1$ since $\beta>1/2$. There the preceding variable
$z$ has beta parameters $1/2,\beta-1/2$. The first ratio is
$2\E(1-z)=(2\beta-1)/\beta$; the second is
$8\E[z(1-z)^2]=(4\beta^2-1)/(\beta(\beta+1)(\beta+2))$.
It follows that $H_{p,\beta}\ge0$ exactly when the left side of
\eqref{r53:orderroot} is nonnegative. That expression increases
strictly in $\beta\in(1/2,1)$, starts at $-1/4$ and ends positive.
It also increases strictly in $p$ at each fixed $\beta>1/2$ and
tends to infinity there, proving the assertions about its roots.

At $\beta=\beta_p$, this positive density gives a measurable
Laplace representation of $n_\theta(r)/r^{s-1}$:
in \eqref{r53:cycledensity} send $x$ to rate $A+2Cx$, multiply by
$(2C)^{1+\beta}W_Q^{p-2}$ and average over the $U_i$, retaining the
prefactor $\prod_i b_i^2/2$. At $C=0$ the integrand is a nonnegative
constant times $r^{-1-\beta}e^{-Ar}$, which has a positive rate
mixture directly. All rates are bounded
below by $\lambda_{\min}(K)$. Weight
the resulting rate measure by $2\Gamma(s)z^{-s}$ and push
$(\theta,z)$ to $\theta/z$ to obtain the positive full-face source.
By the word expansion, an open face uses its principal precision.
A proper connected induced subgraph is a star on at most $p+1$
vertices or another $K_{j,2}$ with $2\le j<p$. For a star with
$d\ge1$ leaves take $|P|=d$, $|Q|=1$ in the same block determinant
and generating factors. The full-face density divided by $r^{2d-1}$
is a positive constant times
$e^{-Ar}\E[W_Q^{d-1} F(rC)]$: the coefficient after $k=l+d$
is $1/(l!(l+1)!)$. The spherical formula gives a positive rate
mixture. Thus the star face admits order $2d$ and every larger order
by multiplying by a completely monotone power of $r$. On axes use
the preceding gamma density. For $K_{j,2}$ induct on $p$, since
$s_p>s_j$. Restrict the positive sources to
the respective open faces; disconnected faces carry no jumps.
Adding the sources gives $2\nu=\mathcal U_{q_{s_p}}M$, with finite
second moment by Tonelli. Lemma~\ref{r55:source} gives feasibility
at $s_p$ and every larger order.

Conversely, feasibility makes $n_\theta(r)/r^{s-1}$ completely
monotone at every interior angle, by shrinking angular windows in
the angular density formula and using continuity. Let the angle tend to a
point supported on one vertex of each part, both coordinates positive.
Then both $C$ and $W_Q$ tend to positive deterministic numbers, so
the fixed-$C$ function above must be completely monotone. If
$1/2<\beta<\beta_p$, its continuous inverse density is
negative near $x=1$. Translation to rates $A+2Cx$ preserves the
sign. The signed density has finite exponentially damped mass at
every positive argument, so uniqueness of Laplace transforms
excludes any positive representation. This excludes every smaller
order as well by the nesting in Lemma~\ref{r55:source}.
\end{proof}

\begin{corollary}[Arbitrary trees attached to the cycle]\label{r53:prune}
Attach $t$ vertices in arbitrary trees to the rank-one $K_{p,2}$
of Theorem~\ref{r53:four}, with arbitrary nonzero edge weights on
the trees and any positive definite precision respecting the signs.
The least order is $s_p+2t$. For the four-cycle with equal alternating
edge products and $n$ total vertices it is
$2(n-4)+(3+\sqrt{65})/2$.
\end{corollary}
\begin{proof}
Let $u$ be a new leaf, $v$ its neighbor, $a=K_{uu}$ and
$h=|K_{uv}|$. For a connected open face containing $u$ and $v$,
let $S$ be its vertices other than $u$, and use the corresponding
principal matrix. Denote by $f_S$ its Cartesian jump density at the
original principal precision $K_{SS}$. Schur complementation at $u$
gives the determinant factor $a+\lambda_u$ and the remaining
precision $K_{SS}+\operatorname{diag}\lambda_S-
 h^2(a+\lambda_u)^{-1} e_v e_v^\top$. Differentiation of the log
Laplace transform in $\lambda_v$ therefore identifies the transform
of $x_v\nu$ on the enlarged space with the transform of
$x_v\exp(h^2 x_v/(a+\lambda_u))\nu_S(dx_S)$ in the old coordinates.
First take $\lambda_S$ sufficiently large to justify the shifted
transform. Expanding the positive exponential and inverting in
$\lambda_u$, the constant term stays at $x_u=0$ and the open face
has density
\[
 f_{S\cup\{u\}}(x)=h^2x_v e^{-a x_u}
       F(h\sqrt{x_vx_u}) f_S(x_S).
\]
The strictly positive multiplier $x_v$ from differentiation cancels
on the open face. Exponentially damping the measures in the old
coordinates before inversion ensures uniqueness on the whole face.

For any $y_S>0$ write the ray test in Cartesian form as
$r^{|S|-s}f_S(ry_S)$; this differs from the normalized angular test
only by a positive constant and a rescaling of $r$. If the enlarged
face passes at order $s$, let $y_u\downarrow0$. Since $F(0)=1$, the
old test at $s-2$ is a pointwise limit of completely monotone functions
after division by $h^2y_v$. This necessity applies also to real ray
exponents below zero. On an axis the test $r^{-s}e^{-K_{ii}ry_i}$
passes exactly at $s\ge0$. The proof of Theorem~\ref{r53:four}
excludes all ray exponents below $s_p$ on the full bipartite face
by multiplication by negative powers of $r$.

For sufficiency, use the positive definite Schur precision
$K'=K_{SS}-(h^2/a)e_ve_v^\top$. Its square law is the exponential
tilt of the principal law at $K_{SS}$ by $\exp(h^2Y_v/a)$, hence
$f_S(x)=e^{-h^2x_v/a}f'_S(x)$ for their jump densities
(by the tilted L\'evy exponent). Relative to an explicit positive
rate mixture for the old face at order $s-2$ with precision $K'$,
the new test multiplies by a constant and
\[
 e^{-r(a y_u+h^2y_v/a)}F(hr\sqrt{y_uy_v}).
\]
By the semicircle representation this factor is a positive Laplace
mixture with rates at least
$(\sqrt{a y_u}-h\sqrt{y_v/a})^2\ge0$.
The old mixtures have strictly positive rates and are measurable in
the angle, so multiplying them gives the required face source.

Induct on the number of attached leaves. On a connected face
containing the entire bipartite core, prune its attached leaves to
that core. The upper construction applies for every positive definite
choice of diagonals, and the boundary test gives the matching lower
bound on the full face. If the observed part of the core is a
proper $K_{j,2}$, use $s_j<s_p$. If it is a star, use order at most
$2p$, and if it is empty the face is a tree, with order
$2(|S|-1)$ by pruning to an axis. A connected face meets the core
in a connected set whenever it meets it, since each attached tree
has only one attachment point. All these face sources therefore
exist at order $s_p+2t$; their sum and Lemma~\ref{r55:source}
give sufficiency, and the full-face boundary test gives necessity.
\end{proof}

\section{Gaussian squares and quantum covariance spectra}\label{r54:shifts}
\subsection{Infinite divisibility of shifted Gaussian squares}

Kogan \cite[Section~5]{q47:Kogan} asks whether a Gaussian vector whose
centered squares fail infinite divisibility can acquire it after a
translation. Kogan proved exclusions under additional sign and row-sum
conditions in her Lemma~5.1. Marcus and Rosen characterized divisibility
along every multiple of a prescribed shift \cite{r49:MR08}. We first treat nonsingular covariance; the following corollary includes all ranks.
\begin{theorem}[A necessary condition independent of the mean]\label{q49:norescue}
If the coordinate squares of $X\sim N(Q^{-1}h,Q^{-1})$, $Q>0$ real
symmetric, are jointly infinitely divisible, there exists a diagonal
sign matrix $S$ such that $SQS$ is an $M$-matrix. In particular the
coordinate squares of $X-\E X$ are infinitely divisible.
\end{theorem}
\begin{proof}
Suppose the squares are infinitely divisible. Every real diagonal
perturbation $Q\mapsto Q+D$ that keeps $Q+D>0$ preserves this property
at the same $h$: multiply the Gaussian density by
$\exp(-\sum_i D_{ii}X_i^2/2)$ and normalize. This Esscher tilt preserves
infinite divisibility because each convolution root has a finite
normalizer and can be tilted by the same functional. The property also
passes to a principal precision block $Q_I$ with source $h_I$: tilt by
adding $t$ to the complementary diagonal and take the weak limit of the
marginal on $I$ as $t\to\infty$, using the Schur complement. Ordinary
Gaussian marginalization also preserves divisibility of the retained squares.

Write $c_{ij}=-Q_{ij}$ on the nonzero edges. A diagonal sign change makes
all edge weights positive exactly when every cycle has positive sign
product: assign vertex signs along a spanning tree in each component.
A shortest negative cycle has no chord, since a chord cuts it into two
shorter cycles whose signs multiply to its sign. Pass to its principal
precision block. On a chordless cycle of length at least four, marginalize
a vertex $v$ with neighbours $i,j$. The Schur complement joins the
previously nonadjacent neighbours by the new weight $c_{iv}c_{vj}/Q_{vv}$.
The remaining edges retain their weights, giving a shorter chordless
negative cycle with some source. Iteration reduces the obstruction to a
triangle, with arbitrary source.

On a negative triangle an invertible real diagonal change of Gaussian
coordinates makes the three off-diagonal precision entries $+1$:
the required signs are consistent around the triangle, and the three
positive coordinate factors can realize any prescribed three edge
magnitudes. This only rescales the coordinate squares. Fix any vertex
$k$; tilt so that its diagonal is $z$ near $1$ and the other two are
a common $M>1$. These precisions are positive definite near $z=1$
by their Schur complements. After marginalizing $k$, the edge weight
is $c=1/z-1$ and the two sources are $h_i-h_k/z$, $h_j-h_k/z$.
For the required two-vertex necessity, consider the mixed covariance. Add $tI$ to the reduced precision and let $R$ be its
inverse and $m$ the new mean. Infinite divisibility on the positive
orthant implies nonnegative mixed covariance, while Wick's formula
gives $2R_{ij}^2+4R_{ij}m_i m_j$. Multiply by $t^4$ and let
$t\to\infty$: $t^2R_{ij}\to c$, $tm_i\to h_i-h_k/z$.
For $c\ne0$ this requires
\[
 (h_i-h_k/z)(h_j-h_k/z)/c\ \ge-1/2 .
\]
Letting $z$ approach $1$ from both sides yields
$(h_i-h_k)(h_j-h_k)=0$. As this holds for each $k$, the three sources
must equal $a\mathbf1$, for some real $a$.

Finally tilt this triangle to precision $\epsilon I+\mathbf1\mathbf1^{\mathsf T}$
at the same source and rescale all Gaussian coordinates by $\sqrt\epsilon$.
As $\epsilon\downarrow0$ the mean tends to zero and the covariance tends
to the orthogonal projection on $\mathbf1^\perp$. For a centered real Gaussian vector with covariance $R$, Wick's
formula gives third mixed cumulant $8R_{12}R_{23}R_{31}$ for its
coordinate squares. In this limit it equals $-8/27$. Every
infinitely divisible law on the positive orthant with these finite
moments has nonnegative third mixed cumulant, namely
$\int x_1x_2x_3\,d\nu(x)$ for its L\'evy measure. This contradicts
weak closure. Thus all cycles are positive and the stated
sign change exists. The centered conclusion follows by the positive
walk expansion or the Griffiths--Bapat criterion \cite[Theorem~1.1]{q47:Kogan}.
\end{proof}

\begin{corollary}[Singular covariance and arbitrary means]\label{q49:singular}
For a Gaussian vector of any rank, infinite divisibility of its coordinate
squares implies infinite divisibility of its centered coordinate squares.
More precisely, in the singular case delete constant coordinates and choose
one representative of each class of proportional nonzero centered Gaussian
coordinates. Joint infinite divisibility of the original squares holds
exactly when the representatives have nonsingular covariance and infinitely
divisible squares at their original means, and all coordinates in each
class are proportional almost surely at their original means as well.
\end{corollary}
\begin{proof}
An infinitely divisible probability on $\R^m$ has support whose Zariski
closure is its affine hull. Indeed work on this hull after translation.
For a nonzero polynomial $P$ of degree $d\ge1$ vanishing on the
support, write the law as the sum of $d$ independent identically
distributed roots. The support contains every sum of $d$ root support
points. Take the $d$-fold difference of $P$ on these sums, independently
choosing two alternatives for each summand. This gives $d!$ times the
polarization of the top homogeneous part evaluated on the $d$ arbitrary
differences of root support points. The differences span the direction
space of the original support: the support of the sum is the closure
of the set of these sums and one can keep all but one summand fixed.
Multilinearity then kills the entire top homogeneous part on the hull,
a contradiction.

Write the coordinates as $X_i=m_i+\ell_i(g)$ for a standard Gaussian
on its support parameter space. If distinct nonproportional forms
$\ell_i$ have a minimal linear dependence of size $k$, then $k\ge3$
and, after relabeling, $\ell_1,\ldots,\ell_{k-1}$ are independent,
while $\ell_k$ is their linear combination with all coefficients nonzero.
The $k$ quadratic forms $\ell_i^2$ are linearly independent, because
$\ell_k^2$ has a mixed term in these basis coordinates. Thus the $k$
random variables $X_i^2$ have full affine span, by comparing quadratic
parts of any almost sure affine identity. Yet $\sum_i c_iX_i=c$ almost
surely for some nonzero $c_i$, so
\[
 \prod_{\sigma\in\{-1,1\}^k}
       \left(c-\sum_i\sigma_i c_i t_i\right)
\]
is a nonzero polynomial in $t_1^2,\ldots,t_k^2$ that vanishes at
$t=X$. This contradicts the support property for the marginal law.
For two proportional centered coordinates with $X_2=aX_1+b$,
$a\ne0$, $b\ne0$ makes $X_1^2,X_2^2$ affinely independent by their
quadratic and linear terms, while the same sign product supplies a
polynomial relation. Thus $b=0$ is necessary. The surviving
representatives now have nonsingular covariance; projection gives
divisibility of their squares, and Theorem~\ref{q49:norescue} gives it
after centering. Reinstating constant or proportionally repeated
coordinates preserves divisibility, also proving sufficiency.
\end{proof}

For a fixed covariance and fixed direction $m$, the set of admissible
amplitudes $t\in\R$ is empty, all of $\R$, or $[-r,r]$ with $r\ge0$.
Indeed, if it contains any $t_0\ne0$, it contains zero by the theorem.
For $0\le s\le1$ the Gaussian log Laplace transform at mean $st_0m$
is $(1-s^2)$ times the centered log transform plus $s^2$ times the
log transform at mean $t_0m$, by the quadratic dependence on the mean.
Nonnegative combinations of infinitely divisible exponents preserve
infinite divisibility. Symmetry in the mean and weak closure give the
assertion. This excludes the onset of divisibility beyond a positive
amplitude proposed in \cite[Section~5]{q47:Kogan}.

\subsection{What an ordinary covariance spectrum determines}
Fix the standard orthogonal complex structure $J$ on $\mathbb R^{2n}$.
In Avagyan's normalization a quantum covariance is a real symmetric
matrix $D>0$ with $D+iJ\ge0$. The orthogonal matrices commuting with
$J$ form the passive group $U(n)=O(2n)\cap\operatorname{Sp}(2n,\mathbb R)$.
Avagyan asked which ordinary spectra determine a single passive orbit
and which determine both the thermal and squeezing parameters. He classified simple spectra, proved sufficiency when all but at most
one eigenvalue coincide, and obtained an obstruction from a quartet
formed by two nonpure modes under his unique-pairing and multiplicity
hypotheses \cite[Corollary~5.5, Lemma~5.3 and Proposition~5.1]{r87:Avagyan}.
The following rule applies with arbitrary multiplicities.

\begin{theorem}[Reciprocal pairs and a rank-one remainder]\label{r87:sympl}
Let $a_1\le\cdots\le a_{2n}$ be positive. They form the ordinary spectrum
of a quantum covariance exactly when $a_i a_{2n+1-i}\ge1$ for $1\le i\le n$.
Remove both entries of each such pair with product one. The following are equivalent:
\begin{enumerate}
\item All quantum covariances with this spectrum form one passive orbit.
\item All have the same thermal and squeezing parameters.
\item All have the same thermal parameters.
\item The remaining list is empty, or all its entries except possibly one coincide.
\end{enumerate}
Here the equivalences concern admissible spectra. In the other admissible cases,
the sum of the squared thermal parameters ranges over a nondegenerate interval
already on a connected family of covariances.
\end{theorem}

\begin{proof}
The real block form of $D+iJ\ge0$ and its Schur complement give
$D\ge B:=J^TD^{-1}J$. The ordered eigenvalues of $B$ are $1/a_{2n+1-i}$,
proving necessity of the product inequalities. Conversely, place opposite
entries of the ordered list on the two axes of each canonical plane.

Suppose $a_i=1/a_{2n+1-i}=t$. The spaces for the eigenvalues of $D$ at
most $t$ and of $B$ at least $t$ have dimensions summing to more than
$2n$. A unit vector $v$ in their intersection satisfies
$t\ge\langle Dv,v\rangle\ge\langle Bv,v\rangle\ge t$,
hence $Dv=Bv=tv$. It follows that $D(Jv)=t^{-1}Jv$.
Thus $v,Jv$ span an orthonormal plane reducing both $D$ and $J$, also
when $t=1$. Removing this plane deletes opposite entries of the sorted
list, leaving the other pairs intact. Induction splits off every
product-one pair by a passive choice of basis. If the remaining list
has the form in (4), its covariance is $aI+(b-a)ww^T$ for a unit
vector $w$, including $a=b$. The unitary group of the remaining
space acts transitively on the unit sphere. This proves (4)$\Rightarrow$(1).
Passive changes preserve both sets of parameters, so (1)$\Rightarrow$(2)$\Rightarrow$(3).

For the converse write the remaining eigenvalues as $b_1,\ldots,b_m$
and $D_0=\operatorname{diag}(b_1,\ldots,b_m)$. Keep $D_0$ fixed and vary
the orthogonal complex structure $K=O^T K_0 O$, $O\in SO(m)$, where
$K_0$ pairs opposite entries of the remaining sorted list. This is
equivalent to varying the covariance in a fixed symplectic space.
Since all these paired products are strictly greater than one,
$D_0+iK$ is positive definite in a neighborhood of $K_0$.
The squared thermal parameters in these modes have sum
\[
 F(K)=-\tfrac12\operatorname{tr}((D_0K)^2)
     =\sum_{i<j}b_i b_j K_{ij}^2,
\]
because the eigenvalues of $D_0K$ are $\pm i\nu_j$, where $\nu_j$
are the thermal parameters. If (4) fails there are four distinct
indices carrying values $a,b,c,d$ with $a\ne d$ and $b\ne c$.
Compare the pairings $(a,b),(c,d)$ and $(a,c),(b,d)$, keeping
all other pairs fixed. On their signed coordinate complex structures
the two values of $F$ differ by $(a-d)(b-c)$. The signs of the
two-dimensional blocks do not affect $F$ and may be chosen so both
structures have the orientation of $K_0$. Thus the real analytic
function $O\mapsto F(O^TK_0O)$ is nonconstant on the connected group
$SO(m)$, and hence on every neighborhood of the identity. A connected
physical neighborhood therefore yields an interval of distinct values.
The removed planes only add a fixed constant to $F$. This disproves
(3) whenever (4) fails.
\end{proof}

\section{Wishart determinant products}\label{r54:wishart}

Here $\mathcal W_p(\alpha,\Sigma)$ denotes the real Wishart law with
Laplace transform $\det(I+2\Sigma H)^{-\alpha/2}$ at $H\ge0$.
For $W\sim\mathcal W_p(\alpha,\Sigma)$ with $\alpha>p-1$ and
$\Sigma>0$, let $D_j=\det W_{jj}$ on disjoint diagonal blocks.
Write $p_j$ for the block dimensions, restricting the matrix to the
selected coordinates if necessary. Genest, Ouimet and Richards
formulate the following product bound in \cite[Conjecture~1.1]{wish46:GOR}
and give a stronger quantitative bound for two blocks in
\cite[Corollary~3.4]{wish46:pair}. In the same authors' mixed-sign
moment inequality \cite[Proposition~3.9]{wish46:GOR}, Conjecture~1.1
is a hypothesis. Corollary~\ref{r88:wishfloor} gives the sharp conditional
lower bound and obtains their mixed-sign inequality throughout the
finite negative-moment range. Their Conjecture~3.8 concerns joint
lower-tail events.

\begin{theorem}[Wishart minor products and determinant tilts]
\label{wish50:product}\label{wish48:intpower}
For arbitrary block sizes, positive definite scale and real
$\alpha>p-1$, and all $\nu_j\ge0$,
\[
 \E\prod_j D_j^{\nu_j}\ \ge\ \prod_j\E D_j^{\nu_j}.
\]
Equality holds exactly when $\Sigma_{ij}=0$ for every distinct $i,j$
with $\nu_i\nu_j>0$. More precisely, for a distinguished block $0$
of size $m$ and $F=\prod_{j\ne0}D_j^{\nu_j}$, the response
\[
 R(u)=\frac{\E[D_0^u F]}{\E D_0^u},\qquad
 u>-(\alpha-m+1)/2,
\]
is increasing in the weak sense. It is strictly increasing if some
$\Sigma_{0j}\ne0$ with $\nu_j>0$, and constant otherwise.
\end{theorem}

\begin{lemma}[A radial derivative of a noncentral Wishart response]
\label{wish50:derivative}
Let $D>0$ in dimension $q$, $b\in\mathbb R^q$, $B=bb^\top$,
and $\beta>q$. For $x\ge0$ write $Q_x f$ for the expectation of
$f(X)$ at
\[
 X=V+(Z+\sqrt x\,b)(Z+\sqrt x\,b)^\top,\qquad
 V\sim\mathcal W_q(\beta-1,D),\quad Z\sim N(0,D)
\]
with $V,Z$ independent. For smooth $f$ on the symmetric matrices
whose derivatives through order two have polynomial growth on the
positive cone, put
\[
 \mathcal L_b f(X)=\beta\,df(X)[B]+\tfrac12\sum_{l=1}^q
 d^2f(X)[E_l,E_l],\qquad
 E_l=(X^{1/2}e_l)b^\top+b(X^{1/2}e_l)^\top .
\]
For $x>0$ one has
\[
 \frac{d}{dx}Q_x f=\tfrac12\int_0^1 s^{\beta/2-1}
                      Q_{sx}\mathcal L_b f\,ds .
\]
\end{lemma}
\begin{proof}
For $e_H(X)=e^{-\tr(HX)}$, Gaussian integration and the central
Wishart transform give $Q_xe_H=a e^{-xc}$, with
$a=\det(I+2DH)^{-\beta/2}$, $A=H(I+2DH)^{-1}$ and $c=b^\top A b$.
The operator on $e_H$ gives
$(-\beta b^\top H b+2b^\top H X H b)e_H$.
Along the matrix direction $R=HBH$, differentiation in $H$ gives
$d\log a[R]=-\tfrac\beta2 b^\top(H-A)b$ and $dc[R]=c^2$,
using $dA[R]=(I+2HD)^{-1}R(I+2DH)^{-1}$. Hence
$Q_{sx}\mathcal L_b e_H=(2sxc^2-\beta c)ae^{-sxc}$.
Multiplying by $s^{\beta/2-1}/2$ gives
$-ac\,\partial_s(s^{\beta/2}e^{-sxc})$, proving the identity.

The same transforms hold for purely imaginary symmetric $H$ by
holomorphic continuation from the positive cone, using exponential
moments in a neighborhood of zero and the determinant branch continued
from zero. Fourier inversion with the trace pairing now gives the
identity on smooth compactly supported functions: the multiplier is
bounded by $C(1+\|H\|^2)(1+\tr X)$ on imaginary tests, the trace
moments are uniform over $0\le s\le1$, and the Fourier transform
decreases rapidly. Differentiation on the left uses the shifted normal
density, whose score at a row $z$ is
$b^\top D^{-1}(z-\sqrt x\,b)/(2\sqrt x)$ at $x>0$.
Finally multiply $f$ by a smooth cutoff equal to one on
$\|X\|_{\rm HS}\le N$ and zero on $\|X\|_{\rm HS}\ge2N$.
Polynomial moment bounds, uniform for $sx\in[0,x]$, dominate the
right-hand side as $N\to\infty$, and the Gaussian score dominates the
left derivative.
\end{proof}

\begin{lemma}[Positive response for a product of determinants]
\label{wish50:radial}
Partition the $q$ coordinates into blocks of sizes $q_j$,
take $K\ge0$, and set
$f(X)=\prod_j\det(K_{jj}+X_{jj})^{\eta_j}$, $\eta_j\ge0$.
Then $Q_x f$ is nondecreasing on $[0,\infty)$, strictly if
some active block $\eta_j>0$ has $b_j\ne0$.
\end{lemma}
\begin{proof}
With a single block, condition on $V$, put $T=K+V>0$ and use
the determinant lemma. The multiplier of $\det T$ is
$1+(Z+\sqrt{x}b)^\top T^{-1}(Z+\sqrt{x}b)$. Diagonalizing
$D^{1/2}T^{-1}D^{1/2}$ writes the quadratic term as a positively
weighted sum of independent noncentral $\chi^2_1$ variables with
noncentralities $xc_l^2$, where $c_l$ are the coordinates of
$D^{-1/2}b$ in that eigenbasis. For each coordinate, the Poisson
mixture with mean $xc_l^2/2$ adds independent $\chi^2_2$ increments
to a central $\chi^2_1$. Coupling the counts by Poisson processes
gives the claimed order and strictness.

With at least two blocks, let $T_j=K_{jj}+X_{jj}+\varepsilon I$,
$A_j=T_j^{-1}$ and $f_\varepsilon=\prod_j(\det T_j)^{\eta_j}$.
There is a smooth extension off the positive cone with polynomial
growth through order two, by extending each scalar power smoothly
below $\varepsilon^{q_j}$. Summing the two derivatives of
$\log\det T_j$ in Lemma~\ref{wish50:derivative} gives
\[
 \mathcal L_b\log\det T_j
 =\beta b_j^\top A_jb_j-b_j^\top A_j X_{jj} A_j b_j
    -(b_j^\top A_j b_j)\tr(A_j X_{jj})
 \ \ge\ (\beta-q_j-1)b_j^\top A_j b_j .
\]
Indeed $0\le A_j^{1/2}X_{jj}A_j^{1/2}\le I$. The exponential chain
rule adds half the sum of the squares of the directional derivatives
of $\log f_\varepsilon$, times $f_\varepsilon$. Thus
\[
 \mathcal L_b f_\varepsilon\ \ge f_\varepsilon\sum_j\eta_j(\beta-q_j-1)b_j^\top A_jb_j .
\]
Every $q_j\le q-1$ and $\beta>q$. At fixed $x>0$ the Gaussian score
and $f_\varepsilon\le f_1$ for $0<\varepsilon\le1$ give
$(Q_x f_\varepsilon)'\to(Q_x f)'$. In the integral formula use
Fatou's lemma on the displayed nonnegative lower bound. Its pointwise
limit is positive in the active case since $V>0$ almost surely.
This proves strict increase for $x>0$. Continuity at zero follows
from the defining Gaussian representation and polynomial growth.
If every active $b_j=0$, the active marginal law is constant.
\end{proof}

\begin{proof}[Proof of Theorem~\ref{wish50:product}]
For a nonempty complement of dimension $q=p-m$, partition the scale
about the distinguished block as
$\left(\begin{smallmatrix}A&C\\ C^\top&S\end{smallmatrix}\right)$
and put $D=S-C^\top A^{-1}C$. Write
$A^{-1/2} W_{00} A^{-1/2}=O\operatorname{diag}(\lambda)O^\top$
with increasing eigenvalues and an independent Haar eigenbasis.
Whiten the first block, then subtract its linear regression from the second.
At the resulting block diagonal scale, write $U$ for the first block and $J$ for the cross block.
The change to $G=U^{-1/2}J$ and the Schur complement $V$ has Jacobian $(\det U)^{q/2}$.
The Wishart density factors into the laws of $U,V$ and the independent Gaussian rows of $G$.
Undoing the regression and rotating the rows by $O^\top$ therefore gives the complement as
\[
 V+\sum_{i=1}^m (Z_i+\sqrt{\lambda_i}\,b_i)
                  (Z_i+\sqrt{\lambda_i}\,b_i)^\top,qquad
 b_i=C^\top A^{-1/2}O e_i .
\]
Here $V\sim\mathcal W_q(\alpha-m,D)$ and
$Z_i\sim N(0,D)$ are all independent of one another and of
$(O,\lambda)$. At fixed $O$ freeze all other rows and apply
Lemma~\ref{wish50:radial} to row $i$ with $\beta=\alpha-m+1>q$.
The response of $F$ increases in each $\lambda_i$. Its average
$h(\lambda)$ over $O$ is bounded by a polynomial in
$1+\sum_i\lambda_i$. If $\Sigma_{0j}\ne0$ for some $\nu_j>0$, the
$j$-component of $b_m$ is nonzero for almost every $O$, so $h$
strictly increases in $\lambda_m$ for fixed preceding coordinates.
In the absence of such a cross scale $h$ is constant.

Under the tilt by $D_0^u$, the whitened distinguished block has law
$\mathcal W_m(\alpha+2u,I)$ by its density. Its conditional regression
and independent Haar eigenbasis persist. The admissible range gives
$\alpha+2u>m-1$, and the polynomial bound on $h$ gives finite tilted
moments throughout. For $v>u$, the reweighting from $u$ to $v$ is
proportional to $g(\lambda)=\prod_i\lambda_i^{v-u}$. The ordered
white Wishart eigenvalues are associated, since their density
is MTP$_2$ \cite[Theorem~3.14, PDF numbering]{wish46:GOR}.
Hence $\operatorname{Cov}_u(g,h)\ge0$ by truncation, and $R(v)\ge R(u)$.

For strictness, condition on the first $m-1$ ordered eigenvalues
$r$. The conditional density of the last eigenvalue $y>r_{m-1}$ is
proportional to
$y^{(\alpha+2u-m-1)/2} e^{-y/2}\prod_{i<m}(y-r_i)$.
The conditional covariance of $g,h$ is positive in the active case
by strict increase in $y$. For ordered $r'\ge r$ the ratio of the
two conditional densities on their common support is a positive
constant times $\prod_{i<m}(y-r'_i)/(y-r_i)$, nondecreasing in $y$,
and zero below the raised endpoint. Thus the conditional law increases
stochastically in $r$ and both conditional means are increasing.
Association of the marginal $r$ makes their covariance nonnegative.
Total covariance gives strictness. At $m=1$ the scalar covariance
argument alone suffices.

Taking $u=0$ and $v=\nu_0>0$ separates one factor. Iterate on the
remaining marginal Wishart matrices, which retain shape $\alpha$.
If a pair of active blocks has nonzero cross scale, separate one of
them first to obtain strictness. If the active scale is block
diagonal, independence follows from the Wishart transform.
An empty complement or an empty active set gives equality directly.
\end{proof}

\begin{corollary}[The smallest conditional determinant moment]\label{r88:wishfloor}
In Theorem~\ref{wish50:product} let $F=\prod_{j\ne0}D_j^{\nu_j}$,
$D=S-C^\top A^{-1}C$ as in its proof, and $G\sim\mathcal W_q(\alpha,D)$.
The largest constant $c$ for which
$\mathbb E[w(W_{00})F]\ge c\,\mathbb E[w(W_{00})]$ for every
nonnegative bounded measurable weight $w$ is
\[
 c=\mathbb E\prod_{j\ne0}\det(G_{jj})^{\nu_j}.
\]
Consequently the mixed-sign inequality of
\cite[Proposition~3.9]{wish46:GOR} holds on the whole finite negative-moment
range $0<s<(\alpha-m+1)/2$ for the distinguished power $D_0^{-s}$.
\end{corollary}
\begin{proof}
Condition on $W_{00}$ in the regression used above. For its fixed
eigenbasis the conditional expectation increases in each whitened
eigenvalue. Setting them all to zero leaves $V+\sum_i Z_iZ_i^\top$,
which has the central law $\mathcal W_q(\alpha,D)$ by its Laplace
transform. This proves the lower bound. As $W_{00}\to0$ all shifts
tend to zero uniformly in the eigenbasis. Gaussian polynomial
domination gives convergence of the conditional expectation to $c$.
The law of $W_{00}$ charges every neighborhood of zero, so weights
supported in these neighborhoods prove sharpness.
Truncation permits $w=D_0^{-s}$, with finiteness in the stated range by
the tilt in Theorem~\ref{wish50:product}. Applying that theorem to $G$
bounds $c$ below by
$\prod_{j\ne0}(\det D_{jj}/\det S_{jj})^{\nu_j}\mathbb E[D_j^{\nu_j}]$,
the Schur factors of the cited inequality.
\end{proof}

\begin{corollary}[Minor products for independent Gaussian rows]
\label{r67wish:rank}
Let the rows $(X_{ri})_i$, $1\le r\le N$, be independent centered
Gaussian vectors, with arbitrary block covariances that may vary with
$r$. For fixed positive semidefinite matrices $K_i$ of the block
dimensions $q_i$, put $D_i=\det(K_i+\sum_r X_{ri}X_{ri}^\top)$.
For all $\nu_i\ge0$,
\[
                 \E\prod_iD_i^{\nu_i}\ \ge\ \prod_i\E D_i^{\nu_i},
\]
with zero powers set to one. In the Wishart sample case $K_i=0$ with
common row covariance $\Sigma\ge0$, an active block ($\nu_i>0$) of
width $q_i>N$ or singular marginal covariance makes both sides zero.
Otherwise equality holds exactly when the active blocks have zero
pairwise cross covariances.
\end{corollary}
\begin{proof}
For jointly centered Gaussian vectors $Y_i$, arbitrary $c_i\ge0$ and
$s_i>0$, the factors $(c_i+\|Y_i\|^2)^{s_i/2}$ satisfy the product
inequality. Indeed, for a fixed vector $y$ and independent standard
Gaussian $g,z$ of vector and scalar type,
\[
 \E_{g,z}|\langle g,y\rangle+\sqrt c\,z|^s
       =\E|z|^s\,(c+\|y\|^2)^{s/2}.
\]
Take such auxiliaries independently for each factor. At fixed $g_i$
apply the scalar inequality of Ouimet and Greaves
\cite[Theorem~1]{r67wish:OG} to the jointly centered
$\langle g_i,Y_i\rangle+\sqrt{c_i}\,z_i$, then average over the $g_i$.
If every factor has positive expectation, the inequality is strict as
soon as a cross covariance of two $Y_i$ is nonzero: the associated
bilinear form is then nonzero for almost every pair of Gaussian probes.

Freeze all rows except $r$ and denote the fixed matrix in block $i$ by
$V_i\ge0$. The determinant in the remaining row is
$\det V_i+X_{ri}^\top\operatorname{adj}(V_i)X_{ri}$, where the
adjugate is positive semidefinite by continuity from invertible
matrices. The preceding inequality shows that replacing this row by
independent Gaussian blocks with the same block marginals can only
decrease the product moment. After all rows have been replaced the
blocks are independent, giving the bound.

For strictness in the Wishart case discard inactive blocks. At the
last replacement the fixed matrices $V_i$ are independent, each
formed from $N-1$ samples of a nonsingular block marginal. When
$q_i<N$, the adjugate is invertible almost surely. When $q_i=N$, it
has rank one and its range has full support on projective space:
the Gaussian sample matrix has positive density, and every
hyperplane is spanned by the rows of some full-rank matrix. This
includes $N=q_i=1$, where the adjugate is the scalar one. Thus a
nonzero $\Sigma_{ij}$ gives a nonzero cross covariance of
$\operatorname{adj}(V_i)^{1/2}X_{ri}$ and
$\operatorname{adj}(V_j)^{1/2}X_{rj}$ with positive probability.
Every active conditional factor has positive expectation. The last
replacement is then strict. Vanishing active cross covariances imply
independence.
\end{proof}

\subsection{All real joint moments of disjoint principal minors}
Wilks discussed the joint moments of disjoint principal minors in 1934
\cite[p.~326]{r75:Wilks}. Genest, Ouimet and Richards state this problem
as Eq.~(3) in Section~5 of \cite{r56:wilks-source}, after treating
nested minors in their Theorem~1. Their two-block formula for disjoint
minors uses a matrix hypergeometric function
\cite[Theorem~3.1]{wish46:pair}. Gaunt and Ouimet give a finite
recursion for arbitrary block counts and nonnegative integer powers
\cite{r56:stein}. We use one radial variable to turn the integer moments
into a moment table on a cube. Its finite differences then determine the
real moments, with an explicit error for each finite truncation.

The moment calculation uses the following elementary fact about
positive measures. It also explains the Bernstein reconstruction of
strip measures for a self-adjoint pencil.

\begin{lemma}[Two uses of finite differences on a cube]\label{wil56:cube}
Let $\lambda$ be a finite positive measure on $[0,1]^d$, with
$M_j=\int z^j\,d\lambda(z)$. For nonnegative multi-indices $m,k$ put
$B_m(k)=\sum_{j\le k}(-1)^{|j|}\binom{k}{j}M_{m+j}$, where binomial
coefficients and powers are products over coordinates. Then
$B_m(k)=\int z^m(1-z)^k\,d\lambda\ge0$ and
\[
 \sum_{0\le k_i\le N}
 \left(\prod_i\binom{N}{k_i}\right)
 B_k(N\mathbf1-k)\,\delta_{k/N}\ \longrightarrow\ \lambda
 \quad\text{weakly}.
\]
If $z_i>0$ almost surely in every coordinate, then for all
$\beta_i\ge0$ one has, in $[0,\infty]$,
\[
 \int z^{m-\beta}\,d\lambda
       =\sum_{k\ge0}B_m(k)\prod_i(\beta_i)_{k_i}/k_i!.
\]
\end{lemma}
\begin{proof}
Expand $(1-z)^k$ to obtain the integral for $B_m(k)$.
In the first formula, conditionally on $z$, take independent
$K_i\sim\mathrm{Bin}(N,z_i)$. Their normalized vector converges in
probability to $z$ by the variance bound $1/(4N)$ in every coordinate.
Integration against $\lambda$ proves weak convergence. In the second
formula expand each $z_i^{-\beta_i}$ at $1-z_i$ and use monotone convergence.
\end{proof}

Let $W\sim\mathcal W_p(\alpha,\Sigma)$ with $\Sigma>0$ and
$\alpha>p-1$. For disjoint nonempty index sets $I_i$, $1\le i\le d$, put
$q_i=|I_i|$, $D_i=\det W_{I_i I_i}$ and
$c_i=\det\Sigma_{I_i I_i}/q_i^{q_i}$. Set $a=p\alpha/2$,
$h(t)=\sum_i q_it_i$, and define
\[
 A(t)=2^{h(t)}(a)_{h(t)}\prod_i c_i^{t_i},
 \qquad (x)_u=\Gamma(x+u)/\Gamma(x)\quad(x,x+u>0).
\]
For an independent symmetric matrix variable $H_i$ in each block, let
$H$ have blocks $H_i$ and be zero elsewhere. The symmetric differential
matrix $\partial_i$ has diagonal entries $\partial/\partial (H_i)_{ll}$
and off-diagonal entries $\tfrac12\partial/\partial (H_i)_{lk}$
($l<k$, using the independent upper triangular entries). For
$j\in\mathbb N_0^d$ write
\begin{equation}\label{wil56:coeff}
 M_j=\frac{1}{A(j)}\left[
 \prod_i\det(\partial_i)^{j_i}\,
       \det(I-2\Sigma H)^{-\alpha/2}\right]_{H=0}.
\end{equation}
Only the Taylor polynomial of degree $h(j)$ is needed here. Take $B_m(k)$ from Lemma~\ref{wil56:cube} with this table.

\begin{theorem}[Wilks' joint moments]\label{wil56:main}
The joint moment $\E\prod_i D_i^{\nu_i}$ is finite exactly when
\begin{equation}\label{wil56:domain}
                 \nu_i>-(\alpha-q_i+1)/2\qquad(1\le i\le d).
\end{equation}
Throughout this domain, put $m_i=\max(0,\lceil\nu_i\rceil)$ and
$\beta_i=m_i-\nu_i$. Then
\begin{equation}\label{wil56:series}
 \E\prod_i D_i^{\nu_i}
   = A(\nu)\sum_{k\in\mathbb N_0^d}
               B_m(k)\prod_i\frac{(\beta_i)_{k_i}}{k_i!}.
\end{equation}
Every summand is nonnegative. For integer $k\ge0$ we use the
rising factorial polynomial, so $(0)_0=1$ and $(0)_k=0$ for $k>0$.
\end{theorem}
\begin{proof}
Let $l,L$ denote the extreme eigenvalues of $\Sigma$ in increasing
order. The density at scale $\Sigma$ lies between the densities at
$lI$ and $LI$ multiplied by $(l^p/\det\Sigma)^{\alpha/2}$ and
$(L^p/\det\Sigma)^{\alpha/2}$, respectively. This follows by comparing
the exponential factors in the Wishart densities. At either scalar
scale the diagonal blocks are independent by the Laplace transform.
For a block of size $q$ at scale $LI$, triangular factorization gives
independent squared diagonal factors distributed as
$L\chi^2_{\alpha-r+1}$, $1\le r\le q$.
It follows that its determinant moment of order $u$ is finite exactly
for $u>-(\alpha-q+1)/2$, and then equals
$(2L)^{qu}\Gamma_q(\alpha/2+u)/\Gamma_q(\alpha/2)$, where
$\Gamma_q(x)=\pi^{q(q-1)/4}\prod_{r=1}^q\Gamma(x-(r-1)/2)$.
The two density comparisons give \eqref{wil56:domain}, including
divergence at each boundary face. In this domain $a+h(\nu)>0$, since
$\sum_i q_i\le p$ and $\sum_i q_i(q_i-1)\ge0$.

Whiten the matrix, writing $Y=\Sigma^{-1/2}W\Sigma^{-1/2}=RS$,
where $R=\operatorname{tr}Y$ and $\operatorname{tr}S=1$.
In these coordinates Lebesgue measure has radial factor
$R^{p(p+1)/2-1}$. Multiplying by the determinant power in the Wishart
density leaves $R^{a-1}e^{-R/2}$. Thus $R$ has the gamma law of
shape $a$, scale $2$, and is independent of $S$.
If $E_i$ is the coordinate inclusion for $I_i$, then
$Q_i=\Sigma_{I_i I_i}^{-1/2}E_i^*\Sigma^{1/2}$ has orthonormal rows.
Consequently
\[
 Z_i=\frac{D_i}{c_iR^{q_i}}=q_i^{q_i}\det(Q_iSQ_i^*)\in(0,1].
\]
Indeed $\operatorname{tr}(Q_iSQ_i^*)\le1$ and the arithmetic--geometric
mean inequality bounds the determinant. Independence gives
$\E\prod_i D_i^{t_i}=A(t)\E\prod_i Z_i^{t_i}$ on
\eqref{wil56:domain}. Applying each $\det(\partial_i)$ to
$e^{\operatorname{tr}(HW)}$ multiplies it by $D_i$.
The moment generating function is analytic near the origin, so
\eqref{wil56:coeff} gives $M_j=\E\prod_i Z_i^{j_i}$. Applying the real-power identity in Lemma~\ref{wil56:cube} to this
angular law proves
\eqref{wil56:series}.
\end{proof}

\begin{corollary}[Certified finite sums]\label{wil56:bound}
Let $T_N$ denote the right side of \eqref{wil56:series} restricted to
$0\le k_i\le N_i$. For each $i$ with $\beta_i>0$ choose
$0<s_i<\nu_i+(\alpha-q_i+1)/2$. With $e_i$ the corresponding unit vector,
\[
 0\le \E\prod_i D_i^{\nu_i}-T_N\le
 A(\nu)\sum_{i:\beta_i>0}
     \frac{(\beta_i)_{N_i+1}}{(\beta_i+s_i)_{N_i+1}}\,
                     U(\nu-s_i e_i),
\]
where
\[
 U(t)=\frac{(L^p/\det\Sigma)^{\alpha/2}(2L)^{h(t)}}{A(t)}
         \prod_i\frac{\Gamma_{q_i}(\alpha/2+t_i)}
                       {\Gamma_{q_i}(\alpha/2)}.
\]
In particular the error is $O(\sum_{i:\beta_i>0}(N_i+1)^{-s_i})$.
\end{corollary}
\begin{proof}
Bound the omitted set by the union of the tails $k_i>N_i$.
On the $i$th tail the ratio $(\beta_i)_{k_i}/(\beta_i+s_i)_{k_i}$
is decreasing in $k_i$. Extract its value at $N_i+1$ as an upper
bound, enlarge the modified sum to all indices, and expand the angular
moment $\E[Z_i^{-s_i}\prod_j Z_j^{\nu_j}]$ as in the proof.
The density comparison at $LI$ bounds this moment by $U(\nu-s_i e_i)$.
The gamma ratio gives the claimed order.
\end{proof}

A second bound is obtained by changing coordinates so that each diagonal
scale block is the identity. Under an invertible block diagonal congruence
$W\mapsto BWB^*$, both $R=\operatorname{tr}(\Sigma^{-1}W)$ and the
$Z_i$ are invariant. Thus the smaller of the two values of $U(t)$,
before and after standardization, can be used with the original
prefactor $A(\nu)$. At block diagonal scale the standardized bound
is the exact angular moment.

\begin{corollary}[Gaussian samples of deficient rank]\label{wil56:sample}
The same formula, domain and error bound hold at positive integer
$\alpha=n<p$ for $W=XX^*$ with $n$ independent $N(0,\Sigma)$ columns
when every selected block has $q_i\le n$.
\end{corollary}
\begin{proof}
For $\Sigma^{-1/2}X$ the squared Euclidean radius in dimension $pn$
is $R\sim\chi^2_{pn}$ and is independent of the direction. Thus
$S=\Sigma^{-1/2}XX^*\Sigma^{-1/2}/R$ satisfies the same angular
calculation. Each selected block has full rank almost surely. Density
comparison at $lI,LI$ is now on the $p\times n$ sample space, with
exactly the same constants. At scalar scale the disjoint row blocks
are independent, giving \eqref{wil56:domain} and the same $U(t)$.
The moment generating function and all remaining steps are unchanged.
\end{proof}

The distance to the first divergent moment gives the exact rate of the
positive series. Suppose $p>1$ and at least one $\beta_i>0$. On the
diagonal truncation $N_i=N$, for fixed parameters in \eqref{wil56:domain},
\begin{equation}\label{r78wil:rate}
 \E\prod_i D_i^{\nu_i}-T_N\asymp N^{-b},\qquad
 b=\min_{i:\beta_i>0}\left(\nu_i+\frac{\alpha-q_i+1}{2}\right).
\end{equation}
When all $\beta_i$ vanish the sum has only its $k=0$ term. For $p=1$
the angular variable is identically one and there is no truncation error.

To verify the rate, tilt the angular law by $\prod_j Z_j^{\nu_j}$ and
normalize it to be a probability. Under this tilt its $i$th marginal
satisfies $\mathbb P_\nu(Z_i\le z)\asymp z^{b_i}$ as $z\downarrow0$,
where $b_i=\nu_i+(\alpha-q_i+1)/2$. Indeed the scalar-scale density
comparisons in the proof of Theorem~\ref{wil56:main} reduce this
estimate to $\Sigma=I$: $\tr(\Sigma^{-1}W)\asymp\tr W$ makes each
angular coordinate comparable to its identity-scale version, and
those coordinates are unaffected by scalar multiplication of $W$.
At identity scale the trace fractions $T_i=\tr W_{ii}/\tr W$
(including the unused block if present) have a Dirichlet law with
parameters $q_i\alpha/2$ and are independent of the block shapes
$W_{ii}/\tr W_{ii}$. The tilt changes those parameters to
$q_i(\alpha/2+\nu_i)$, with zero exponent on the unused block.
For a block of size $q=q_i$, put $B_i=q^q\det(W_{ii}/\tr W_{ii})$. Then $Z_i=T_i^q B_i$. By the
determinant moment formula and Gauss's gamma multiplication formula,
under the tilt
\[
 B_i\ \stackrel d=\ \prod_{r=1}^{q-1}
 \operatorname{Beta}\!\left(\frac{\alpha-r}{2}+\nu_i,
                                  \frac r2+\frac rq\right)
\]
with independent factors. This identity follows by comparing all
nonnegative integer moments on $[0,1]$. If $q\ge2$ the last beta factor
has lower tail comparable to $z^{b_i}$, and the inverse $b_i$th moment
of the product of the other factors, including $T_i^q$, is finite.
Conditioning on that product gives the claimed lower tail. For $q=1$ it
comes directly from the trace fraction since $p>1$.

Conditionally on the angular coordinates, the normalized binomial
series in coordinate $i$ has upper tail
$\mathbb P(U_{i,N}>Z_i\mid Z_i)$, where
$U_{i,N}\sim\operatorname{Beta}(\beta_i,N+1)$ is independent.
The corresponding single-coordinate truncation error, as a fraction
of the full moment, is thus comparable to
$\E U_{i,N}^{b_i}\asymp N^{-b_i}$.
The rectangular error lies between the largest of these tails and
their sum, proving \eqref{r78wil:rate}. The same argument uses the
Gaussian sample-space density comparisons in Corollary~\ref{wil56:sample}
when $\alpha$ is an integer below $p$ and every selected block has size
at most $\alpha$.

\begin{corollary}[Overlapping compressions]\label{wil56:overlap}
Let $V_i:\mathbb R^{q_i}\to\mathbb R^p$ be any injective linear maps.
For $D_i=\det(V_i^*WV_i)$ and $c_i=\det(V_i^*\Sigma V_i)/q_i^{q_i}$,
formula \eqref{wil56:series} holds for every $\nu_i\ge0$, with
$H=\sum_i V_iH_iV_i^*$ in \eqref{wil56:coeff}. For these exponents the
error bound of Corollary~\ref{wil56:bound} holds with $U(\nu-s_i e_i)$
replaced by
\[
 \frac{q_i^{-q_i u_i}\Gamma_{q_i}(\alpha/2-u_i)}
      {(a)_{-q_i u_i}\Gamma_{q_i}(\alpha/2)},\qquad
 u_i=\max(0,s_i-\nu_i),\quad 0<s_i<\nu_i+(\alpha-q_i+1)/2 .
\]
\end{corollary}
\begin{proof}
The angular calculation applies with $Q_i=(V_i^*\Sigma V_i)^{-1/2}
V_i^*\Sigma^{1/2}$. Since $Q_iQ_i^*=I$ we still have $0<Z_i\le1$.
All the determinant moments with nonnegative exponents are finite
by Gaussian radial moments and this angular bound. The modified tail
moment is at most $\E Z_i^{-u_i}$, because every other angular factor is
at most one. Its value is the one-block Wishart moment divided by its
radial factor, giving the displayed expression.
\end{proof}

\subsection{Matching inequalities and polarization of subspace volumes}

The second and third questions following Conjecture~1.1 in
\cite{wish46:pair} ask for algebraic inequalities from Wishart moments
and a Wishart counterpart of real linear polarization. The matching
expansion for entry moments is due to Kuriki and Numata
\cite{r59:KurikiNumata}. Antisymmetrizing that expansion gives a
determinant inequality. The weighted real scalar polarization theorem of Galicer, Ortega-Moreno and Pinasco \cite{r59:GOP} supplies the case of lines. Successive applications on orthogonal complements, followed by Cauchy--Binet, yield below a sharp bound for Gram volume dilations on subspaces of arbitrary dimensions.

\begin{corollary}[An inequality for alternating sums of weighted matchings]
\label{r59:matching}
Partition $\{1,\dots,p\}$ into blocks $I_i$ of sizes $q_i$. Fix
$\Sigma>0$, $\alpha>p-1$, and integer powers $k_i\ge0$.
For each $i$ take $k_i$ copies of $I_i$ and choose a permutation
$\sigma_{ih}$ of each copy $h$. For each $a\in I_i$ in that copy,
introduce a pair of distinct vertices with labels $a$ and
$\sigma_{ih}(a)$. Write $\ell(u)$ for the label of vertex $u$ and
$M_0$ for the matching consisting of these prescribed pairs. For a
perfect matching $M$ of all vertices, let $c(M)$ be the number of
connected components of $M\cup M_0$, counting a common edge as one
component. Then
\[
 \begin{aligned}
 &\sum_{\boldsymbol\sigma}\!
 \Big(\prod_{i,h}\operatorname{sgn}\sigma_{ih}\Big)
 \sum_M\alpha^{c(M)}\prod_{\{u,v\}\in M}\Sigma_{\ell(u),\ell(v)}
 \\[-2pt]
 &\hspace{3em}\ge\
 \prod_i(\det\Sigma_{I_iI_i})^{k_i}
 \prod_i\prod_{j=1}^{q_i}\prod_{b=0}^{k_i-1}(\alpha-j+1+2b).
 \end{aligned}
\]
Equality holds exactly when $\Sigma_{I_iI_j}=0$ for distinct
$i,j$ with $k_i k_j>0$.
\end{corollary}
\begin{proof}
For integer $\alpha=n\ge p$, expand the block determinants of
$W=\sum_{s=1}^{n}G_sG_s^\top$ by permutations, where the $G_s$ are
independent $N(0,\Sigma)$. A prescribed pair uses the same sample
index for its two vertices. In Wick's formula a matched pair also
requires equal sample indices and contributes the displayed covariance.
The number of index assignments is $n^{c(M)}$. Thus the left side is
$\mathbb E\prod_i(\det W_{I_iI_i})^{k_i}$. The same identity holds
at every real $\alpha>p-1$. Indeed each entry moment of total degree
$N=\sum_i q_i k_i$ is polynomial in $\alpha$, by differentiating
$\exp[-(\alpha/2)\log\det(I-2\Sigma H)]$ at $H=0$ in symmetric
directions. The marginal determinant moment is the displayed factor
for block $i$, by the independent chi-square factors in Bartlett's
decomposition. Theorem~\ref{wish50:product} gives the inequality
and its equality condition.
\end{proof}
For real powers throughout the finite moment domain, insert these
matching sums in the integer moment table of Theorem~\ref{wil56:main}.
Its normalization and finite differences give a convergent series
for the real moments in terms of the same matching sums.

\begin{theorem}[Sharp polarization inequality for subspace volumes]
\label{r59:volpol}
Let $E_i$ be finitely many subspaces of dimensions $q_i\ge1$ in a
real Hilbert space $H$, let $a_i>0$, and set $P=\sum_i q_i a_i$ and
$m=\max_i q_i$. Write
$v_i(T)=\det((T|_{E_i})^*(T|_{E_i}))^{1/2}$. There exists
$T:H\to\mathbb R^m$ of Hilbert--Schmidt norm one such that
\[
                 \prod_i v_i(T)^{a_i}
           \ \ge\ \prod_i (a_i/P)^{q_i a_i/2}.
\]
The constant is sharp for mutually orthogonal $E_i$. The output
dimension $m$ is the least possible for a positive universal bound.
For lines $E_i$ the assertion is the weighted real linear polarization
inequality.
\end{theorem}
\begin{proof}
Put $A_j=\sum_{q_i\ge j}a_i$. Construct orthonormal rows
$u_1,\ldots,u_m$ successively. At step $j$, for each active $i$
put $F_{ij}=E_i\cap\operatorname{span}(u_1,\ldots,u_{j-1})^\perp$
and choose a unit vector $v_{ij}\in F_{ij}$. This is possible
since $\dim F_{ij}\ge q_i-j+1$. Apply the weighted scalar polarization
inequality of Galicer, Ortega-Moreno and Pinasco \cite{r59:GOP}
in the span of the chosen vectors, with weights $a_i/A_j$.
It gives a unit vector $u_j$ orthogonal to the preceding rows and
\[
 \prod_{q_i\ge j}\|P_{F_{ij}}u_j\|^{a_i}
 \ \ge\ \prod_{q_i\ge j}|\langle v_{ij},u_j\rangle|^{a_i}
 \ \ge\ \prod_{q_i\ge j}(a_i/A_j)^{a_i/2}.
\]
Define $(Tx)_j=\sqrt{A_j/P}\langle u_j,x\rangle$. Since
$\sum_j A_j=P$, its Hilbert--Schmidt norm is one. Cauchy--Binet
bounds $v_i(T)$ below by the absolute determinant of the first $q_i$
rows restricted to $E_i$. Gram--Schmidt expresses that determinant
as $\prod_{j\le q_i}\sqrt{A_j/P}\|P_{F_{ij}}u_j\|$.
Multiply over $i$ and use the preceding inequality at each $j$.

For orthogonal $E_i$, the energies $e_i=\|T|_{E_i}\|_{\rm HS}^2$
sum to at most one. Applying arithmetic--geometric mean to the squared
singular values gives $v_i(T)\le(e_i/q_i)^{q_i/2}$.
Weighted arithmetic--geometric mean with weights $q_i a_i/P$ gives
the asserted upper bound. If the output dimension is less than $m$,
a subspace of dimension $m$ has zero volume.
\end{proof}

\subsection{Which block moments determine a Wishart response?}\label{cycle59:sec}
Let $p=\sum_i q_i$, $\alpha>p-1$, and let $Z_i$ be independent
$\mathcal W_{q_i}(\alpha,I)$ matrices. Fix nonnegative measurable
$f_i$ with $z_i=\mathbb E f_i(Z_i)>0$ and
$\mathbb E[f_i(Z_i)e^{\epsilon\operatorname{tr}Z_i}]<\infty$ for some
$\epsilon>0$. Write $\mathbb E_f$ for the independent block laws
tilted by $f_i/z_i$, and $L_i=\mathbb E_f Z_i-\alpha I$.
Take $C(t)_{ij}=t_{ij}C_{ij}$, where $C_{ji}=C_{ij}^\top$,
$t_{ji}=t_{ij}$, and $C(t)_{ii}=0$. With $W$ at scale $I+C(t)$, put
$G(t)=\mathbb E\prod_i f_i(W_{ii})/\prod_i z_i$ and $F(t)=\log G(t)$.

\begin{theorem}[Taylor coefficients and block degrees]\label{cycle59:degree}
For a monomial $t^m=\prod_{i<j}t_{ij}^{m_{ij}}$ let
$d_i=\sum_{j\ne i}m_{ij}$, using $m_{ji}=m_{ij}$. Its coefficient
in $F$ vanishes if any $d_i$ is odd. Otherwise it is determined by
the tilted polynomial moments of block $i$ through degree $d_i/2$,
for all $i$. This order is sharp in general. In particular, for a
simple cycle $\gamma=(i_1,\ldots,i_r,i_1)$ of length $r\ge3$, with
$t_\gamma$ the product of its edge variables,
\begin{equation}\label{cycle59:traces}
 \begin{split}
 [t_{ij}^2]F&=(2\alpha)^{-1}\operatorname{tr}(L_iC_{ij}L_jC_{ji}),\\
 [t_\gamma]F&=\alpha^{1-r}
 \operatorname{tr}(L_{i_1}C_{i_1i_2}\cdots L_{i_r}C_{i_ri_1}).
 \end{split}
\end{equation}
\end{theorem}
\begin{proof}
The scale density divided by its scale-$I$ density is
$\det(I+C)^{-\alpha/2}\exp\{\operatorname{tr}((I-(I+C)^{-1})W)/2\}$.
Near zero this is analytic and bounded in modulus by a constant times
$e^{\epsilon\operatorname{tr}W}$, so coefficient extraction under
the tilted expectation is valid. First take $f_i(Z)=e^{-\operatorname{tr}(H_iZ)}$
with $H_i\ge0$. The tilted scale is $K_i=(I+2H_i)^{-1}$. Factoring
the joint determinant transform by its block marginals gives
\[
 R_m(K):=[t^m]G(t)
  =[t^m]\det\!\big(I-C(t)\operatorname{diag}(K_i-I)\big)^{-\alpha/2}.
\]
Exponentiating the trace-log expansion counts collections of oriented
closed walks.
Each visit to $i$ consumes two incident edge factors and contributes
one $K_i-I$. Thus $R_m$ vanishes at odd degree and otherwise has
polynomial degree at most $d_i/2$ in $K_i$.

This gives an exact calculation for general weights. The map
$\Phi_i:P\mapsto[K\mapsto\mathbb E_{\alpha,K}P(Z_i)]$ is a linear
automorphism on polynomials of degree at most $d_i/2$ in symmetric
matrix entries. The Wishart transform gives the degree bound, and
a zero image gives a zero Laplace transform for $P(Z_i)$ times its
scale-$I$ law. Its density is positive on the open cone, so $P=0$.
For general $f_i$ one has $[t^m]G=\mathbb E_f Q_m$ with
$Q_m=(\bigotimes_i\Phi_i^{-1})R_m$. Indeed after multiplication by
$\prod z_i$ these two functionals integrate $\prod f_i$ against
finite signed measures on the block diagonals. Their Laplace transforms
agree by the exponential tests, hence the measures agree.
Taking the formal logarithm uses only coefficients with componentwise
smaller or equal exponents, giving the claimed block bounds and parity.

For a doubled edge or a simple cycle, every proper nonempty
submonomial has an odd vertex, so $[t^m]F=[t^m]G$. The trace-log
coefficient counts two walks for a doubled edge and $2r$ for a simple
cycle. Each uses every participating matrix factor once, and
$\Phi_i^{-1}(K_i-I)=(Z_i-\alpha I)/\alpha$. This gives
\eqref{cycle59:traces}.

For sharpness take two scalar blocks of correlation $\rho$, put
$a=\alpha/2$ and let $Q_n(x)=(-1)^n L_n^{(a-1)}(x/2)$ for $n\ge1$.
The Laguerre generating function at $x/2$ is
$(1-u)^{-a}\exp[-xu/(2(1-u))]$. Multiplying the two generating
functions and taking the Wishart expectation gives $(1-\rho^2uv)^{-a}$.
Coefficient comparison gives
$\mathbb E[Q_n(W_{11})Q_l(W_{22})]=
\mathbf1_{\{n=l\}}(a)_n\rho^{2n}/n!$. At $\rho=1$ this also follows
by using a single gamma variable for both coordinates. Thus $Q_n$ is
orthogonal to all lower degrees, and $f_i=1+\delta Q_n$ is positive for
small $\delta>0$ since $Q_n$ has positive leading coefficient.
Its lower tilted moments are unchanged, whereas
$F(\rho)=\log(1+\delta^2(a)_n\rho^{2n}/n!)$. Replacing just one
weight by one preserves its first $n-1$ moments and makes $F=0$.
\end{proof}

The same proof applies to an integer sample shape $n<p$ when each
$q_i\le n$: the displayed scale likelihood is then the likelihood of
the $n$ Gaussian sample vectors, and each block law still has positive
density on its full cone.

For $f_i(Z)=\det(Z)^{\nu_i}$ throughout \eqref{wil56:domain}, the
tilted shape is $\alpha+2\nu_i$, so $L_i=2\nu_i I$. The simple-cycle
coefficient is consequently the coefficient in the joint cumulant of
the distinct linear observations $\operatorname{tr}W_{ii}$ at its
vertices multiplied by $\prod_{i\in\gamma}(\nu_i/\alpha)$. Indeed
the logarithmic Wishart transform gives $2^r\alpha$ times the cyclic
trace for those observations by the same $2r$-walk count.

The first repeated-vertex coefficient can have the opposite sign. For
three scalar blocks with correlations $x=t_{12}$, $y=t_{23}$ and
$t_{13}=0$, write $u_i=\mathbb E_f Z_i/\alpha-1$. Then
\[
 [x^2y^2]F=\tfrac14 u_1u_3
          \big(\operatorname{Var}_f Z_2-4\mathbb E_f Z_2+2\alpha\big).
\]
Indeed the determinant in $R_m$ is $1-x^2k_1k_2-y^2k_2k_3$,
$k_i=K_i-1$. Its inverse power contributes
$a(a+1)k_1k_2^2k_3$, $a=\alpha/2$. Use
$\Phi_2^{-1}(K_2^2)=Z_2^2/[\alpha(\alpha+2)]$ and subtract the
product of the two edge coefficients in taking the logarithm.
For determinant powers this is $-4\nu_1\nu_2\nu_3/\alpha^2$,
strictly negative even when all three powers are positive.

\subsection{Absolute Gaussian moments in known hypergeometric functions}
Genest, Ouimet and Richards single out a special-function formula for
three centered Gaussian coordinates \cite[Introduction]{wish46:pair};
see also \cite{r56:stein}. Nabeya treated trivariate integer moments
\cite{r57:Nabeya}. Ogasawara gave multivariate series for real powers,
including truncated and noncentral laws \cite{r57:Ogasawara}. Here the
parities of the edge multiplicities identify a finite sum in the
standard Srivastava--Daoust family. Its three-variable specialization is
a sum of two classical $H_B$ functions.

For a graph $G$ on $[d]$ with edge set $E$ let
$J_{ie}=\mathbf1_{\{i\in e\}}$. For vertex parameters $v_i$ and edge
parameters $b_e>0$ write the following specialization of the
Srivastava--Daoust generalized Lauricella function
\cite[(5)]{r57:SD}:
\[
 \mathcal F_J(v;b;x)
  =\sum_{k\in\mathbb N_0^E}
     \prod_i(v_i)_{(Jk)_i}
     \prod_{e\in E}\frac{x_e^{k_e}}{(b_e)_{k_e}k_e!}.
\]
The upper coefficient rows are the incidence rows of $J$; each lower
parameter is attached to one variable. This notation specifies the
standard function also when some incidence coefficients are zero.

\begin{theorem}[Absolute moments in every dimension]\label{r57:SDmoment}
For $X\sim N_d(0,\Sigma)$ with $\Sigma>0$ the exact real moment
domain for $\prod_i|X_i|^{2\nu_i}$ is $\nu_i>-1/2$ for every $i$.
Write $K=\Sigma^{-1}$, $\ell_i=K_{ii}$,
$C_{ij}=K_{ij}/\sqrt{\ell_i\ell_j}$, and $a_i=\nu_i+1/2$.
Take $G$ to have the nonzero off-diagonal entries of $C$ as edges,
and write $C_e=C_{ij}$ for $e=ij$. For $\epsilon\in\{0,1\}^E$ with
$J\epsilon$ even at every vertex, put $s=J\epsilon/2$. Then
\begin{align}
 \E\prod_i|X_i|^{2\nu_i}
    &=\sqrt{\det C}\prod_i\frac{(2/\ell_i)^{\nu_i}\Gamma(a_i)}
                                 {\sqrt\pi}\, B(1),\label{r57:GF}\\
 B(t)&=\sum_{\epsilon:\,J\epsilon\ {\rm even}}
       \prod_e(-2tC_e)^{\epsilon_e}\prod_i(a_i)_{s_i}\,
       \mathcal F_J(a+s;\tfrac12\mathbf1+\epsilon;(t^2C_e^2)_e).
       \notag
\end{align}
Here $B(1)$ is the analytic continuation of the indicated single sum
from zero along the real segment. For $\nu_i\ge0$ the value at every
positive semidefinite covariance is the limit from $\Sigma+\varepsilon I$.
\end{theorem}
\begin{proof}
The Gaussian density is bounded above by a product of Gaussian tails
and bounded below on every compact box. Integrating in a box where
only one coordinate can approach zero gives the claimed necessity;
the upper bound gives sufficiency. Rescale $x_i=\sqrt{2/\ell_i}\,y_i$.
The normalized integral
\[
 B(t)=\left(\prod_i\Gamma(a_i)\right)^{-1}
       \int_{\mathbb R^d}e^{-y^\top[I+t(C-I)]y}
                         \prod_i|y_i|^{2\nu_i}\,dy
\]
is holomorphic on the strip on which $I+(\operatorname{Re}t)(C-I)>0$.
Every compact substrip has a common Gaussian majorant for the integral
and each fixed derivative. This strip contains $[0,1]$.

Near zero, expand the cross terms $e^{-2tC_e y_i y_j}$. The sum of
absolute integrands is integrable there. A term with multiplicities
$n_e$ survives coordinate sign integration exactly when all entries of
$Jn$ are even. Write $n=2k+\epsilon$. Integration at vertex $i$ gives
$\Gamma(a_i+s_i+(Jk)_i)$, while
$(2k_e+\epsilon_e)!=4^{k_e}k_e!(1/2+\epsilon_e)_{k_e}$.
This proves the finite sum near zero and the integral continues that
sum across the entire segment. For boundary covariances couple
$X_\varepsilon=X+\sqrt{\varepsilon}Z$ with independent standard
Gaussian $Z$. If $0\le\varepsilon\le1$ and $\nu_i\ge0$,
$\prod_i(|X_i|+|Z_i|)^{2\nu_i}$ is an integrable majorant (zero power one).
\end{proof}

\begin{corollary}[The trivariate formula]\label{r57:HBmoment}
In dimension three use $x=(C_{12}^2,C_{23}^2,C_{31}^2)$ and
$\tau=C_{12}C_{23}C_{31}$. In \eqref{r57:GF} one has
\[
 B(t)=H_B(a;1/2;t^2x)
       -8t^3\tau a_1a_2a_3\,H_B(a+\mathbf1;3/2;t^2x).
\]
Here $H_B(v;b;x)$ abbreviates Srivastava's classical triple function
with three equal lower parameters \cite[p.~43]{r57:HB}.
\end{corollary}
\begin{proof}
For the triangle all three edge parities are equal. Its incidence rows
give the upper Pochhammer indices $m+h,m+n,n+h$ in edge order
$12,23,31$, precisely those of $H_B$. If an edge is absent its zero
entry gives the same specialization.
\end{proof}

There is also a convergent evaluation at $t=1$ using only series
coefficients at zero. For $C\ne I$, the strip endpoints are
$u=-1/\lambda_{\max}(C-I)<0$ and $v=-1/\lambda_{\min}(C-I)>1$.
Put
\[
 g(t)=\tan\!\left(\frac{\pi(t-(u+v)/2)}{2(v-u)}\right),\qquad
 w(t)=\frac{g(t)-g(0)}{1-g(0)g(t)} .
\]
The tangent maps the strip biholomorphically to the unit disk (apply
the exponential and the Cayley map). The second transformation is a
disk automorphism taking $t=0$ to zero. The Taylor series of
$B(t(w))$ therefore converges at $w(1)$; each coefficient of order $k$
uses only edge terms of total degree at most $k$ in the original germ.
This prescription continues the sum even when its individual
power series diverge at the desired covariance.

\subsection{A named hypergeometric formula for arbitrary Wishart blocks}
Gaunt and Ouimet identify a special-function representation beyond two Wishart blocks as missing \cite[Introduction]{r56:stein}. The radial part of Bartlett's decomposition leaves a Dirichlet moment on a simplex. Its monomial coefficients give Srivastava--Daoust functions \cite{r57:SD} for arbitrary block dimensions, including real powers.

\begin{theorem}[Wishart compression moments as Srivastava--Daoust functions]\label{r57:wilksSD}
Let $W\sim\mathcal W_p(\alpha,\Sigma)$, with $\alpha>p-1$ and $\Sigma>0$. For $1\le i\le d$, let $V_i\in\mathbb R^{p\times q_i}$ have full column rank, and put
\[
 D_i=\det(V_i^\top W V_i),\quad c_i=\det(V_i^\top\Sigma V_i),\quad
 u=\sum_iq_i\nu_i,\quad A=p\alpha/2 .
\]
The exponents $\nu_i$ are real. Define $m_i=\max(0,\lceil\nu_i\rceil)$ and $b_i=m_i-\nu_i$. The following finite collection of Srivastava--Daoust functions gives $\E\prod_iD_i^{\nu_i}$ whenever it is finite. In particular it covers every exponent in \eqref{wil56:domain} for disjoint principal blocks.

Index $N=p(p+1)/2$ variables $y_j$ by the lower-triangular matrix entries. Put $a_{kk}=(\alpha-k+1)/2$ and $a_{k\ell}=1/2$ for $k>\ell$, so $\sum_j a_j=A$. For each of the $2^{N-p}$ off-diagonal sign choices $\sigma$, let $L_\sigma(y)$ have diagonal entries $y_{kk}$ and lower off-diagonal entries $\sigma_{k\ell}y_{k\ell}$. Form the polynomials
\[
\begin{aligned}
 P_i(y)&=c_i^{-1}\det\!\left(V_i^\top\Sigma^{1/2}
       L_\sigma(y)L_\sigma(y)^\top\Sigma^{1/2}V_i\right),\\
 H_i(y)&=\Big(\sum_jy_j^2\Big)^{q_i}-P_i(y)
       =\sum_{\gamma} h_{i\gamma} y^\gamma ,
 \qquad \prod_iP_i(y)^{m_i}=\sum_{\eta}d_{\eta} y^\eta .\end{aligned}
\]
The coefficient arrays depend on $\sigma$. Each nonzero $h_{i\gamma}$ has $|\gamma|=2q_i$. For every numerator monomial index $\eta$, define
\[
 L_\eta=\frac{\Gamma(A)}{\Gamma(A+|\eta|/2)}
                \prod_j\frac{\Gamma(a_j+\eta_j/2)}{\Gamma(a_j)}
\]
and the following Srivastava--Daoust germ in the variables $z h_{i\gamma}$:
\[
 S_{\sigma,\eta}(z)=
 \sum_{(k_{i\gamma})\ge0}
 \frac{\displaystyle\prod_i(b_i)_{\sum_\gamma k_{i\gamma}}
       \prod_j(a_j+\eta_j/2)_{\frac12\sum_{i,\gamma}\gamma_j k_{i\gamma}}}
      {\displaystyle(A+|\eta|/2)_{\sum_{i,\gamma}q_i k_{i\gamma}}}
        \prod_{i,\gamma}\frac{(z h_{i\gamma})^{k_{i\gamma}}}{k_{i\gamma}!}.
\]
Only nonzero coefficients $h_{i\gamma}$ require summation variables. Then
\[
 \E\prod_i D_i^{\nu_i}
 =2^{u}\frac{\Gamma(A+u)}{\Gamma(A)}
   \Big(\prod_i c_i^{\nu_i}\Big)\,2^{-(N-p)}
      \lim_{z\uparrow1}\sum_{\sigma,\eta}d_\eta L_\eta S_{\sigma,\eta}(z).
\]
Equivalently, group the displayed series by total degree $\sum_{i,\gamma}k_{i\gamma}$ before summing. These combined coefficients are nonnegative, and their partial sums at $z=1$ increase to the angular moment. If $A+u>0$, the same identity holds in $[0,\infty]$. If $A+u\le0$, the moment is infinite.
\end{theorem}
\begin{proof}
Write $W=\Sigma^{1/2}LL^\top\Sigma^{1/2}$ by Bartlett's decomposition. The squared diagonal entries of $L$ have independent chi-square laws with degrees $\alpha-k+1$, the strict lower entries are independent standard normals, and both families are independent. Thus $R=\sum_j L_j^2$ has the $\chi^2_{2A}$ law, its normalized squared entries $s_j=L_j^2/R$ have the Dirichlet law with parameters $(a_j)$, and the strict lower signs are uniform. These three objects are independent. With $y_j=\sqrt{s_j}$ we have $D_i=c_iR^{q_i}P_i(y)$. Since $L_\sigma(y)L_\sigma(y)^\top$ has trace one, its compression through $\Sigma^{1/2}V_i$ is bounded above by $V_i^\top\Sigma V_i$. Hence $0<P_i(y)\le1$ in the simplex interior and $0\le H_i(y)=1-P_i(y)\le1$. The radial moment equals $2^u\Gamma(A+u)/\Gamma(A)$ for $A+u>0$ and diverges otherwise, in which case independence and strictly positive angular factors imply divergence of the full moment.

At fixed $\sigma$ consider the angular expectation
\[
 B_\sigma(z)=\E_s\!\left[\prod_i P_i(\sqrt s)^{m_i}
                     \prod_i(1-zH_i(\sqrt s))^{-b_i}\right].
\]
It is holomorphic for $\operatorname{Re}z<1$, using the branches that equal one at zero. Indeed each $1-zH_i$ has real part bounded below on compact subsets by a positive constant. For small $|z|$ expand each factor by the binomial and multinomial series. Absolute convergence is uniform on the simplex, since the polynomial consisting of the absolute values of the coefficients of each $H_i$ is bounded there. The monomials integrate by
\[
 \E_s\prod_j s_j^{v_j}
     =\frac{\Gamma(A)}{\Gamma(A+\sum_j v_j)}
        \prod_j\frac{\Gamma(a_j+v_j)}{\Gamma(a_j)}\qquad(v_j\ge0).
\]
Taking $v_j=(\eta_j+\sum_{i,\gamma}\gamma_jk_{i\gamma})/2$ gives
$B_\sigma(z)=\sum_\eta d_\eta L_\eta S_{\sigma,\eta}(z)$ as germs and supplies the indicated continuation. The ordinary Taylor coefficients of $B_\sigma$ are nonnegative: before expanding $H_i$ into monomials, the coefficient of $z^n$ integrates the sum over $\sum_i n_i=n$ of $\prod_i P_i^{m_i}(b_i)_{n_i}H_i^{n_i}/n_i!$. The Taylor series converges for $|z|<1$, and its partial sums at one increase to the Abel limit. The nonnegative integrand of $B_\sigma(z)$ increases to $\prod_iP_i(\sqrt s)^{\nu_i}$ as $z\uparrow1$. Monotone convergence applies even when some factors in that product tend to zero and others diverge at the same face. Averaging signs and multiplying by the radial moment proves the formula.
\end{proof}

\begin{proposition}[Removing a common principal block]\label{w57:core}
Suppose $I_i=C\cup P_i$, where the petals $P_i$ are pairwise disjoint
and disjoint from the core $C$ of size $c>0$. Write $a=\alpha/2$,
$q_i=|I_i|$, $s=\sum_i\nu_i$, and let
$\Sigma^{[C]}$ denote the Schur complement of $\Sigma_{CC}$ on the
union of the petals. For $\alpha>p-1$ and $\Sigma>0$,
\[
 \E\prod_i (\det W_{I_i I_i})^{\nu_i}
   = \det(2\Sigma_{CC})^s
       \frac{\Gamma_c(a+s)}{\Gamma_c(a)}
       \E\prod_{i:P_i\ne\varnothing}
                   (\det Y_{P_iP_i})^{\nu_i},
 \qquad Y\sim\mathcal W(\alpha-c,\Sigma^{[C]}),
\]
where the last expectation is evaluated by Theorem~\ref{wil56:main}.
The exact finite domain is
\[
 s>-(\alpha-c+1)/2,qquad
 \nu_i>-(\alpha-q_i+1)/2\quad(P_i\ne\varnothing).
\]
For nonnegative exponents the joint moment is at least the product of
the marginal moments, strictly so when at least two exponents are
positive. In particular this proves the product bound for any two
principal minors, including crossed and nested pairs.
\end{proposition}
\begin{proof}
The partitioned Wishart factorization used in the proof of
Theorem~\ref{wish50:product} gives $W_{CC}$ and its Schur complement
$Y$ as independent Wishart matrices of the displayed shapes and scales.
This factorization also underlies Wilks' nested-minor formula
\cite{r56:wilks-source}. The determinant identity
$\det W_{I_iI_i}=\det W_{CC}\det Y_{P_iP_i}$ (empty determinant one)
gives the formula and the domain, since all factors are positive
almost surely. Use the disjoint product bound of
Theorem~\ref{wish50:product} on $Y$. For $X=\det W_{CC}>0$, increasing
powers satisfy $\E X^{u+v}\ge\E X^u\,\E X^v$ for $u,v\ge0$ by the
independent-copy covariance identity. The inequality is strict if
$u,v>0$ since $X$ has a nondegenerate law. Iterate and use independence.
The disjoint case of the last assertion is Theorem~\ref{wish50:product}.
\end{proof}

This applies also to any two injective linear compressions $V_1,V_2$:
choose a basis of the intersection of their ranges, extend it to a
basis of each range and concatenate the additional basis vectors.
The Wishart matrix in these linearly independent coordinates has
positive definite scale. Changes within either basis contribute only a
positive constant to its determinant. For two positive exponents,
equality in the product bound holds exactly when $V_1^*\Sigma V_2=0$.
Indeed a nontrivial intersection forces strictness by the proposition;
otherwise apply the disjoint equality criterion in these coordinates.

\subsection{The convergence domain for all two-coordinate minors}
Let $G$ be any simple graph on $[p]$, $p\ge2$. Write $D_i=W_{ii}$
and $D_{ij}=W_{ii}W_{jj}-W_{ij}^2$. For real exponents $b_i,z_{ij}$ put
\[
 t_i=b_i+\sum_{j:\,ij\in E(G)} z_{ij},\qquad
 L(S)=(\alpha-1)(|S|-1)+2\sum_{ij\in E(G[S])}z_{ij}.
\]
These expressions also specify the real-part tests for complex powers.

\begin{theorem}[Arbitrary graphs of Wishart minors]\label{w57:edges}
Let $W\sim\mathcal W_p(\alpha,\Sigma)$, $\Sigma>0$, $\alpha>p-1$.
The moment $\E\prod_iD_i^{b_i}\prod_{ij\in E(G)}D_{ij}^{z_{ij}}$
is finite exactly when
\begin{equation}\label{w57:edgedomain}
 t_i> -\alpha/2\quad (i\in[p]),\qquad L(S)>0
                      \quad(S\subseteq[p],\ |S|\ge2).
\end{equation}
It is holomorphic on the tube over this domain and diverges at every
real boundary point. The same conclusions hold for $W=XX^\top$ with
$\alpha=n\ge2$ independent $N(0,\Sigma)$ columns, even for $n<p$.

There is also a two-sided estimate on every bounded set of real
exponents inside \eqref{w57:edgedomain}. Let $T$ run over spanning
trees of the complete graph, each with a total ordering of its edges.
Let $\mathcal P_{T,k}$ be the vertex partition into the components of
the first $k$ edges, and put $L_{T,k}=\sum_{S\in\mathcal P_{T,k}}L(S)$,
where $L(\{i\})=0$. The moment is comparable to
\begin{equation}\label{w57:edgequant}
  \left[\prod_{i=1}^p\Gamma(\alpha/2+t_i)\right]
      \sum_T\prod_{k=1}^{p-1}L_{T,k}^{-1}.
\end{equation}
The comparison constants may depend on the bounded set, $p,\alpha$
and $\Sigma$; they are uniform up to its boundary.
\end{theorem}
\begin{proof}
Comparison with two scalar scales in the Wishart density, or in the
Gaussian sample density at integer shape, reduces the proof to
$\Sigma=I$. Replacing $I$ by $cI$ multiplies the moment by
$c^{\sum_i t_i}$, uniformly bounded on the specified set of exponents.
At identity scale the diagonal entries are independent $\chi^2_\alpha$
variables, independent of the correlation matrix $R=(r_{ij})$.
For real shape this follows by putting
$W_{ij}=\sqrt{D_iD_j}\,r_{ij}$: the off-diagonal Jacobian is
$\prod_i D_i^{(p-1)/2}$, and $\det W=(\prod_i D_i)\det R$.
For samples, independence is the polar decomposition of each
independent standard Gaussian row. Since
$D_{ij}=D_iD_j d_{ij}^2$ with $d_{ij}=\sqrt{1-r_{ij}^2}$,
the diagonal factor gives precisely the first conditions and gamma
factors in \eqref{w57:edgedomain}--\eqref{w57:edgequant}.

For every fixed tree, its edge distances $d_{ij}$ are independent,
with squared distance distributed as
$\mathrm{Beta}((\alpha-1)/2,1/2)$. To verify this at real shape,
order vertices with parents before children. Given the principal matrix
$Y$ on the $m$ earlier vertices, the new column off its diagonal
equals $Y^{1/2}g$, and its diagonal equals $\|g\|^2+Z$ with
$g\sim N(0,I_m)$ and $Z\sim\chi^2_{\alpha-m}$ independent of each
other and of $Y$. This follows by substituting the scalar Schur
complement and the new off-diagonal column in the marginal Wishart
density. For the parent $i$, $r_{ij}^2$ is the squared component
of $g$ along the unit vector $Y^{1/2}e_i/\sqrt{Y_{ii}}$, divided by
$\|g\|^2+Z$. Orthogonal Gaussian decomposition gives the asserted
beta law, independent of the previously exposed matrix. In the
sample case each new normalized row is uniform on the sphere in
$\mathbb R^n$, giving the same conditional law without a nonsingularity
assumption on the earlier Gram matrix.

The numbers $d_{ij}$ obey the triangle inequality: realize $R$ as a
Gram matrix of unit vectors, then $d_{ij}$ is the operator norm of
the difference of their rank-one orthogonal projections. Let $F$ be
the distribution function of an edge distance. The transformation
$u=F(d)^{1/(\alpha-1)}$ is increasing, has density
$(\alpha-1)u^{\alpha-2}$ on $(0,1)$, and satisfies
$c u\le d\le C u$ for fixed $c,C>0$. The latter follows from the beta
density at zero and continuity on the rest of $[0,1]$.

Choose a minimum spanning tree for the distances and order its edges
by increasing length, breaking ties arbitrarily. For any pair $i,j$,
let $k(i,j)$ be the largest edge index on its tree path. Its distance
is at least the length of that edge by minimality, and at most $p-1$
times its length by the triangle inequality. On this sector the angular
integrand is thus bounded above by a constant times
$\prod_{ij\in E(G)}u_{k(i,j)}^{2z_{ij}}$. Dropping the minimum-tree
restriction and using independence for the fixed tree gives the integral
over $0<u_1<\cdots<u_{p-1}<1$. In the ratios
$v_k=u_k/u_{k+1}$, $u_p=1$, the monomial times the edge densities is
\[
             (\alpha-1)^{p-1}\prod_{k=1}^{p-1}v_k^{L_{T,k}-1}\,dv_k .
\]
Indeed a pair contributes to the $k$th prefix precisely when both
vertices belong to the same component of $\mathcal P_{T,k}$. Summing
the resulting upper bounds gives the tree sum in \eqref{w57:edgequant}.

Conversely take \emph{any} ordered tree and restrict its independent
lengths to $v_k<\epsilon$ for all $k$, with
$0<\epsilon<c/(2pC)$ and $\epsilon<1$. On every path the sum of the
lengths of the other edges is less than half the longest length.
The reverse triangle inequality and the path upper bound therefore
give a two-sided comparison of each $d_{ij}$ with $u_{k(i,j)}$.
The same integral now has each upper endpoint $\epsilon$.
It diverges if any $L_{T,k}\le0$ and otherwise is a constant times
$\prod_k \epsilon^{L_{T,k}}/L_{T,k}$. The constants of the distance
comparisons and the lower bound from these epsilon powers are uniform
on bounded sets of exponents. There are finitely many ordered trees,
so their maximum and sum are comparable. Every subset $S$ with
$|S|\ge2$ is a component of the first $|S|-1$ edges of some ordered
tree, with all other vertices isolated. Positivity of every prefix
sum is therefore equivalent to \eqref{w57:edgedomain}'s subset
conditions. Finally each compact subset of the complex tube is covered
by finitely many real-part boxes whose vertices still pass the tests;
the sum of the real integrands at the vertices is an integrable
dominating function. This gives holomorphy.
\end{proof}

Along an affine path from the interior to a real boundary point,
parametrized by its distance $\epsilon$ from that point, the exact
order of growth is $\epsilon^{-h}$. Here $h$ is the number of vanishing
diagonal margins plus the maximum over ordered trees of the number of
vanishing prefix sums there. This follows from
\eqref{w57:edgequant} and $\Gamma(x)\sim1/x$ at zero. For instance,
at $p=\alpha=3$, take $b_i=1$,
$z_{12}=-1+\epsilon$, $z_{13}=z_{23}=-1/2+\epsilon$. All diagonal
margins stay positive, while the ordered tree beginning with edge
$12$ has prefix sums $2\epsilon$ and $6\epsilon$. The moment grows
as $\epsilon^{-2}$. The nonnegative series \eqref{wil56:series}
evaluates the moment throughout \eqref{w57:edgedomain} using the
overlapping differential pencil of Corollary~\ref{wil56:overlap}.
Indeed $p\alpha/2+\sum_i t_i>0$ follows by summing the diagonal tests,
and the independent radial gamma factor and the cube expansion then
apply without any sign restriction.

For nonsingular Wishart laws and $p\ge3$ one may also insert the
factor $(\det W)^\eta$ for real $\eta$. The additional necessary
condition is $\alpha'=\alpha+2\eta>p-1$: near a rank-$(p-1)$ positive
semidefinite matrix with every one- and two-coordinate minor positive,
the determinant vanishes to first order and all the other factors stay
bounded above and below. Such a matrix is the Gram matrix of $p$
generic vectors in $\mathbb R^{p-1}$. A scalar Schur complement in a
nonsingular $(p-1)$ block gives a local positive normal coordinate.
Under this condition change shape in the Wishart density to $\alpha'$.
The moment is the one in Theorem~\ref{w57:edges} at that shape times
$2^{p\eta}(\det\Sigma)^\eta\Gamma_p(\alpha'/2)/\Gamma_p(\alpha/2)$,
so its remaining convergence conditions are exactly
\eqref{w57:edgedomain} with $\alpha'$ substituted for $\alpha$.

\subsection{Likelihood ratio power for independence of several Gaussian blocks}
Genest, Ouimet and Richards ask whether their product conjecture yields explicit power bounds for the likelihood ratio test of independence of $d$ groups of Gaussian variables \cite[question~1 following Conjecture~1.1]{wish46:pair}. The statistic is a positive power of $\Lambda=|W|/\prod_i D_i$, where $W\sim\mathcal W_p(\alpha,\Sigma)$, the nonempty blocks partition all $p$ coordinates, and $D_i=|W_{I_i I_i}|$. Thus a rejection region is $\{\Lambda\le c\}$ with $0<c<1$. Write $\pi_\Sigma(c)=\mathbb P_\Sigma(\Lambda\le c)$, $q_i=|I_i|$, $a=\alpha/2$ and $h=\prod_i|\Sigma_{I_i I_i}|/|\Sigma|$. Fischer's inequality gives $0<\Lambda\le1$ and $h\ge1$. The Gaussian sample has integral degrees of freedom; the bounds hold throughout the Wishart range $\alpha>p-1$.

\Needspace{8\baselineskip}
\begin{theorem}[Explicit bounds on the power]\label{r57:power}
For $0<s<(\alpha-p+1)/2$, set
\[
 g(s)=\frac{\Gamma_p(a-s)}{\Gamma_p(a)}
       \prod_i\frac{\Gamma_{q_i}(a)}{\Gamma_{q_i}(a-s)},\qquad
 u(s)=\frac{\Gamma_p(a-s)}{\Gamma_p(a)}
       \prod_i\left(\frac{\Gamma_{q_i}(a+(d-1)s)}{\Gamma_{q_i}(a-s)}\right)^{1/d}.
\]
For $0<t<(\alpha-p+1)/4$,
\[
 \frac{[g(t)-(ch)^{-t}]_+^2}{u(2t)}
 \le\pi_\Sigma(c)\le\min\{1,(ch)^t u(t)\}.
\]
The lower bound is positive whenever $(ch)^tg(t)>1$. With dimensions, shape and cutoff fixed, the lower bound tends to one as $h\to\infty$ if $t=(\log h)^{-1/2}$.
\end{theorem}
\begin{proof}
Multiplying the Wishart density by $|W|^{-s}$ changes its shape to $\alpha-2s$ and produces the factor $|2\Sigma|^{-s}\Gamma_p(a-s)/\Gamma_p(a)$. Consequently $M_s=\E\Lambda^{-s}$ is this factor times $\E_{\alpha-2s,\Sigma}\prod_iD_i^s$. Theorem~\ref{wish50:product} bounds the latter below by its marginal moment product. H\"older's inequality bounds it above by $\prod_i(\E_{\alpha-2s,\Sigma}D_i^{ds})^{1/d}$. The one-block gamma formula gives $h^s g(s)\le M_s\le h^s u(s)$. For $Y=\Lambda^{-t}$, splitting at $K=c^{-t}$ yields $\E Y\le K+(\E Y^2)^{1/2}\mathbb P(Y\ge K)^{1/2}$. Markov's inequality supplies the upper power bound. For the stated choice of $t$ as $h\to\infty$, $g(t),u(2t)\to1$ and $(ch)^{-t}\to0$.
\end{proof}

\begin{corollary}[Certified recovery of the nonnull power]\label{r57:powercdf}
For $d\ge2$, $\pi_\Sigma(c)$ can be enclosed to any prescribed accuracy by finite evaluations of the joint moment formula. Put $m_0=1$ and, for $k\ge1$, put
\[
 m_k=|2\Sigma|^k\frac{\Gamma_p(a+k)}{\Gamma_p(a)}\,
          \E_{\alpha+2k,\Sigma}\prod_i D_i^{-k}.
\]
For $0\le x\le1$ let
\[
 H_n(x)=\sum_{j=0}^{\lfloor nx\rfloor}\binom{n}{j}
                \sum_{l=0}^{n-j}(-1)^l\binom{n-j}{l}m_{j+l}.
\]
With $0<\delta<\min(c,1-c)$ and $e=(4n\delta^2)^{-1}$,
\[
 H_n(c-\delta)-e\ \le\ \pi_\Sigma(c)\ \le\ H_n(c+\delta)+e.
\]
\end{corollary}
\begin{proof}
A positive full-determinant tilt gives $m_k=\E\Lambda^k$. The block exponents $-k$ satisfy \eqref{wil56:domain} at shape $\alpha+2k$, since $\alpha>p-1$. Given $\Lambda$, draw $B_n\sim\operatorname{Bin}(n,\Lambda)$. Then $H_n(x)=\mathbb P(B_n/n\le x)$ and conditional Chebyshev gives $\mathbb P(|B_n/n-\Lambda|\ge\delta)\le e$. This proves the two inequalities. The level $\Lambda=c$ has probability zero, since $|W|-c\prod_iD_i$ is a nonzero polynomial (evaluate at $W=I$) and $W$ has a density. Taking $\delta=n^{-1/4}$ for sufficiently large $n$ sends both bounds to $\pi_\Sigma(c)$. For a fixed $n$ the sums are finite linear combinations of $m_0,\ldots,m_n$. Theorem~\ref{wil56:main} and Corollary~\ref{wil56:bound} enclose each $m_k$, and interval arithmetic with the displayed coefficients preserves the bounds. The evaluation intervals can be taken to have total propagated width at most $1/n$, giving a terminating precision test by the preceding convergence.
\end{proof}

\part{Recovery of probability laws from transforms and moments}

How much of a law is retained by observing a transform? For a
stochastic integral the transform averages dilates of its L\'evy
measure. For a stopped Brownian path the observations are tensors of
iterated integrals. Sections~\ref{r54:scales} and~\ref{path:sec}
give the distinct hypotheses under which these data determine the
relevant range or the law itself.
\section{Ranges of stochastic integrals and fractional orders}
\label{r54:scales}
For a deterministic integrand $f$, write $\Phi_f\rho$ for its
improper integral against a L\'evy process with time-one law $\rho$.
The ordinary range $R$ uses convergence in probability, the essential
range $R^e$ permits deterministic centering before taking the limit,
and the absolute range $R^0$ requires
$\int |C_\rho(f(t)z)|\,dt<\infty$ for every $z$, where $C_\rho$ is
the distinguished log characteristic function
\cite[Sections~1.6 and~3]{SatoFractional2010}.
When $f>0$, its scale measure is the image of Lebesgue measure
under $t\mapsto f(t)$ on its positive set. It acts on jump measures
by averaging their dilates. In log radius, comparison of two such
ranges reduces first to factorization by a positive convolution kernel.
At the critical index the drift has its own convergence test.
Theorem~\ref{r80:comparison} combines both steps for the two families below.
In the opposite direction, equality of jump ranges determines the scale
measure up to a dilation and a constant factor when it has finite second
moment (Proposition~\ref{r81:upsilon}).

For $p>0$ and $\alpha<2$, the logarithmic mapping $\Lambda_{p,\alpha}$
and the barred fractional mapping $\bar\Phi_{p,\alpha}$ have respective
scale measures on $0<u<1$
\[
  \frac{(-\log u)^{p-1}}{\Gamma(p)}u^{-1-\alpha}\,du,
  \qquad \frac{(1-u)^{p-1}}{\Gamma(p)}u^{-1-\alpha}\,du .
\]
Write $L_{p,\alpha}$ and $K_{p,\alpha}$ for their ordinary ranges,
with superscripts $0$ and $e$ for their absolute and essential ranges.
For $L^0_{\infty,\alpha}$ and $L^e_{\infty,\alpha}$ take the
intersection over all $p>0$. Use triplets $(A,\nu,\gamma)$ with
truncation $x\mathbf1_{\{|x|\le1\}}$. A law with weak mean zero
satisfies $\gamma+\lim_{R\to\infty}\int_{1<|x|\le R}x\nu(dx)=0$.
Absolute weak mean zero additionally requires the radial image of
$x\nu(dx)$ to have finite total variation on $(1,\infty)$.

\subsection{The jump measure and compensated drift in logarithmic coordinates}
For a triplet $\rho=(A,\eta,b)$ use the truncation
$x\mathbf1_{\{|x|\le1\}}$ fixed above. Write $y=\log|x|$ and
$\xi=x/|x|$, let $\eta_{\log}$ be the image of $\eta$ under
$x\mapsto(y,\xi)$, and put
\[
 \begin{aligned}
 M_\rho(dy,d\xi)&=e^{\alpha y}\eta_{\log}(dy,d\xi),\\
 g_\rho(v)&=e^{(\alpha-1)v}
     \left(b+\int_{1<|x|\le e^v}x\,\eta(dx)\right),\quad v\ge0 .
 \end{aligned}
\]
Let $k_p(t)=t^{p-1}/\Gamma(p)$ for $t>0$. Define the right fractional
integral of a positive measure on log radius and direction, and that of
a vector function on $[0,\infty)$, by
\[
 \begin{aligned}
 (I_p M)(F)&=\int_0^\infty k_p(t)\!\int
            \mathbf1_F(y-t,\xi)\,M(dy,d\xi)\,dt,\\
 (I_p g)(v)&=\lim_{T\to\infty}\int_0^T k_p(t)g(v+t)\,dt .
 \end{aligned}
\]
The first definition uses Tonelli's theorem. The second requires the
indicated improper vector limit.

\begin{theorem}[A common fractional integral for jumps and compensated drift]
\label{r61:satocore}
Fix $p>0$, $\alpha<2$. The law $\rho$ belongs to the essential
domain of $\Lambda_{p,\alpha}$ exactly when
\[
 \int e^{-\alpha y}(1\wedge e^{2y})\,(I_pM_\rho)(dy,d\xi)<\infty.
\]
Every essential output $\mu$ then has
$A_\mu=(2-\alpha)^{-p}A$ and $M_\mu=I_pM_\rho$, with arbitrary location.
On this domain the ordinary integral exists exactly when
$(I_pg_\rho)(0)$ exists. The absolute integral exists exactly when
$(I_p|g_\rho|)(0)<\infty$. For the ordinary output, including the
absolutely defined case, the whole profile satisfies
\[
                      g_\mu(v)=(I_pg_\rho)(v)\qquad(v\ge0).
\]
In particular the output drift is $(I_pg_\rho)(0)$. The positive
operators satisfy $I_q I_p=I_{p+q}$. Whenever the direct order
$p+q$ is absolutely defined, both successive integrals are
absolutely defined and their profile is $I_qI_pg_\rho=I_{p+q}g_\rho$.
\end{theorem}
\begin{proof}
In the variable $t=-\log u$ the scale measure is
$k_p(t)e^{\alpha t}dt$. Scaling an input jump by $e^{-t}$ subtracts
$t$ from $y$, so the tilted jump measure is exactly $I_p M_\rho$.
The displayed integrability condition is its untilted L\'evy norm.
The integrated covariance multiplier is the gamma integral
$\int k_p(t)e^{-(2-\alpha)t}dt=(2-\alpha)^{-p}$.
These are the two essential-domain conditions of
\cite[Proposition~3.18]{SatoFractional2010}; deterministic centering
can prescribe any drift. With the stated truncation the contribution
to the drift over $dt$ is $k_p(t)g_\rho(t)\,dt$. Ordinary definability
therefore amounts to convergence of its primitive. Absolute
definability amounts to integrability of its norm: the remaining
terms in the L\'evy exponent are absolutely integrable for each test
vector by the covariance bound and the integrated L\'evy norm.
Conversely, subtract these terms from the exponent and test the
coordinate directions.

In the ordinary case add the output's first moment on
$1<|x|\le e^v$ to its drift. Under dilation the added term integrates
the input on $e^t<|x|\le e^{t+v}$ with weight
$k_p(t)e^{(\alpha-1)t}dt$. This exchange is absolute since the
output jump measure has finite absolute first moment on the fixed
annulus. Multiplying by $e^{(\alpha-1)v}$ gives the profile identity.
The beta convolution $k_q*k_p=k_{p+q}$ yields the positive semigroup
by Tonelli. On the absolute direct domain the lower order $p$ also
exists absolutely, by comparison of the weights for $t\ge1$ and
local boundedness of both the jump norm and $g_\rho$. The positive
semigroup gives the second essential domain, and
$I_q|I_pg_\rho|(0)\le I_{p+q}|g_\rho|(0)$. Fubini now gives the
composition at zero. At other $v$ its absolute convergence follows
by the same finite-annulus argument at order $p+q$. This is Sato's
absolute-domain composition \cite[Theorem~7.3(i)]{SatoFractional2010}
written in the two logarithmic coordinates above.
\end{proof}

\subsection{Exponential scale kernels}
Write $\Psi_\alpha$ for Sato's exponential scale kernel
$u^{-1-\alpha}e^{-u}\,du$ on $(0,\infty)$ and
$\Phi_{\beta,\alpha}=\bar\Phi_{\alpha-\beta,\alpha}$ for his beta
kernel. The dependence of the jump measure on the scale answers two
questions in \cite[Section~5, Comments~6 and~7]{r79:Sato06}.

\begin{corollary}[Left limits of the exponential-kernel ranges]\label{r79:left}
For every $\alpha\le0$ and for each of the ordinary, absolute and
essential ranges $\mathcal R$,
\[
   \bigcap_{\beta<\alpha}\mathcal R(\Psi_\beta)=\mathcal R(\Psi_\alpha).
\]
\end{corollary}
\begin{proof}
Formula (5.21) in the author manuscript of \cite[Proposition~5.12]{SatoFractional2010}
prints strict inclusion down to negative indices. The cited strict inclusion
in \cite[Section~4, (4.28)]{r79:Sato06} concerns $0<\alpha\le2$.
For $\gamma\le0$ the three ranges agree. Their radial jump densities
are $r^{-1-\gamma}h_{\gamma,\xi}(r)$ with
$h_{\gamma,\xi}$ completely monotone in the classical sense and
vanishing at infinity \cite[Theorems~5.8 and~5.10(i)]{SatoFractional2010}.
For $\beta<\alpha$ the factor $r^{\beta-\alpha}$ preserves these
properties, giving one inclusion.

For the converse take $\beta_n\uparrow\alpha$ strictly below $\alpha$.
If the jump measure $\nu$ vanishes the criterion already applies.
Otherwise choose the angular marginal of $(1\wedge|x|^2)\nu(dx)$
as reference measure. Disintegrating this finite measure fixes a
single radial density $g_\xi(r)$ for $\nu$, up to radial null sets.
Changing the angular normalization only multiplies each radial
factor by a positive constant. Thus $r^{1+\beta_n}g_\xi(r)$
has a completely monotone version $h_{n,\xi}(r)$ for every $n$
on one common set of directions of full measure. By continuity,
$h_{n,\xi}(r)=r^{\beta_n-\beta_1}h_{1,\xi}(r)$ for every $r>0$.
Their finite pointwise limit $h_\xi(r)=r^{\alpha-\beta_1}h_{1,\xi}(r)$
is completely monotone by the pointwise closure of Laplace transforms
of positive measures. The L\'evy tail condition gives
$\int_1^\infty r^{-1-\alpha}h_\xi(r)\,dr<\infty$ almost everywhere.
Since $\alpha\le0$ and $h_\xi$ is decreasing, this forces
$h_\xi(\infty)=0$. The same radial criterion now gives membership
at $\alpha$.
\end{proof}

\begin{corollary}[Absolute definability through the exponential kernel]\label{r79:abs}
For every $-\infty<\beta<\alpha<2$,
\[
D^0(\Psi_\alpha)
  =\{\mu\in D^0(\Psi_\beta):\ \Psi_\beta\mu\in D^0(\Phi_{\beta,\alpha})\}.
\]
\end{corollary}
\begin{proof}
The forward inclusion is \cite[Lemma~3.2]{r79:Sato06}. For the reverse
inclusion, ordinary composition, including its domain equality
\cite[Theorem~3.1]{r79:Sato06}, gives $\mu\in D(\Psi_\alpha)$.
When $\alpha<0$ this whole domain is absolute. When $0\le\alpha<2$
and $\alpha\ne1$, ordinary and absolute domains again agree
\cite[Theorem~4.2 and Section~5.1]{SatoFractional2010}.

At $\alpha=1$ the input has mean zero and its L\'evy measure $\nu$
satisfies $\int_{|x|>1}|x|\nu(dx)<\infty$. For $\Psi_1$ and
$\Phi_{\beta,1}$, an ordinarily admissible law with jump measure
$\eta$ belongs to the absolute domain precisely when
$T_\eta(t)=\int_{|x|>e^t}x\,\eta(dx)$ is in
$L^1([0,\infty);\mathbb R^d)$ \cite[Theorem~2.8, (2.46), and (3.1)]{r79:Sato06}.
Write $\tau(du)=u^{-1-\beta}e^{-u}\,du$ and $c=\int_0^\infty
u\,\tau(du)=\Gamma(1-\beta)>0$. The jump measure $\nu'$ of
$\Psi_\beta\mu$ is the mixture of the images of $\nu$ under
$x\mapsto ux$ against $\tau$. Put $\log^+r=\max(0,\log r)$.
The two threshold indicators for $u|x|$ and $|x|$ differ on a
$t$-interval of length $|\log^+(u|x|)-\log^+|x||$. The triangle
inequality and Tonelli's theorem therefore give
\[
\begin{split}
\|T_{\nu'}-cT_\nu\|_1
&\le \iint u|x|\,|\log^+(u|x|)-\log^+|x||\,\tau(du)\nu(dx)\\
&\le \int u|\log u|\,\tau(du)\int_{|x|>1}|x|\,\nu(dx)
 +\int u^2\,\tau(du)\int_{|x|\le1}|x|^2\,\nu(dx)<\infty .
\end{split}
\]
Here $\log^+$ is Lipschitz in the log coordinate and
$\log^+(u|x|)\le u|x|$ on $|x|\le1$. The finite bound justifies
the vector integrals of the differences. The hypothesis gives
$T_{\nu'}\in L^1$, hence $T_\nu\in L^1$ and
$\mu\in D^0(\Psi_1)$.
\end{proof}

The critical domain test is unchanged by certain oscillations in log
radius. Write $F_\nu(t)=\int_{|x|>e^t}x\,\nu(dx)$ for an input with
finite first moment. Given a positive locally integrable function $w$
on $[0,\infty)$ and a number $b>0$, put
$H(T)=\int_0^T(w(t)-b)\,dt$.

\begin{theorem}[Preservation of the critical stochastic integral domains]
\label{r80:criticaldomains}
The integral $\int_0^\infty(w-b)F_\nu\,dt$ converges for every
finite-first-moment input if and only if $H$ is bounded. Under this
condition let $f_w$ be the decreasing integrand with scale measure
$u^{-2}w(\log(1/u))\,du$ on $(0,1)$. Its absolute, ordinary,
compensated and essential domains coincide, respectively, with the
four domains of $s\mapsto(1+s)^{-1}$.

For $w(t)=1+a\cos t$, $0<a<1$, the kernel $f_w$ is positive and
strictly decreasing. For every $c>0$,
$\int_1^\infty|f_w(s)-c/s|\,ds=\infty$, and $sf_w(s)$ has no limit
at infinity. This provides an example for \cite[Section~5, Comment~8]{r79:Sato06}
beyond the integrable perturbations in \cite[Example~4.5(iv)]{r80:Sato07}.
\end{theorem}
\begin{proof}
Both $F_\nu$ and $|F_\nu|$ have finite variation and tend to zero.
For either function $h$, integration by parts gives convergence of
$\int_0^T(w-b)h\,dt$ when $H$ is bounded. For the converse, every
finite signed measure $\lambda$ on $(0,\infty)$ gives a scalar
compound Poisson input with $F_\nu(t)=\lambda((t,\infty))$. Place
the jump masses $e^{-v}\lambda^\pm(dv)$ at $\pm e^v$, respectively.
Fubini's theorem gives
\[
 \int_0^T(w-b)F_\nu\,dt
     =\int H(\min\{v,T\})\,\lambda(dv).
\]
As a functional on finite signed measures with the variation norm,
the right side has norm $\max_{0\le v\le T}|H(v)|$. The uniform
boundedness principle gives the necessity.

Suppose $H$ bounded. Define $s(t)=\int_0^t e^v w(v)\,dv$ and
$f_w(s(t))=e^{-t}$. For $r\ge1$ the integrated jump test equals
\[
 \int_0^\infty (1\wedge r^2e^{-2t})e^t w(t)\,dt
       =r\int_0^\infty e^{-|t-\log r|}w(t)\,dt .
\]
Every unit interval has $w$-mass at most $b+2\|H\|_\infty$, and
every interval of fixed length $L>2\|H\|_\infty/b$ has mass at least
$bL-2\|H\|_\infty$. Thus $s(t)\to\infty$ and the last display is
bounded above and below by positive constants times $r$. For
$0<r<1$ the integrated jump test is $r^2\int_0^\infty e^{-t}w(t)\,dt$,
with finite coefficient. The same coefficient controls the Gaussian part.
The triplet criterion \cite[Proposition~2.1 and (2.8)]{r79:Sato06}
therefore gives the essential domain of all infinitely divisible
inputs with finite first moment.

With the truncation $x\mathbf1_{\{|x|\le1\}}$, scaling such an input
of mean $m$ by $u=e^{-t}$ gives drift $u(m-F_\nu(t))$.
The ordinary drift test is convergence of
$\int_0^T w(t)(m-F_\nu(t))\,dt$. It forces $m=0$, since
$F_\nu(t)\to0$ and $\int_0^T w=bT+O(1)$. For a centered input
the absolute drift test is $\int_0^\infty w|F_\nu|\,dt<\infty$.
Under the essential jump and Gaussian bounds, the absolute drift test
is equivalent to absolute integrability of each characteristic
exponent. Indeed the truncated jump remainder is bounded by
$C_z(1\wedge|ux|^2)$ at each test vector $z$.
The first paragraph applied to $F_\nu$ and $|F_\nu|$ removes $w$
from both tests. Compensation allows the shift by $-m$. These are
exactly the four domains of \cite[Theorem~2.8]{r79:Sato06}, including
all three strict inclusions.

Finally, for $w(t)=1+a\cos t$ direct integration gives
$s(t)e^{-t}=1+\frac a2(\cos t+\sin t)-(1+\frac a2)e^{-t}$.
The subsequences $t=2\pi n$ and $t=(2n+1)\pi$ give distinct limits.
Writing $L(t)=s(t)e^{-t}$, the absolute deviation in the theorem
integrates $|L(t)-c|w(t)/L(t)$ in $t$. For large $t$ the ratio
$w(t)/L(t)$ is bounded below, and $L$ differs by an integrable term
from a nonconstant periodic function. The integral therefore diverges.
\end{proof}

The order of decay in this example is forced by the essential domain
alone. Suppose a positive decreasing $f$ has precisely the finite-first-
moment laws as its essential domain, and let $\tau$ be its scale measure.
Put $K(r)=\int(1\wedge r^2u^2)\tau(du)$. Then $k r\le K(r)\le C r$
for all sufficiently large $r$, with positive $k,C$. Indeed, if
$K(r_n)/r_n\ge2^n$ along a sequence $r_n\ge2^n$, jump masses
$2^{-n}/r_n$ at $r_n$ give a finite first moment and an infinite
integrated jump test. If $K(r_n)/r_n\le2^{-n}$, masses $1/r_n$
give the reverse contradiction. For fixed $h>2C/k$ and large $r$,
\[
 kr/(2h)\le K(r/h)-h^{-2}K(r)
            \le \tau((1/r,\infty))\le K(r)\le Cr .
\]
The middle bound follows pointwise from the two truncated squares.
Since $\tau((1/r,\infty))$ is the length of $\{s:f(s)>1/r\}$,
monotonicity gives $f(s)\asymp s^{-1}$. The converse for the essential
domain is \cite[Example~4.5(i)]{r80:Sato07}.

\subsection{Two dilations}\label{avg:sec:range}
On the probability laws $\mathcal P_d$ on $\mathbb R^d$ set
$K_b\mu=\mu*T_b\mu$ for dilation $T_b$ by $b>1$. On infinitely divisible
laws $\mathrm{ID}_d$ this is Sato's integral mapping for
$f=b\mathbf1_{[0,1]}+\mathbf1_{(1,2]}$ by independent L\'evy increments.
Put $(D\nu)(B)=\nu(bB)$ and $\Delta=I-D$, testing signed measures away
from zero. An ID law belongs to $L_\infty(b)$ exactly when its jump measure
satisfies $\Delta^j\nu\ge0$ for all $j\ge1$ \cite[Proposition~2.3]{MSW}.
The corresponding exponent differences have Gaussian matrices $(1-b^{-2})^jA$
and unrestricted truncation-adjusted drifts.
\begin{theorem}[The two-dilation range]\label{avg:thm:two-scales}
For all $b>1$ and finite $d$,
\[
 \bigcap_{n\ge1}K_b^n(\mathrm{ID}_d)=L_\infty(b).
\]
\end{theorem}
\begin{proof}
Put $c_N=(2\lceil N/2\rceil+1)^{-1}$.
Set $p_N(t)=(1-t)^2+c_N(1+t)^2$. The coefficient of $t^j$ in
$(1-t)^2(1+t)^N$ divided by $\binom Lj$, $L=N+2$, is
$((L-2j)^2-L)/(L(L-1))$ for $0\le j\le L$, with minimum $-c_N$.
Thus $(1+t)^Np_N(t)$ has nonnegative coefficients. Since $K_b=D^{-1}(I+D)$
on jump measures, raising this polynomial to the $k$th power gives
$p_N(D)^k\nu\ge0$ at depth $Nk$. 

At unbounded depth let $N\to\infty$ at fixed $k$ to get $\Delta^{2k}\nu\ge0$.
For odd $j$, $\sum_{h=0}^{m-1}D^h\Delta^{j+1}\nu=\Delta^j\nu-D^m\Delta^j\nu\ge0$.
The last term tends vaguely to zero by finiteness of L\'evy mass away from zero.
This gives all differences. Sato proved the reverse inclusion
\cite[Example~4.5]{SatoRanges}.
\end{proof}

\subsection{Iterated ranges for finite real scale measures}
\label{ap:finite-section}
Allow any finite positive measure $\tau$ on $\mathbb R$ with
$\int s^2\,\tau(ds)<\infty$. A mass at zero contributes nothing to
$C_{\Phi_\tau\mu}(z)=\int C_\mu(sz)\,\tau(ds)$. If no nonzero scales
remain, the infinite range is $\{\delta_0\}$. Otherwise let $G$ be the
closed subgroup of $\mathbb R^*$ generated by the ratios of nonzero
support points, and put $a=\int s\,\tau(ds)$. Use the odd truncation
$h(x)=x\chi(|x|)$ with $\chi$ continuous, compactly supported and one
near zero. Let $T_s$ denote pushforward by $x\mapsto sx$. For a L\'evy measure $\zeta$ write
$k_s(\zeta)=\int(h(sx)-s h(x))\,\zeta(dx)$.

\begin{theorem}[Infinite range for real dilation kernels]\label{avg:thm:scale-kernel}\label{r68:realranges}
The possible jump measures in $\bigcap_{n\ge1}\Phi_\tau^n(\mathrm{ID}_d)$
are precisely the L\'evy measures $\nu=\int\zeta\,d\pi$ formed as
positive mixtures of components satisfying
\[
             T_s\zeta=|s|^\beta\zeta\qquad(s\in G).
\]
Here $0<\beta<2$ if $G$ contains a scale of modulus different from one;
for $G\subseteq\{1,-1\}$ take $\beta=0$. The Gaussian covariance is
arbitrary. If $a\ne0$ the drift is arbitrary. If $a=0$, choose any
$g<0$ in $G$. The unique drift for this jump measure is
\[
       \gamma=\int\frac{k_g(\zeta)}{|g|^\beta-g}\,d\pi .
\]
It is zero if $-1\in G$. Thus the infinite range depends only on $G$
and on whether $a$ vanishes. For positive scales the range is
$\mathrm{ID}_d$, $L_\infty(b)$ or $L_\infty$, according as
$G=\{1\}$, $G=b^{\mathbb Z}$ or $G=(0,\infty)$.
\end{theorem}
\begin{proof}
Divide all scales by one nonzero support point $s_0$. The new map is
$T_{1/s_0}\Phi_\tau$ and commutes with dilations, so its $n$th image
is $\Phi_\tau^n(T_{s_0^{-n}}\mathrm{ID}_d)=\Phi_\tau^n(\mathrm{ID}_d)$.
Its first moment is $a/s_0$, which vanishes precisely when $a$ does.
Now $1$ is in the support and all nonzero scales are in $G$.
Write $P\nu=\int T_s\nu\,\tau(ds)$, $m(q)=\int_{s\ne0}|s|^q\,\tau(ds)$
and $K\nu=\int_{s\ne0} k_s(\nu)\,\tau(ds)$, ignoring the atom at zero
throughout. The triplet map is
$(A,\nu,\gamma)\mapsto(m(2)A,P\nu,a\gamma+K\nu)$,
where $a$ now denotes the normalized first scale moment.
For $w(x)=1\wedge|x|^2$ put $\|\nu\|_w=\int w\,d\nu$.
Since $1\wedge s^2\le w(sx)/w(x)\le1\vee s^2$,
$c\|\nu\|_w\le\|P\nu\|_w\le C\|\nu\|_w$ with
$c=\int_{s\ne0}(1\wedge s^2)\,\tau(ds)>0$.
Also $|h(sx)-sh(x)|\le C_h(1+s^2)w(x)$: at small $x$ the difference
vanishes unless $|sx|$ is bounded away from zero. For $s$ in a compact
subset of $\mathbb R^*$ those differences have common compact support
away from zero. Truncating $s$ therefore gives vague continuity of
both $P$ and $K$ on sets of bounded $w$-mass.

For an output at every depth, its $j$th jump predecessors have
$w$-mass at most $c^{-j}\|\nu_0\|_w$. Diagonal vague extraction
gives a chain $\nu_j=P\nu_{j+1}$ at every $j$. All chains with
$\|\nu_0\|_w\le1$ form a compact metrizable convex set, by these bounds,
continuity and lower semicontinuity of the mass. The metrizable
Choquet theorem \cite{PhelpsChoquet} expresses a normalized chain
as a barycenter of extreme points. By additivity of $\|\nu_0\|_w$,
a nonzero extreme point $(\eta_j)$ spans an extreme ray.
Apply Deny's splitting argument \cite{DenyConvolution} to the identity
$(\eta_j)_j=\int(T_s\eta_{j+1})_j\,\tau(ds)$. Integrals over a Borel
set and its complement are positive chains, hence proportional to
$(\eta_j)$. A test with positive $\eta_0$-mass fixes the scalar, and
a countable determining family gives $(T_s\eta_{j+1})_j=c(s)(\eta_j)_j$
almost everywhere. Vague continuity extends this to the support with
$c(s)>0$. Since $1$ is a support point, successive coordinates are
proportional, and support ratios, products and continuity make $\eta_0$
an eigenmeasure for a continuous positive character of $G$.

The character is $|s|^\beta$: signs in the kernel of modulus have finite
order and the proper modulus map has closed subgroup image in $(0,\infty)$.
If some $|s|>1$, summing $w$-mass over both inward and outward annuli
forces $0<\beta<2$. The index is measurable by testing one such dilation
on the first determining function with positive $\eta_0$-mass.
With $\zeta=\eta_0$ we obtain
$\nu_j=\int m(\beta)^{-j}\zeta\,d\pi$ from the same barycenter,
where the overall normalization is absorbed into $\pi$. Conversely
this formula gives jump predecessors at every depth since $m(\beta)\ge c$.
For positive scales the trivial group admits every jump measure;
for $b^{\mathbb Z}$ the cone is that for $\delta_1+\delta_b$ in
Theorem~\ref{avg:thm:two-scales}; for $(0,\infty)$ the homogeneous
components have radial measures $r^{-1-\beta}dr$ and their mixtures
characterize $L_\infty$ \cite[Sec.~1]{SatoRanges}.
At depth $j$ take $A_j=m(2)^{-j}A$. If $a\ne0$ and $\gamma_0$ is
prescribed, recursively set $\gamma_{j+1}=(\gamma_j-K\nu_{j+1})/a$.

When $a=0$, oppositely signed support points give some $g<0$ in $G$.
On a component put $l=k_g(\zeta)/(|g|^\beta-g)$. The identity
$k_{sg}(\zeta)=k_s(T_g\zeta)+s k_g(\zeta)$ and commutativity give
$k_s(\zeta)=(|s|^\beta-s)l$ for every $s\in G$, so $K\zeta=m(\beta)l$.
In extracting the chain for the given law the equation
$\gamma=K\nu_1$ passes to the limit by the continuity proved above.
Thus $\gamma=\int l\,d\pi$. Conversely assign the $j$th predecessor
drift $\gamma_j=\int m(\beta)^{-j}l\,d\pi$. These integrals are finite
since $|l|\le C_g\int w\,d\zeta$, and $K\nu_{j+1}=\gamma_j$ at every
depth. To see that the value is intrinsic, put
$q=|g|^\beta$ and approximate $(q-g)^{-1}$ uniformly by polynomials
on the compact interval of possible $q$ (including its endpoints).
For each integer $j\ge0$, $\int q^j k_g(\zeta)\,d\pi=k_g(T_g^j\nu)$.
The limit therefore depends only on $\nu$; taking $g=-1$ gives zero
whenever that sign belongs to $G$.
\end{proof}

The possible groups are $\{1\}$, $\{1,-1\}$, $b^{\mathbb Z}$,
$(-b)^{\mathbb Z}$, $\{\pm b^n:n\in\mathbb Z\}$, $(0,\infty)$ and
$\mathbb R^*$, with $b>1$. Indeed the modulus image is closed, and
when the modulus has trivial kernel its inverse defines a continuous
sign character, trivial on $(0,\infty)$ and determined by the
generator on $b^{\mathbb Z}$. Sato's equal-modulus examples
\cite[Examples~4.1--4.2]{SatoRanges} give the centered and freely
translated symmetric cases. For $\tau=\delta_1+\delta_{-b}$ the range
$\mathcal C_b$ has unrestricted drift. Moreover
$\mathcal C_b\subseteq\mathcal C_c$ exactly when $c=b^k$ for a positive
odd integer $k$: odd powers preserve the component equation; conversely
$\sum_{n\in\mathbb Z}b^{-n\beta}\delta_{(-b)^n u}$ for a unit $u$
has support invariant under $-c$ only for such $c$, and every mixture
for $-c$ has invariant support in the punctured space.
For the same orbit and $\tau=b\delta_1+\delta_{-b}$, the first
scale moment vanishes. Taking $\chi=1$ on $[0,1]$ and $\chi=0$ on
$[b,\infty)$ gives $k_{-b}(\zeta)=bu$, so the unique drift is
$bu/(b^\beta+b)$.

\paragraph{The six kernel families in Sato's concluding remarks.}
Section~5 of \cite{SatoRanges} lists six ordinary iterated-range problems.
Three have finite scale measures. The truncated exponential
$e^{-t}\mathbf1_{(0,c)}(t)$, $c>0$, occurs in the periodic
Ornstein--Uhlenbeck setting of \cite{r61:Pedersen}.
Its scale density $u^{-1}\mathbf1_{(e^{-c},1)}(u)$ vanishes on
$(0,e^{-c})$. Thus the positive decreasing density hypothesis on the
whole interval $(0,t_0)$ in \cite[Theorem~3.4]{MaejimaSatoNested}
does not apply. Each finite product of these scales still has a
positive lower endpoint.
The cosine $\cos(\pi t/2)\mathbf1_{(0,1)}(t)$ has the type~$A$
one-step range of \cite{r61:arcsine}. Both scale supports contain an
interval, so Theorem~\ref{avg:thm:scale-kernel} gives $L_\infty$ for
both limits. For the cosine this also follows from
\cite[Theorem~3.4]{MaejimaSatoNested}: the product of two independent
cosine scales has decreasing density
\[
 \frac4{\pi^2}\int_0^{\pi/2}
       \frac{dt}{\sqrt{x^2+(1-x^2)\sin^2t}},\qquad 0<x<1 .
\]
This follows by substituting $u^2=x^2+(1-x^2)\sin^2t$ in the product
density integral. Even iterates have the same intersection as all iterates.

For the Mittag--Leffler mapping $\Upsilon_\theta$, $0<\theta<1$, the scale
$Y=S_\theta^{-\theta}$, with $\E e^{-t S_\theta}=e^{-t^\theta}$, has full positive support and moments
$\E Y^q=\Gamma(1+q)/\Gamma(1+\theta q)$ for $q>0$
\cite[Examples~2.4(2) and~7.7(2)]{BRT}. Simon's theorem
\cite[part~(b)]{r61:Simon} gives a decreasing density at one step
exactly for $\theta\le1/2$. A finite iterate has this property for every
$\theta$. Indeed, for $K(q)=\log\E Y^q$ the digamma formula
\cite[(5.9.12)]{r77:gamma} gives
$K''(q)=\int_0^\infty y e^{-qy}\ell_\theta(y)\,dy$, where
\[
 \ell_\theta(y)=(e^y-1)^{-1}-(e^{y/\theta}-1)^{-1}.
\]
Integrating twice gives the jump density $\ell_\theta(y)/y$ at $-y$
for the infinitely divisible law of $\log Y$. Put $b=(1-\theta)/\theta$ and $N=\lceil\max(1,1/b)\rceil$.
For $x=e^{-y}$ one has
$e^y\ell_\theta(y)\ge(1-x^b)/(1-x)\ge\min(1,b)$.
Thus $N\ell_\theta(y)/y$ contains the jump density $e^{-y}/y$ of
$\log U$, $U$ uniform on $(0,1)$. Subtracting it in the
L\'evy--Khintchine formula factors the product of $N$ independent
copies of $Y$ as $UZ$ in law, with $Z>0$ independent of $U$.
Its density $\E[Z^{-1}\mathbf1_{\{Z>u\}}]$ is nonincreasing and positive
for every $u>0$ by full support. Its second moment is $(\E Y^2)^N$.
The decreasing-density theorem \cite[Theorem~3.4]{MaejimaSatoNested}
therefore gives $L_\infty$ for the infinite range of $\Upsilon_\theta^N$.
The iterated ranges are nested, so the subsequence of depths $Nk$ has
the same intersection. This argument covers the original one-dimensional
mapping and its extension to every finite dimension.

The remaining cases use infinite or parameter-dependent scale measures.
For the critical logarithmic map $\Lambda_{q,1}$, Sato's $q=1$
limit imposes weak mean zero. For $q\ne1$, Theorem~\ref{ap:crit-depth}
characterizes the ordinary infinite intersection by the stable-index
mixture and $\lceil1/q\rceil$ successive drift tests. Once these hold,
the same mixture determines predecessors at every greater depth. For the barred map $\bar\Phi_{p,\alpha}$,
Theorem~\ref{ap:barred-range} covers $0<p<1$ and all $\alpha<2$.
Sato's ordinary iterated-range theorem covers $p\ge1$
\cite[Theorem~1.1]{SatoRanges}; the later inversion theorem retains
that restriction \cite[Theorem~6.4]{SatoInversions}.
For $\Psi_{\alpha,\beta}$ with $1\le\alpha<2$, Maejima and Ueda's
iteration theorem excludes $\alpha=1+n\beta$, $n=0,1,\ldots$
\cite[Theorem~2.6]{MaejimaUedaComposition}. The point
$(\alpha,\beta)=(1,1)$ is covered by \cite[Theorem~1.1]{SatoRanges}.
For $\alpha=1$, \cite[Theorem~3.1]{SatoUedaWeak} treats the one-step
range for all $\beta>0$, whereas its Theorem~4.1 gives the ordinary
iteration for $\Psi_{1,1}$. Theorem~\ref{avg:thm:resonant-tempering}
includes the remaining resonant parameters.

\subsection{Whether one integral realizes the limit class}
\begin{proposition}[A single stochastic integral cannot have range $L_\infty$]
\label{avg:prop:single-range}
Let $\Phi_f$ denote the improper integral of a real locally square-integrable
$f$ on $[0,\infty)$ against a L\'evy process on $\mathbb R^d$, $d\ge1$.
Write $R^0(\Phi_f)\subset R(\Phi_f)\subset R^e(\Phi_f)$ for its ranges on the full absolute, ordinary and essential domains. If $R^e(\Phi_f)$ contains a Gaussian law with
nonzero covariance, $R^0(\Phi_f)$ contains a non-Gaussian law of finite variance.
Consequently none of these ranges equals the completely selfdecomposable class
$L_\infty$, or either of Sato's classes $L^0_{\infty,\alpha}$ and
$L^e_{\infty,\alpha}$ for any $\alpha<2$
\cite[Remark~7.13]{SatoFractional2010}.
\end{proposition}
\begin{proof}
Even for an essential integral, which permits deterministic centering, the
Gaussian matrix of the output is $\int_0^\infty f(s)^2 A_{\mathrm{in}}\,ds$
\cite[Proposition~3.18]{SatoFractional2010}. A nonzero Gaussian matrix thus
forces $0<c:=\int_0^\infty f(s)^2\,ds<\infty$. Take a rate-one compound
Poisson input with equally likely jumps $\pm e$, $|e|=1$. Its log
characteristic function $C$ satisfies $|C(f(s)z)|\le f(s)^2\langle z,e\rangle^2/2$,
so the integral exists absolutely and belongs to all three ranges. Its output
is symmetric with zero Gaussian matrix and jump measure $\nu$ satisfying
$\int |x|^2\nu(dx)=c$. It is therefore non-Gaussian with finite variance.

Every nonzero jump measure in $L_\infty$ has infinite second moment, since
its radial marginal is a positive mixture of $r^{-1-\beta}\,dr$, $0<\beta<2$
\cite[(1.6)]{SatoFractional2010}. Both target classes in Remark~7.13 lie in
$L_\infty$ by \cite[Theorem~7.11]{SatoFractional2010}. To see that both contain
each centered Gaussian even at $\alpha=1$, recall that they are the intersections
over $p>0$ of the absolute and essential ranges, respectively, of kernels with
$\int_0^\infty\ell_{p,\alpha}(s)^2\,ds=
\Gamma(p)^{-1}\int_0^1(-\log u)^{p-1}u^{1-\alpha}\,du=(2-\alpha)^{-p}$.
Rescaling the input covariance realizes the Gaussian in every absolute range.
\end{proof}

\subsection{Can a different kernel realize the essential range?}
Maejima, P\'erez-Abreu and Sato consider the class
$\mathcal A^q_{\alpha,p}(\mathbb R^d)$ with arbitrary covariance and
location and with L\'evy measures given by positive scale mixtures
with
\[
 \tau(du)=\Gamma(p)^{-1}(1-u^q)^{p-1}u^{-1-\alpha}
                    \mathbf1_{(0,1)}(u)\,du.
\]
Their Theorems~4.3--4.4 identify the ordinary range of the canonical
kernel for $\alpha<1$ and its essential range for $1\le\alpha<2$.
In Remark~4.1 they ask if another ordinary kernel works in the latter
case \cite{r81:MPS}.

\begin{theorem}[Ordinary representation of fractional jump classes]
\label{r81:mps}
For $p,q>0$, $d\ge1$ and $\alpha<2$,
$\mathcal A^q_{\alpha,p}(\mathbb R^d)$ is the full ordinary range of a
locally square-integrable real integrand on $[0,\infty)$ if and only if
$\alpha<1$.
\end{theorem}
\begin{proof}
The positive direction is \cite[Theorem~4.3]{r81:MPS}. For $1\le\alpha<2$
the target contains every point mass and also a law whose nonzero jump
measure lies on one positive ray: take the $\tau$-mixture of a unit jump
on that ray. Every nonzero target jump measure $\widetilde\nu$ satisfies
\[
             \int_{|y|\le1}|y|\,\widetilde\nu(dy)=\infty .
\]
Indeed, each nonzero predecessor jump $x$ contributes
$|x|\int_0^{\min(1,1/|x|)}u\,\tau(du)=\infty$; apply Tonelli.

Suppose $R(\Phi_f)$ equals the target. If $f$ takes both signs on sets
of positive measure, the L\'evy measure of any output with jumps is
\[
   \int_0^\infty\int \mathbf1_{\{f(s)x\in\cdot\}}\,\nu(dx)\,ds
   \quad\text{on }\mathbb R^d\setminus\{0\}.
\]
It cannot be supported on one positive ray: for every nonzero $x$ at
least one of the two signs puts $f(s)x$ outside the ray. Thus $f$ has
constant sign on its nonzero set. A nonzero point mass in the range
must come from a deterministic drift, since the Gaussian covariance and
jump measure of a nontrivial random input cannot disappear under such
$f$. Ordinary convergence of that drift forces $f\in L^1(0,\infty)$.
For a rate-one Poisson process with unit jumps, the expected sum of the
absolute jump contributions is now $\int|f|<\infty$. Its ordinary
output has a nonzero jump measure of finite total first moment, in
contradiction to the displayed divergence.
\end{proof}

The accompanying question in \cite[Remark~4.1]{r81:MPS} concerns other
generalized Upsilon maps with the same jump range. The following
argument identifies all pairs giving equality. For a scale measure
$\sigma$ write $D_c\sigma$ for its image under $u\mapsto cu$ and
disregard any scale mass at zero. For arbitrary real $b$, the maps
$\Upsilon^{(\sigma,b)}$ use predecessor measure $|x|^b\rho(dx)$
with $\rho$ locally finite off zero \cite[(1.6)]{r81:MPS}.
This weighting is a bijection on locally finite measures, so the
output range intersected with the L\'evy measures is independent of $b$.

\begin{proposition}[Scale measures with the same jump range]\label{r81:upsilon}
For any nonzero positive scale measure $\tau$ on $(0,\infty)$ with
$0<\int u^2\tau(du)<\infty$, this intersection is the same for $\sigma$
and $\tau$ exactly when $\sigma=C D_c\tau$ for $C>0$, $c\ne0$.
\end{proposition}
\begin{proof}
Realize the profile $\tau$ on one ray. This forces $\sigma$ to
have constant sign, as above, and forces the predecessor onto that ray
up to reflection. Reflecting $\sigma$ if needed, write the equality of
radial measures as $\tau=\sigma\circledast\eta$, where $\circledast$
is convolution under multiplication of positive radii. Taking second
moments gives finite positive second moments for $\sigma$ and $\eta$.
Hence $\sigma$ itself is a L\'evy measure realized on one ray by a unit
jump. Range equality gives $\sigma=\tau\circledast\theta$, also with
finite positive second moment for $\theta$.
Tilt each measure by $u^2$, normalize its mass and take the image
under $\log u$. The resulting probability measures satisfy
$P_\tau=P_\tau*P_\theta*P_\eta$. The characteristic function of
$P_\tau$ is nonzero near zero. Hence the characteristic function of
$P_\theta*P_\eta$ equals one there, forcing this convolution to be
$\delta_0$. Its two independent summands are constant, so $\theta$
is a single atom and $\sigma=C D_c\tau$. Conversely such a change of
scale and mass is absorbed by dilating and weighting the predecessor.
The fractional $\tau$ above has the required second moment for every
$\alpha<2$.
\end{proof}

\subsection{Finite depth for the critical logarithmic kernels}\label{ap:crit-section}
Let $T=\Lambda_{q,1}$, $q>0$, be the stochastic integral mapping on
$\mathrm{ID}_d$ with scale measure
$(-\log u)^{q-1}u^{-2}\,du/\Gamma(q)$ on $0<u<1$.
Write $R(T^n)$ for the range of the $n$th ordinary improper iterate.
Sato identified $\bigcap_n R(T^n)$ for $q=1$ and asked for the other
orders \cite[Theorem~1.1 and Section~5]{SatoRanges}.
Use truncation $x\mathbf1_{\{|x|\le1\}}$ and write $(A,\nu,\gamma)$ for
the Gaussian covariance, jump measure and drift. Put $S=\mathbb S^{d-1}$.
The class $L_\infty^{(1,2)}$ consists of the infinitely divisible laws whose jump measures have
the stable-mixture representation
\begin{equation}\label{ap:crit-mixture}
 \nu(B)=\int_{(0,1)\times S}\int_0^\infty
 \mathbf1_B(r\xi)r^{-2-s}\,dr\,M(ds,d\xi),\qquad
 M\ge0,\quad \int\frac{M(ds,d\xi)}{1-s}<\infty .
\end{equation}
The mixing measure $M$ is unique. Write $a(ds)=\int_S\xi\,M(ds,d\xi)$.

\begin{theorem}[Finite depth at stable index one]\label{ap:crit-depth}
For $N=\lceil1/q\rceil$ and every finite dimension,
\[
 \bigcap_{n\ge1}R(T^n)= L_\infty^{(1,2)}\cap R(T^N).
\]
For a law with \eqref{ap:crit-mixture}, set
$b_N=-\int s^{Nq-1}a(ds)$ and successively form, for $j=N,\ldots,1$,
\begin{equation}\label{ap:crit-recursion}
 b_{j-1}=\lim_{R\to\infty}\int_0^R\frac{t^{q-1}}{\Gamma(q)}
 \left[b_j+\int s^{jq-1}(1-e^{-st})\,a(ds)\right]dt .
\end{equation}
Membership is equivalent to existence of all these vector limits with
$b_0=\gamma$. The Gaussian covariance $A$ is arbitrary.
\end{theorem}
\begin{proof}
Write $U_p$ for the induced map on jump measures at order $p$. There
$U_pU_q=U_{p+q}$ on the natural domains and the ranges decrease with
the order \cite[Theorem~7.1 and Corollary~7.2]{SatoFractional2010}.
Thus $\bigcap_n R(U_q^n)=\bigcap_{p>0} R(U_p)$. A jump measure is in
$R(U_p)$ precisely when it occurs in the essential range at order $p$,
where arbitrary locations are allowed. Sato's description of the
intersection of these essential ranges gives exactly \eqref{ap:crit-mixture}
\cite[Theorems~6.9 and~7.11]{SatoFractional2010}. Thus every
ordinary output of arbitrarily large depth has this form.

One application multiplies the mixing measure at index $1+s$ by
$\int_0^\infty t^{q-1}e^{-st}dt/\Gamma(q)=s^{-q}$ and preserves $A$.
Since $U_q$ is injective \cite[Theorem~6.6(ii)]{SatoFractional2010},
the only possible $j$th predecessor has covariance $A$ and jump measure
given by $s^{jq}M$. For an input with these data and drift $b_j$, dilation
by $u=e^{-t}$ contributes the truncation correction
$u\int_{1<|x|\le e^t}x\,\nu_j(dx)$. Substitution of $s^{jq}M$ shows that
the output drift over $0\le t\le R$ is the integral in
\eqref{ap:crit-recursion}. It is finite since $(1-e^{-st})/s\le t$.
The Gaussian and jump terms of the truncated exponents converge to the
prescribed output terms by domination by the full covariance and the
L\'evy norm of $\nu_{j-1}$. Thus ordinary convergence is equivalent to
convergence of this drift, which fixes the output location.

At $j=N$ the large-jump first moment is $\int s^{Nq-1}M(ds,d\xi)<\infty$.
The bracket in \eqref{ap:crit-recursion} tends to the input mean. Convergence
against $t^{q-1}dt$ on $(1,\infty)$ forces this mean to be zero. Hence
$b_N$ is forced in every $N$-step history. Conversely the stated recursion
constructs such a history. For $j>N$ choose $b_j=-\int s^{jq-1}a(ds)$.
The bracket then equals $-\int s^{jq-1}e^{-st}a(ds)$, and the integral
of its norm is bounded, with weight $t^{q-1}/\Gamma(q)$, by
$\int s^{(j-1)q-1}|a|(ds)<\infty$. Fubini gives the output $b_{j-1}$.
This extends every admissible history to all depths.
\end{proof}

For $q\ge1$ there is one limit. With
$P_q(v)=\Gamma(q)^{-1}\int_0^v t^{q-1}e^{-t}\,dt$, it reads
\[
 \lim_{R\to\infty}\int P_q(sR)\,\frac{a(ds)}s=-\gamma .
\]
For $q=1$ the integral on the left equals $\int_{1<|x|\le e^R}x\,\nu(dx)$, so this
is precisely weak mean zero. For $q>1$ the condition implies weak mean
zero. Indeed, for $B\sim\mathrm{Beta}(1,q-1)$, the beta--gamma identity gives
$P_1(v)=\E P_q(v/B)$. Fubini applies because $P_1(sR)/s\le R$, and
the $q$-cutoff converges uniformly over arguments at least $R$ as
$R\to\infty$.

\paragraph{Absolute integrability and composition at the critical index.}
Sato asks for the absolute all-orders class in Remark~7.12 of
\cite{SatoFractional2010} and asks about ordinary composition and strict
range nesting in Remarks~7.5 and~7.8 there. The absolute test involves
one vector-valued function of the radial coordinate. Its first two orders
already distinguish the ordinary maps.

We first treat a single ordinary or absolute output on its full
essential jump range. The same truncation also gives an input criterion
below order one. In Theorem~\ref{r61:satocore} the positive jump
measure already has the fractional semigroup law. What remains is to
recover the input's compensated drift from the output and test its
ordinary or absolute convergence. The resulting inverse formula
feeds the nesting and continuity results below.

Sato identifies the order-one output and records the remaining orders after
Theorem~4.5 of \cite{SatoInversions}.

\begin{theorem}[An output test for the two critical ranges]\label{ap:thm:critical-output}
Fix $p>0$. To test either $L_{p,1}$ or $L^0_{p,1}$, first require
$\mu\in L^e_{p,1}$ by the positive jump criterion of
\cite[Theorem~6.9]{SatoFractional2010}. In log radius $y=\log|x|$,
let $h(y)dy$ be the locally finite vector measure obtained by pushing $x\,\nu_\mu(dx)$
forward under this map. Write $p=j+r$ with $j$ a nonnegative integer
and $0<r\le1$, and set $h_j=(-\partial_y)^jh$ in the distributional
sense. It has a locally integrable representative. Put $K_r(u)=u^{-r}$
for $u>0$ and $K_r(u)=0$ otherwise. For $t>0$ define
\[
 M(t)=
 \begin{cases}
 \displaystyle\frac1{\Gamma(1-r)}\int_0^\infty
 [K_r(u)-K_r(u-t)]h_j(u)\,du,&0<r<1,\\[5pt]
 h_j(0+)-h_j(t+),&r=1 .
 \end{cases}
\]
Put $A(Y)=Y^{-1}\int_0^Y M(v^{1/p})\,dv$ and let $\gamma$ be the
drift for truncation $x\mathbf1_{\{|x|\le1\}}$. Then
\[
 \mu\in L_{p,1}\quad\Longleftrightarrow\quad
 A(Y)\to a\quad\text{and}\quad
 Y[A(Y)-a]\to\Gamma(p+1)\gamma\qquad(Y\to\infty).
\]
Membership in $L^0_{p,1}$ requires in addition
$\int_0^\infty |M(v^{1/p})-a|\,dv<\infty$.
The Gaussian covariance is arbitrary. These tests use only the output
triplet and describe both classes asked for in
\cite[Remark~6.13]{SatoFractional2010}.
\end{theorem}
\begin{proof}
Use the essential jump criterion to take a jump predecessor $\eta$.
Write $m$ for the vector log-radius image of $x\eta(dx)$, and
$m_*$ for the corresponding positive image of $|x|\eta(dx)$.
Then $h=I_p m$ by the critical log-scale kernel, and
$\int_{[1,\infty)}s^{p-1}m_*(ds)<\infty$
by \cite[Theorem~6.2]{SatoFractional2010}.
Distributional differentiation of $k_p$ gives
$h_j=I_rm$, locally integrable by the same tail bound and local
integrability of $k_r$. For $r=1$ use its right-continuous version
$h_j(u)=m((u,\infty))$.

We claim $M(t)=m((0,t])$. For $0<r<1$ integrate the difference kernel
against $k_r(s-u)$ on $0<u<s$. By the beta integral the first term,
including $1/\Gamma(1-r)$, gives $1$ for $s>0$, and the second
gives $1$ precisely for $s>t$. Absolute Fubini holds against $m_*$.
Indeed the absolute integral in $u$ is at most $2$ for every $s>0$,
since each nonnegative term integrates to at most $1$. For
$s>\max(4t,1)$, on $u<s/2$ the factor $k_r(s-u)$ is $O_r(s^{r-1})$
and the difference of the two powers has finite $L^1(0,\infty)$ norm.
On $s/2<u<s$ the difference is $O_{r,t}(s^{-r-1})$ and the integral
of $k_r(s-u)$ is $O_r(s^r)$. Thus the absolute integral is
$O_{r,t}(s^{r-1})$, integrable against $m_*$ since $r\le p$.
These estimates include atoms at $s=t$. At $s=0$ both contributions
are zero, as required by the truncation. For $r=1$ subtract the two
right-continuous tails to obtain the same claim.

A predecessor with drift $b$ has critical drift bracket $b+M(t)$:
the truncation correction before the factor $k_p(t)$ is
$\int_{1<|x|\le e^t}x\,\eta(dx)$. The jump integral already converges,
and the covariance is matched by choosing the input covariance equal to that of $\mu$.
The truncated drift at $t=T$ is therefore
\[
 \int_0^T k_p(t)[b+M(t)]\,dt
   =\frac{Y}{\Gamma(p+1)}[b+A(Y)],\qquad Y=T^p .
\]
Convergence to $\gamma$ forces $A(Y)\to-b$ and then gives exactly the
second limit. Conversely the two limits prescribe $b=-a$ and give
ordinary convergence to the specified output. The jump predecessor is
unique by \cite[Theorem~6.6(ii)]{SatoFractional2010}. Absolute
definability is equivalent to absolute integrability of the bracket
against $k_p(t)dt$ by \cite[Proposition~3.18]{SatoFractional2010}.
Changing variables $v=t^p$ gives the stated additional test.
\end{proof}

\begin{corollary}[Critical input domains below order one]\label{ap:cor:critical-input}
For $0<p<1$ the essential domain of $\Lambda_{p,1}$ consists of the
input triplets $(A,\eta,b)$ with
$\int_{|x|>e}|x|(\log|x|)^{p-1}\eta(dx)<\infty$
\cite[Theorems~6.2--6.3]{SatoFractional2010}. Let $m$ be the log-radius
image of $x\eta(dx)$ and set $V(t)=m((0,t])$,
$W(t)=\int_{(0,t]}s^p m(ds)$. Within that domain, ordinary definability
is equivalent to existence of a finite vector limit
\[
 \lim_{t\to\infty}\big\{t^p[b+V(t)]-W(t)\big\}.
\]
Absolute definability is equivalent to
$\int_0^\infty |b+V(v^{1/p})|\,dv<\infty$.
Both domain inclusions are strict. Even absolute definability can hold
for an input without a weak mean. This describes the domains in
\cite[Remark~6.4]{SatoFractional2010}.
\end{corollary}
\begin{proof}
The drift bracket is $b+V(t)$. Integrating against $k_p(t)dt$ on a
finite interval and interchanging the two finite integrals gives the
display divided by $\Gamma(p+1)$ before passage to the limit.
For absolute convergence substitute $v=t^p$.

For explicit one-dimensional examples take $b=0$, $A=0$, and prescribe
$V(v^{1/p})$ to equal $c_n$ on $[2^n,2^n+\ell_n)$ and zero elsewhere,
$n\ge1$. Each edge of height change $c\ne0$ at $v$ is realized by a jump
at $\operatorname{sgn}(c)e^{v^{1/p}}$ with rate $|c|e^{-v^{1/p}}$.
For the bounded heights and widths below the essential moment is finite
since $\sum_n 2^{n(p-1)/p}<\infty$. With $c_n=1$ and $\ell_n=2^{-n}$
the absolute test holds, yet $V(t)$ has no limit, so there is no weak
mean. With $c_n=(-1)^n/n$ and $\ell_n=1$ the ordinary integral
converges conditionally, including its incomplete last pulse.
A nonzero pure drift belongs only to the essential domain.
\end{proof}

\begin{theorem}[Absolute ranges at stable index one]\label{ap:thm:all-absolute}
With the truncation and notation of \eqref{ap:crit-mixture}, define
\[
 F(t)=\int_{(0,1)}e^{-st}a(ds),\qquad t\ge0 .
\]
Sato's class $L^0_{\infty,1}=\bigcap_{p>0}R^0(\Lambda_{p,1})$
consists exactly of the laws in $L_\infty^{(1,2)}$ satisfying
\begin{equation}\label{ap:eq:absolute-profile}
 \gamma=-\int_0^\infty F(t)\,dt,\qquad
 \int_0^\infty t^j |F^{(j)}(t)|\,dt<\infty
 \quad(j=0,1,2,\ldots).
\end{equation}
The Gaussian covariance is arbitrary. Equivalently,
$g(v)=e^v F(e^v)$ belongs to $W^{j,1}(\mathbb R;\mathbb R^d)$ for every $j$
and $\gamma=-\int_{\mathbb R} g(v)\,dv$.
\end{theorem}
\begin{proof}
The essential all-orders class is $L_\infty^{(1,2)}$ and the absolute
ranges decrease with the order \cite[Corollary~7.6, Theorem~7.11 and
Remark~7.12]{SatoFractional2010}. Thus one need only test positive
integer orders $p=j+1$ on \eqref{ap:crit-mixture}.
The unique possible input jump measure has spectral measure $s^pM$ and
its covariance equals $A$, by the same multiplier and injectivity as in
Theorem~\ref{ap:crit-depth}. The input has finite first moment since its
large-jump first moment is $\int s^{p-1}M(ds,d\xi)<\infty$.
For an input drift $b$, its integrated truncation correction is given by
the integrand in \eqref{ap:crit-recursion} with one step of order $p$,
namely $t^{p-1}/\Gamma(p)$ times
\[
 b+\int s^{p-1}(1-e^{-st})a(ds).
\]
The integrated Gaussian and jump terms are finite in the L\'evy norm by
the prescribed output triplet. Hence absolute definability is equivalent
to integrability of the norm of this weighted bracket
\cite[Proposition~3.18]{SatoFractional2010}. Its limit before weighting is
the input mean and must vanish. The bracket is therefore
$-\int s^j e^{-st}a(ds)=-(-1)^jF^{(j)}(t)$.
The order-$j+1$ test is precisely the $j$th bound in
\eqref{ap:eq:absolute-profile}, with output drift
$-(-1)^j\int_0^\infty t^j F^{(j)}(t)\,dt/j!$.

Under all the bounds these drifts coincide. For $j\ge1$ the derivative of
$t^jF^{(j-1)}$ is integrable, so the function has a limit at infinity.
The limit is zero because $t^{j-1}|F^{(j-1)}|$ is integrable. Its value
at zero is also zero. Integration by parts thus reduces each signed
integral to its predecessor and yields the stated drift. Finally, with
$D=d/dv$ and $t=e^v$, one has
$t^{j+1}F^{(j)}(t)=(D-j)\cdots(D-1)g(v)$.
These monic polynomials in $D$, including the identity for $j=0$,
give the Sobolev formulation by $dt=t\,dv$.
\end{proof}

\begin{proposition}[Ordinary factorization in one direction]\label{ap:prop:ordinary-factor}
For $p,q>0$, $\Lambda_{q,1}\Lambda_{p,1}$ extends
$\Lambda_{p+q,1}$ on its ordinary domain. Consequently, for every $p>0$,
\[
 R(\Lambda_{p,1})\supsetneq\bigcup_{r>p}R(\Lambda_{r,1}).
\]
\end{proposition}
\begin{proof}
Put $k_p(t)=t^{p-1}/\Gamma(p)$ for $t>0$. For input triplet
$(A,\nu,b)$ the drift integrand at order $p$ is $k_p(t)H(t)$ with
$H(t)=b+\int_{1<|x|\le e^t}x\,\nu(dx)$. If the direct order $p+q$
exists, $\int_0^\infty k_{p+q}(t)H(t)\,dt$ converges. The order-$p$
Gaussian and jump integrals are finite: on $t\ge1$ its weight is
bounded by a constant times the order-$(p+q)$ weight; on $0<t<1$
local integrability of $k_p$ suffices. Moreover
$k_p(t)/k_{p+q}(t)=Ct^{-q}$, so integration by parts against the
convergent direct primitive gives $b_p=\int_0^\infty k_p(t)H(t)\,dt$.
Writing $U_p\nu$ for the order-$p$ output jumps, the drift bracket
of this output at $v$ is
\[
 b_p+\int_{1<|x|\le e^v}x\,(U_p\nu)(dx)
   =\int_0^\infty k_p(t)H(t+v)\,dt .
\]
Indeed the difference term is the integral of
$k_p(t)[H(t+v)-H(t)]$, absolutely convergent since the displayed jump
annulus has finite absolute first moment under $U_p\nu$.

Truncate the second step at $v=T$. Its drift equals
\[
 \int_0^\infty k_{p+q}(x)H(x)
 I_{\min(1,T/x)}(q,p)\,dx ,
\]
where $I_y$ is the regularized incomplete beta integral.
To justify the exchange, first integrate over $0<v<T$, $v<x<L$
with $x=t+v$. The discarded $x>L$ integral tends uniformly to zero
in $v\in[0,T]$ by integration by parts against the convergent direct
primitive: $k_p(x-v)/k_{p+q}(x)=C x^{-q}(1-v/x)^{p-1}$ has supremum
and total variation $O_{p,q}(L^{-q})$ on $x>L>2T$.
The beta factor equals one on $x\le T$ and decreases to zero thereafter,
so the tail is bounded by $\sup_{u\ge T}|\int_T^u k_{p+q}H|$.
This tends to zero as $T\to\infty$. Covariance and jump measure agree
by the beta identity $k_q*k_p=k_{p+q}$ and Tonelli. Swapping $p,q$
gives the asserted range inclusion. Its strictness follows from the
symmetric example in the proof of \cite[Theorem~7.7]{SatoFractional2010},
whose order-$p$ absolute output lies outside every higher essential range.
\end{proof}
For the general pushforward formula on a common domain of definition see
\cite[Theorem~1 and Corollary~2]{JurekComposition}. The truncation above
checks directly which domain inclusion holds for the ordinary integrals.

\Needspace{5\baselineskip}
\begin{example}[A strict extension of the direct map]\label{r61:satocounter}
In one dimension let $n=k^6$, $h_k=k^{-3}8^{-n}$ and $c_k=k^{-2}2^{-n}$
for $k\ge2$, and put
\[
 F(t)=\sum_{k\ge2}c_k\big[e^{-h_kt}(1-e^{-h_kt})\big]^n .
\]
Each term is the Laplace transform of a signed measure on
$[nh_k,2nh_k]\subset(0,1/2)$ of variation $c_k2^n=k^{-2}$.
Their sum $a$ is realized in \eqref{ap:crit-mixture} by putting its
positive and negative parts on the two respective directions.
The terms in $F$ are nonnegative. The $k$th integral is
$c_k/[h_kn\binom{2n}{n}]=O(k/\sqrt{n})=O(k^{-2})$ by the beta
integral and the central binomial estimate. With $\gamma=-\int F$, the
order-one absolute test holds. This is absolute weak mean zero:
the angular first-moment density of the output jumps over $r>1$ is
$F(\log r)\,dr/r$. But at $t_k=(\log2)/h_k$ one has
$t_k F(t_k)\ge (\log2)c_k/(h_k4^n)=k\log2$.
Thus $g$ is unbounded. Since $W^{1,1}(\mathbb R)$ consists of functions
with bounded continuous representatives, the order-two test fails.

The same law distinguishes the ordinary mappings in
\cite[Remark~7.5]{SatoFractional2010}. Since $\mu\in L_\infty^{(1,2)}$
and $\mu\in R(\Lambda_{1,1})$, Theorem~\ref{ap:crit-depth} with $q=1$
puts $\mu$ in $R(\Lambda_{1,1}^2)$. At order two the only possible
input has spectral measure $s^2M$; convergence of its truncated output
drift forces the input mean to vanish. The truncated drift then equals
\[
 \int_0^T tF'(t)\,dt=T F(T)-\int_0^T F(t)\,dt.
\]
It is unbounded on $T=t_k$. Hence $\mu\notin R(\Lambda_{2,1})$ and
$\Lambda_{2,1}\ne \Lambda_{1,1}^2$ on their ordinary domains.

\end{example}

\paragraph{Continuity from below in the fractional order.}
For $\alpha\le0$, Sato proved the absolute-range identity below
and asked whether it persists for $0<\alpha<2$
\cite[Theorem~7.9 and Remark~7.10]{SatoFractional2010}.
Compactness of the inverse jump measures settles the jump part. Only the
location can fail at the endpoint.
\begin{theorem}[Continuity from below in the fractional order]\label{ap:thm:fractional-continuity}
For $p>0$ and $0<\alpha<2$,
\[
 \bigcap_{0<q<p}L^e_{q,\alpha}=L^e_{p,\alpha}.
\]
The same identity holds for the absolute ranges $L^0_{q,\alpha}$ if
$\alpha\ne1$. At $\alpha=1$, for every $p>0$ there is a law in dimension one
which lies in every lower absolute range and outside the ordinary
order-$p$ range.
\end{theorem}
\begin{proof}
Let $S=\mathbb S^{d-1}$ and write
$\tau(dy,d\xi)=e^{\alpha y}\nu(dy,d\xi)$, where $y=\log|x|$ and
$\nu$ here denotes the output jump measure in log-polar coordinates.
The essential range criterion is $\tau=I_q\sigma_q$, with $I_q$ acting on
the real coordinate by the right integral with kernel
$k_q(u)=u^{q-1}/\Gamma(q)$ for $u>0$
\cite[Theorem~6.9]{SatoFractional2010}. This joint formulation follows by
applying log-polar coordinates to the input jump measure; conversely, the
resulting candidate input is a L\'evy measure whenever its output is one,
by restricting the scale integral to $1/2<u<1$.

Take $q\uparrow p$, $q\ge p/2$, and multiply $\sigma_q(ds,d\xi)$ by
$w_q(s)=(1+s_+)^{q-1}/\Gamma(q)$. On each compact strip
$[b,+\infty]\times S$ these weighted measures have uniformly bounded
mass: integrate $I_q\sigma_q$ on $[b-2,b-1]\times S$, where
$(s-y)^{q-1}/(1+s_+)^{q-1}$ is bounded below uniformly for $s\ge b$.
A subsequence therefore converges vaguely on $(-\infty,+\infty]\times S$.
Write its limit as $w_p\sigma$ on the finite part and $C(d\xi)$ at $+\infty$.
For a compactly supported continuous test $\phi(y,\xi)$ the corresponding
test on weighted inputs is
\[
 \int_{y<s}\phi(y,\xi)
        (s-y)^{q-1}(1+s_+)^{1-q}\,dy .
\]
It vanishes for sufficiently negative $s$, converges uniformly on finite
compact sets to its value at $q=p$ by local $L^1$ convergence of the
kernels, and tends uniformly at $+\infty$ to $\int\phi(y,\xi)\,dy$.
Thus $\tau=I_p\sigma+dy\,C(d\xi)$. Any one lower order $q_0>0$ in the
intersection excludes $C$: positivity and $\tau=I_{q_0}\sigma_{q_0}$ give
\[
 R^{-1}\tau([R,2R]\times S)
 \le C_{q_0}\int_{(R,\infty)\times S}s^{q_0-1}\sigma_{q_0}(ds,d\xi)
 \longrightarrow0.
\]
Indeed the kernel average over $[R,2R]$ is at most $C_{q_0}s^{q_0-1}$,
by direct integration for $R<s<3R$ and pointwise comparison for $s\ge3R$.
The weighted tail is finite by the same fixed-interval estimate used above.
This proves the reverse essential inclusion; nesting gives the forward one.
For $0<\alpha<1$ the absolute and essential ranges coincide, and for
$1<\alpha<2$ the absolute range adds exactly the zero-mean condition,
independent of $q$ \cite[Proposition~6.8 and Theorem~6.12]{SatoFractional2010}.

Fix any critical endpoint $p>0$. In one dimension let $n=2^{k^2}$,
$h_k=e^{-n^2}$, $c_k=k4^n h_k^{p-1}$, $k\ge2$. The functions
\[
 H_k(t)=c_k[e^{-h_kt}(1-e^{-h_kt})]^n
       =\int e^{-st}\sigma_k(ds)
\]
come from signed binomial measures supported on $[nh_k,2nh_k]$ of
variation $c_k2^n$. Put $a_k(ds)=s^{2-p}\sigma_k(ds)$ and $a=\sum a_k$;
realize its positive and negative parts on the two directions in
\eqref{ap:crit-mixture}. The support lies in $(0,1/2)$ and, for every $q>0$,
\[
 \sum_k\int s^{q-1}|a_k|(ds)
 \le C_{p,q}\sum_k k8^n n^{q+1-p}h_k^q<\infty.
\]
In particular $a$ is finite and every order-$q$ input with spectral
measure $s^qM$ has finite first moment. Its mean must vanish for a
convergent output drift, as in Theorem~\ref{ap:thm:all-absolute}.

Set $\delta=p-2$. With $t_k=(\log2)/h_k$ the unique maximum point,
\[
 A_k=\int_0^\infty t^\delta H_k(t)\,dt=O_p(k/\sqrt n),\qquad
 B_k=\int_0^\infty t^{\delta+1}|H'_k(t)|\,dt=O_p(k).
\]
Indeed $A_k/k=\int_0^\infty y^\delta[4e^{-y}(1-e^{-y})]^n\,dy$.
The bracket has a unique nondegenerate maximum one at $\log2$.
A Gaussian upper bound near the maximum gives $O(n^{-1/2})$; on
its complement the bracket is bounded by $\rho<1$ and its square
is integrable against $y^\delta\,dy$. For $B_k$, integrate by parts
on the two monotonicity intervals to get
$B_k\le 2t_k^{p-1} H_k(t_k)+|p-1|A_k=O_p(k)$.

Take $\max(0,p-1)<q<p$ and set $\theta=q+1-p\in(0,1)$. Then
$D^\theta H_k(t):=\int s^\theta e^{-st}\sigma_k(ds)
=\int s^{q-1}e^{-st}a_k(ds)$. Apply
$s^\theta=c_\theta\int_0^\infty(1-e^{-sy})y^{-1-\theta}\,dy$
and put $y=ut$. With $\|v\|_{1,\delta}=\int_0^\infty t^\delta|v(t)|dt$,
\[
 \int_0^\infty t^{q-1}|D^\theta H_k(t)|\,dt
 \le c_\theta\int_0^\infty
 \|H_k(\cdot)-H_k((1+u)\cdot)\|_{1,\delta}\,u^{-1-\theta}\,du
 \le C_{p,q}(A_k^{1-\theta}B_k^\theta+A_k).
\]
For $0<u<1$ the numerator is at most $C_p\min(A_k,uB_k)$ by
integrating along dilations. For $u\ge1$ it is at most
$[1+(1+u)^{1-p}]A_k$, integrable against $u^{-1-\theta}du$ since
$q>0$. The last bound is summable in $k$ since $n=2^{k^2}$.
The absolute order-$q$ drift test follows. The signed output drift
equals $\gamma=-\sum_k\int s^{-1}a_k(ds)$: integration termwise is
justified by that summability, and the gamma integral applies to each
finite atomic $a_k$. This defines the same location for every such $q$.
Nesting gives membership at all $0<q<p$.

At the endpoint, $H=\sum H_k$ is nonnegative and $H'=\sum H'_k$
locally uniformly since $H_k(0)=0$ and
$\sum_k\int s\,|\sigma_k|(ds)<\infty$. For the unique possible
centered input the truncated output drift times $\Gamma(p)$ is
\[
 \int_0^T t^{p-1}H'(t)\,dt
   =T^{p-1}H(T)-(p-1)\int_0^T t^{p-2}H(t)\,dt .
\]
The term at zero vanishes since $H(t)=O(t)$. The integral on the right
is finite even at infinity since $\sum A_k<\infty$, whereas
$t_k^{p-1}H(t_k)\ge k(\log2)^{p-1}$. Thus the law is outside even
the ordinary order-$p$ range.
\end{proof}

\subsection{The barred fractional kernels below order one}
For $p>0$ and $\alpha<2$, let $T=\bar\Phi_{p,\alpha}$ have scale measure
\[
 \sigma_{p,\alpha}(du)=\Gamma(p)^{-1}(1-u)^{p-1}u^{-1-\alpha}
 \,\mathbf1_{(0,1)}(u)\,du .
\]
Write $R_\infty(T)=\bigcap_{n\ge1}R(T^n)$ for ordinary improper iterates,
and $L_\infty^I$ for the completely selfdecomposable laws whose stable
mixing indices lie in $I$, with arbitrary Gaussian covariance and location.

\begin{theorem}[Barred fractional limit ranges]\label{ap:barred-range}
For $0<p<1$ and every finite dimension,
\[
 R_\infty(\bar\Phi_{p,\alpha})=
 \begin{cases}
 L_\infty,&\alpha\le0,\\
 L_\infty^{(\alpha,2)},&0<\alpha<1,\\
 L_\infty^{(1,2)}\cap\{\mu:\mu\text{ has weak mean zero}\},&\alpha=1,\\
 L_\infty^{(\alpha,2)}\cap\{\mu:\mu\text{ has mean zero}\},&1<\alpha<2.
 \end{cases}
\]
These are the same classes as for $p\ge1$ in \cite[Theorem~1.1]{SatoRanges}.
The interval $0<p<1$ is listed in \cite[Section~5]{SatoRanges}.
For $\alpha<0$ the result also follows from Theorem~\ref{avg:thm:scale-kernel},
since the scale measure is finite and its support contains $(0,1)$.
\end{theorem}
\begin{proof}
In the contraction coordinate $t=-\log u$ the scale measure becomes
$e^{\alpha t}g(t)\,dt$, where $g(t)=(1-e^{-t})^{p-1}/\Gamma(p)$.
Concavity of $1-e^{-t}$ and $p<1$ give $g(\lambda t)\ge
\lambda^{p-1}g(t)$ for $\lambda\ge1$. Taking $N=\lceil1/p\rceil$ and
rescaling the convolution simplex shows that $h=g^{*N}$ satisfies
$h(\lambda t)\ge\lambda^{Np-1}h(t)$. Thus $h$ is nondecreasing.
It is continuous for $t>0$ by dominated convergence on the simplex,
using $g(t)=O(t^{p-1})$ near zero. Hence $h(t)=\eta([0,t])$ for a
positive locally finite measure $\eta$, including the mass $h(0+)$ at zero.
With $\ell$ denoting Lebesgue measure on $[0,\infty)$ this gives
\[
 (e^{\alpha t}g(t)\,dt)^{*N}=(e^{\alpha t}\ell)*(e^{\alpha t}\eta).
\]
On jump measures the first factor acts as $U$, the jump transform of
$\Lambda_{1,\alpha}=\bar\Phi_{1,\alpha}$. By Tonelli, a jump measure
with an $Nk$-step $T$-predecessor thus belongs to $R(U^k)$. All
intermediate measures are L\'evy measures: restricting each remaining
outer factor to a compact interval of positive mass bounds the
intermediate $1\wedge|x|^2$ integral by a constant times the output one.
The logarithmic jump orders add, their ranges decrease, and their
common essential range has precisely the indices
$I=(\max(0,\alpha),2)$
\cite[Theorem~7.1, Corollary~7.2 and Theorem~7.11]{SatoFractional2010}.
Consequently membership in $R(U^k)$ for all $k$ forces the asserted
spectral support. The location conditions for every ordinary one-step
output are exactly those in \cite[Theorem~4.18]{SatoFractional2010}.

Conversely, at index $\beta\in I$ one step multiplies the stable mixing
measure by $m(\beta)=\Gamma(\beta-\alpha)/\Gamma(\beta-\alpha+p)$.
Its reciprocal is bounded on $I$. Dividing the mixing measure by
$m(\beta)$ and the covariance by $m(2)$ gives an admissible jump and
Gaussian predecessor. By the essential and ordinary range criteria
\cite[Theorems~4.15 and~4.18]{SatoFractional2010}, the stated target
location gives an ordinary law predecessor with these data, the jump
predecessor being unique by \cite[Theorem~4.9(ii)]{SatoFractional2010}.
For $\alpha<1$ its location is unrestricted. For $1\le\alpha<2$ the
ordinary domain requires a finite first moment and mean zero
\cite[Theorem~4.2]{SatoFractional2010}. Thus in every case the
predecessor again belongs to the displayed class, and induction gives
an ordinary history of every depth.
\end{proof}

\paragraph{The resonant tempered kernels.}
Let $1\le\alpha<2$, $\beta>0$, and let $\Psi_{\alpha,\beta}$ denote the
ordinary stochastic integral mapping with scale measure
$h(u)\,du=u^{-1-\alpha}e^{-u^\beta}\,du$ on $(0,\infty)$.
Maejima and Ueda denote its integrand by $G^*_{\alpha,\beta}$.
Their limit-range theorem excluded $\alpha=1+n\beta$, $n=0,1,\ldots$
\cite[Theorems~2.4 and~2.6]{MaejimaUedaComposition}. Sato lists those parameters in
\cite[Section~5]{SatoRanges} and gives the critical case $\alpha=\beta=1$
in Theorem~1.1 there. The later one-step range formulas distinguish mean
zero from weak mean zero \cite[Theorem~4.6]{SatoInversions}
\cite[Theorem~3.1]{SatoUedaWeak}.

\begin{theorem}\label{avg:thm:resonant-tempering}
For every $1\le\alpha<2$, $\beta>0$, the intersection of the ordinary
iterated ranges of $\Psi_{\alpha,\beta}$ consists of the laws in
$L_\infty^{(\alpha,2)}$ having mean zero when $\alpha>1$ and weak mean
zero when $\alpha=1$. The Gaussian covariance is arbitrary. Here the
superscript means that the stable indices of the jump measure lie in
$(\alpha,2)$.
\end{theorem}
\begin{proof}
On jump measures a scale $u$ acts by $x\mapsto ux$. The scale calculation
in \cite[(3.1)]{MaejimaUedaComposition} gives, under multiplicative convolution,
\[
 h(u)\,du=
 \bigl(\mathbf1_{(0,1)}(u)u^{-1-\alpha}\,du\bigr)
 *_\times\bigl(\beta u^{\beta-\alpha-1}e^{-u^\beta}\,du\bigr),
\]
since $\int_r^\infty\beta u^{\beta-1}e^{-u^\beta}du=e^{-r^\beta}$.
Write $U,V$ for the two jump transforms. A depth-$n$ output jump measure
is $U^nV^n\eta$ by Tonelli. The measure $V^n\eta$ is a L\'evy measure:
the scale measure of $U^n$ has positive finite mass on $[1/2,3/4]$, where
$1\wedge|ux|^2\ge(1\wedge|x|^2)/4$. Thus the output belongs to every
jump range of $U^n$. These are Sato's logarithmic kernels of integer
order $n$, whose intersection gives precisely the mixture
\cite[Theorems~7.1 and~7.11]{SatoFractional2010}
\[
 \nu(B)=\int_{(\alpha,2)\times S^{d-1}}M(dq,d\xi)
              \int_0^\infty\mathbf1_B(r\xi)r^{-q-1}\,dr,
 \qquad \int (2-q)^{-1}\,dM<\infty .
\]
The output centering is necessary already in the first ordinary range,
by the one-step formulas cited above.

Conversely, for a target with this mixture, covariance $A$, and the stated
centering, put $m(q)=\int_0^\infty u^qh(u)\,du=
\beta^{-1}\Gamma((q-\alpha)/\beta)$. For $k\ge1$ take the law $\rho_k$
with mixing measure $m(q)^{-k}M$, covariance $m(2)^{-k}A$, and mean zero.
These are L\'evy triplets since $1/m$ is bounded on $(\alpha,2)$.
The large-jump first moment is $\int(q-1)^{-1}m(q)^{-k}\,dM<\infty$,
also at $\alpha=1$ since $(q-1)m(q)\to1$ as $q\downarrow1$ there.
Put $\rho_0$ equal to the target. One application to the covariance and
jump measure of $\rho_k$ gives those of $\rho_{k-1}$ by homogeneity.
Thus $\rho_{k-1}$ lies in the essential one-step range and its centering
puts it in the ordinary range. Its ordinary preimage has mean zero by
the domain condition \cite[Theorems~4.4 and~4.6]{SatoInversions}.
It must be $\rho_k$: the covariance is fixed by $m(2)$, and the jump
transform is injective. Indeed, on each angular Borel set $C$, the output
radial density times $r^{1+\alpha}$ is the Laplace transform at $r^\beta$
of the pushforward of $\mathbf1_{\{x/|x|\in C\}}|x|^\alpha\eta(dx)$ by
$x\mapsto|x|^{-\beta}$, which determines the input jump measure $\eta$.
This constructs ordinary predecessors at every depth.
\end{proof}

\subsection{Comparison of barred and logarithmic fractional kernels}
Sato's Theorem~6.16 gives the integer-order essential inclusion and
the noncritical ordinary and absolute inclusions at those orders
\cite{SatoFractional2010}. Comparing the two scale densities at different
indices also determines the exact index shift required at each order.

\begin{theorem}[Comparison of the two fractional ranges]\label{r80:comparison}
For $p>0$ and $a,b<2$, the ordinary, absolute and essential ranges each
satisfy the following equivalences with the same convention on both sides:
\begin{align*}
 K_{p,a}\subseteq L_{p,b}&\quad\Longleftrightarrow\quad
 a-b\ge\max\{0,(p-1)/2\},\\
 L_{p,a}\subseteq K_{p,b}&\quad\Longleftrightarrow\quad
 a-b\ge\max\{0,(1-p)/2\}.
\end{align*}
All inclusions are strict except at $p=1$ and $a=b$. At equal indices
the barred range is strictly smaller when $p<1$ and strictly larger
when $p>1$, answering both questions of Sato's Remark~6.17.
\end{theorem}
\begin{proof}
In $t=-\log u$, the untilted logarithmic and barred densities are
$a_p(t)=t^{p-1}/\Gamma(p)$ and $b_p(t)=(1-e^{-t})^{p-1}/\Gamma(p)$
on $t>0$. Their Laplace transforms are $A(z)=z^{-p}$ and
$B(z)=\Gamma(z)/\Gamma(z+p)$. Write $d=a-b$. For $d\ge0$ the
respective ratios of the source kernel to the target kernel, after
removing the common tilt $e^{at}$, are
\[
 Q_K(z)=(z+d)^p B(z),\qquad Q_L(z)=A(z)/B(z+d).
\]
Set $r_p(t)=(1-e^{-pt})/(1-e^{-t})$. The digamma integral represents
$-(\log Q_K)'$ and $-(\log Q_L)'$ as the Laplace transforms of
$r_p(t)-pe^{-dt}$ and $p-e^{-dt}r_p(t)$, respectively.
The scalar logarithmic complete monotonicity conditions for the shifted
gamma ratios are due to Qi and Guo \cite{r81:QG}; the equal-index case is
also \cite[Theorem~1.3]{r80:Qigamma}.
For $p\ge1$, $r_p(t)\le p$ by convexity, while
\[
 r_p(t)=e^{-(p-1)t/2}\frac{\sinh(pt/2)}{\sinh(t/2)}
        \ge pe^{-(p-1)t/2}.
\]
For $p\le1$ both inequalities reverse. The displayed index conditions
therefore make the respective Laplace densities nonnegative. If $D$
denotes the corresponding density, $Q(+\infty)=1$ yields
$\log Q(z)=\int_0^\infty e^{-zt}D(t)\,dt/t$.
The sum of the convolution powers of $D(t)\,dt/t$ divided by their
factorials, including $\delta_0$, is a positive measure $m$ with
Laplace transform $Q$. Restoring the tilt factors the source scale
measure as the convolution of the target scale measure with $e^{at}m(dt)$.
Consequently every source jump measure is a target output: apply the
positive scale mixture $e^{at}m(dt)$ to its predecessor and use Tonelli.
The intermediate measure is a L\'evy measure since the outer mass on
a fixed compact subinterval of $(0,1)$ bounds its L\'evy norm by a
constant times the output norm. Covariances factor identically at
Laplace argument $2-a$.

To see necessity even for the ordinary or absolute inclusion, apply the
source kernel to equally weighted jumps at $\pm1$ on a line and zero
drift. The output is absolutely defined. Its jumps have radii at most
one; so must a target predecessor, since the target kernel has mass at
scales arbitrarily near one. In log radius the Mellin transform of this
predecessor must be a positive constant times the source/target ratio,
with variable $s=z+a>\max(a,b)$. If $a<b$, the source transform stays
finite as $s\downarrow b$ while the target transform diverges, which
would make a nonzero Laplace transform of a positive measure start at
zero. Hence $d\ge0$. As $z\to\infty$, the standard gamma ratio expansion
gives
\[
 Q_K(z)=1+\frac{p[d-(p-1)/2]}{z}+O(z^{-2}),\qquad
 Q_L(z)=1+\frac{p[d+(p-1)/2]}{z}+O(z^{-2}).
\]
A positive Laplace transform decreasing to one cannot be smaller than
one. This gives the two thresholds. Failure of the reverse threshold
gives strictness, with the same symmetric test law.

Off the target critical index $b=1$, ordinary and absolute ranges
coincide for either kernel. They agree with the essential range for $b<1$
and require mean zero for $b>1$ \cite[Theorems~4.18 and~6.12]{SatoFractional2010}.
In the latter case $a\ge b>1$, so the source condition gives mean zero.
If $a>b=1$ the source output jumps have a finite large-jump moment
of any order $1<s<a$. Indeed for an input radius $R>1$ the large-jump
mass and its $s$th moment are, respectively,
\[
 I_0(R)=\int_0^{\log R} e^{at}k(t)\,dt,\qquad
 I_s(R)=R^s\int_0^{\log R}e^{(a-s)t}k(t)\,dt,
\]
where $k=a_p$ or $b_p$. Their ratio is bounded: near $R=1$ use
$I_s(R)\le R^s I_0(R)$, and at infinity each integral is asymptotic
to its upper-end exponential times $k(\log R)$ divided by the
exponent, by l'H\^opital's rule and $k'(t)/k(t)\to0$.
Outer scales bounded away from zero pass this moment bound to the
intermediate predecessor. Center its drift. Its vector tail bracket
then satisfies $|H(t)|\le C e^{-(s-1)t}$. Both critical scale densities
are integrable against this bound, so the target integral is absolutely
defined. Its output is centered: the compensated exponent with
subtraction of the whole linear term integrates by Fubini, using the
output first moment on large jumps and second moment on small jumps.
Thus it gives precisely the specified centered source law.

The remaining case has equal critical indices $a=b=\alpha=1$.
Write $F=A/B$; then $a_p\,dt=(b_p\,dt)*m$ for $p>1$ and
$b_p\,dt=(a_p\,dt)*m$ for $0<p<1$, with the untilted measures from
the corresponding factorization above. At $p=1$ the kernels coincide.
At $\alpha=1$
and $p>1$, Proposition~\ref{ap:prop:ordinary-factor} and
\cite[Corollary~7.6]{SatoFractional2010} put the ordinary and
absolute logarithmic outputs in their respective order-one ranges.
They therefore have weak mean zero, respectively absolutely weak
mean zero. Together with membership in $K^e_{p,1}$ these are the
barred range tests \cite[Theorems~4.18 and~4.21]{SatoFractional2010}.

It remains to carry the drift through $b_p\,dt=(a_p\,dt)*m$ when
$\alpha=1$ and $0<p<1$. Put $c=1/\Gamma(p)$ and $d=b_p-c$.
Then
\[
                  m=c\,a_{1-p}(t)\,dt+h
\]
for a finite signed measure $h$. Indeed the cumulative function of
$m$ is $a_{1-p}*b_p$ by Laplace transforms. The function
$a_{1-p}*d$ decreases from one to zero: after rescaling its
integral to $(0,1)$ it is a positive beta-weighted average of
$s^{1-p}d(s)$ evaluated at $s=tv$. The latter function decreases
from $c$ to zero. On differentiating, its decrease follows from
\[
 (1-e^{-s})[1-(1-e^{-s})^{1-p}]\le (1-e^{-s})e^{-s}\le s e^{-s}.
\]
The beta integral gives $(a_{1-p}*d)(0+)=1$. Extending this
function by zero to the negative half-line, its distributional
derivative is the required finite $h$.

Take an ordinary barred input, with jump measure $\eta$ and drift
$b$. It has finite first moment and mean zero
\cite[Theorem~4.2]{SatoFractional2010}. Its bounded drift bracket
$H(s)=b+\int_{1<|x|\le e^s}x\,\eta(dx)$ has convergent primitive:
$\int b_pH$ converges and $b_p-c$ is integrable. The intermediate
triplet for the scale mixture $e^s m(ds)$ has drift
$\int H(s)m(ds)$, which converges by the displayed decomposition
and Dirichlet's test. Its bracket at $v$ is $\int H(s+v)m(ds)$,
by adding its finite first moment on the annulus $1<|x|\le e^v$
as in Theorem~\ref{r61:satocore}.

The contribution of $c\,a_{1-p}\,dt$ to the iterated drift is
$c\int H$: the direct order-one logarithmic integral exists by
the bounded first moment and the convergent primitive, and
Proposition~\ref{ap:prop:ordinary-factor} applies. For $h$, the
integrals $\int_0^L a_p(v)H(s+v)\,dv$ converge as $L\to\infty$ and
are uniformly bounded in $s,L\ge0$. Use boundedness of $H$ below
$v=1$, and integration by parts with the bounded primitive of $H$
and decreasing $a_p$ above it. Dominated convergence against $|h|$
shows that the truncations $v\le L$ and $s+v\le L$ have the same
limit. The full iterated drift is consequently $\int b_p H$, the
specified output drift. For absolute inputs Tonelli directly gives
$\int a_p(v)\int |H(s+v)|\,m(ds)\,dv=\int b_p(t)|H(t)|\,dt<\infty$.
This proves both inclusions at the critical index.

\end{proof}

\subsection{Multipliers of fractional monotonicity}
For $p>0$ set $k_p(x)=x^{p-1}/\Gamma(p)$ for $x>0$ and zero otherwise.
For a positive measure $\sigma$ on $(0,\infty)$ write
$I_p\sigma(r)=\int_{(r,\infty)} k_p(s-r)\,\sigma(ds)$.
Admissibility means local finiteness and
$\int_{[1,\infty)}s^{p-1}\sigma(ds)<\infty$.
Write $\mathcal M_p$ for these densities, with equality almost everywhere.
When acting on measures, $I_p\sigma$ denotes the measure with this
density, and $I_0$ is the identity. The beta integral gives $I_pI_q=I_{p+q}$.

\begin{theorem}[Multipliers at all fractional orders]\label{ap:cm-multiplier}
A finite locally bounded measurable function $g$ on $(0,\infty)$ maps
$\mathcal M_p$ into itself by multiplication for every $p>0$ if and
only if it agrees almost everywhere with a completely monotone function.
In particular $r^{-a}$ preserves each fractional order for every $a>0$,
answering Sato's question in \cite[Remark~5.20]{SatoFractional2010}.
\end{theorem}
\begin{proof}
A completely monotone $g$ is a positive mixture of exponentials. For
$\lambda\ge0$ and $s>r$, expand $e^{-\lambda r}$ at the right endpoint:
\[
 e^{-\lambda r}k_p(s-r)=e^{-\lambda s}
 \sum_{j\ge0}\frac{\lambda^j\Gamma(p+j)}{j!\Gamma(p)}k_{p+j}(s-r).
\]
The beta identity and Tonelli realize the product with $I_p\sigma$ as
$I_p$ of a positive measure, and then allow mixing in $\lambda$.
This measure $\tau$ is admissible: for each $u>0$, integrating the
locally integrable product over $(u/4,u/2)$ bounds
$\int_{[u,\infty)}s^{p-1}\tau(ds)$ by a finite constant.
Conversely $e^{-\delta r}=I_p(\delta^pe^{-\delta s}\,ds)(r)$ lies in
every order for every $\delta>0$. Thus $g(r)e^{-\delta r}$ has an
$I_n$ representation for every positive integer $n$. Taking $n>j+1$
and differentiating its integral $j$ times gives a continuous
representative with nonnegative signed derivatives $(-1)^jD^j$.
These representatives agree by continuity. Using $\delta=1$ first to
fix a smooth version of $g$, and then $\delta\downarrow0$, proves the
derivative inequalities for $g$. Finally $r^{-a}$ is completely monotone.
For this multiplier the positive kernel expansion and Tonelli hold at
each $r>0$, giving the product of Sato's specified representatives
also when a fractional density takes the value $+\infty$.
\end{proof}
For $\alpha<\alpha'<2$, this yields the strict inclusions
$K_{p,\alpha}\supsetneq K_{p,\alpha'}$ at every $p>0$, also for
the absolute and essential ranges $K^0,K^e$ in Sato's notation.
The radial kernel for an output at $\alpha'$ is multiplied by
$r^{-(\alpha'-\alpha)}$ when written at $\alpha$.
The corresponding location conditions follow from
\cite[Theorems~4.18 and~4.21]{SatoFractional2010}; in particular finite
mean zero implies absolute weak mean zero. For strictness take the
symmetric one-dimensional jump measure with radial density
$r^{-1-\alpha}k_p(1-r)$, zero Gaussian part and zero location.
At $\alpha'$ its putative kernel $r^{\alpha'-\alpha}k_p(1-r)$ tends
to zero at the origin. A nonzero $I_p$ of a positive measure is
bounded below near the origin by restricting its measure to a compact
interval of positive mass.

\Needspace{6\baselineskip}
\begin{corollary}[Strict nesting at fixed fractional order]\label{ap:cor:log-alpha-order}
For every $p>0$ and $-\infty<\alpha<\alpha'<2$,
\[
 L_{p,\alpha}\supsetneq L_{p,\alpha'},\qquad
 L^0_{p,\alpha}\supsetneq L^0_{p,\alpha'},\qquad
 L^e_{p,\alpha}\supsetneq L^e_{p,\alpha'} .
\]
This gives the affirmative answer to \cite[Remark~7.20]{SatoFractional2010}.
\end{corollary}
\begin{proof}
In log radius $y$, the jump criterion for the essential range at $\alpha'$
is a radial density $r^{-1-\alpha'}h(\log r)$ on each direction, with
$h=I_p\sigma$ on $\mathbb R$
\cite[Theorem~6.9]{SatoFractional2010}. Set $a=\alpha'-\alpha$.
The exponential expansion in Theorem~\ref{ap:cm-multiplier} is valid
for all $y<s$ on the whole line. Hence $e^{-ay}I_p\sigma(y)$ is again
$I_p$ of a positive measure, by Tonelli and the beta identity.
The measure is admissible since the product is locally integrable:
integrating it on $[b-2,b-1]$ bounds the representing measure on
$[b,\infty)$ with weight $(1+s-b)^{p-1}$, for each $b\in\mathbb R$.
This proves essential inclusion. For $\alpha<1$ the two other ranges
equal the essential one. For $\alpha>1$ they add the same finite mean-zero
condition \cite[Proposition~6.8 and Theorem~6.12]{SatoFractional2010}.

If $\alpha=1$, a law $\mu$ in the ordinary range at $\alpha'>1$ has
finite first moment and mean zero. Let $\eta$ be a jump predecessor
at index one and let $m(du)=k_p(-\log u)u^{-2}du$ on $(0,1)$.
By Tonelli,
\[
 \int m(du)\int_{u|x|>1}u|x|\,\eta(dx)
   =\int_{|z|>1}|z|\,\nu_\mu(dz)<\infty .
\]
Restricting to $1/2<u<3/4$ gives a finite large-jump first moment for
$\eta$. Choose the input centered and match its Gaussian covariance
to the output. Its drift integrand is $-\int_{u|x|>1}ux\,\eta(dx)$,
absolutely integrable by the display. Its integral equals
$-\int_{|z|>1}z\,\nu_\mu(dz)$, the drift of $\mu$. Thus
$\mu\in L^0_{p,1}$ even for $p<1$.

For strictness use the order-$p$, index-$\alpha$ output of a symmetric
unit-jump compound Poisson input. It exists absolutely since the scale
has finite second moment, and its radial factor is $k_p(-y)$.
At $\alpha'$ the required factor is $e^{ay}k_p(-y)$.
Any nonzero $I_p\tau(y)$ on $\mathbb R$ is at least
$c|y|^{p-1}$ for all sufficiently negative $y$, by restricting $\tau$
to a bounded interval of positive mass. The proposed factor is
$o(|y|^{p-1})$, so even essential membership at $\alpha'$ fails.
\end{proof}

\subsection{The fractional monotonicity of a hyperbolic power}

For $p>0$, a function on $\mathbb R$ is monotone of order $p$ in
Sato's sense if it is the density of the right fractional integral
$I_p\sigma$ of a positive measure. Its density is
$\Gamma(p)^{-1}\int_{(r,\infty)}(s-r)^{p-1}\sigma(ds)$.

\begin{proposition}[The order of a hyperbolic power]\label{newsato:power}
For $a>0$ the largest order of monotonicity of
$h_a(r)=(\sqrt{1+r^2}-r)^a$ on $\mathbb R$ is $a+1$. Its endpoint
representation is
\[
 h_a(r)=\frac{2^a}{B(1/2,a)}
   \int_r^\infty (s-r)^a(1+s^2)^{-a-1/2}\,ds .
\]
Here $B$ denotes Euler's beta function. Sato excluded all orders above
$a+1$ in Example~2.19(a). The integral above attains the endpoint
in his Remark~2.20 \cite{SatoFractional2010}.
\end{proposition}
\begin{proof}
Put $A=\sqrt{1+r^2}$ and $b=r/A$. Substituting $s=r+Ax^2$ in the
integral gives
\[
 2A^{-a}\int_0^\infty
   \left(\frac{x^2}{x^4+2b x^2+1}\right)^{a+1/2}\,dx
   =A^{-a}[2(1+b)]^{-a}B(1/2,a)
\]
by the Cauchy--Schl\"omilch integral \cite[Theorem~4.1]{newsato:CS}.
Since $(A+r)^{-1}=A-r$, this proves the formula and the measure in it is
positive. Sato's upper bound now gives the asserted maximum.
\end{proof}

\section{Probability laws from their convolution and moment data}\label{r55:law}
\subsection{The same regularization curve with opposite endpoint behavior}
For a symmetric Gaussian convolution semigroup $(\mu_t)$ on a compact
group with Haar probability $m$, put $t_A=\inf\{t>0:\mu_t\ll m\}$.
Saloff-Coste asks whether the density belongs to $L^2(m)$ for every
$t>t_A$, and whether a product of elliptic Gaussian semigroups on tori
has a continuous density for every $t>2t_A$
\cite[Problem~6 and the first assertion of Problem~8]{r87:saloff}.
Saloff-Coste's Example~2.2 gives affirmative answers to Problems~6 and~7
for independent circle factors. Here the independent factors have
dimensions tending to infinity, with correlated circle coordinates
within a block. Finite-dimensional lattice smoothing estimates include the total variation
bounds of Debris-Alazard et al.\ and the $L^1$ flatness estimates of
Luzzi, Ling and Bloch \cite{r87:latticesmoothing,r87:latticeflat}.
Ling, Luzzi and Yan use Rogers's formula and Construction~A to study
averaged finite-dimensional R\'enyi flatness \cite{r87:renyiflat}.
We use finite-dimensional bounds to prescribe endpoint behavior for
two semigroups on the same compact group.

\begin{theorem}[Regularization thresholds on a product of tori]\label{r87:gausstimes}
On one countable product of finite-dimensional flat tori there are two
symmetric Gaussian convolution semigroups $\mu^-,\mu^+$, elliptic on
each factor. Put
\[
 a_1=(2\pi e)^{-1},\qquad
 a_p=(2\pi p^{1/(p-1)})^{-1}\ (1<p<\infty),\qquad
 a_\infty=(2\pi)^{-1}.
\]
For every $1\le p\le\infty$, $\mu^-_t$ has an $L^p$ density exactly
when $t>a_p$, and $\mu^+_t$ has one exactly when $t\ge a_p$.
In each case the law is singular whenever it lacks an $L^1$ density
and its density is continuous whenever it is bounded.
\end{theorem}
\begin{proof}
Let $g_s$ be the $N(0,sI_n)$ density and let $F_{\Lambda,s}$ be its
periodization on the quotient of $\mathbb R^n$ by a covolume-one
lattice $\Lambda$. With $F=F_{\Lambda,s}$ and Haar probability on
the quotient, unfolding gives
\[
 \|F\|_p^p=\int_{\mathbb R^n}g_s(x)
 \left(g_s(x)+\sum_{v\in\Lambda\setminus\{0\}}g_s(x+v)\right)^{p-1}dx
 \ \ge\ \int g_s^p=(a_p/s)^{n(p-1)/2}.
\]
For a random unimodular lattice in dimension $n\ge2$, Siegel's
mean value theorem \cite{r87:Siegel} gives expectation one for the
sum over nonzero vectors at each fixed $x$. For $1<p\le2$, concavity
and subadditivity of the power $p-1$, followed by integration, give
\[
 0\le\mathbb E(\|F\|_p^p-1)\le (a_p/s)^{n(p-1)/2}.
\]

For $p>2$ let $q=\lceil p-1\rceil$ and $S(x)=\sum_{v\ne0}g_s(x+v)$.
We claim that $\mathbb E\int g_s S^q\le 1+O(e^{-cn})$ for $s>a_q$,
with $q,s$ fixed and $n>q$ sufficiently large. Rogers's independent
tuple mean formula \cite[(8)]{r87:Rogers} integrates a sum over
independent lattice $r$-tuples against ordinary Lebesgue measure on
$(\mathbb R^n)^r$. In the expansion of $S^q$, the independent
$q$-tuples therefore contribute exactly one after integration.

For tuples of rank $1\le r<q$, express their entries in a basis of
the lattice in their real span. The resulting rank-$r$ integer matrix
$A$ of size $q\times r$ is taken modulo right $\mathrm{GL}_r(\mathbb Z)$.
For a fixed representative the basis is unique for each tuple, so
enlarging the sum over these bases to all independent $r$-tuples gives
an upper bound. Applying (8) and integrating the resulting $q+1$
Gaussians gives
\[
 (2\pi s)^{-n(q-r)/2}\Delta_A^{-n/2},\qquad
 \Delta_A=\det(B^TB),\quad B=\begin{pmatrix}1&0\\ \mathbf1&A\end{pmatrix}.
\]
Fix $r+1$ independent rows of $B$. Each remaining row has first
coordinate one and can replace some basis row, giving a different
nonzero integral minor. Cauchy--Binet yields $\Delta_A\ge q-r+1$.
There are polynomially many classes with $\Delta_A\le H$ for fixed $q$.
Indeed all maximal minors of $A$ have modulus at most $\sqrt H$.
Put an invertible row block of $A$ in column Hermite normal form,
whose entries are bounded by its determinant in absolute value.
Cramer's rule then bounds every other entry by $rH$.
Thus $\sum_A\Delta_A^{-n/2}=O_q((q-r+1)^{-n/2})$, by reserving
a fixed power for summability. Since
$(q-r+1)^{1/(q-r)}\ge q^{1/(q-1)}$, each deficient rank contributes
exponentially little for $s>a_q$, proving the claim.

Under the probability measure $g_s(x)\,dx\,d\mathbb P(\Lambda)$,
Minkowski and monotonicity of norms give
\[
 (\mathbb E\|F\|_p^p)^{1/(p-1)}
   \le (a_p/s)^{n/2}+\big(\mathbb E\textstyle\int g_s S^q\big)^{1/q}.
\]
As $q\le p$, the excess $\mathbb E\|F\|_p^p-1$ decays exponentially
for each $s>a_p$.

For each $j\ge2$ put $h_j=a_1/(4j)$. By uniform continuity select a
finite mesh $R_j\subset(1,j]$ such that each $p\in[1,j]$ admits
$r\in R_j$ with $r\ge p$ and $a_r\le a_p+h_j$.
At each of the finitely many pairs $(r,s)=(r,a_r+h_j)$ the bounds
above also give $\mathbb E\|F_{\Lambda,s}-1\|_r^r\to0$ as
$n\to\infty$. Indeed, for fixed $r>1$ and $\delta>0$ the ratio
$|x-1|^r/(x^r-1-r(x-1))$ is bounded on
$\{x\ge0:|x-1|\ge\delta\}$. On its complement the numerator is
at most $\delta^r$, and $\int(F_{\Lambda,s}-1)=0$.
Markov's inequality and a union bound therefore give one lattice
$\Lambda_j$ of dimension $n_j\ge4^j$ with
$\|F_{\Lambda_j,a_r+h_j}-1\|_r\le2^{-j}$ for every $r\in R_j$.

On the same product of quotient tori take independent Brownian motions
with time multipliers $c_j^\pm=1\pm 1/j$, respectively. Continuous paths
in the product topology and the Gaussian small-time bound on every
finite cylinder give symmetric Gaussian semigroups.
For fixed $1\le p<\infty$, eventually $c_j^-t\ge a_p+2h_j$ if $t>a_p$,
and $c_j^+t\ge a_p+2h_j$ already at $t=a_p$ and above. Convolution
contraction at the selected $r\ge p$ bounds the $L^p$ distance of
the factor density to one by $2^{-j}$. The finite products converge
in $L^p$, including $p=1$: telescoping bounds the norm of a tail
difference by a constant times $\sum_{j>N}2^{-j}$ since their norms
multiply and remain bounded. Cylinder probabilities identify the limit.

For $1<p<\infty$ the lower bound on the $j$th factor diverges for
$c_j^-t$ at $t\le a_p$ and for $c_j^+t$ at $t<a_p$. At the negative
choice's endpoint use $(1-1/j)^{-n_j(p-1)/2}\to\infty$.
Taking a coordinate marginal contracts $L^p$, excluding a product
density there. In the corresponding excluded range for $p=1$ one has
$c_j^\pm t\le a_1$ eventually. Project the ball of radius $\sqrt{n_ja_1}$
onto the $j$th torus. Its Haar mass is $O(n_j^{-1/2})$ by the ball
volume formula and Stirling, while its marginal Gaussian mass is at
least $\mathbb P(\chi^2_{n_j}\le n_j)\to1/2$. The limsup has Haar
mass zero and product Gaussian mass one by Borel--Cantelli.
For $p=\infty$ the factor density at zero is at least
$(2\pi c_j^\pm t)^{-n_j/2}$, diverging in the excluded range. It is
continuous and Haar measure has full support, so this bounds the
essential supremum. In the admitted range the density at $t/2$ is in
$L^2$ since $a_\infty=2a_2$; convolving it with itself gives a
continuous density.
\end{proof}
Here $a_2/a_1=a_\infty/(2a_1)=e/2>1$, giving negative answers to
both assertions for each semigroup.

\subsection{Gaussian divisors and convolution partners}
Call a probability law $F$ on $\R$ strongly non-Gaussian if for every
probability law $G$, $F*G$ has a Gaussian convolution factor of positive
variance exactly when $G$ has one. Il'inskii records this question of
A.~M.~Kagan, proves that bounded non-point-mass laws fail the condition,
and discusses unbounded supports \cite[\S2, Theorem~2.1 and
Remarks~2.2--2.3]{IlinskiiGaussian}.
\begin{theorem}[Gaussian divisors under convolution]\label{thm:kagan}
A probability law on $\R$ is strongly non-Gaussian if and only if it is a point mass.
\end{theorem}
\begin{proof}
Translations preserve Gaussian factors. Fix a non-point-mass $F$.
Write $\gamma_v$ for the centered Gaussian density of variance $v$,
$\ell(x)=e^{-|x|}/2$, $H=F*\ell$, and $h=H*\gamma_1$.
Set $q(x)=e^{-x^2/2}h(x)$ and $g=h-q$. The probability density
$g/\int g$ has no Gaussian factor: $g$ is continuous, vanishes at zero,
and convolution with a nondegenerate Gaussian has an everywhere positive
continuous density.

We claim that $F*q\le(1-c)F*h$ for some $c>0$. Choose two intervals
inside $[-R,R]$ of positive $F$-mass at distance greater than $2r>0$.
For each $x$, at least one is disjoint from $(x-r,x+r)$. On this interval
$h(x-y)\ge e^{-R}h(x)$ by the translation bound for $\ell$, and
$1-e^{-(x-y)^2/2}\ge1-e^{-r^2/2}$. Integration gives $F*g\ge c_0h$
uniformly. Also $F*q\le\|h\|_\infty F*e^{-(\cdot)^2/2}\le C h$,
since $\ell*\gamma_1(x)\ge c_1e^{-|x|}\ge c_2e^{-x^2/2}$.
These bounds and $h=g+q$ give the claim.

Completing the square gives
$e^{-x^2/2}\gamma_1(x-z)=2^{-1/2}e^{-z^2/4}\gamma_{1/2}(x-z/2)$.
Thus $q=B*\gamma_{1/2}$ for a finite positive measure $B$. For
$0<t<1/2$ put $h_t=H*\gamma_{1-t}$ and $q_t=B*\gamma_{1/2-t}$.
Then $g=(h_t-q_t)*\gamma_t$. The translation bound for $\ell$ gives
$h=h_t*\gamma_t\le a_t h_t$, where $a_t=\int e^{|y|}\gamma_t(y)\,dy\to1$.
Comparison of the two centered Gaussian densities gives
$q_t\le b_t q$, $b_t=(1-2t)^{-1/2}\to1$. Hence
\[
 F*(h_t-q_t)\ge\big[a_t^{-1}-b_t(1-c)\big]F*h\ge0
\]
for sufficiently small $t>0$. This density has mass $\int g$; after
normalization its convolution with $\gamma_t$ has density $F*g/\int g$.
\end{proof}
\subsection{A prescribed radius for bivariate moment matrices}
Let $m_\alpha=\int x^\alpha\,d\mu(x)$ for a probability measure on
$\mathbb R^2$ with all moments. Following Roybal, put
\[
 H_{R,N}=\bigl[R^{-|\alpha+\beta|}m_{\alpha+\beta}\bigr]_
 {|\alpha|,|\beta|\le N},\qquad
 \gamma_R=\inf_{N\ge0}\lambda_{\min}(H_{R,N})\quad(R>0).
\]
These are the moment matrices of $X/R$ when $X$ has law $\mu$.
Roybal \cite[Section~3]{Roybal} asks whether positivity of $\gamma_1$ implies positivity of
every $\gamma_R$, and whether bounded evaluation at a point of
$(\mathbb C\setminus\mathbb R)^2$ implies locally uniform boundedness of
the orthonormal polynomial kernel diagonal on $\mathbb C^2$. These are
the proposed implications (3.1)$\Rightarrow$(3.2) and
(3.4)$\Rightarrow$(3.3), respectively. Roybal's Theorem~3.1 proves both
when the mixed moments factor across the given coordinates.

The second converse alone already fails by the classical one-variable
criterion \cite[Theorem~1.1]{Roybal}. Take a probability measure $\nu$
with an indeterminate Hamburger moment sequence. Its polynomial
evaluation at $i$ is bounded on $L^2(\nu)$. Independently take $T$
uniform on $[0,1]$ with probability $1/2$ and equal to zero otherwise.
For $X\sim\nu$ set $Y=X+T$. The atom of $T$ gives
$\int|P(x,x)|^2\,\nu(dx)\le2\mathbb E|P(X,Y)|^2$,
so evaluation at $(i,i)$ is bounded. Evaluation at $(i,3i)$ is
unbounded by testing $(y-x)^n$, of norm at most one. The moment
matrices are positive definite because both factors before the shear
have infinite support. The following construction also settles the
scale implication and prescribes its exact threshold.

\begin{theorem}[A prescribed critical radius]\label{roy:thm:critical}
For every $r>0$ there is a probability measure $\mu_r$ on $(0,\infty)^2$
whose finite moment matrices are positive definite and for which
\[
                 \gamma_R>0\quad\Longleftrightarrow\quad 0<R\le r .
\]
Its polynomial evaluations are locally uniformly bounded in the open
polydisc of radius $r$. Evaluation at $(ir,2ir)$ is unbounded.
In particular, $\mu_1$ gives negative answers to both implications.
\end{theorem}

\begin{proof}
Choose $v=\log 16$, $q=e^{-2v}$, $r_0=\sqrt{q/2}$, and independent
$\log S\sim N(0,v)$ and $T$ with $\mathbb P(T=q^j)=2^{-j}$,
$j=1,2,\ldots$. Start with $X=(S,ST)$, whose moments are
\[
 \mathbb E X_1^aX_2^b=e^{v(a+b)^2/2}\frac{q^b}{2-q^b}
       \qquad(a,b\in\mathbb N_0).
\]
For a polynomial $p(x,y)=\sum a_{k,l}x^ky^l$, write
$p(s,st)=\sum_n s^nh_n(t)$ and
$b_n^2=\sum_{l=0}^n|a_{n-l,l}|^2$. Then $h_n(t)=\sum_l a_{n-l,l}t^l$.
If $p(X)=0$ almost surely, conditioning on each $T=q^j$ gives a radial
polynomial vanishing on $(0,\infty)$, so $h_n(q^j)=0$ for all $n,j$.
Thus $p=0$. The Rayleigh quotient, applied after changing monomial
coefficients by $R^{k+l}$, gives
\[
 \gamma_R=\inf_{p\ne0}\frac{\|p(X)\|_2^2}{\sum_n R^{2n} b_n^2}.
\]

The Gram matrix of the normalized radial monomials $S^n/e^{vn^2}$ has
entries $e^{-v(n-m)^2/2}$. Its diagonal is one and its off-diagonal row
sums are bounded by $2\sum_{k\ge1}4^{-k}=2/3$. Applying the resulting
lower bound to the coefficients $e^{vn^2}h_n(t)$ and integrating in $t$
gives
\[
 \|p(X)\|_2^2\ge\tfrac13\sum_n e^{2vn^2}\|h_n(T)\|_2^2 .
\]
To recover the coefficients of $h_n$, interpolate at $q,\ldots,q^{n+1}$.
Every Lagrange numerator has coefficient $\ell^1$-norm at most
$\prod_{k\ge1}(1+q^k)$. Every denominator has absolute value at least
$q^{n(n+1)/2}\prod_{k\ge1}(1-q^k)^2$. Indeed,
$|q^i-q^j|=q^{\min(i,j)}(1-q^{|i-j|})$ and the sum of $\min(i,j)$
over $1\le j\le n+1$, $j\ne i$, is at most $n(n+1)/2$.
The two infinite products are finite and strictly positive, and
Cauchy--Schwarz gives
\[
 \sum_{j=1}^{n+1}|h_n(q^j)|
 \le \Big(\sum_{j=1}^{n+1}2^j\Big)^{1/2}\|h_n(T)\|_2 .
\]
Consequently $b_n\le C2^{n/2}q^{-n(n+1)/2}\|h_n(T)\|_2$ with $C$
independent of $n$. Multiplication by $r_0^n$ changes the factor on the
right to $Cq^{-n^2/2}=Ce^{vn^2}$. The Gram bound proves $\gamma_{r_0}>0$.
It also gives $\gamma_R>0$ for every $0<R\le r_0$.

Conversely, $p_n(x,y)=\prod_{j=1}^n(y-q^jx)$ has $y^n$ coefficient
one and vanishes whenever $T$ is one of the first $n$ nodes. On the
other nodes $\prod_{j=1}^n|T-q^j|\le q^{n(n+1)/2}$, and their total
probability is $2^{-n}$. Independence gives
\[
 \|p_n(X)\|_2\le e^{vn^2}2^{-n/2}q^{n(n+1)/2}=r_0^n.
\]
The Rayleigh quotient is at most $(r_0/R)^{2n}$, giving $\gamma_R=0$
for $R>r_0$. Now take $X_r=(r/r_0)X$. The scale $R$ for $X_r$ is
the scale $Rr_0/r$ for $X$, and homogeneity gives $\|p_n(X_r)\|_2\le r^n$.

For $|z_1|,|z_2|\le u<r$, coefficient Cauchy--Schwarz gives
$|p(z)|\le\gamma_r^{-1/2}(1-u^2/r^2)^{-1}\|p(X_r)\|_2$.
At $z=(ir,2ir)$, each factor of $p_n(z)$ has modulus $r(2-q^j)$, so
$|p_n(z)|/\|p_n(X_r)\|_2\ge (2-q)^n$. For any orthonormal polynomial basis
adapted to the degree filtration, the degree-$N$ kernel diagonal is the
squared norm of evaluation on that finite-dimensional space. Thus the
kernel is locally bounded on the stated polydisc, including at
$(ir/2,ir/2)$, and infinite at the exhibited point.
\end{proof}

\subsection{When the number of atoms changes the rank of a moment map}
For a polynomial map $A:\mathbb R^n\to\mathbb R^m$, write
$S_k(c,x)=\sum_{i=1}^k c_iA(x_i)$ with $c_i\ge0$. A moment vector
$s$ is a regular value at length $k$ when $DS_k$ is onto at every
representation of $s$ of that length. Di Dio and Schm\"udgen ask
whether a regular value at one length can be singular at a larger
length \cite[authors' version, Definition~63 and Problem~5]{r86:momentcone}.

\begin{theorem}[Prescribing the first singular length]\label{r86:rankthreshold}
For each $q\ge2$ there are $q+1$ polynomial observables on
$\mathbb R^q$ and an interior moment vector $s$ that is a regular
value of $S_k$ exactly for $1\le k<q$. At every length $k$ there is a
regular representation of $s$ with distinct nodes and positive weights.
For every $k\ge q$ there is also a singular representation with
distinct nodes and positive weights.
\end{theorem}
\begin{proof}
Choose the vertices $v_1,\ldots,v_q$ of a simplex with barycenter zero
in $\mathbb R^{q-1}$. For $(u,y)\in\mathbb R^{q-1}\times\mathbb R$, put
\[
 G(u)=\prod_{j=1}^q\|u-v_j\|^2,\qquad
 A(u,y)=(1,u,G(u)y),\qquad s=(1,0,0).
\]
At a node of positive weight $c$ with $G(u)\ne0$, the columns for its
weight and coordinates alone are onto: the corresponding Jacobian
determinant of $cA(u,y)$ is $c^qG(u)$. For any length one obtains such
a representation by taking distinct $u$ near zero with mean zero,
equal weights and $y=0$. The one-node representation also shows that
$s$ is interior. If all positive-weight nodes have $u$ among the
vertices, their weights give barycentric coordinates of zero, so each
vertex has total weight $1/q$. This requires at least $q$ nodes.

Conversely, distribute mass $1/q$ at each vertex among one or more
distinct $y$-coordinates to obtain any prescribed length $k\ge q$.
Since $G(v_j)=0$ and $\nabla G(v_j)=0$, the last row of $DS_k$ vanishes
at this representation. The other rows have rank $q$ by the weight
column and the $u$-derivatives at any node. Thus the first singular
length is exactly $q$.
\end{proof}

\begin{proposition}[Mutual transformations of affine moments]\label{r86:affinetrans}
On full-dimensional compact convex bodies of the same dimension, affine
moments of full-support probabilities admit injective transformations in
both directions. A strong transformation exists exactly when an affine
bijection between the bodies takes one barycenter to the other.
Consequently the unit disk and the square with uniform area measures
answer di Dio's Open Problem~5.1 in the negative
\cite[Definition~1.5 and Open Problem~5.1]{r86:transform}.
\end{proposition}
\begin{proof}
For affine observables on bodies $K,H\subset\mathbb R^d$, the pullback
condition on the coordinate functions forces every transformation
$f:K\to H$ to be affine. The moment identity is precisely
$f(b_K)=b_H$ for the two barycenters. Each barycenter is interior:
otherwise a supporting affine functional, positive on an open subset
of the body, would have zero mean. Hence for small $a>0$ the map
$x\mapsto b_H+a(x-b_K)$ embeds $K$ in $H$ and pulls the affine
observables on $H$ onto those on $K$. Interchanging the bodies gives
the reverse transformation. A strong transformation is also surjective,
so its linear part has rank $d$ and it is an affine bijection.
Conversely such a barycenter-preserving bijection is strong. An affine
bijection preserves extreme points, of which the disk has infinitely
many and the square has four.
\end{proof}

A second question in the same paper concerns encoding observations on
an unbounded set by one real parameter. In
\cite[Remark~4.14 and Open Problem~5.4]{r86:transform}
a measurable section $g:\mathbb R^n\to[0,\infty)$ of a continuous
surjection $f$ satisfies $|g(x)-\|x\||\le\varepsilon$ and $f(g(x))=x$.
The question asks whether a continuous extension to $\mathbb R[x,g]$
whose $g$-moments have a Stieltjes representation forces the original
polynomial functional to have a positive representing measure.

\begin{proposition}[A positive moment law along a single direction]
\label{r86:direction}
For every such $g$ on $\mathbb R^2$, there is a linear functional
on $\mathbb R[x_1,x_2,g]$ with $L(1)=1$, Stieltjes $g$-moments
and $L(x_2^2)<0$. It respects uniform limits on the parameter domain
as in the note to \cite[Theorem~4.4]{r86:transform}; it also respects
pointwise limits under a common polynomial growth bound.
\end{proposition}
\begin{proof}
Let $x_j$ lie on the first positive coordinate axis for even $j\ge0$
and on the second for odd $j$, at radius $R4^j$.
Choose $R>3(1+\varepsilon)$ so $t_j=g(x_j)>0$ and
$t_{j+1}\ge2t_j$. For $P_N(z)=\prod_{i=0}^N(1-z/t_i)$ the numbers
$c_{j,N}=-1/P'_N(t_j)$ alternate in sign, starting positive. They
converge to nonzero $c_j$ for each fixed $j$, since
$\sum_i 1/t_i<\infty$. The derivative product gives, for $j\le N$,
\[
 |c_{j,N}|\le C t_j\,2^{-j(j+1)/2},\qquad
 C=\prod_{l\ge1}(1-2^{-l})^{-2}.
\]
Indeed $t_j/t_i\ge2^{j-i}$ for $i<j$, and the factors
$1-t_{\min(i,j)}/t_{\max(i,j)}$ have product bounded below by $C^{-1}$.
The leading coefficient in Lagrange interpolation at $t_0,\ldots,t_N$
gives $\sum_{j=0}^N c_{j,N}t_j^k=0$ for $k<N$. Since $t_j\asymp4^j$,
dominated convergence yields $\sum_j c_jt_j^k=0$ for every $k\ge0$.
Integrate against the finite signed measure
$\delta_0+\sum_jc_j\delta_{x_j}$. Its total variation integrates every
polynomial in $\|x\|$ and $g(x)$, so the stated continuity follows
by dominated convergence. Its $g$-moments are those of
$\delta_{g(0)}$. Only the negative, odd-indexed masses contribute to
$L(x_2^2)$, proving the claim.
\end{proof}

If one instead requires continuity under all locally uniform polynomial
approximations on the half-line, the implication changes. Applying this
assumption at zero makes the marginal functional continuous for the
compact-open topology on polynomials, hence bounded by a constant times
the supremum on some $[0,M]$. Its even moments then force every positive
Stieltjes representative to be supported there. For each polynomial $v$
in the plane, approximate $v\circ f$ locally uniformly on the half-line
by polynomials. Continuity and $f\circ g=\mathrm{id}$ identify $L(v)$
with integration against the image under $f$ of that compactly supported
measure. Thus this stronger continuity condition itself forces a positive
representation.

Classical Bernstein--Shapiro approximation theorems prescribe slow
convergence by choosing the element to approximate
\cite{r86:AlmiraOikhberg}. For di Dio's question the code is fixed
and the representing probability may vary.
\begin{proposition}[Degree of a polynomial approximation to a fixed code]
\label{r86:momdegree}
Let $K\subset\mathbb R^d$ be compact and $g:K\to[0,1]$ measurable.
For continuous $g$, the degrees needed in
\cite[Theorem~4.7]{r86:transform} admit a common bound over all
probability measures on $K$ at each accuracy $\varepsilon>0$.
If $g$ is discontinuous, for every prescribed degree bound
$H(\varepsilon)$ increasing as $\varepsilon\downarrow0$, there is a
single full-support probability $\mu$ such that every polynomial $p$
with $\int|p-g|\,d\mu\le\varepsilon$ satisfies
$\deg p>H(\varepsilon)$ for all sufficiently small $\varepsilon$.
Every code of \cite[Theorem~4.4]{r86:transform} on $[0,1]^2$ has the
latter behavior. This gives the unrestricted degree-growth answer to
Open Problem~5.5 of that paper.
\end{proposition}
\begin{proof}
For continuous $g$, approximate $\sqrt g$ uniformly by a polynomial
$q$ to accuracy $\eta$ and set $p=((q+\eta)/(1+2\eta))^2$.
Then $0\le p\le1$, $\|p-g\|_\infty\to0$, and
$|p^k-g^k|\le k|p-g|$ for $k\ge1$.

At a discontinuity $a$, fix $\delta>0$ such that points $x$ with
$|g(x)-g(a)|\ge\delta$ accumulate at $a$. Choose a full-support
probability $m$ on $K$, set $\varepsilon_j=\delta 2^{-j-3}$,
and put $n_j=\lceil H(\varepsilon_{j+1})\rceil$. The restrictions
of polynomials of degree at most $n_j$ with $\|p\|_{L^1(m)}\le5$
form a compact subset of $C(K)$ by finite dimensionality and full
support. Choose such a discontinuity witness $x_j$ close enough to
$a$ that $|p(x_j)-p(a)|<\delta/2$ throughout this compact family.
Put
\[
 \mu=\tfrac14m+\tfrac14\delta_a+
             \tfrac12\sum_{j\ge1}2^{-j}\delta_{x_j}.
\]
For error at most one, $\|p\|_{L^1(m)}\le5$. If $\deg p\le n_j$,
the errors at $a,x_j$ sum to at least $\delta/2$, so their weighted
contribution is at least $\delta 2^{-j-2}>\varepsilon_j$.
For $\varepsilon_{j+1}\le\varepsilon\le\varepsilon_j$ this forces
$\deg p>n_j\ge H(\varepsilon)$. Finally each code on the square is
injective, since it has a left inverse. A continuous injection would
identify the square with a nondegenerate interval, yet deleting an
interior point of the interval disconnects it and deleting any point
of the square leaves it connected.
\end{proof}

\subsection{The largest origin atom compatible with a complex moment extension}
Let $\gamma_{m,n}=\int z^m\bar z^n\,d\mu$ for $m,n\ge0$ and write
$\mathrm{PDE}(\gamma)$ for the positive definite extensions to
$S=\{(m,n)\in\mathbb Z^2:m+n\ge0\}$, with $(m,n)^*=(n,m)$.
Cicho\'n, Stochel and Szafraniec asked whether a unique extension
forces a unique plane measure \cite[\S6, p.~286]{r85:CSS}. We obtain
the implication by maximizing the mass at the common origin of the
directional fibers. The fiber determinacy criterion of Putinar and
Schm\"udgen \cite[\S8]{r85:PutSch} and the one-dimensional maximal
atom formula \cite[\S7.5]{r85:SchmudgenMom} give the precedents.

Choose a Borel square root $u(w)^2=w$ on the circle $\mathbb T$.
An off-origin point has coordinates $w=z/\bar z$ and $t=z/u(w)\in
\mathbb R\setminus\{0\}$. A lift of $\Gamma\in\mathrm{PDE}(\gamma)$
is a measure $\rho$ on $\mathbb T\times\mathbb R$ such that
\[
 \Gamma_{k+\ell,-\ell}
   =\int t^ku(w)^kw^\ell\,d\rho(w,t)
   \qquad(k\ge0,\ \ell\in\mathbb Z).
\]
At $k=0$ take $t^0=1$ also at zero. A lift exists for every positive
definite sequence on $S$ by the representation formula
\cite[p.~265, (2)]{r85:CSS}: the boundary measure $\nu$ there gives
the measure $(\zeta\mapsto\zeta^2)_*\nu$ at $t=0$. Conversely that
boundary measure on directions can be lifted with $u$. The map
$(w,t)\mapsto t u(w)$ sends each lift to a plane measure for $\gamma$.
Call such a plane measure compatible with $\Gamma$.

\begin{theorem}[Largest directional origin mass]\label{r86:origin}
Disintegrate a lift as $d\rho=d\rho_w(t)\,d\lambda(w)$, with probability
fibers. The base measure $\lambda$ and the conditional polynomial
functionals $L_w(p)=\int p\,d\rho_w$ are determined by $\Gamma$
up to null sets. Put
\[
 c_\Gamma(w)=\inf_{\substack{p\in\mathbb Q[t]\\p(0)=1}}L_w(p^2).
\]
Among all lifts, the largest measure $\rho(dw,\{0\})$ in the
ordering of measures is $c_\Gamma\,d\lambda$. Exactly one lift
maximizes its total mass. Its plane image is the unique compatible
measure maximizing the atom at zero, of mass $\int c_\Gamma\,d\lambda$.
\end{theorem}
\begin{proof}
Fourier uniqueness first recovers $\lambda$ from $k=0$ and then,
for every $k$, the measure $u^k L_w(t^k)\,d\lambda$ from all $\ell$.
The countable infimum defining $c=c_\Gamma$ is measurable.
Real coefficient approximation and separation of real and imaginary
parts give $L_w(|p|^2)\ge c(w)|p(0)|^2$ for complex polynomials.
Every lift therefore satisfies $\rho_w(\{0\})\le c(w)$.

Each finite linear combination $F$ of the displayed monomials is a
polynomial in $t$ for fixed $w$. Thus
\[
 \int |F|^2\,d\rho-\int |F(w,0)|^2c(w)\,d\lambda(w)\ge0.
\]
Subtracting the moments of $c\,d\lambda\otimes\delta_0$ from $\Gamma$
leaves a positive definite sequence on $S$. Its lift exists by (2)
of \cite{r85:CSS}. Adding the subtracted measure gives a lift that
attains the upper bound in every direction up to a null set.
If a lift maximizes the total, its conditional atom equals $c(w)$
almost everywhere. Subtract this atom from the conditional functional.
The residual functional has infimum zero on the squares with $p(0)=1$,
so none of its representing measures charges zero. An indeterminate
Hamburger sequence always has a representing measure charging zero
\cite[Proposition~9(iii)]{r85:CSS}. Hence the residual conditional
law is unique, proving uniqueness of the maximizing lift and its image.
\end{proof}

If only $\gamma$ is prescribed, its largest possible origin mass is
\[
 \max_{\Gamma\in\mathrm{PDE}(\gamma)}\int c_\Gamma\,d\lambda_\Gamma
  =\inf_{\substack{P\in\mathbb R[x,y],\ P\ge0\ {\rm on}\ \mathbb R^2\\
                       P(0,0)=1}} L_\gamma(P).
\]
Indeed the Riesz--Haviland theorem \cite[Theorem~1.12]{r85:SchmudgenMom}
represents $L_\gamma-a\,\mathrm{ev}_{(0,0)}$ for the infimum $a$
on the right. Adding $a\delta_0$ attains that bound. Every plane
measure for $\gamma$ lifts to some extension by assigning its origin atom
to one chosen direction, so Theorem~\ref{r86:origin} gives the equality.
The maximizing plane measure can now be nonunique even when $a>0$.
Add $a\delta_0$ to any indeterminate Hamburger law translated to
$y=1$. The zero integral of $(y-1)^2(x^2+y^2)$ forces every
representation onto $\{0\}\cup\{y=1\}$, and its integral of $1-y$
fixes the atom at $a$.

\begin{corollary}[Uniqueness of complex moment extensions]\label{r85:css}
For every complex moment sequence $\gamma$, the set
$\mathrm{PDE}(\gamma)$ is a singleton if and only if $\gamma$ has
exactly one representing measure and that measure assigns zero mass
to the origin. In all other cases there are continuum many extensions.
\end{corollary}
\begin{proof}
If $\Gamma$ is the only extension, every plane representation is
compatible with it and has zero mass at the origin by
\cite[Proposition~5]{r85:CSS}. All attain the maximum, so
Theorem~\ref{r86:origin} makes them equal. The converse is
\cite[Theorem~6]{r85:CSS}. Two distinct extensions yield a continuum
by convexity, and the index set is countable.
\end{proof}

\section{The Gaussian endpoint and its cumulant diagrams}
\label{om42:sec}
Let $\mu_\omega$ be the braided central limit law for $|\omega|=1$,
$q=\operatorname{Re}\omega=\cos\theta$, $-\pi\le\theta\le\pi$.
Write $M_n,\kappa_n,r_n$ for its moments, classical cumulants and free
cumulants. The odd moments vanish. We use the oriented-pairing formula of
\cite[Eq.~(1)]{om42Alahmade}. For a pairing $\pi$ of $[2k]$, label its
chords in order of their left endpoints. Independently mark each chord
by $x=(u,e)$, with $u$ uniform on $[0,1]$ and $e$ a symmetric sign;
denote this mark law by $\rho$. A crossing of chords $i<j$ carries
the factor
\[
 Q(x_i,x_j)=\exp\big(i\theta e_i e_j\operatorname{sgn}(u_j-u_i)\big).
\]
Then $M_{2k}$ is the sum over pairings of the expected product of these
factors: the independent ranks induce a uniform order of the chords,
and the signs supply their orientations.

\begin{theorem}[Exact Gaussian-endpoint divisibility]\label{om42:order}
For each $k\ge2$ there is a constant $L_k>0$ such that
\[
 \kappa_{2k}(q)=L_k(q-1)^{\lceil k/2\rceil}
                  +O\big((q-1)^{\lceil k/2\rceil+1}\big).
\]
Alahmade proved divisibility by $(q-1)^2$ for every $k\ge3$
and conjectured divisibility by $(q-1)^{\lceil k/2\rceil}$
\cite[Theorem~1 and Conjecture~1, restated as Problem~1]{om42Alahmade}.
The positive leading coefficient gives the exact multiplicity.
\end{theorem}
\begin{proof}
Realize a random chord $X=(a,b)$ by ordering two independent uniform
endpoints. Set $C(X,Y)=1$ if $a<c<b<d$, $-1$ if $c<a<d<b$, and zero
otherwise, where $Y=(c,d)$. With $Z=(X,u)$ of law $\lambda$ let
$K(Z,Z')=C(X,Y)\operatorname{sgn}(u'-u)$.
Sorting $2k$ independent endpoints gives a uniform pairing, each pairing
having $2^kk!$ labelled preimages. Consequently
\[
 b_k:=\frac{M_{2k}}{(2k-1)!!}
 =\mathbb E\exp\!\left(i\theta\sum_{i<j}e_i e_j K(Z_i,Z_j)\right).
\]
Expand the exponential over loopless multigraphs $G$ on the labelled
chords, with edge multiplicities $n_{ij}$, number of edges $m$, weight
$W(G)=\mathbb E\prod_{i<j}K(Z_i,Z_j)^{n_{ij}}$, and
$G!=\prod_{i<j}n_{ij}!$. Sign averaging requires all degrees to be
even. Weights factor over components, so the logarithm of
$\sum_{k\ge0}b_k(t^2/2)^k/k!$ selects connected graphs. Replacing every
$u_i$ by $1-u_i$ kills the odd $m$ terms. Thus
$\kappa_{2k}=(2k-1)!!\sum_G(i\theta)^m W(G)/G!$, over connected
even-degree graphs with even $m$. This coefficientwise expansion is
finite at each power of $\theta$. All vertices have degree at least
two, so $m\ge2\lceil k/2\rceil$.

For a rooted cycle of length $l\ge2$ let $D_l(z)$ be its weight after
integrating all vertices except the root $z$; a double edge is a
two-cycle. For odd $l$ the independent chord and rank integrals factor,
and reversing the cycle in the antisymmetric chord integral negates it.
Thus $D_l=0$. For $l=2s$, symmetry of $K$ gives
$D_{2s}(z)=\int |K^s(z,w)|^2d\lambda(w)$, with $K^s$ the kernel of
the $s$th operator power. Its integral is positive since $K$ is a
nonzero self-adjoint Hilbert--Schmidt operator, and $D_{2s}$ vanishes
where $D_2=0$ by iterating a zero row.
At the least possible $m$, even $k$ gives one cycle of length $k$.
Odd $k$ gives a unique degree-four vertex and two cycles meeting there,
of lengths $a+b=k+1$ and weight $\int D_aD_b\,d\lambda\ge0$.
The choice $a=2,b=k-1$ gives positive weight. Hence the first coefficient
does not cancel; $(i\theta)^{2s}/(q-1)^s\to2^s$ for $s=\lceil k/2\rceil$.
The moments, and hence the cumulants, are polynomials in $q$.
\end{proof}

\begin{proposition}[Laplace asymptotics at the Gaussian endpoint]\label{r78:braid-laplace}
Locally uniformly for real $t$,
\[
 \lim_{\theta\downarrow0}e^{-t^2/(2\theta)}
   \int e^{tx/\sqrt\theta}\,d\mu_{e^{i\theta}}(x)
   =\prod_{n=1}^{\infty}\frac{1}{\cosh(t^2/(2\pi n))}.
\]
\end{proposition}
\begin{proof}
In the notation of Theorem~\ref{om42:order}, let $W_i=(Z_i,e_i)$ and
$H(W_i,W_j)=e_ie_jK(Z_i,Z_j)$. With $N$ independent Poisson of mean
$t^2/(2\theta)$ the normalized Laplace transform equals
$\mathbb E\exp(i\theta\sum_{i<j\le N}H(W_i,W_j))$ by the formula for $b_k$.
Indeed $M_{2k}\le(2k-1)!!$ gives absolute exponential integrability,
and the odd moments vanish. The symmetric kernel $H$ has conditional
mean zero and the same nonzero eigenvalues as $K$ via $f(Z)\mapsto e f(Z)$.

For a deterministic sample of size $p$, the sum divided by $p$ converges
in law to $\tfrac12\sum_j\lambda_j(G_j^2-1)$, where the $\lambda_j$
are those eigenvalues and $G_j$ are independent standard normals.
For a finite real orthonormal spectral expansion this follows by the
multivariate central limit theorem for the sums of the eigenfunctions
and the law of large numbers for their squares. A discarded kernel of
Hilbert--Schmidt norm $d$ contributes variance at most $d^2/2$ uniformly
in $p$, since two distinct pairs have zero covariance. This proves the
full limit. At $t\ne0$ independence of $N$ passes this convergence to
the Poisson sample, and $\theta N\to t^2/2$ in probability. The argument also applies if the parameter varies toward any nonzero $t$.

To compute the eigenvalues, write $K=C\otimes S$ with
$S(u,v)=\operatorname{sgn}(v-u)$. For a chord $X=(a,b)$ let
$Df(X)=f(b)-f(a)$. The crossing kernel is
$C(X,(c,d))=\mathbf1_{(a,b)}(c)-\mathbf1_{(a,b)}(d)$.
Indeed $Cg(a,b)=\int_a^b h(x)\,dx$, with
$h(x)=2\int_x^1g(x,d)\,dd-2\int_0^xg(c,x)\,dc$.
Thus its range lies in the difference subspace $D(L^2[0,1])$. On mean-zero functions $D/\sqrt2$ is an isometry
and $CDf(a,b)=-2\int_a^b f(u)\,du$. The Fourier modes $e^{2\pi inu}$,
$n\ne0$, thus yield all the nonzero eigenvalues $i/(\pi n)$ of $C$.
The antiperiodic Fourier modes $e^{\pi i(2m+1)u}$ yield the eigenvalues
$2i/(\pi(2m+1))$ of $S$, $m\in\mathbb Z$. For each $n\ge1$ and each
positive odd $l$, the limiting quadratic form therefore uses two
eigenvalues of each sign with magnitude $2/(\pi^2nl)$. Their joint
characteristic-function factor at the required argument is
$(1+t^4/(\pi^4n^2l^2))^{-1}$. Euler's product \cite[(4.36.2)]{r77:gamma} for the hyperbolic
cosine gives the displayed result. Near zero the phase in the Poisson
expectation has second moment $t^4\|H\|_{\rm HS}^2/8$ by degeneracy.
This gives uniform convergence to one as $t\to0$ and completes the
local uniformity assertion.
\end{proof}

The same eigenfunctions give the leading coefficients in
Theorem~\ref{om42:order}. For even $l\ge2$ put
$T_l=\operatorname{tr}K^l=4(2^l-1)\zeta(l)^2/\pi^{2l}$ and
$c_k=L_k/((2k-1)!!\,2^{\lceil k/2\rceil})$. Then
\[
 c_k=\begin{cases}
 (k-1)!T_k/2,&k\text{ even},\\
 \displaystyle\frac{k!}{8}
 \sum_{\substack{a+b=k+1\\a,b\ge2\text{ even}}}
 T_aT_b\left(1+\frac{\zeta(a+b)}{2\zeta(a)\zeta(b)}\right),
 &k\text{ odd}.
 \end{cases}
\]
For a chord of length $y$, the diagonal of $K^l$ equals
$T_l[1-\sum_{n\ge1}\cos(2\pi n y)/(n^l\zeta(l))]$.
Indeed the normalized eigenvectors of $C$ are $De^{2\pi inu}/\sqrt2$,
and those of $S$ have constant modulus one. The chord length has density
$2(1-y)$ on $[0,1]$, so integration of a product of these symmetric
diagonals agrees with Lebesgue integration. Orthogonality of cosines
gives the factor in parentheses. A labelled cycle has total weight
$(k-1)!/2$, including the double edge with divisor $2!$. Choosing the
root and the remaining vertices of the two rooted cycles, then dividing
by two for their exchange, gives $k!/8$ per ordered length decomposition.

\begin{proposition}[An exact moment recurrence]
For a list $\boldsymbol x=(x_1,\ldots,x_h)$ define\label{om42:rec}
$F_{0,h}=\mathbf1_{\{h=0\}}$ and
\[
 F_{n+1,h}(\boldsymbol x)=\int F_{n,h+1}(\boldsymbol x,y)\,d\rho(y)
 +\sum_{j=1}^h\prod_{l=j+1}^h Q(x_j,x_l)
                  F_{n,h-1}(\boldsymbol x\setminus x_j).
\]
Then $M_{2k}=F_{2k,0}(\varnothing)$, giving an exact state-valued
recurrence for the moments in \cite[Problem~2]{om42Alahmade}.
\end{proposition}
\begin{proof}
The list records open chords in their opening order. The next endpoint
either opens a chord with an independent mark, or closes chord $j$,
crossing precisely the later chords still open. Thus each pairing
appears once with its full weight and an empty terminal list.
More explicitly $F_{n,h}=0$ for $n<h$ or odd $n-h$.
On each rank-order chamber and for each sign vector, $F_{n,h}$ is
polynomial in the ranks of degree at most $(n-h)/2$: closing does
not raise this bound, and opening splits the new rank integral at the
existing ranks and raises its degree by at most one.
\end{proof}

\begin{theorem}[Diagrams for both cumulants]\label{om42:diagrams}
Let $E(\pi)$ be the crossing edges. For disjoint $A,B\subset E(\pi)$
with even $B$-degrees and even $|B|$, set
\[
 w(B)=\frac1{k!}\sum_{\tau\in S_k}\prod_{(i,j)\in B,\ i<j}
           \operatorname{sgn}(\tau(j)-\tau(i)),\qquad
 V(A,B)=(q-1)^{|A|}(q^2-1)^{|B|/2}w(B).
\]
Partition the endpoints according to the connected components of
$([k],A\cup B)$, including isolated chords, and denote this partition
by $\eta$. Write $\overline\eta$ for its least noncrossing coarsening
and $\mathbf1$ for the one-block partition. Then
\[
 \kappa_{2k}(q)=\sum_{\pi,A,B:\,\eta=\mathbf1}V(A,B),\qquad
 r_{2k}(q)=\sum_{\pi,A,B:\,\overline\eta=\mathbf1}V(A,B).
\]
These give the weighted combinatorial interpretations asked for in
\cite[Problem~3]{om42Alahmade}.
\end{theorem}
\begin{proof}
At a crossing expand $Q=1+(q-1)+i\sin\theta\,e_i e_j
\operatorname{sgn}(u_j-u_i)$, taking the last two terms on $A,B$
respectively. Sign averaging imposes the vertex parities; rank reversal
kills odd $|B|$, and the remaining average is $V$. In the product of
moments over an endpoint partition $\sigma$, precisely $\eta\le\sigma$
is required: each chord and each selected edge must remain within one
block, and relative ranks in different blocks are independent.
M\"obius inversion in the full partition lattice selects
$\eta=\mathbf1$. Within the noncrossing lattice it selects
$\overline\eta=\mathbf1$; see also the connected-partition transform
of \cite{om42Lehner}.
\end{proof}

\begin{corollary}[The diagram rule at $\omega=\pm i$]\label{om42:special}
At $q=0$ retain the pairings $\pi$ with $\overline\pi=\mathbf1$ whose
full crossing graph has even vertex degrees and an even number of edges.
Then $r_{2k}(0)$ is the sum of
$(-1)^{|E(\pi)|/2}w(E(\pi))$ over these pairings. This gives
the diagrammatic interpretation in \cite[Problem~4]{om42Alahmade}.
\end{corollary}
\begin{proof}
For any noncrossing $\sigma\ge\pi$ no crossing joins distinct blocks of
$\sigma$, so the full pairing weight factors in the product of block
moments. Noncrossing inversion therefore selects
$\overline\pi=\mathbf1$. At $q=0$ every crossing supplies the sine
term, giving exactly the stated parity rule.
\end{proof}

\subsection{Odd cumulants and free infinite divisibility}
Ejsmont and Lehner ask whether a homogeneous self-adjoint polynomial
whose distribution on free copies ignores all odd cumulants of the
common input law preserves free infinite divisibility
\cite[Conjecture~5.2]{r87:ELvariance}. They classify quadratic forms on arbitrary free families and give a
higher-degree polynomial preserving free infinite divisibility without
cancelling odd cumulants \cite[Theorem~4.4 and Proposition~4.8]{r87:ELsums}.
Here we address the converse implication.

\begin{theorem}[Homogeneous counterexamples in every degree at least three]
\label{r87:freecancel}
For every $d\ge3$ there is a self-adjoint homogeneous noncommutative
polynomial of degree $d$ whose moments on free identically distributed
inputs depend only on their even free cumulants, and whose value on
free standard semicircular inputs has negative fourth free cumulant.
In fact its output law on semicircular inputs can be made arbitrarily
close in moments to the arcsine law on $[-2,2]$.
\end{theorem}
\begin{proof}
Use disjoint pairs of free copies of $X$ and write
$Y_j=(X_{2j+1}-X_{2j+2})/\sqrt2$ for $0\le j\le N$.
These variables are free and
$\kappa_k(Y_j)=2^{-k/2}(1+(-1)^k)\kappa_k(X)$.
Consequently every joint moment of the $Y_j$ depends only on the even
free cumulants of $X$. At standard semicircular input the $Y_j$
are again free standard semicirculars. Set $S=Y_0$ and
$D_N=N^{-1}\sum_{j=1}^N Y_j^2$. Then $\|D_N\|\le4$ and
$\|D_N-1\|_2^2=1/N$ by freeness and unit variance of $Y_j^2$.

Write $2T_d(x/2)=\sum_k c_kx^k$, where $T_d(\cos u)=\cos(du)$.
The powers $k$ have the parity of $d$. With $h_k=(d-k)/2$, the polynomial
\[
 P_N=\tfrac12\sum_k c_k(S^kD_N^{h_k}+D_N^{h_k}S^k)
\]
is homogeneous of degree $d$ in the original inputs.
On semicircular inputs the uniform norm bounds and telescoping powers
give $P_N\to 2T_d(S/2)$ in tracial $L^2$ and in every moment.
Writing $S=2\cos u$ in law, the density of $u$ on $[0,\pi]$ is
$(1-\cos(2u))/\pi$. For $d\ge3$ each $\cos(jdu)$, $j\ge1$, is
orthogonal to both $1$ and $\cos(2u)$ for Lebesgue measure there.
Hence $2\cos(du)$ has the arcsine law, by the cosine moments.
Its mean, second and fourth moments are $0,2,6$. Its fourth free
cumulant is $6-2\cdot2^2=-2$, so $\kappa_4(P_N)<0$ for all
sufficiently large $N$. The free cumulants of a compactly supported
freely infinitely divisible law are conditionally positive definite
\cite[Theorem~13.16]{r87:NicaSpeicher}; in particular its fourth free
cumulant is nonnegative. The semicircular input is freely infinitely
divisible, which proves the claim.
\end{proof}

For a concrete cubic take $N=10$ and $P=(SD_N+D_NS)/2-S^3/10$.
Its law is centered. The noncrossing moment formula gives
$\tau(P^2)=7/10$, $\tau(P^4)=37703/40000$ and
$\kappa_4(P)=-1497/40000$. To reproduce the calculation, expand each
power into words in $S,D_N$ and sum over noncrossing partitions in which all
entries of each block have the same letter. An $S$-block has weight $1$ at size two and $0$
otherwise; a $D_N$-block of size $j$ has weight $N^{1-j}$, since
each semicircular square is free Poisson of rate one.

\begin{corollary}[Cancellation for arbitrary free families]\label{r87:freefamilies}
There is a homogeneous polynomial of degree six whose moments on
every free family depend only on the even free cumulants of the
individual inputs, and which fails to preserve free infinite divisibility.
This answers the reverse implication asked about in
\cite[Section~4.3]{r87:ELsums}.
\end{corollary}
\begin{proof}
In the explicit cubic just given, replace each $Y_j$ by
$Z_{j,M}=(2M)^{-1/2}\sum_{k=1}^M i[U_{jk},V_{jk}]$, with all
inputs taken from disjoint pairs in a free family. By
\cite[Theorem~4.4]{r87:ELsums} every commutator distribution depends
only on the even cumulants of the two inputs. The commutators are
free across pairs, so the same even-cumulant dependence holds for
all moments of the resulting polynomial. At free standard semicircular
input each $i[U,V]$ has mean zero and variance two, by freeness.
The $Z_{j,M}$ are free and converge jointly in moments to standard
free semicirculars: the cumulants of each normalized sum of orders
$r>2$ are $O(M^{1-r/2})$. The fourth free cumulant of the degree-six
output therefore tends to $-1497/40000$ and is negative for some
finite $M$.
\end{proof}

\subsection{The fourth-moment criterion at negative parameters}
\label{r88:qchaos}
For a variance-one multiple integral of order $n$, the candidate limit
is a centered $q^{n^2}$-Gaussian of variance one. Deya, Noreddine and
Nourdin proved the fourth-moment criterion for fully symmetric kernels
when $q\geq0$, and found a counterexample at $q=-1/2$ in order two
\cite[Theorem~1.7 and Remark~3.3]{r88:DNN}. Their order-two formula
isolates that exceptional parameter. Kemp and Miyagawa ask, immediately
after their Theorem~1.3, whether other negative parameters admit
counterexamples \cite{r88:KempMiyagawa}.

\begin{theorem}[The negative-parameter classification in fourth chaos]
\label{r88:negativeq}
Let $-1<q<0$ and let $\rho$ be the unique root of $1+\rho+\rho^9=0$ in $(-1,0)$.
For real fully symmetric kernels of order four, the fourth-moment
criterion holds exactly for $q>\rho$. More precisely, for every sequence
$X_N=I_4^q(f_N)$ with $\tau_q(X_N^2)\to1$, convergence of
$\tau_q(X_N^4)$ to $2+q^{16}$ forces convergence in moments to the
standard $q^{16}$-Gaussian precisely in that parameter range.

For $q<\rho$ one finite-rank kernel has variance one, fourth moment
$2+q^{16}$ and positive third moment. At $q=\rho$ a variance-normalized
sequence has the target fourth moment in the limit, while its sixth
moment converges to a different value.
\end{theorem}
\begin{proof}
\emph{Above the critical parameter.} Write \(h_j=[j]_q!=\prod_{i=1}^j\frac{1-q^i}{1-q}\) with \(h_0=1\).
In a Wick pairing of four ordered blocks of four legs, the edge
multiplicities on the opposite pairs \(12,34\), \(13,24\), and
\(14,23\) agree within each pair, by the vertex degree equations.
Denote the three multiplicities by \(a,b,c\), so \(a+b+c=4\).
Every such pairing has the same tensor integral \(J(a,b,c)\) for a
fully symmetric real kernel \(f\). This integral is symmetric in
\(a,b,c\) by vertex relabeling. The total weight of these pairings is
\[
    h_4^4\,q^{b^2}/(h_a h_b h_c)^2.
\]
Indeed chords with all endpoints on distinct vertices contribute
\(b^2\) crossings. At each vertex the shuffle of the three target
groups has inversion enumerator \(h_4/(h_a h_b h_c)\), with the
reverse cyclic order of the targets as the order of zero crossings.
The matchings within each of the six individual edge bundles
contribute in total \((h_a h_b h_c)^2\).

Set \(T_i=J(0,i,4-i)=\|f\frown_i f\|_2^2\) for \(i=1,2\) and
\(J=J(1,1,2)\), where \(\frown_i\) integrates \(i\) shared variables.
Contract each of the two doubled edges in the last diagram, then
pair the two resulting tensors after permuting their legs. Thus
\(|J|\le T_2\) by Cauchy--Schwarz. Removing the three fully paired
block diagrams leaves, for \(X=I_4^q(f)\),
\[
 \frac{\tau_q(X^4)-(2+q^{16})\tau_q(X^2)^2}{h_4^4}
  =\frac{2(1+q+q^9)}{h_3^2}\,T_1
    +\frac{1+2q^4}{(1+q)^4}\,T_2
    +\frac{2q+q^4}{(1+q)^2}\,J .
\]
The second coefficient exceeds the modulus of the third, because
\(1+2q^4+(2q+q^4)(1+q)^2\ge1-8/27>0\).
The first coefficient is positive when \(q>\rho\). Thus convergence
of the fourth moment forces \(T_1,T_2\to0\), and
\(\|f_N\frown_3 f_N\|_2^2=T_1\) by the same relabeling.

For any fixed moment order, a pairing which does not match whole
blocks has two vertices joined by \(1\le j\le3\) legs. Integrating
these legs first bounds its tensor integral in modulus by
\(\|f_N\frown_j f_N\|_2\|f_N\|_2^{m-2}\) for a moment of order \(m\).
Here the product integral of the remaining tensors is bounded by
the product of their \(L^2\) norms: integrate each shared variable
by Cauchy--Schwarz in its two incident factors, then repeat with their
partial \(L^2\) norms. The variance is \(h_4\|f_N\|_2^2\), so all
such pairings vanish in the limit. For each pairing of whole blocks,
the internal weights and tensor integrals give one variance per pair
and the external crossings give \(q^{16}\) per crossing. The sum
therefore tends to the corresponding \(q^{16}\)-Gaussian moment.
The converse follows from convergence of the fourth moment.

\emph{Below the critical parameter.} Write $I$ for the fourth integral, $S$ for tensor symmetrization by
averaging, and $D(X)=\tau(X^4)-(2+q^{16})\tau(X^2)^2$.
Choose three mutually orthogonal real orthonormal families
$(e_i),(u_i),(v_i)$ in $L^2(\mathbb R_+)$ and put
\[
 A_N=I\left(N^{-1}\!\sum_{i,j\le N}S(e_i\otimes u_j^{\otimes3})\right),
 \qquad
 B_N=I\left(N^{-1}\!\sum_{1\le i<j\le N}
                    S(v_i^{\otimes2}\otimes v_j^{\otimes2})\right),
\]
where indices start at one. We show that, for positive constants
$c_A,c_B,c_3$ depending only on $q$,
\[
 D(A_N)\longrightarrow c_A(1+q+q^9),\qquad
 D(B_N)\longrightarrow c_B(1+2q^4),\qquad
 \tau(B_N^3)\longrightarrow c_3 .
\]

Use the Wick pairing formula \cite[Theorem~2.7]{r88:DNN}. Draw the
$m$ Wick blocks in cyclic order. Each block contains two distinct
colors with respective multiplicities $(1,3)$ or $(2,2)$.
A leg must pair with a leg of the same color in a different block.
With $2m$ color occurrences and prefactor $N^{-m}$, patterns with
fewer than $m$ distinct colors contribute $O(N^{-1})$ in total.
In a surviving pattern each color thus occurs in two blocks and joins
them by a bundle of parallel legs. The resulting graph of bundles has
degree two at every vertex. In the fourth moment, all leg diagrams
consisting of two disconnected block pairs contribute exactly
$(2+q^{16})\tau(X^2)^2$: two pairings of the blocks are noncrossing
and the third forces $4\cdot4$ crossings. This removes all
disconnected terms in $D$.

Here is the weight of a surviving simple cycle. For an edge of width
$w$, all bijections of its two sets of legs contribute
$[w]_q!$, with $[k]_q!=\prod_{j=1}^k(1-q^j)/(1-q)$.
For two bundles of widths $a,b$ incident at a vertex, place the two
groups of legs in the order giving no mutual crossings. Any shuffle
adds its inversion number, so summing over their positions in $S$
contributes
\[
 \binom{a+b}{a}^{-1}\frac{[a+b]_q!}{[a]_q![b]_q!}.
\]
The relative ordering of legs inside a single bundle is accounted for
by its bijection. These choices are independent. Edges with four
distinct endpoints contribute $ab$ further crossings precisely if
their endpoints alternate. Thus all local factors are positive even
for negative $q$, and depend only on the widths. Choices of distinct
colors have the same limiting multiplicity for every cycle with the
same edge types: assign distinct indices to its distinguishable edges
within each family. For $B_N$, sorting the two incident indices gives
the required $i<j$ at each vertex.

For $A_N$, edges of widths one and three form two perfect matchings.
The cycle cases are the six ordered pairs of different matchings.
Two have no alternating edges, two have the width-one matching
crossing, and two have the width-three matching crossing. For $B_N$,
there are three undirected four-cycles with width two on every edge.
One has no alternating edges and the other two have four forced
crossings each. In the third moment of $B_N$, the surviving graph is
a triangle and has no alternating edges. These counts prove the
three limits. The third moment of $A_N$ vanishes exactly because the
number of legs in the $e$-family is odd.

Finally, kernels on orthogonal subspaces satisfy
$D(A_N+tB_N)=D(A_N)+t^4D(B_N)$: any surviving mixed diagram must
pair the two $A_N$ blocks to one another and likewise the two $B_N$
blocks. The four nonalternating and two alternating words give
the mixed term $2(2+q^{16})t^2\tau(A_N^2)\tau(B_N^2)$.
For the same reason mixed third moments vanish. For all sufficiently
large $N$ take $t=(-D(A_N)/D(B_N))^{1/4}>0$. The sum has positive
third moment and zero $D$; division by its standard deviation yields
the assertion.

\emph{At the critical parameter.} Keep the kernels $A_N$ from the preceding case and put
$[4]_q=1+q+q^2+q^3$. The variance $V=[4]_q!/4$ of $A_N$ is
independent of $N$: the $N^2$ symmetrized kernels are orthogonal and
each has ordinary squared norm $1/4$. Set $h=[4]_q/4$. In the sixth
moment the surviving colors define a pair $(P,Q)$ of perfect
matchings of the six blocks, of widths one and three. A common edge
of $P,Q$ gives a two-vertex component with weight $V$. On every
longer cycle each vertex gives a shuffle factor $[4]_q/4$ and each
width-three edge gives $[3]_q!$. Its weight per vertex is therefore
$W=[4]_q\sqrt{[3]_q!}/4$, with $W^2/V=h$. In addition, each pair
of alternating edges of widths $a,b$ contributes $q^{ab}$, including
edges from different components. As before, color identifications
beyond the paired occurrences contribute $O(N^{-1})$.

Let $F_j(x)$ be the sum of $x^{C(P,Q)}$ over ordered pairs with
$|P\cap Q|=j$, where $C$ adds the products of widths of all
alternating edge pairs. The possible $j$ are $0,1,3$. The term $F_3(q)$
is the sixth moment of a variance-one $q^{16}$-Gaussian. The difference
from that moment for $A_N/\sqrt V$ tends to
$h^2 F_1(q)+h^3F_0(q)$. To verify that this is nonzero, put
$g(x)=1+x+x^9$. The remainder of
$4F_1(x)+(1+x+x^2+x^3)F_0(x)$ modulo $g$ equals $6R(x)$, where
\[
 R(x)=-3x^8-x^6-4x^5+x^4+10x^3+15x^2+14x+6.
\]
For reproducibility, the following recursion gives the entire finite
calculation. All polynomial operations can be performed modulo $g$.
\begin{quote}\small\ttfamily
M(empty) = \{empty matching\}.\\
M(S): pair the least a in S with each other b, adjoining (a,b)\\
\hspace*{1em}to every matching in M(S minus \{a,b\}).\\
For P,Q in M(\{1,2,3,4,5,6\}), set j = size(P intersect Q).\\
List P-edges with weight 1 and Q-edges with weight 3.\\
For unordered edge pairs (a,b,w),(c,d,z), add w*z to C\\
\hspace*{1em}if a<c<b<d or c<a<d<b, starting with C=0.\\
Add x\textasciicircum C to F[j], starting with zero polynomials.
\end{quote}
Modulo two, $R(x)=x^2(x+1)^6$ while $g(0)=g(1)=1$ in
$\mathbb F_2$. Since $g$ is monic, the two polynomials are coprime
over $\mathbb Q$. Hence $R(\rho)\ne0$. The fourth moment tends to
$2+\rho^{16}$ by the four-cycle computation. The same color count at
eight blocks bounds the eighth moments uniformly, so the sixth-moment
discrepancy also rules out weak convergence to the candidate law.
\end{proof}

In order two the same enumeration gives the excess divided by \(h_2^4\)
as \((1+2q)T_1\), recovering the exceptional value \(-1/2\).
In order three it gives \(2(1+q+q^4)T_1/(1+q)^2+qJ(1,1,1)\)
after division by \(h_3^4\), with \(|J(1,1,1)|\le T_1\).
The coefficient left after this bound has numerator
\((1+q)(2+q+q^2)+2q^4>0\) over \((1+q)^2\).
Thus the fourth moment criterion holds in order three for every \(-1<q<0\).

\begin{corollary}[A negative interval valid in every chaos order]
\label{r88:qallpositive}
For \(-1/2<q<0\), the fourth moment theorem holds for real fully
symmetric kernels in every fixed order \(n\): with variances tending
to one, fourth moment convergence to \(2+q^{n^2}\) is equivalent to
convergence in moments to the standard \(q^{n^2}\)-Gaussian.
At \(q=-1/2\) this equivalence holds in every order except \(n=2\).
\end{corollary}
\begin{proof}
Order one consists of Gaussian fields with parameter \(q\).
For \(n\ge2\) the four-block enumeration in
Theorem~\ref{r88:negativeq} applies with \(a+b+c=n\) and
overall factor \(h_n^4\). Put \(r=-q\) and
\(T_a=J(0,a,n-a)=\|f\frown_a f\|_2^2\).
In the fourth moment excess retain the positive terms with \(b=0\).
For \(a,b,c>0\), contracting the two \(a\)-fold edges and applying
Cauchy--Schwarz bounds \(|J(a,b,c)|\) by \(T_a\).
For \(b>0\), the two boundary terms with \(a=0\) or \(c=0\) contribute
\(2q^{b^2}T_{n-b}/(h_bh_{n-b})^2\) after the overall factor is removed;
the fully paired diagrams have already been subtracted. Thus the
excess divided by \(h_n^4\) is at least
\[
 \sum_{a=1}^{n-1}\frac{T_a}{(h_a h_k)^2}
 \left(1-2r^{k^2}
 -\sum_{j=1}^{k-1} r^{j^2}
     \left(\frac{h_k}{h_jh_{k-j}}\right)^2\right),
 \qquad k=n-a .
\]
For \(k=1\) the bracket is \(1-2r\). For \(k\ge2\) and \(r\le1/2\),
\(h_k/h_{k-1}=(1-q^k)/(1-q)\le1\).
For \(2\le j<k\), the ratio \(h_k/(h_jh_{k-j})\) equals
\(\prod_{i=1}^j(1-q^{k-j+i})/(1-q^i)\). In its numerator the odd
exponents start at three or later, so the numerator is at most
\(\exp(r^3/(1-r^2))\). The denominator is at least
\(\prod_{i\ge1}(1-r^{2i})\ge1-r^2/(1-r^2)\).
The ratio is therefore at most \((3/2)e^{1/6}<9/5\).
Since \(\sum_{j\ge2}r^{j^2}\le r^4/(1-r^5)\le2/31\), the bracket
for \(k\ge2\) is at least
\(1-1/8-1/2-(9/5)^2(2/31)>0\).
All nontrivial contractions consequently vanish when the excess
tends to zero. Even at \(q=-1/2\), when \(n\ge3\), the only lost
term \(T_{n-1}\) equals \(T_1\), which is still controlled.
The whole-block pairing argument above gives the limiting moments,
and the order-two counterexample at \(-1/2\) is due to
Deya, Noreddine and Nourdin.
\end{proof}

\section{Radial certificates for polynomial positivity}\label{r55:positive}
A positive measure integrates every nonnegative polynomial to a nonnegative number. A linear functional which is positive on polynomial squares may fail this test.
The next lemma gives a separator with a prescribed factorial bound.
\begin{lemma}[A separator of controlled growth]\label{avg:prop:gevrey-squares}
For each $\sigma>1$ there is a linear functional $L$ on $\R[x,y,z]$ which is nonnegative on all squares, strictly negative on the Motzkin sextic $M$, and satisfies $L(|x|^{2n})\le AB^{2n}((2n)!)^\sigma$ for some $A,B$ and all $n$. In the bound $|x|$ is the Euclidean norm of the three-coordinate vector.
\end{lemma}
\begin{proof}
The nonnegative Motzkin sextic $M=x^4y^2+x^2y^4+z^6-3x^2y^2z^2$
is not a sum of squares \cite{Motzkin}. Sums of squares of polynomials
of degree at most three form a closed cone in degrees at most six:
Gaussian integration bounds the traces of their positive Gram matrices
in an orthonormal polynomial basis. Separate $M$ from the cone and add
a small Gaussian integral. The resulting functional $L_0$ has $L_0(M)<0$
and $L_0(q^2)\ge\delta|q|^2$ for $\deg q\le3$, with $|q|$ its coefficient
norm. Extend $L_0$ by zero above degree six.

Take $\alpha=1/\sigma<1$, $d\mu(x)\propto e^{-|x|^\alpha}dx$ on $\mathbb R$
and $\mu_S$ the law of $SX$, $X\sim\mu^{\otimes3}$. The bounded function $g(x)=\cos(|x|^\alpha\tan(\pi\alpha/2))$ has every
$\mu$-moment zero: odd moments vanish by symmetry and the degree-$2j$ gamma
integral contains $\cos((2j+1)\pi/2)$. Thus $(1\pm g/2)\mu$ give Hamburger
indeterminacy. Hence the univariate
moment matrices dominate $cI$ uniformly \cite[Theorem~1.1]{BergChenIsmail}.
Tensoring and scaling yields $|[x^\gamma]p|\le C S^{-|\gamma|}
\|p\|_{L^2(\mu_S)}$. For $q=p_{\le3}$ only degrees four through six of
$p-q$ enter $L_0(q(p-q))$, and $L_0((p-q)^2)=0$. Completing the square in
$\delta|q|^2-2C'S^{-4}|q|\|p\|_{L^2(\mu_S)}$ gives
$L=L_0+(C'^2/\delta)S^{-8}\mu_S\ge0$ on all squares for $S\ge1$.
Homogeneity gives $L(M)=L_0(M)+O(S^{-2})<0$ for large $S$.
The gamma integral bounds $\int\|x\|^r d\mu_S\le AB^r(r!)^\sigma$.
The finite coefficient part of $L$ only changes the constants in the bound.
\end{proof}
\subsection{The critical factorial weight and the cost of radial square certificates}
\label{kms:sec}
Write $B_{N,t}(x)=\sum_{n=0}^{N}|x|^{2n}/(n!)^t$, and let
$N_{F,\varepsilon}(t)$ be the least $N$ for which
$F+\varepsilon B_{N,t}$ is a sum of squares. Klep, Magron and Sch\"otz
give an upper bound $O_F(\varepsilon^{-1/(1-t)})$ for $0\le t<1$ and
ask about degree growth for $1\le t\le2$ and maximality of $t=2$
\cite[Theorem~2 and subsequent discussion]{KlepMagronSchoetz}.

\Needspace{8\baselineskip}
\begin{theorem}[Sharp radial exponent and degree costs for two fixed forms]\label{kms:thm:critical}
For $t\ge0$, every globally nonnegative polynomial in every finite
dimension admits such a certificate for each $\varepsilon>0$ exactly
when $t\le2$. For the quaternary Choi--Lam quartic
\[
 Q=x^2y^2+y^2z^2+z^2x^2+w^4-4xyzw
\]
and the ternary Motzkin sextic $M=x^4y^2+x^2y^4+z^6-3x^2y^2z^2$,
put $2h=\deg F$. As $\varepsilon\downarrow0$,
\[
 \log N_{F,\varepsilon}(2)\asymp_F \varepsilon^{-1/(2h)},\qquad
 N_{F,\varepsilon}(t)\asymp_{F,t}\varepsilon^{-1/[h(2-t)]}
 \quad(0\le t<2).
\]
The same endpoint estimate holds for the least truncation using the
coordinate perturbation $\sum_{i=1}^{d}\sum_{n=0}^N x_i^{2n}/(2n)!$.
\end{theorem}
\begin{proof}
For $t=2$, use the coordinate certificate of
\cite[Theorem~1]{KlepMagronSchoetz} with $\varepsilon/d$.
The difference of the perturbations is a sum of even monomials with
nonnegative coefficients. Decreasing $t$ adds squares. For $t>2$,
choose $1<\sigma<t/2$. The functional in
Lemma~\ref{avg:prop:gevrey-squares} satisfies $L(M)<0$, is
nonnegative on squares, and has $L(|x|^{2n})\le AB^{2n}((2n)!)^\sigma$.
Thus $\sup_N L(B_{N,t})<\infty$ and small $\varepsilon$ excludes every $N$.

For the degree estimate fix $0\le t\le2$, and put
$\ell_N=\log(e+N)$ if $t=2$ and $\ell_N=N^{1-t/2}$ otherwise.
For $N\ge h$ take $d\mu_N(x)=(1+x^2)^{-(h+2)}v_N(|x|)\,dx$, where
\[
 v_N(r)=
 \begin{cases}
 (1+r/(N\ell_N))^{-4N},&t=2,\\
 (1+(r/N)^{2/t})^{-4N},&0<t<2,\\
 \max(1,r/N)^{-4N},&t=0 .
 \end{cases}
\]
Write $w_N=d\mu_N/dx$. Uniformly in $N$, $-\log w_N$ is bounded on
$[-2,2]$ and $\int_{\R}(-\log w_N)/(1+x^2)\,dx\le C$.
At the endpoint split the integral at $1$ and $a=N\ell_N$; the
contribution of $-\log v_N$ is $O(N(1+\log a)/a)$. For $0<t<2$,
substitution $x=Nu$ bounds it by
$8\int_0^\infty u^{-2}\log(1+u^{2/t})\,du<\infty$; for $t=0$
replace this logarithm by $\log^+u$.

These logarithmic estimates bound every coefficient of every polynomial
$P$ of degree at most $N$ by $C\|P\|_{L^2(\mu_N)}$. Indeed, for
$z=e^{i\theta}$ off the real axis, let $P_z$ denote the Poisson density
of the half-plane containing $z$. Its integral against $-\log w_N$ is
uniformly bounded by the two estimates above, since
$P_z(x)\le C/(1+x^2)$ for $|x|>2$. The Poisson inequality for
$\log|P|$ (or factorization into linear factors) and Jensen's inequality
give
\[
 |P(z)|^2\le C\int_{\R}P_z(x)|P(x)|^2w_N(x)\,dx
       \le C|\sin\theta|^{-1}\|P\|_{L^2(\mu_N)}^2.
\]
Cauchy's coefficient formula now uses only the integrability of
$|\sin\theta|^{-1/2}$. The weights also satisfy, for $1\le m\le2N$,
\[
 \int |x|^m\,d\mu_N(x)\le C^{m+1}(m!)^{t/2}\ell_N^m .
\]
Indeed, bound $r^mv_N(r)$ by its maximum: at $t=2$ the maximizing
radius is $N\ell_N m/(4N-m)$; for $s=2/t>1$ it is
$N(m/(4Ns-m))^{1/s}$; at $t=0$ use $N^m$. In addition all moments
through $2h$ are bounded uniformly by the prefactor of $v_N$.

Both forms are nonnegative and are outside the closed cone of squares
of degree at most $2h$. For the Choi--Lam example \cite{ChoiLamForms}, AM--GM proves nonnegativity; in a
quadratic square representation the absent $x^4,y^4,z^4$ first exclude
$x^2,y^2,z^2$ from all factors. The absent $x^2w^2,y^2w^2,z^2w^2$
then exclude $xw,yw,zw$, leaving no product that produces $xyzw$.
Separate $F$ from the cone, adding a sufficiently small Gaussian
expectation to the separator to obtain $L_0(F)<0$ and
$L_0(q^2)\ge\delta\|q\|_{\mathrm{coeff}}^2$ for $\deg q\le h$,
with the Euclidean coefficient norm. Set $L_0=0$ on higher monomials.
Let $\nu_{N,S}$ be the image under $x\mapsto Sx$ of
$\mu_N^{\otimes d}$, $S\ge1$. Tensoring the coefficient bound gives
$|[x^\alpha]P|\le C^d S^{-|\alpha|}\|P\|_{L^2(\nu_{N,S})}$ for
$\deg P\le N$. If $q$ is the truncation of $P$ through degree $h$,
then $L_0((P-q)^2)=0$, while only degrees $h+1,\ldots,2h$ of $P-q$
contribute to $L_0(q(P-q))$. Hence
\[
 |L_0(q(P-q))|\le C S^{-h-1}\|q\|_{\mathrm{coeff}}\|P\|_{L^2(\nu_{N,S})}.
\]
Completing the square shows that
$L_N=L_0+D S^{-2h-2}\int\!\cdot\,d\nu_{N,S}$ is nonnegative on squares
of degree at most $2N$ for fixed large $D$. Homogeneity and the uniform
$2h$th moment imply $L_N(F)=L_0(F)+O(S^{-2})\le-c_F<0$ after fixing
$S$ independently of $N$. The product moment estimate, using
$|x|^{2n}\le d^{n-1}\sum_i |x_i|^{2n}$ and $(2n)!\le4^n(n!)^2$, gives
\[
 0\le L_N(|x|^{2n})\le C^{2n+1}(n!)^t\ell_N^{2n}\qquad(0\le n\le N).
\]
With $b=(2C\ell_N)^{-1}$ we obtain $L_N(B_{N,t}(b\,\cdot))\le 2C$.
Any certificate therefore requires $2C\varepsilon\ge c_F b^{2h}$.
Inverting $\ell_N$ proves the lower bounds; $N<h$ is likewise excluded
for small $\varepsilon$ by closedness at fixed degree.

For the converse write $R=|x|^2$. For each of the two forms $RF$ is
a sum of squares. For $M$, put
$u=xy(x^2-z^2)$, $v=xy(y^2-z^2)$. Direct expansion gives
\[
 RM=u^2+uv+v^2+[xz(y^2-z^2)]^2+[yz(x^2-z^2)]^2+(x^2y^2-z^4)^2 .
\]
For $Q$, the positive definite quadratic form
$H(a,b,c)=a^2+b^2+c^2+2(ab-ac-bc)/3$ gives
\[
 RQ=(w^3-xyz)^2+
 \sum_{\mathrm{cyc}(x,y,z)} H(xy^2-yzw,xz^2-yzw,xw^2-yzw).
\]
Also fix $K$ so that $KR^h\pm F$ are both sums of squares: $R^h$
has a positive definite diagonal Gram matrix on the degree-$h$ monomials.

Let $A(r)$ be an even real polynomial with $A(0)=1$, and set
$f=A(\sqrt R)^{2h}$. Since $R$ divides $1-f$, the first term in
\[
 F=(1-f)^2F+2fF-f^2F
\]
is a sum of squares. The last two terms become sums of squares on
adding $KR^h(2f+f^2)$, since $f$ is a square. It suffices, therefore,
to majorize $K r^{2h}(2A^{2h}+A^{4h})$ by $\varepsilon B_{N,t}(r)$
for real $r$. The nonnegative difference is an even univariate
polynomial; symmetrizing its univariate square representation writes
it as $s_0(r^2)+r^2s_1(r^2)$ with $s_0,s_1$ sums of squares.
Substitution of $R$ gives the required multivariate certificate.

For $t<2$ take $m$ odd, $N=2hm$, $D=Cm^{t/2}$, and set
$A(r)=(-1)^{(m-1)/2}T_m(r/D)/(mr/D)$ with the Chebyshev polynomial
$T_m$. On $|r|\le D$, $|A(r)|\le\min(1,D/(m|r|))$, since
$|\sin(m\theta)|\le m|\sin\theta|$. For $|r|>D$ use
$|T_m(r/D)|\le (2|r|/D)^m$. For $j=1,2$ this yields
\[
 |r|^{2h}|A(r)|^{2hj}\le (D/m)^{2h}
 \begin{cases}1,& |r|\le D,\\
 (2|r|/D)^{2mhj},& |r|>D.\end{cases}
\]
If $C\ge 2(2h)^{t/2}$, the latter power is bounded by the $n=mhj$
term of $B_{N,t}$, since $n!\le n^n$. Hence
$\varepsilon=C_{F,t}m^{-h(2-t)}$ suffices.

At $t=2$ put $J_m=\sum_{k=1}^m k^{-1}$,
$G_m(r)=\prod_{k=1}^m(1+i\delta r/k)$, and
$A(r)=\operatorname{Im}G_m(r)/(\delta rJ_m)$.
The bound $|\sum_k\arctan(\delta r/k)|\le \delta |r|J_m$ gives
$|A(r)|\le |G_m(r)|\min(1,(\delta |r|J_m)^{-1})$. Thus the last
two powers cost at most $(\delta J_m)^{-2h}|G_m(r)|^{4h}$ each.
Euler's product bounds the coefficients of $|G_m(r)|^{4h}$ by those of
$[\sinh(\pi\delta r)/(\pi\delta r)]^{2h}$, and its coefficient of
$r^{2n}$ is at most $(2h\pi\delta)^{2n}/(2n)!$ by the multinomial
formula. For $2h\pi\delta\le1$ the polynomial $|G_m|^{4h}$ is
therefore bounded by $B_{2hm,2}$. This yields the certificate for
$\varepsilon=C_F J_m^{-2h}$ and proves the endpoint upper bound.

For the coordinate assertion write $\Theta_N=\sum_{i=1}^d\sum_{n=0}^N x_i^{2n}/(2n)!$.
The difference $dB_{N,2}-\Theta_N$ is a sum of squares coefficientwise.
In the other direction $\Theta_N(x)-d B_{N,2}(x/(2\sqrt d))$ is also a
sum of squares: use $(2n)!\le4^n(n!)^2$ and
$\sum_i x_i^{2n}-d(R/d)^n\ge0$ in the square cone. For the last fact
sum the Taylor remainders of $u^n$ at $v=R/d$ with $u=x_i^2$;
the linear terms cancel and each remainder is $(u-v)^2$ times a
polynomial with nonnegative coefficients in $u,v$ ($n\ge2$).
Homogeneity absorbs the fixed change of scale.
\end{proof}

\begin{corollary}[Degree cost for forms with a radial square multiplier]
\label{kms:cor:radial-multiplier}
For every homogeneous nonnegative form $F$ of degree $2h$ outside the
sum-of-squares cone, the lower bounds in Theorem~\ref{kms:thm:critical}
hold, with constants depending on $F,t$. If $R^k F$ is a sum of squares
for some integer $k\ge1$, both two-sided estimates hold, including the
coordinate endpoint. In particular they hold for every positive
definite form outside the square cone, by \cite{ReznickUniformDenominators}.
\end{corollary}
\begin{proof}
The lower-bound argument used only separation of $F$ at degree $2h$ and
homogeneity. For the upper bound let $e=2\lceil k/2\rceil$.
Since $R\mid 1-f$ and $R^{e-k}$ is a sum of squares, $(1-f)^eF$ is
a sum of squares. Expand $1-(1-f)^e=\sum_{j=1}^e c_j f^j$.
Each $c_j f^j F+|c_j|K R^h f^j$ is a sum of squares, so it suffices to
bound $|r|^{2h}\sum_{j=1}^e |c_j| |A(r)|^{2hj}$.
The Chebyshev proof works for $1\le j\le e$ with $N=ehm$ and
$C\ge2(eh)^{t/2}$. At $t=2$ take $eh\pi\delta\le1$ and majorize all
powers by $(\delta J_m)^{-2h}|G_m|^{2eh}$, again of radial order at most
$ehm$. The binomial coefficients depend only on the fixed multiplier
exponent. The coordinate comparison is unchanged.
\end{proof}
For the coordinate weights $1/((2n)!)^s$, the universal critical order
is $s=1$: the coordinate theorem gives sufficiency, and for $s>1$
the separator above with $1<\sigma<s$ excludes all truncations.

\begin{corollary}[Degree bounds for every nonnegative polynomial]
\label{kms:cor:all39}
For every fixed nonnegative polynomial $P$, as $\varepsilon\downarrow0$,
\[
 N_{P,\varepsilon}(t)=O_{P,t}(\varepsilon^{-1/(2-t)})\quad(0\le t<2),
 \qquad \log(1+N_{P,\varepsilon}(2))=O_P(\varepsilon^{-1/2}).
\]
\end{corollary}
\begin{proof}
A nonnegative constant is already a square. Write $2h=\deg P\ge2$, $R=|x|^2$. For fixed
small $a>0$, $B_{N,t}/2-a(1+R)^h$ has nonnegative coefficients in $R$
for $N\ge h$. Set $P_\varepsilon=P+a\varepsilon(1+R)^h$. The spherical
square bound of Fang and Fawzi \cite[Theorem~2 and Lemma~13]{FangFawzi}
gives $(1+R)^kP_\varepsilon$ as a sum of squares for
$k=\lceil C_P\varepsilon^{-1/2}\rceil$. Indeed apply their theorem to
the normalized nonnegative homogenization of $P$ on the sphere, padding
the dimension to at least $h$. At this order the required constant shift
is at most $a\varepsilon$. Homogenize the spherical squares as in their
Lemma~13, set the dummy variables to zero, and dehomogenize. Increasing
$k$ preserves the certificate. Fix $K$ such that $K(1+R)^h\pm P_\varepsilon$
are sums of squares for $0<\varepsilon\le1$, using the positive definite
Gram matrix on all monomials of degree at most $h$.

Use the even cutoff $A(r)$ of the preceding proof, and set
$E=A(i)>1$, $f=(A(\sqrt R)/E)^{2h}$. Then $1+R$ divides $1-f$.
The truncated binomial inverse gives
\[
 q_k(s)=(1-s)^k\sum_{j=0}^{k-1}\binom{k+j-1}{j}s^j,\qquad
 1-q_k(s)^2=\sum_{j=k}^{4k-2}c_j s^j.
\]
The sum starts at $k$ since $q_k(s)=1+O(s^k)$, and
$\sum |c_j|\le1+(2^k\binom{2k-1}{k-1})^2\le65^k$.
Now $q_k(f)^2P_\varepsilon$ is a sum of squares by divisibility.
Each remaining term $c_j f^jP_\varepsilon$ is a sum of squares after
adding $K|c_j|(1+R)^h f^j$. It suffices to ensure that $65E^{-2h}\le1/2$
and, uniformly for $1\le j\le4k$,
\[
                 (1+r^2)^h|A(r)|^{2hj}\le C_h B_{N,t}(r).
\]
Indeed the total correction is then bounded by $KC_h2^{-k} B_{N,t}
\le\varepsilon B_{N,t}/2$. As in the preceding proof, the nonnegative
even univariate remainder transfers to a multivariate sum of squares.
This certifies $P_\varepsilon+\varepsilon B_{N,t}/2$, and adding the
reserved square sum certifies $P+\varepsilon B_{N,t}$.

For $t<2$ take odd $m$, $N=4hkm$, and Chebyshev scale $D=2N^{t/2}$.
For $|r|\le D$ use $|A|\le\min(1,D/(m|r|))$ with $m/D\ge1$.
For $|r|>D$ the preceding estimate with $n=mhj\le N$ applies, since
$D\ge2n^{t/2}$. These bounds give the displayed inequality. Also
$E=\sinh(m\,\mathrm{arsinh}(1/D))/(m/D)$.
Choose the least odd $m$ with $m/D\ge C_h'$ for a sufficiently
large fixed constant. This gives the required $E$ and
$N=O_{h,t}(k^{2/(2-t)})$.

At $t=2$ take $\delta=(8\pi hk)^{-1}$, $N=4hkm$ in the Euler cutoff.
Then $|G_m|^{8hk}\le B_{N,2}$ coefficientwise, and $\delta J_m\ge1$
gives the displayed inequality. At $r=i$,
\[
 E=\frac{\prod_{j\le m}(1+\delta/j)-\prod_{j\le m}(1-\delta/j)}
          {2\delta J_m}.
\]
The logarithms of the products are $\pm\delta J_m+O(\delta^2)$ uniformly.
Choosing $m=\lceil\exp(C_h' k)\rceil$ with a large fixed $C_h'$ makes
$E$ sufficiently large and gives $\log N=O_h(k)$.
\end{proof}

\subsection{Gaussian upper bounds at an unattained infimum}
Klep, Magron and Sch\"otz asked for quantitative upper bounds in
Gaussian polynomial optimization \cite[concluding discussion]{KlepMagronSchoetz}.
Gribling, Klep and Kumar obtained order $1/r$ at an attained minimum
and asked which weaker hypotheses guarantee a prescribed rate
\cite[Theorem~1 and Sec.~6]{r67:GKK}.
Let $\gamma_d$ be standard Gaussian measure and put
\[
 \lambda_r(f)=\min_{\deg q\le r,\ \int q^2\,d\gamma_d=1}
                      \int f q^2\,d\gamma_d,\qquad
 e_f(R)=\min_{|x|\le R} f(x)-f_*,\qquad f_*=\inf_{\mathbb R^d} f .
\]
The same minimum results on allowing normalized sums of squares,
since each gives a convex combination of single-square values.

\begin{theorem}[The exact rate of the Gaussian hierarchy]\label{r67:gaussianrate}
Let $g=f-f_*$ for a nonconstant polynomial $f$ of degree $D$ on
$\mathbb R^d$ with finite infimum $f_*$. Define
\[
 M_f(t)=\min_{|a|\le1/t}\ \max_{|h|\le t}g(a+h),\qquad t>0 .
\]
There exist $c,b,C>0$ depending only on $d,D$ such that, with
$t=r^{-1/2}$ and $r$ sufficiently large,
\[
 c M_f(bt)\le\lambda_r(f)-f_*\le C M_f(t).
\]
In particular there are $u_f>0$ and a positive rational $\beta_f$ with
$M_f(t)\sim u_f t^{2\beta_f}$ as $t\downarrow0$, and
$\lambda_r(f)-f_*=\Theta_f(r^{-\beta_f})$. For fixed $d,D$ only finitely many
exponents occur, and each exponent defines a semialgebraic set of coefficients.
In the definition of $M_f(t)$ one may equivalently minimize, up to
constants depending only on $d,D$, the local curvature expression
\[
                g(a)+t^2\max_{|h|\le t}\|D^2 f(a+h)\|.
\]
For constant $f$ the hierarchy is exact at every order.
\end{theorem}
\begin{proof}
Use the smooth Hermite kernel $\Lambda_n$ of Petrushev and Xu
\cite[Corollary~2.3 and Theorem~2.4]{r68:PX}, with a real cutoff
equal to one on $[0,3/2]$ and supported in $[0,2]$. Take
$n=\lfloor r/2\rfloor$. In their Hermite-function normalization
$\Lambda_n(a/\sqrt2,x/\sqrt2)$ is $e^{-|x|^2/4}$ times a polynomial
in $x$ of total degree at most $2n$. Their two estimates give
\[
 |\Lambda_n(a/\sqrt2,x/\sqrt2)|\le C_k r^{d/2}
           (1+\sqrt r\,|x-a|)^{-k},\qquad
 \int |\Lambda_n(a/\sqrt2,x/\sqrt2)|^2\,dx\ge c_d r^{d/2}
\]
for every fixed $k$ and all $|a|\le\sqrt r$ at large $r$.
Indeed the lower estimate applies to
$|a|/\sqrt2\le(3/4)\sqrt{3n}$. Its normalized square is therefore
an admissible density $q^2\,d\gamma_d/dx$, with
$\mathbb E|X-a|^j\le C_{j,d}t^j$ for $j\le D$.
Every coefficient of $h\mapsto g(a+th)$ is bounded by
$C_{d,D}\max_{|h|\le1}g(a+th)$, by equivalence of norms on the
finite-dimensional polynomial space. Integration proves the upper bound.

For the lower bound, any normalized trial polynomial $q$ of degree at
most $r$ gives $\psi=(2\pi)^{-d/4}q e^{-|x|^2/4}$ with
\[
 \|\nabla\psi\|_2^2+\tfrac14\||x|\psi\|_2^2
       \le r+d/2 .
\]
Indeed the Hermite functions of total order $j$ are eigenfunctions
of $-\Delta+|x|^2/4$ of eigenvalue $j+d/2$ in these coordinates.
Partition space into cubes of side $\epsilon t$, and let $v$ equal the
average of $\psi$ on each cube. The cube Poincar\'e inequality,
summed over the partition, gives $\|\psi-v\|_2\le C_d\epsilon$ for
$r\ge d$. The union $U$ of cubes meeting $|x|\le4/t$ has
$\|\psi\mathbf1_U\|_2\ge1/2$, by the preceding quadratic moment bound.

In each cube $Q\subset U$, the set $F_Q=\{g\ge(\max_Q g)/2\}\cap Q$
occupies at least a fraction $\theta_{d,D}>0$. After rescaling to
the unit cube and dividing by the maximum, the gradient is bounded
uniformly by finite-dimensional norm equivalence; a fixed
neighborhood of any maximum therefore suffices. Put $F=\bigcup F_Q$.
Since $v$ is constant on every cube,
\[
 \|\psi\mathbf1_F\|_2\ge
 \sqrt{\theta_{d,D}}\,(1/2-C_d\epsilon)-C_d\epsilon .
\]
Fix $\epsilon>0$ small enough to make this positive. At large $r$ every
center in $U$ has norm at most $5/t$, and its cube contains the ball
of radius $\epsilon t/2$ about that center. With
$b=\min(1/5,\epsilon/2)$, every $F_Q$ thus has $g\ge M_f(bt)/2$.
Integration against $|\psi|^2$ and minimization prove the lower bound.

The positive function $M_f$ is semialgebraic by quantifier elimination
and tends to zero by testing at fixed approximate minimizers. Its
Puiseux leading term gives the stated exact rate. For uniformity,
eliminate the quantifiers defining its graph with the coefficients
and $f_*$ as parameters \cite{r59:BPR}. The resulting sign
polynomials in $(t,y)$ have both degrees bounded in terms of $d,D$.
After fixing parameters discard the identically zero polynomials.
At every graph point one remaining polynomial must vanish, since
otherwise a neighborhood would have the same signs. By
one-dimensional semialgebraicity one polynomial vanishes on the
graph for all small $t$. At least two of its monomials $t^iy^j$
have the same lowest exponent $i+2j\beta_f$, giving finitely many
possible positive ratios for $2\beta_f$. For each such ratio $p/q$,
the condition that $M_f(t)^q/t^p$ stays between two positive
constants for all sufficiently small $t$ is semialgebraic in the
coefficients, again by quantifier elimination (including $f_*$).
Finally, polynomial norm equivalence on the unit ball bounds the
scaled Hessian by the local maximum. In the other direction take
one step of length $t$ opposite the gradient, using $g\ge0$.
Taylor's formula bounds $t|\nabla f(a)|$ by $g(a)$ plus half the
displayed curvature term. Taylor's formula on the ball gives the
claimed equivalence.
\end{proof}

Thus a prescribed rate $O(r^{-\beta})$ holds precisely when
$M_f(t)=O(t^{2\beta})$. In terms of the earlier ball minima alone,
$e_f(R)=O(R^{-k})$ gives
$\beta_f\ge\min\{k/2,k/(D-2+k)\}$: take
$R=r^{\min(1/2,1/(D-2+k))}$ and bound the Hessian on the unit
neighborhood of the ball by $C_f(1+R)^{D-2}$.
For a nonconstant univariate polynomial let $2m$ be the largest
multiplicity of a real zero of $f-f_*$. Its coercivity and the
local factorization at the finitely many minima give
$M_f(t)\asymp_f t^{2m}$: center the interval at a flattest minimum
for the upper bound; any interval near one minimum contains a
point at distance at least $t$ from it. Away from the minima $g$
is bounded below. Hence the exact univariate rate is $r^{-m}$.

For $f_m(x,y)=x^2+(1+y^2)^m(xy-1)^2$ from
\cite[Example~4]{r67:GKK}, $m\ge0$, one has $e_{f_m}(R)\asymp R^{-2}$.
Take $y=R/2$, $x=1/y$ for the upper bound. For the reverse inequality
either $|xy|\ge1/2$ and $x^2\ge1/(4R^2)$, or $(xy-1)^2\ge1/4$.
The theorem recovers the upper exponent in their two-sided estimate
$\lambda_r(f_m)=\Theta(r^{-1/(m+2)})$. The squared objective
$f_0^2$ exhibits the prescribed $O(1/r)$ rate without attaining its
infimum. Indeed $M_{f_0^2}=M_{f_0}^2$ since $f_0\ge0$ and $\inf f_0=0$,
so Theorem~\ref{r67:gaussianrate} gives $\lambda_r(f_0^2)=\Theta(r^{-1})$.

\subsection{The multivariate rate on convex bodies}
For a finite positive measure $\mu$ with compact support $K$, let
$\lambda_r^\mu(f)$ be the minimum of $\int_K fq^2\,d\mu$ over multivariate
polynomials $q$ of degree at most $r$ with $\int_K q^2\,d\mu=1$.
We omit $\mu$ for Lebesgue measure. The push-forward value
$\lambda_r^{\rm pf}(f)$ uses $q=p\circ f$ with $\deg p\le r$.
Thus for nonconstant $f$ the corresponding multivariate degree is at
most $r\deg f$.

Slot and Laurent proved the $O(r^{-2}\log^2 r)$ multivariate bound on
every convex body and asked in Section~6 whether $O(r^{-2})$ holds both
there and on general compact sets under their lower volume condition
\cite[Theorem~11 and Sec.~6]{r69:SL}.
Theorem~\ref{r69:convexrate} gives the convex rate using a quadratic
upper estimator at the minimizer. Theorem~\ref{r69:compactlower} gives
a Lipschitz domain and a linear objective attaining the logarithmic loss
along a subsequence.

\begin{lemma}[Concentration on doubling sublevel sets]\label{r69:sublevel}
Let $H$ be a polynomial with $0\le H\le1$ on $K$. Suppose
$V(t)=\mu\{H\le t\}>0$ and $V(2t)\le 2^\gamma V(t)$ for all $0<t<1$,
where $\gamma\ge0$. If $0\le f-f_*\le A H^\alpha$ on $K$ with $\alpha>0$,
then $\lambda_r^\mu(f)-f_*=O(r^{-2\alpha})$.
\end{lemma}
\begin{proof}
For $u=\sin^2 v\in[0,1]$, $0\le v\le\pi/2$, put
$F_N(u)=(\sin(Nv)/(N\sin v))^2$, taking $F_N(0)=1$.
Its cosine expansion in $2v$ makes it a polynomial of degree $N-1$ in $u$.
It is bounded by $\min(1,1/(N^2u))$ and is at least $4/\pi^2$ on
$[0,\delta]$, where $\delta=(16N^2)^{-1}$, by the sine inequalities on
$[0,\pi/2]$. Fix an integer $k$ with $2k>\gamma+\alpha$ and use the
trial polynomial $q=F_N(H)^k$. Then $\int q^2\,d\mu\ge
(4/\pi^2)^{2k}V(\delta)$. On $2^j\delta<H\le2^{j+1}\delta$ one has
$H^\alpha q^2\le(2^{j+1}\delta)^\alpha(16\,2^{-j})^{2k}$ and the
mass is at most $2^{(j+1)\gamma}V(\delta)$. Summing for $j\ge0$ and
including $H\le\delta$ gives $\int H^\alpha q^2\,d\mu\le C\delta^\alpha
V(\delta)$. Since $\deg q\le k(N-1)\deg H$, one can take $N\asymp r$.
\end{proof}

\begin{theorem}[Sharp rate on every convex body]\label{r69:convexrate}
On every compact convex body $K\subset\mathbb R^d$ of nonempty interior,
with Lebesgue measure, every polynomial $f$ satisfies
$\lambda_r(f)-\min_K f=O_f(r^{-2})$. Every nonconstant affine $f$ has
error $\Theta_f(r^{-2})$. Both statements persist when the Lebesgue
density is bounded above and bounded away from zero on $K$.
\end{theorem}
\begin{proof}
Take a minimizer $a$ and $L>\max_K\|D^2 f\|$. Convexity gives
$\nabla f(a)\cdot(x-a)\ge0$, so Taylor's theorem gives
$0\le f(x)-f(a)\le h(x)=\nabla f(a)\cdot(x-a)+L|x-a|^2$ on $K$.
Set $H=h/\max_K h$. The map $x\mapsto(a+x)/2$ sends
$K\cap\{H\le2t\}$ into $K\cap\{H\le t\}$ by convexity of $H$, so
$V(2t)\le2^dV(t)$ and $V(t)>0$ by iteration. Apply Lemma~\ref{r69:sublevel}
with $\alpha=1$.

For the lower estimate fix $y\in\operatorname{int}K$. On each segment
$x=(1-s)y+sz$, $0\le s\le1$, ending at the boundary in a given radial
direction, an affine $f$ satisfies $f(x)-\min_K f\ge(1-s)(f(y)-\min_K f)$.
For any polynomial $p$ of degree at most $r$, the orthonormal Legendre
expansion on $[1/2,1]$ and $|P_j|\le1$ give
\[
 \max_{[1/2,1]}|p|^2\le2(r+1)^2\int_{1/2}^1p^2\,ds
 \le 2^d(r+1)^2\int_0^1p^2s^{d-1}\,ds.
\]
Thus at most half this last weighted mass lies in $[1-\epsilon,1]$ for
$\epsilon=2^{-d-1}(r+1)^{-2}$. The weighted average of $1-s$ against
$p^2$ is at least $\epsilon/2$. Apply this to the restriction of any
trial polynomial to each ray and integrate in polar coordinates about
$y$, where the radial Jacobian is a direction-dependent constant times
$s^{d-1}$. Since $f(y)>\min_K f$ for a nonconstant affine $f$, the lower
bound follows. Comparing the two nonnegative integrals in the Rayleigh
quotient transfers both bounds to the stated densities.
\end{proof}

\subsection{A compact domain attaining the logarithmic loss}
For univariate trials, Slot and Laurent proved $O(r^{-2}\log^2 r)$
on connected compact sets with a polynomial lower volume density and
asked whether the logarithmic factor is necessary
\cite[Theorems~1 and~4, Sec.~6]{r69:SLpush}.
For the linear objective below, both trial spaces are indexed by the
same degree. The construction gives matching lower bounds for both.

\begin{theorem}[A sharp logarithmic loss for both hierarchies]
\label{r69:compactlower}\label{r69:pushlower}
There is a compact $K\subset\mathbb R^2$, the closure of a bounded
Lipschitz domain, satisfying
$|K\cap B(z,\delta)|\ge c\delta^2$ uniformly for $z\in K$ and
$0<\delta<1/8$, and a linear $f$ with $\min_K f=0$, for which
\[
 0<\limsup_{r\to\infty}\frac{r^2\lambda_r(f)}{\log^2 r}
 \le\limsup_{r\to\infty}\frac{r^2\lambda_r^{\rm pf}(f)}{\log^2 r}
 <\infty .
\]
\end{theorem}
\begin{proof}
Set $t_j=2^{-8^j}$, $j\ge1$, and
$b(x)=\max\{t_j^{1/4}:t_j\le x\}$ for $0<x\le1$, $b(0)=0$. Take
\[
 K=\{(x,y):0\le x\le1,\ |y|\le x+b(x)\},\qquad f(x,y)=x .
\]
Put $h(x)=x+b(x)$. Upper semicontinuity gives
$G(v)=\min\{x:h(x)\ge v\}$ for $0\le v\le h(1)$, and
$h(x+s)\ge h(x)+s$ makes $G$ nondecreasing and $1$-Lipschitz.
The interior is $G(|y|)<x<1$, $|y|<h(1)$, an open connected set with
closure $K$, since each section at fixed $x$ joins the horizontal axis
and the two boundary graphs are limits of interior points. Its boundary consists of the Lipschitz graph
$x=G(|y|)$ and the segment $x=1$, meeting at two convex corners since
$h$ is affine near $1$, so the domain is Lipschitz. At
$x\le3/4$ all increments $(h,v)$ with $0\le h\le1/8$, $|v|\le h$
stay in $K$. The part with $x\ge3/4$ belongs to the fixed solid
trapezoid $K\cap\{x\ge1/2\}$, which has a uniform lower volume density.
This proves the geometric assertion.

Fix $t=t_j$, $A=t^{1/4}$. Since $t_{j+1}^{1/4}=t^2$, the part of $K$
with $x<t$ lies in $[0,t]\times[-2t,2t]$. Meanwhile
$R=[t,1]\times[-A,A]\subset K$. For each real polynomial $q$ of
degree at most $r$, the univariate Chebyshev exterior bound on $[t,1]$
and the endpoint Legendre bound give, for $0\le x<t$,
\[
 |q(x,y)|^2\le C(r+1)^2 e^{Cr\sqrt t}\int_t^1 |q(s,y)|^2\,ds .
\]
Indeed the affine image of $[0,t]$ lies between $-(1+t)/(1-t)$ and
$-1$, and $\operatorname{arcosh}((1+t)/(1-t))=O(\sqrt t)$; the
Legendre expansion on $[t,1]$ bounds the squared supremum there by
$(r+1)^2/(1-t)$ times the squared $L^2$ norm.

In the transverse direction, for $|y|\le A/2$,
\[
 |q(s,y)|^2\le C(r+1)A^{-1}\int_{-A}^A |q(s,v)|^2\,dv .
\]
To see this, apply Cauchy--Schwarz to the Fourier coefficients of the
trigonometric polynomial $\sin\theta\,q(s,A\cos\theta)$ of degree
at most $r+1$. Its squared norm over the circle is at most
$2A^{-1}\int_{-A}^A|q(s,v)|^2\,dv$, and $\sin^2\theta\ge3/4$ on the
central half of the interval. Integrating the two bounds gives
\[
 \int_{K\cap\{x<t\}}q^2\le C\frac{t^2(r+1)^3}{A}e^{Cr\sqrt t}
                              \int_R q^2 .
\]
Take $r=\lfloor c_0t^{-1/2}\log(1/t)\rfloor$ for a sufficiently
small fixed $c_0>0$. The multiplier is
$O(t^{1/4-Cc_0}\log^3(1/t))=o(1)$, hence the normalized average of
$x$ is at least $t/2$ for every trial polynomial and large $j$.
For these orders $t\asymp r^{-2}\log^2r$. The original push-forward
upper bound applies to this $K$, and $p(x)$ is an admissible
bivariate trial polynomial of the same degree.
\end{proof}

\begin{corollary}[Subanalytic domains]\label{r69:tamedomain}
Let $K\subset\mathbb R^d$ be compact and globally subanalytic, equal to
the support of its restricted Lebesgue measure. For every polynomial
$f$ the push-forward hierarchy, and hence the multivariate hierarchy,
has error $O_f(r^{-2})$.
\end{corollary}
\begin{proof}
For nonconstant $f$ on $K$, set
$H=(f-\min_K f)/(\max_K f-\min_K f)$. By stability under parametric
integration \cite[Theorem~1.3]{r69:CM}, $V(t)=|K\cap\{H\le t\}|$ is
constructible, a finite sum of products of globally subanalytic functions
and logarithms of positive such functions. Each one-variable subanalytic
germ has a convergent Laurent--Puiseux expansion. Factoring its leading
power before taking a logarithm expresses $V(t)$ as a polynomial in
$\log t$ with Laurent--Puiseux coefficients. Since $V(t)>0$ for $t>0$,
it has leading term $c t^a|\log t|^j$ with $c>0$. Thus $V(2t)/V(t)$
is bounded at zero, and monotonicity bounds it away from zero. The trial
polynomials in Lemma~\ref{r69:sublevel} are polynomials in $f$ itself.
\end{proof}

\subsection{Boundary contact for the push-forward hierarchy}
Laurent asked whether the logarithmic factor can be removed from the
push-forward bound on every convex body \cite[PDF p.~120]{r70:Laurent}.
The interval example $f(x)=x^{2k}$ of Slot and Laurent ($k>1$) has
push-forward error $\Theta(r^{-2})$ and multivariate error
$O(r^{-2k}\log^{2k}r)$ \cite[Theorem~3 and its proof]{r69:SLpush}.
Theorem~\ref{r70:convexpush} answers Laurent's question with a fixed
quadratic on a uniformly convex body. Theorem~\ref{r73:contact}
identifies the boundary area that decides the quadratic rate at
regular minima. Corollary~\ref{r73:smoothlimit} distinguishes the
rate on each fixed smooth planar body from a rate valid throughout
the smooth class. For a Gaussian measure on the whole real line,
Slot and Wiedmer also constructed a degree-six objective whose
push-forward bounds fail to converge \cite[Theorem~7]{r71:SW}.

The smallest Gaussian quadrature node of the objective's distribution
equals the push-forward bound \cite[Sec.~1.3]{r69:SLpush}. Global and
local converses for doubling use zero spacings and neighboring
quadrature weights \cite{r71:MT,r71:TV}. Under the endpoint mass
hypothesis below, the first node alone tests doubling at zero.

\begin{theorem}[An endpoint criterion and its radial form]\label{r71:edgedoubling}
For a finite positive measure $\nu$ on $[0,1]$ put $V(t)=\nu([0,t])$
and $z_r=\min_{0\ne p,\,\deg p\le r}\int s p^2\,d\nu/\int p^2\,d\nu$.
Assume $V(t)>0$ near zero. Suppose for some $b,c>0$ and all small $t>0$,
\[
  \nu|_{[t,b]}\ \ge\ c\,V(t)t^{-1}\,ds|_{[t,b]} .
\]
Then $z_r\gtrsim r^{-2}$, and $z_r=O(r^{-2})$ if and only if
$V(2t)=O(V(t))$ at zero. The hypothesis allows singular parts.
It holds in particular for densities comparable near zero to a positive
nondecreasing function $m$, in which case doubling is equivalent to
$m(2t)=O(m(t))$.

Let $K$ be a compact set in the closed unit ball, star-shaped about the
origin, whose radial function $R$ has essential supremum one. For Lebesgue
measure and $f(x)=1-|x|^2$, the push-forward error is $\Theta(r^{-2})$
precisely when $A(2t)=O(A(t))$ at zero, where
$A(t)=\sigma\{\omega:R(\omega)^2\ge1-t\}$ and $\sigma$ is surface measure.
\end{theorem}
\begin{proof}
Doubling implies the upper bound by Lemma~\ref{r69:sublevel} with
$H(s)=s$. For $0<t<b/4$ put
$\ell(2t)=\sup\{a:\nu|_{[2t,b]}\ge a\,ds|_{[2t,b]}\}$.
The Chebyshev exterior bound from Theorem~\ref{r69:pushlower} amplifies
the supremum of $|p|$ on $[2t,b]$ by at most $e^{C_b r\sqrt t}$ on
$[0,t]$. The squared supremum on $[2t,b]$ is at most
$(r+1)^2/(b-2t)$ times the squared Lebesgue norm there by the Legendre
bound. Thus
\begin{equation}\label{r71:edgecompare}
 \int_{[0,t]}p^2\,d\nu\le C_b(r+1)^2 e^{C_br\sqrt t}
     \frac{V(t)}{\ell(2t)}\int_{[2t,b]}p^2\,d\nu .
\end{equation}
By hypothesis $V(t)/\ell(2t)\lesssim t/J$, where $J=V(2t)/V(t)\ge1$.
At $t=c_0/(r+1)^2$ the multiplier is at most one for small $c_0$,
so $z_r\ge t/2$. If $J\to\infty$ along $t\downarrow0$, take
$r=\lfloor c_1\log J/\sqrt t\rfloor$ with small fixed $c_1>0$.
The multiplier is $O((\log J)^2 J^{-1/2})$, hence
$r^2z_r\gtrsim(\log J)^2$. For $w\asymp m$ nondecreasing up to
comparison, $V(t)\lesssim tm(t)\lesssim t\ell(t)$.
Doubling of $m$ gives doubling of $V$ by integration; conversely
$V(4t)\gtrsim t m(2t)$ gives doubling of $m$ from doubling of $V$.
For the radial objective, polar integration gives density
$\tfrac12(1-s)^{d/2-1}A(s)$, comparable near zero to $A(s)$.
\end{proof}

\begin{theorem}[Boundary contact and the two optimization rates]\label{r73:contact}
Let $K\subset\mathbb R^d$ be a compact convex body and $f$ a polynomial,
with $\nabla f\ne0$ at every minimizer on $K$. Write $f_*=\min_K f$ and
\[
 A(t)=\mathcal H^{d-1}\{z\in\partial K:f(z)\le f_*+t\}.
\]
For Lebesgue measure the multivariate error is $\Theta(r^{-2})$. The
push-forward error has this same order if and only if $A(2t)=O(A(t))$
as $t\downarrow0$. The push-forward assertion also holds for the
closure of a bounded $C^1$ domain under the same gradient condition.
\end{theorem}
\begin{proof}
Subtract $f_*$ and normalize the range to $[0,1]$. Fix
$\overline B(b,h)\subset\operatorname{int}K$. From $b$ each ray
$\omega$ meets the boundary at $z=b+R(\omega)\omega$. This map is
bi-Lipschitz: the gauge of $K-b$ is subadditive and at most $|x|/h$,
so $1/R$ is Lipschitz on the sphere; the inverse is radial normalization
at distance at least $h$. Thus the angular measure of $\{f(z)\le s\}$
is comparable to $A(s)$.

Use inward depth $u$ on the segment $x=(1-u)z+ub$, $0\le u\le1$.
Every minimizer $z$ is on the boundary and satisfies
$\nabla f(z)\cdot(b-z)\ge h|\nabla f(z)|$ by the first-order condition
on the ball. Near the compact set of minimizing endpoints and for
$0\le u\le u_0$ the derivative $\partial_u f(x)$ is therefore bounded
above and below by positive constants. The complement of this
neighborhood in the ray parameters is compact and has positive minimum.
For small $s$ the level $f(x)=s$ is consequently reached once on
precisely the rays with $f(z)\le s$. Changing from $u$ to $s$ in
the polar volume $R^d(1-u)^{d-1}\,du\,d\omega$ gives a push-forward
density comparable to $A(s)$. Apply Theorem~\ref{r71:edgedoubling}.
The same derivative bound and positive minimum give $f(x)\ge c u$
on all segments. The one-variable Legendre estimate on rays in
Theorem~\ref{r69:convexrate} now supplies the multivariate lower bound;
that theorem gives the upper bound.

For a $C^1$ boundary the field $\nabla f/|\nabla f|^2$ points strictly
inward at every minimizer. Its time $u$ flow from $z\in\partial K$ has
value $f(z)+u$ and volume Jacobian bounded above and below in transverse
boundary charts. A finite cover of the minimizing set by such charts,
shrunk inside larger transverse charts, covers all sufficiently low
levels. Summing over the cover with bounded multiplicity again gives a
density comparable to $A(s)$.
\end{proof}

\begin{theorem}[Push-forward bounds on convex bodies]\label{r70:convexpush}
On a convex body with Lebesgue measure each convex polynomial objective
has push-forward error $O(r^{-2})$. For the fixed quadratic
$f(x,y)=1-x^2-y^2$, and every integer $m\ge2$, there is a uniformly
convex planar body with $C^m$ boundary, $\min_K f=0$, and
\[
 \lambda_r(f)=\Theta(r^{-2}),\qquad
 0<\limsup_{r\to\infty}\frac{r^2\lambda_r^{\rm pf}(f)}{\log^2r}<\infty .
\]
There is also a uniformly convex body with $C^\infty$ boundary on which
$r^2\lambda_r^{\rm pf}(f)\ge c(\log r/\log\log r)^2$ along a sequence.
More generally, for every positive sequence $u_r=o(r^{-2}\log^2 r)$
there is such a smooth body with $\limsup_r\lambda_r^{\rm pf}(f)/u_r=\infty$.
\end{theorem}
\begin{proof}
For a nonconstant convex polynomial, take $H=(f-\min_K f)/(\max_K f-
\min_K f)$ in Lemma~\ref{r69:sublevel}. Contraction by one half toward
a minimizer proves doubling, and the trial polynomial depends only on $f$.
A constant objective is exact.

For the examples let $a_j\ge8a_{j-1}$, $t_j=e^{-a_j}$ and
$\rho_j\downarrow0$. Fix a smooth nondecreasing cutoff $\phi$ on
$[0,\infty)$ that is zero on $[0,1]$ and one on $[2,\infty)$. Put
\[
 G(y)=\varepsilon\sum_{j\ge1}t_j\phi(|y|/\rho_j),\qquad
 K=\{(x,y):x^2+y^2+G(y)\le1\}.
\]
When $\sum_j t_j\rho_j^{-k}<\infty$ for $0\le k\le m$, the series
converges in $C^m$. Choose $\varepsilon>0$ so that $\|G''\|_\infty\le1$.
The defining function has Hessian at least $I$ and a unique critical point
at its interior minimum. Its level one is a regular $C^m$ curve of
curvature bounded away from zero. The points $(\pm1,0)$ minimize $f$.
Since $\nabla f$ is nonzero at both minimizers,
Theorem~\ref{r73:contact} gives the multivariate rate for all the examples.

In coordinates $s=f(x,y)$ and $y$, the two branches together have
Jacobian $(1-s-y^2)^{-1/2}$ and the membership condition is $G(y)\le s$.
With $t=t_j$, the range $0\le s<\varepsilon t$ requires $|y|<2\rho_j$.
For $2\varepsilon t\le s\le1/2$, every $|y|\le\rho_{j-1}$ is allowed
at large $j$, since $\sum_{i\ge j}t_i\le2t_j$. Thus the push-forward
density $w(s)$ is at most $C\rho_j$ on the first interval and at least
$c\rho_{j-1}$ on the second. Put $\ell_j=\log(\rho_{j-1}/\rho_j)$.
When $\ell_j\to\infty$, take $r_j=\lfloor c_0\ell_j/\sqrt{t_j}\rfloor$
with fixed small $c_0>0$. Applying \eqref{r71:edgecompare} at
$\varepsilon t_j$ with $b=1/2$ bounds the mass below $\varepsilon t_j$
by $O(\ell_j^2e^{-\ell_j/2})$ times the mass above it. Thus
$\lambda_{r_j}^{\rm pf}(f)\ge\varepsilon t_j/2$. The matching upper bound follows from
\cite[Theorem~1]{r69:SLpush}.

Take $a_j=8^j$. The choice $\rho_j=e^{-a_j/(2m)}$ gives $C^m$
convergence and $\ell_j\asymp_m a_j\asymp\log r_j$.
With $\rho_j=e^{-a_j/j}$, the series converges in every $C^k$ norm and
$\ell_j\asymp a_j/j\asymp\log r_j/\log\log r_j$. For the last assertion
keep $\rho_j=e^{-a_j/j}$ and choose $a_j\ge8a_{j-1}$ so large that
$r^2u_r/(\log r)^2\le j^{-4}$ for all $r\ge e^{a_j/2}$. We still have
$\ell_j\asymp a_j/j$ and $\log r_j\sim a_j/2$, whence
$\lambda_{r_j}^{\rm pf}(f)/u_{r_j}\ge c j^2$ for large $j$.
\end{proof}

\begin{corollary}[The logarithmic loss on a fixed smooth planar body]
\label{r73:smoothlimit}
Let $K$ be a planar convex body with $C^\infty$ boundary and Lebesgue
measure. Let $f$ be a polynomial with $\nabla f\ne0$ at every minimizer
on $K$. For this fixed pair the push-forward error satisfies
\[
 \lambda_r^{\rm pf}(f)-\min_K f=o\big(r^{-2}\log^2 r\big).
\]
For every positive sequence $u_r=o(r^{-2}\log^2r)$ there are a fixed
$C^\infty$ uniformly convex planar body and the single quadratic
$f(x,y)=1-x^2-y^2$ for which the error divided by $u_r$ is unbounded.
\end{corollary}
\begin{proof}
Normalize the minimum and maximum to zero and one. Write $g=f|_{\partial K}$.
By Theorem~\ref{r73:contact} the push-forward has density $w(s)$
comparable near zero to the boundary length $A(s)=|\{g\le s\}|_{\partial K}$.
If all zeros have finite order, they are isolated and hence finite.
In an arclength coordinate at a zero of order $k$ one has
$g(y)\asymp |y|^k$. Thus $A(s)\asymp s^{1/k_*}$, with $k_*$ the largest
order, and doubling gives the quadratic rate.

Otherwise, at a zero flat to every order, Taylor's bound gives
$A(s)\ge c_k s^{1/k}$ for each integer $k\ge1$. In particular
$w$ is bounded near zero and $V(t)=\int_0^t w(s)\,ds\ge c'_k t^{1+1/k}$.
Fix $\varepsilon>0$, set $z=(\varepsilon\log r/r)^2$,
$N=\lfloor r/4\rfloor$ and $m=\lfloor r/2\rfloor$. Use the trial
$F_N(s)^2Q_m(s)$ from Lemma~\ref{r69:sublevel}, with
\[
 Q_m(s)=\frac{T_m((1+z-2s)/(1-z))}{T_m((1+z)/(1-z))}.
\]
This polynomial $Q_m$ equals one at zero and has supremum norm one
on $[0,1]$. Markov's inequality gives $Q_m\ge1/2$ on $[0,c/r^2]$
for a fixed small $c$, where $F_N$ too is bounded below.
On $[z,1]$ the Chebyshev formula gives
$|Q_m|\le2e^{-2m\sqrt z}$. Since
$F_N(s)\le\min(1,1/(N^2s))$, boundedness of $w$ near zero and
finiteness of the measure elsewhere imply $\int s F_N^4\,d\nu=O(r^{-4})$.
The squared trial has mass at least $c''_k r^{-2-2/k}$, so its
Rayleigh quotient is at most $z+C_k r^{-2+2/k}e^{-4m\sqrt z}$.
Choose $k>2/\varepsilon$ and then let $\varepsilon\downarrow0$.
The last assertion is the smooth construction of
Theorem~\ref{r70:convexpush}; its contact points $(\pm1,0)$ have
nonzero gradient.
\end{proof}

\section{Expected signatures and Brownian exit}\label{path:sec}

\subsection{Determination from a positive expected-signature radius}
\label{sig:sec:determination}
Chevyrev and Lyons asked whether a law in their class $\Phi(\mathbb R^d)$
can have the same expected signature as a law outside that class
\cite[Remark~6.19]{ChevyrevLyons}. We use finite-dimensional developments
at parameters near zero to identify the entire joint law of the word
coordinates.

Fix $d\ge1$. A normalized formal group-like series $s=(s^w)_w$ has real
coordinates indexed by finite words over $\{1,\ldots,d\}$, with
$s^\emptyset=1$ and
\[
 s^u s^v=\sum_{w\in\operatorname{Sh}(u,v)}s^w .
\]
Shuffles preserve the orders of both words and are counted with
multiplicity. These identities define a closed subset $\widehat G_d$ of
the countable product of coordinate lines. A word-integrable law
$\mu$ on $\widehat G_d$ has $\int |s^w|\,d\mu<\infty$ for every $w$.
Its expected signature has coordinates $m^w=\int s^w\,d\mu$.
For a series $a$ write
\[
 p_r(a)=\sum_{n\ge0}r^n\sum_{|w|=n}|a^w| .
\]
The inner sum is the projective tensor norm for the coordinate $\ell^1$
norm. Changing the norm of $\mathbb R^d$ rescales the radius by at most
fixed factors. The entire group-like space $G_d=\{s\in\widehat G_d:
p_k(s)<\infty\ \forall k\in\mathbb N\}$ is Borel. Its seminorm topology
and its word coordinates give the same Borel sets, since the seminorms
are measurable and the ambient entire tensor space is separable.

For a word $w=i_1\cdots i_m$, $m\ge1$, define skew matrices on rows and
columns $0,\ldots,m$ by
\[
 B_i^w=\sum_{j:i_j=i}(E_{j-1,j}-E_{j,j-1}),\qquad
 D_n^w(s)=\sum_{|v|=n}s^v B_{v_1}^w\cdots B_{v_n}^w,\qquad D_0^w=I .
\]
The only walk from $0$ to $m$ in $m$ steps moves right at every step, so
$(D_m^w)_{0m}=s^w$.

\begin{theorem}[Second moments of tridiagonal developments]\label{sig:thm:one-sided}
Suppose a word-integrable law $\mu$ on $\widehat G_d$ satisfies, for each
nonempty word $w$,
\begin{equation}\label{sig47:growth}
 \limsup_{n\to\infty}\left(\max_{a,b}
       \int |(D_n^w)_{ab}|^2\,d\mu\right)^{1/(2n)}<\infty .
\end{equation}
Its expected signature determines $\mu$ among all word-integrable laws
on $\widehat G_d$.
\end{theorem}
\begin{proof}
Let $\nu$ have the same expected signature, put
$\lambda=(\mu+\nu)/2$, and set $h=d(\mu-\nu)/d\lambda$. All word
polynomials are integrable and have the same integrals under $\mu,\nu$,
by shuffling. In particular their squares are integrable, and each
second moment in \eqref{sig47:growth} transfers to $\lambda$.
For each $w$, $U_w(z)=\sum_n z^nD_n^w$ is therefore holomorphic in
$L^2(\lambda)$ on some disc about zero. On every smaller closed disc the
sum of the expected absolute norms is finite, giving absolute convergence
almost surely. For real $z$ in the disc $U_w(z)$ is orthogonal. Indeed
signed word reversal gives the inverse of a formal group-like series
and evaluates to the transpose for skew generators
\cite[Section~4]{ChevyrevLyons}. This formal identity evaluates in each
disc of absolute convergence.

For any word polynomial $P$ and any finite product $F$ of entries of
$U_w(x)$ at real rational parameters inside their respective discs, we
claim $\int hPF\,d\lambda=0$. For $F=1$ this is the moment equality.
If the claim holds for products with $j$ factors, let $F$ be such a
product. For any word $v$, the entries of
$z\mapsto\int hPF U_v(z)\,d\lambda$ are holomorphic on its disc,
since $hF$ is bounded and $P\in L^2(\lambda)$. Each Taylor
coefficient vanishes by applying the induction hypothesis to the word
polynomial $P(D_n^v)_{ab}$. Evaluation at a rational parameter proves
the claim with one more factor.

Use a common full-measure set for the countably many words and rational
parameters. These bounded matrix entries generate the coordinate
$\sigma$-field up to null sets, since $s^w$ is the coefficient of $z^m$
in $(U_w(z))_{0m}$, recovered by finite differences along rational
parameters tending to zero. With $P=1$ the claim gives zero against the
algebra of these entries. The bounded monotone class theorem now gives
$h=0$ almost surely.
\end{proof}

\begin{theorem}[Positive radius determines the signature law]\label{sig:thm:positive-radius}
If $p_\rho(m)<\infty$ for some $\rho>0$, the expected signature of a
word-integrable law on $\widehat G_d$ uniquely determines that law.
In particular a reference law in $\Phi(\mathbb R^d)$ has no distinct
word-integrable group-like competitor.
\end{theorem}
\begin{proof}
Set $C=p_\rho(m)$. For a word $v$ of length $n$, the shuffle square has
$\binom{2n}{n}$ terms. For each resulting word $u$ we have
$|m^u|\le C\rho^{-2n}$. Hence
\begin{equation}\label{sig55:shuffle-square}
 \|s^v\|_{L^2(\mu)}\le C^{1/2}(2/\rho)^n,\qquad
 \int p_r(s)\,d\mu<\infty\quad(2dr<\rho).
\end{equation}
The second bound follows by summing over the $d^n$ words. For fixed $w$,
the triangle inequality in $D_n^w$ gives \eqref{sig47:growth}, since each
matrix product has norm at most $(\max_i\|B_i^w\|_{\mathrm{op}})^n$.
Theorem~\ref{sig:thm:one-sided} applies. Condition (P1) in
\cite[Definition~6.12]{ChevyrevLyons} implies the required positive
radius for every member of $\Phi(\mathbb R^d)$.
For a continuous geometric rough path its sample signature is entire.
Each finite-dimensional unitary development is a pointwise limit of
polynomials in its word coordinates and hence a bounded Borel function
on the space of entire signatures. The signature law therefore
determines the expectation of every such development, the path
characteristic function in \cite[Problem~1.2]{LiuLyuNi}.
\end{proof}

\subsection{Which series are expected signatures?}
On pp.~137--138 of \cite[Ch.~6]{r68sig:thesis}, Chevyrev asks two related existence questions for probability laws on the entire group $G_d$. One prescribes the expected word coordinates, the other the characteristic functions of finite-dimensional unitary developments.
For a real tensor series $m$, write $L_m(w)=m^w$ on the word basis of the
commutative shuffle algebra. Positivity on squares means that every matrix
\[
 \left[\sum_{z\in\operatorname{Sh}(u,v)}m^z\right]_{u,v\in I}
\]
is positive semidefinite for finite word sets $I$.

\begin{theorem}[Positive-radius existence]\label{r68sig:exist}
Suppose $m^\emptyset=1$ and $p_\rho(m)<\infty$ for some $\rho>0$.
It is the expected signature of a word-integrable law on $\widehat G_d$
if and only if $L_m$ is positive on shuffle squares. The law is unique.
If $m$ has infinite radius, it is supported on $G_d$ and its expectation
exists in the entire tensor topology.
\end{theorem}
\begin{proof}
Necessity follows by integration. For sufficiency form the complex
Hilbert completion $H$ of shuffle polynomials modulo zero norm for
$\|p\|^2=L_m(\bar p p)$; here multiplication means shuffle.
Cauchy--Schwarz makes the null space an ideal:
$L_m(\overline{qp}\,qp)=L_m(\bar p\,\bar q q p)=0$ if $\|p\|=0$.
Let $D_n^w$ now have entries in this algebra, replacing the coordinate
$s^v$ by its word $v$. For fixed $w$ and any polynomial $p$, all entries
of $D_n^w p$ satisfy
\[
 \limsup_n\|(D_n^w)_{ab}p\|^{1/n}
 \le 2d\max_i\|B_i^w\|/\rho .
\]
Indeed for a word $u$ of length $k$ the squared norm of $v u$, $|v|=n$,
is bounded by $p_\rho(m)\rho^{-2n-2k}(2n+2k)!/(n!^2k!^2)$.
Take square roots and sum over the $d^n$ words and the finitely many
coefficients of $p$. Thus for real $x$ in a fixed interval about zero
$U_w(x)p=\sum_n x^nD_n^wp$ converges on polynomial vectors.
The formal transpose is its formal inverse by the shuffle antipode.
The isometry identity on polynomial vectors follows by summing the
absolutely convergent double series for squared norms. The transpose
also acts isometrically and is adjoint there. The extensions give a
unitary matrix on $H^{|w|+1}$. Each entry is a bounded self-adjoint
operator on $H$, by symmetry of real polynomial multiplication.
Any two entries commute, even from different words and parameters:
in their matrix elements on polynomials, moving the outer bounded
operator to the other slot gives an absolutely convergent double series
of commuting polynomial factors in either order.

Use just rational parameters in the respective intervals. Their
entries generate a commutative unital $C^*$-algebra with cyclic vector
$1$. Indeed for a word of length $k$ the $(0,k)$ entry of $D_j^w$
vanishes for $j<k$ and equals $w$ for $j=k$, whence
$x^{-k}(U_w(x))_{0k}p\to wp$ in $H$. Taking $p=1$ yields every word.
In the cyclic spectral representation $H=L^2(\sigma)$, denote the
function representing $p$ by $F_p$. The same bounded multiplication
operators, rescaled by $x^{-k}$, converge on $1$ in probability to $F_w$.
Their products with $F_p$ therefore converge in probability to $F_wF_p$
and in $L^2$ to $F_{wp}$. Hence $F_{wp}=F_wF_p$ almost surely.
The real word functions (all approximants on $1$ are real) satisfy
all countably many shuffle identities on one full-measure set.
Their joint law on $\widehat G_d$ realizes $m$. Uniqueness is
Theorem~\ref{sig:thm:positive-radius}. At infinite radius
\eqref{sig55:shuffle-square} gives $\int p_k(s)<\infty$ for every
integer $k$, proving both the support and the expectation assertions.
\end{proof}

Let $E_d$ be the entire real tensor series, with the seminorms $p_k$
defined above. Pad a degree-$N$ word array by zeros in $E_d$. Call it
a normalized shuffle array if it satisfies the shuffle identities
through total degree $N$ and its empty coordinate is one.

\begin{theorem}[Tight finite arrays and cubature]\label{r69sig:tight}
Let $N_n\ge n$ and let $\lambda_n$ be probability laws on normalized
degree-$N_n$ shuffle arrays. Their zero extensions are tight in $E_d$
if and only if integers $R_k\ge1$ can be chosen so that
\[
       \lambda_n\{p_k>R_k\}\le2^{-k}\qquad (1\le k\le n).
\]
Every weak limit along $n\to\infty$ is supported on $G_d$. Conversely,
for any probability on $G_d$ one such envelope works for all its
truncations. At level $n$, any $a_n$ integrable real tests of its
truncation, including the constant one, can be matched by a law on
at most $a_n+n$ arrays with the same envelope.
\end{theorem}
\begin{proof}
Tightness gives the bounds by continuity of each $p_k$. Conversely,
each $p_j$ is uniformly tight: take $k\ge j$ with $2^{-k}$ as small
as required, use $p_j\le p_k$ for $n\ge k$, and enlarge the bound for
the finitely many earlier laws. Hence, with probability uniformly
arbitrarily close to one, all $p_j$ satisfy bounds $b_j$ simultaneously
(allocate summable error probabilities). The resulting closed set is
compact in $E_d$. On it, truncation beyond degree $J$ costs at most
$2^{-J}b_{2j}$ in $p_j$, and each finite set of coordinates is bounded.
Diagonal extraction gives compactness in every seminorm. Each fixed
shuffle identity is eventually satisfied and is closed. A law on $G_d$
has tight $p_k$ and truncation only decreases them. Apply finite
cubature to the listed tests and the $n$ bad-set indicators
\cite[Corollary~2]{r68sig:cubature}. For general projective-limit
moment problems see \cite{r69sig:projective}.
\end{proof}

\begin{corollary}[Realization of moments and characteristic functions]
\label{r68sig:allmoments}\label{r68sig:bochner}
Prescribe a normalized real word sequence $m$, a matrix characteristic
function $\psi$, or both. A law on $G_d$ realizes the prescribed data
exactly when the following finite tests hold with one envelope $(R_k)$.
When $\psi$ is specified it must also be continuous in each matrix size.

For word data include every coordinate of length at most $n$ and require
its mean to equal $m^w$. For matrix data enumerate a norm-dense rational
family of skew-Hermitian tuples in all sizes as $(M_j)$, and put
$B_j=\max(1,\max_i\|M_{j,i}\|)$. Use the degrees
\[
 k_n=n+2+\lceil2\max_{j\le n}B_j\rceil,\qquad
 N_n=\max\{n,\lceil\log_2(2nR_{k_n})\rceil\};
\]
for word data alone set $N_n=n$. For each $j\le n$ include the real and
imaginary entries of $\sum_{|w|\le N_n}x^w M_{j,w}$ clipped to $[-1,1]$,
where $M_{j,w}$ is the ordered product. Require their means to agree
with the corresponding entries of $\psi(M_j)$ within $1/n$.
At every $n$ these requirements and the envelope of
Theorem~\ref{r69sig:tight} must hold on at most $a_n+n$ normalized
degree-$N_n$ arrays with nonnegative weights of sum one, where $a_n$
counts all listed scalar tests including the constant once. Each test
is semialgebraic. With characteristic data, every sequence of witnesses
converges weakly in $E_d$ to the unique realizing law.

For word data alone an equivalent condition is that $L_m$ be the wordwise
limit of functionals with entire coefficients and mass at most one,
increasing from zero in the cone of functionals nonnegative on shuffle squares.
\end{corollary}
\begin{proof}
For a realizing law use Theorem~\ref{r69sig:tight}, choosing $R_k$ to
bound the bad sets for the full series. The full development is unitary.
On $p_{k_n}\le R_{k_n}$ its omitted tail is at most $2^{-N_n}R_{k_n}$
for $j\le n$. On the complement clipping costs at most two per entry.
The expected error is thus at most $1/n$.

Conversely Theorem~\ref{r69sig:tight} gives subsequential limits on $G_d$.
When word data are specified, a fixed coordinate $w$ eventually has
second moment $L_m(w^2)$ by shuffle multiplication, giving uniform
integrability. For each fixed $M_j$, the clipped full development is a
bounded continuous test on $E_d$. On a zero-extended array it equals
the finite test, and on $G_d$ it equals the unitary development.
Both prescribed means therefore pass to the same limit. Continuity and
bounded convergence extend matrix identities to all tuples. Characteristic
functions identify the law \cite[Corollary~4.12]{ChevyrevLyons}.

For the cone assertion restrict a realizing law to increasing compact
sets of mass tending to one. Each $p_k$ is bounded there. Conversely
Theorem~\ref{r68sig:exist} represents the positive entire differences
after normalization (zero mass forces a zero functional). Their measures
sum to mass one. The $w$-square integral is $L_m(w^2)<\infty$, so the
$w$-integral converges absolutely to the telescoping limit $m^w$.
\end{proof}

\subsection{Separation of signatures over a Banach space}
Chevyrev asked whether unitary developments separate the entire
group-like elements over every separable Banach space
\cite[Ch.~6, p.~137]{r68sig:thesis}. Cox, Khedher and Maessen asked
whether scalar shuffle tests characterize exponentials of Lie series
in infinite dimensions. Their Proposition~2.22 gives equality under the
approximation property for strongly uniform crossnorms \cite{r68:CKM}.
Use completed
projective tensor powers $V_n=V^{\widehat\otimes_\pi n}$ of a real
Banach space. A series $g=(g_n)$ with $g_0=1$ is entire if
$\sum_n R^n\|g_n\|_\pi<\infty$ for every $R>0$. The coproduct is
the multiplicative extension of $v\mapsto v\otimes1+1\otimes v$;
the group-like identity is $\Delta g=g\otimes g$.

\begin{theorem}[Tensor separation and group-like series]\label{r68:banach}\label{r68:weakshuffle}
The following properties of $V$ are equivalent.
\begin{enumerate}
\item Elementary tensors of functionals in $V^*$ separate each $V_n$.
\item Developments induced by bounded linear maps from $V$ into finite
dimensional skew-Hermitian matrices separate entire group-like series.
\item Every normalized formal series satisfying the shuffle identity
against algebraic dual words is the exponential of a series of
homogeneous elements of the closed free Lie algebra.
\end{enumerate}
All three fail on a single separable space, with counterexamples
supported at degree two in the logarithm of the group-like series and
in the nonconstant part of the weakly group-like series, respectively.
\end{theorem}
\begin{proof}
Under (1), a differing tensor is detected by a finite coordinate map
$Q:V\to\mathbb R^m$ formed from its detecting functionals.
Apply finite-dimensional matrix separation
\cite[Theorem~2.3.8]{r68sig:thesis} to the images to obtain (2).
For (3), the bidegree $(j,k)$ component of $\Delta g-g\otimes g$
pairs to zero with every product of $j+k$ functionals by the shuffle
identity. Projective associativity and (1) give $\Delta g=g\otimes g$.
The tensor form of Chen's theorem \cite[Theorem~D.1]{r68:CKM}
then gives the closed Lie logarithm. Conversely, exponentials of Lie
series always satisfy this coproduct identity.

If (1) fails, split $V=W\oplus\mathbb R e$ with a bounded projection.
Decompose an undetected nonzero tensor by its $W$ and $e$ slots.
A nonzero block, after deleting its scalar slots, gives a nonzero
$z\in W_m$ annihilated by all products of $W^*$ functionals: extend
each by zero on $e$ and test the deleted slots by the scalar projection.
Extend $w_1\otimes\cdots\otimes w_m\mapsto
[w_1,[\ldots,[w_m,e]\ldots]]$ continuously to $W_m$, and denote its
value at $z$ by $L$. Projection onto the block with $e$ last gives
$z\otimes e\ne0$. Every matrix development kills $L$, since every
entry of its value is a finite sum of elementary functionals on $z$.
The entire group-like series $\exp L$ therefore has the same
developments as $1$. This persists for bounded operator developments
on any Hilbert space: compress to finite-rank projections increasing
strongly to the identity. Products converge strongly on elementary
tensors, and their uniform bounds extend convergence to each $V_n$.
For (3), $1+z\otimes e$ agrees with $1$ on every algebraic word, but
the $(m,1)$ coproduct of its leading term has nonzero projection
$z\otimes e$ onto $W_m\otimes\mathbb R e$. Its logarithm cannot be Lie.

Finally take a separable space $X$ without the approximation property
\cite{r68:Enflo}. Grothendieck's trace criterion \cite{r68:Grot} gives
$\sum_i\|f_i\|\|x_i\|<\infty$, $\sum_i f_i(\cdot)x_i=0$ and
$\sum_i f_i(x_i)=1$. Put $U=\overline{\operatorname{span}}\{f_i\}$
in $X^*$ and $V=U\oplus X$. The evaluation pairing detects
$t=\sum_i f_i\otimes x_i\in U\widehat\otimes_\pi X$, while every
elementary functional kills it: $\sum_i\phi(x_i)f_i=0$ in $U$
for every $\phi\in X^*$ by the adjoint of the zero operator.
Writing $t^{\rm op}$ for its flip, the two distinct tensor blocks give
$t-t^{\rm op}\ne0$ and $t+t^{\rm op}\ne0$. The first is a limit of
commutators. Its exponential is invisible to all developments, while
$1+t+t^{\rm op}$ passes every shuffle test and its symmetric degree-two
logarithm lies outside the closed commutator space.
\end{proof}

\paragraph{A clock and the uniform approximation of rough paths.}
Fix $T>0$. For a geometric or weakly geometric rough path $X$, let
$S^j_{s,t}$ denote the $j$th level of its Lyons extension, with $S^0=1$.
An external bounded-variation clock gives a canonical joint signature.
In the block with $l$ clock letters between spatial blocks of lengths
$j_0,\ldots,j_l$, its value is
\begin{equation}\label{r69:clock}
 \int_{s<u_1<\cdots<u_l<t}
       S^{j_0}_{s,u_1}\otimes S^{j_1}_{u_1,u_2}
          \otimes\cdots\otimes S^{j_l}_{u_l,t}\,du_1\cdots du_l .
\end{equation}
For $l=0$ this denotes $S^{j_0}_{s,t}$. At each degree use the finite
direct sum of the spatial tensor blocks, with its Euclidean sum norm.
Permuting slots permutes the blocks isometrically. For two tensors
$x,y$, each block of $x\otimes y$ is the tensor product of its two
corresponding blocks. Spatial submultiplicativity and summation of
squares give $\|x\otimes y\|\le\|x\|\|y\|$, which suffices below.
Formula~\eqref{r69:clock} is the levelwise limit of products over
partitions of the spatial increments followed by
$\exp(\Delta t\,e_0)$, where $e_0$ is the clock letter. Indeed terms
whose clock times lie in distinct mesh cells are Riemann sums for
the displayed integral by Chen's identity and continuity. Terms with
multiple times in a cell have total simplex volume tending to zero
and bounded integrands. These products are group-like. Splitting the
simplex at an intermediate time gives Chen's identity, and the
increment bound in degree $j_0+\cdots+j_l+l$ follows by integrating
the spatial bounds. Thus the formula defines the joint rough path
lift, agrees with ordinary iterated integrals for smooth paths, and
commutes with finite coordinate maps. This construction needs only
admissible spatial norms.

Cox, Khedher and Maessen proved approximation on bounded balls and proposed global weighted approximation as a direction for future research \cite[Sec.~1.1]{r68:CKM}. Their outlook leaves the choice of weights open.
For finite-dimensional drivers with a clock, Cuchiero, Schmocker and
Teichmann proved density for $\exp(\lambda R^\gamma)$,
$\gamma\ge\lfloor1/\alpha\rfloor$, using the prefix generators below
\cite[Theorem~5.4 and Remark~5.5]{r69:CST}.
The scalar integral in \eqref{r70sig:weight} is the classical
Bernstein criterion for weighted polynomial approximation on the line
\cite[Theorem~6]{r71sig:DD}. We determine the path size to which it applies
for ordinary terminal words and for words with a clock.
Use the predual setting of
\cite[Setting~3.11]{r68:CKM}: $E=F^*$ and each spatial tensor level
is dual to the corresponding completion of the tensors over $F$.
Let $\mathcal X$ be the whole space of based weakly geometric
$\alpha$-H\"older rough paths on $[0,T]$, in the weak-star topology on
their prefix paths through degree $n=\lfloor1/\alpha\rfloor$.
Write $R(X)=\max_{1\le j\le n}\sup_{s<t}\|X^j_{s,t}\|^{1/j}/|t-s|^\alpha$.
\needspace{5\baselineskip}
\begin{theorem}[Sharp weights for global signature approximation]
\label{r69sig:global}\label{r70sig:log}
Suppose $E\ne0$ and $T>0$. Put $p=1/\alpha$, $n=\lfloor p\rfloor$
and $Y=(1+R)^s$. Use
$s=p$ for ordinary terminal predual words on every $A\subseteq\mathcal X$
closed on each $R$-ball and separated by these words. Use $s=n$ for
terminal words with the external clock \eqref{r69:clock} on all of
$\mathcal X$. In dimension one use $s=1$ in both statements.
Let $\Omega:[1,\infty)\to(0,\infty)$ be increasing, with
$\Omega(y)/\log y\to\infty$ and $t\mapsto\Omega(e^t)$ convex on
$[0,\infty)$.
For every $\lambda>0$, density in $\mathcal B_\psi$ for the indicated
space or family of spaces, with $\psi=\exp(\lambda\Omega(Y))$, holds
if and only if
\begin{equation}\label{r70sig:weight}
                      \int_1^\infty\frac{\Omega(y)}{y^2}\,dy=\infty .
\end{equation}
Here $\mathcal B_\psi$ is the closure of bounded continuous functions in
$\sup |f|/\psi$. The sufficient condition persists for any larger
admissible weight. In particular, for $Y/\log(e+Y)^\sigma$ the
threshold is $\sigma=1$ (up to bounded-interval modifications),
and the indicated powers $s$ are universally sharp.
\end{theorem}
\begin{proof}
The $R$-balls are weak-star compact. Words are continuous there by
\cite[Theorem~3.16 and Proposition~6.1]{r68:CKM} and uniform Riemann
approximation in \eqref{r69:clock}. Their polynomial growth in $1+R$
and shuffle multiplication give an algebra in $\mathcal B_\psi$
\cite[Lemma~2.7]{r69:CST}. By weighted Riesz representation
\cite[Theorem~2.4]{r69:DT}, take a signed annihilator $\eta$
with $\int\psi\,d|\eta|<\infty$.
Take $M_k=\sup_{y\ge1}y^ke^{-\lambda\Omega(y)}$ and
$q_k=M_k/M_{k-1}$. Logarithmic convexity gives $1\le q_k\uparrow\infty$
(the root growth of $M_k$ follows by fixing arbitrarily large $y$).
For $T(y)=\sup_{k\ge0}y^kM_0/M_k$ one has
$\log T(y)=\sum_{k\ge1}\log^+(y/q_k)$, hence
$\int_1^\infty \log T(y)\,dy/y^2=\sum_k1/q_k$.
Writing $t=\log y$ and $\Phi(t)=\lambda\Omega(e^t)$ gives
$\log M_k=\sup_{u\ge0}(ku-\Phi(u))=\Phi^*(k)$.
For $t>0$ take a supporting slope $r\ge0$. Since $\Phi^*$ is
nondecreasing, $j=\lfloor r\rfloor$ satisfies
$jt-\Phi^*(j)\ge\Phi(t)-t$. Thus
$\Phi(t)-t\le\log T(e^t)-\log M_0\le\Phi(t)$,
so \eqref{r70sig:weight} is exactly the Denjoy--Carleman condition
$\sum_k1/q_k=\infty$ \cite[Lecture~19, Sec.~2]{r70sig:Kahane}.
For $k\ge1$ put $a_k=(M_0/M_k)^{1/k}$ and
$B_k=\max(M_k/M_0,k^k)$, with $B_0=1$. Then
$\int(k+Y)^k\,d|\eta|\le C^{k+1}B_k$. The sequence $a_k$ is
decreasing and $a_k\ge1/q_k$. Under \eqref{r70sig:weight} its sum
diverges. Cauchy condensation shows that the sum still diverges
after replacing $a_k$ by $B_k^{-1/k}=\min(a_k,1/k)$: the condensed
terms are $\min(2^j a_{2^j},1)$. Thus the log-convex $B_k$ is a Denjoy--Carleman
sequence for quasi-analyticity by the root criterion
\cite[Lecture~19, Sec.~2]{r70sig:Kahane}.

For ordinary words in dimension at least two take any finite predual
projection and skew-Hermitian matrix development $U_z$, scaling the
matrices by $z\in\mathbb C$. On each of $N=\lceil Y\rceil$ equal
time intervals the projected H\"older control $C(1+R)^p(v-u)$ is
bounded. Enlarging $C$ for the first $n$ levels, Lyons' bound
\cite[Sec.~5, (5.1)]{ChevyrevLyons} bounds each interval coefficient
by $C_0^j/\Gamma(1+j/p)$. The corresponding development
and first derivative are uniformly bounded on each complex disk, and
at real $t$ the development is unitary. Multiplication gives
$\|U_z(X)\|\le e^{C_BN|z-t|}$ for $|t|\le B$, $|z-t|\le1$.
Cauchy's estimate on radius $\min(1,k/(C_BN))$ bounds
$\|\partial_t^kU_t(X)\|$ by $[C'_B(N+k)]^k$.

The entries of $\int U_t\,d\eta$ have derivative bounds
$C_B^{k+1}B_k$ and zero Taylor coefficients by word annihilation.
Quasi-analyticity makes them identically zero. Matrix coefficients
at real parameters span a bounded unital self-adjoint algebra
by tensor products and conjugation. They are continuous on balls
by uniform factorial convergence and separate $A$ by
\cite[Theorem~2.3.8]{r68sig:thesis} after predual projection.
To see that such an algebra determines $\eta$, approximate bounded
continuous real $f$ on
a compact ball by Stone--Weierstrass, then clip to
$[-\|f\|_\infty,\|f\|_\infty]$. The clip still integrates to zero by
polynomial approximation on the approximant's bounded range.
Exhausting the balls gives $\int f\,d\eta=0$. For $\exp(\lambda Y)$,
positive expected signature radius of the projected Jordan parts
also gives uniqueness by Theorem~\ref{sig:thm:positive-radius}.

For the external clock, the prefix generators
$h_{k,w}=\int_0^T t^k\langle X^j_{0,t},w\rangle\,dt$,
$k\ge0$, $|w|=j\le n$, are clock words up to shuffle and the
factor $k!$: shuffle $w$ with $k$ clock letters, then append one.
Polynomial density in $t$, predual separation and Chen's identity
recover all increments. Each finite real linear combination $h$
satisfies $|h|\le C_h(1+R)^n$. The integrated character
$\int e^{ith}\,d\eta$ thus has the same derivative bounds and zero
Taylor coefficients by shuffle. Quasi-analyticity and the bounded
algebra of characters give $\eta=0$. In dimension one the ordinary
signature is the exponential of the separating endpoint increment on
$A$, and first-level prefixes determine the whole path. These
generators obey the character bound with $s=1$.

When \eqref{r70sig:weight} fails, fix $D\ge1$ and put
$W(m)=\exp(\lambda\Omega(D(1+|m|)))$ for $m\in\mathbb Z$.
Monotonicity and integral comparison give
$\sum_m\log W(m)/(1+m^2)<\infty$, and $|m|^k/W(m)\to0$ for each $k$.
Koosis's integer approximation theorem puts $\mathbf1_{\{0\}}$
outside the polynomial closure in the norm $\sup_m|f(m)|/W(m)$
\cite[Part~II, Theorem~13 and its proof]{r71sig:Koosis}.
The dual of weighted $c_0(\mathbb Z)$ therefore gives a real sequence
$c_m$ with $c_0=1$, $\sum_m|c_m|W(m)<\infty$, and all power sums zero.
This is the discrete counterpart of Ingham's compact-support Fourier
estimate \cite{r70sig:Ingham}.

In dimension at least two take independent vectors $x,y$ and a fixed
piecewise linear loop with signature $g$ whose first nonconstant
level is $n+1$. Explicitly, start with the segment in direction $y$
and take $n$ successive path commutators with the segment in direction
$x$, using concatenation and reversal. The lowest term of $g-1$ is
$\operatorname{ad}_x^n y\ne0$ by multiplication at each commutator.
Let $X_m$ repeat the loop $|m|$ times at equal time intervals, reversed
for $m<0$. Its signature is $g^m$, whose coefficients are polynomials
in $m$ by the degreewise binomial expansion. They distinguish $m$
already at degree $n+1$. Complete cells have trivial step-$n$
signature. An increment spanning several cells therefore has only
two end pieces. On any interval the Lipschitz length bound is
$C|m|(t-s)$. These bounds give $R(X_m)\le C|m|^{1/p}$, and a
nonconstant first-level increment within one cell gives the reverse
inequality. The family has finite intersection with each ball.
Choose $D$ so $Y(X_m)\le D|m|$ for $m\ne0$. The signed measure
$\sum_m c_m\delta_{X_m}$ has finite weighted total variation and
annihilates every word. It detects the indicator of $X_0$ on this
family, which belongs to the weighted space by the bounded-ball criterion.

For the clock choose a nonzero degree-$n$ Lie polynomial $L$ on
$x,y$ and take $X^m_{u,v}=\exp_{\le n}((v-u)mL)$. Its full spatial lift
is the untruncated exponential by Lyons uniqueness, $R(X^m)=C|m|^{1/n}$,
and \eqref{r69:clock} makes every clock word a polynomial in $m$.
The same construction gives a weighted signed annihilator at $s=n$.
A bounded continuous function of a spatial prefix pairing recovering
$m$ detects its atom at zero. In dimension one use straight lines
of slopes $m$ for both assertions.
\end{proof}

On these integer families every separating subspace of the word
algebra contains a nonconstant polynomial $P(m)$. At the critical
weight $\Omega(y)=y/\log(e+y)$ and $Y\asymp1+|m|$, every
$\exp(\varepsilon|P(m)|)/\psi(X_m)$ is unbounded for $\varepsilon>0$.
Thus the moderate-growth hypothesis of \cite[Theorem~3.9]{r69:CST}
fails even at a weight for which density holds.

\begin{corollary}[When clock words distinguish actual geometric paths]
Give $V$ its projective tensor norms\label{r69sig:geometric} and let $p>1$ be nonintegral,
$m=\lfloor p\rfloor$. On based geometric $p$-rough paths on $[0,T]$,
the algebraic dual words of the terminal clock signature separate
points if and only if products of $V^*$ functionals separate the
closed homogeneous Lie spaces in every degree through $m$. This is
also equivalent to the uniform density of these words on every
compact set in the rough path topology.
\end{corollary}
\begin{proof}
Equality of the clock words gives equality of all prefix moments
$I_{k,w}$ through degree $m$. Polynomial density recovers each scalar
prefix. Passing to the truncated logarithm preserves these equalities:
each product of lower tensor levels, tested by a word, is a product
of its evaluations on the corresponding subwords. Lie separation now
recovers every truncated log prefix, and hence the whole path.

If separation fails, take a nonzero undetected homogeneous Lie element
$L$ of degree $j\le m$. The full increments $\exp((t-s)L)$ satisfy
Chen's relation and have finite $j$-variation since the level $kj$
is $(t-s)^kL^{\otimes k}/k!$. These are genuine geometric
$p$-rough path signatures. Indeed approximate $L$ in tensor norm by
finite Lie polynomials $L_i$, each lying over a finite-dimensional
subspace. The resulting weakly geometric $j$-rough paths are geometric
at every $p>j$ by finite-dimensional rough path approximation
\cite[Corollary~8.24]{r69:FV}. Their lifts through degree $m$ converge
in $p$-variation by the explicit powers and
$\sum_{[s,t]\in\mathcal P}|t-s|^{p/j}\le T^{p/j}$ for every partition.
Lyons uniqueness identifies the full lifts. Every finite spatial
projection kills $L$ and sends this path to the constant path; by
\eqref{r69:clock} their clock signatures have identical scalar words.
Finally word functions are continuous and form a unital shuffle algebra,
so compact density is equivalent to separation by Stone--Weierstrass.
In particular the degree-two counterexample of
Theorem~\ref{r68:banach} gives two such paths for every $2<p<3$.
\end{proof}

\subsection{The radius and law at Brownian exit}
Lyons and Ni asked whether the expected signature at first exit from a
bounded domain determines the law of that signature
\cite[Introduction]{LyonsNiExit}. Their Remark~3.6 concerns the compact-like sufficient condition available at the time. For bounded strongly Lipschitz $C^m$ domains, $m\ge\lfloor d/2\rfloor+1$, Liu, Lyu and Ni proved uniqueness among competitors satisfying their condition (ROC) \cite[Corollary~3.24]{LiuLyuNi}. For the radius itself, Li and Ni proved finiteness
on bounded $C^{2,\alpha}$ domains in dimensions $2$ through $8$
\cite{LiNi}, and posed the question in higher dimensions
\cite[Future Work]{BrownianRadiusUCL}.

Let $B$ be standard Brownian motion in $\mathbb R^d$. Write $S_t^w$
for its iterated Stratonovich integral on $[0,t]$ in the order of the
letters of $w$. Almost surely on every finite interval these coordinates
form an entire group-like signature. For any fixed finite matrix tuple
$(A_i)$ its development
\[
 Y_t(z)=I+\sum_{n\ge1}z^n\sum_{|v|=n}S_t^v A_{v_1}\cdots A_{v_n}
\]
solves $dY_t=z\sum_i Y_t A_i\circ dB_t^i$. These are the geometric
extension and linear development identities for the Brownian rough path
\cite[Sections~4--5]{ChevyrevLyons}.

\begin{theorem}[The Brownian clock and the expected-signature radius]\label{br42:thm:radius}
Let $d\ge2$ and let $\tau$ be an almost surely finite stopping time with
word-integrable stopped signature. Write
\[
 R=\sup\{r\ge0:p_r(\mathbb E S_\tau)<\infty\},\qquad
 \beta=\sup\{b\ge0:\mathbb E e^{b\tau}<\infty\}.
\]
In the coordinate $\ell^1$ tensor norm the critical values satisfy
\[
       \frac{\sqrt{\beta}}{\sqrt{8}\,d}\ \le\ R\ \le\ 4\sqrt{\beta}.
\]
In particular they are positive simultaneously and infinite
simultaneously. A positive exponential moment of $\tau$ already
implies word integrability.
\end{theorem}
\begin{proof}
For each word $w$ of length $n\ge1$ take $A_i=B_i^w$. Each edge of the path $0,\ldots,n$ occurs in exactly one $A_i$.
At a vertex the entries of $A_i v$ use only its two neighbors.
Bounding the square of their difference by twice the sum of their
squares and summing over vertices and labels gives
$\sum_i\|A_i v\|^2\le4\|v\|^2$. Since the $A_i$ are skew, the same
bound gives $\sum_i\|Y A_i\|_{\mathrm{HS}}^2\le4\|Y\|_{\mathrm{HS}}^2$. For $z=x+ib$, set
$H_t=\operatorname{tr}(Y_t^*Y_t)$. The It\^o drift of $H_t$ is
\[
 (|z|^2-\operatorname{Re} z^2)\sum_i\|Y_t A_i\|_{\mathrm{HS}}^2
 \le 8 b^2 H_t ,
\]
using $A_i^*=-A_i$ in the Stratonovich correction. Hence
$e^{-8 b^2t}H_t$ is a nonnegative local supermartingale starting at
$n+1$. Optional sampling and Cauchy--Schwarz give, uniformly on $|z|=r$,
\[
 \mathbb E\|Y_\tau(z)\|_{\mathrm{HS}}
 \le\sqrt{n+1}\,(\mathbb E e^{8r^2\tau})^{1/2}.
\]
The coefficient of $z^n$ in the $(0,n)$ entry is $S_\tau^w$.
Pathwise Cauchy's formula and Tonelli therefore bound
$\mathbb E|S_\tau^w|$ by the last display times $r^{-n}$.
Summing over words gives $R\ge r/d$ whenever $8r^2<\beta$.

For the converse take $A_1=\left(\begin{smallmatrix}1&0\\0&-1\end{smallmatrix}\right)$,
$A_2=\left(\begin{smallmatrix}0&1\\1&0\end{smallmatrix}\right)$ and
$A_i=0$ for $i>2$. For real $\ell$ put $y_t=(1,0)Y_t(\ell)$.
The linear development is invertible, so $y_t$ never vanishes.
For $i=1,2$ let $u_i=y_tA_i y_t^{\mathsf T}/|y_t|^2$.
Writing $y_t=(a,b)$ gives
$u_1=(a^2-b^2)/(a^2+b^2)$ and $u_2=2ab/(a^2+b^2)$.
Thus $W_t=\sum_{i=1}^2\int_0^t u_i\,dB^i$ is Brownian motion.
Since $A_i^2=I$ for $i=1,2$, the row's It\^o drift is $\ell^2 y_t$.
The radial Hessian correction for $\log|y_t|$ is
$\frac12\ell^2(2-2\sum_i u_i^2)=0$, which gives
$\log|y_t|=\ell W_t+\ell^2t$. The nonnegative exponential martingale
$Z_t=\exp(-\ell W_t-\ell^2t/2)$ has $\mathbb E Z_\tau\le1$. Consequently
\[
 (\mathbb E e^{\ell^2\tau/4})^2
 =(\mathbb E\sqrt{|y_\tau|Z_\tau})^2\le\mathbb E|y_\tau|.
\]
Let $V_n$ be the degree-$n$ coefficient of this row development,
so $y_\tau=\sum_{n\ge0}\ell^n V_n$ pathwise. Expanding
$\mathbb E|V_n|^2$ by the shuffle identity gives
\[
 \mathbb E|V_n|^2\le {2n\choose n}
             \sum_{|u|=2n}|(\mathbb E S_\tau)^u|.
\]
Indeed, an output word $u$ and a choice of the $n$ positions of its
first input recover both input words. Their row inner product has
absolute value at most one since each generator has norm at most one.
For $2|\ell|<\rho<R$ the series therefore converges in $L^1$, since
$\|V_n\|_{L^2}\le 2^n\rho^{-n}p_\rho(\mathbb E S_\tau)^{1/2}$.
Consequently $\beta\ge\ell^2/4$ whenever $2|\ell|<R$, which gives
the upper bound, including the infinite endpoint.
\end{proof}

For a clock independent of the Brownian path, the fixed-time Brownian
signature formula recalled in \cite{LyonsNiExit} gives an exact check.
Its expected tensor at time $t$ is
$\exp(\tfrac t2\sum_i e_i\otimes e_i)$, with nonnegative word
coefficients. Conditioning on the clock and summing by Tonelli yields
$p_r(\mathbb E S_\tau)=\mathbb E\exp(dr^2\tau/2)$ and hence
$R=\sqrt{2\beta/d}$ for this class.

\begin{corollary}[The expected signature of Brownian exit]\label{sig:thm:stopped}
At first exit from any bounded open set in $\mathbb R^d$, for every
interior starting point, the expected signature determines the signature
law among all word-integrable formal group-like laws. Its radius is
finite and positive for every $d\ge2$, and infinite for $d=1$.
\end{corollary}
\begin{proof}
Let $D$ be the bounded open set. In each unit interval there is a
fixed positive probability that the Brownian increment exceeds the
diameter of $D$ in norm. The Markov property thus gives a geometric
upper tail for $\tau_D$. To bound survival below, take a cube
$Q=x+(-a,a)^d$ around the starting point with closure contained in $D$.
Its positive Dirichlet eigenfunction
$\phi(y)=\prod_j\cos(\pi(y_j-x_j)/(2a))$ is bounded by $\phi(x)=1$.
It\^o's formula stopped on the cube gives
\[
 \mathbb P_x(\tau_D>t)\ge
 \mathbb E_x[\phi(B_t)\mathbf1_{\{\tau_Q>t\}}]
 =\exp(-d\pi^2t/(8a^2)).
\]
Thus $0<\beta<\infty$. Theorem~\ref{br42:thm:radius} gives the two-sided
radius conclusion for $d\ge2$. In dimension one,
$S_{\tau_D}^{(n)}=(B_{\tau_D}-B_0)^n/n!$ with bounded displacement.
Theorem~\ref{sig:thm:positive-radius} gives determinacy in every dimension.
\end{proof}

\part{Spectral tests for processes and operators}

The last two sections use spectral data in two ways. The spectral
density of a stationary Gaussian process controls correlations of
separated observations. Continuous-time processes with the same white-noise
law on the integer lattice can be monofractal or multifractal on every
interval (Theorem~\ref{r76:ayache}). For a fixed self-adjoint operator its entire
list of possible basis diagonals retains the multiplicities at forced
spectral cuts. Section~\ref{r54:diagonals} follows these multiplicities
through the published spectral reductions.
\section{Stationary Gaussian spectra and invariant measures for weighted shifts}\label{r54:stationary}

\subsection{A spectral characterization of interlaced mixing}
\label{mix40:sec}
Bradley's Questions~8 and~9 ask for a spectral characterization of
interlaced linear and Gaussian maximal correlation mixing
\cite{BradleyMixing2005}. His Theorem~6.3 and the implication following
(6.8) give continuity as a necessary condition. Theorem~6.5 gives
sufficiency for strictly positive continuous densities. The remaining
issue is the behavior at the zeros.
For a spectral density $g$ on the circle, write $r_g^*(n)$ for the
supremum of $|\int\overline A B g|/(\|A\|_g\|B\|_g)$ over Fourier
polynomials whose coefficient supports have distance at least $n$,
with normalized Haar measure and $\|A\|_g^2=\int |A|^2g$.
For a real stationary Gaussian process with even density $g$, this
equals the interlaced maximal correlation, by the Gaussian first-chaos
identity \cite[(7.1)]{BradleyMixing2005}. A spectral measure of mass zero gives the zero process and has all coefficients zero. A positive measurable function $L$ is slowly varying at infinity if $L(cx)/L(x)\to1$ for each $c>0$. The full characterization for nonzero spectral measures is
Theorem~\ref{mix43:full}. Theorem~\ref{mix:thm:exactlag43} determines the
first lag of strict inequality from the same local orders.

\paragraph{A periodic compression of the spectral density.}
There is an exact matrix version of the spectral test. For a density $g>0$
almost everywhere and a period $m$, set $\theta_l=(t+2\pi l)/m$,
$0\le l<m$, and
\[
 G_{ij}(t)=m^{-1}\sum_{l=0}^{m-1}g(\theta_l)e^{\mathrm i(j-i)\theta_l},
 \qquad 0\le i,j<m .
\]
Then $r_g^*(n)$ equals
\[
 \sup_{m\ge2n}\ 
 \max_{\substack{I,J\ne\varnothing\\
        \operatorname{dist}_{\mathbb Z/m\mathbb Z}(I,J)\ge n}}
 \operatorname*{ess\,sup}_{t\in[0,2\pi)}
 \|G_{II}(t)^{-1/2}G_{IJ}(t)G_{JJ}(t)^{-1/2}\|.
\]
Indeed, split a Fourier polynomial on $I+m\mathbb Z$ into
$\sum_{j\in I}e^{\mathrm i j\theta}a_j(m\theta)$.
Folding its scalar product with the corresponding polynomial for $J$
gives the integral against $G_{IJ}$. The matrix $G$ is positive definite
almost everywhere by the discrete Fourier transform. Polynomial vectors
are dense in the weighted block norm: truncate a measurable vector and
approximate it in the scalar $L^2$ space for the integrable weight
$\operatorname{tr}G$. Conjugating by the block square roots gives the
displayed multiplication-operator norm. Conversely, any finite pair of
supports at distance at least $n$ lies, after translation, in a single
period with cyclic distance still at least $n$, by taking the period
larger than the total diameter plus $n$. This proves equality.

For $r_f(n)$ restrict the same supremum to coefficients supported in
$(-\infty,0]$ and $[n,\infty)$, respectively.

\begin{theorem}[The first strict past--future and interlaced angles]\label{mix:thm:exactlag43}
Let $P(z)=\prod_{j=1}^J(z-e^{i\theta_j})^{k_j}$, with distinct roots,
$k_j\ge1$, and $D=\deg P$. Suppose $f=|P|^2w$, where $w$ is positive
and measurable, bounded above and away from zero on compact subsets
off the roots, and for positive measurable slowly varying $L_j$,
\[
 w(\theta_j+u)\asymp L_j(1/|u|)\qquad(u\to0).
\]
The ordinary past--future coefficient $r_f(n)$ uses the supports
$(-\infty,-1]$ and $[n-1,\infty)$. Both $r_f(n)$ and $r_f^*(n)$
first fall strictly below one at $n=D+1$. If in addition $w$ is
continuous off the roots and $w(\theta_j+u)\sim c_jL_j(1/|u|)$,
$0<c_j<\infty$, then $r_f^*(n)\to0$.
\end{theorem}
\begin{proof}
For $n\le D$ take the $D$ independent functions
$(1-ze^{-i\theta_j})^{-l}$, $1\le l\le k_j$, in $L^2(f)$.
Their squared integrability follows from Potter's bound
\cite[Theorem~1.5.6]{BGTregular}. Some nonzero linear combination $G$
has its first $n-1$ Taylor coefficients zero. As $t\uparrow1$,
$G(tz)$ and $G(z/t)$ converge to $G$ in $L^2(f)$ by dominated
convergence, since $|1-tz|^2=(1-t)^2+t|1-z|^2$ on the circle.
Their respective Fourier series have supports in $[n-1,\infty)$ and
$(-\infty,-1]$. Finite truncations give $r_f(n)=r_f^*(n)=1$.

Write $r_\infty(g)=\lim_n r_g^*(n)$. A Laurent polynomial filter $p$
of exponent diameter $d$ gives $r_{|p|^2g}^*(n+d)\le r_g^*(n)$
by support separation. A relative change of weight by at most
$\delta<1$ takes a correlation bound $r$ to at most
$(r+\delta)/(1-\delta)$ by Cauchy--Schwarz. Uniformly approximating
$\sqrt h$ by Laurent polynomials and then applying the same argument to
$1/h$ shows $r_\infty(hg)=r_\infty(g)$ for positive continuous $h$.
When $J=0$, equivalence to the constant norm gives $r_f^*(1)<1$,
and the invariance gives the assertion for continuous $w$.

For coefficient supports at distance $D+1$ or more, put $u=PA$,
$v=PB$ for the respective Fourier polynomials $A,B$. Merge overlapping
intervals $[i,i+D]$ over their support sites $i$. Each resulting
interval has one color. Each filtered restriction is orthogonal in
coefficient space to all $\chi_j(t)=e^{-it\theta_j}$, since each
translate of $P$ vanishes at every root. Let $S$ project on each
first-color interval onto the orthogonal complement of these characters,
and vanish elsewhere. The unweighted unitary $T=1-2S$ satisfies
$Tu=-u$, $Tv=v$. We establish bounds uniform in the selected intervals.

Put $R_j=|1-e^{i(\theta-\theta_j)}|^2$ and fix $0<s<1/2$.
We claim $T$ is bounded on each $L^2(R_j^{\pm s})$. Modulate the
chosen root to zero. The squared positive-weight norm is comparable to
$H_{\mathbb Z}(x)$, where
$H_I(x)=\sum_{l<t\in I}|x_t-x_l|^2/(t-l)^{1+2s}$, by Parseval
and $\sum_{q\ge1}(1-\cos q\theta)/q^{1+2s}\asymp |\theta|^{2s}$.
For each finite interval $I$, with $\bar x_I$ its mean, we need
\[
 \sum_{t\in I}\frac{|x_t-\bar x_I|^2}
  {\operatorname{dist}(t,I^c)^{2s}}\le C_s H_I(x).
\]
Translate $I$ to $\{1,\ldots,N\}$. For $N\ge2$ put
$m=\lfloor\log_2 N\rfloor-1$. Take consecutive shells
$J_h=[2^h,2^{h+1}-1]\cap\mathbb Z$ for $0\le h<m$ and
$J_m=[2^m,N]\cap\mathbb Z$. Their lengths are between $2^h$ and
$3\cdot2^h$, and $|J_m|\ge N/2$. Writing $m_h$ for the mean on
$J_h$, pairwise variance within each shell and Jensen's inequality
across neighboring shells give
\[
 \sum_{h=0}^m2^{-2sh}\sum_{t\in J_h}|x_t-m_h|^2
 +\sum_{h=0}^{m-1}2^{(1-2s)h}|m_h-m_{h+1}|^2\le C_s H_I(x).
\]
Indeed the relevant distances are at most $8\cdot2^h$, and each pair
is counted once. For $\sigma=1/2-s>0$, telescoping $m_h-m_m$ and
convolving with the summable kernel $(2^{-\sigma k})_{k\ge0}$ give
\[
 \sum_{h=0}^m2^{2\sigma h}|m_h-m_m|^2
 \le (1-2^{-\sigma})^{-2}
     \sum_{h=0}^{m-1}2^{2\sigma h}|m_h-m_{h+1}|^2 .
\]
Finally $\sum_I|x_t-\bar x_I|^2\le N^{2s}H_I(x)$ by pairwise
variance, so $|m_m-\bar x_I|^2\le2N^{2s-1}H_I(x)$.
Since $\sum_{t=1}^N t^{-2s}\le C_s N^{1-2s}$, these estimates prove
the displayed inequality with $t$ in place of
$\operatorname{dist}(t,I^c)$. Reflection gives the other endpoint.
For $N=1$ the left side is zero.

On each chosen interval set $y_t=x_t-\bar x_I$, extending the
collection by zero on the complement. Its internal energies sum to
at most $H_{\mathbb Z}(x)$. For pairs from different intervals use
$|y_t-y_l|^2\le2|y_t|^2+2|y_l|^2$ and
\[
 \sum_{l\notin I}|t-l|^{-1-2s}
 \le C_s\operatorname{dist}(t,I^c)^{-2s}.
\]
The interval estimate thus gives $H_{\mathbb Z}(y)\le C_sH_{\mathbb Z}(x)$
for every collection of disjoint intervals.

To remove the remaining characters on $I$, center each one by
subtracting its mean. For every nonzero frequency difference $\omega$,
summation by parts and bounded geometric sums give
\[
 N^{-1/2}\left|\sum_{t\in I}y_te^{it\omega}\right|
 \le C_\omega N^{-1/2}
 \left(2\|y|_I\|_2+\sum_{t,t+1\in I}|x_{t+1}-x_t|\right)
 \le C_\omega H_I(x)^{1/2}.
\]
The centered-character Gram matrix divided by $N$ converges to the
identity uniformly in the position of the interval by the same
geometric-sum bound. Thus the squared unweighted norm of the
projection onto their span is at most $C H_I(x)$ for all large $N$.
The bounded lengths follow from $\|y|_I\|_2^2\le N^{2s}H_I(x)$.
Since the projected sequences have disjoint coefficient supports and
$R_0^s\le4^s$, their sum has squared $L^2(R_0^s)$ norm at most
$C\sum_I H_I(x)$. This proves the positive-weight bound for $T$.
Unweighted self-adjointness gives its negative-weight bound by duality.

Around each root use disjoint arcs of radius $a$ and their annuli
$h-1\le\log_2(a/|\theta-\theta_j|)<h$, $h\ge1$. Keep the complement
as one bulk block. Between blocks $\alpha,\gamma$ let $d$ be $|h-l|$
on one branch, $h+l$ on different branches, and $h$ to the bulk
($d=0$ on the bulk diagonal). Their Fourier projections satisfy
$\|Q_\alpha T Q_\gamma\|\le C e^{-c d}$. Indeed the two signed
$R_j$ estimates give $C2^{-s|h-l|}$ on one branch; different branches
give both $C2^{-sh}$ and $C2^{-sl}$ by using the output and input
roots in turn. The bulk bounds are $C2^{-sh}$. There are uniformly
boundedly many blocks at any fixed distance. On each annulus
$w^{1/2}\asymp L_j(2^h/a)^{1/2}$ by uniform slow variation.
Potter's bound makes the ratio of any output and input weights at most
$C_\eta e^{\eta d}$, for each $\eta>0$, including between branches.
Summation of block rows and columns with $\eta<c$ gives
$\|T\|_{L^2(w)}\le C'$. Finite block truncations converge in both
norms for Fourier polynomials, and the extensions agree by convergence
in measure. Hence $\|u\|_w^2+\|v\|_w^2\le C''\|u+v\|_w^2$, by
applying $(1\pm T)/2$. Normalizing and choosing the relative phase
gives $r_f(D+1)\le r_f^*(D+1)<1$.

For the asymptotic assertion put $b_{jh}=\frac12\log L_j(2^h/a)$.
Given $\delta>0$, choose $H$ with $|b_{j,h+1}-b_{jh}|\le\delta$
for $h\ge H$ on every branch. Define $\beta$ on each arc by linear
interpolation in $\log_2(a/|\theta-\theta_j|)$ of the values
$b_{j,\max(h,H)}-b_{jH}$ at nonnegative integers $h$, and set it
to zero on the bulk. For $U=e^\beta$, uniform slow variation shows
that $w/U^2$ extends continuously to the positive value
$c_j e^{2b_{jH}}$ at each root. The oscillation of $\beta$ between any
two blocks is at most $\delta(d+2)$. Comparing both multiplication
factors to the exponential of one input-block value and its reciprocal,
the corresponding block of $UTU^{-1}-T$ has norm at most
$C(e^{2\delta(d+2)}-1)e^{-cd}$. Row and column sums give
$\|UTU^{-1}-T\|\le e_\delta\to0$, uniformly in the intervals.
The same truncation argument identifies this extension on $Uu,Uv$.
By unitarity $\langle TUu,TUv\rangle=\langle Uu,Uv\rangle$ while
$\langle UTu,UTv\rangle=-\langle Uu,Uv\rangle$. Their difference
is bounded by $(2e_\delta+e_\delta^2)\|Uu\|\|Uv\|$, so
$r_{|P|^2U^2}^*(D+1)\le e_\delta+e_\delta^2/2$. Continuous-multiplier
invariance and $\delta\downarrow0$ give $r_\infty(f)=0$.
\end{proof}

\Needspace{8\baselineskip}
\begin{proposition}[A bound on the spectral zeros]\label{mix42:zero-count}
Let $f\ge0$ be continuous and nonzero on the circle with normalized
Haar measure $d\mu$. If $r=r_f^*(n)<1$, each zero $z$ satisfies
\[
 \int \frac{f(\zeta)}{|\zeta-z|^2}\,d\mu(\zeta)
       \le \frac{3n\|f\|_\infty}{1-r}.
\]
There are finitely many zeros. If $k_z$ is the largest integer $k$
for which $f(\zeta)|\zeta-z|^{-2k}$ is integrable, then
$k_z\ge1$ and $\sum_z k_z\le n-1$. This bound is sharp.
\end{proposition}
\begin{proof}
In the periodic compression above choose a grid of $M=2L$ frequencies
starting at a zero $z$, with $L>n$. After diagonal modulation the
quadratic form $Q$ on the residues has Fourier eigenvalues
$f(ze^{2\pi i l/M})$, so the constant vector has zero energy.
Take $a,b$ to be the indicators of $[0,L-n-1]$ and $[L,2L-n-1]$.
Their cyclic separation gives $Q(a+b)\ge(1-r)(Q(a)+Q(b))$, also
for singular fibers: apply the defining bound to periodic lifts
multiplied by the same Fej\'er-localizing polynomial, and pass to the
limit by continuity. Since $a+b$ differs from the constant vector on
$2n$ sites, $Q(a+b)\le2n\|f\|_\infty$. The half-cycle indicator
$h=\mathbf1_{[0,L-1]}$ then satisfies
$Q(h)\le2Q(a)+2Q(h-a)\le6n\|f\|_\infty/(1-r)$.
Its finite Fourier transform vanishes at nonzero even modes and gives
\[
 Q(h)=\frac4M\sum_{\substack{0\le l<M\\l\ {\rm odd}}}
 \frac{f(ze^{2\pi i l/M})}{|1-e^{2\pi i l/M}|^2}.
\]
With a continuous cutoff vanishing near $z$, this last sum converges
to twice the corresponding cutoff integral. Increasing the cutoffs
to one off $z$ proves the first assertion.

Take finitely many distinct zeros $z_j$ and positive integers $k_j$
for which the indicated integrals are finite. The rational functions
$(1-\zeta/z_j)^{-l}$, $1\le l\le k_j$, are linearly independent in
$L^2(f\,d\mu)$ since $f$ is positive on an arc. If $\sum_j k_j\ge n$,
a nonzero linear combination $R$ has its first $n-1$ Taylor coefficients
at zero equal to zero. For $t\uparrow1$, both $R(t\zeta)$ and
$R(\zeta/t)$ converge to $R$ in $L^2(f\,d\mu)$ by domination:
$|1-tu|^2=(1-t)^2+t|1-u|^2$ on $|u|=1$. Their uniformly convergent
series on the circle have exponents $\ge n-1$ and $\le-1$,
respectively. Finite truncations thus give correlations arbitrarily
close to one across distance $n$, a contradiction. Applying this to
all finite choices of divisors proves the count. Finally, for $f=|p|^2$
with all $D$ roots of $p$ on the circle counted with multiplicity, this
argument gives $r_f^*(n)=1$ for $n\le D$. For $n>D$ it vanishes since
$|p|^2$ has Fourier bandwidth $D$.
\end{proof}

\paragraph{A spectral criterion without a hypothesis at the zeros.}
\begin{theorem}[Interlaced spectral mixing]\label{mix43:full}
A nonzero spectral measure of a centered complex weakly stationary sequence
satisfies $r^*(n)\to0$ if and only if it has a continuous density $f$
with finitely many zeros, and at each zero $\theta_j$ there exist an integer
$k_j\ge1$ and a positive slowly varying $L_j$ such that
\[
 f(\theta_j+u)\sim |u|^{2k_j}L_j(1/|u|)\qquad (u\to0).
\]
For a real stationary Gaussian sequence this characterizes
$\rho^*$-mixing. These are Questions~8 and~9 of \cite{BradleyMixing2005}.
\end{theorem}
\begin{proof}
Sufficiency is Theorem~\ref{mix:thm:exactlag43}. Conversely interlaced
mixing gives a continuous density \cite[Theorem~6.3 and (6.8)]{BradleyMixing2005},
and Proposition~\ref{mix42:zero-count} gives finitely many zeros.
Translate one zero to $0$. We show that all its normalized rescalings
tend to a single monomial.

\emph{Local polynomial limits.}
First, every continuous nonnegative symbol $g$ has a trigonometric
polynomial of degree at most $3n$ within
$64\|g\|_\infty r_g^*(n)$ in supremum norm. Indeed the convolution
operator $T_g$ on $\ell^2(\mathbb Z)$ satisfies
$\|P_A T_gP_B\|\le\|g\|_\infty r_g^*(n)$ on sets at distance at least
$n$. Give consecutive length-$n$ blocks independent fair bits. Put a block
in $A$ when the three centered bits read $010$, and in $B$ when they
read $000$. These sets are separated by at least $n$. For two blocks
whose indices differ by at least three, the probability of their inclusion
in $A,B$ respectively is $1/64$. Thus $T_g-64\E(P_A T_gP_B)$ has
bandwidth at most $3n-1$. Averaging its $n$ translates gives a Laurent
operator of this bandwidth and the claimed error.

Suppose $s_j\downarrow0$, $c_j>0$, and
$f(x)/c_j\le C(1+|x|/s_j)^D$ on $[-\pi,\pi]$, for some finite $D$.
We claim that $f(s_j\,\cdot)/c_j$ is locally uniformly precompact with
polynomial limits. On $[-H,H]$ take
$m=\lfloor ((1+H)s_j)^{-1}\rfloor$ and
$D_m(x)=m^{-1}\sum_{h=0}^{m-1}e^{ihx}$. For an integer $2K>D$ the
symbols $g_j=f|D_m|^{2K}/c_j$ are uniformly bounded, since
$|D_m(x)|\le\min(1,\pi/(m|x|))$. The finite-filter inequality in
Theorem~\ref{mix:thm:exactlag43} gives
$r_{g_j}^*((K+1)m)\le r_f^*(m)\to0$. The band estimate approximates
$g_j$ to $o(1)$ uniformly by bounded polynomials of degree $O(m)$.
By Bernstein's inequality their rescalings to $s_j u$ are equicontinuous.
On $[-H,H]$, $|D_m(s_j u)|$ converges uniformly to
$|\sin v/v|$, $v=u/[2(1+H)]$, which is bounded away from zero.
Division and a diagonal subsequence give the compactness claim.

For any limit $q$ and smooth test functions $a,b$ of disjoint compact
supports on the real line, test the correlation bound on sequences
$s_j a(s_j l)$ and $s_j b(s_j l)$. Their separation tends to infinity.
After the change $x=s_j u$, Poisson summation gives the limiting Fourier
transforms $\widehat a(u),\widehat b(u)$. The periodized transforms
are bounded by $C_N(1+|u|)^{-N}$ on $|u|\le\pi/s_j$ for every $N$,
by rapid decay of the translates at $2\pi l/s_j$. Dominated convergence
after division by $c_js_j$ gives
$\int q\overline{\widehat a}\widehat b=0$. Taking $a$ to be an
approximate identity near zero shows that the inverse Fourier transform
of $q$ is supported at the origin. The structure of point-supported
distributions and the growth bound make $q$ a nonnegative polynomial
of degree at most $D$.

\emph{A growth bound at every scale.}
Write $M(s)=\max_{|x|\le s}f(x)$. Choose $n$ with $r_f^*(n)<1$ and
put $D=2n+1$. Proposition~\ref{mix42:zero-count} implies
$\int_{-\delta}^{\delta} f(x)|x|^{-2n}\,dx=\infty$ for every $\delta>0$.
Thus $M(s)/s^D$ is unbounded near zero. Write
$H(s)=\max_{s\le t\le\pi}M(t)/t^D$ and call $t$ a record if
$M(t)/t^D=H(t)$. At records the normalized profile $g_t(u)=f(tu)/M(t)$
is bounded by $\max(1,|u|^D)$ for $|tu|\le\pi$. Its limiting polynomials
satisfy $P(0)=0$ and $\max_{[-1,1]}P=1$. Their degree is at most
$D-1$ by nonnegativity on the whole line and oddness of $D$.
Finite-dimensional coefficient bounds on $P(r\,\cdot)$ give
$\max_{[-r,r]}P\ge b_Dr^{D-1}$ uniformly for $0<r<1$.
Fix $r$ small enough that $b_Dr^{D-1}>3r^D$.
Compactness then gives $M(rt)>2r^DM(t)$ for all small records.
Every maximizer $t$ defining $H(s)$ is a record and tends to zero with
$s$. Necessarily $rt<s$, since otherwise $rt$ would beat the maximum.
It follows that $s^DH(s)\le M(t)\le M(s)/(2r^D)$. Consequently
$M(us)\le C u^D M(s)$ for $u\ge1$, $us\le\pi$. The polynomial
compactness just proved now applies to $g_s$ at \emph{all} small scales.

\emph{Uniqueness of the power.}
Put $d=D-1$. For a nonzero polynomial $P(u)=\sum_{j=0}^d a_j u^j$ set
$V(P)=\sum j a_j^2/\sum a_j^2$. Under $u\mapsto e^{-h}u$ its derivative
in $h$ is minus twice the variance of $j$ for weights proportional to
$a_j^2e^{-2hj}$. It strictly decreases unless $P$ is a monomial.
For $s=e^{-t}$ let $v(t)$ be $V$ of the polynomial interpolating $g_s$
at $d+1$ fixed nodes of $[-1,1]$. Compactness and normalization ensure
this polynomial is nonzero for large $t$. If $g_s\to P$ along a sequence
and $h\to h_0$, then $g_{e^{-h}s}$ converges to the normalized dilation
of $P$ by $e^{-h_0}$. Hence
$\sup_{0\le h\le1}(v(t+h)-v(t))\le o(1)$. On any closed range of
$v(t)$ avoiding the integers, $v(t+1)\le v(t)-\delta$ for some
$\delta>0$ and all large $t$, by compactness and strict decrease.
If $v(m)$ at integer times had two distinct limit extremes, choose a
noninteger $c$ strictly between them. In a short interval below $c$
there is strict decrease, while from below that interval a single
upward step tends to zero in size. Thus $v(m)$ could never cross $c$
upwards at late times. This contradiction gives convergence, also at
real times by the bound for steps of length at most one.
For any tangent $P$ the value of $V$ on it and on its unit-time
dilation must now be equal. Thus $P$ is a monomial. Its exponent is the
single limit of $v$, and normalization, nonnegativity, and $P(0)=0$
give $P(u)=u^{2k}$ for some $k\ge1$. Every subsequence has this same
limit. Taking $L(x)=x^{2k}f(1/x)$ for large $x$ proves the two-sided
asymptotic and slow variation asserted in the theorem. The Gaussian
maximal-correlation identity \cite[(7.1)]{BradleyMixing2005} proves
the last statement.
\end{proof}

\begin{theorem}[Limiting correlations for two one-sided spectral profiles]
\label{r63:brad-asym}
Let $f$ be a continuous nonzero density on the circle with finitely
many zeros $\theta_j$. Suppose that on each side of every zero
\[
 f(\theta_j\pm s)\sim s^{2k_{j,\pm}}L_{j,\pm}(1/s)\qquad(s\downarrow0),
\]
where the $k_{j,\pm}$ are positive integers and the $L_{j,\pm}$ are
positive and slowly varying. Then
\[
 \lim_{n\to\infty}r_f(n)=\lim_{n\to\infty}r_f^*(n)
 =b,\qquad
 b=\max_j\limsup_{s\downarrow0}
 \frac{|f(\theta_j+s)-f(\theta_j-s)|}{f(\theta_j+s)+f(\theta_j-s)} .
\]
The empty maximum is zero. In particular, unequal orders on the two
sides of any zero give $r_f(n)=r_f^*(n)=1$ for all $n$. When the
orders agree and $L_{j,+}/L_{j,-}\to c_j\in(0,\infty)$ at every
zero, the limit is $\max_j |c_j-1|/(c_j+1)$.
\end{theorem}
\begin{proof}
For $b<1$ fix $b<b'<1$. On sufficiently small disjoint arcs around
the zeros let $f_0$ be the even average about that zero. On a smaller
arc keep this average, interpolate continuously to $f$ using a convex
combination in the outer part, and take $f_0=f$ elsewhere. Then
$1-b'\le f/f_0\le1+b'$ off the zeros. The two orders at each zero
agree, since otherwise their ratio tends to zero or infinity.
The average of the two slowly varying factors is slowly varying, so
$r_{f_0}^*(n)\to0$ by Theorem~\ref{mix43:full}.
For separated polynomials normalize their $f$-norms equally and
align their relative phase to make their $f$-inner product
nonnegative. With $r_n=r_{f_0}^*(n)$, comparison of the squared
norms of their sum and difference gives
\[
 \frac{1+r_f^*(n)}{1-r_f^*(n)}
     \le \frac{1+b'}{1-b'}\,\frac{1+r_n}{1-r_n}
 \qquad(r_n<1).
\]
Indeed the base norm ratio is bounded by $(1+r_n)/(1-r_n)$ even
when the two base norms are unequal. Let $n\to\infty$ and $b'\downarrow b$.
For $b=1$ the upper bound is immediate.

For the lower bound modulate the Fourier coefficients to move a
maximizing zero to zero, and choose $s\downarrow0$ along a sequence
on which the signed side difference divided by the side sum tends
to $\eta$ with $|\eta|=b$. With $c_s=(f(s)+f(-s))/2$, for every
$u\ne0$ one then has
\[
 f(su)/c_s\longrightarrow
       |u|^{2k}(1+\eta\,\operatorname{sgn}u).
\]
Here $k$ is the common order when the orders agree and the lesser
one otherwise; a side with vanishing normalized weight makes no
contribution. Potter's bound gives
$f(su)/c_s\le C(1+|u|)^D$ on $|su|\le\pi$ for some $D$:
near zero apply it separately to each side relative to its value
at $s$, using an error exponent below both positive orders; away
from zero use boundedness of $f$ and the polynomial lower bound on
$c_s$. For real $a\in C_c^\infty((-\infty,0))$ and
$d\in C_c^\infty((0,\infty))$, sample the coefficients
$s a(sl),s d(sl)$ as in the proof of Theorem~\ref{mix43:full}.
Their past--future separation tends to infinity. Poisson summation
and the polynomial bound give convergence of their normalized
correlation to the Fourier correlation for the displayed limiting
density. The squared norms reduce by Plancherel to those of
$a^{(k)},d^{(k)}$, since their squared Fourier moduli are even.
The constant part of the cross inner product vanishes by disjoint
supports. Its remaining modulus is $|\eta|$ times that of the
Hilbert-transform pairing between $a^{(k)}$ and $d^{(k)}$.

On either half-line the $k$th derivatives of real compactly supported
smooth functions are dense in real $L^2$, since an annihilator has
distributional $k$th derivative zero and no nonzero polynomial there
is square-integrable. After reflecting the negative half-line, the
cross Hilbert kernel is $1/[\pi(x+y)]$ on $x,y>0$, up to sign.
In logarithmic coordinates the unitary map $h(x)\mapsto e^{t/2}h(e^t)$
turns it into convolution with $1/[2\pi\cosh(t/2)]$. This positive
kernel has integral one and operator norm one, as seen by testing
normalized indicators of long intervals. The limiting past--future
correlations therefore have supremum $|\eta|$. Since $r_f(n)\le
r_f^*(n)$ this proves both equalities. If there are no zeros,
Theorem~\ref{mix43:full} gives the result directly.
\end{proof}

For example, the even density
$f(\theta)=\cos^2\theta\,(1+\eta\,\operatorname{sgn}(\cos\theta))$,
$0\le\eta<1$, gives both limits equal to $\eta$ for its real stationary
Gaussian process. The even density
$f(\theta)=(\max(\cos\theta,0))^2+(\max(-\cos\theta,0))^4$
has just the two zeros $\pm\pi/2$ and gives both coefficients equal
to one at every lag.

\subsection{The singularity spectrum of a stationary Gaussian process}
\label{r76:regularity}
For a continuous function $x$ on $\R$, let $\alpha_x(t)$ be the supremum of
the $h\in[0,1]$ for which the oscillation on $[t-r,t+r]$ is $O(r^h)$.
For a stationary Gaussian process with spectral density $g^2$, Ayache asks
whether the critical Besov index
\[
 s=\sup\{u\ge0:\mathcal F^{-1}g\in B^u_{2,\infty}(\R)\}\in(0,1)
\]
and the almost sure identity $\alpha_X(t)=s$ at each fixed $t$ force
almost sure monofractality. In Question~1.6 of the author's manuscript,
this means that the closure of the indices with positive-dimensional
level sets has empty interior \cite{r76:ayjfa}. Ayache's Lille seminar asks whether
the same hypothesis forces equality at every point on a common event
\cite[p.~34, (4.15)]{r76:aytalk}. Ayache obtained a positive-dimensional
exceptional set with positive probability for multifractional Brownian
motion with nonstationary increments \cite[long version, Theorem~2.5(iv)]{r77:ay13}.
Esser, Jaffard and Vedel obtained almost sure full dimension in a Gaussian
wavelet example whose fixed-point exponent varies with location
\cite[Proposition~3.7 and Theorem~5.15]{r76:ejv}.
Here the carrier at frequency $R_j$ has amplitude $R_j^{-s}$ and its
envelope varies on the spatial scale $R_j^{-\kappa}$. Varying that scale
separates fixed-time regularity from the singularity spectrum of a path.

\begin{theorem}[Bandwidth and the singularity spectrum]\label{r76:ayache}
Fix $0<s<1$. There are centered continuous stationary Gaussian processes
$X_\kappa$ on $\R$, $0\le\kappa\le1$, with absolutely continuous spectra,
critical Besov index $s$, and independent observations on sets at distance
greater than $1/2$. All have the same law on the integer lattice, and the
exponent at each fixed time equals $s$ almost surely. Their increment variances satisfy
$V_\kappa(v)/V_0(v)\to1$ as $v\to0$ exactly when $\kappa\le s$.

Write $X=X_\kappa$. If $\kappa=0$, almost surely $\alpha_X(t)=s$ at every
$t$. For each fixed $0<\kappa\le1$ set $h_*=\min\{1,s+\kappa,s/\kappa\}$. On one
event of probability one, for every nonempty open interval $I$,
\[
 \alpha_X(I)=[s,h_*],\qquad
 \dim_{\mathrm H}\{t\in I:\alpha_X(t)=s+\delta\}
 =\frac{\kappa-\delta}{\kappa+\delta}\qquad(0\le\delta\le h_*-s).
\]
The paths have constant exponent also at $\kappa=1$. For $0<\kappa<1$
their smoother points form a Lebesgue null set of Hausdorff dimension one
in each $I$. Each such $\kappa$ answers Ayache's two questions negatively. In particular
$X_s$ realizes the entire interval $[s,\min(2s,1)]$ while its increment
variance is asymptotic to that of the monofractal process $X_0$.
\end{theorem}

\begin{proof}
\emph{Carrier frequencies and bandwidths.}
Fix a real $\phi\in C_c^\infty([-1/4,1/4])$ of $L^2$ norm one.
Convolving real Gaussian white noise with $\phi$ gives a stationary
process $A$ of variance one and spectral density proportional to
$|\widehat\phi|^2$. Its restrictions to sets at distance greater than
$1/2$ are independent. It has smooth paths: each mean-square derivative
obtained by differentiating $\phi$ has increment variance $O(|v|^2)$;
Gaussian moments give continuous versions, and integration identifies
successive path derivatives. Set $R_j=2^{(j+2)!}$, $L_j=\log R_j$ and $W_j=R_j^\kappa$.
For independent copies $A_j,B_j$ put
\[
 Y_j(t)=A_j(W_jt)+iB_j(W_jt),\qquad
 X(t)=\sum_{j\ge1} R_j^{-s}\Re(e^{iR_jt}Y_j(t)).
\]
Circular symmetry makes each modulated process stationary. Its spectral
density is a sum of translates, by $\pm R_j$, of the rescaled envelope
density. The measures of the independent summands have summable masses
$R_j^{-2s}$. All covariances vanish at distances greater than $1/2$,
since $W_j\ge1$. Thus the integer samples are independent Gaussians of
variance $\sum_j R_j^{-2s}$ for every $\kappa$.

On each compact interval $K$, almost surely for all sufficiently large $j$,
\begin{equation}\label{r76:envelopes}
 \sup_K|Y_j|\le L_j,\qquad \sup_K|Y'_j|\le W_j L_j .
\end{equation}
On a cell of length $W_j^{-1}$ the supremum of a derivative of order
$k=0,1$ is bounded by its modulus at the left endpoint plus the integral
of the next absolute derivative. Its $L^p$ norm is at most
$C_\phi W_j^k\sqrt p$. The maximum over $O_K(W_j)$ cells costs a factor
$O_K(W_j)^{1/p}$. With $p=\lceil L_j\rceil$, the failure probability
is at most $2(C_K/\sqrt{L_j})^p$, which is summable. Hence the series
converges locally uniformly almost surely. Use countably many compacts to
obtain one event on $\R$.

The increment variance $V(v)$ is bounded by
$C\sum_j R_j^{-2s}\min(1,R_j^2 v^2)\le C_s|v|^{2s}$ for $|v|\le1$,
since $W_j\le R_j$. At $v=\pi/R_j$ the $j$th term contributes at least
$cR_j^{-2s}$: its carrier reverses sign, and the envelope correlation
tends to one for $\kappa<1$ and is zero for $\kappa=1$. For the
nonnegative square root $g$ of the spectral density, Plancherel gives
$V(v)=c\|k(\cdot+v)-k\|_2^2$, $k=\mathcal F^{-1}g$.
The first-difference criterion for $B^u_{2,\infty}$ at $0<u<1$ and
nesting in $u$ give critical index $s$.

To compare the variances, put $U(v)=2\sum_j R_j^{-2s}(1-\cos(R_jv))$
and $m=\max\{j:R_j\le |v|^{-1}\}$ for small $v\ne0$. The covariance
$\rho$ of $A$ satisfies $1-\rho(y)=\|\phi(\cdot+y)-\phi\|_2^2/2\le C y^2$.
For $\kappa\le s$ this bounds each summand of $V_\kappa-U$ in modulus
by $Cv^2$. The tail after $m+1$ is bounded by $C R_{m+2}^{-2s}\le Cv^2$,
since eventually $R_{m+2}^s\ge R_{m+1}>|v|^{-1}$. Thus
$|V_\kappa-U|\le C(m+2)v^2=o(U)$ by $U(v)\ge c v^2 R_m^{2-2s}$.
For $\kappa>s$, take $v_n=2\pi/R_n$. All terms of $U(v_n)$ from $n$
on vanish, giving $U(v_n)\le C v_n^2 R_{n-1}^{2-2s}$. The $n$th
variance term of $V_\kappa(v_n)$ alone is at least
$c v_n^2 R_n^{2\kappa-2s}$, since $1-\rho(y)\sim\|\phi'\|_2^2 y^2/2$
as $y\to0$ and $\rho(2\pi)=0$ at $\kappa=1$. Consequently
$V_\kappa(v_n)/V_0(v_n)\to\infty$ when $\kappa>s$.

\emph{Reading the exponent from the envelopes.}
Set $d(t)=\liminf_j(-\log|Y_j(t)|)/L_j$, with $-\log0=+\infty$.
With $s/0=+\infty$, almost surely at every $t$,
\begin{equation}\label{r76:read-exponent}
 \alpha_X(t)=\min\{1,s/\kappa,s+d(t)\}.
\end{equation}
By \eqref{r76:envelopes}, $d(t)\ge0$. When taking an increment of length
$|v|$, freezing the envelopes at $t$ costs at most
$C_K\sum_j R_j^{-s}L_j\min(1,W_j|v|)=O_{K,q}(|v|^q)$ for
$0<q<\min(1,s/\kappa)$, since $\sum_j L_jR_j^{\kappa q-s}<\infty$.
If also $q<s+d(t)$, the frozen coefficients are eventually bounded by
$R_j^{-q}$. Summing $R_j^{-q}\min(2,R_j|v|)$ on the two sides of
$R_j=|v|^{-1}$ proves the lower bound.

During one period $2\pi/R_j$ the frozen $j$th carrier has an increment
of modulus at least $R_j^{-s}|Y_j(t)|$. All other frozen carriers
contribute at most
$C_K R_j^{-1}\sum_{k<j}R_k^{1-s}L_k+C_K\sum_{k>j}R_k^{-s}L_k
\le R_j^{-1+o(1)}$ by the choice of $R_j$.
Taking $0<q<\min(\alpha_X(t),s/\kappa)$ gives $d(t)\ge q-s$.

It remains to show $\alpha_X(t)\le s/\kappa$ when $\kappa>s$.
In each lattice cell of length $L_j^2/W_j$, the real process
$\Re(e^{iR_jt}Y_j(t))$, sampled at spacing $W_j^{-1}$, gives at least $cL_j^2$
independent standard real Gaussians. Except with probability
$2e^{-c'L_j^2}$ these have a value above one and one below minus one.
Summing over the $O_K(W_j)$ full cells and over $j$ shows this eventually
throughout each compact. The oscillation from all other summands on
one such cell is at most
\[
 C_K (L_j^2/W_j)\sum_{k<j}L_kR_k^{1-s}
 +C_K\sum_{k>j}R_k^{-s}L_k=o(R_j^{-s}),
\]
by \eqref{r76:envelopes} and $\kappa>s$. Each point is within one cell
length of such an oscillation of order $R_j^{-s}$. As the cell length is
$R_j^{-\kappa+o(1)}$, the upper bound follows, proving
\eqref{r76:read-exponent}.

\emph{Distribution of the decay rates $d(t)$.}
At each fixed $t$, $\mathbb P(|Y_j(t)|\le R_j^{-v})=O(R_j^{-2v})$ for
all $v>0$, so $d(t)=0$ almost surely. Fubini gives this identity for
Lebesgue almost every $t$ on almost every path. To locate all larger
rates, for rational $v>0$ cover $K$ with cells of length
$b_j=R_j^{-v}/(2W_jL_j)$. Mark those with $|Y_j|\le2R_j^{-v}$ at
the left endpoint. By \eqref{r76:envelopes} these eventually cover
$\{|Y_j|\le R_j^{-v}\}\cap K$. Their expected number is
$O_K(L_j R_j^{\kappa-v})$. Almost surely none occur eventually if
$v>\kappa$. Otherwise the sum over $j$ of their expected $q$th
length costs is finite for $q>(\kappa-v)/(\kappa+v)$. The limsup
cover and rational $v\uparrow\delta$ give, simultaneously,
\[
 d(t)\le\kappa\quad(t\in\R),\qquad
 \dim_{\mathrm H}\{t\in K:d(t)\ge\delta\}
 \le(\kappa-\delta)/(\kappa+\delta)\quad(0<\delta\le\kappa).
\]
This proves the assertion at $\kappa=0$.

For $\kappa>0$, write $a=2^{-m}$,
$0\le m\le\lceil\kappa\log_2R_j\rceil$. Each full lattice cell
of length $H=L_j^3/(W_j a)$ in $K$ contains at least $cW_jH$ trial
intervals in its middle half, of length $W_j^{-1}$ and separated by
at least $W_j^{-1}$. On each trial the real part of $Y_j$ takes
opposite signs at its independent normal endpoints with probability
$1/2$. At its first zero between them, conditional on the real part,
the imaginary part is standard normal. Thus the modulus belongs to
$[2a,3a]$ at that zero with unconditional probability at least $ca$
including the sign change. Distinct trials are independent by the
range of dependence. The probability that a cell has no qualifying
point is at most $e^{-c L_j^3}$. There are $O_K((1+L_j)W_j)$ cells
over all $m$, so eventually all full cells qualify on one event.
On the interval of radius $r=a/(4W_j L_j)$ around a qualifying point,
\eqref{r76:envelopes} gives $a\le |Y_j|\le4a$.

For any real $0<\delta<\kappa$ choose
$m_j=\lfloor\delta\log_2 R_j\rfloor$. Then
$r_j=R_j^{-(\kappa+\delta)+o(1)}$ and
$H_j=R_j^{-(\kappa-\delta)+o(1)}$. Start inside $I$ with a closed
interval of radius $r_{j_0-1}$, for $j_0$ large. In the middle half of
each parent retain $n_j=\lfloor c r_{j-1}/H_j\rfloor\ge2$ children of
radius $r_j$, centered at qualifying points in distinct full cells whose
indices are divisible by three. A small absolute $c$ suffices. The
children are separated by at least $H_j$ and $d(t)=\delta$ at every
limit point. Splitting mass equally gives each level-$j$ interval mass
$M_j^{-1}$, $M_j=\prod_{k=j_0}^j n_k$. For $q<(\kappa-\delta)/(\kappa+\delta)$,
$q>0$, $M_j^{-1}\le C_q r_j^q$ since $\log R_{j-1}=o(\log R_j)$.
At length $r_j\le r\le r_{j-1}$ a test interval meets at most
$C(1+r/H_j)$ children. Its mass is bounded by
$C_q r_j^q+Cr/(r_{j-1}M_{j-1})\le C'_q r^q$. This gives the matching
dimension on $\{d=\delta\}\cap I$ by the mass distribution principle.
To attain $d=\kappa$ even though its level has dimension zero, replace
$\delta$ at level $j$ by $\kappa-\kappa/\sqrt j$. Then
$\log H_j=-(\kappa/\sqrt j)L_j+O(\log L_j)$,
$\log r_{j-1}\ge-2\kappa L_{j-1}-O(\log L_{j-1})$, and
$L_j/L_{j-1}=j+2$, so $H_j/r_{j-1}\to0$ and nesting still works.
The occupancy events were uniform in the dyadic amplitudes, hence the
conclusions hold simultaneously for all real $\delta$ and intervals.

By \eqref{r76:read-exponent}, below $h_*$ the $s+\delta$ level is
$\{d=\delta\}$. At $h_*$ it contains this exact rate and lies in
$\{d\ge h_*-s\}$, whose upper dimension is the same. The $s$-level
has full Lebesgue measure. For $0<\kappa<1$ the dimensions of the
levels $s+\delta$ tend to one as $\delta\downarrow0$, proving the final claim.
\end{proof}

\begin{corollary}[Order of contact of increment variances]\label{r78:variances}
For each $0<s<1$ and $0<\kappa<1$ there are two centered continuous
stationary Gaussian processes of finite dependence range with absolutely
continuous spectra, critical index $s$ and the same law on $\mathbb Z$.
Almost surely one has exponent $s$ everywhere and the other has on
every interval the multifractal spectrum of Theorem~\ref{r76:ayache}
at bandwidth $\kappa$. Their increment variances satisfy
\[
 0<\limsup_{v\to0}\frac{|V_1(v)-V_2(v)|}{|v|^{2s/\kappa}}<\infty .
\]
For every prescribed $M>0$, choosing $2s/\kappa>M$ gives
$|V_1(v)-V_2(v)|=o(|v|^M)$ as $v\to0$.
\end{corollary}
\begin{proof}
Use the carriers, weights and bandwidths of Theorem~\ref{r76:ayache}.
Write $\rho$ for the normalized covariance of the real envelope, supported
in $[-1/2,1/2]$, and periodize it to
$\rho_P(x)=\sum_{n\in\mathbb Z}\rho(x+2n)$. This is a smooth circle
covariance since its Fourier coefficients are nonnegative samples of
$|\widehat\phi|^2$ up to normalization. Use independent circular Gaussian
envelopes at each level with real-component covariances, respectively,
\[
 \rho(v)\rho(W_jv)\quad\hbox{and}\quad \rho(v)\rho_P(W_jv).
\]
Both products are positive definite and vanish for $|v|>1/2$. They have
spectral densities by integrability. Since $\rho_P=\rho$ on $[-1,1]$,
their modulated series have increment variances differing by at most
$C\sum_{W_j|v|>1}R_j^{-2s}=O(|v|^{2s/\kappa})$. For a matching lower bound choose $v_j$ to be
the multiple of $2\pi/R_j$ nearest $2/W_j$. Then $W_jv_j\to2$,
the earlier levels agree exactly and the later levels contribute
$o(R_j^{-2s})$. Since $\rho_P(2)-\rho(2)=1$, the $j$th level
contributes $(2+o(1))R_j^{-2s}$ in modulus. The first covariance
still gives independent trials at separation $W_j^{-1}$ and Gaussian
derivative bounds $O(W_j^k)$ in root mean square at each fixed order $k$.
Thus the proof of Theorem~\ref{r76:ayache} gives its stated spectrum.
For both series the same variogram bounds at all increments and at
$\pi/R_j$ give critical index $s$, and the integer samples are iid with
variance $\sum_j R_j^{-2s}$.

For the second series realize the envelope as $G_j(t,W_jt)$ on the
cylinder $\mathbb R\times(\mathbb R/2\mathbb Z)$, with independent
circular Gaussian fields of real-component covariance
$\rho(t-u)\rho_P(\theta-\psi)$. Fix a compact interval $K$ and let $D_j$
be the event that the $C^1$ norm on its cylinder is at most $L_j=\log R_j$.
The fields have smooth versions. Applying the fundamental theorem of
calculus in each coordinate bounds the supremum of a derivative by
its value at a corner, integrals along the two edges and the interior
integral of the absolute mixed derivative. Gaussian moment bounds
therefore give $C_K\sqrt p$ for the $L^p$ norm of the cylinder $C^1$ norm, so
$\sum_j\mathbb P(D_j^c)<\infty$.
Put $F_j(t)=\min_\theta|G_j(t,\theta)|$ and $u_j=R_j^{-\delta}$ for
rational $\delta>0$. On $D_j$ a phase grid of mesh $u_j/L_j$ detects
$F_j(t)\le2u_j$ within error $u_j$. The circular Gaussian small-ball
bound at each grid point gives
$\mathbb P(D_j, F_j(t)\le2u_j)\le C L_j u_j$ uniformly in $t$.
On $D_j\cap D_{j+1}$ a time grid of mesh $u_{j+1}/L_{j+1}$ detects any
point where $F_i(t)\le u_i$ for both $i=j,j+1$ within error $u_i$.
Independence of the levels bounds the probability of this event by
$C_K L_jL_{j+1}^2u_j$, summable in $j$. Borel--Cantelli on a countable
exhaustion and over rational $\delta$ gives envelope decay rate $d(t)=0$
at every $t$: a positive lower limit would force all sufficiently late
consecutive pairs to qualify. The envelope and derivative bounds
\eqref{r76:envelopes} hold up to a fixed factor by $D_j$. Freezing the envelopes and summing the carrier increments as in
\eqref{r76:read-exponent} gives exponent at least $s$.
An exponent greater than $s$ would give $d(t)\ge q-s>0$ for
$s<q<\min(\alpha_X(t),s/\kappa)$, impossible since $\kappa<1$.
\end{proof}

Equality of the variances on
$|v|<a$ determines the centered increment process on each interval of
length less than $a$. Its covariance at $u,v$ relative to the left endpoint
is $[V(u)+V(v)-V(u-v)]/2$. Thus exact local equality determines the
local path law up to its initial value.

\subsection{Information in two shrinking intervals}
Let \(X\) be a centered continuous real Gaussian process with stationary
increments, pinned at zero. Write \(E_\epsilon(t)\) for the closed Gaussian
linear span of the increments observed in \((t-\epsilon,t+\epsilon)\),
\(P_\epsilon(t)\) for its projection in the space of all increments, and
\(E_0(t)=\bigcap_{\epsilon>0}E_\epsilon(t)\). For two Gaussian subspaces
\(E,F\) write \(I(E:F)\) for the mutual information of their generated
fields. Norros and Saksman proved local independence for fractional
Brownian motion and asked which variances or spectra give the same
property \cite[Section~5, Q.1--Q.3]{r83:NS}.

\begin{proposition}[Information left in shrinking Gaussian spaces]
For a fixed separation \(d\ne0\), suppose
\(C_\delta=P_\delta(0)P_\delta(d)\) is Hilbert--Schmidt for some
\(\delta>0\).\label{r83:germ} Then
\[
   \lim_{\epsilon\downarrow0} I(E_\epsilon(0):E_\epsilon(d))
      =I(E_0(0):E_0(d)).
\]
If there is no such radius, the information is infinite at every radius.
Thus local independence at \(d\) requires and suffices that such a
radius exists and the two germs are orthogonal. For all nonzero separations
the orthogonality condition is equivalent to \(E_0(0)=0\).
\end{proposition}
\begin{proof}
For \(C=P_EP_F\), the Gaussian information formula
\cite[Theorem~2 and Corollary~1]{r83:NS} gives infinite information
unless \(C\) is Hilbert--Schmidt and \(\|C\|<1\); otherwise
\(I(E:F)=-\frac12\log\det(I-C^*C)\).
For \(\epsilon\le\delta\),
\(C_\epsilon=P_\epsilon(0)C_\delta P_\epsilon(d)\). The outer
projections converge strongly to the germ projections. Finite-rank
approximation of \(C_\delta\) therefore gives convergence in
Hilbert--Schmidt norm to \(C_0=P_0(0)P_0(d)\). The products
\(C_\epsilon^*C_\epsilon\) converge in trace norm to \(C_0^*C_0\),
by the Schatten product inequality. If \(\|C_0\|<1\), continuity of
the Fredholm determinant gives the limit. If \(\|C_0\|=1\), the germ
information is infinite, hence so is the information on every larger
window by monotonicity. The formula also shows that zero information
is exactly orthogonality. Translation on the increment space is a
strongly continuous unitary group by mean-square continuity, with
\(E_0(d)=U_dE_0(0)\). Any nonzero \(f\in E_0(0)\) has
\(\langle f,U_df\rangle\to\|f\|^2\) as \(d\to0\).
\end{proof}

\begin{corollary}[A variance test at rational times]\label{r84:gram}
For any continuous Gaussian process with stationary increments, center and pin its increments so that $X(0)=0$. Write $v(t)=\operatorname{Var}(X(t)-X(0))$. Then local independence at all distinct times has the following
necessary and sufficient condition in terms of $v$.
Set $a_m=2^{-m}$ and let $Y_{m,n}$ record the consecutive increments on
the $2^n$ equal cells of $[-a_m,a_m]$, with $Y_{m,n}^{(d)}$ its
translate to separation $d$. Put
\[
 A=\mathbb E[Y_{m,n}Y_{m,n}^{\mathsf T}],\quad
 C(d)=\mathbb E[Y_{m,n}(Y_{m,n}^{(d)})^{\mathsf T}],\quad
 b(q)=\mathbb E[Y_{m,n}X(q)] .
\]
For $j\ge2$ fix a finite rational $a_m/2$-net $Q_{j,m}$ within
$[1/j,j]$. With $A^\dagger$ the Moore--Penrose inverse, the two
conditions are
\[
 \begin{split}
 &\lim_{m\to\infty}\sup_{n\ge0} b(q)^{\mathsf T}A^\dagger b(q)=0
       \quad(q\in\mathbb Q\cap(-1,1)),\\
 &\forall j\ge2\ \exists m\ge1:\quad
   \max_{d\in Q_{j,m}}\sup_{n\ge0}
   \operatorname{tr}\big(A^\dagger C(d)A^\dagger C(d)^{\mathsf T}\big)
    <\infty .
 \end{split}
\]
All matrix entries use only rational values of $v$, via
$\mathbb E[X(t)X(u)]=(v(t)+v(u)-v(t-u))/2$.
\end{corollary}
\begin{proof}
The first quadratic form is the squared norm of the projection of
$X(q)$ on the mesh span, by least squares. Refining the mesh exhausts
$E_{a_m}(0)$ by continuity. Shrinking $a_m$ then gives the projection
on $E_0(0)$. The rational $X(q)$ in the statement span a dense subspace
of $E_1(0)$, which contains the germ. Thus the first test is equivalent
to $E_0(0)=0$.

The trace is the squared Hilbert--Schmidt norm of the product of the
two mesh projections. This follows by normalizing with
$A^{\dagger/2}$ on the marginal covariance range, since both marginals
have covariance $A$. Summing squared inner products in orthonormal bases
adapted to the nested mesh spans shows that its supremum over $n$ is
$\|P_{a_m}(0)P_{a_m}(d)\|_{\rm HS}^2$, allowing infinity.
A Hilbert--Schmidt product at radius $a$ and separation $d$ gives one
at radius $a/2$ for all $d'$ with $|d'-d|\le a/2$, by inclusion of
both observation windows. On the compact interval $[1/j,j]$ a finite
subcover therefore gives a common radius whenever each separation
admits one. Conversely the finite net in the second test supplies such
a radius at every real separation in this interval, by the same
inclusions. Apply Proposition~\ref{r83:germ} and use symmetry in $d$.
\end{proof}

Write \(v(t)=\mathbb E X(t)^2\). We use the derivative spectral measure
\(\nu\), so that the covariance form on step functions \(f,g\),
integrated against the increments of \(X\), is
\((2\pi)^{-1}\int\widehat f\,\overline{\widehat g}\,d\nu\), with
\(\widehat f(\xi)=\int e^{-it\xi}f(t)\,dt\).
Its convolution kernel is the distribution \(R=\mathcal F^{-1}\nu=v''/2\).

The sampling condition below determines when the local increment norm
can be replaced by a weighted Fourier norm. Let \(w>0\) be Borel and
polynomially moderate, meaning
\(w(x+y)\le C(1+|y|)^M w(x)\), and let \(\nu\) be a positive locally
finite measure. Write \(d\rho=d\nu/w\) and
\[
 \|f\|_w^2=(2\pi)^{-1}\int|\widehat f|^2w\,d\xi,\qquad
 \|f\|_\nu^2=(2\pi)^{-1}\int|\widehat f|^2\,d\nu .
\]

\begin{proposition}[Spectral mass on short windows]\label{r84:mass}
The two norms are equivalent on the smooth test functions supported in
some fixed interval of positive length if and only if for some \(L>0\)
\[
                    \nu([x-L,x+L])\asymp w(x)\qquad(x\in\mathbb R).
\]
The equivalence constants are uniform over translates of that interval.
\end{proposition}
\begin{proof}
Let \(PW_b\) consist of the Fourier transforms of \(L^2\) functions
supported in \([-b,b]\). Lin's sampling theorem \cite{r84:Lin} identifies
unweighted norm equivalence for \(\rho\) on some \(PW_b\), \(b>0\), with
uniform upper and positive lower bounds on the mass of intervals of
some fixed length. In its usual fixed-bandwidth statement the sufficient
interval length must be small compared with \(1/b\); here \(b\) can
shrink. Bounds assumed only off a compact set give the stated version
by enlarging the interval length. On each such interval \(w\) is
comparable to its value at the center.

To transfer the norms, take \(0<b'<b\) and a nonzero Schwartz function
\(\phi\) whose inverse Fourier transform is supported in
\([-(b-b'),b-b']\). For \(\widehat f\in PW_{b'}\) the functions
\(\widehat f(\xi)\phi(\xi-y)\) belong to \(PW_b\). Integrate their
unweighted sampling inequalities against \(w(y)\,dy\). Moderateness
in both directions and rapid decay give
\(\int|\phi(\xi-y)|^2 w(y)\,dy\asymp w(\xi)\); the lower bound follows
by integrating on a bounded set of positive \(|\phi|^2\)-mass.
Tonelli yields precisely the weighted inequalities on the smaller
window. Conversely start with the weighted inequalities and integrate
against \(dy/w(y)\) to obtain the unweighted ones for \(\rho\).
It suffices to run this converse on Fourier transforms of smooth tests:
the test products still have smooth inverse transforms with the stated
support, and \(L^2\) approximation on the smaller time interval gives
uniform convergence of Fourier transforms and hence the same
inequalities throughout \(PW_{b'}\). Translation in time preserves
both Fourier moduli.
\end{proof}

\begin{corollary}[Sampling on every bounded window]\label{r84:allwindows}
The two norms are equivalent on every bounded time window, with constants
depending on the window length, precisely when \(\rho\) samples every
\(PW_b\). Equivalently, \(\sup_x\rho((x,x+1))<\infty\) and for every
\(t>0\) there exist \(\delta,D>0\) such that every interval \(I\)
of length at least \(D\) meets at least \(t|I|\) disjoint intervals
each of length and \(\rho\)-mass at least \(\delta\).
In particular the conclusion holds if the frequency line is partitioned
into half-open bins \(J\) of bounded lengths tending to zero at infinity
and \(\nu(J)\asymp |J|w(x_J)\) for some \(x_J\in J\), uniformly in \(J\).
\end{corollary}
\begin{proof}
Apply the transfer in both directions for arbitrary \(b'<b\).
The geometric equivalence is Makarov--Poltoratski's characterization
\cite[Theorem~3.1]{r84:MP}. Under the bin hypothesis,
\(\rho(J)\asymp |J|\). The upper unit-interval bound follows by covering
with intersecting bins. For a fixed \(t>0\), choose \(\ell<1/(4t)\).
Outside a fixed compact set all bins have length \(<\ell/4\); each
interval of length \(\ell\) there contains whole bins of total length
at least \(\ell/2\), hence has mass at least \(c\ell/2\).
Disjoint such intervals with gaps of length \(\ell\) outside the
exceptional compact set occur at least \(t|I|\) times in each
sufficiently long \(I\). Take \(\delta=\min(\ell,c\ell/2)\).
\end{proof}

\begin{theorem}[Point germs and regularity away from zero]\label{r83:nsweights}\label{r83:sob}
Suppose the derivative spectral measure \(\nu\) satisfies the short-window
mass criterion of Proposition~\ref{r84:mass} for an even positive
polynomially moderate weight \(w\) with
\(\int_{\mathbb R}(1+\xi^2)^{-1}w(\xi)\,d\xi<\infty\).

Put \(R=\mathcal F^{-1}\nu=v''/2\), and let \(m\) be the largest integer
\(k\ge0\) with \(\int |\xi|^{2k}w(\xi)\,d\xi<\infty\), or \(-1\) if there
is no such integer. The point germ at \(t\) is represented by
\(\delta_t,\ldots,\delta_t^{(m)}\), with empty span when \(m=-1\).
For each \(d\ne0\), local independence at separation \(d\) is equivalent to
\[
 R^{(j)}(d)=0\quad(0\le j\le 2m),\qquad
 \int_{\mathbb R}\frac{|\widehat{\chi R}(\xi)|^2}{w(\xi)^2}\,d\xi<\infty
 \quad\text{for some }\chi\in C_c^\infty\text{ equal to one near }d .
\]
When the integral is finite, Proposition~\ref{r83:germ} identifies the
limiting information with that of the two Gaussian derivative vectors
(with value zero for empty vectors).
If the integral fails for every such cutoff, the information is infinite
at every scale. In particular, local independence at every nonzero
separation holds precisely when \(m=-1\) and the integral test holds
locally throughout the punctured line.
\end{theorem}
\begin{proof}
Give \(H_w\) the squared Fourier norm \(\int|\widehat f|^2w\).
Proposition~\ref{r84:mass} gives norm equivalence on sufficiently short
windows. Summing the mass bounds over a grid of intervals of length
\(2L\), with endpoint overlaps of multiplicity at most two, shows that
\(\int|\xi|^{2k}\,d\nu<\infty\) precisely when
\(\int|\xi|^{2k}w(\xi)\,d\xi<\infty\).

For such an open interval \(J\), increments are the \(H_w\) closure
of \(C_c^\infty(J)\). A compactly supported step function can be mollified
inside \(J\). Conversely Riemann step approximations to a test function
have uniformly bounded variation, hence Fourier bound
\(C(1+|\xi|)^{-1}\); dominated convergence gives convergence in \(H_w\).
\par
A germ is a distribution supported at a single point, hence a finite
sum of delta derivatives. The Fourier transform of a nonzero sum of
order \(k\) has modulus comparable to \(|\xi|^k\) at infinity. Every
permitted derivative belongs to all the shrinking spaces by mollification.
The polynomial lower bound on \(w\) ensures \(m<\infty\). For \(m\ge0\)
the moments through order \(2m\) give classical continuous derivatives
of \(R\); up to signs these are exactly the cross covariances of the two
germs. If \(m\ge0\), their orthogonality for every nonzero \(d\) is
impossible since \(R(0)>0\).
\par
The cross covariance on two small intervals has kernel \(R(t-u)\).
Hilbert--Schmidt forms on two copies of \(H_w\) have Fourier tensor weight
\(1/(w(\xi)w(\eta))\). Smooth cutoffs are bounded multipliers in these
spaces by the weight inequality and Young's inequality. Localizing a
Hilbert--Schmidt form first to a rectangle about \((d,0)\), and then to
the support of \(\chi(t-u)\psi(u)\) inside that rectangle, gives a global
Hilbert--Schmidt kernel \(\chi(t-u)R(t-u)\psi(u)\) with \(\psi\ne0\).
Conversely such a kernel with \(\psi=1\) near zero restricts to the
original form on a smaller rectangle. Its Fourier transform is
\(\widehat{\chi R}(\xi)\widehat\psi(\xi+\eta)\); the squared norm is
comparable to the displayed integral because
\[
 \int\frac{|\widehat\psi(z)|^2}{w(z-\xi)}\,dz\asymp\frac1{w(\xi)} .
\]
The upper bound uses rapid decay, and the lower integrates over a bounded
set with positive squared mass. This computation applies to distributions
as well, since \(\chi R\) has compact support and thus a smooth Fourier
transform. The shrinking-projection criterion now gives the equivalence.
\end{proof}

The same test admits a formulation using only \(v\) near zero and the
chosen separation, as requested in \cite[Section~5, Q.2]{r83:NS}.
\begin{corollary}[A local variance criterion]\label{r83:nslocal}
Let \(w\) satisfy the weight assumptions of Theorem~\ref{r83:nsweights}
and \(\int w=\infty\). Put \(R_w=\mathcal F^{-1}w\).
Suppose for some \(a>0\) and \(\chi\in C_c^\infty\) equal to one near
zero that \(\widehat{\chi(R-aR_w)}/w\in L^2\). Then at each \(d\ne0\)
local independence holds precisely when
\(\widehat{\eta R}/w\in L^2\) for some \(\eta\in C_c^\infty\) equal to
one near \(d\). Failure gives infinite information at every radius.
For \(w(\xi)=(1+\xi^2)^s,\ |s|<1/2\), the two tests are respectively
local membership of \(R-aR_w\) at zero and of \(R\) at \(d\) in
\(H^{-2s}\). In particular, if \(v(h)=a|h|+r(h)\) near zero, \(a>0\),
and \(r\) near zero and \(v\) near \(d\) are \(C^2\), the information
is \(O(\epsilon^2)\).
\end{corollary}
\begin{proof}
Use \((2\pi)^{-1}\int|\widehat f|^2w\) as reference squared norm.
On a small fixed interval about zero the localized tensor calculation
of Theorem~\ref{r83:nsweights} makes the marginal covariance error
with kernel \((R-aR_w)(t-u)\) a Hilbert--Schmidt operator in the
reference norm. The reference germ is zero since \(\int w=\infty\).
Compress the error to shrinking intervals: their projections converge
strongly to zero, so its compressed norm tends to zero. On a sufficiently
small interval the true squared norm is therefore between \(a/2\)
and \(3a/2\) times the reference squared norm. By stationarity this
holds around every point. The resulting local identifications preserve
the trivial germs and the Hilbert--Schmidt property of the cross form.
The same tensor calculation away from zero and Proposition~\ref{r83:germ}
give the criterion and its infinite-information alternative. For the last
assertion choose \(w=1\). The marginal kernel is \(a\delta_0+r''/2\)
near zero and the cross kernel \(v''(t-u)/2\) is bounded on the two
intervals. In the reference \(L^2\) norms its Hilbert--Schmidt norm is
at most \(\epsilon\sup|v''|\). Equivalence of marginal norms and the
information bound give the stated rate.
\end{proof}

\begin{theorem}[Prescribing the separations of dependence]\label{r83:bad}
Every relatively closed \(F\subset(0,\infty)\) is the set of positive
separations at which local independence fails for a continuous Gaussian
process with stationary increments and almost surely nowhere
differentiable paths. Its derivative spectral density lies between
\(1/2\) and \(3/2\), and all point germs are zero. For \(d\in F\) the
information is infinite at every radius; for \(d\notin F\) it is zero
at every sufficiently small radius.
\end{theorem}
\begin{proof}
Choose distinct points \(d_j\) dense in \(F\) and \(c_j>0\) with
\(\sum c_j\le1/2\), using a finite or empty list when appropriate.
The density \(1+\sum_jc_j\cos(d_j\xi)\) has covariance kernel
\[
 R=\delta_0+\tfrac12\sum_j c_j(\delta_{d_j}+\delta_{-d_j}).
\]
It defines a stationary-increment process by the Gaussian spectral
isometry on interval indicators. Its increment variance is bounded
above by \(3|t-s|/2\), giving a continuous version by Gaussian moment
bounds. Theorem~\ref{r83:nsweights} with \(w=1\) gives zero germs and
infinite information wherever every neighborhood of \(d\) contains an
atom. Elsewhere the cross form vanishes on sufficiently small windows,
which are independent by Gaussianity.

Split the density into \(1/2\) and its nonnegative remainder, realizing
the process as \(B/\sqrt2+Z\) for independent two-sided Brownian
motion \(B\) and continuous \(Z\). Given \(Z\), the probability that
three consecutive intervals of the mesh \(h=2^{-n}\) all carry increments
of absolute size at most \(6Mh\) is at most \((C M\sqrt h)^3\), by the
Gaussian density bound and independent Brownian increments. Each fixed
compact interval meets \(O(h^{-1})\) such triples. The resulting bounds
are summable in \(n\). Differentiability at an interior point would
force such a triple at every sufficiently fine mesh for some integer
\(M\), by bounding the increments from that point. Borel--Cantelli
on a countable compact exhaustion excludes this almost surely.
\end{proof}
Conversely, for any stationary-increment Gaussian process with zero point
germs the set of bad positive separations is relatively closed. A
Hilbert--Schmidt cross form at radius \(\delta\) and separation \(d\)
restricts to one at radius \(\delta/2\) and separation \(d'\) whenever
\(|d'-d|<\delta/2\). Proposition~\ref{r83:germ} therefore makes the
set of good separations open.
Taking \(F=(0,\infty)\) answers Q.3 of \cite{r83:NS} and disproves the
sufficiency of rough paths and atomless spectrum asked about after Q.1.

\begin{corollary}[Good separations with a nonzero point germ]\label{r84:good}
Suppose the point germ is nonzero and a Hilbert--Schmidt cross form
exists at some radius for every $d>0$. The set of positive separations
at which local independence holds is relatively closed and is disjoint
from a neighborhood of zero. Every such set $Z$, including the empty
set, occurs for a Gaussian process with continuously differentiable
paths and one-dimensional point germs. In these examples the entire
limiting information is
\[
 \lim_{\epsilon\downarrow0}I(E_\epsilon(0):E_\epsilon(d))
 =-\tfrac12\log\!\left(1-\frac{R(d)^2}{R(0)^2}\right),\qquad d>0,
\]
where $R$ is the covariance of the derivative process.
\end{corollary}
\begin{proof}
By Proposition~\ref{r83:germ} the good set is the common zero set of
$\langle f,U_d g\rangle$ for $f,g\in E_0(0)$. These functions are
continuous, and the coefficient of any nonzero $f=g$ is positive at zero.

The prescription of a characteristic function's zero set is classical
\cite[Corollary~1]{r84:Ilinskii}. Here we also control the spectral
density. Choose $a>0$ with $(0,2a)\cap Z=\varnothing$ and cover
$(0,\infty)\setminus Z$ by countably many intervals
$(d_n-b_n,d_n+b_n)$ contained in that complement, with $b_n<a$.
For $b>0$, let $\eta_b$ be the autocorrelation of a smooth nonnegative
bump positive precisely on $(-b/2,b/2)$. The functions
\[
 2\eta_{b_n}(t)+\eta_{b_n}(t-d_n)+\eta_{b_n}(t+d_n)
\]
are nonnegative and positive definite, since their Fourier transforms
are $(2+2\cos(d_n\xi))\widehat\eta_{b_n}(\xi)\ge0$.
Choose positive coefficients small enough that their sum $\phi$
converges in the Schwartz space. Its positive zeros are exactly $Z$,
and $\phi(0)>0$. Put $R(t)=e^{-|t|}\phi(t)$. Its spectral density is
comparable to $w(\xi)=(1+\xi^2)^{-1}$, since it is a positive constant
times the convolution of $w$ with the nonzero nonnegative Schwartz
function $\widehat\phi$. For the lower bound integrate over a bounded
set of positive mass. The upper bound follows from
$w(\xi-y)\le 2(1+y^2)w(\xi)$.

The stationary Gaussian process $Y$ with covariance $R$ has a
continuous version by $R(0)-R(t)=O(|t|)$ and Gaussian moment bounds.
Set $X(t)=\int_0^tY(u)\,du$. Theorem~\ref{r83:nsweights} with $m=0$
identifies its germ with the span of $Y(t)$ and gives a local
Hilbert--Schmidt cross form for every $d>0$, since $R$ is smooth there.
The cross projections converge in Hilbert--Schmidt norm by the proof
of Proposition~\ref{r83:germ}. Their limiting product has a single
possible nonzero singular value $|R(d)|/R(0)<1$, the strict inequality
following from the positive spectral density. Their squared singular
values therefore give the stated information limit: $C^*C$ converges
in trace norm as $C$ converges in Hilbert--Schmidt norm, and the series
$\sum_{j\ge1}\operatorname{tr}((C^*C)^j)/(2j)$ converges uniformly
along the limit since the operator norms stay below one uniformly.
\end{proof}

\subsection{What the law on a bounded interval determines}
The same covariance kernel tests two different comparisons. For local
independence the kernel couples increments in disjoint intervals. For
equivalence of two centered Gaussian laws it is the difference of the
two covariance kernels on one interval. We use the classical criterion
in covariance form \cite[Section~3]{r84:SX}: equivalence requires and
suffices that the two covariance norms be comparable and that the
difference of their inner products be a Hilbert--Schmidt form in
either completion. Feldman obtained a single-interval criterion
for stationary processes with reference spectral measure
\((1+\xi^2)^{-u}\,d\xi\), \(u\) a positive integer, and any finite
competing spectral measure. It combines a covariance-norm condition
with square integrability of the transformed spectral difference on
the observation square \cite[pp.~1211--1212]{r84:Feldman}.
Safikhani and Xiao also give spectral-density
conditions and a criterion for adding an independent discrete component
\cite[Theorems~3.5 and~4.2]{r84:SX}.

For arbitrary spectra the obstruction to comparability also has a local
spectral description. Keich eliminated an operator nondegeneracy condition
on smaller intervals in a comparison with Wiener measure
\cite[Theorem~1]{r85:Keich}. For two stationary-increment laws the
remaining obstruction, under an essential spectral gap, is recovery of
an atom present in only one spectrum.

\begin{theorem}[Which spectral atoms obstruct local equivalence?]
\label{r85:atomobstruction}
Let \(\nu_0,\nu_1\) be the derivative spectral measures of two centered
continuous Gaussian stationary-increment processes, and write
\(\nu=\nu_0+\nu_1\), \(h=d(\nu_1-\nu_0)/d\nu\). For a bounded
open interval \(I\), define
\[
 E_I=\overline{\{\widehat f:f\in C_c^\infty(I;\mathbb C)\}}^{L^2(\nu)},
 \qquad T_I=P_{E_I}M_h|_{E_I}.
\]
Call an atom \(\lambda\) exclusive if only one measure charges it, and
recoverable on \(I\) if \(\mathbf1_{\{\lambda\}}\in E_I\).
The two covariance norms are comparable on each bounded interval if and
only if \(\|T_I\|_{\rm ess}<1\) on each such interval and no exclusive
atom is recoverable on any bounded interval. The two increment laws are
equivalent on every bounded interval if and only if every \(T_I\) is
Hilbert--Schmidt and no exclusive atom is recoverable on a bounded
interval. In particular the Hilbert--Schmidt condition suffices when
the sets of atoms coincide.

If every \(T_I\) is Hilbert--Schmidt, then for any prescribed bounded
window one can make the two laws equivalent there by dividing the total
mass equally between the two measures at finitely many exclusive atoms.
\end{theorem}
\begin{proof}
In the sum covariance norm the two forms have operators \((1-T_I)/2\)
and \((1+T_I)/2\). They are comparable precisely when \(\|T_I\|<1\).
The Gaussian covariance criterion \cite[Theorem~3.1]{r84:SX} then gives
equivalence precisely when \(T_I\) is also Hilbert--Schmidt. The real
observations correspond to \(F(-\xi)=\overline{F(\xi)}\);
complexification preserves these operator tests. Smooth derivative tests
generate the increment field by approximating interval indicators with
mollifiers, using \(\int (1+\xi^2)^{-1}d\nu<\infty\) and continuity.

Fix \(\overline I\subset J\). If \(T_I F=F\), positivity gives
\(\int(1-h)|F|^2\,d\nu=0\). Thus \(h=1\) \(|F|^2\nu\)-almost everywhere.
For all sufficiently small \(|a|\), translation of the tests shows that
\(e^{-ia\xi}F\) belongs to \(E_J\) and still satisfies
\(T_J(e^{-ia\xi}F)=e^{-ia\xi}F\). If \(\|T_J\|_{\rm ess}<1\), this
eigenspace is finite-dimensional. Take any \(n\) distinct points of the
closed support of \(|F|^2\nu\). Choose small \(a>0\) so their values
of \(e^{-ia\xi}\) are distinct and all translates through \((n-1)a\)
lie in the eigenspace. A relation among those \(n\) translates would
give a polynomial of degree below \(n\) vanishing at the \(n\) values:
its composition with the exponential is continuous and zero
\(|F|^2\nu\)-almost everywhere. Thus \(n\) is at most the eigenspace
dimension. The support is finite, and Vandermonde inversion at its
points puts each individual atom indicator in that same eigenspace on
\(J\). Each atom is exclusive. The argument for \(-1\) is identical.

In particular absence of recoverable exclusive atoms excludes both
endpoint eigenvalues on \(I\); the essential norm bound then gives
\(\|T_I\|<1\). Conversely, recovery of an exclusive atom gives an
endpoint eigenvector on its window. This proves the two criteria.
For the last assertion take the exclusive atoms whose indicators are
endpoint eigenvectors on \(J\). There are only finitely many by
orthogonality, and the set is symmetric. Equal division sets \(h\) to
zero there, leaves \(\nu\) unchanged, and changes \(T_I\) by finite
rank. An endpoint eigenvector after division would already have been an
endpoint eigenvector before division and would vanish at all these atoms.
The support conclusion just proved forces it to vanish identically.
The Hilbert--Schmidt criterion gives the claimed equivalence.
\end{proof}

Recovery of a particular atom on \(I\) is equivalent to
\[
 \inf_{f\in C_c^\infty(I;\mathbb C),\ \widehat f(\lambda)=1}
 \int_{\mathbb R\setminus\{\lambda\}} |\widehat f|^2\,d\nu=0,
\]
since an approximation to its indicator has value tending to one at
the atom and can be normalized there.
If \(\nu\) has all moments \(m_k=\int|\xi|^k\,d\nu<\infty\) and
\(\sum_{n\ge1}m_{2n}^{-1/(2n)}=\infty\), de Jeu's extended Carleman
theorem \cite[Theorem~2.3]{r85:deJeu} gives \(E_I=L^2(\nu)\) for
every nonempty interval: its exponentials are dense and smooth
mollifiers centered there approximate them. The full-line criterion
of Corollary~\ref{r84:global} then applies on each such interval.

\begin{proposition}[Information under an equivalent change of Gaussian law]
\label{r84:transfer}
Let two centered continuous Gaussian stationary-increment processes,
pinned at zero, have equivalent laws on every bounded observation interval
containing zero. The condition of a zero point germ holds for both if
it holds for either one. In that case their sets of separations of local
independence coincide. At each remaining separation both have infinite
mutual information at every positive radius.
\end{proposition}
\begin{proof}
Fix a bounded window containing the observation intervals and zero.
The covariance criterion identifies the two completions of its step
functions (modulo their common null space) by the identity, with
equivalent norms. The identification preserves the closed span of the
increments in each subinterval, hence also their decreasing intersection.
The two cross forms on any pair of subintervals differ by the
restriction of the Hilbert--Schmidt covariance difference.
Thus one cross form is Hilbert--Schmidt precisely when the other is.
Proposition~\ref{r83:germ} gives the result. If a cross form fails to be
Hilbert--Schmidt at every radius, the Gaussian information formula gives
infinite information at every radius.
\end{proof}

\begin{theorem}[Equivalence to a law with known local covariance norm]\label{r84:equiv}
Let \(\nu_0,\nu_1\) be even positive derivative spectral measures with
\(\int (1+\xi^2)^{-1}\,d\nu_i<\infty\), and put
\(D=\mathcal F^{-1}(\nu_1-\nu_0)\).
Suppose the centered pinned process for \(\nu_0\) has continuous paths
and its covariance norm on each bounded time interval is equivalent
to \(H_w\), for an even polynomially moderate weight \(w\) with
\(\int w(\xi)/(1+\xi^2)\,d\xi<\infty\).
The process for \(\nu_1\) has a continuous modification locally
equivalent in law to the reference on every bounded window containing
zero if and only if
\[
            \widehat{\chi D}/w\in L^2(\mathbb R)
                   \qquad(\chi\in C_c^\infty(\mathbb R)).
\]
In this case the covariance norm for \(\nu_1\) is also locally
equivalent to \(H_w\). If a second weight \(w_1\) represents that norm
then \(w_1\asymp w\). For fractional Brownian motion with Hurst index
\(H\) as reference, the test is \(v_1''-(|t|^{2H})''\in
H^{2H-1}_{\rm loc}(\mathbb R)\).
\end{theorem}
\begin{proof}
On the test completion over a bounded interval \(I\) use the reference
inner product \(q_0\). The tensor calculation of
Theorem~\ref{r83:nsweights} identifies a Hilbert--Schmidt covariance
difference on every such interval with the displayed cutoff test.
Indeed its kernel is \(D(t-u)\). Two-variable localization on an
interior rectangle with a given lag, followed by the change of variables
\((t,u)\mapsto(t-u,u)\), gives the test near that lag; conversely a
cutoff equal to one near \(\overline{I-I}\) suffices on \(I\).

Under the test the competing form is \(q_0((1+K_I)\cdot,\cdot)\) with
\(K_I\) self-adjoint Hilbert--Schmidt, so \(1+K_I\) is positive and
Fredholm. A nonzero null vector would be a compactly supported
distribution \(f\), since the reference norm is equivalent to \(H_w\).
Take a larger interval \(J\) containing \(\overline I\) in its interior.
All small translates of \(f\) belong to the completion on \(J\), and
the translated tests have the same competing norm. Thus each translate
lies in \(\ker(1+K_J)\). Any finite collection of distinct translates
is linearly independent: a dependence would give an exponential
polynomial vanishing wherever the nonzero entire transform \(\widehat f\)
does not vanish. This contradicts the finite dimension of that kernel.
Hence \(1+K_I\) is invertible and the covariance norms are comparable.
For any second representing weight, the squared norm of
\(e^{iyt}\varphi(t)\), with a fixed nonzero test \(\varphi\), is uniformly
comparable to the weight at \(y\), by moderateness.

The Gaussian covariance criterion now gives equivalence on the countable
rational-time observations of each pinned window. Indeed one can
mollify the oriented interval indicators from zero in both covariance
norms, using spectral admissibility, and conversely approximate each
smooth test by step functions. Uniform continuity of the reference
rational observations on each compact interval transfers by equivalence.
Their continuous extensions give a modification of the competing
process by mean-square continuity, and generate the respective path
laws. Conversely equivalent path laws give comparable test norms and a
Hilbert--Schmidt difference in the reference norm; apply the tensor
calculation. For fractional Brownian motion the derivative density is
proportional to \(|\xi|^{1-2H}\). Off a compact band it is comparable to
\(w(\xi)=(1+\xi^2)^{1/2-H}\); dividing by \(w\) gives a density bounded
above and below at infinity and locally integrable, which samples every
Paley--Wiener space by Corollary~\ref{r84:allwindows}. The stated
Sobolev exponent follows.
\end{proof}

Suppose both increment covariance norms are comparable to $L^2(J)$ on a fixed interval $J$. The interval endpoints give the following exact test. Write $L$ for its length, with observations pinned
inside the window or taken through increments. Its two laws are
equivalent if and only if $D$ is represented on $(-L,L)$ by a function
with
\begin{equation}\label{r84:endpoint}
                 \int_{-L}^L(L-|x|)|D(x)|^2\,dx<\infty .
\end{equation}
Indeed Hilbert--Schmidt forms on $L^2(J)$ have exactly the
$L^2(J\times J)$ kernels. In the difference coordinate $x=t-u$,
the length of the slice is $L-|x|$. Conversely an $L^2$ representative
on the square gives an $L^2_{\rm loc}$ representative of $D$ on
$(-L,L)$ by testing in the second coordinate on interior rectangles.
For a Wiener reference law this is the stationary-increment form of
Shepp's kernel criterion \cite{r84:Shepp}.

For example, compare derivative density $1$ with $1+c\cos(a\xi)$,
where $a>0$ and $0<c<1$. The latter adds atoms only at lags $\pm a$
to the covariance kernel. The two window laws coincide for $L\le a$
and are singular for $L>a$. For density $1+c\sin(a|\xi|)$ instead,
the laws are equivalent for $L<a$ and singular already at $L=a$.
Indeed inverse Fourier transformation with damping $e^{-\eta|\xi|}$
and $\eta\downarrow0$ gives, on $(-a,a)$,
$D(x)=c((a+x)^{-1}+(a-x)^{-1})/(2\pi)$. Its square has logarithmically
divergent integral against the weight in \eqref{r84:endpoint} at $L=a$.
Gaussian dichotomy gives singularity on every larger window as well.

For stationary Gaussian processes, Feldman recorded the full-line
criterion in terms of the nonatomic part and atomic masses of the
spectrum \cite[p.~1211]{r84:Feldman}. The same criterion applies to
the derivative spectral measure of a pinned stationary-increment process.

\begin{corollary}[What the full-line law determines]\label{r84:global}
Let \(\nu_0,\nu_1\) be the derivative spectral measures of two pinned
centered continuous Gaussian processes with stationary increments.
Their laws on \(C(\mathbb R)\) are equivalent if and only if their
nonatomic parts coincide, their sets of atoms coincide, and on the common
set of atoms \(A\),
\[
 \sum_{a\in A}\left(\frac{\nu_1\{a\}-\nu_0\{a\}}
                              {\nu_1\{a\}+\nu_0\{a\}}\right)^2<\infty .
\]
Otherwise the laws are mutually singular.
\end{corollary}
\begin{proof}
Put \(\mu=\nu_0+\nu_1\) and \(b=d(\nu_1-\nu_0)/d\mu\).
The full-line observations \(X'(f)\), for real compactly supported
smooth \(f\), generate the pinned path field, by approximation of
interval increments. Their Fourier transforms are dense in the
real symmetric subspace of \(L^2(\mu)\): an orthogonal vector \(g\)
would have \(\mathcal F^{-1}(g\mu)=0\) in distributions, hence
\(g=0\). Temperedness follows from
\(\int(1+\xi^2)^{-1}\,d\mu<\infty\) by Cauchy--Schwarz.
Relative to the mean covariance, the two covariance operators multiply
by \(1-b\) and \(1+b\). The classical Gaussian covariance criterion
\cite[Section~3]{r84:SX} requires comparability of the two norms
and a Hilbert--Schmidt difference in the common norm. A multiplication
operator is Hilbert--Schmidt exactly when it vanishes on the nonatomic
part and its multipliers on the atoms are square summable: on a
nonatomic set where its modulus is bounded below, disjoint normalized
indicators preclude compactness. One may complexify the real symmetric
space for this test. The displayed summability and common atom set
give \(\|b\|_\infty<1\), since only finitely many atomic values can
have modulus greater than \(1/2\); comparability conversely precludes
unmatched atoms. The Gaussian dichotomy gives the last assertion.
\end{proof}

\subsection{Exact finite observations with different spectral types}

Positive definite extension asks which stationary covariances agree on a
prescribed interval \cite[Section~7.5]{r84:JN}. For finite positive spectral
measures this is the interval determinacy problem of \cite{r84:MPdet}.
For fractional Brownian motion, Dzhaparidze and van Zanten gave an
explicit atomic derivative spectrum reproducing its law on a fixed
interval \cite[Theorem~8.1 and Corollary~8.2]{r84:DVZ}. For the
densities below, we can prescribe either singular spectral type while
preserving the Gaussian measure class on every larger bounded window.

\begin{theorem}[Preserving the law on a prescribed window]\label{r84:types}
Let \(w\) be even and measurable with \(c\rho\le w\le C\rho\) almost
everywhere, where \(\rho\) satisfies the weight hypotheses of
Theorem~\ref{r83:nsweights}. Suppose the pinned centered Gaussian process
of derivative spectral density \(w\) has a continuous version. Fix \(T>0\).
There are purely atomic and singular continuous derivative measures whose
pinned Gaussian laws coincide with the given law on \(C([-T,T])\),
are mutually equivalent on every bounded window, and are pairwise
mutually singular on \(C(\mathbb R)\). Each of the two singular types
admits a continuum of such choices with pairwise disjoint supports
within that type. The atomic supports are locally finite.
\end{theorem}
\begin{proof}
Let \(B\) be the real Banach space of even measurable functions modulo
null sets with norm \(\|f\|_B=\|f/\rho\|_\infty\). For an integer
\(M\ge1\) subdivide each \([j,j+1)\), \(j\ge0\), into
\(N_j=\lceil M e^{j+1}\rceil\) equal bins. Choose a probability
\(\theta_J\) carried by a compact subset of the interior of each bin
\(J\); use reflected bins and probabilities on the negative half-line.
Set
\[
 D_M f=\sum_J\Big(\int_J f(\xi)\,d\xi\Big)\theta_J,
 \qquad E_Mf=\mathcal F^{-1}(D_Mf-f\,d\xi).
\]
For every \(A>0\) and integer \(k\ge0\),
\[
 \|E_M f\|_{C^k([-A,A])}\le C_{A,k,\rho} M^{-1}\|f\|_B .
\]
Indeed, the bin error at derivative order \(k\) is bounded, up to the
Fourier normalization, by \(\int_J |f|\) times the oscillation of
\((i\xi)^k e^{it\xi}\) on that bin. On the \(j\)-th positive cell
the oscillation is at most \(C_{A,k}(1+j)^k/N_j\). The polynomial upper
bound on \(\rho\) gives a summable estimate for each derivative, uniform
in the probabilities. Polynomial cell masses ensure that the sum
represents the inverse Fourier transform as a tempered distribution.

Take a real even \(\chi\in C_c^\infty\) equal to one on a neighborhood
of \([-2T,2T]\), and define \(K_M f=\mathcal F(\chi E_M f)\).
The reciprocal \(1/\rho\) grows at most polynomially. Integration by parts
in the last Fourier transform and the displayed estimates therefore give
\(\|K_M\|_{B\to B}\le C_{\rho,\chi}/M\). Put
\(g=(1+K_M)^{-1}w\) for sufficiently large \(M\). The Neumann series
gives \(\|g-w\|_B<c/2\), hence \(g\ge c\rho/2\). The positive even
measure \(\nu=D_Mg\) satisfies
\[
 \mathcal F^{-1}(\nu-w\,d\xi)=(1-\chi)E_Mg.
\]
This smooth function vanishes on the prescribed lag interval. All
covariances of increments on \([-T,T]\) therefore agree, as follows
by testing first against smooth functions and then approximating interval
indicators. Pinning gives equality of the process laws there.
By stationarity of the increments these local laws yield a continuous
version of the replacement on a countable covering of the line.

Every bin has mass comparable to its length times \(\rho(j)\),
with mesh tending to zero at infinity. Theorem~\ref{r84:allwindows}
therefore gives the covariance norm on each bounded interval. Smoothness
of the covariance difference and Theorem~\ref{r84:equiv} give equivalence
on every such interval. Choosing \(\theta_J\) to be a point mass at a
fixed relative position \(\alpha\in[1/4,3/4]\) yields disjoint locally
finite supports for different \(\alpha\). For the continuous choices use
a scaled middle-thirds Cantor set inside each bin: fix the even-indexed
Cantor digits to one prescribed binary sequence, and give the remaining
digits independent fair probabilities. Distinct prescribed sequences
give disjoint compact supports in each bin. The support unions are
locally finite collections of compact sets. These singular measures
are mutually singular with the original density, and any two measures
of different types are mutually singular. The global singularity of
the associated Gaussian laws follows from Theorem~\ref{r84:global}.
\end{proof}

\begin{corollary}\label{r84:typedbad}
For each relatively closed \(F\subset(0,\infty)\) and prescribed
\(T>0\), the nowhere differentiable Gaussian process of
Theorem~\ref{r83:bad} has purely atomic and singular continuous
replacements with identical laws on \([-T,T]\), zero point germs,
and the same set \(F\) of failures of local independence. In all three
models the information at a lag in \(F\) is infinite at every radius.
At each positive lag outside \(F\) it tends to zero as the radius shrinks.
\end{corollary}
\begin{proof}
Use Theorem~\ref{r84:types} with \(\rho=1\) and the density in
Theorem~\ref{r83:bad}. Equality of local increment laws on translated
intervals of length \(2T\) transfers almost sure nowhere differentiability.
By Theorem~\ref{r84:mass} each germ is zero, and local independence at
\(d>0\) is equivalent to \(R\in L^2\) on some neighborhood of \(d\).
Adding a smooth function preserves that condition and its failure.
For any two windows observed inside \([-T,T]\), equality of joint laws
preserves the exact value of mutual information. The horizon matters:
if two sufficiently small windows were independent at every nonzero lag,
then \(R\) would vanish on the punctured line, by testing the cross
covariances on each corresponding rectangle. A distribution supported at
zero has polynomial Fourier transform, which excludes a nonzero singular
spectral measure.
\end{proof}

\subsection{Observable jumps and the Cameron--Martin test}
Kania separated the Gaussian path law from a random Cameron--Martin
shift for autonomous Ornstein--Uhlenbeck equations
\cite[Proposition~2.1 and Theorem~3.3]{r88:Kania}. His
Problems~P2, P3 and P6 ask for an intrinsic infinite-activity test,
a comparison allowing degenerate Gaussian noise, and a two-time
version for evolution families. All three use the same control range.

Fix $T>0$. Let $H,E$ be separable real Hilbert spaces, let $B:E\to H$
be Hilbert--Schmidt, and let $U_i(t,s)$, $i=0,1$, be strongly
continuous evolution families on $0\le s\le t\le T$, with $U_i(s,s)=I$.
The two zero-initial-state convolutions use the same noise $B\,dW+dJ$,
where $W$ is cylindrical Brownian motion on $E$ and $J$ is an independent
$H$-valued L\'evy process with triplet $(b,0,\nu)$.
Assume the Gaussian convolutions have continuous versions and the full
convolutions have càdlàg versions with jumps $\Delta J$.
Write $V=L^2(0,T;H)$, $F=L^2(0,T;E)$ and define
\[
 D_i u(t)=\int_0^t U_i(t,s)Bu(s)\,ds,\qquad
 \psi_sx(t)=\mathbf1_{s<t}\big(U_0(t,s)-U_1(t,s)\big)x.
\]
Restrict $D_i$ to $F_i=(\ker D_i)^\perp$ when taking an inverse,
and put $L_i=D_iF_i$. No density assumption on $BE$ is needed.
Use the standard truncation $\chi_X(y)=y\mathbf1_{\|y\|_X\le1}$
in each Hilbert space $X$.

\begin{theorem}[A triplet test for the observable jump discrepancy]
\label{r88:jumpcomparison}
The full path laws are equivalent exactly under the following conditions.
First $L_0=L_1=:L$ and $RR^*-I_{F_1}$ is Hilbert--Schmidt, where
$R=D_1^{-1}D_0:F_0\to F_1$. Next, let $M$ be the image of
$ds\,\nu(dx)$ under $(s,x)\mapsto\psi_s x$, omitting zero, and set
\[
 h=\int_0^T\psi_s b\,ds+
 \int_0^T\!\int_H
    \big[\chi_V(\psi_sx)-\psi_s\chi_H(x)\big]\,\nu(dx)\,ds .
\]
The three jump conditions are
\begin{gather*}
 M(V\setminus L)=0,\qquad
 \int_L (1\wedge\|D_1^{-1}y\|_{F_1}^2)\,M(dy)<\infty,\\
 h-\int_L\big[\chi_V(y)-D_1\chi_{F_1}(D_1^{-1}y)\big]M(dy)\in L .
\end{gather*}
Failure of the Gaussian conditions gives mutual singularity. When they
hold, the absolutely continuous part in either direction has mass
$\mathbb P(Z\in L)$, where $Z=\int_0^T\psi_s\,dJ_s$ in $V$.
\end{theorem}
\begin{proof}
The Gaussian covariances on $V$ are $D_iD_i^*$, with Cameron--Martin
spaces $L_i$. Under range equality the closed graph theorem makes $R$
a bounded isomorphism. The Feldman--H\'ajek criterion in control
coordinates is precisely the stated Hilbert--Schmidt test
\cite{r88:DaPrato}. Passing from paths to $V$ preserves
the laws being compared: on càdlàg paths continuous at the fixed
terminal time the inclusion is Borel and injective and hence has a
Borel inverse on its image. The paths have this terminal continuity
almost surely.

Each full path records its jump times and sizes. Compensated annular
sums along a deterministic almost surely convergent subsequence
recover $J$, including its prescribed drift, from this record.
Conditionally on the record each full law is a translate of its
Gaussian law. The difference of the conditional means is $Z$:
large jumps give finite sums and the small-jump identity follows
by the Hilbert-space isometry, since $\sup_s\|\psi_s\|_{H\to V}<\infty$.
Unequal Gaussian covariance classes remain singular under all
translations. In the common class the Cameron--Martin theorem gives
conditional equivalence exactly on $\{Z\in L\}$ and conditional
singularity elsewhere. The record has the same law in both experiments
and is recoverable from the path, so disintegration gives the stated
mass (conditional densities can be taken relative to the sum of the
two conditional laws).

The deterministic integral $Z$ has triplet $(h,0,M)$ by the
L\'evy--Khintchine formula. The injective restriction $D_1:F_1\to V$
has Borel image and Borel inverse on that image. If $Z\in L$ almost
surely then $D_1^{-1}Z$ is infinitely divisible: for each $n$,
conditioning on $n-1$ summands places every $n$th convolution root
in a single coset, and the $n$-fold sum in $L$ implies that this
coset is $L$, since $V/L$ is a real vector space. The bounded
pushforward formula for its triplet gives all three conditions.
Conversely they define a L\'evy measure on $F_1$ and an $F_1$-valued
drift whose pushforward has exactly the given triplet.
The truncation-difference integral is absolutely convergent: where
only the ambient truncation is active, the lifted measure of the
set is finite, and where only the lifted truncation is active,
$\|y\|_V\le\|D_1\|$ and the ambient measure of the set is finite.
\end{proof}

In the analytic setting of \cite[Corollary~4.1 and Section~7, P2]{r88:Kania},
write $R_\beta=Q^{-1/2}KA^{-\beta}$ and suppose it is bounded for
$0<\beta<1/2$, with the fractional-power convention of that paper.
Duhamel's formula supplies the single-jump control
$q_sx(u)=-\mathbf1_{s<u} R_\beta A^\beta S(u-s)x$.
Since $\int_0^T r^{-2\beta}\,dr<\infty$,
$\sup_s\|q_sx\|_{F}\le C_T\|x\|$. The compensated small jumps therefore
converge in $L^2(\Omega;F)$ under the defining L\'evy measure condition,
the large jumps give a finite sum, and the drift has a Bochner integral.
Pushing this control through $D_1$ gives $Z$. Thus the L\'evy--It\^o
interpretation in P2 works for every $H$-valued L\'evy noise in this
setting. If the displayed jump sum in P2 is read without compensation,
truncating at $\|x\|_H>\varepsilon$ yields an almost surely convergent
series in $F$ exactly when the deterministic compensators
$\int_0^T\int_{\varepsilon<\|x\|_H\le1}q_sx\,\nu(dx)\,ds$ converge in $F$
as $\varepsilon\downarrow0$. Indeed, in the increasing size-cutoff
filtration the compensated sums form an $L^2$-bounded Hilbert martingale
and converge almost surely; the remaining large-jump sum is finite.

\begin{corollary}[Exact overlap for autonomous evolutions]
\label{r88:jumpmass}
Suppose $U_i(t,s)=S_i(t-s)$ and the Gaussian conditions hold. Write
$\lambda=M(V\setminus L)$. If $\lambda=\infty$, the path laws are
singular. If $\lambda<\infty$, test the last two jump conditions of
Theorem~\ref{r88:jumpcomparison} with $M|_L$ and
$h_0=h-\int_{V\setminus L}\chi_V(y)M(dy)$. The common absolutely
continuous mass is $e^{-\lambda}$ if these tests hold and zero otherwise.
\end{corollary}
\begin{proof}
Right translation $T_d$ on $V$ with zero extension preserves $L$,
because it intertwines the control and response translations.
Since $T_d^N=0$ when $Nd\ge T$, $(I-T_d)y\in L$ implies $y\in L$.
For each $x$, $\psi_{s+d}x=T_d\psi_sx$; consequently the time map
$s\mapsto\psi_s x$ hits each coset other than $L$ at most once.
Fubini gives $M(y+L)=0$ for all $y\notin L$.
For finite $\lambda$ split off the uncompensated Poisson process of
outside jumps. Conditional on at least one such jump, conditioning
on all other terms gives probability zero of the target coset.
With no outside jumps, the residual triplet is $(h_0,0,M|_L)$.
Its mass in $L$ is zero or one by Janssen's coset zero--one law, and
the lifting test identifies the latter case. For infinite $\lambda$,
Janssen's infinite-intensity theorem gives zero directly
\cite[Theorem~10 and Corollary~11]{r88:Janssen}.
\end{proof}

\paragraph{Banach state spaces.}
Brze\'zniak and van Neerven studied equivalence of Gaussian
Ornstein--Uhlenbeck processes in Banach spaces \cite{r88:BvN}.
Theorem~\ref{r88:jumpcomparison} also gives the comparison in
\cite[Section~7, P4]{r88:Kania} for every separable real Banach state
space $X$, including UMD spaces, under the same path hypotheses.
Here $B:E\to X$ is bounded with $E$ Hilbert, the Gaussian convolutions
are $X$-valued, and the jump noise is a Radon $X$-valued L\'evy process.
Choose a countable norming family $(\ell_j)_{j\ge1}$ in the dual unit
ball and set $jx=(2^{-j/2}\ell_j(x))_j$. In
$V=L^2(0,T;\ell^2)$ use the controls with kernel $jU_i(t,s)B$, replace
$\psi_sx(t)$ by $\mathbf1_{s<t}j(U_0(t,s)-U_1(t,s))x$, and use
$\chi_X$ on the input side of the drift. These changes give the same
necessary and sufficient conditions, independently of the norming family.

The squared Hilbert--Schmidt norm of the control kernel is $\sum_j2^{-j}\|B^*U_i(t,s)^*\ell_j\|_E^2$. The embedded control
ranges are the Gaussian reproducing-kernel Hilbert spaces with their
minimum-control norms. For the jumps the map
$\ell\mapsto\ell(\cdot)$ from $X^*$ to $L^2(\nu|_{\{\|x\|_X\le1\}})$
is everywhere defined by the scalar L\'evy condition and bounded by the
closed graph theorem. Applied to $(U_0(t,s)-U_1(t,s))^*\ell_j$, this
bound and Tonelli give $\int_0^T\int_{\|x\|_X\le1}\|\psi_sx\|_V^2
\,\nu(dx)\,ds<\infty$. Scalar evaluation and the compensated isometry
therefore identify the embedded discrepancy with the integral in the
theorem. The embedding of càdlàg $X$-paths continuous at $T$ into $V$
is Borel and injective, hence has a Borel inverse on its image.
The jump record and conditional laws are preserved, completing the proof.

\subsection{What a rectangular Fourier window estimates}
Pillai proved orthogonality for Mat\'ern fields on a fixed domain in
\(\mathbb R^4\) using localized observations. In
\cite[Section~10]{r86:Pillai} he asks about joint consistency of a
different estimator built from the untapered grid Fourier coefficients
and their exact marginal variances. Known smoothness is assumed there;
the numerical example takes \(\nu=3/2\).
Stein studied cross-frequency dependence of raw periodograms under
fixed-domain sampling \cite{r86:Stein}. The consistency theorem for the
debiased spatial Whittle likelihood in \cite[Theorem~1]{r86:debiased}
uses a growing observation domain.

Let \(Y\) be the centered real stationary Mat\'ern field with spectral
density \(m f_\alpha(\xi)=m c_\nu(\alpha^2+|\xi|^2)^{-(\nu+2)}\),
where \(m=\sigma^2\alpha^{2\nu}>0\) and \(\alpha>0\).
On the grid \(hj\), \(h=2\pi/n\), \(j\in\{0,\ldots,n-1\}^4\), set
\[
 Z_k=h^4\sum_jY(hj)e^{-2\pi i k\cdot j/n},\qquad
 u_\alpha(k)=\mathbb E_{1,\alpha}|Z_k|^2,\qquad
 M_n(\alpha)=\frac1{|\mathcal K_n|}\sum_{k\in\mathcal K_n}
                         \frac{|Z_k|^2}{u_\alpha(k)} .
\]
Here \(\mathcal K_n\) retains one representative per nonzero conjugate
pair modulo \(n\), including the self-conjugate modes once. \(M_n\)
is the exact profile estimate of \(m\) for fixed \(\alpha\).

\begin{theorem}[The smoothness threshold for the rectangular grid periodogram]
\label{r86:whittle}
At true parameters \((m_0,\alpha_0)\), for \(0<\nu<2\) and every
compact interval \(K\subset(0,\infty)\),
\[
 \mathbb E\sup_{\alpha\in K}|M_n(\alpha)-m_0|^2\longrightarrow0 .
\]
Thus every \(K\)-valued random choice of the inverse-range parameter yields a
consistent scale estimate, including at the simulated \(\nu=3/2\).
For \(\nu\ge2\), \(M_n(\alpha_0)\) is inconsistent. If \(\nu>2\),
it converges in \(L^1\) to
\[
 \frac{G^2}{\mathbb E_{1,\alpha_0}G^2},\qquad
 G=\sum_{e\in\{0,1\}^4}(-1)^{e_1+\cdots+e_4}Y(2\pi e).
\]
This limit has law \(m_0\chi_1^2\). At \(\nu=2\), for a deterministic
\(0<a<1\) specified below, the scale estimate at the true inverse range
converges in \(L^2\) to \(aG^2/\mathbb E_{1,\alpha_0}G^2+(1-a)m_0\).
Consequently the profiled pair
\((M_n(\widehat\alpha),\widehat\alpha)\) cannot be jointly consistent
for \(\nu\ge2\).
\end{theorem}
\begin{proof}
Use the spectral representation of the field. Fix \(0<\delta<1/4\);
call a mode interior if \(\delta n\le k_l\le(1-\delta)n\) for every
\(l\). Write \(z_l=e^{-2\pi i k_l/n}\). The spectral kernel of
\(W_k=h^{-4}\prod_l(1-z_l)Z_k\) is
\[
 H_k(\xi)=\prod_{l=1}^4\left[(1-z_l)
       \sum_{r=0}^{n-1}(z_l e^{ih\xi_l})^r\right].
\]
On compact frequency sets it tends uniformly over interior modes to
\(P(\xi)=\prod_l(1-e^{2\pi i\xi_l})\), the kernel of \(G\). Indeed each
factor equals \((1-e^{2\pi i\xi_l})(1-z_l)/(1-z_l e^{ih\xi_l})\).
On \(|\xi|_\infty\le\delta n/2\) these ratios are uniformly bounded.
On the complement, partition into cubes \(n(q+[-1/2,1/2)^4)\).
On the portion of cube \(q\) in the complement,
\(f_{\alpha_0}\le C_\delta n^{-2\nu-4}(1+|q|)^{-2\nu-4}\).
Orthogonality of the exponentials bounds the integral of \(|H_k|^2\)
over each whole cube by \(2^8 n^8\). Thus the complementary contribution
to \(\int |H_k|^2 f_{\alpha_0}\) is \(O_\delta(n^{4-2\nu})\).

Suppose \(0<\nu<2\) and put \(s_n=n^{4-2\nu}\). On interior modes
\(\int |H_k|^2 f_{\alpha_0}\asymp_\delta s_n\). For the lower bound
integrate over \(k+[-1/4,1/4]^4\), where each geometric sum has
modulus at least \(c n\) and \(f_{\alpha_0}\ge c n^{-2\nu-4}\).
Fix \(0<\eta<1/4\), and write \(B_k\) for the frequencies with
\(\xi/n\) at sup-distance less than \(\eta\) from \(k/n\) on the
unit torus. Outside \(B_k\), some coordinate sum is bounded by
\(C_\eta\), while each of the other three has squared integral
\(n^2\) over a period. Thus the integral of \(|H_k|^2\) outside
\(B_k\) in each period cube is \(O_\eta(n^7)\). The low-frequency
bound and the sum over cubes give
\[
 \int_{B_k^c}|H_k|^2 f_{\alpha_0}
       \le C_\delta+C_{\delta,\eta}s_n/n=o(s_n).
\]
If both torus distances from \(k/n\) to \(\pm l/n\) exceed \(2\eta\),
the sets \(B_k\) and \(B_{\pm l}\) are disjoint. Cauchy--Schwarz
then makes both the ordinary and the unconjugated covariances of
the unit-variance coefficients at interior \(k,l\) tend uniformly
to zero (reflection of a spectral kernel gives the latter).
By Wick's formula the covariance of their squared moduli is the sum
of the squared moduli of these two covariances. All normalized
periodogram ordinates have uniformly bounded second moments.
The proportion of pairs with a boundary mode or a distance at most
\(2\eta\) is \(O(\delta+\eta^4)+o(1)\), by counting lattice points.
Let \(n\to\infty\), then \(\delta,\eta\downarrow0\). The variance of
\(M_n(\alpha_0)\) tends to zero, and its mean equals \(m_0\).

On \(|\xi|_\infty>\delta n/2\) the ratio \(f_\alpha/f_{\alpha_0}\)
is \(1+O_{\delta,K}(n^{-2})\) uniformly in \(\alpha\in K\).
On the remaining frequencies \(H_k\) is bounded on interior modes,
and \(\int(f_\alpha+f_{\alpha_0})\) is bounded uniformly in \(K\).
The lower bound above consequently gives
\(u_\alpha(k)/u_{\alpha_0}(k)=1+O_{\delta,K}(s_n^{-1}+n^{-2})\)
on interior modes. On all modes these ratios and their reciprocals
are bounded by \(C_K\), by pointwise comparison of the densities.
Using the triangle inequality in \(L^2\) on the sum of the ordinates
times \(\sup_{\alpha\in K}|u_{\alpha_0}(k)/u_\alpha(k)-1|\) gives
\(\|\sup_{\alpha\in K}|M_n(\alpha)-M_n(\alpha_0)|\|_2
\le O_K(\delta)+o_{\delta,K}(1)\). Letting \(\delta\downarrow0\)
proves the asserted uniform convergence.

For \(\nu>2\) these estimates and dominated convergence give
\(W_k\to G\) in mean square uniformly on the interior modes. Since
\(\mathbb E_{1,\alpha_0}G^2>0\), the ratios \(|Z_k|^2/u_{\alpha_0}(k)\)
tend there in \(L^1\) to the stated limit. The fraction of other
modes is \(O(\delta)+O(1/n)\); each ratio and the proposed limit have
expectation \(m_0\). Letting \(\delta\downarrow0\) proves the claim.

At \(\nu=2\), write \(V=\mathbb E_{1,\alpha_0}G^2\). For almost every
\(\theta\in\mathbb T^4\) put
\[
 b(\theta)=c_2\prod_j|1-e^{-2\pi i\theta_j}|^2
       \sum_{q\in\mathbb Z^4}|\theta+q|^{-8},\qquad
 a=\int_{\mathbb T^4}\frac{V}{V+b(\theta)}\,d\theta\in(0,1),
\]
with normalized Haar measure. Then
\(M_n(\alpha_0)\to a G^2/V+(1-a)m_0\) in \(L^2\).
Indeed, for interior modes split \(W_k=L_k+F_k\) at
\(|\xi|_\infty=\delta n/2\) into independent low and high frequencies.
The preceding compact convergence and bounded low-frequency kernels give
\(L_k\to G\) uniformly in \(L^2\), hence also in \(L^4\) by Gaussianity.
The unit-scale high variance tends uniformly to \(b(k/n)\).
To see this, put \(y=\xi/n\). The product of the squared geometric sums
divided by \(n^4\) is the mass-one product Fej\'er kernel on the torus
centered at \(k/n\). The periodized multiplier
\[
 \sum_q n^8 f_{\alpha_0}(n(y+q))
           \mathbf1_{\{|y+q|_\infty>\delta/2\}}
\]
is uniformly bounded and converges uniformly, near every interior center,
to \(c_2\sum_q|y+q|^{-8}\), by summability in \(q\). Fej\'er
concentration gives the claimed variance, including the four multiplier
factors. Already on high frequencies the squared-kernel mass outside
\(B_k\) is \(O_{\delta,\eta}(1/n)\) by the preceding period-cube
estimate. Thus on macroscopically separated, nonreflected interior
pairs both covariances of the \(F_k\) vanish uniformly. The same
pair count and Wick identity show that the interior averages of \(F_k\)
and of \(|F_k|^2-\mathbb E|F_k|^2\), with any uniformly bounded
deterministic weights, tend to zero in \(L^2\) (take \(n\to\infty\),
then \(\eta\downarrow0\)). The first average is Gaussian and also
tends to zero in \(L^4\), so its product with \(G\) tends to zero in
\(L^2\). Applying this to the expansion of \(|L_k+F_k|^2\) with
unit-scale variance denominators \(V+b(k/n)+o(1)\) leaves the Riemann
sums of \((G^2+m_0b)/(V+b)\). Reflection preserves the summands.
The omitted modes have proportion \(O(\delta)+o(1)\) and uniformly
bounded ordinate second moments. Letting \(\delta\downarrow0\)
gives the stated limit.

Finally \(|\partial_\alpha\log f_\alpha|\le2(\nu+2)/\alpha\).
Integration against the squared grid kernel gives the same logarithmic
bound for every \(u_\alpha(k)\). If \(\widehat\alpha\to\alpha_0\) in
probability, all \(u_{\widehat\alpha}(k)/u_{\alpha_0}(k)\) tend uniformly
to one in probability. Since \(M_n(\alpha_0)\) is tight, replacing
\(\alpha_0\) by \(\widehat\alpha\) changes it by \(o_{\mathbb P}(1)\).
This excludes joint consistency.
\end{proof}

\begin{theorem}[Range inconsistency below the scale threshold]
\label{r86:whittlerange}
Suppose \(0<\nu<2\). On every compact search set
\(K\subset(0,\infty)\) that contains the true inverse range
\(\alpha_0\) and at least one other value, no minimizer of Pillai's
profiled criterion converges in probability to \(\alpha_0\).
This includes intervals with the truth at either endpoint and fixed
candidate grids. The scale estimate remains consistent by
Theorem~\ref{r86:whittle}.
\end{theorem}
\begin{proof}
Write \(p=4+2\nu\), \(N=|\mathcal K_n|\), \(T_k=|Z_k|^2/u_{\alpha_0}(k)\)
and \(J=[\min K,\max K]\). The quantities
\[
 b_k=-\partial_\alpha\log u_{\alpha_0}(k),\qquad B=\sum_k b_k
\]
measure the sensitivity of the marginal variances to the inverse range.
We first locate frequencies carrying a fixed proportion of \(B\).

For a mode \(k\), put \(a_j=1+\operatorname{dist}(k_j,n\mathbb Z)\).
Multiply \(h^{-4}Z_k\) by \(1-z_j\) in each nonzero coordinate and by
\(1/n\) in each zero coordinate, with \(z_j=e^{-2\pi i k_j/n}\).
Write \(H_k\) for the resulting Fourier kernel and
\(U_\alpha=\int |H_k|^2 f_\alpha\) for its unit-scale variance. This
preserves \(T_k\) and \(b_k\). Order the \(a_j\) increasingly and set
\[
 t=a_4,\quad D=t^{-p}\prod_{j=1}^4a_j^2,\qquad
 E_s=a_s^{4-s-p}\prod_{j\le s}a_j^2\quad(1\le s\le4).
\]
Uniformly in \(k\) and \(\alpha\in J\),
\begin{align}
 1+D&\lesssim U_\alpha\lesssim 1+\sum_{s=1}^4 E_s, &
 1&\lesssim-\partial_\alpha U_\alpha
    \lesssim 1+a_3^{-2}U_\alpha.                         \label{r87:whweights}
\end{align}
Indeed each normalizing multiplier has modulus comparable to \(a_j/n\).
On \([1/4,1/2]^4\) the geometric-sum formula bounds its normalized
product away from zero. This gives both constant lower bounds; a box of
side \(1/2\) about the signed frequency representative gives the lower
bound \(cD\). On \([-R,R]\), \(1\le R\le n\), the squared integral
of the \(j\)-th normalized factor is \(O(R)\) if \(a_j>4R\), by the
geometric-sum formula, and \(O(a_j^2)\) otherwise, by orthogonality
over a period. Thus a dyadic shell of radius \(R\) with precisely \(s\)
indices \(a_j\le4R\) contributes at most
\(C R^{4-s-p}\prod_{j\le s}a_j^2\). The sum in each of the five
ranges is geometric, giving \(O(1)\) for \(s=0\) and \(O(E_s)\)
otherwise. Outside the central period cube, summing over translated
periods costs \(O(n^{-p}\prod_j a_j^2)\), since \(p>4\).
For \(-\partial_\alpha f_\alpha\asymp(1+|\xi|)^{-p-2}\), the shells
meeting \(|\xi|_\infty\le a_3/16\) have at most two active indices
when \(a_3\) is large. Their total is bounded because
\(4+s-(p+2)<0\) for \(s\le2\). On the complement the density derivative
is bounded by \(Ca_3^{-2} f_\alpha\); bounded \(a_3\) is immediate.
This proves \eqref{r87:whweights}. Both \(U_\alpha\) and
\(-\partial_\alpha U_\alpha\) are comparable on \(J\) to their values
at \(\alpha_0\) by pointwise comparison of the densities and derivatives.
In particular \(b_k\lesssim U_{\alpha_0}^{-1}+a_3^{-2}\), and
\(\sum_k a_3^{-2}=O(n^2)\) since fixing \(a_3=l\) leaves
\(O(l^2n)\) choices.

Put \(\mu=(p-2)/2\). If \(5\le p<8\), fixing \(t\) and writing
\(x_i=a_i/t\) for the other three coordinates gives, by the
decreasing-integrand bound on the grid,
\[
 \sum_{a_4=t}(1+D)^{-1}\lesssim
 t^3\int_{(0,1)^3}\min\{1,t^{-2(3-\mu)}(x_1x_2x_3)^{-2}\}\,dx
 \lesssim t^\mu(1+\log t)^2.
\]
The last inequality follows by setting \(x_i=e^{-y_i}\) and integrating
in \(\sum y_i\). To see where that order is attained, take \(n/4\le
t\le n/2\), and \(a_i\in[2^{j_i},2^{j_i+1})\) for \(i\le3\) with
\(\sum j_i=\lfloor\mu\log_2 t\rfloor\). Choose the ratios
\(j_i/\log_2t\) near an open patch in the plane \(\sum\lambda_i=\mu\),
with \(0<\lambda_1<\lambda_2<\lambda_3<1\), close to
\(((\nu-1)_+,\min(\nu,1),1)\). At this last point the exponents
\(2\sum_{i\le s}\lambda_i+(4-s-p)\lambda_s\) of \(E_2,E_3\)
are \(-2\min(\nu^2,1)\) and \(-1\), respectively.
Such strictly ordered triples exist arbitrarily close to that point
by moving toward the interior of the ordered simplex in the same plane.
Choose a smaller patch on which both exponents stay strictly negative.
Here \(E_1=a_1^{5-p}\le1\), including at \(p=5\), and \(E_4=D\asymp1\).
There are \(\asymp(\log t)^2\) disjoint boxes, by varying \(j_1,j_2\)
in a small rectangle in the projection of the patch and fixing \(j_3\)
by the sum; each contains \(\asymp t^\mu\) modes. Thus \(U_{\alpha_0}\)
and \(b_k\) are comparable to one on these boxes, with
\(\min a_j\to\infty\) uniformly.

If \(4<p<5\), put \(\kappa=5-p\). There is a further lower bound
\(U_{\alpha_0}\gtrsim a_1^\kappa\). For large \(r=a_1\) integrate
the first coordinate on an interval of length \(1/2\) about its signed
peak, where its squared factor is at least \(c r^2\). Integrate the
other three on \([-r/8,r/8]\); their normalized factors have modulus
at least \(c|1-e^{2\pi i\xi_j}|\) there, by the geometric-sum formula,
and hence squared integrals at least \(cr\). The density on this box
is at least \(c r^{-p}\). For bounded \(r\) use \eqref{r87:whweights}.
Fixing \(r,t\), let \(\beta=t^{(p-6)/2}r^{-1+\kappa/2}\le1\). The
same grid bound in the two remaining coordinates gives
\[
 \sum_{a_1=r,\,a_4=t}U_{\alpha_0}^{-1}
 \lesssim t^2 r^{-\kappa}\!\int_{(0,1)^2}
             \min\{1,\beta^2(x_2x_3)^{-2}\}\,dx_2dx_3
 \lesssim t^\mu r^{-1-\kappa/2}(1+\log t).
\]
Here the integral is \(O(\beta(1+\log(1/\beta)))\) and \(r\le t\).
Summation in \(r\) is bounded. To attain this order, fix the first
frequency at zero, so \(a_1=1\). Take \(n/4\le t\le n/2\) and use
the two dyadic coordinates \(a_2,a_3\) with
\(j_2+j_3=\lfloor\mu\log_2t\rfloor\). Take \(j_3/\log_2 t\) in a
small fixed interval just below one. Then \(\lambda_2=\mu-\lambda_3\)
is near \(\nu>0\), and the exponents of \(E_2,E_3\) are
\((4-p)\lambda_2<0\) and \(p-2-(p-1)\lambda_3<0\). This gives
\(\asymp\log t\) boxes of \(\asymp t^\mu\) modes with
\(D\asymp U_{\alpha_0}\asymp b_k\asymp1\) and \(\min_{j\ge2}a_j\to\infty\).

In either case let \(\mathcal A_n\) be the indicated modes together
with their reflections, taking the representatives in \(\mathcal K_n\).
Coordinate sign and permutation multiplicities are bounded; for the
second construction keep the zero in the first coordinate.
The upper and lower counts, summed over \(t\), give
\(\sum_{\mathcal A_n}b_k\gtrsim B\), since \(\mu+1=p/2>2\).
On each compact spectral cube, \(H_k\) tends uniformly there to a
single nonzero kernel \(P_*\). It is \(\prod_{j=1}^4(1-e^{2\pi i\xi_j})\)
in the first case. In the second its first factor is replaced by
\(\int_0^1e^{2\pi i v\xi_1}\,dv\). This follows from the geometric-sum
ratio for \(a_j\to\infty\) and ordinary Riemann sums at zero.

Fix \(\alpha_1\in K\setminus\{\alpha_0\}\). Put
\(r_k=u_{\alpha_0}(k)/u_{\alpha_1}(k)\), \(c_k=|r_k-1|\),
\(C_n=\sum c_k\), and \(X=C_n^{-1}\sum c_kT_k\).
On \(J\) all variance ratios are bounded above and below and
\(-\partial_\alpha\log u_\alpha(k)\asymp b_k\), hence
\(c_k\asymp|\log r_k|\asymp b_k\) and \(C_n\asymp B\).
Up to an additive constant the profiled criterion is
\(L(\alpha)=N\log M_n(\alpha)+\sum\log u_\alpha(k)\).

If \(\alpha_1<\alpha_0\), then \(r_k<1\) and
\[
 L(\alpha_1)-L(\alpha_0)
 \le -\sum_k\log r_k-C_nX/M_n(\alpha_0).
\]
The largest eigenvalue of the nonnegative Gaussian quadratic form \(X\)
is bounded below. Indeed, let \(G\) be the real Gaussian variable with
kernel \(P_*\mathbf1_{[-1,1]^4}\) in the spectral representation of
the true field. The Rayleigh quotient in its unit direction is
\(C_n^{-1}\sum c_k|\mathbb E[Z_kG]|^2/(u_{\alpha_0}(k)\mathbb E G^2)\).
The ratios are bounded below on \(\mathcal A_n\), by kernel convergence
and \(U_{\alpha_0}\asymp1\) there. Thus \(X\) stochastically dominates
a fixed positive multiple of \(\chi_1^2\). As \(M_n(\alpha_0)\to m_0\),
the criterion difference is below \(-cB\) with probability bounded away
from zero.

If \(\alpha_1>\alpha_0\), then \(r_k>1\). Write \(A=\sum\log r_k\) and
\(F=\sum c_k/r_k\). The inequality \(\log r>1-1/r\) for \(r>1\) and
\(b_k\asymp1\) on \(\mathcal A_n\) give \(A-F\ge cB\).
Split the true field into independent parts \(Y_l,Y_h\) on
\(\{|\xi|_\infty\le R\}\) and its complement. For any field \(y\) put
\(D(y)=\sum c_k |Z_k(y)|^2/u_{\alpha_0}(k)\). Given small \(\epsilon>0\),
choose \(R\) with \(f_{\alpha_0}\le(1+\epsilon)f_{\alpha_1}\) on the
complement. Then \(\mathbb E D(Y_h)\le m_0(1+\epsilon)F\).
Since \(A,F\asymp B\), Markov's inequality, with \(\epsilon\) chosen
small, gives \(D(Y_h)\le m_0(A-c'B)\) with probability bounded
below. For \(g\in C^4([0,2\pi]^4)\), the normalized grid coefficient
is bounded by \(C\|g\|_{C^4}\), uniformly in the mode and grid:
a zero coordinate averages, and discrete summation by parts in a
nonzero coordinate bounds the one-dimensional operator by
\(2\sup|g|+2\pi\sup|\partial_j g|\). Iterate in the four coordinates.
Since \(U_{\alpha_0}\gtrsim1\), this gives \(D(Y_l)\lesssim
B\|Y_l\|_{C^4}^2\). The band-limited field is a centered Gaussian
random element of \(C^4\) (its derivatives through order four have
continuous versions by their Gaussian increment bounds). Each ball
about zero there has positive probability. On a sufficiently small such
ball and the independent event for \(Y_h\), the triangle inequality for
\(D^{1/2}\) gives \(D(Y)\le m_0(A-c''B)\). Since
\(M_n(\alpha_0)\to m_0\), on events of probability bounded away from
zero,
\[
 L(\alpha_1)-L(\alpha_0)\le -A+D(Y)/M_n(\alpha_0)\le-c'''B.
\]

Finally \(N\sup_J|M_n'|\lesssim\sum b_kT_k=O_{\mathbb P}(B)\) by
expectation and variance-ratio comparison. The uniform scale consistency
of Theorem~\ref{r86:whittle} gives \(\sup_J |L'|/B=O_{\mathbb P}(1)\).
Consistency of a minimizer on \(K\) would bring its criterion value
within \(o_{\mathbb P}(B)\) of \(L(\alpha_0)\). Both strict comparisons
above contradict this conclusion.
\end{proof}

\subsection{A Brownian motion perturbed by weighted fractional noise}
Khalaf et al.\ ask when mixed weighted fractional Brownian motion
is a martingale or has a law equivalent to Brownian motion
\cite[Definition~1 and Remark~1]{r85:Khalaf}. Write $Y=B+X$, where $B$
is standard Brownian motion and $X$ is independent, centered Gaussian
with the covariance of Bojdecki, Gorostiza and Talarczyk \cite{r85:BGT},
\[
 R(s,t)=\int_0^{s\wedge t}u^a\big[(s-u)^b+(t-u)^b\big]du,
 \qquad a>-1,\quad |b|<1,\quad |b|<1+a .
\]
The case $a=0$ is the fractional Brownian mixture studied by Cheridito
\cite{r85:Cheridito}. The calculation below applies Shepp's covariance
criterion \cite{r84:Shepp}; the square-integrable perturbation condition
for Brownian motion plus independent Gaussian noise appears also in
\cite[Theorem~20]{r85:BaudoinNualart}.

\begin{corollary}[Equivalence and the martingale property]\label{r85:weighted}
On each $[0,T]$, $0<T<\infty$, the law of $Y$ is equivalent to standard
Wiener measure exactly when $b>1/2$. It is singular otherwise.
In its natural filtration $Y$ is a martingale exactly when $b=0$.
On $[0,\infty)$ its law is singular to standard Wiener measure
for every parameter pair in the stated range.
\end{corollary}
\begin{proof}
Brownian increments give the $L^2(0,T)$ norm on step integrands.
The additional covariance form is positive. Equivalence therefore
requires and suffices that it have an $L^2((0,T)^2)$ kernel:
necessity is the Gaussian covariance criterion, and sufficiency follows
since adding a positive Hilbert--Schmidt operator to the identity
preserves the lower norm bound. Interval increments generate the path
field. Away from $s=t$ the distributional kernel is
$K(s,t)=b(s\wedge t)^a|s-t|^{b-1}$. For $b>0$ this is the entire
kernel, as direct integration on $[0,s]\times[0,t]$ recovers $R(s,t)$.
For $b>1/2$ its squared norm is
\[
 \frac{2b^2}{2b-1}\int_0^T s^{2a}(T-s)^{2b-1}\,ds<\infty ,
\]
since $a>b-1>-1/2$. For $b\ne0$ with $b\le1/2$, the square already
diverges near the diagonal inside any interior time interval.
For $b=0$ the added covariance is multiplication by $2s^a$ and has
no $L^2$ kernel, even on an interval separated from zero. Gaussian
dichotomy proves singularity. In this last case $Y$ is a Brownian
motion with deterministic clock $t+2t^{a+1}/(a+1)$, hence a
martingale (a constant multiple of Brownian motion if $a=0$).
For $b\ne0$ two ordered, separated increments have nonzero covariance:
integrate $K$, of constant strict sign, over their rectangle.
This is incompatible with the martingale property.

It remains to check the infinite horizon for $b>1/2$. There
$d=a+b>0$ and $\operatorname{Var}Y(t)=t+c_{a,b}t^{1+d}$,
where $c_{a,b}>0$. At $t_n=2^n$ put $h_n=t_n^{1/2+d/4}$.
Gaussian tails give $|B(t_n)|\le h_n$ eventually almost surely,
whereas $\mathbb P(|Y(t_n)|\le h_n)\le C t_n^{-d/4}$.
The two assertions follow from Borel--Cantelli and separate the laws.
\end{proof}

\begin{corollary}[Semimartingales and the parameters determined by a path]
\label{r86:weighted}
Write $Y^{a,b}$ for the mixture in Corollary~\ref{r85:weighted}.
It is a semimartingale in its natural filtration exactly when
$b=0$ or $b>1/2$. On each $[0,T]$, $T>0$, two such laws are
equivalent exactly when their parameter pairs coincide or both
exponents $b,b'$ exceed $1/2$. All other pairs give singular laws.
\end{corollary}
\begin{proof}
Fix $J=[u,v]\subset(0,T)$ and let $V_h$ be the sum of squared increments
on its uniform dyadic mesh of width $h$. If $b<0$, then almost surely
\[
 h^{-b}V_h\longrightarrow \frac{2}{1+b}\int_u^v s^a\,ds .
\]
Indeed an increment starting at $s$ has variance
$h+2\int_s^{s+h}t^a(s+h-t)^b\,dt$, uniformly asymptotic to
$2s^ah^{1+b}/(1+b)$. Increments $r\ge2$ cells apart have covariance
bounded in modulus by $C h^{1+b}(r-1)^{b-1}$ by the off-diagonal kernel
in Corollary~\ref{r85:weighted}; adjacent covariances satisfy the same
$O(h^{1+b})$ bound by Cauchy--Schwarz. The Gaussian covariance identity
for squares gives $\operatorname{Var}(h^{-b}V_h)=O(h)$, summable on
dyadic meshes. Thus $V_h$ diverges for $b<0$, precluding a continuous
semimartingale. The growth exponent identifies $b$; the limits on
rational subintervals then identify the function $s^a$ and hence $a$.
For $b>0$ the weighted noise alone has expected sum of squared
increments $O(h^b)$, summable on dyadic meshes. Brownian quadratic
variation and Cauchy--Schwarz consequently give $V_h\to v-u$. At $b=0$
the deterministic clock gives instead the limit
$\int_u^v(1+2s^a)\,ds$ almost surely (the variance of $V_h$ is $O(h)$).
These observations distinguish the three signs of $b$ and the different
clocks at $b=0$.

For positive exponents the increment covariance norm on $J$ is
comparable to $L^2(J)$: its lower bound is the Brownian part and its
upper bound follows by integrating $|K(s,t)|$ in either variable.
Equivalent laws must therefore have an $L^2(J^2)$ difference of
increment kernels. When $b\ne b'$ its leading diagonal singularity is
of order $|s-t|^{\min(b,b')-1}$, which excludes $\min(b,b')\le1/2$.
When $b=b'\le1/2$ and $a\ne a'$, choose $J$ where $s^a-s^{a'}$ is
bounded away from zero. All positive equivalence cases follow from
Corollary~\ref{r85:weighted} and Gaussian dichotomy gives the rest.

It remains to exclude semimartingales for $0<b\le1/2$.
We use the conditional-increment argument of \cite{r85:Cheridito}.
Write $D_j$ for the increments on $J$, and $M_j$ for the conditional
mean of $D_j$ given the preceding mesh increments. The Gram matrix of
these past increments is bounded above by $ChI$ by the norm comparison.
For $r\ge2$, positivity of $K$ on $J^2$ gives
$\mathbb E D_jD_{j-r}\ge c h^{1+b} r^{b-1}$.
Gaussian least squares thus gives
\[
 \|M_j\|_2^2\ge c h^{1+2b}\sum_{r=2}^{j-1}r^{2b-2}.
\]
With $n=|J|/h$ and $j$ in the last half of the mesh, this yields
$\sum_j\|M_j\|_2\to\infty$, growing at least as $n^{1/2-b}$ for
$b<1/2$ and as $\sqrt{\log n}$ at equality.
For $A=\sum_j|M_j|$, Gaussianity gives
$\mathbb EA=\sqrt{2/\pi}\sum_j\|M_j\|_2$ and
$\|A\|_2\le\sum_j\|M_j\|_2$. Paley--Zygmund gives a fixed positive
lower bound for $\mathbb P(A>\mathbb EA/2)$.
The bounded predictable step integral $\sum_j\operatorname{sgn}(M_j)D_j$
differs from $A$ by a martingale sum of second moment at most
$\sum_j\mathbb E D_j^2=O(1)$. These integrals are unbounded in
probability, which excludes a semimartingale. For $b=0$ use the
deterministic clock; for $b>1/2$ use finite-horizon equivalence to
Wiener measure and invariance of the semimartingale property under
equivalent changes of measure.
\end{proof}

\subsection{A fractional L\'evy process reveals its jump history}
Norros and Saksman proved local independence for fractional Brownian
motion and asked whether fractional L\'evy processes have the same
property \cite[\S5, Q.4]{r83:NS}. For the nonanticipative moving average
of Marquardt \cite{r83:Marquardt}, finite-variation and first-derivative
criteria are known \cite[Theorems~2.1, 2.4 and~4.1]{r83:FLpath}. Cohen and Maejima obtained linear recovery from the family
of kernels indexed by all real times \cite{r83:CM2011}. In the
finite-activity case, higher derivatives identify the observation field
itself.

\begin{theorem}[Jump history from an interval of increments]\label{r83:levylocal}
Let the times $s$ form a two-sided homogeneous Poisson process of positive
finite intensity, with independent identically distributed real marks
$J_s$ independent of the times. Suppose $\mathbb P(J_s=0)=0$,
$\mathbb E J_s=0$ and $0<\mathbb E J_s^2<\infty$. For $0<\beta<1/2$ put
\[
 X(t)=\sum_s J_s\big((t-s)_+^\beta-(-s)_+^\beta\big).
\]
For every $a<b$ the completed field generated by the increments on
$(a,b)$ equals the marked Poisson history on $(-\infty,b)$.
For the well-balanced kernel $|t-s|^\beta-|s|^\beta$, the corresponding
increment field on every such interval equals the marked Poisson field on
the entire real line. In both models any two nonempty bounded observation
intervals have infinite mutual information. Up to a constant factor the
covariance in each case is that of
fractional Brownian motion with Hurst index $H=\beta+1/2$.
\end{theorem}
\begin{proof}
The anchored kernel belongs to $L^2(ds)$ because its remote tail is
$O_t(|s|^{\beta-1})$. Thus truncation defines the usual mean-square
integral. Translation and scaling of the increment kernel give stationary
increments and $\mathbb E|X(t)-X(u)|^2=C|t-u|^{2\beta+1}$.
Polarization yields the covariance assertion.

On $[-M,M]$ the sum over $s<-2M-2$ has a locally uniformly convergent
first derivative. Indeed, conditional on the times, the centered independent
series at zero has summable variances since
$\sum_{s<-2M-2}|s|^{2\beta-2}<\infty$ almost surely. The change of its
derivative between zero and any point of the interval is bounded termwise by
$C_M|J_s||s|^{\beta-2}$, a summable majorant by the Poisson intensity
formula. The remote series vanishes at zero, so integration gives uniform
convergence of the series itself. The same majorant, up to constants
depending on the order, bounds every remote derivative of order at least
two. The remaining terms are a finite sum of continuous powers. This
gives a continuous version, smooth between successive jumps.

Fix a deterministic $t\in(a,b)$. Almost surely it is not a jump time.
List the jumps before $t$ backwards as $s_1>s_2>\cdots$, and put
$y_j=(t-s_j)^{-1}$ and $c_j=J_{s_j}y_j^{2-\beta}$. Then
$y_1>y_2>\cdots\downarrow0$, every $c_j$ is nonzero and
$\sum |c_j|<\infty$ almost surely. The increment observations give
every positive derivative at $t$ by finite differences. In particular
they determine
\[
 p_m=\frac{X^{(m+2)}(t)}{\beta(\beta-1)\cdots(\beta-m-1)}
     =\sum_{j\ge1}c_j y_j^m,\qquad m\ge0 .
\]
Since $|p_m-c_1y_1^m|\le y_2^m\sum_{j\ge2}|c_j|$, the limits
$|p_m|^{1/m}\to y_1$ and $p_m/y_1^m\to c_1$ recover the most recent
jump and its mark. Subtract its contribution from every $p_m$ and repeat.
These countably many measurable limits recover every marked jump before $t$.
Taking a deterministic sequence $t_n\uparrow b$ proves one field
inclusion. Conversely every observed increment is a limit of truncated
sums over times before $b$, since the anchoring term cancels. Continuity
allows a countable choice of endpoints.

For the well-balanced kernel the same convergence estimates apply at
both tails. At a fixed $t$ there are almost surely no two distinct jumps
equidistant from $t$ (the Poisson process has diffuse second factorial
moment measure). Now use all jumps, put $y_j=(t-s_j)^{-1}$ and
$c_j=J_{s_j}|y_j|^{2-\beta}$, and order by decreasing $|y_j|$.
Differentiation gives the very same sequence $p_m=\sum_j c_j y_j^m$,
since the $n$th derivative of $|t-s|^\beta$ off $s$ is the falling
factorial times $|t-s|^\beta(t-s)^{-n}$. The ratio
$p_{m+1}/p_m\to y_1$ recovers the side as well as the distance.
Subtracting successive terms recovers the entire marked configuration.

The last jump before any fixed time earlier than both right endpoints
has a continuous distribution and is measurable in both observation
fields. Partitioning its law into $N$ equiprobable sets gives a common
discrete observation of entropy $\log N$ for every $N$.
\end{proof}

\subsection{Invariant measures for scalar weighted shifts}
Write the scalar shift as $B_w(x_0,x_1,\ldots)=(w_1x_1,w_2x_2,\ldots)$, where the bounded
weights are nonzero and the scalars are real or complex. Write
$W_0=1$, $W_n=w_1\cdots w_n$, and use $U_n,V_n$ for two weight
sequences $u,v$. A nontrivial probability is one different from $\delta_0$.

\begin{theorem}[An invariant law detects chaos for weighted shifts]
\label{r79:support}
Let $B_w$ be a continuous backward shift with nonzero weights on a real
or complex Fr\'echet sequence space $X$ over $I=\mathbb N_0$ or
$\mathbb Z$, with unconditional coordinate basis. Define $W_0=1$ and
$W_{j+1}=w_{j+1}W_j$. The following conditions are equivalent:
$B_w$ admits an invariant probability different from $\delta_0$;
the fixed vector $(W_j^{-1})_{j\in I}$ belongs to $X$;
$B_w$ is chaotic; and $B_w$ admits a full-support Bernoulli invariant
probability. The latter can be taken exact for a unilateral shift
and strongly mixing for a bilateral shift. On general sequence spaces
the full-support implication of \cite[Section~4, Question~3]{r79:CM}
fails even for an unweighted unilateral shift on a separable Hilbert
sequence space with dense finitely supported vectors.
\end{theorem}
\begin{proof}
Let $Y$ have an invariant law $\mu\ne\delta_0$. Invariance gives the
same distribution to all $W_jY_j$, and this distribution has positive
mass away from zero. Choose $a>0$ with $q=\mathbb P(|W_jY_j|>a)>0$
and a compact $K\subset X$ with $\mu(K^c)<q/2$. The coordinate
modulus map into $X$ is continuous: its increment at $x,y$ is a
coordinate multiplier of $x-y$ with coefficients of modulus at most
one, and these multipliers are equicontinuous by unconditionality
\cite[Section~1, (1)]{r79:CM}. Hence $u=\int_K|y|\,d\mu(y)$ exists
in $X$, by uniform finite-valued approximation on the compact range
and completeness. Evaluation at $j$ yields
$u_j\ge aq/(2|W_j|)$. The vector $(W_j^{-1})_{j\in I}$ is therefore
a bounded coordinate multiplier of $u$ and belongs to $X$.

Membership is the classical chaos criterion \cite{r80:GE}.
By \cite[Theorems~4.1(i) and~4.2(i)]{r80:Ag}, unconditional
convergence of $\sum_{j\in I}W_j^{-1}e_j$ gives a random vector
$(\xi_j/W_j)_{j\in I}$ with full support, where the $\xi_j$ are
independent copies of a scalar variable. Coordinate rescaling recovers
the independent sequence and intertwines $B_w$ with its Bernoulli
shift, giving the stated mixing properties. A full-support probability
differs from $\delta_0$. The paragraph preceding Question~3 in
\cite{r79:CM} treats boundedly complete unconditional Banach bases.

For the counterexample put $\Delta x_j=x_{j+1}-x_j$ and take
$X=\{x:\Delta x\in\ell^2(\mathbb N_0)\}$ with squared norm
$|x_0|^2+\|\Delta x\|_2^2$. The map $x\mapsto(x_0,\Delta x)$ is an
isometry onto $\mathbb K\oplus\ell^2$, and every coordinate is
continuous. The shift is bounded since $(Bx)_0=x_0+\Delta x_0$ and
$\Delta Bx=B\Delta x$. Finitely supported sequences are dense. Keep
$x_0$ and the first $N$ differences, obtaining endpoint $b=x_N$.
Over the next $L$ coordinates take differences $-b/L$, then take zero
differences. The resulting finitely supported vector differs from $x$
in norm by at most $\|(\Delta x_j)_{j\ge N}\|_2+|b|/\sqrt L$.
Under any invariant law the differences $\Delta Y_j$ all have the law
of $\Delta Y_0$, and $\Delta Y_j\to0$ almost surely by square
summability. Thus every invariant probability is concentrated on constant
sequences, all of which are fixed. In particular $\delta_{\mathbf1}$ is
nontrivial, while no invariant law has full support.
\end{proof}

\Needspace{6\baselineskip}
\begin{theorem}[Gaussian replacement and common invariant measures]
\label{ws42:shift}
The following conclusions answer Questions~7.4, 7.6 and 7.7 of
\cite{ws42:GMM}, respectively.
\begin{enumerate}
\item On $\ell_p(\mathbb N_0)$, $1\le p<\infty$, $B_u,B_v$ admit
nonorthogonal nontrivial invariant product probabilities if and only if
\begin{equation}\label{ws42:scale}
 \sum_{n\ge0}(|U_n|^{-p}+|V_n|^{-p})<\infty,\qquad
 \sum_{n\ge0}\left(\log\frac{|U_n|}{a|V_n|}\right)^2<\infty
 \quad\text{for some }a>0 .
\end{equation}
In that case they admit equivalent centered Gaussian invariant product
probabilities with full support.
\item On the same $\ell_p$, the shifts have a common nontrivial invariant
probability if and only if $|u_j|=|v_j|$ for every $j\ge1$ and
$\sum_n|U_n|^{-p}<\infty$. On $c_0$ the same criterion holds with $|U_n|\to\infty$ in place of summability.
\item On $c_0(\mathbb N_0)$, $B_w$ has a nontrivial invariant probability
if and only if $|W_n|\to\infty$.
\end{enumerate}
\end{theorem}
\begin{proof}
Theorem~\ref{r79:support} gives the necessary $\ell_p$ summability
for each invariant probability and the full assertion on $c_0$, since
$(W_n^{-1})$ belongs to $c_0$ exactly when $|W_n|\to\infty$.

For an invariant product the coordinates $W_nX_n$ are independent with a common law. Let $P,Q$ be the respective laws of the normalized moduli under the two products. The empirical proportion of zero coordinates tends almost surely to $P(\{0\})$ and $Q(\{0\})$ under the respective laws. Nonorthogonality forces these limits to agree, say with value $p_0<1$.

Put $F(y)=|y|^2/(1+|y|^2)$ and $t_n=\log|V_n/U_n|$. Apply to each coordinate the same operation under both laws: conditionally on $Y_n=|V_nX_n|$, draw a Bernoulli bit with success probability $(1+F(Y_n))/3$, with independent auxiliary variables across coordinates. A nonzero measure dominated by the two original laws has a nonzero image dominated by both bit laws. These are Bernoulli products with respective parameters $B(t_n)$ and $b$, where
\[
 B(t)=\tfrac13(1+\int F(e^t y)\,dP(y)),\qquad
 b=\tfrac13(1+\int F(y)\,dQ(y)).
\]
All parameters lie in $[1/3,2/3]$. The squared Hellinger distance between two Bernoulli laws in this interval is comparable to the squared difference of their parameters, by the mean value theorem applied to their square roots. Kakutani's product theorem \cite{ws42:Kakutani} therefore gives $\sum_n(B(t_n)-b)^2<\infty$. The limits of $B$ at $-\infty,+\infty$ are $1/3,(2-p_0)/3$. Since $Q$ also assigns mass $1-p_0$ to $(0,\infty)$, $b$ lies strictly between these limits. Differentiation under the integral gives a continuous positive derivative for $B$. Hence $t_n\to c=B^{-1}(b)$ and $\sum_n(t_n-c)^2<\infty$. This is \eqref{ws42:scale} with $a=e^{-c}$.

Conversely choose independent isotropic Gaussian coordinates of scales
$|U_n|^{-1}$ and $(a|V_n|)^{-1}$, respectively. The two products are
invariant by independence and isotropy, and summability of expected
$p$th powers puts them on $\ell_p$. They have full support by positive
finite-dimensional densities and independence of their vanishing tails.
In real dimension $d=1$ or $2$, Gaussian coordinates of scale ratio
$e^s$ have affinity $(\cosh s)^{-d/2}$. Since
$\log\cosh s\le s^2/2$, \eqref{ws42:scale} gives product equivalence by
Kakutani's theorem \cite{ws42:Kakutani}. The Gaussian equivalence test
itself is \cite[Corollary~5.8]{ws42:GMM}.

For a common invariant law, $U_nX_n$ and $V_nX_n$ both have the
nontrivial law of $X_0$. Strict increase of $s\mapsto\E F(e^s X_0)$
forces $|U_n/V_n|=1$ for every $n$, hence equality of the individual
weight moduli. Conversely equal moduli give the same isotropic Gaussian
product in the preceding construction. For two shifts on $c_0$
with equal weight moduli and $|U_n|\to\infty$, the independent coordinates
$|U_n|^{-1}\xi_n$ with uniform unit scalars $\xi_n$ give a common
invariant law. Necessity follows from the same radial test and part (3).
\end{proof}

\begin{corollary}[Frequent hypercyclicity with a unique invariant probability]
\label{r79:fullshift}
There is a frequently hypercyclic weighted backward shift on
$c_0(\mathbb N_0)$ whose only invariant Borel probability is $\delta_0$.
This answers \cite[Question~1.6]{r79:GM}.
\end{corollary}
\begin{proof}
The shift constructed in \cite[proof of Proposition~9.9]{r79:BGMM} is
frequently hypercyclic. Its weight products satisfy $|W_n|\le M$ along
an infinite subsequence for some $M<\infty$, by the density estimate
for $\{n:|W_n|>M\}$ there. Part~(3) of Theorem~\ref{ws42:shift}
applies.
\end{proof}

For arbitrary invariant probabilities, call two operators orthogonal if
every pair of their respective invariant probabilities giving zero mass
to the origin is mutually singular \cite{ws42:GMM}.

\begin{theorem}[Separation inside one similarity class]\label{ws42:separate}
On real or complex $\ell_p$, $1\le p<\infty$, and on $c_0$, there is a
continuum of mutually similar, pairwise orthogonal weighted shifts. There
are also similar, nonorthogonal weighted shifts admitting no equivalent
nontrivial invariant probabilities. Each individual shift in these
constructions has a full-support strongly mixing Gaussian invariant law.
These examples answer Questions~7.2, 7.3 and 7.5 of \cite{ws42:GMM}
negatively.
\end{theorem}
\begin{proof}
Use the Rudin--Shapiro polynomials \cite{ws42:Rudin}, starting with
$P_0=Q_0=1$ and taking $P_{r+1}=P_r+z^{2^r}Q_r$,
$Q_{r+1}=P_r-z^{2^r}Q_r$. Their coefficients are signs and the
$P_r$ are successive prefixes of one sequence $(\varepsilon_k)$.
For any Hilbert space isometry $U$, the parallelogram identity gives
$\|P_r(U)h\|^2+\|Q_r(U)h\|^2=2^{r+1}\|h\|^2$. Thus for any
probability-preserving $S$ and bounded $h$, the averages
$2^{-r}\sum_{k<2^r}\varepsilon_k h(S^kx)$ tend to zero almost surely:
their squared $L^2$ norms are summable by taking $Uh=h\circ S$.

Fix $\lambda>1$, $d\ge1$, and a nonempty set $J$ of residues modulo
$d$. Put $m_{dk+j}=\varepsilon_k$ for $j\in J$ and $m_{dk+j}=0$
otherwise, using $0\le j<d$. Give $B_t$ the weights
$\lambda e^{t(m_n-m_{n-1})}$, $n\ge1$, $t\in\mathbb R$.
The bounded invertible diagonal with entries $e^{tm_n}$ conjugates
$B_t$ to the constant-weight shift $\lambda B$. For independent standard
isotropic Gaussians $\xi_n$, the law of
$(\lambda^{-n}e^{-tm_n}\xi_n)_n$ has full support on either space by
the preceding tail argument. The coordinate map intertwines $B_t$ with
the independent one-sided shift, so the law is invariant and strongly mixing.

Write $E_J=\{x:x_{dk+j}=0\ (k\ge0,\ j\in J)\}$. We recover $t$
off $E_J$ almost surely for every $B_t$-invariant law. With $F$ as
above, define the bounded averages
\[
 R_r(b,x)=\frac1{|J|2^r}\sum_{j\in J}\sum_{k<2^r}
 \varepsilon_k F(\lambda^{dk+j}e^{b\varepsilon_k}x_{dk+j}).
\]
At $b=t$ each argument of $F$ is
$y_{k,j}=\lambda^je^{tm_j}(B_t^{dk}x)_j$, so
$R_r(t,x)\to0$ by the signed-average estimate for $S=B_t^d$.
For $a>0$, both $R_r(t+a,x)-R_r(t,x)$ and
$R_r(t,x)-R_r(t-a,x)$ dominate the unsigned average of
\[
 g_a(y_{k,j}),\qquad
 g_a(z)=\min\{F(e^az)-F(z),\ F(z)-F(e^{-a}z)\}.
\]
This function is positive exactly off zero. By the pointwise ergodic
theorem the average has a limit, positive almost surely off $E_J$.
Indeed on the invariant event where that limit is zero, every
nonnegative summand's conditional expectation is zero, hence each
$g_a(y_{0,j})$ vanishes almost surely there. Invariance gives the same
for every $k$, forcing $x\in E_J$. Testing the countable set of rational
$b$ now gives the same Borel function for all parameters,
\[
 \inf\{b\in\mathbb Q:\liminf_r R_r(b,x)>0\}=t
 \quad\text{almost surely on }E_J^c .
\]
Thus every measure absolutely continuous with respect to invariant
laws for two distinct parameters is concentrated on $E_J$.

Take $d=1$ and $J=\{0\}$. Then $E_J=\{0\}$, giving orthogonality.
The cumulative-weight ratios are $e^{(s-t)(m_n-m_0)}$, bounded above
and away from zero. This violates both proposed separation alternatives
in Question~7.2 already at window length zero, and similarity gives the
counterexample to Question~7.3. For Question~7.5 take $d=2$, $J=\{1\}$.
The nonzero vector $x_n=\lambda^{-n}\mathbf1_{\{n\text{ even}\}}$
is fixed by every $B_t^2$. The two-point orbit measures for distinct
parameters both charge $x$, so are nonorthogonal. If invariant laws for
the two shifts were equivalent they would both be concentrated on $E_J$.
An invariant law with $X_1=0$ almost surely has $X_0=0$ almost surely,
hence every coordinate vanishes by invariance. This excludes equivalence
for every nontrivial pair.
\end{proof}
Question~7.1 of \cite{ws42:GMM} asks which pairs on $\ell_p$ are
orthogonal. The following test uses only finite weight prefixes. Put
$Y=[0,1]^{\mathbb N_0}$, $R(x)_i=|x_i|^p/(1+|x_i|^p)$ and
\[
 (T_wz)_i=\frac{|w_{i+1}|^p z_{i+1}}{1+(|w_{i+1}|^p-1)z_{i+1}},
 \qquad
 A_{w,N}(P)=\max_{z\in Y}\frac1N\sum_{k<N}P(T_w^kz).
\]
For a polynomial in the first $L$ coordinates this maximum is over
$[0,1]^{L+N-1}$. Put $s(t)=t/(1-t)$, $s(1)=\infty$, and let
$K_m=\{\sum_i s(z_i)\le m,\ \max_{j<m}z_j\ge1/(m+1)\}$ for $m\ge1$.

\begin{theorem}[Finite tests for orthogonality]\label{ws43:finite}
On real or complex $\ell_p$, $1\le p<\infty$, $B_u,B_v$ are orthogonal
if and only if for all integers $m,r\ge1$ there are rational-coefficient
polynomials $P,Q$, nonnegative on a finite cube $[0,1]^L$, $L\ge m$, and integers
$N,M\ge1$, such that $P+Q\ge1$ on the projection of $K_m$ and
$A_{u,N}(P)+A_{v,M}(Q)<2^{-r}$. The projection is given by truncating
the sum defining $K_m$ to $i<L$. In this case one Borel set separates
all invariant pairs assigning zero mass to the origin.
\end{theorem}
\begin{proof}
The compact sets $K_m$ exhaust $R(\ell_p\setminus\{0\})$. Since
$RB_w=T_wR$, failure of orthogonality gives invariant probabilities
$\mu,\nu$ on $Y$ with $(\mu\wedge\nu)(K_m)>0$ for some $m$.
Conversely $R(\ell_p)$ is forward invariant, hence invariant modulo
each corresponding probability. Restrict there, remove the atom at zero
and normalize. The common mass survives. Lifting $z$ to coordinates
$s(z_i)^{1/p}\theta_i$ with independent uniform unit scalars preserves
domination and intertwines the shifts in law.

For compact $K\subset Y$, the common mass $(\mu\wedge\nu)(K)$ is the
infimum of $\mu(f)+\nu(g)$ over continuous $f,g\ge0$ with $f+g\ge1$
on $K$. Indeed, partition $K$ according to the smaller of the two
densities relative to $\mu+\nu$, approximate both pieces from inside
by compacta, and multiply a continuous partition separating them by a
cutoff supported near $K$ and equal to one there. Regularity makes the
error arbitrarily small. Compact convex minimax interchanges this
infimum and the supremum over the two compact convex sets of invariant
probabilities. For a continuous $f$ the maximum of $\mu(f)$ over
$T_w$-invariant $\mu$ equals $\inf_N A_{w,N}(f)$: invariance gives one
inequality and limit points of maximizing empirical measures give the
other. Rational cylinder polynomials approximate $f,g$ uniformly from
above, so the certificate infimum in the statement equals the maximal
common mass on $K_m$. A truncated point extends by zeros, proving the
projection formula. For certificates indexed by $r$, invariance and
summability give $P_r\to0$ almost surely for each $u$-law and $Q_r\to0$
for each $v$-law. The inverse image under $R$ of the union over $m$ of
$K_m\cap\{P_r\to0\}$ is the required separator.
\end{proof}

\subsection{A finite phase test for a diagonal plus a backward shift}
Let \(T=D_\lambda+B_\omega\) on complex \(\ell^2(\mathbb N)\), where
\(T e_1=\lambda_1 e_1\) and
\(T e_{k+1}=\lambda_{k+1} e_{k+1}+\omega_k e_k\). The \(\lambda_k\)
are distinct and unimodular and the \(\omega_k\) are positive and bounded.
The test below also covers the original assumption \(\lambda_k\to1\)
of \cite[Question~8.9]{r79:GMMbook}.

For \(j\ge1\), normalize to unit norm the finite vector
\[
 v_j\ \propto\ \sum_{k=1}^j
    \left(\prod_{l=1}^{k-1}\frac{\lambda_j-\lambda_l}{\omega_l}\right)e_k
\]
using a positive normalization constant. This is the eigenvector formula in
\cite[Proposition~4.22(0)]{r79:GMMbook}. Put
\(V_N=[v_1\ \cdots\ v_N]\) in the first \(N\) coordinates and
\(y_N=\sum_{j=1}^N2^{-j}v_j\). The matrices \(V_N\) are invertible.
With \(z^a=\prod_{j=1}^Nz_j^{a_j}\), set
\[
 H_N=\{z\in\mathbb T^N:z^a=1\text{ for all }a\in\mathbb Z^N
                               \text{ with }\lambda^a=1\},\qquad
 C_N(z)=V_N\operatorname{diag}(z)V_N^{-1}.
\]
Each \(H_N\) is determined by a finite basis of the lattice of integer
relations; when these relations are trivial it is the full torus.

\begin{theorem}[A scalar criterion from the finite eigenvector matrices]
\label{r81:phase}
The operator \(T\) is hypercyclic if and only if the following limit is zero:
\[
 \lim_{N\to\infty}\min_{z\in H_N}
 \left(\frac{\det(G_N(z)+y_Ny_N^*)}{\det G_N(z)}-1\right),
 \qquad G_N(z)=I_N+C_N(z)C_N(z)^* .
\]
The limit always exists. Each minimum is over a compact group given by
finite-dimensional algebraic equations in the phase coordinates.
\end{theorem}
\begin{proof}
Write \(F_N=\operatorname{span}\{e_1,\ldots,e_N\}\).
Kronecker's theorem gives \(H_N\) as the closure of the nonnegative
powers of \((\lambda_1,\ldots,\lambda_N)\). Thus all phase multipliers
on \(F_N\) used here are limits of powers of \(T|_{F_N}\). For any
\(y\in\ell^2\), put
\[
 d_N(y)=\inf_{n\ge0,\,x\in F_N}
       (\|x\|^2+\|T^nx-y\|^2)^{1/2}.
\]
These distances decrease in \(N\) and are 1-Lipschitz in \(y\).
For \(y=y_N\), one may minimize over \(z\in H_N\) in place of \(n\).
For each fixed \(z\), minimizing the quadratic in \(x\) gives
\(y_N^*G_N(z)^{-1}y_N\) as its squared minimum. The matrix determinant
lemma identifies this with the displayed determinant ratio minus one.

Put \(y=\sum_{j\ge1}2^{-j}v_j\); then \(\|y-y_N\|\le2^{-N}\).
Thus the stated limit is \(\lim_N d_N(y)^2\).
If \(T\) is hypercyclic, transitivity provides arbitrarily small \(x\)
with an iterate arbitrarily close to \(y\). At that fixed time, finite
truncation of \(x\) preserves the approximation; the limit is zero.
Conversely, a zero limit gives \(x_k\to0\) and \(T^{n_k}x_k\to y\).
The vector \(y\) is cyclic: the averages of
\(\overline\lambda_j^{\,n}T^ny\) over \(0\le n<m\) converge in norm to \(2^{-j}v_j\)
as \(m\to\infty\), by distinctness and the absolutely summable coefficients.
For any finitely supported \(u\) its orbit has compact closure, since it
lies in a finite span of these eigenvectors. Pass to a subsequence with
\(T^{n_k}u\to w\). Given a nonempty target open set \(O\), choose a
polynomial \(P\) with \(w+P(T)y\in O\). Then \(u+P(T)x_k\to u\)
and its \(n_k\)-th iterates eventually lie in \(O\). Density of the
finitely supported vectors gives transitivity, hence hypercyclicity.
\end{proof}

\begin{corollary}[A finite set of tail tests at positive radius]\label{r81:shiftradius}
Suppose \(\lambda_k\to1\) and
\(R=\liminf_{n\to\infty}(\omega_1\cdots\omega_n)^{1/n}>0\).
Then \(T\) is hypercyclic if and only if, for each \(k\) with
\(|\lambda_k-1|\ge R\),
\[
 \sum_{j=k}^\infty
 \left(\prod_{l=k}^{j-1}\frac{\omega_l}{|\lambda_k-\lambda_{l+1}|}
 \right)^2=\infty .
\]
There are only finitely many such \(k\). Whenever the tests hold, \(T\)
is strongly mixing for a full-support invariant Gaussian probability.
\end{corollary}
\begin{proof}
A nonzero eigenvector of \(T^*\) with its first nonzero entry at \(k\) has
eigenvalue \(\bar\lambda_k\). The coordinate recurrence gives the displayed
products as its coordinate moduli, after normalization at \(k\). For
\(|\lambda_k-1|<R\) the series always diverges: the geometric means of the
denominators tend to \(|\lambda_k-1|\), whereas the numerator has lower
geometric mean \(R\). When \(\lambda_k=1\) the former means tend to zero.
Thus the tests are precisely \(\sigma_p(T^*)=\varnothing\), which is
necessary for hypercyclicity by the one-dimensional quotient.

Replace \(\lambda_j\) in the unnormalized formula for \(v_j\) by \(z\)
and sum over all \(k\ge1\), obtaining the eigenvector field \(E(z)\) on
\(|z-1|<R\). The root test gives locally uniform absolute convergence,
since \(\lambda_k\to1\). The closed span \(L\) of this field is invariant
and contains every \(v_j\) with \(|\lambda_j-1|<R\); hence it has finite
codimension. If \(L\ne H\), \(T^*|_{L^\perp}\) has an eigenvector, contrary
to the tests. An open arc of the unit circle with closure inside the disk
already spans \(H\) by analytic uniqueness. Integrating
\(E(e^{it})E(e^{it})^*\) against Lebesgue measure on this arc gives
an injective trace-class covariance \(Q\) with \(TQT^*=Q\).
Its centered circular Gaussian law has full support. Every matrix
coefficient of \(T^nQ\) tends to zero by the Riemann--Lebesgue lemma,
so finite coordinate blocks separated by \(n\) are asymptotically
independent under this Gaussian law. Approximation by cylinder functions
gives strong mixing. Ergodicity and full support imply hypercyclicity.
\end{proof}

Part~(4) of \cite[Proposition~4.22]{r79:GMMbook} covers
\(R>\sup_k|\lambda_k-1|\). Part~(5) in the author manuscript asserts
nonhypercyclicity when \(R<\sup_k|\lambda_k-1|\).
Its root-test argument uses an upper limit where \(R\) is defined
as a lower limit. For example, take \(\lambda_1=-1\),
\(\lambda_j=e^{i/j}\) for \(j\ge2\), and \(\omega_j\) alternating between
constant blocks of \(1\) and \(3\), with successive endpoints
\(2^{2^m}\) for \(m\ge1\). The lower and upper limits of the weight
geometric means are \(1\) and \(3\); the only tested index is \(k=1\)
and its product has upper geometric mean \(3/2\). The resulting
operator is strongly mixing for a full-support Gaussian probability.

\subsection{Typical triangular operators with a prescribed diagonal}
Write \(\mathcal C(a)\) for the contractions on \(\ell^2(\mathbb N)\)
whose matrices are upper triangular with diagonal \(a=(a_j)\), where
\(|a_j|<1\). We use the strong operator topology. In the full contraction
ball the infinite-multiplicity backward shift describes the generic unitary
equivalence class by Eisner and M\'atrai \cite[Theorem~5.2]{r81:EM}.
Whether a generic operator is chaotic under the triangular and unimodular
diagonal restrictions is \cite[Question~8.10]{r79:GMMbook}.

\begin{theorem}[Coisometric completion with a prescribed diagonal]
\label{r81:triangular}
If \(\sum_j(1-|a_j|)=\infty\), the contractions unitarily equivalent to
the backward shift of countably infinite multiplicity form a comeager
set in \(\mathcal C(a)\). If the sum is finite, \(\mathcal C(a)\)
contains no coisometry. In particular, for every \(M>1\), a typical
upper triangular operator of norm at most \(M\) with distinct unimodular
diagonal entries is unitarily equivalent to \(M\) times this shift.
It is chaotic.
\end{theorem}
\begin{proof}
First suppose the sum diverges. On any subsequence on which it diverges,
temporarily renumber its entries and put \(r_i=(1-|a_i|^2)^{1/2}\).
Define an upper triangular matrix \(S\) with diagonal \(a_i\) and entries
\[
 S_{ij}=r_i r_j\prod_{i<k<j}(-\bar a_k),\qquad i<j .
\]
Renard studies the contractive completion with the additional prescribed
first superdiagonal \(r_ir_{i+1}\) \cite[Corollary~1.2]{r81:Renard}.
For real diagonal entries his formula (1.2) is the display above.
For complex diagonal entries the conjugation here follows from row
orthogonality.
The telescoping identity
\(\sum_{j>i}r_j^2\prod_{i<k<j}|a_k|^2=1\)
follows because every infinite tail product vanishes. It gives row
norms one. For \(i<l\), the inner product of the two rows on column
\(l\) is canceled by the sum on \(j>l\), by the same identity at \(l\).
The orthonormal rows define a coisometry. Moreover the unit vector
with entries \(r_j\prod_{k<j}(-a_k)\) lies in its kernel, again by
telescoping. A divergent nonnegative series can be partitioned into
countably many divergent subseries, by distributing successive finite
blocks of mass at least one to each index infinitely often.
Applying the construction on the disjoint subsequences in their
original order gives, on every tail, an upper triangular coisometry
\(S'\) with precisely the prescribed tail diagonal and an infinite
dimensional kernel.

Let \(A\) be the first \(N\)-by-\(N\) corner of any matrix in
\(\mathcal C(a)\). Take a coisometry \(K\) from the tail onto the first
\(N\) coordinates with \(K^*\) mapping into \(\ker S'\). Then
\[
 \begin{pmatrix}
 A &(I-AA^*)^{1/2}K\\ 0&S'
 \end{pmatrix}
\]
is a coisometry in \(\mathcal C(a)\). Its kernel contains
\(\ker S'\cap\ker K\), which has infinite dimension. The completion
agrees with the given operator on the first \(N\) basis vectors, so
these coisometries are strongly dense.

For a contraction \(C\), coisometry is equivalent to the image of the
unit ball being dense in the unit ball: the image is weakly compact and
its support function is \(\|C^*h\|\) in direction \(h\). Testing density
on countable dense subsets gives a \(G_\delta\) condition in the strong
topology. On coisometries, strong convergence implies strong convergence
of adjoints by weak convergence together with equality of their vector
norms. Thus the kernel projection \(I-C^*C\) is strongly continuous on
this set; having rank at least \(m\) is relatively open, since a nonzero
Gram determinant of \(m\) projected vectors witnesses it. The density
just proved now gives infinite-dimensional kernel on a comeager set.
Every such \(C\) satisfies \(C^nx\to0\) for each \(x\), first on finite
coordinate spaces, whose spectral radii are less than one, then by
contractivity on the whole space. The Wold decomposition of \(C^*\)
therefore gives the asserted pure backward shift.

If the sum is finite, each column's squared off-diagonal norm is at most
\(1-|a_j|^2\). Hence every such \(C\) differs from a diagonal unitary by
a Hilbert--Schmidt operator. It is Fredholm of index zero. A coisometry
would then be unitary, and an upper triangular unitary preserves and
reduces each finite coordinate space, forcing unit-modulus diagonal.

For the last claim apply the dense completion argument to the full
triangular class after division by \(M\). It preserves each chosen
diagonal, for which \(|a_j|=1/M\), so density holds in this class as
well as on each fiber. The same \(G_\delta\) argument applies there.
For \(M\) times a backward shift, any finite target can be reached from
arbitrarily small inputs by putting its rescaled copy far along the
sequence. Finite initial data can also be approximated by periodic
vectors, repeating a long block with successive factors \(M^{-p}\)
for period \(p\). This gives hypercyclicity and dense periodic points.
\end{proof}

\subsection{Visit frequencies of hypercyclic operators}
For a bounded linear operator $T$ on a separable Banach space, put
$N_T(z,V)=\{n\ge1:T^nz\in V\}$ and write $\overline d$ for upper
asymptotic density. Hypercyclic vectors have dense orbit. For a
hypercyclic $T$, its parameter of return to zero is
\[
 c(T)=\sup_{z\text{ hypercyclic}}\overline d\bigl(N_T(z,B(0,1))\bigr).
\]
Replacing the unit ball by any ball centered at zero gives the same
number by scalar multiplication. Upper frequent hypercyclicity asks for
a single vector visiting each nonempty open set with positive upper density.

\begin{theorem}[Return frequency from invariant probabilities]
\label{r79:upper}
Let $T$ be hypercyclic. There is a comeager set of $z$ such that every open $V$ satisfies
\[
 \overline d\bigl(N_T(z,V)\bigr)\ \ge\ c(T)\sup_{\nu}\nu(V),
\]
where the supremum runs over the ergodic $T$-invariant Borel probabilities.
If $c(T)>0$ and $T$ admits an invariant probability with full support,
its upper frequently hypercyclic vectors form a comeager set.
This answers \cite[Question~8.12]{r79:GMMbook}, following the case of dense uniformly recurrent points in Theorem~5.23 of the same work.
\end{theorem}
\begin{proof}
Write $c=c(T)$. Fix an ergodic invariant probability $\nu$, an open $V$
and $0<a<c\nu(V)$. Choose a Borel $V'\subset V$ and $\epsilon>0$
such that $V'+B(0,\epsilon)\subset V$ and $a<c\nu(V')$, using the
increasing inner neighborhoods of $V$. For any nonempty open $W$ take
$y_0\in\operatorname{supp}\nu$. By \cite[Proposition~4.7]{r79:GM},
some $x\in W-y_0$ satisfies $T^nx\to0$ along a set $D$ with
$\overline d(D)\ge c$. Choose a ball $B$ about $y_0$ with $x+B\subset W$.

On $L^2(\nu)$ the Koopman isometry $Uf=f\circ T$ fixes constants.
The mean ergodic theorem gives
$K^{-1}\sum_{j=0}^{K-1}U^j\mathbf1_{V'}\to\nu(V')$ in $L^2(\nu)$.
Choose $\eta>0$ with $a<c(\nu(V')-\eta)$. Applying $U^n$ to the
average and pairing with $\mathbf1_B$, we obtain, for one $K$ and every $n$,
\[
 \frac1K\sum_{j=0}^{K-1}\nu(B\cap T^{-(n+j)}V')
          \ \ge\ \nu(B)(\nu(V')-\eta).
\]
For all sufficiently large $n\in D$, $\|T^{n+j}x\|<\epsilon$ for
every $0\le j<K$. If $y\in B$ and $T^{n+j}y\in V'$, then
$T^{n+j}(x+y)\in V$. Sum the last estimate over these $n\le N$.
Each time $n+j$ occurs at most $K$ times, hence
\[
 \frac1{\nu(B)}\int_B \#\bigl(N_T(x+y,V)\cap[1,N+K]\bigr)\,d\nu(y)
 \ \ge\ (\#(D\cap[1,N])-O(1))(\nu(V')-\eta).
\]
For arbitrarily large $N$ the right side exceeds $a(N+K)$. Some
$z=x+y\in W$ therefore has visit proportion greater than $a$ through
time $N+K$. For every $m$, the union over $L\ge m$ of the open sets
$\{z:\#(N_T(z,V)\cap[1,L])>aL\}$ is dense. Baire's theorem and a
countable sequence of thresholds increasing to $c\sup_\nu\nu(V)$
give the bound for fixed $V$. Intersect these comeager sets over all
finite unions of members of a countable base. Every open $V$ is an
increasing union of such sets, so continuity from below for each $\nu$
gives the simultaneous assertion. For a ball centered at zero, the supremum
on the right is one by $\delta_0$, and the definition of $c(T)$ gives the
matching upper bound for every hypercyclic vector. Finally, ergodic
decomposition of a full-support invariant probability gives a component
of positive mass for every nonempty open set.
\end{proof}

\subsection{Full-support invariant measures on compact models}
Dayan and Glasner proved that a compact frequently hypercyclic system
admits a full-support invariant probability \cite[Theorem~2.4]{r80:DG}.
Grivaux and Matheron asked whether one can require ergodicity. The
following compactification preserves the existence of both kinds of measure.
\begin{theorem}[Full-support ergodicity in a compact model]\label{r80:compact}
Let $X=F^*$ be norm-separable with separable predual $F$, and let $T=A^*$
for $A\in\mathcal B(F)$. There is a compact metrizable system $(K,S)$
and a dense equivariant topological embedding $J:X\to K$ such that
$S$ admits a full-support invariant probability if and only if $T$ does. The same
equivalence holds for full-support ergodic probabilities.
Frequent hypercyclicity and density of periodic points pass from $T$ to $S$.
\end{theorem}
\begin{proof}
Choose norm-dense sequences $(h_j)$ in $X$ and $(a_j)$ in the unit ball
of $F$, with $h_0=0$. Record, for $n,j\ge0$, the observations
\[
  \|T^nx-h_j\|\in[0,\infty],\qquad (T^nx)(a_j)\in\widehat{\mathbb K},
\]
where $\widehat{\mathbb K}$ is the one-point compactification of the
scalar field. Let $J$ be this map, $K$ its image closure in the
countable product of compact metrizable spaces, and $S$ the shift of
the time index. The time-zero distances to dense centers detect norm
convergence to every $x\in X$, so $J$ is a topological embedding.
Since $SJ=JT$ and the image is dense, the preimage of each nonempty
open set is nonempty and open. This transfers frequent hypercyclicity,
dense periodic points, and full-support invariant probabilities, preserving ergodicity.

For the converse write $d_j(z)$ for the time-zero distance observations
and put $E=\{z:d_0(z)<\infty\}$. This nonempty open set is forward
invariant since $d_0(Sz)\le\|T\|d_0(z)$ there. If $Jx_i\to z\in E$,
the sequence $(x_i)$ is bounded. Weak-star compactness of bounded balls
and the evaluations on $(a_j)$ give a unique weak-star limit $\pi z$.
Thus $\pi:E\to X$ is well defined and $\pi S=T\pi$ by weak-star
continuity of $T$. On $\{d_0\le m\}$ it is weak-star continuous, since
the evaluations metrize the weak-star topology on bounded balls.
The norm Borel sets are weak-star Borel: norm-closed balls are weak-star
closed, and norm separability gives a countable base of norm-open balls.
Hence $\pi$ is Borel. Weak-star lower semicontinuity gives
$\|\pi z-h_j\|\le d_j(z)$.

For a full-support invariant probability $\lambda$ on $K$,
$E\subset S^{-1}E$ and invariance show these sets agree modulo $\lambda$.
Since $\lambda(E)>0$, its normalized restriction to $E$ is invariant
and so is its image under $\pi$. For ergodic $\lambda$, one has
$\lambda(E)=1$ and the image is ergodic. Every ball $B(h_j,r)$ receives positive
mass: $\{d_j<r\}$ is nonempty and open, lies in $E$ by the triangle
inequality, and maps inside the ball. These balls give full norm support.
\end{proof}

\begin{corollary}[Frequent hypercyclicity on a compact connected space]
\label{r80:GMcompact}
There is a frequently hypercyclic homeomorphism of a compact connected
metric space with dense periodic points and no full-support ergodic
probability. This answers \cite[Question~1.15]{r79:GM}.
\end{corollary}
\begin{proof}
Apply Theorem~\ref{r80:compact} to the chaotic, frequently hypercyclic
operator on $\ell^2$ without a full-support ergodic probability in
\cite[Lemma~6.4 and Example~7.7 in the published numbering]{r79:GMMbook}. Every
operator on a Hilbert space is an adjoint. The compact model $K$ is
connected as the closure of a connected space. Its map $S$ is onto
since it has a dense forward orbit visiting each open set infinitely often.

The inverse limit of $(K,S)$ is compact, connected and metrizable, and
its natural shift is a homeomorphism. Every basic open set is the inverse
image of an open set under the projection onto the oldest coordinate
specified. For any inverse history of a frequently hypercyclic point,
visits to this basic open set after that many steps are a fixed translate
of its forward visits downstairs. Periodic points downstairs have periodic
lifts, which are dense by the same description of basic open sets.
A full-support ergodic probability upstairs would project to one on $K$.
\end{proof}

\subsection{Boundary contact for universal Hardy-space multipliers}
An operator is \(2\)-universal for ergodic systems if every ergodic
probability-preserving transformation of a standard nonatomic probability space
has an isomorphic realization under this operator by a full-support invariant
probability of finite second moment.
Grivaux proved that on a Hilbert space this property forces the unimodular
point spectrum to have full Haar measure, and asked whether it forces every
point of the circle to be an eigenvalue \cite[Theorem~1.9 and Question~5.4]{r80:GrUniversal}.
For affine multipliers the distinction occurs exactly at internal tangency.

\begin{theorem}[Affine multiplier classification]\label{r80:tangent}
Let \(B(x_0,x_1,\ldots)=(x_1,x_2,\ldots)\) on complex \(\ell^2(\mathbb N_0)\)
and \(A=aI+bB\), where \(a,b\in\mathbb C\). Then \(A\) is \(2\)-universal for
all ergodic systems if and only if
\[
 |b|>1+|a| \qquad\text{or}\qquad |b|=1+|a|\ \text{ and }\ a\ne0 .
\]
\end{theorem}
\begin{proof}
For \(b\ne0\), the coordinate recurrence for an eigenvector gives the point
spectrum \(D(a,|b|)\), the open disk with center \(a\) and radius \(|b|\).
For \(b=0\) it is the singleton \(\{a\}\). By \cite[Theorem~1.9]{r80:GrUniversal}
the intersection with the unit circle must have full Haar measure. The
farthest point of the circle from \(a\ne0\) is \(-a/|a|\), at distance
\(1+|a|\). For \(a=0\) the distance is constant. This proves necessity.

For sufficiency write \(w(\lambda)=(\lambda-a)/b\) and \(K(w)=(w^j)_{j\ge0}\).
In the strict case set \(p=1\). In the equality case set
\(\lambda_0=-a/|a|\) and \(p(\lambda)=(\lambda-\lambda_0)^6\).
On the circle define \(F(\lambda)=p(\lambda)K(w(\lambda))\), with
\(F(\lambda_0)=0\) in the equality case. Then \(AF(\lambda)=\lambda F(\lambda)\).
In the strict case \(F\) is smooth. At tangency, writing
\(\lambda=\lambda_0e^{it}\) gives
\[
 1-|w(\lambda)|^2=\frac{2|a|}{(1+|a|)^2}(1-\cos t).
\]
For \(i=0,1,2\), differentiation of the geometric series gives
\(\|K^{(i)}(w)\|_2=O((1-|w|^2)^{-i-1/2})\) near the boundary.
Since \(w(\lambda_0e^{it})\) has bounded first and second derivatives,
the \(i\)-th derivative of \(K(w(\lambda_0e^{it}))\) is
\(O(|t|^{-2i-1})\). The derivatives of \(F\) of orders \(i=0,1,2\)
are therefore \(O(|t|^{5-2i})\). They extend through \(t=0\), giving
\(F\in C^2(\mathbb T;\ell^2)\).

Let \(v_n\) be the \(n\)-th Fourier coefficient of \(F\), with
\(F(\lambda)=\sum_n v_n\lambda^n\). Two integrations by parts give
\(\sum_n\|v_n\|<\infty\). Each coordinate of \(F\) is a polynomial in
\(\lambda\), so \(v_n=0\) for \(n<0\). The eigenvector identity gives
\(Av_n=v_{n-1}\). If a continuous functional annihilates every \(v_n\),
it annihilates \(F(\lambda)\) on the circle. The analytic function
\(w\mapsto\ell(K(w))\) on the disk then vanishes along an arc in the disk
and hence vanishes identically. Its Taylor coefficients show \(\ell=0\).
Thus the \(v_n\) have dense linear span. Their first coordinates are the
coefficients of \(p\), so deleting the finitely many indices from \(0\) to
\(\deg p\) leaves a family contained in the proper first-coordinate
hyperplane. The chain \(z_k=v_{-k}\) satisfies all hypotheses of
\cite[Theorem~1.4]{r80:GrUniversal}, including \(Az_0=0\).
Remark~2.2 of the same paper supplies the finite second moment.
\end{proof}

Under the coefficient identification \(H^2(\mathbb D)\cong\ell^2(\mathbb N_0)\),
\(2B-I=M_\phi^*\) for \(\phi(z)=2z-1\). Its unimodular point spectrum is
\(\mathbb T\setminus\{1\}\), and \(1\notin\phi(\mathbb D)\).
Theorem~\ref{r80:tangent} therefore gives negative answers both to
Grivaux's Question~5.4, for either class of source systems, and to
Question~5.5, which asks whether circle containment is necessary and
sufficient for \(2\)-universality for invertible systems
\cite[Questions~5.4--5.5]{r80:GrUniversal}.

\begin{theorem}[Finite defects of the universality criterion]\label{r80:holes}
Grivaux proved that the orbit criterion allows only finitely many missing
unimodular eigenvalues \cite[opening of Section~5]{r80:GrUniversal}.
Every finite \(S\subset\mathbb T\) occurs as exactly
\(\mathbb T\setminus\sigma_p(A)\) for a Hardy-space multiplier adjoint
satisfying the stronger criterion for all ergodic systems.
\end{theorem}
\begin{proof}
In Grivaux's argument, for the chain \(Az_n=z_{n+1}\) of the criterion, unconditional convergence
makes \(F(\lambda)=\sum_n\lambda^nz_{-n}\) uniformly convergent on
\(\mathbb T\), and \(AF(\lambda)=\lambda F(\lambda)\).
Choose a nonzero functional annihilating the span left after the prescribed
finite deletion. The full chain has dense span, so \(\ell(F(\lambda))\)
is a nonzero Laurent polynomial. Every missing eigenvalue is one of its zeros.

To realize the prescribed set, write \(h(e^{it})=\prod_{s\in S}|e^{it}-s|^2\).
Since \(h\) is a real trigonometric polynomial, it is the real part on
\(\mathbb T\) of an analytic polynomial \(g\). Put
\(\psi(z)=z\exp(-\varepsilon g(z))\), taking \(\varepsilon>0\) small
enough that \(|\psi'-1|<1/4\) and \(|\psi(z)-z|<1/2\) on
\(|z|\le2\). Integration along segments gives univalence there, and
Rouch\'e's theorem gives \(D(0,3/2)\subset\psi(D(0,2))\).
Set \(\phi(z)=\overline{\psi^{-1}(\bar z)}\) on \(D(0,3/2)\).
For \(\lambda\in\mathbb T\), the only zero of \(\phi(z)-\bar\lambda\)
in that disk is \(\overline{\psi(\lambda)}\), of multiplicity one and
modulus \(\exp(-\varepsilon h(\lambda))\le1\). The quotient by this
linear factor and its reciprocal belong to \(H^\infty(\mathbb D)\).
Since \((M_z-uI)^*=B-\bar u I\), its kernel is trivial for \(|u|=1\)
and one-dimensional for \(|u|<1\), by the coordinate recurrence.
Thus \(M_\phi^*\) has eigenvalue \(\lambda\) exactly when
\(|\psi(\lambda)|<1\), which is equivalent to \(\lambda\notin S\).

Take \(p(\lambda)=\prod_{s\in S}(\lambda-s)^6\) and define
\(F(\lambda)=p(\lambda)K(\psi(\lambda))\) on \(\mathbb T\setminus S\),
with zero values on \(S\), where \(K(w)=(w^j)_{j\ge0}\).
This is an eigenvector field for \(M_\phi^*\). Near each \(s\in S\),
\(1-|\psi(\lambda)|^2=1-e^{-2\varepsilon h(\lambda)}
\asymp|\lambda-s|^2\). The derivative estimates from the proof of
Theorem~\ref{r80:tangent} show that \(F\) is \(C^2\). Its coordinates
extend holomorphically through the closed unit disk, so its negative
Fourier coefficients vanish. The others are absolutely summable in norm
and span densely, since \(\psi\) maps any open arc disjoint from \(S\)
to a nonconstant arc in the disk where the analytic identity argument
of that proof applies. The first coordinate of \(F\) is the polynomial
\(p\), giving the required finite deletion. Theorem~1.4 and Remark~2.2
of \cite{r80:GrUniversal} now give \(2\)-universality for all ergodic systems.
\end{proof}

\subsection{Backward histories and Hardy multipliers analytic across the circle}
Flytzanis obtained a square-summable backward orbit from a noninvertible
Hilbert-space model with finite second moment
\cite[p.~270]{r81:Flytzanis}. Without a moment hypothesis the backward
histories still give analyticity throughout the open unit disk.

\begin{lemma}[An analytic field from a noninvertible law]\label{r81:backwardkernel}
Let \(C\) be bounded on a separable complex Banach space \(X\) and \(m\ge1\).
If \(A=C^m\) is noninvertible modulo some invariant Borel probability,
there is a nonzero holomorphic field \(G:\mathbb D\to X\) satisfying
\(CG(z)=zG(z)\). In particular \(\mathbb D\subset\sigma_p(C)\).
\end{lemma}
\begin{proof}
For an invariant law \(\mu\) take its stationary two-sided histories
\((X_j)_{j\in\mathbb Z}\), with \(AX_j=X_{j+1}\). Their finite-dimensional
laws follow by forward iteration from \(\mu\) at the oldest time. Almost
surely, some ball of finite radius \(M\) is visited by \(X_{-n}\) with
positive asymptotic frequency. This follows from Birkhoff's theorem for
balls of integer radius, whose conditional probabilities increase to one.
For every \(\delta>0\) the interval \([n,(1+\delta)n]\) therefore contains
a visit for all large \(n\), by subtracting the two asymptotic counts.
Forward iteration from that visit gives
\(\limsup_n\|X_{-n}\|^{1/n}\le1\).

Conditionally on the common time-one value, draw two independent histories
\(X_j,Y_j\). Their marginals retain this bound. With positive probability
\(X_0\ne Y_0\): otherwise the conditional law of \(X_0\) given \(AX_0\)
is almost surely a point mass, whose measurable location gives a
measure-theoretic inverse. Choose such a pair and set
\(F(z)=\sum_{n\ge0}(X_{-n}-Y_{-n})z^n\). This is holomorphic on the
disk, with \(F(0)\ne0\) and \(AF(z)=zF(z)\).
Now put \(G(z)=\sum_{k=0}^{m-1}z^{m-1-k}C^kF(z^m)\).
The terms have distinct Taylor residue classes modulo \(m\), so
\(G\not\equiv0\), and telescoping gives \(CG(z)=zG(z)\).
At any zero of \(G\) its first nonzero Taylor coefficient is an eigenvector
with eigenvalue at that point.
\end{proof}

The same history argument locates the part of a stationary law which the
remote future fails to determine. For a bounded operator \(A\) on a
separable complex Banach space \(X\), let \(M\) be the closed linear span
of the values of all holomorphic maps \(F:\mathbb D\to X\) satisfying
\(AF(z)=zF(z)\). Put \(q:X\to X/M\).

\begin{theorem}[Linear observations determined by the remote future]
\label{r82:history}
For every \(A\)-invariant probability \(\mu\), the map \(q\) on
\((X,\mu)\) is measurable in the completed tail
\(\mathcal T=\bigcap_{n\ge1}\sigma(A^n x)\), with each field completed
under \(\mu\). If \(\mathcal T\) is trivial, the support of \(\mu\) is in
a single coset of \(M\). If in addition \(\mu\) has full support, the
values of these fields on any nonempty open subset of \(\mathbb D\) span
\(X\) densely, and \(\bigcup_{n\ge1}\ker A^n\) is dense.
\end{theorem}
\begin{proof}
The operator \(A\) preserves \(M\), since it multiplies each field value by its eigenvalue. The space \(M\) is also the closed span of all Taylor coefficients at zero,
by Cauchy's formula and the Taylor series. Their recurrence gives
\(A^{j+1}v_j=0\) for the coefficient of order \(j\).
Fix \(n\ge1\). Draw two stationary histories independently conditional
on their shared value at time \(n\). The differences \(D_j=X_j-Y_j\)
satisfy \(D_n=0\) and \(AD_j=D_{j+1}\). The subexponential bound on
backward histories in Lemma~\ref{r81:backwardkernel} applies to each
marginal. Thus
\[
    F(z)=\sum_{k\ge0}D_{n-1-k}z^k
\]
is almost surely holomorphic on the disk, with \(AF(z)=zF(z)\).
Its coefficient of order \(n-1\) is \(D_0\), so \(qX_0=qY_0\) almost
surely. The conditional law of \(qX_0\) given \(X_n\) is a point mass.
Since the quotient is a separable Banach space, the location of this mass
has a measurable version. Hence \(q\) is measurable in each completed
\(\sigma(A^n x)\), and therefore in \(\mathcal T\). A trivial tail makes
\(q\) constant almost surely. Full support then gives \(M=X\). Finally,
any functional annihilating all field values on a nonempty open subset
annihilates all of them throughout the disk by analytic uniqueness.
\end{proof}
In particular the conclusion applies to ordinary universality for all
ergodic systems: the fair one-sided Bernoulli shift has trivial tail, and
its prescribed realization has full support. The need for spanning fields
on the open disk already follows without imposing a moment on this realization.

For a nonzero model with full-support exact law, a nonzero analytic field
gives an eigenvector at each point of the disk, dividing out any zero of
the field as in Lemma~\ref{r81:backwardkernel}. Taking limits from inside
the disk puts the full circle in the approximate point spectrum.
The unilateral backward shift shows that the algebraic conclusions alone
do not imply the existence of such a law: its fields \(K(z)\) span densely
on every open subdisk, but its powers tend strongly to zero.

\begin{theorem}[A proper component for an analytic Hardy multiplier]\label{r81:hardy}
Let \(R\) be a nonconstant holomorphic function on a neighborhood
of \(\overline{\mathbb D}\), and let \(A=R(B)\) on complex
\(\ell^2(\mathbb N_0)\), identified with \(H^2(\mathbb D)\).
Thus \(A=M_{R^\sharp}^*\), where \(R^\sharp(z)=\overline{R(\bar z)}\).
The following are equivalent:
\begin{enumerate}
\item some \(A\)-invariant probability is noninvertible as a
measure-preserving system;
\item \(A\) is \(2\)-universal for all standard nonatomic ergodic systems;
\item On some component \(\Omega\) of \(\{w\in\mathbb D:|R(w)|<1\}\),
the restriction \(R:\Omega\to\mathbb D\) is proper, and \(R\) is not
a finite Blaschke product times a unimodular constant.
\end{enumerate}
Ordinary universality for all such systems is equivalent to these conditions.
\end{theorem}
\begin{proof}
Take \(\rho>1\) such that \(R\) is holomorphic on a neighborhood of the
closed disk of radius \(\rho\), and set
\(U_\rho=\{w:|w|<\rho,\ |R(w)|<1\}\).
A component \(\Omega\) of \(U_1\) has proper restriction precisely
when \(|R|=1\) everywhere on \(\partial\Omega\): a boundary point
mapping into the disk violates compactness of preimages, and the converse
follows by compactness of \(\overline\Omega\). This also says that
\(\Omega\) is a whole component of \(U_\rho\). Indeed a path in \(U_\rho\)
leaving \(\Omega\) would first meet its boundary there. Conversely a
component of \(U_\rho\) contained in \(\mathbb D\) has boundary in
\(\{|R|=1\}\), since its closure is in \(\{|w|<\rho\}\).

Given a noninvertible invariant law, Lemma~\ref{r81:backwardkernel}
with \(C=A\) supplies a nonzero holomorphic field \(F\) on \(\mathbb D\)
with \((R(B)-\lambda I)F(\lambda)=0\). Remove the finite set of
values in \(\mathbb D\) attained at zeros of \(R'\) in \(\{|w|\le\rho\}\).
For each remaining \(\lambda\), factor the finitely many closed-disk
zeros of \(R-\lambda\). The residual factor and its reciprocal are holomorphic
near \(\overline{\mathbb D}\), hence give inverse operators at \(B\).
The factors \(B-wI\) with \(|w|=1\) are injective by the coordinate
recurrence, and all factors commute. Thus the eigenspace is the kernel
of \(\prod_{R(w)=\lambda,\ |w|<1}(B-wI)\). For the distinct roots here
the finite coordinate recurrence has basis \(K(w)=(w^j)_{j\ge0}\).
In particular, the expansion of \(F(\lambda)\) in these vectors is unique.

On \(U_\rho\) above the remaining disk define the coefficient \(c(w)\) by
this expansion for \(|w|<1\), and by zero for \(|w|\ge1\). It is
holomorphic, including at sheets crossing \(|w|=1\). To see this near
\(w_0\), choose \(s\) with \(\max(1,|w_0|)<s<\rho\) such that no root
over \(R(w_0)\) lies on \(|w|=s\). For \(\lambda\) near \(R(w_0)\)
the roots inside this larger circle form a fixed number \(d\) of
distinct analytic branches, by the argument principle and the implicit
function theorem. Inverting their Vandermonde matrix on the first \(d\)
coordinates of \(F\) gives holomorphic coefficients. By uniqueness these
are the values already specified, including zeros on exterior branches.
Every component of \(U_\rho\) not contained in \(\mathbb D\) contains an
open subset with \(|w|>1\). Removing preimages of the exceptional values
deletes only finitely many points and preserves connectedness. The identity
theorem forces \(c\) to vanish on all such components. Since \(F\not\equiv0\),
some component of \(U_\rho\) must therefore lie in \(\mathbb D\).
If \(R\) is a nonconstant finite Blaschke product times a unimodular
constant, \(A\) is the adjoint of multiplication by the inner function
\(R^\sharp(z)=\overline{R(\bar z)}\). Its powers tend strongly to zero:
the ranges of the isometric multiplier powers have zero intersection by
increasing order of vanishing at any zero. Thus the only invariant
probability is \(\delta_0\).

Conversely a proper restriction to \(\Omega\) is onto, by openness and
closedness of its image in \(\mathbb D\), and has finite degree
\(e\ge1\), since its fibers are compact and discrete. For \(0\le r<e\) form the fields
\[
 H_r(\lambda)=\sum_{w\in\Omega,\ R(w)=\lambda}w^r K(w),
 \qquad K(w)=(w^j)_{j\ge0},
\]
counting multiplicity. These holomorphic fields on the disk obey
\(AH_r(\lambda)=\lambda H_r(\lambda)\). In fact they are symmetric
sums of local branches at regular values and locally bounded at critical
values by properness, so the latter singularities are removable.
Put \(H'_0=H_0/e\) and \(H'_r=H_r-(H_r)_0 H'_0\) for \(r>0\), where the
subscript \(0\) outside parentheses denotes the first coordinate.
Their first coordinates are now one and zero, respectively.

We record why a single polynomial makes all the boundary fields \(C^2\).
The circle contacts \(S=\{u\in\mathbb T:|R(u)|=1\}\) are finite:
otherwise \(|R|=1\) on the whole circle by real analyticity, and division by
the Blaschke factors of the disk zeros and the maximum principle applied
to the quotient and its inverse show that \(R\) is an excluded inner
function. On the circle discard the finite set \(D_0\) of contact images and
the critical values from \(\{|w|\le\rho\}\) that belong to the circle. On any
remaining arc all roots from \(\Omega\) continue analytically and stay
strictly inside the disk.
For such a root \(w\) near a discarded point \(s\), both
\((1-|w|)^{-1}\) and \(|R'(w)|^{-1}\) grow at most as fixed powers of
\(|\lambda-s|^{-1}\). For the first bound use
\(|1-|R(u)|^2|\ge c\operatorname{dist}(u,S)^k\) on the circle when
\(S\ne\varnothing\), by the finite orders of its zeros. With
\(u=w/|w|\) and \(|R(w)|=1\), the left side is \(O(1-|w|)\);
Lipschitz continuity of \(R\) near the closed disk then gives
\(\operatorname{dist}(\lambda,R(S))\le C(1-|w|)^{1/k}\) near
boundary contact. The second bound follows in the same way from the
finite orders of the zeros of \(R'\) near \(\overline{\mathbb D}\)
and \(|R(w)-R(w_0)|\le C|w-w_0|\) near such a zero \(w_0\).
Away from the respective zero sets the quantities are bounded.
Implicit differentiation of the branches through order two and the
estimates on \(K,K',K''\) in Theorem~\ref{r80:tangent} now give
polynomial bounds on the fields and their first two angular derivatives.
Choose \(p(\lambda)=\prod_{s\in D_0}(\lambda-s)^N\) with \(N\) large.
Then \(p H'_r\) extend by zero across the discarded points to \(C^2\)
fields on the circle. Every scalar coordinate before multiplication by
\(p\) is bounded and holomorphic in the disk, since the \(e\) roots in
the sums have modulus less than one. Its boundary coefficients at
negative frequencies vanish.

The Fourier coefficients of the resulting fields are norm-summable.
Their union spans densely: annihilating all fields on a regular arc
annihilates the individual \(K(w)\) there by the invertible Vandermonde
matrix on the \(e\) roots. A functional annihilating those kernels is zero
by analytic uniqueness. Each field yields a chain terminating at zero;
the first-coordinate functional kills all other chains and all but
finitely many elements of the first one, where it reads the polynomial
\(p\). Thus \cite[Theorem~2.6 and opening of Section~5]{r80:GrUniversal}
gives \(2\)-universality. Conversely ordinary universality realizes the
fair one-sided Bernoulli shift, an ergodic, noninvertible system.
\end{proof}

To answer \cite[Question~5.6]{r80:GrUniversal}, take
\(\phi(z)=\tfrac32(z-1/2)^2\). The squared moduli of the two preimages of
\(|\lambda|\le1\) average to \(1/4+2|\lambda|/3\le11/12\), so
\(\overline{\mathbb D}\subset\phi(\mathbb D)\). But
\(\phi^{-1}(\mathbb D)=D(1/2,\sqrt{2/3})\) meets the exterior of
the closed unit disk. The intersection of this preimage with \(\mathbb D\) is connected and has
boundary points on \(\mathbb T\) where \(|\phi|<1\), so its restriction
is not proper. Since \(M_\phi^*=\phi(B)\), every invariant
probability makes this operator invertible modulo null sets. Both
sufficiency assertions in Question~5.6 therefore fail, including the
one with no moment hypothesis. Example~4.2 there treats invertible
source systems.

\subsection{Spectral subspaces that generate an ergodic system}
Let \(U\) be the Koopman unitary of an invertible ergodic transformation on
a standard nonatomic probability space. Write \(P_V=\mathbf1_V(U)\) for
the spectral projection associated to a Borel set \(V\subset\mathbb T\).

\begin{lemma}[Generation by an open spectral set]
For each nonempty open \(V\subset\mathbb T\), the functions in
\(P_VL^2\) generate the entire completed \(\sigma\)-field.\label{r80:spectralgenerator}
\end{lemma}
\begin{proof}
The spectrum of \(U\) is \(\mathbb T\). Indeed, for each \(n\) the normalized
indicator of a Rokhlin tower base has \(n\) orthogonal successive translates.
The normalized sum of these translates with coefficients
\(1,\bar\lambda,\ldots,\bar\lambda^{n-1}\) is a unit approximate eigenvector
at each \(\lambda\in\mathbb T\), with error at most \(2/\sqrt n\).
Consequently \(P_VL^2\) is infinite dimensional: otherwise a smaller open arc
in \(V\) avoiding its finitely many eigenvalues would have zero spectral
projection. A countable dense subset of \(P_VL^2\) generates the same
completed \(\sigma\)-field, by almost everywhere subsequence convergence.
Thus \(\mathcal G=\sigma(P_VL^2)\) is a standard probability factor, invariant
in both time directions. Its ergodic action is nonatomic, since an ergodic
probability space with an atom consists of a finite orbit up to null sets.

If \(D=L^2\ominus L^2(\mathcal G)\ne0\), take \(0\ne h\in D\). For each
\(n\) choose a height-\(n\) tower in the factor. Some translate \(C\) of
its base has \(\|\mathbf1_C h\|_2>0\), since the orbit of the base covers
the space up to a null set. Its first \(n\) translates are again disjoint.
Multiplication by \(\mathbf1_C\) preserves \(D\), by conditional expectation.
Applying the preceding approximate-eigenvector argument to
\(\mathbf1_C h/\|\mathbf1_C h\|_2\) gives \(\sigma(U|_D)=\mathbb T\).
But \(D\) reduces \(U\) and is orthogonal to \(P_VL^2\), so its spectral
projection on the open set \(V\) vanishes. This contradiction gives
\(L^2(\mathcal G)=L^2\).
\end{proof}

\begin{theorem}[Prescribing the unimodular point spectrum]\label{r80:outsidechain}
For every nonempty proper open \(V\subset\mathbb T\), an invertible
operator on a separable complex Hilbert space is \(2\)-universal for
invertible ergodic systems and has unimodular point spectrum exactly \(V\).
For all such universally representing operators on nonzero separable complex
Hilbert spaces the possible normalized Haar measures of the unimodular
point spectrum form exactly \((0,1]\).
In particular, a proper open arc of length less than \(2\pi\) gives a
negative answer to Grivaux's Question~5.3 and, by
Theorem~\ref{r80:holes}, an affirmative answer to Question~5.10
\cite{r80:GrUniversal}.
\end{theorem}
\begin{proof}
Set \(E=\mathbb T\setminus V\).
Put \(E'=\{\bar\eta:\eta\in E\}\) and take the bounded domain
\[
 \Omega=\{z:1/2<|z|<2,\ \max_{w\in E'}\operatorname{Re}(\bar wz)<1\}.
\]
It is connected, as every point can be joined radially to the circle
\(|z|=3/4\). Let \(H\) be the closure of the Laurent polynomials in
\(L^2(\Omega,d{\rm area})\). Its elements have holomorphic representatives
by the disk mean-value estimate. Multiplication by \(z\) and \(z^{-1}\)
preserve \(H\) and are bounded inverses there. Set \(A_0=M_z^*\).
For \(w\in\Omega\) let \(k_w\) represent point evaluation, with the inner
product linear in its first variable. Then \(A_0k_w=\bar w k_w\).
The field \(\lambda\mapsto k_{\bar\lambda}\) is norm-holomorphic on
\(\{\lambda:\bar\lambda\in\Omega\}\): indeed \(\langle k_{\bar\lambda},g\rangle=\overline{g(\bar\lambda)}\)
for \(g\in H\), and the
disk mean-value estimate bounds the norms locally.
If \(A_0v=\eta v\) with \(\eta\in E\), then
\(\langle p,v\rangle=p(\bar\eta)\langle1,v\rangle\) for every Laurent
polynomial \(p\), using the inverse of \(M_z\) for negative powers.
At \(w=\bar\eta\) the entire functions \(e^{m(\bar wz-1)}\) equal \(1\),
while their \(L^2(\Omega)\) norms tend to zero by dominated convergence.
Their Taylor polynomials converge uniformly on \(\overline\Omega\cup\{w\}\).
Thus evaluation at \(w\) is unbounded on polynomials, so
\(\langle1,v\rangle=0\) and then \(v=0\). This proves the asserted point
spectrum on the circle.

Choose nonnegative smooth functions \(u_j\) compactly supported in
\(\mathbb T\setminus E\) whose positivity sets cover this complement.
Choose \(c_j>0\) so that the sums of the \(C^2\)-norms of \(c_ju_j\)
and of \(c_ju_j(\lambda)k_{\bar\lambda}\), each extended by zero, are
finite. Write \(s=\sum c_j u_j\) and \(F=\sum c_j u_j k_{\bar\lambda}\).
Thus \(s\ge0\) has zero set \(E\), \(A_0F(\lambda)=\lambda F(\lambda)\),
and \(\langle F(\lambda),1\rangle=s(\lambda)\). The coefficients
\(v_n=\widehat F(n)\), \(n\in\mathbb Z\), satisfy
\(\sum_n\|v_n\|<\infty\) and \(A_0v_n=v_{n-1}\).
Their closed span contains \(F(\mathbb T)\) by Fourier inversion and equals
\(H\), by the identity theorem applied to any function orthogonal to all
\(k_{\bar\lambda}\) on an open arc disjoint from \(E\).

On \(Z=H\otimes\ell^2\) take \(A=A_0\otimes I\), with the same unimodular
point spectrum (use the coordinates of any eigenvector).
For an invertible ergodic nonatomic source \((X,\mu,T)\), let
\(U f=f\circ T\). Given \(f=(f_i)\in L^2(\mu;\ell^2)\), define
\[
   \Phi_f(x)=\sum_{n\in\mathbb Z}v_n\otimes f(T^nx).
\]
The series converges absolutely almost everywhere and in \(L^2(\mu;Z)\),
with norm at most \(\|f\|_2\sum_n\|v_n\|\) in the latter space. Its
recurrence yields \(\Phi_f\circ T=A\Phi_f\).

For a residual set of \(f\), the law of \(\Phi_f\) has full support. For
each open ball \(B\subset Z\), the condition \(\mu(\Phi_f\in B)>0\) is
open by continuity in probability. For density start with bounded \(f\),
which approximate every input in \(L^2\), and take a finite sum
\(y=\sum_{|n|\le m}v_n\otimes a_n\in B\). Such sums are dense by the
density of the coefficients \(v_n\) in linear span. Pad with \(a_n=0\)
up to \(|n|\le N\), taking \(N\) large enough that
\[
  \max\{\|f\|_\infty,\max_n\|a_n\|\}
  \sum_{|n|>N}\|v_n\| < \operatorname{dist}(y,Z\setminus B).
\]
On the disjoint levels \(T^nD\), \(|n|\le N\), of a Rokhlin tower set
\(f=a_n\), leaving it unchanged elsewhere. Its code belongs to \(B\) on
\(D\) up to null sets. Shrinking the base to arbitrarily small positive
measure makes the \(L^2\) perturbation as small as desired. Intersect
over a countable base of balls.

Also residual is the condition that \(\{f_i:i\ge1\}\) has dense linear
span in scalar \(L^2(\mu)\). For each member of a countable dense set and
each positive rational tolerance, approximation by a finite linear
combination is an open condition. It is dense by truncating the
coordinate sequence and assigning an arbitrarily small nonzero multiple
of the target to an unused coordinate. Choose \(f\) satisfying both
residual conditions and let \(\mathcal F=\sigma(\Phi_f)\) modulo null sets.
In channel \(i\) the scalar functional given by \(1\in H\) reads
\(s(U)f_i\). Its translates at all integer times are measurable in
\(\mathcal F\), since \(A\) is invertible. Uniform Laurent polynomial
approximation to \((s+\varepsilon)^{-1}\) and then dominated convergence
in the spectral calculus give \(P_Vf_i\in L^2(\mathcal F)\).
Their span is dense in \(P_VL^2(\mu)\). Lemma~\ref{r80:spectralgenerator}
therefore gives the entire source \(\sigma\)-field. On standard Borel
spaces this means that \(\Phi_f\) is an isomorphism modulo null sets onto
its law, which is invariant, has full support, and has finite second moment.

To see why strictly positive Haar measure is necessary, let
\(S=\mathbb T\cap\sigma_p(A)\) for any such universal model.
It is Haar measurable as the projection of the closed set of pairs
\((\zeta,h)\) with \(\|h\|=1\) and \(Ah=\zeta h\).
For each \(\lambda\in\mathbb T\) of infinite order, realize the
ergodic rotation \(z\mapsto\lambda z\) by a Hilbert-valued
\(L^2\) function \(G\). Full support makes \(G\) nonconstant.
Fourier expansion of \(G(\lambda z)=A G(z)\) gives
\(A\widehat G(n)=\lambda^n\widehat G(n)\), with
\(\widehat G(n)\ne0\) for some \(n\ne0\).
Thus the inverse images of \(S\) under the countably many nonzero
power maps cover every infinite-order point. Each map preserves Haar
measure, so \(S\) has positive measure. The construction above realizes
each value in \((0,1)\) by an arc, and realizes one by the circle
with a single point removed.
\end{proof}

For a concrete choice, let \(B\) be the unilateral backward shift and set
\(A_0=I+\tfrac12 B\), which is invertible by the Neumann series.
Its eigenvectors on \(V=\{\lambda\in\mathbb T:|\lambda-1|<1/2\}\)
are \(K(2(\lambda-1))\), with \(K\) as in Theorem~\ref{r80:tangent}.
Their span is dense by the identity theorem and their first coordinate is
one. The same smooth cutoffs and encoding therefore give the explicit
model \((I+\tfrac12 B)\otimes I_{\ell^2}\).

\subsection{Norm moments in linear models of discrete-spectrum systems}
Flytzanis proved that, when a total family of adjoint orbits is bounded,
an ergodic invariant probability with spanning support has discrete
spectrum. He asked at the end of \cite[Section~3]{r81:Fly79} whether
the norm in this class of models must be integrable. His construction of
a nonintegrable model for every nonatomic ergodic source did not impose
this condition on the adjoint.

\begin{theorem}[Linear models with no positive norm moment]\label{r81:flymom}
Every standard nonatomic ergodic system with discrete spectrum is isomorphic
to the action of a bounded invertible operator on a separable complex Hilbert
space with an invariant probability whose support spans the space and whose
norm moments of every positive order are infinite. A total family of
functionals has relatively compact orbits under the adjoint.
\end{theorem}
\begin{proof}
Let \(U f=f\circ T\). The unit-modulus eigenfunctions of \(U\) admit a
countable generating family \((g_j)\), since their linear span is dense
in \(L^2\). Their eigenvalue group is infinite and hence dense in
\(\mathbb T\): ergodic eigenspaces are one-dimensional, whereas \(L^2\)
is infinite dimensional. We may normalize any eigenfunction with
eigenvalue of finite order \(q\) so that its distribution is Haar on the
\(q\)-th roots of unity. In the infinite-order case its distribution is
Haar on the circle, by invariance under rotation through the eigenvalue.

Put \(m_n=2^{8n}\), \(S_n=\exp(\sqrt{m_n})\) and
\(p_n(z)=((1+z)/2)^{m_n}\). Define an invertible matrix \(V_n\) by
requiring that it send \((1,z,\ldots,z^{m_n})\) to
\[
 (S_np_n(z),\epsilon_nz,\ldots,\epsilon_n z^{m_n}),
 \qquad \epsilon_n=2^{-n}/\sqrt{m_n}.
\]
Invertibility follows from \(p_n(0)\ne0\). Choose a normalized eigenfunction
\(h_n\) with eigenvalue \(\lambda_n\) of order greater than \(m_n^2\) or
infinite order, sufficiently close to one that
\[
 T_n=V_n\operatorname{diag}(1,\lambda_n,\ldots,\lambda_n^{m_n})V_n^{-1},
 \qquad \|T_n-I\|<2^{-n}.
\]
This is possible by density and avoidance of finitely many roots of unity.

Concatenate the displayed vectors evaluated at \(z=h_n(x)\) for all \(n\),
and append \((2^{-j}g_j(x))_j\), to obtain \(F(x)\) in the Hilbert direct
sum. Indeed, \(\mathbb P(|\arg h_n|<2^{-n})=O(2^{-n})\) for both Haar
laws just described. Off this event, the inequality
\(|\cos(t/2)|\le e^{-c t^2}\) for \(|t|\le\pi\) gives
\[
 S_n |p_n(h_n)|\le
 \exp(\sqrt{m_n}-c m_n 2^{-2n})\le2^{-n}
\]
for all sufficiently large \(n\). The other entries of the block have
squared norm \(4^{-n}\). Borel--Cantelli gives \(\|F\|<\infty\) almost surely.
On the other hand \(|\arg h_n|\le1/\sqrt{m_n}\) has probability at least
\(c'/\sqrt{m_n}\), for a uniform \(c'>0\), and on that event
\(|p_n(h_n)|\) is bounded below by a fixed positive constant. Thus for each
\(p>0\), \(\mathbb E\|F\|^p\ge c_p S_n^p/\sqrt{m_n}\) for every \(n\).

The direct sum of the \(T_n\) and the diagonal phases of the \(g_j\) defines
a bounded invertible operator \(A\), with \(F\circ T=AF\). The appended
coordinates generate the source, giving an isomorphism modulo null sets.
Restrict \(A\) to the closed linear span of the support of the law of \(F\);
this span is invariant under \(A\) and \(A^{-1}\) by invariance of the
law and continuity. Finitely supported coordinate functionals are total
after this restriction. Their adjoint orbits have compact closure because
each finite block is similar to a diagonal family of phases.
\end{proof}

The necessary adjoint conditions of Flytzanis' Lemma~1(ii) do not by
themselves give a spanning invariant law, answering the sufficiency question
following that lemma in \cite{r81:Fly79}.
\begin{proposition}\label{r81:flyadjoint}
A bounded invertible Hilbert-space operator can have a total family of
unimodular adjoint eigenfunctionals, with countable unimodular adjoint
point spectrum generating the eigenvalue group of a nonatomic ergodic
rotation, although every invariant probability is supported on a single
proper closed hyperplane.
\end{proposition}
\begin{proof}
On \(\mathbb C\oplus\ell^2(\mathbb N)\) put \(z_n=e^{i/n}\) and
\[
 A(a,b)=\big(a,(z_nb_n+(z_n-1)a)_n\big).
\]
Both \(A\) and its inverse are bounded as \((z_n-1)_n\in\ell^2\).
The functionals \(a\) and \(a+b_n\) have respective multipliers
\(1,z_n\) and form a total family. Testing an arbitrary eigenfunctional
first on \(0\oplus\ell^2\), then on \((1,0)\), shows that these are all
possible multipliers. Their group is \(\{e^{iq}:q\in\mathbb Q\}\).
On its compact dual, translation by the character given by inclusion
into \(\mathbb T\) is an ergodic rotation with exactly this eigenvalue
group, since the compact-group characters span \(L^2\).
For \(a\ne0\),
\[
 \|A^k(a,b)\|\ge |a|\Big(\sum_n|z_n^k-1|^2\Big)^{1/2}-\|b\|_2
 \longrightarrow\infty :
\]
the terms with \(k\le n\le2k\) are bounded below by a positive constant.
Poincar\'e recurrence confines every invariant probability to \(a=0\).
\end{proof}

\subsection{Full support from the coordinate marginal of a Hilbert-space shift}
Let \(K\) be a separable complex Hilbert space and \((A_n)_{n\in\mathbb Z}\)
bounded positive operators, each bounded below by a positive multiple of
the identity. The multiples may depend on \(n\). Write \(H_A(K)\) for
the space of sequences with finite squared norm
\(\sum_n\langle A_nx_n,x_n\rangle\), and suppose the bilateral left
shift \(T\) is boundedly invertible. For a closed subspace \(L\subset K\),
write \(H_A(L)\) for the closed subspace of sequences with entries in \(L\).

\begin{theorem}[Full support from the one-coordinate marginal]
\label{r82:fullsupport}
Let \(\mu\) be any \(T\)-invariant probability on \(H_A(K)\), and let
\(L\) be the closed complex linear span of the support of its
time-zero marginal. Then \(H_A(L)\) carries a full-support ergodic
invariant probability. In particular the existence of an invariant law
whose one-coordinate marginal spans \(K\) is equivalent to the existence
of a full-support ergodic law on \(H_A(K)\).
\end{theorem}
\begin{proof}
Draw independent paths \(X^j\) of law \(\mu\). Their coordinate marginals
are identical at all integer times. By independence and separability,
almost surely \(\{X^j_n:j\ge1\}\) spans \(L\) densely for every \(n\):
each open set from a countable base with positive marginal mass is hit
by the samples. Choose \(M_j\ge1\) with
\(\mu(\|X\|>M_j)<2^{-j}\) and put \(a_j=2^{-j}/M_j\).
Borel--Cantelli gives \(\sum_j a_j^2\|X^j\|^2<\infty\) almost surely.
We may take an ergodic component of the diagonal path shift on which both
properties hold almost surely, by ergodic decomposition of this standard
Borel system. Independence among paths was needed only before this step.

On this component take an independent Bernoulli shift of rows of
independent standard complex Gaussian scalars \(g_{nj}\), and set
\[
             Y_n=\sum_{j\ge1}a_jg_{nj} X^j_n .
\]
Given the paths, the sum of the expected squared norms of the contributions
over \((n,j)\) is \(\sum_j a_j^2\|X^j\|^2\). Thus the centered independent
Hilbert-space series converges in conditional \(L^2\) and almost surely.
The vectors multiplying its coefficients span \(H_A(L)\) densely: in each
coordinate its norm is equivalent to the norm of \(K\), and finite
coordinate vectors are dense. The conditional Gaussian law therefore has
full support. Indeed, any finite linear combination can be approximated
inside any prescribed ball with positive probability by finitely many
Gaussian coefficients, independently of a tail small with positive
probability by its vanishing second moment. The formula for \(Y\)
intertwines the path and row shifts with \(T\). The Bernoulli shift is
mixing, so its product with the chosen ergodic component is ergodic.
This gives the claimed law. The converse follows by projection onto the
zeroth coordinate, a continuous surjection.
\end{proof}

Put
\[
 q(v)=\sum_{n\in\mathbb Z}\langle A_nv,v\rangle,\qquad
 D=\{v\in K:q(v)<\infty\},\qquad K_1=\overline D^{\,K}.
\]
Thus \(v\in D\) precisely when the constant path with value \(v\)
belongs to \(H_A(K)\).

\begin{theorem}[Support and the first moment]\label{r82:supports}
Every invariant probability of finite first moment on \(H_A(K)\) is
supported on \(H_A(K_1)\). This entire subspace is the support of an
invariant Gaussian law with independent identically distributed
coordinates and finite second moment. If \(K\) is finite-dimensional,
the support inclusion holds for every invariant probability.

The last assertion also holds if, in a common direct-integral
representation of \(K\), all \(A_n\) act by multiplication by
measurable matrix fields on almost everywhere finite-dimensional
fibres. In particular it holds for commuting weights. In these cases
the shift is frequently hypercyclic precisely when \(D\) is dense in
\(K\), equivalently when there is a full-support invariant Gaussian law
with independent identically distributed coordinates.
\end{theorem}
\begin{proof}
Write \(X=(X_n)\) for a stationary path and \(\rho\) for its common
coordinate marginal. For any bounded Borel scalar function \(f\), the
path \((f(X_n)X_n)_n\) has norm at most \(\|f\|_\infty\|X\|\).
When \(\mathbb E\|X\|<\infty\), its Bochner mean is the constant path
with value \(v_f=\int f(v)v\,d\rho(v)\), so \(v_f\in D\). If \(u\)
is orthogonal to every \(v_f\), then \(\langle v,u\rangle=0\) for
\(\rho\)-almost every \(v\), by testing against bounded scalar
functions. Thus the closed span of the \(v_f\) contains
\(\operatorname{supp}\rho\). All coordinates belong to \(K_1\) almost
surely. For finite second moment Tonelli gives the sharper statement
\(\mathbb E\|X\|^2=\int q(v)\,d\rho(v)\), hence \(\rho(D)=1\).

Conversely choose \(v_j\in D\) with dense span in \(K_1\), and
\(a_j>0\) with \(\sum_j a_j^2(\|v_j\|^2+q(v_j))<\infty\).
With independent standard complex Gaussians \(g_{nj}\) take
\(X_n=\sum_j a_jg_{nj}v_j\). The coordinate rows are iid, and the
series over \((n,j)\) converges in mean square in \(H_A(K)\), since
its squared mean norm is \(\sum_j a_j^2q(v_j)\). The coefficient
vectors span \(H_A(K_1)\) densely, giving the stated support. The
shift of iid rows is mixing.

For the finite-dimensional assertion let \(L\) be the span of
\(\operatorname{supp}\rho\), and suppose \(L\ne0\).
Choose \(R\) large enough that the restriction of \(\rho\) to the ball
of radius \(R\) still spans \(L\). This is possible by choosing finitely
many independent support points in its interior. The truncated
second-moment matrix \(C_R=\int_{\|v\|\le R}v\otimes v\,d\rho(v)\)
then satisfies \(C_R\ge c I_L\) on \(L\) for some \(c>0\).
For \(B_n=P_LA_n|_L\) put
\[
 W_N=\sum_{|n|\le N}\operatorname{tr} B_n,\qquad
 S_N=\sum_{|n|\le N}\langle A_nX_n,X_n\rangle
                         \mathbf1_{\{\|X_n\|\le R\}} .
\]
If \(W_N\to\infty\), then \(S_N/W_N\to0\) almost surely since
\(S_N\le\|X\|^2<\infty\), and \(0\le S_N/W_N\le R^2\).
This contradicts \(\mathbb E S_N=\sum_{|n|\le N}
\operatorname{tr}(B_n C_R)\ge c W_N\) by dominated convergence.
Consequently \(\sum_n\operatorname{tr}B_n<\infty\), and \(L\subset D\).

In the direct-integral case use the standard separable representation
\(K=\int^\oplus K_s\,dm(s)\) and write \(A_n(s)\) for the matrix
fields. On the canonical stationary path space take a jointly
measurable representative \(V(\omega,s)\) of \(X_0\); one obtains one
by approximation by simple \(K\)-valued maps with summable norm errors.
Then \(V(T^n\omega,s)\) represents \(X_n\) and is stationary in \(n\)
for every fixed \(s\). By Tonelli applied to the path energy and then
Fubini on its finiteness event, almost every \(s\) gives an almost
surely finite-energy stationary path in \(K_s\). The finite-dimensional
argument shows that \(q_s(V)=\sum_n\langle A_n(s)V,V\rangle<\infty\)
almost surely there. Interchanging the two variables again, almost
every \(v=X_0(\omega)\) satisfies \(q_s(v(s))<\infty\) almost everywhere.
The fields \(v(s)\mathbf1_{\{q_s(v(s))\le j\|v(s)\|^2\}}\) have
energy at most \(j\|v\|^2\) and tend to \(v\) in \(K\). Hence
\(v\in K_1\). Commuting positive operators on a separable Hilbert
space admit such a representation with scalar fibres by the cyclic
decomposition of their abelian von Neumann algebra.

Finally a frequently hypercyclic operator on a reflexive space admits
an invariant probability of full support by
\cite[Theorem~1.3]{r79:GM}. Conversely the iid law above is ergodic,
so full support implies frequent hypercyclicity by Birkhoff's theorem.
\end{proof}
Yin, Chen and Xiang characterized frequent hypercyclicity and chaos
for shifts with positive diagonal operator weights \cite{r82:Yin}.
The support assertion above also covers continuous joint spectra and
measurable families of noncommuting finite-dimensional matrices.

\begin{corollary}[Frequent hypercyclicity with empty point spectrum]
\label{r82:tape}
There is a bounded invertible operator on a separable complex Hilbert
space, with empty point spectrum, carrying a full-support ergodic law with
a non-atomic Bernoulli factor given by iid complex Gaussian observations and infinite entropy. The operator is
frequently hypercyclic and its only invariant probability of finite
first moment is \(\delta_0\).
\end{corollary}
\begin{proof}
Steenbeck constructed such a shift without eigenvalues and with a nontrivial
invariant law \cite[Theorem~1.3]{r81:Steen}, answering Question~8.3 of
\cite{r79:GMMbook}. In his Definition~2.1 and Lemmas~2.7 and~2.11,
\(K=H^1(\mathbb T;\mathbb C)\), and the paths have coordinates
\(X_n(\theta)=k_{\theta+n\alpha}\), where \(k_u\) is the reproducing
kernel at \(u\), \(\theta\) has Haar law and the rotation is irrational.
Their zeroth marginal spans \(K\): by continuity, an \(H^1\) function
vanishing almost everywhere on the circle vanishes everywhere.
Theorem~\ref{r82:fullsupport} gives the full-support ergodic law. In its
construction \(\langle X_n^j,1\rangle=1\) at every coordinate. Hence the
observations \(\langle Y_n,1\rangle=\sum_j a_jg_{nj}\) form an iid
nondegenerate complex Gaussian sequence. This Bernoulli factor has
infinite entropy, since equiprobable partitions of its marginal can
have arbitrarily many cells.

Birkhoff's theorem on a countable base yields frequent hypercyclicity.
This answers \cite[Question~3.19]{r79:GM} and
\cite[Question~8.2]{r79:GMMbook}. An element $v\ne0$ of the finite-weight space $D$ in
Theorem~\ref{r82:supports} would give the nonzero fixed vector
$(v)_{n\in\mathbb Z}$. Thus $D=\{0\}$, and that theorem forces every
invariant law of finite first moment to be $\delta_0$. This answers
the second-moment question following \cite[Question~5.6]{r80:GrUniversal}.
Flytzanis gave the Hilbert-space spectral obstruction at the second moment
\cite[p.~270]{r81:Flytzanis}.
\end{proof}

\begin{corollary}[Hereditary frequent hypercyclicity without eigenvalues]
\label{r82:hereditary}
Every frequently hypercyclic bilateral shift on \(H_A(K)\) is hereditarily
frequently hypercyclic, for both lower and upper density. In particular
the operator in Corollary~\ref{r82:tape} has these properties and no
nonzero periodic vector. This strengthens the existence assertion in
Bayart--Grivaux--Matheron--Menet \cite[Question~9.18]{r79:BGMM}.
\end{corollary}
\begin{proof}
Grivaux--Matheron's full-support theorem \cite[Theorem~1.3]{r79:GM} and
Theorem~\ref{r82:fullsupport} give a coding from an ergodic base
\((B,\eta,R)\) times an independent bilateral Bernoulli shift, with
full-support conditional law for almost every base point. Write \(S\)
for the product map. For a nonempty open set \(U\) let \(f\) be the
indicator of its inverse image under the coding and \(p=\mathbb E(f\mid B)\).
Thus \(p>0\) almost everywhere. For every fixed \(A\subset\mathbb N\),
\[
 \frac1N\sum_{n<N}\mathbf1_A(n)(f-p)\circ S^n\longrightarrow0
 \quad\text{almost surely}.
\]
Indeed, condition on the entire base and Bernoulli rows \([-m,m]\) to
obtain \(f_m\), so \(\mathbb E(f_m\mid B)=p\). Given the base, the
centered terms for \(f_m-p\) in each residue class modulo \(2m+1\) are
independent and bounded, hence obey the strong law even after
multiplication by \(\mathbf1_A(n)\). Birkhoff's theorem for the
ergodic product bounds the limsup of the omitted error by
\(\|f-f_m\|_1\), which tends to zero by martingale convergence.

If \(A\) has positive lower density \(d\), choose \(\epsilon>0\) with
\(\eta(p<\epsilon)<d/2\). The unweighted ergodic theorem for the
base then gives lower density at least \(\epsilon d/2\) for returns
to \(U\) inside \(A\). With positive upper density use a subsequence
realizing that density. Taking the intersection of the probability-one
events for any prescribed countable family of pairs \((A,U)\) proves
the two hereditary assertions. A nonzero periodic vector would span,
with its orbit, a nonzero finite-dimensional space carrying an eigenvector.
\end{proof}
The bare existence of a nonchaotic example in Question~9.18 also follows
from \cite[Theorem~3.1(a)]{r79:BGMM} and the classical restricted
Kalisch operator. Take as its parameter set the support of Lebesgue
measure restricted to a positive-measure compact set missing every root
of unity. Its indicator eigenvectors vary continuously on that support,
span after deletion of any Lebesgue null set, and have exactly the
prescribed unimodular point spectrum \cite{r82:BGprescribed}. Menet and Papathanasiou gave a frequently hypercyclic, nonchaotic
operator-weighted shift on a sum with $\ell^1$ coordinate space
\cite[Proposition~3.17]{r82:MP}. Corollary~\ref{r82:tape} has empty
point spectrum on a Hilbert space.

\subsection{Universal Hilbert shifts without unimodular eigenvalues}

Glasner and Weiss constructed a Hilbert-space representation that
models every free ergodic action of a fixed countable group
\cite{r83:GlasnerWeiss}. Steenbeck constructed an invertible Hilbert
operator with empty point spectrum carrying a nontrivial invariant law
\cite[Theorem~1.3]{r81:Steen}. We build fixed shifts for which every
source in the indicated class has a full-support model, while every
nontrivial invariant law has infinite norm moments of all positive
orders. Theorem~\ref{r81:flymom} constructed a model and one such law
for each discrete-spectrum source.

\begin{theorem}\label{r83:universal}
There is a bounded operator on a separable complex Hilbert space,
with no unimodular eigenvalues, universal for all ergodic endomorphisms
of standard nonatomic probability spaces. Its only invariant probability
having a finite norm moment of any positive order is \(\delta_0\). There is also an invertible
Hilbert-space operator with empty point spectrum with the same assertions
for invertible ergodic sources. In particular ordinary universality in
\cite[Questions~5.7 and~5.8]{r80:GrUniversal} need not imply
\(2\)-universality. Conversely, on any nonzero separable complex Banach
space a \(2\)-universal operator for invertible ergodic systems must
have a unimodular eigenvalue.
\end{theorem}

\begin{proof}
We record the necessary condition first. A full-support Bochner-integrable
model \(G\) of an irrational Haar rotation by \(\zeta\) on the circle
satisfies \(AG(w)=G(\zeta w)\). Its Fourier coefficients obey
\(A\widehat G(k)=\zeta^k\widehat G(k)\). If every coefficient for
\(k\ne0\) vanished, scalar Fourier uniqueness after applying a countable
separating family would make \(G\) constant almost everywhere,
contrary to full support. A second moment implies the required
integrability. This is the rotation argument of
Theorem~\ref{r80:outsidechain}; Grivaux proved the stronger perfectly
spanning conclusion for Hilbert spaces
\cite[Proposition~5.2]{r80:GrUniversal}.

Let \(J=\mathbb N_0\) for arbitrary sources, or \(J=\mathbb Z\) for
invertible ones. We construct one operator for each \(J\), independently
of the source. Put \(h_r=2^{r^2+4}\), \(r\ge1\), and
\(M=\prod_{r\ge1}\{0,\ldots,h_r\}\), with its compact product topology.
Write \(\mathcal P_r\) for the cylinders specified by the first \(r\)
digits, \(p_r=|\mathcal P_r|\). For \(u\in M\) let \(k_u\) have
coordinates \(2^{-r}\mathbf1_C(u)\), indexed by
\(r\ge1,C\in\mathcal P_r\), in an ordinary complex \(\ell^2\).
Take \(K=\overline{\operatorname{span}}\{k_u:u\in M\}\).
The map \(u\mapsto k_u\) is continuous and bounded. Projection of
\(2^r e_{r,C}\) onto \(K\) gives a vector \(v_C\) of norm at most
\(2^r\) with \(\langle k_u,v_C\rangle=\mathbf1_C(u)\). These vectors
span densely, and \(v_C\) is the sum of its descendants at each
later level, by testing against every \(k_u\). The sum of the level-one
vectors gives a functional \(\ell\) with \(\ell(k_u)=1\).

Put \(w_r=p_r^{2r}\). Choose integer centres \(t_r\) so that the intervals
\([t_r-h_r,t_r]\) are positive and disjoint and
\(2^{-(t_r-h_r)}4^r p_r w_r\le2^{-r}\).
A cylinder \(C\in\mathcal P_r\) with last digit \(j\) gets slot
\(m(C)=t_r-j\). Define the Hilbert path space \(H\) by
\[
 \|x\|^2=\sum_{n\in J}(1+|n|)^{-2}\|x_n\|_K^2
 +\sum_{n\in J}\sum_{r\ge1}w_r\sum_{C\in\mathcal P_r}
        2^{-|n-m(C)|}|\langle x_n,v_C\rangle|^2.
\]
At each fixed \(n\) the quadratic form is equivalent to the squared
norm on \(K\): its upper bound follows from the condition on \(t_r\)
and \(2^{-|n-m|}\le2^{|n|}2^{-m}\). The shift \(T(x)_n=x_{n+1}\)
is bounded by the adjacent weight ratios, and is boundedly invertible
in the bilateral case. The constant path with a nonzero value \(v\)
has infinite energy. Indeed choose \(C_0\) with
\(\langle v,v_{C_0}\rangle\ne0\). At each later level, Cauchy--Schwarz
over its descendants gives
\(p_r\sum_{C\in\mathcal P_r}|\langle v,v_C\rangle|^2
\ge |\langle v,v_{C_0}\rangle|^2\). The diagonal terms \(n=m(C)\)
give divergence since \(w_r\ge p_r\). Thus no formal path \((\lambda^nv)\) with
\(|\lambda|=1\) belongs to \(H\). In the bilateral case the first
sum excludes the remaining nonzero eigenvalues, and invertibility
excludes zero.

Let \(X\) have any invariant law different from \(\delta_0\).
By stationarity its zeroth marginal is nontrivial. Since the cylinder
vectors form a countable total family, for some \(C_0\) and \(a>0\)
the event \(|\langle X_0,v_{C_0}\rangle|>a\) has probability
\(\eta>0\). For every later level \(r\), the identity expressing
\(v_{C_0}\) as the sum of at most \(p_r\) descendants and the union
bound give a descendant \(C\) with
\(\mathbb P(|\langle X_0,v_C\rangle|>a/p_r)\ge\eta/p_r\).
The same lower bound holds at coordinate \(m(C)\). The corresponding
diagonal term in the squared norm is
\(w_r|\langle X_{m(C)},v_C\rangle|^2\), so for every \(s>0\),
\[
 \mathbb E\|X\|^s\ \ge\ (\eta/p_r)(a\sqrt{w_r}/p_r)^s
                 =\eta a^s p_r^{s(r-1)-1}\longrightarrow\infty .
\]

Now fix a source \((\Omega,\mu,S)\). For each \(r\) choose a Rokhlin
tower of height \(h_r\) with disjoint levels \(S^{-j}B_r\),
\(0\le j<h_r\), and remainder \(G_r\) of measure at most \(1/h_r\).
This preimage form is valid for aperiodic endomorphisms
\cite[Theorem~1.2]{r83:AvilaCandela}; ergodicity and nonatomicity
imply aperiodicity. Let \(u_r(\omega)=j\) on the corresponding level
and \(u_r=h_r\) on \(G_r\). The measurable map \(u=(u_r)_r\) gives
an equivariant path \(X_n(\omega)=k_{u(S^n\omega)}\).

To verify its norm, write \(n=t_r-j+b\) in the terms with last
digit \(j\), ignoring terms with \(n\notin J\). For fixed \(r,j,b\),
at most one cylinder contributes. For \(j<h_r\), a contribution
implies \(S^{t_r+b}\omega\in B_r\); for \(j=h_r\) it implies
\(S^{t_r-h_r+b}\omega\in G_r\). In the unilateral case a relevant
exponent is nonnegative whenever the corresponding term occurs.
Take \(L_r=r+\lceil\log_2((h_r+1)w_r)\rceil\), and enlarge \(t_r\)
at construction so that \(t_r-h_r>L_r\). For \(|b|\le L_r\) the
probability of any contribution at level \(r\) is at most
\(2(2L_r+1)/h_r\). This is summable, since \(L_r=O(r^4)\).
For \(|b|>L_r\) the deterministic energy bound at level \(r\) is
\((h_r+1)w_r\sum_{|b|>L_r}2^{-|b|}\le 2^{1-r}\).
Borel--Cantelli and boundedness of \(k_u\) give \(X\in H\) almost
surely. Overwriting any finite prefix of \(u(\omega)\) by fixed digits
leaves the current digit at every later level intact. The preceding
estimates still apply, and the finitely many initial levels have
uniformly bounded energy. Enumerate the resulting equivariant
paths as \(X^j\), including the unmodified one. At each \(\omega\)
and each coordinate \(n\), their values span \(K\) densely by
continuity of \(k\) and density of all finite-prefix replacements in \(M\).

Consider the closed subspace
\(\mathcal L=\{Y\in L^0(\mu;H):Y\circ S=TY\}\) of the complete
metrizable space of measurable maps with convergence in probability.
We show that, for each nonempty open ball \(U\subset H\), the open
condition \(\mu(Y\in U)>0\) is dense in \(\mathcal L\).
Take \(Y\in\mathcal L\), a finitely supported \(z\in U\), and
\(\delta>0\) with \(B(z,3\delta)\subset U\). Put
\(I_d=J\cap[-d,d]\). By coordinate spanning, there are \(d\), finitely
many indices \(j\), scalar constants \(c_{ij}\), \(i\in I_d\), and
a set \(D\) of positive measure where
\[
 \big\|\mathbf1_{I_d}(Y-z)+
       (\textstyle\sum_j c_{ij}X^j_i)_{i\in I_d}\big\|<\delta,
 \qquad \|\mathbf1_{J\setminus I_d}(Y-z)\|<\delta .
\]
Indeed these countably many alternatives with rational complex
coefficients cover a conull set. After shrinking \(D\) to a set of
positive measure and choosing \(M_0>d\), we may suppose that
\(\sum_j(1+\max_i |c_{ij}|)\|\mathbf1_{J\setminus I_{M_0}}X^j\|
<\delta\) on \(D\).

Choose a marker \(P\) of arbitrarily small positive measure such
that \(P,S^{-1}P,\ldots,S^{-2M_0}P\) are disjoint modulo null sets
and \(D\cap S^{-M_0}P\) has positive measure. To see this, cover a
conull set by countably many sets \(C\) disjoint from each of their
first \(2M_0\) preimages: use a countable separating algebra to choose
\(C=C'\setminus\bigcup_{p=1}^{2M_0}S^{-p}C'\) around each aperiodic
point. One such \(C\) has positive mass for the atomless measure
\((S^{M_0})_*(\mathbf1_D\mu)\ll\mu\), and subdividing it gives \(P\).
Set
\[
 H_j=\sum_{i\in I_d}c_{ij}\mathbf1_{S^{-(M_0-i)}P},\qquad
 Z_n=\sum_j H_j(S^n\omega)X^j_n(\omega).
\]
The multiplier levels are disjoint, so \(Z\in\mathcal L\) by bounded
coordinatewise multiplication. On \(D\cap S^{-M_0}P\) the prescribed
coefficients occur at exactly the requested times throughout
\(I_{M_0}\): at coordinate \(n\), a term indexed by \(i\) can occur only when
both \(S^{M_0}\omega\) and \(S^{M_0+n-i}\omega\) belong to \(P\).
For \(n\ne i\) these visits differ by \(|n-i|\le2M_0\), excluded by
the marker condition. The unilateral indices have \(M_0+n-i\ge0\);
negative times in the bilateral case use invertibility. The tail estimate on \(D\)
gives \(Y+Z\in U\) there. As \(\mu(P)\to0\), \(Z\to0\) in probability.
Indeed the active events at any fixed finite collection of coordinates
have probability \(O(\mu(P))\), while outside it the norm is bounded by
a fixed finite sum of tails of the \(X^j\). This proves density.

Lastly the condition that \(\ell(Y_0)\) generate the source
\(\sigma\)-field modulo null sets is a dense \(G_\delta\).
For each member \(E\) of a countable generating algebra and each
\(k\ge1\), impose
\(\int|\mathbf1_E-f(\ell(Y_0))|\,d\mu<1/k\) for some bounded
continuous \(f\) on \(\mathbb C\). These are open conditions and
their intersection expresses generation, by density of bounded
continuous functions in \(L^1\) of any Borel probability on
\(\mathbb C\). For density, approximate the scalar observable
\(\ell(Y_0)\) within \(\varepsilon\) uniformly by a Borel
injection of \(\Omega\) modulo null sets: partition its range into
squares of diameter \(\varepsilon\), and encode a fixed Borel
injection of \(\Omega\) into an interval inside each square.
If \(a(\omega)\) denotes this scalar correction, the equivariant
perturbation \((a(S^n\omega)X_n(\omega))_n\) has norm at most
\(\varepsilon\|X(\omega)\|\) and gives exactly the correction in
coordinate zero since \(\ell(X_0)=1\). By Baire's theorem some
\(Y\) generates and hits every basic ball with positive probability.
It is an isomorphism modulo null sets onto its full-support invariant
law, as the spaces are standard. This proves both assertions.
\end{proof}

\subsection{Hereditary frequent hypercyclicity with a rigid sequence of powers}
Badea and Grivaux constructed a frequently hypercyclic chaotic
Hilbert-space operator which fails to be topologically mixing
\cite{r82:BadeaRigid}. Bayart, Grivaux, Matheron and Menet asked
whether this can happen with hereditary frequent hypercyclicity
\cite[Questions~9.16 and~9.20]{r79:BGMM}.

Eisner and Grivaux also obtained uniform rigidity and chaos
\cite[Remark~4.8]{r82:EG}; a recent odometer example is
\cite[Example~5.20]{r82:odometer}. Here the question is how often a
single orbit visits a countable list of open sets when each visit must
fall in its prescribed set of times. The spectral estimate below makes
the asymptotic frequency explicit for each fixed time set.
\begin{proposition}[An arc bound for visits at prescribed times]\label{r82:arcbound}
Let $T$ act boundedly on a separable complex Hilbert space and suppose
$TF(z)=zF(z)$ for a bounded measurable field on the unit circle.
Let $\sigma$ be a probability with $\sigma(I)\le C |I|^d$ on every
arc for some $0<d<1$. Suppose the values of $F$ span densely after
deletion of any $\sigma$-null set. Then $T$ admits a full-support
ergodic Gaussian law $\gamma$. For each fixed $A\subset\mathbb N$
and open $U$, almost surely
\[
 \frac1N\sum_{0\le n<N}\mathbf 1_A(n)\mathbf 1_U(T^nx)
 =\gamma(U)\frac{|A\cap[0,N)|}{N}+o(1).
\]
The same event may be taken for any prescribed countable family of
such pairs. In particular $T$ is hereditarily frequently
hypercyclic for both lower and upper density.
\end{proposition}
\begin{proof}
There is a centered circular
Gaussian probability \(\gamma\) with covariance
\(Q=\int F(z)\otimes F(z)\,d\sigma(z)\): this is positive and trace
class since \(F\) is bounded. The spanning assumption gives full
support, and \(TQT^*=Q\) gives invariance.

If \(h\) is bounded, then
\[
 \sum_{|j|\le N}|\widehat{h\sigma}(j)|^2\le C_h N^{1-d}.
\]
Indeed the Fej\'er kernel of order \(2N+2\) has Fourier coefficients
at least \(1/2\) for \(|j|\le N\), and is bounded by
\(C N(1+N\operatorname{dist}(t,\mathbb Z))^{-2}\).
Integrate against \(h\sigma\) and its conjugate, using the arc bound
on successive dyadic annuli about each point. For the centered
Gaussian character \(f_u(x)=\exp(i\operatorname{Re}\langle x,u\rangle)
-\int\exp(i\operatorname{Re}\langle x,u\rangle)\,d\gamma(x)\),
the Gaussian characteristic function gives
\[
 |\langle f_u\circ T^j,f_u\rangle_{L^2(\gamma)}|
 \le C_u |\widehat{h_u\sigma}(j)|,\qquad
 h_u(z)=|\langle F(z),u\rangle|^2 .
\]
Indeed the covariance of the two real linear observations is one
half the real part of that Fourier coefficient, and the exponential
is Lipschitz on their bounded covariance interval. Cauchy--Schwarz
and the Fej\'er estimate give, for any fixed \(|b_n|\le1\),
\[
 \int\left|N^{-1}\sum_{n<N}b_n f_u(T^nx)\right|^2d\gamma(x)
 \le C_u N^{-d/2}.
\]
For \(N_j=\lfloor j^p\rfloor\), \(p>2/d\), Chebyshev's inequality
and Borel--Cantelli imply almost sure convergence to zero along
\(N_j\), and boundedness of \(f_u\) and \(N_{j+1}/N_j\to1\) fill
the gaps. With \(b_n=1\) this proves ergodicity, since the characters
span densely in \(L^2(\gamma)\) by uniqueness of characteristic
functions on finite-dimensional projections. For general \(b_n\),
Birkhoff applied to the absolute error of approximation by finite
linear combinations of characters extends the conclusion to every
fixed centered \(f\in L^1(\gamma)\). Taking \(b_n=\mathbf1_A(n)\)
and \(f=\mathbf1_U-\gamma(U)\) shows that the visit frequency inside
\(A\) equals \(\gamma(U)|A\cap[0,N)|/N+o(1)\) almost surely.
Countable intersection gives both hereditary assertions.
\end{proof}

\begin{theorem}\label{r82:rigidhfhc}
There is a chaotic operator \(T\) on a separable complex Hilbert space
which is hereditarily frequently hypercyclic for both lower and upper
density and satisfies \(\|T^{n_k}-I\|\to0\) for a sequence
\(n_k\to\infty\). In particular it is not topologically mixing.
\end{theorem}
\begin{proof}
Let \(m_k=100^k\) and let \(E\subset[1/4,1/2]\) consist of
numbers with binary expansion beginning with \(01\), whose digits at
positions \(m_k+1,\ldots,4m_k\) vanish for every \(k\).
Give the remaining digits independent fair values. Its law \(\sigma\)
has support \(E\) and satisfies the arc bound with \(d=1/5\).
Indeed among the first \(M\) digits at least \(M/5-O(1)\) are free:
the fraction is smallest at the end of a forced block, and at
\(M=4m_k\) at most \(2+3\sum_{j\le k}m_j\le 2+(300/99)m_k\)
digits have been fixed. A dyadic cell of length \(2^{-M}\) has mass
at most \(C2^{-M/5}\), giving the bound for all arcs.

For an absolutely continuous \(g\) with \(g'\in L^2[0,1]\), put
\[
 (K_g x)(t)=g(t)x(t)-\int_0^t g'(s)x(s)\,ds .
\]
This operator has norm at most \(\|g\|_\infty+\|g'\|_2\);
\(e_a=\mathbf1_{[a,1]}\) is its eigenvector of eigenvalue \(g(a)\).
Take \(T=K_{\exp(2\pi i\,\cdot)}\) restricted to the closed span
of \(e_a\) for \(a\in E\). The field \(a\mapsto e_a\) is continuous
on the support of \(\sigma\), so Proposition~\ref{r82:arcbound} applies after
identifying \(a\) with \(\exp(2\pi i a)\). Dyadic truncations are
dense in \(E\) and their eigenvalues are roots of unity. Thus periodic
vectors span densely, and frequent hypercyclicity gives chaos.

Finally, for \(n=2^{m_k}\), the digit restriction puts \(E\) in the
union of the intervals \([j/n,j/n+n^{-4}]\), \(0\le j<n\).
Let \(g_n(t)=\exp(2\pi i nt)\) on these intervals and interpolate
linearly back to \(1\) at each succeeding grid point. Then
\(\|g_n-1\|_\infty=O(n^{-3})\) and
\(\|g_n'\|_2^2=O(n^3 n^{-4}+n^2 n^{-6})=O(n^{-1})\).
The actions of \(K_{g_n}\) and \(T^n\) agree on every spanning
eigenvector. Thus \(T^n\to I\) in norm along this subsequence.
A small ball about zero and a ball bounded away from zero cannot meet
under any of these sufficiently large powers, proving the assertion.
\end{proof}

Perfect spanning means that the unimodular eigenvectors still span
densely after any countable set of eigenvalues is discarded. This
qualitative property guarantees hereditary upper-frequent hypercyclicity
\cite[Theorem~3.20]{r79:BGMM}. It leaves room for much longer intervals
of almost periodic behavior than those used in Theorem~\ref{r82:rigidhfhc}.

\begin{theorem}[Perfect spanning without hereditary frequent hypercyclicity]
\label{r82:thinrigid}
A chaotic Hilbert-space operator can have perfectly spanning
unimodular eigenvectors and fail to be hereditarily frequently
hypercyclic. This gives a negative answer to
\cite[Question~9.22]{r79:BGMM}.
\end{theorem}
\begin{proof}
We retain the restricted Kalisch operator from
Theorem~\ref{r82:rigidhfhc} and change its spectral set. Choose
powers of two \(q_1\ge16,q_2,\ldots\) recursively, each at least four
times its predecessor, and take
\[
 E=\Big\{\tfrac14+\sum_{j\ge1}\epsilon_j/q_j:
                       \epsilon_j\in\{0,1\}\Big\},\qquad
 B_k=(k+1)^{2^{q_k}} .
\]
This is a perfect compact set, and truncations are dyadic and dense.
Its fair digit law is atomless with full support, so the continuous
field of indicator eigenvectors still spans densely after removal of
any countable set of eigenvalues. In particular \(T\) is hereditarily
upper-frequently hypercyclic by \cite[Theorem~3.20]{r79:BGMM}, hence
hypercyclic; its dyadic eigenvalues give a dense set of periodic
vectors.

For \(q=q_k\), every point of \(E\) lies in a grid interval
\([j/q,j/q+2/q_{k+1}]\). On these intervals interpolate the
multiplier as in Theorem~\ref{r82:rigidhfhc}, obtaining
\(\|T^q-I\|\le C(q^{3/2}/\sqrt{q_{k+1}}+q^2/q_{k+1})\).
Also \(\|T^r\|\le1+2\pi r\) from the unrestricted exponential
multiplier. Thus we may require \(q_{k+1}\ge2q_k B_k\) and choose it
large enough that, regardless of later choices,
\[
 \max_{0\le r<q_k,\,0\le l\le B_k}
        \|T^{lq_k+r}-T^r\|\le 1/k .
\]
Indeed \(\|(T^q)^l-I\|\le(1+\|T^q-I\|)^l-1\).

We use the residue-class construction of
\cite[proof of Theorem~5.3]{r79:BGMM}. At stage \(k\) enumerate all
subsets of \([0,q_k)\cap\mathbb Z\) of size at least \(q_k/2\) as
\(F_0,\ldots,F_{c-1}\). In the integer block
\([(k+1)^j q_k,(k+1)^{j+1}q_k)\) let \(A\) consist exactly of the
integers whose residue modulo \(q_k\) belongs to \(F_j\).
Fill all intervals outside these stages completely. Stages are
disjoint, since each ends at most at \(q_kB_k\le q_{k+1}/2\).
At each stage start \(q_k\), at least half the preceding interval
\([0,q_k)\) has been filled. Each subsequent whole cell of length
\(q_k\) is at least half full. At a point inside the next cell this
gives density at least \(1/4\); between stages the density increases.
Thus \(A\) has positive lower density.

Choose two open balls whose concentric doubles are disjoint and fix
any vector \(x\). At every stage at least \(q_k/2\) of the initial
iterates \(T^r x\), \(0\le r<q_k\), miss one doubled ball. The same
ball is missed in this fashion at infinitely many stages. At such a
stage let \(F_j\) be the set of residues missing that doubled ball.
For all sufficiently large such \(k\), \(\|x\|/k\) is smaller than
the radius of the original ball. The norm estimate therefore excludes
every return to the original ball along \(A\) throughout the
\(j\)-th block. At the right endpoint the proportion of all returns
inside \(A\) is at most \(1/(k+1)\). No \(x\) can visit both balls
with positive lower density along \(A\).
\end{proof}

\subsection{An analytic group with no eigenvectors and a common invariant law}
Steenbeck constructed an invertible Hilbert-space operator with no
eigenvalues and a nontrivial invariant probability. Remark~1.4 of
\cite{r81:Steen} asks for the corresponding analytic semigroup.
We first arrange polynomial growth in the discrete construction and then
restrict its suspension to a bounded band of frequencies.

\begin{theorem}\label{r81:analyticgroup}
There is a bounded operator \(D\) on a separable complex Hilbert space
such that \(e^{tD}\) has no eigenvalues for any real \(t\ne0\), while the
whole group preserves a common Borel probability different from \(\delta_0\).
The law may be chosen ergodic for the group, with norm in weak \(L^1\)
and of infinite first moment. In particular the semigroup requested in
\cite[Remark~1.4]{r81:Steen} is entire analytic.
\end{theorem}
\begin{proof}
Let \(K=H^1(\mathbb R/\mathbb Z)\) with evaluation kernels \(k_s\), and fix
an irrational \(\alpha\). We use the normalized piecewise affine hat
functions \(u_{r,j}\) on the periodic mesh of size \(h_r=2^{-r}\).
Their supports have length \(2h_r\) and their supremum norms are
\(O(\sqrt{h_r})\). If \(P_r\) projects onto their span, then
\(\sum_j|\langle x,u_{r,j}\rangle|^2\ge c h_r^2\|P_rx\|^2\)
by \cite[Lemma~2.9]{r81:Steen}; the mesh spaces increase densely.
Take \(h_r^{-2}\) copies of each hat, for \(r\ge4\), and assign them to
distinct integer slots \(m\), writing \(v_m\) for the assigned vector
and zero in unused slots. Then
\(\sum_m|\langle x,v_m\rangle|^2=\infty\) for each nonzero \(x\in K\).
Choose disjoint intervals \(I_r\) of lengths \(4h_r\). Each slot for
level \(r\) can be chosen so that \(v_m(\cdot+m\alpha)\) is supported
in \(I_r\), by density of the irrational rotation. Consequently
\[
 g(s)=\sum_m|v_m(s+m\alpha)|^2
       \le C\sum_{r\ge4}h_r^{-2}\mathbf1_{I_r}(s).
\]
For Haar-uniform \(s\), this gives \(\mathbb P(g(s)>M)=O(M^{-1/2})\)
for \(M\ge1\).

Put \(b_j=(1+|j|)^{-2}\), \(w_j=(1+|j|)^{-6}\), and give bilateral
\(K\)-valued sequences the Hilbert norm
\[
 \|x\|_H^2=\sum_j b_j\|x_j\|_K^2+
             \sum_{j,m}w_{j-m}|\langle x_j,v_m\rangle|^2 .
\]
This is a Hilbert direct sum of coordinate spaces with norms equivalent
to the norm of \(K\). Its left shift \(T\) satisfies
\(\|T^l\|\le(1+|l|)^3\) for all \(l\in\mathbb Z\) by the two weight
ratios. Its formal eigenvectors are \((\lambda^j x_0)_j\).
For \(|\lambda|\ne1\), \(\lambda\ne0\), the first sum excludes them,
and for \(|\lambda|=1\), the diagonal \(m=j\) in the second sum excludes
them. Invertibility excludes zero. The random sequence
\(X(s)_j=k_{s+j\alpha}\) belongs to \(H\) almost surely. Indeed
\(\|k_s\|_K\) is bounded, and its second sum is
\(\sum_l w_l g(s+l\alpha)\). Put \(a_l=w_l^{1/3}/\sum_nw_n^{1/3}\).
A union bound gives, for \(M\ge1\),
\[
 \mathbb P\Big(\sum_l w_l g(s+l\alpha)>M\Big)
 \le C M^{-1/2}\sum_l(w_l/a_l)^{1/2}=O(M^{-1/2}).
\]
Consequently \(\mathbb P(\|X\|_H>L)=O(L^{-1})\).
Coordinate truncations give measurability, and \(TX(s)=X(s+\alpha)\)
makes its law \(\mu\ne\delta_0\) invariant.

On \(Y=L^2([0,1);H)\) extend each \(y\) to the line by
\(\widetilde y(s+j)=T^jy(s)\) for \(0\le s<1\).
The translations \((U_ty)(s)=\widetilde y(s+t)\) form a strongly
continuous group with \(\|U_t\|\le C(1+|t|)^3\);
continuity follows first for continuous functions supported in \((0,1)\)
and then by density. For independent \(x\sim\mu\) and uniform
\(u\in[0,1)\), take \(y(s)=T^{\lfloor s+u\rfloor}x\). Translation
replaces \((x,u)\) by
\((T^{\lfloor u+t\rfloor}x,\{u+t\})\), preserving their product law.
The resulting invariant law \(\nu\) on \(Y\) differs from \(\delta_0\);
\(U_1\) acts pointwise by \(T\) and has no eigenvalues.

For Schwartz \(h\) write \(J_hy=\int h(t)U_ty\,dt\).
Choose \(h\) with smooth compactly supported Fourier transform so that
\((J_h)_*\nu\ne\delta_0\). Such kernels give a strong approximate
identity by dilation of one with integral one, using the polynomial
group bound. On \(Y_0=\overline{J_hY}\), take a Schwartz \(k\) whose
Fourier transform equals one on the support of \(\widehat h\).
Then \(J_kJ_h=J_h\) by convolution, so \(J_k=I\) on \(Y_0\).
The translates \(U_tJ_k=J_{k(\cdot-t)}\) are norm differentiable
by the polynomial bound. Thus \(U_t|_{Y_0}=e^{tD}\) for a bounded
generator \(D\). Its common invariant law is \((J_h)_*\nu\).
If \(U_{t_0}v=\lambda v\) there for some \(t_0>0\) and nonzero \(v\),
the two-sided polynomial bounds force \(|\lambda|=1\).
For \(e^{i\beta t_0}=\lambda\), the continuous, nonzero,
\(t_0\)-periodic function \(s\mapsto e^{-i\beta s}U_s v\) has a
nonzero vector Fourier coefficient. Translation of that coefficient
gives an eigenvector for every \(U_t\), in particular for \(U_1\),
a contradiction. Negative times follow by inversion.

The suspension is a factor of the two-torus flow with velocity
\((\alpha,1)\), using coordinates \((s+\alpha u,u)\) modulo one on its
source. This flow is ergodic, so the final law is group-ergodic.
The initial weak \(L^1\) bound passes to the suspension by
\(\|y\|_Y\le\max(1,\|T\|)\|x\|_H\) and then to \(J_h y\) by boundedness.
If this last variable had finite first moment, write it as \(Z(s,u)\)
over the source parameters. For almost every \(u\) it is integrable in
\(s\) and obeys \(Z(s+\alpha,u)=U_1 Z(s,u)\). Its vector Fourier
coefficients in \(s\) must vanish since \(U_1\) has no eigenvalues.
Fourier uniqueness after pairing against a countable separating family
then gives \(Z=0\) almost everywhere, contradicting the chosen law.
\end{proof}

The discrete and continuous constructions admit a common moment refinement.
Steenbeck's outlook in \cite{r81:Steen} mentions sharpness of moment
conditions. Here the endpoint is sharp for the laws constructed above.

\begin{proposition}\label{r82:sublinear}
Let $(\phi_i)_{i\ge1}$ be any countable family of finite nonnegative
nondecreasing functions on $[0,\infty)$ with $\phi_i(t)=o(t)$ at infinity.
The full-support ergodic law in Corollary~\ref{r82:tape} and the
group-ergodic law in Theorem~\ref{r81:analyticgroup} can each be chosen
so that its norm $N$ obeys
\[
   \sup_{t>0}t\,\mathbb P(N>t)<\infty,\qquad
   \mathbb E\phi_i(cN)<\infty\quad(i\ge1,\ c>0),\qquad
   \mathbb EN=\infty .
\]
\end{proposition}
\begin{proof}
Majorize $\phi_i$ on $[1,\infty)$ by
\[
 \psi_i(t)=\int_0^t\sup_{u\ge\max(1,s)}\frac{\phi_i(u)}u\,ds .
\]
This is increasing, concave, Lipschitz and $o(t)$, with $\psi_i(0)=0$.
In particular it is subadditive. In the construction of
Theorem~\ref{r81:analyticgroup}, use a subsequence of dyadic mesh sizes
$h_r$ decreasing to zero. The mesh spaces still increase densely and each
level still contributes at least a constant times $\|P_rx\|^2$ to the
divergent sum. The norm of the rotation coding obeys
\[
  \|X(s)\|\le C_0+C\sum_{r,l} h_r^{-1}\sqrt{w_l}\,
                       \mathbf1_{I_r}(s+l\alpha),\qquad |I_r|=4h_r .
\]
For fixed $b>0$ and $i$, the quantity
$h\sum_l\psi_i(b C\sqrt{w_l}/h)$ tends to zero as $h\downarrow0$,
by dominated convergence and $\sum_l\sqrt{w_l}<\infty$.
Choose the $r$th mesh so that this quantity is at most $2^{-r}$
for $i\le r$ and $b=r$. Subadditivity and integration now give
$\mathbb E\psi_i(c\|X\|)<\infty$ for every $c>0$. The weak $L^1$
estimate for $X$ still holds: distinct dyadic sizes satisfy
$\sum_{h_r\le u}h_r\le2u$. Suspension and bounded frequency localization
multiply the norm bound only by a constant, and the Fourier argument
in Theorem~\ref{r81:analyticgroup} gives divergence of the first moment
both for $X$ and for the resulting nonzero group law.

Apply Theorem~\ref{r82:fullsupport} to the polynomially bounded shift
just constructed, using $X(\theta+c_j)$ for a dense family of translates
$c_j$ of a single Haar phase. The coordinate kernels span $K$ for each
phase and each time, and this array of paths is a factor of the rotation.
Choose $0<a_j\le2^{-j}$ so small that
$\mathbb E\psi_i(j a_j\|X\|)\le2^{-j}$ for $i\le j$; this follows by
dominated convergence. For the conditional Gaussian series $Y$ there,
write $\tau^2=\sum_j a_j^2\|X(\theta+c_j)\|^2$. The inequality
$\tau\le\sum_j a_j\|X(\theta+c_j)\|$ and a union bound with thresholds
proportional to $\sqrt{a_j}$ give
$\mathbb P(\tau>t)\le C(\sum_j\sqrt{a_j})^2/t$.
In particular the conditional series converges. Since
$\mathbb E_g\|Y\|^2=\tau^2$, conditional Markov gives
$\mathbb P(\|Y\|>t)\le\mathbb E\min(1,\tau^2/t^2)=O(t^{-1})$.
Concavity and subadditivity yield
\[
 \mathbb E\psi_i(c\|Y\|)\le \mathbb E\psi_i(c\tau)
       \le\sum_j\mathbb E\psi_i(ca_j\|X\|)<\infty.
\]
Finally the Gaussian fourth-moment bound
$\mathbb E_g\|Y\|^4\le3\tau^4$ and Paley--Zygmund give
$\mathbb E_g\|Y\|\ge c_0\tau$. Since $\tau\ge a_1\|X(\theta+c_1)\|$,
the first moment diverges. Every $a_j$ is positive, as required for
full support.
\end{proof}

\begin{corollary}[An analytic group with full support]\label{r82:analyticfull}
The law of the analytic group in Proposition~\ref{r82:sublinear} may
be chosen with full support. For every $t\ne0$, the only invariant
probability of finite first moment for this time map is $\delta_0$.
\end{corollary}
\begin{proof}
Use the polynomially bounded shift and rotation path \(X(\theta)\) from
Theorem~\ref{r81:analyticgroup}, with the mesh of
Proposition~\ref{r82:sublinear}. Fix dense shifts \((c_j)\) on the
circle. Given Haar-uniform \(\theta\), let \(Y^k\), \(k\ge1\), be the
independent centered Gaussian tape variables obtained as in
Theorem~\ref{r82:fullsupport} from the paths \(X(\theta+c_j)\) and a
fresh array of Gaussian rows for each \(k\). Their conditional laws have
full support in the tape space \(H\), all with conditional second moment
\(\tau(\theta)^2=\sum_j a_j^2\|X(\theta+c_j)\|^2\).
The joint coding is ergodic for the shift since its source is the product
of the irrational rotation and a Bernoulli shift.

Enumerate the trigonometric basis of \(L^2(\mathbb T)\) as \((p_k)\).
For an independent uniform \(u\in[0,1)\), set \(S_u(s)=T^{\lfloor s+u\rfloor}\)
and define, in \(L^2([0,1);H)\),
\[
 y(s)=S_u(s)\sum_{k\ge1}2^{-k}p_k(s+u)Y^k .
\]
Conditionally on \(\theta,u\), the series is centered Gaussian with
second moment bounded by \(C\tau(\theta)^2\), so it converges almost
surely and in mean square. Its support is the whole suspension space:
the functions \(p_k(\cdot+u)v\), \(v\in H\), span densely, and
multiplication by \(S_u\) is boundedly invertible. Translation of \(y\)
is intertwined with the constant-roof suspension of the ergodic joint
coding, by periodicity of the \(p_k\). Thus the common law is group-ergodic.

Use the same frequency cutoff by convolution \(J_h\) chosen in
Theorem~\ref{r81:analyticgroup}. Continuity gives full support on
\(\overline{J_h L^2([0,1);H)}\).
The resulting law has weak \(L^1\) norm and finite prescribed sublinear
moments by the conditional second-moment estimate and the bounds for
\(\tau\) in Proposition~\ref{r82:sublinear}.

On the whole suspension space \(U_1\) is again a bilateral tape shift,
with coordinate space \(K'=L^2([0,1);K)\) and weights \(I\otimes A_n\).
The energy of the constant path with value \(f\in K'\) is
\(\int_0^1 q(f(s))\,ds\) by Tonelli. On the original tape \(q(v)=\infty\)
for every nonzero \(v\), so only the zero constant path has finite energy.
Theorem~\ref{r82:supports}
therefore excludes every nontrivial \(U_1\)-invariant probability of finite
first moment, including on the closed invariant range above. If \(\nu\)
were such a law for \(U_t\), \(t\ne0\), its average
\( |t|^{-1}\int_0^{|t|}(U_s)_*\nu\,ds\) would be group-invariant by
periodicity, nontrivial by invertibility, and still of finite first
moment by boundedness on that time interval. This proves the claim for
every time map and the divergence for the full-support law just constructed.
\end{proof}

\section{The kernel problems for diagonals of self-adjoint operators}\label{r54:diagonals}

Loreaux and Weiss asked whether approximate shifted majorization characterizes the orthonormal-basis diagonals of positive compact operators \cite[Conjecture~3.6]{LW}. The remaining compact positive kernel case also occurs in Bownik and Jasper's compact self-adjoint algorithm \cite[Section~11]{BJ}. For operators with two essential spectral points Bownik and Jasper identify two further kernel classes \cite[Section~12]{newbj:II}. After an affine normalization these consist of positive contractions with both endpoints essential and self-adjoint operators with spectrum bounded by the upper essential endpoint and extending below the lower one, or their reflections.

\subsection{Diagonals of compact operators with finite-dimensional kernel}
\label{schur:sec:finite}

Let $A_p=0_p\oplus\operatorname{diag}(\lambda)$, where
$0\le p<\infty$ and $\lambda_1\ge\lambda_2\ge\cdots>0$ tends to zero.
Write $\mathcal D(A)$ for its diagonals in orthonormal bases. For a nonnegative
null sequence $d$, let $z(d)$ count its zeros and let $d^*$ list its positive
entries in decreasing order. Write $S_x(n)=\sum_{j\le n}x_j$, with $S_x(0)=0$,
and $x\prec\lambda$ if $S_x(n)\le S_\lambda(n)$ for every $n$ and the
extended total sums agree. For decreasing positive $x$ put
\[
 h_r(n)=[S_x(n+r)-S_\lambda(n)]_+,\qquad
 F_r(t)=\sup_{\lambda_{n+1}\le t}h_r(n),\qquad
 J_r(x,\lambda)=\int_0^\infty F_r(t)\,\frac{dt}{t^2}.
\]
Here $n\ge0$. Under majorization $0\le F_r(t)\le r\min(t,\lambda_1)$.

\begin{theorem}[The finite-kernel Schur--Horn criterion]\label{schur:thm:finite40}
A nonnegative null sequence $d$ belongs to $\mathcal D(A_p)$ if and only if
\[
 z(d)\le p,\qquad d^*\prec\lambda,\qquad J_{p-z(d)}(d^*,\lambda)<\infty.
\]
\end{theorem}

Loreaux--Weiss proved sufficiency of $h_r(n)=0$ eventually and necessity of
$h_r(n)/\lambda_{n+1}\to0$, and conjectured the latter sufficient
\cite[Conjecture~3.6]{LW}. The integral separates these two conditions.
The injective case is the theorem of Kaftal--Weiss \cite{KW}.

\begin{proof}
Positivity splits off every prescribed zero as a kernel vector. Hence take
$d=d^*>0$ and induct on $p$, starting with \cite{KW}. Ordinary majorization
is necessary by the Ky Fan principle and the equality of the extended traces.

\emph{Removing one kernel direction.} For $p>0$, compress to $u^\perp$,
where $u=z\oplus y$ is unit, $z\ne0$, and $a=\langle A_pu,u\rangle$.
The compressed kernel has dimension $p-1$. Its positive spectrum $\mu$ is
that of $B=D-vv^*$ with $D=\operatorname{diag}(\lambda)$ and $v=D^{1/2}y$,
by exchanging $L^*L$ and $LL^*$ for $L=A_p^{1/2}|_{u^\perp}$.
The possible positive spectra are precisely those with
\begin{equation}\label{schur:eq:loss40}
 \lambda_j\ge\mu_j\ge\lambda_{j+1},\qquad
 \sum_j\delta_j=a,\qquad \sum_j\delta_j/\lambda_j<\infty,
 \qquad \delta_j=\lambda_j-\mu_j .
\end{equation}
Indeed rank-one interlacing gives the inequalities. Let $P_n,Q_n$ project
onto the leading $n$ eigenvectors of $D,B$. Both projections increase to $I$
since $I-yy^*\ge(1-\|y\|^2)I>0$. The Ky Fan bounds
$\|Q_nv\|^2\le S_\lambda(n)-S_\mu(n)\le\|P_nv\|^2$ give total loss $a$.
The determinant of the compression to $P_n$, divided by that of $P_nDP_n$
on its range, is $1-\|P_ny\|^2\le\prod_{j\le n}\mu_j/\lambda_j$.
This proves summability by taking negative logarithms.

Conversely \eqref{schur:eq:loss40} gives $c=\prod_j\mu_j/\lambda_j>0$.
For the first $n$ terms finite inverse interlacing gives $v_n$ such that
$D_n-v_nv_n^*$ has eigenvalues $\mu_1,\ldots,\mu_n$.
For distinct poles the squared coordinates are the nonnegative residues of
$\prod_{j\le n}(s-\mu_j)/(s-\lambda_j)$. Repeated poles follow by adding
$\epsilon(n-j+1)$ to $\lambda_j$ and $\epsilon(n-j+1/2)$ to $\mu_j$
and passing to a convergent subsequence of the bounded vectors. The
determinant gives $\|D_n^{-1/2}v_n\|^2\le1-c$. Extend these vectors by zero
and take a weak limit $y$ along a subsequence. Compactness of $D^{1/2}$ yields
$v_n\to D^{1/2}y$ in norm, so $D-vv^*$ has spectrum $\mu$ away from zero.
Completing $y$ to a unit vector with a nonzero kernel component realizes this
compression, whose diagonal value is $a$ by the total loss identity.

\emph{Necessity.} Choose a realizing basis vector with nonzero kernel
component, say with value $a=d_k$, and write $d'=d\setminus\{d_k\}$.
Put $T_n=\sum_{j>n}\delta_j$ and $N(t)=\#\{j:\lambda_j>t\}$. Tonelli gives
\begin{equation}\label{schur:eq:tonelli40}
 \int_0^\infty T_{N(t)}\,\frac{dt}{t^2}=\sum_j\frac{\delta_j}{\lambda_j}.
\end{equation}
Since $S_\mu(n)=S_\lambda(n)-a+T_n$, for $n+p\ge k$ one has
$h_p(n;d,\lambda)\le h_{p-1}(n;d',\mu)+T_n$. Thus for small $t$,
$F_p(t;d,\lambda)\le F_{p-1}(t;d',\mu)+T_{N(t)}$, using $\mu\le\lambda$.
Induction and \eqref{schur:eq:tonelli40} give the required integral.

\emph{Sufficiency.} Set $F=F_p(\cdot;d,\lambda)$. Monotonicity gives
$F(t)/(2t)\le\int_t^{2t}F(s)\,ds/s^2\to0$ as $t\downarrow0$.
Only finitely many $j$ have $S_d(j)=S_\lambda(j)$, since equality implies
$h_p(j-1)\ge\lambda_j$. Choose $a=d_k$ strictly below the entries $d_j$
at all such indices, with $k$ beyond them, and let $d'=d\setminus\{d_k\}$.
Then $g_n=a+S_{d'}(n)-S_\lambda(n)$ satisfies
$g_0=a$, $g_n\le a$, $g_n<\lambda_{n+1}$ and $(g_n)_+\le h_p(n)$.
Indeed $a+S_{d'}(n)\le S_d(n+1)$, strictly when $n+1$ is an equality index.
Form the nondecreasing $1$-Lipschitz function
\[
 U(t)=\max\{0,\sup_{n\ge0}(g_n-(\lambda_{n+1}-t)_+)\}.
\]
It satisfies $U(0)=0$, $U(\lambda_1)=a$, $U(\lambda_{n+1})\ge g_n$.
For large $n$, $g_n\le F(\lambda_{n+1})\le\lambda_{n+1}/2$, so a positive
contribution requires $\lambda_{n+1}<2t$ and is at most $F(2t)$.
For each remaining $n$ the strict bound on $g_n$ eliminates its contribution
for small $t$. Consequently $U(t)\le F(2t)$ near zero.
Set $T_n=U(\lambda_{n+1})$, $\delta_j=T_{j-1}-T_j$ and
$\mu_j=\lambda_j-\delta_j$. The Lipschitz bound gives interlacing, the losses
sum to $a$, and $T_{N(t)}\le U(t)$ gives the last condition of
\eqref{schur:eq:loss40} by \eqref{schur:eq:tonelli40}. Hence we may remove
one vector of value $a$. The bound $T_n\ge g_n$ gives $d'\prec\mu$,
including total sums by the loss identity. Finally,
$a+S_{d'}(n+p-1)\le S_d(n+p)$ gives $h_{p-1}(n;d',\mu)\le h_p(n;d,\lambda)$.
As $c'=\inf_j\mu_j/\lambda_j>0$, the new envelope is at most $F(t/c')$.
Its integral is finite, so induction realizes the remaining diagonal.
\end{proof}

\begin{corollary}[Orthostochastic images of nonnegative null sequences]
\label{r66bj:orthostochastic}
Let $a,x$ be nonnegative null sequences indexed by $\mathbb N$, and set
$p=z(a)$, $q=z(x)$. There is a real orthogonal matrix $O$ on $\ell^2$
with $x_i=\sum_j O_{ij}^2a_j$ for every $i$ exactly when
$q\le p$, $x^*\prec a^*$ and the following test holds. For finite $p$
require $J_{p-q}(x^*,a^*)<\infty$. For $p=\infty$ and $q<\infty$
require, for every fixed $r\ge1$, that
$S_{x^*}(n+r)\le S_{a^*}(n)$ for all sufficiently large $n$.
For $p=q=\infty$ ordinary majorization suffices. Pad finite positive
lists with zeros when taking their partial sums.
\end{corollary}
\begin{proof}
The formula says that $x$ is a diagonal of an orthogonal conjugate of $\operatorname{diag}a$.
For finite $p$, Theorem~\ref{schur:thm:finite40} gives necessity. Its construction uses real
square roots of the interlacing residues, weak limits of real vectors and a real eigenbasis
of each compression. Its injective base has a real orthogonal realization by \cite[Corollary~6.1]{KW}.
Real permutations restore the zero rows and the original order.
For infinite rank and infinite kernel, \cite[Corollary~3.5]{LW} gives exactly ordinary majorization
when $q=\infty$ and $\infty$-majorization when $q<\infty$. For finite rank ordinary
majorization is necessary and gives every shifted inequality once $n$ reaches the rank.
In either rank case with infinite kernel, use order $0$ when $q=\infty$ and
order $\infty$ when $q<\infty$ in \cite[Theorem~2.4]{LW}. Its construction
gives a real orthogonal matrix by \cite[Remark~2.5]{LW}. The question discussed immediately before Theorem~2.4 in that paper concerns the existence of an orthostochastic realization. The criterion here decides this for a specified pair of sequences.
\end{proof}

\begin{corollary}[A sharp threshold at every spectrum and kernel dimension]
\label{schur:cor:all40}
For every such $\lambda$, every integer $p\ge1$, and every $\gamma>0$,
there is a strictly positive diagonal of $A_{p-1}$ which is approximately
$p$-majorized, but not eventually $p$-majorized, by $\lambda$.
It is a countable convex combination of finite permutations of the
eigenvalue list of $A_p$, and belongs to $\mathcal D(A_p)$ exactly for
$\gamma>1$.
\end{corollary}
\begin{proof}
For $p=1$ put $w_j=1-\lambda_{j+1}/\lambda_j$,
$s_j=1+\sum_{i<j}w_i$ and
$\delta_j=(\lambda_j-\lambda_{j+1})s_j^{-\gamma}/2$. Then $s_j\to\infty$,
since the product of $1-w_j$ is zero and $0\le w_j<1$.
Set $T_n=\sum_{j>n}\delta_j$, $\mu_j=\lambda_j-\delta_j$ and
$d=(T_0,\mu_1,\mu_2,\ldots)^*$. Interlacing makes $\mu$ decreasing.
Since $T_{n-1}\le\lambda_n$,
$S_d(n)=\max\{S_\mu(n),T_0+S_\mu(n-1)\}\le S_\lambda(n)$,
and the totals agree. Thus $d\in\mathcal D(A_0)$.
Eventually $h_1(n)=T_n>0$, while
$T_n/\lambda_{n+1}\le s_{n+1}^{-\gamma}/2\to0$.
By \eqref{schur:eq:tonelli40}, $J_1<\infty$ exactly when
$\sum_j w_js_j^{-\gamma}<\infty$. Since $s_{j+1}-s_j=w_j\le1$ and
$s_j\ge1$, comparison with $\int_1^\infty s^{-\gamma}\,ds$ gives the threshold.
For the convex combination, choose a positive integer $K$ with
$\mathbb P(K>j)=s_j^{-\gamma}/2$, $j\ge1$. Conditional on $K=k$, use
the sequence with first entry $\lambda_1-\lambda_k$ and $j$th subsequent
entry $\lambda_j-(\lambda_j-\lambda_{j+1})\mathbf1_{\{j<k\}}$.
It mixes just two finite permutations of $(0,\lambda_1,\lambda_2,\ldots)$:
the first and $(k+1)$st slots split $(\lambda_1,0)$, while all other
eigenvalues occupy a single slot. Averaging gives $(T_0,\mu)$, and its
decreasing rearrangement moves just one entry finitely far.
For $p>1$, split $\lambda_1$ into $p$ equal positive diagonal entries using
$p-1$ kernel vectors, and apply the $p=1$ construction to the remaining spectrum.
Once the finite block is included among the largest $n+p$ entries,
$h_p(n)$ equals the tail's $h_1(n-1)$ at the same spectral scale.
The finite block itself averages $p$ finite permutations, preserving the
countable-mixture assertion. Taking $p=1$ and $0<\gamma\le1$ disproves \cite[Conjecture~3.6]{LW}.
\end{proof}

Bownik--Jasper's compact self-adjoint algorithm has a remaining kernel test
precisely in its zero-excess branch \cite[Theorems~11.1 and~11.5]{BJ}.
The reduction in \cite[Theorem~11.5]{BJ} splits the prescribed diagonal at zero between the positive and negative spectral subspaces. Here the excess on a sign is $\liminf_n(S_{\lambda_\pm}(n)-S_{x_\pm}(n))$, with finite lists padded by zeros for partial sums. Positive-excess branches are decided by \cite[Theorem~10.1]{BJ}. The compact positive cases of finite rank, zero kernel or infinite kernel are covered by \cite{KW,LW}.

\begin{corollary}[The finite-kernel test at a compact spectral split]
In the \label{r55bj:compact}zero-excess splitting branch of \cite[Theorem~11.5]{BJ}, let $p<\infty$ be the kernel dimension and let $z(d)$ be the number of prescribed zero entries. Require $z(d)\le p$ before forming the difference $p-z(d)$.
Use $x_\pm,\lambda_\pm$ for the decreasing lists of the strictly positive
magnitudes on each sign of the diagonal and spectrum. For infinite lists
$x,\lambda$ with $x\prec\lambda$, let
$I(x,\lambda)=\{r\ge0:r\text{ is an integer and }J_r(x,\lambda)<\infty\}$,
an initial interval since the integrand increases with $r$.
Its endpoint requires at most one integral. Let
$q=\sup\{r\ge0:r\text{ is an integer and }h_r(n)=0\text{ eventually}\}$.
All integers $0\le r\le q$ are allowed. If $q<\infty$, the sole further
candidate is $q+1$, included exactly when $J_{q+1}<\infty$: integrability
implies approximate $r$-majorization, which implies eventual
$(r-1)$-majorization \cite[Remark~3.2]{LW}. For two finite
lists passing ordinary majorization after zero padding and $|x|\ge|\lambda|$,
put $I(x,\lambda)=\{|x|-|\lambda|\}$, and put $I=\varnothing$ in all other
cases. Then the branch accepts exactly when
\[
 p-z(d)\in I(x_+,\lambda_+)+I(x_-,\lambda_-).
\]
\end{corollary}
\begin{proof}
Theorem~11.5 of \cite{BJ} splits the hidden kernel between the two sign
spaces; Theorem~\ref{schur:thm:finite40} and finite-dimensional Schur--Horn
give exactly the two sets of allowable dimensions.
\end{proof}

The two representative operators in the open problem following
\cite[Theorem~2.8]{r56lw:survey} give particularly simple series.
Their spectra are $0\oplus(2^{1-j})_{j\ge1}$ and
$0\oplus(1/j)_{j\ge1}$. In each case require $z(d)\le1$ and
$d^*\prec\lambda$. For $r=1-z(d)$ write
$M_n=\sup_{m\ge n}[S_{d^*}(m+r)-S_\lambda(m)]_+$.
The remaining condition is respectively
\[
   \sum_{n\ge1}2^{n-1}M_n<\infty
   \qquad\text{or}\qquad
   \sum_{n\ge1}M_n<\infty .
\]
Indeed $F_r(t)=M_n$ on $[\lambda_{n+1},\lambda_n)$, and the mass of
this interval under $dt/t^2$ is $1/\lambda_{n+1}-1/\lambda_n$.
The contribution from $t\ge\lambda_1$ is finite.

\begin{corollary}[Convexity of the diagonals of a positive compact operator]
\label{r62bj:convex}
For every positive compact operator $A$ on a separable Hilbert space,
$\mathcal D(A)$ is convex. For every infinite-rank $A$ with nonzero
finite-dimensional kernel there are countable convex combinations of
finite permutations of its eigenvalue list that lie outside
$\mathcal D(A)$.
\end{corollary}
\begin{proof}
Finite rank and zero kernel are covered by \cite[Corollary~6.7]{KW},
and infinite kernel by \cite[Corollary~4.3]{LW}. Suppose $A=A_p$
with $0<p<\infty$. Take $x,y\in\mathcal D(A)$ and put
$v=\theta x+(1-\theta)y$ with $0<\theta<1$. For $w=x,y,v$ let $Z_w$
be its zero set and $r_w=p-|Z_w|$. Then $Z_v=Z_x\cap Z_y$, and
$v^*\prec\lambda$ by subadditivity of the leading partial sums and
equality of the extended traces. For all sufficiently large $n$, the
indices of the $n+r_v$ largest entries of $v$ contain
$(Z_x\cup Z_y)\setminus Z_v$. Indeed this set is finite and $v$ is
strictly positive there. On those leading indices $x$ has at most
$n+r_x$ nonzero entries, and $y$ at most $n+r_y$. Consequently
\[
 h_{r_v}(n;v^*,\lambda)\le
 \theta h_{r_x}(n;x^*,\lambda)+(1-\theta)h_{r_y}(n;y^*,\lambda).
\]
This is the zero-set count of \cite[Lemma~4.2]{LW} applied to the
actual deficits. Taking the running supremum for small $t$ and
integrating proves $J_{r_v}(v^*,\lambda)<\infty$. Theorem~\ref{schur:thm:finite40}
realizes $v$. The countable averages follow from
Corollary~\ref{schur:cor:all40} with $0<\gamma\le1$.
\end{proof}

The same endpoint integral governs compact and two-point spectra.
For a positive compact spectrum $\lambda$ of infinite rank and a
nonnegative null sequence $x\prec\lambda$, put
\[
 \Phi_x(t)=\sum_i(x_i-t)_+,\qquad G_x(t)=\Phi_\lambda(t)-\Phi_x(t),
 \qquad M_r(t)=\sup_{0<s\le t}[rs-G_x(s)]_+ .
\]
\begin{lemma}[Ordered sums and the endpoint hinge]\label{r63bj:dual}
For each integer $r\ge0$, $F_r(t)\le M_r(t)$ for $t>0$.
If $h_r(n)=o(\lambda_{n+1})$, then $M_r(t)\le F_r(2t)$ for all
sufficiently small $t$. Consequently $J_r(x^*,\lambda)<\infty$ if and
only if $\int_0^{\lambda_1}M_r(t)\,dt/t^2<\infty$, and approximate
$r$-majorization is equivalent to $\liminf_{t\downarrow0}G_x(t)/t\ge r$.
\end{lemma}
\begin{proof}
Write $S_x(k)$ for the sum of the $k$ largest entries. Its upper hinge
is $\Phi_x(t)=\max_{k\ge0}(S_x(k)-kt)$. At $t=\lambda_{n+1}$ the
spectral maximum is attained at $n$. Using $k=n+r$ for the target gives
$h_r(n)\le M_r(\lambda_{n+1})$. In the reverse direction the term $k=n+r$
in $rt-G_x(t)$ is at most $h_r(n)-(\lambda_{n+1}-t)_+$, by using
both $n$ and $n+1$ in the spectral maximum. Each fixed $k$, including
$k<r$, gives a strictly negative term near zero, since
$S_x(k)\le S_\lambda(k)<S_\lambda(k+1)$. If $h_r(n)=o(\lambda_{n+1})$,
then eventually $h_r(n)\le\lambda_{n+1}/2$. A positive remaining term
requires $\lambda_{n+1}<2t$, giving the upper bound.
Monotonicity and integrability of $F_r(t)/t^2$ give $F_r(t)=o(t)$,
which implies the hypothesis of the reverse bound.
Finally $h_r(n)=o(\lambda_{n+1})$ is equivalent to $F_r(t)=o(t)$,
and $M_r(t)=o(t)$ is equivalent to the stated lower limit.
\end{proof}

\begin{lemma}[Averaging at a spectral cut]\label{r69bj:cut}
Let $E=E^*$ be bounded on a separable Hilbert space, let $e(s)=(s-a)_+$ or $(a-s)_+$, and assume
$\operatorname{Tr}e_t(E)<\infty$ for every $t>0$, where $e_t=(e-t)_+$.
For a coordinatewise measurable random diagonal $D$ put $v_i=\mathbb ED_i$,
$J_i=\mathbb Ee(D_i)-e(v_i)$, and
\[
 G_t(d)=\operatorname{Tr}e_t(E)-\sum_i e_t(d_i),\qquad
 L(d)=\liminf_{t\downarrow0}G_t(d).
\]
Then $G_t(v)\ge0$ and $L(v)\ge\mathbb EL(D)+\sum_iJ_i$.
If $L(v)=0$, then $L(D)=0$ almost surely. Every coordinate of $D$
lies on the same closed side of $a$ as its mean, with $D_i=a$ whenever
$v_i=a$, simultaneously almost surely.
\end{lemma}
\begin{proof}
Jensen's inequality for each basis vector's spectral measure gives
$G_t(D)\ge0$. The identity
\[
 G_t(v)=\mathbb EG_t(D)+\sum_i\bigl(\mathbb Ee_t(D_i)-e_t(v_i)\bigr)
\]
has nonnegative terms. Apply Fatou along a sequence realizing the lower
limit for $v$. Equality in the hinge at each coordinate gives the
closed-side assertion, including equality at the cut.
\end{proof}

\begin{theorem}[Convexity and probability means of exact compact diagonals]\label{r67bj:convex}
Let $T=T^*$ be compact on a separable infinite-dimensional Hilbert
space, and put $p=\dim\ker T$. Its diagonal set $\mathcal D(T)$ is
convex unless $0<p<\infty$ and both spectral signs have infinite rank.
In that case a whole open segment between two diagonals is missing.
Every coordinatewise probability mean of diagonals belongs to
$\mathcal D(T)$ exactly when $p=0$ or $p=\infty$.
\end{theorem}
\begin{proof}
Let $v_i=\mathbb E D_i$. Dominated convergence gives $v\in c_0$.
By \cite[Proposition~2.7]{BJ}, the leading-sum inequalities and
excesses $\sigma_\pm$ are the hinge inequalities and lower limits of
Lemma~\ref{r69bj:cut} at zero. The two Jensen gaps coincide since
$s_+-s_-=s$. Write their common sum as $h$, so
$\sigma_\pm(v)\ge\mathbb E\sigma_\pm(D)+h$.
If $v_+\in\ell^1$ and $\operatorname{Tr}T_+<\infty$, Tonelli gives
$\sigma_+(v)=\mathbb E\sigma_+(D)+h\le\sigma_-(v)$ by necessity in
\cite[Theorem~10.1]{BJ}. If $\operatorname{Tr}T_+=\infty$, either
$h=\infty$ or $D_+\in\ell^1$ almost surely and both original excesses
are infinite by the same necessity. Either way $\sigma_-(v)=\infty$.
Reflect for the negative trace inequality. If an excess vanishes,
Lemma~\ref{r69bj:cut} transfers the zero and weak-sign count bounds
from $D$ to $v$. Thus \cite[Theorem~10.1]{BJ} realizes $v$ whenever
the total excess is positive.

For two zero excesses, split the indices by the signs of $v$. In each
realizing basis of $D$ the rows on either strict sign set span its entire
strict spectral subspace together with some kernel directions: the
remaining rows have no spectral mass on that side by
\cite[Theorem~11.5]{BJ} and Lemma~\ref{r69bj:cut}. There the magnitudes
$u$ of $D$ and their mean $w$ are majorized by the strict spectral list
$\ell$, including extended trace equality by Tonelli. The zero rows of
$v$ take one kernel direction each. For $p=0$, realize each sign by
ordinary Schur--Horn \cite{KW}. For finite $p$ and a two-point law, if
one sign has finite rank, its extra dimension is the fixed number of
its rows minus that rank. The other extra dimension is then fixed too;
Corollary~\ref{r62bj:convex} realizes both means.

For $p=\infty$ a finite-rank side takes the extra dimension fixed by
its row count; an infinite-rank side can always take zero by \cite{KW}.
If neither a finite-rank side nor the zero rows take infinitely many,
some infinite-rank side can take infinitely many for $u$ with positive
probability, since every original basis allocates the whole kernel.
This event is measurable: by \cite[Corollary~3.5]{LW} and
Lemma~\ref{r63bj:dual}, it means infinitely many zeros of $u$ or
$(\Phi_\ell(t)-\Phi_u(t))/t\to\infty$.
In the second alternative with positive probability, Jensen and Fatou
give the same divergence for $w$. In the first with positive
probability, the Jensen gap $\mathbb E\Phi_u(t)-\Phi_w(t)$ is at least
$t\sum_{w_i>t}\mathbb P(u_i=0)$. Dividing by $t$ gives divergence
by Tonelli, since every $w_i>0$ on this sign set. Thus $w$ can take an
infinite kernel, and the sign bases and zero rows give a full basis.

For $0<p<\infty$ an infinite-rank sign exists. On it take the countable
mixture of Corollary~\ref{schur:cor:all40} for one hidden direction.
Keep $p-1$ visible zeros and the other sign's eigenbasis list fixed.
The latter sign has zero excess, forcing decoupling. It cannot absorb
the last direction: its full list either has finite dimension or fails
approximate $1$-majorization since $h_1(n)=\ell_{n+1}$.
The chosen mean cannot absorb the direction either.

Finally assume both spectral signs have infinite rank and $0<p<\infty$.
Keep $p-1$ common zero coordinates. For a decreasing positive spectral
list $\ell$, let $\widetilde\ell$ be its decreasing rearrangement after
replacing $\ell_1$ by two entries $\ell_1/2$. A rotation using the
remaining kernel direction realizes this split. Use it on the positive
side for $x$ and on the negative side for $y$, ordering both lists on
the same two sets of indices. Eventually
$S_{\widetilde\ell}(n+1)=S_\ell(n)$. Thus both excesses of $v$ vanish,
and both signs must decouple by \cite[Theorem~11.5]{BJ}. The additional
kernel dimension has to go to one side. Yet on the positive side
$h_1(n;v_+^*,\lambda_+^*)/\lambda_{+,n+1}^*=1-\theta$ eventually,
and on the negative side the corresponding ratio is $\theta$. Neither
satisfies necessary approximate $1$-majorization \cite[Theorem~3.4]{LW}.
This excludes every interior point of the segment.
\end{proof}

\begin{theorem}[A randomly placed kernel vector]\label{r63bj:means}
Let $\lambda$ be a nonincreasing infinite list of positive eigenvalues tending to zero.
Insert a zero at a random position $N$ in the decreasing positive spectrum
$\lambda$, shifting all subsequent entries by one, and let $d$ be
the mean of the resulting lists. If $a_1>a_2>\cdots$ are the distinct
spectral levels and $\lambda_N=a_K$, then
\[
 d\in\mathcal D(0\oplus\operatorname{diag}\lambda)
 \quad\Longleftrightarrow\quad
 \mathbb E\sum_{j=2}^{K}(1-a_j/a_{j-1})<\infty .
\]
For $\lambda_n=q^{n-1}$, $0<q<1$, this asks $\mathbb EN<\infty$.
For $\lambda_n=1/n$ it asks $\mathbb E\log(1+N)<\infty$.
\end{theorem}
\begin{proof}
A deterministic $N$ gives a permutation.
For a nondeterministic $N$ every mean entry is positive. With
$p_i=\mathbb P(N=i)$ the sum of the first $m$ unsorted coordinates is
$S_\lambda(m-1)+\lambda_m\mathbb P(N>m)$. Every later entry is at most
$\lambda_m$. Sorting can increase the prefix by at most
$\sum_{i\le m}(\lambda_m-d_i)_+$. Since $d_i\ge(1-p_i)\lambda_i$,
this summand is at most $p_i\lambda_m$. For large $m$ it vanishes whenever
$\lambda_i>2\lambda_m$, using $p_i\le1/2$ eventually and positivity of
the remaining finitely many $d_i$. Hence
\[
 \lambda_m\mathbb P(N>m)\le h_1(m-1)\le
 \lambda_m\mathbb P(\lambda_N\le2\lambda_m).
\]
For $a_j\le t<a_{j-1}$ this sandwiches $F_1(t)$ between
$a_j\mathbb P(K>j)$ and $a_j\mathbb P(a_K\le2a_j)$ for all large $j$.
Integration against $dt/t^2$ multiplies the two probabilities by
$w_j=1-a_j/a_{j-1}$. Put $C_k=\sum_{j=2}^k w_j$ with $C_0=C_1=0$.
The lower sum equals $\mathbb E C_{K-1}$ up to finitely many terms,
and $C_K-C_{K-1}\le1$. In the upper sum the terms after $K$ have
$a_j\ge a_K/2$. Their weights sum to at most $\log2$ since
$w_j\le\log(a_{j-1}/a_j)$. Thus the integral is finite exactly when
$\mathbb E C_K<\infty$. Apply Theorem~\ref{schur:thm:finite40}.
\end{proof}

\subsection{The positive two-endpoint kernel problem}
\label{newbj:sec}

Let \(0\le E\le I\) be self-adjoint on a separable Hilbert space with
\(\sigma_{\rm ess}(E)=\{0,1\}\). Its spectrum is countable, so it has an
orthonormal eigenbasis. Write \(\lambda\) for its eigenvalue list, with
multiplicities. For a sequence \(x\) in \([0,1]\) define
\[
 f_x(t)=\sum_i\bigl(\min(x_i,t)-t x_i\bigr),\quad 0<t<1,
 \qquad z_j(x)=\#\{i:x_i=j\}\quad(j=0,1).
\]
The Blaschke condition \(f_x(1/2)<\infty\) is equivalent to
\(\sum_i\min(x_i,1-x_i)<\infty\) and gives \(f_x(t)<\infty\) throughout
\((0,1)\). Bownik and Jasper proposed an eventual endpoint inequality in
\cite[Sec.~12, after Theorem~12.1]{newbj:II}. The uniform margin is a special case of the integral endpoint test below.
For Blaschke lists let $G=f_d-f_\lambda$. A diagonal entry equal to
$j\in\{0,1\}$ uses one direction in $\ker(E-jI)$. Write $r_j$ for
the dimension left in that eigenspace after the visible entries are
removed. Thus $z_j(d)+r_j=\dim\ker(E-jI)$ in cardinal arithmetic.
When both visible and total multiplicities are infinite, the residual
dimension is specified separately.

\begin{example}[The limiting endpoint inequalities are insufficient]
\label{newbj:ex}
Let \(\xi_j=2^{-j-2}\), \(j\ge1\). Apply
Corollary~\ref{schur:cor:all40} with \(p=1\) and \(\gamma=1\) to obtain
a strictly positive sequence \(x\) which is a doubly stochastic image of
\((0,\xi)\) and cannot be a diagonal of \(0\oplus\operatorname{diag}\xi\).
Consider
\[
 E=(0\oplus\operatorname{diag}\xi)\oplus\operatorname{diag}(1-\xi),
 \qquad d=x\oplus(1-\xi).
\]
Here \(r_0=1,r_1=0\). Both Blaschke sums are finite. Write \(M\) for
the doubly stochastic matrix sending \((0,\xi)\) to \(x\), with the
zero column indexed by \(0\). Rowwise concavity gives \(G\ge0\).
For a row with \(x_i\ge t\), its Jensen gap for \(s\mapsto\min(s,t)-ts\)
is at least \(tM_{i0}\). Therefore
\[
 \liminf_{t\downarrow0}G(t)/t
 \ \ge\ \lim_{t\downarrow0}\sum_{x_i\ge t}M_{i0}=1.
\]
For \(t>\xi_1\) the low-block gaps vanish, since also \(x_i\le\xi_1\).
The high block is unchanged, so \(G=0\) on \((\xi_1,1)\). This verifies
the upper limiting inequality and the slope congruence. At \(t=1/2\),
equality forces spectral splitting by \cite[Theorem~11.3]{newbj:II}.
There are no spectral or diagonal entries at this cut, so a basis
realizing \(d\) would realize \(x\) on \(0\oplus\operatorname{diag}\xi\).
This contradiction shows that the endpoint conditions (12.3)--(12.4)
of \cite{newbj:II}, even with the interior conditions (12.1)--(12.2),
do not characterize the diagonals.
\end{example}

\paragraph{A necessary integral bound at the two endpoints.} For \(j=0,1\) put
\(G_0(s)=G(s)\), \(G_1(s)=G(1-s)\), and
\[
 F_j(t)=\sup_{0<s\le t}[r_js-G_j(s)]_+,\qquad 0<t\le1/2 .
\]

\begin{theorem}[Two endpoint integrals]\label{r47bj:exact}
Let \(0\le E\le I\) have essential spectrum \(\{0,1\}\), finite endpoint
eigenspaces and eigenvalue list \(\lambda\). Suppose \(d\) is a diagonal of \(E\) in a
complete orthonormal basis and \(f_d(1/2),f_\lambda(1/2)<\infty\).
Then \(z_j(d)\le z_j(\lambda)\) for \(j=0,1\), \(G\ge0\),
\(G'\) is an integer at some common differentiability point of the two
functions \(f_d,f_\lambda\), and
\[
 \int_0^{1/2}\frac{F_0(t)+F_1(t)}{t^2}\,dt<\infty .
\]
\end{theorem}
\begin{proof}
Remove the eigenvectors giving the visible endpoints. The interior and
count conditions are necessary by \cite[Theorem~12.1]{newbj:II}.
We first prove necessity of the integral at zero, writing \(r=r_0>0\).
From a realizing basis choose \(r\) vectors whose projections onto
\(\ker E\) are linearly independent. Write \(U:\mathbb C^r\to H\) for
their isometric inclusion, with the same argument over \(\mathbb R\).
Then \(c=\det(U^*P_{\ker E}U)>0\), and their diagonal entries are positive.

Round the spectrum in \([0,1/m)\) to zero and in \((1-1/m,1]\) to one
to obtain \(E_m\), \(m>2\). The Blaschke condition gives
\(\|E_m-E\|_1\to0\). Its eigenvalue list is obtained by rounding
\(\lambda\), and its diagonal in the realizing basis differs from \(d\)
by at most this trace norm in \(\ell^1\). The corresponding function
\(G_m\) therefore converges uniformly to \(G\), since
\(s\mapsto\min(s,t)-ts\) is \(1\)-Lipschitz uniformly in \(t\).
There is a finite reducing subspace \(W\) containing \(\operatorname{ran}U\):
take all the interior spectral spaces of \(E_m\) and the projections of
\(\operatorname{ran}U\) onto its two endpoint spaces. Its complement
has only endpoint spectrum.

Let \(D\) be the positive part of \(E_m|_W\) restricted to its support,
and \(Y=P_{\operatorname{ran}D}U\). The positive spectrum of the
compression of \(E_m|_W\) to \(W\ominus\operatorname{ran}U\) equals
that of \(B=D^{1/2}(I-YY^*)D^{1/2}\), by exchanging \(L^*L\) and
\(LL^*\) for \(L=E_m^{1/2}|_{W\ominus\operatorname{ran}U}\).
For their ordered eigenvalues \(b_i,a_i\) respectively denoting those
of \(B,D\), we have \(0<b_i\le a_i\) and
\[
 \prod_i b_i/a_i=\det(I-Y^*Y)=\det(U^*P_{\ker E_m}U)\ge c .
\]
Here the kernel projection increases under rounding, and determinant is
monotone on positive matrices. The function
\(H_m(t)=\sum_i(\min(a_i,t)-\min(b_i,t))\) is nondecreasing, with
\(\int_0^\infty H_m(t)\,dt/t^2=\sum_i\log(a_i/b_i)\le-\log c\).
Let \(\alpha_{k,m}\) be the \(r\) selected diagonal entries. Their sum
equals \(\sum_i(a_i-b_i)\). On the orthogonal complement of the selected
vectors, rowwise Jensen applied to the concave function
\(\min(s,t)-ts\) bounds its sum on the remaining diagonal below by
its spectral sum. The sum over all rows thus gives
\[
 G_m(t)\ge \sum_{k=1}^r\min(\alpha_{k,m},t)-H_m(t).
\]
Fix \(\epsilon>0\) below all the selected diagonal entries of \(E\).
For all sufficiently large \(m\), the sum on the right is \(rt\) on
\((0,\epsilon]\). Monotonicity of \(H_m\) and uniform convergence imply
\(F_0(t)\le\liminf_m H_m(t)\) there. Fatou proves the lower endpoint
integral. Apply the same proof to \(I-E\) for the upper endpoint.

\end{proof}

\paragraph{Absorbing a kernel on a reducing tail.}
The pointwise interpolation needed between spectral values follows
from the integral slope congruence. The following construction also
keeps any strict slack needed away from the endpoint.

\begin{lemma}[A diagonal with a prescribed local hinge budget]\label{r62bj:local}
Suppose that $\ell_i\downarrow0$ is a complete positive spectral tail
of a compact operator. Let $g\ge0$ be continuous near zero. On each open spectral gap
assume that $g$ is concave and locally piecewise affine with integer slopes. Add $r$ zero eigenvectors, where $r$ is finite or countably infinite.
Write $F(t)=\sup_{0<s\le t}[rs-g(s)]_+$ for finite $r$.
If $\int_0^\eta F(t)\,dt/t^2<\infty$, an arbitrarily small reducing
tail together with all $r$ zero directions admits a strictly positive
diagonal. Its increase in the hinge sum is bounded by $g$, is strictly
smaller wherever the increase is positive, and vanishes outside the
chosen neighborhood. For $r=\infty$ this conclusion holds provided
$g(t)/t\to\infty$, with increase at most $g/2$.
\end{lemma}
\begin{proof}
For $r=0$ there is nothing to absorb. For $1\le r<\infty$ monotonicity
gives $F(t)/(2t)\le\int_t^{2t}F(u)\,du/u^2\to0$. When $r>1$, split
one sufficiently small eigenvalue $a_0$ and $r-1$ zero directions into
$r$ equal diagonal entries. This costs at most $(r-1)t$, strictly
below the original $g(t)\ge(r-1/2)t$ wherever the cost is positive,
and exactly $(r-1)t$ below $a_0/r$. Subtract that cost from the budget,
take a smaller untouched reducing tail, and denote the remaining
budget by $g$. Now there is one zero. The running supremum of
$h(t)=[t-g(t)]_+$ equals the original $F(t)$ near zero.
Take $\ell_1$ small enough that $h(\ell_i)+\ell_i^2\le\ell_i/2$.
Define
\[
 U(t)=\sup_i\big[h(\ell_i)+\ell_i^2-(\ell_i-t)_+\big]_+,
 \quad a=U(\ell_1),\qquad m_i=\ell_i-U(\ell_i)+U(\ell_{i+1}).
\]
The nondecreasing $1$-Lipschitz envelope has $U(0)=0$, $a>0$,
$\ell_i\ge m_i\ge\ell_{i+1}$, and
$U(t)\le F(2t)+4t^2$ near zero: each positive term has
$\ell_i<2t$. On every gap $[l,u]=[\ell_{i+1},\ell_i]$, telescoping gives
\[
 H(t)=\sum_k\{\min(\ell_k,t)-\min(m_k,t)\}
      =\max\{U(l),U(u)-u+t\}\le U(t).
\]
If a convex piecewise affine function $h$ with integer slopes is
nonincreasing up to $t$, then $h(t)\le h(l)$; otherwise every slope
to the right of $t$ is at least one and $h(t)\le h(u)-u+t$.
This applies to the present $h$ by concavity of $g$. The added squares
at the two ends imply $H(t)>h(t)$, including at an eigenvalue;
equal successive eigenvalues contribute no open gap.
Moreover $H$ is constant and equal to $a$ for $t\ge\ell_1$. Thus
$\sum_k\log(\ell_k/m_k)=\int_0^\infty H(t)\,dt/t^2<\infty$.
The interlacing losses sum to $a$ and satisfy
$\sum_k(\ell_k-m_k)/\ell_k<\infty$. By the proved inverse
compression \eqref{schur:eq:loss40}, $(a,m_1,m_2,\ldots)$ is an
actual diagonal of this tail and its last zero direction. Its cost
is $\min(a,t)-H(t)$ by trace equality. This is strictly below $g$
on the tail interval and vanishes beyond it.

For $r=\infty$ choose distinct summable eigenvalues $a_k\downarrow0$
with $g(t)\ge2kt$ for $0<t\le a_k$, and split the pair consisting
of $a_k$ and the $k$th zero into two entries $a_k/2$. Their total
cost at $t$ is at most $t\#\{k:a_k>t\}\le g(t)/2$, with zero cost
above $a_1$.
\end{proof}

\begin{theorem}[Exact allocation of the two endpoint eigenspaces]
\label{r62bj:positive}\label{r47bj:full}\label{r48bj:full}
Let $0\le E\le I$ have essential spectrum $\{0,1\}$ on a separable
Hilbert space. Let $d$ be a list in $[0,1]$ of the dimension of that
space, $x$ its interior entries and $\lambda$ the interior eigenvalues.
For a Blaschke target require $\lambda$ Blaschke, $G=f_x-f_\lambda\ge0$,
and integral $G'$ at one common differentiability point. Put
$g_0(t)=G(t)$ and $g_1(t)=G(1-t)$. Prescribe residual dimensions $r_j$
by $z_j(d)+r_j=\dim\ker(E-jI)$, $j=0,1$, in cardinal arithmetic.
Exactly these residual dimensions can occur in a realizing basis if
and only if the preceding conditions and the following endpoint tests
hold.
\begin{enumerate}
\item If interior eigenvalues accumulate at $j$, require
$\int_0^{1/2}\sup_{0<s\le t}[r_js-g_j(s)]_+\,dt/t^2<\infty$ for
finite $r_j$, and $g_j(t)/t\to\infty$ for infinite $r_j$.
\item At a gap of the interior spectrum require $r_j=\infty$ if $x$
accumulates there, and otherwise require
$r_j=\lim_{t\downarrow0}g_j(t)/t$.
\end{enumerate}
For a Blaschke $d$, existence without a specified residual means that
one such compatible pair passes.

For a non-Blaschke target fix any compatible residual dimensions. Remove
the visible endpoint eigenvectors and call the remaining operator $E'$.
An endpoint $j$ remains essential precisely when the interior spectrum
accumulates there or $r_j=\infty$. If both endpoints remain essential,
the allocation passes. If only zero remains essential, apply the positive
compact diagonal test to $E'$ and $x$. If only one remains essential,
apply it to $I-E'$ and $(1-x_i)_i$. In finite dimension use Schur--Horn,
including equality of the list length and the dimension. Without a prescribed
allocation, existence is equivalent to the two visible endpoint count bounds.
\end{theorem}
\begin{proof}
Remove the visible endpoint eigenvectors. Applying rowwise Jensen
to $u\mapsto\min(u,t)-tu$ and summing gives $f_\lambda(t)\le f_x(t)$.
In particular a Blaschke target forces a Blaschke spectrum. The slope
congruence is then necessary by \cite[Theorem~12.1]{newbj:II}. For finite $r_j$, the rounding proof
of Theorem~\ref{r47bj:exact} uses only $r_j$ independent projections
of rows onto that residual eigenspace; the other endpoint space has
arbitrary dimension throughout that proof. For infinite residual
multiplicity at zero, the Jensen gap for each row of mean $x_i\ge t$
is at least $t\|P_{\ker E}e_i\|^2$. Every residual mean is positive,
so summation gives $G(t)/t\to\infty$. Reflect at one. If interior
eigenvalues accumulate at zero, $f_\lambda(t)/t\to\infty$. Since
$f_x(t)/t$ stays bounded when $x$ has a lower gap, $x$ must also
accumulate at zero.

At a lower spectral gap $\sum_i(1-\lambda_i)<\infty$. With $r_0$
finite, the residual $I-E$ is trace class, so $x$ cannot accumulate
at zero. Conversely if $x$ has a lower gap its residual trace sum
$\sum_i(1-x_i)$ is finite. Trace equality forces
$r_0=\sum_i(1-x_i)-\sum_i(1-\lambda_i)$, precisely the lower slope
of $G$ on the gap. Reflection proves the upper statement.

For sufficiency reserve exactly the visible endpoint subspaces, with
the specified residual dimensions. Apply Lemma~\ref{r62bj:local} on
the spectral accumulation sides, using disjoint tails of width less
than $1/3$. Each relevant $g_j$ is concave between eigenvalues and its
slopes are integers since slopes change by integers at each interior
breakpoint. On the complement of the unused gapped endpoint spaces this
gives an actual interior diagonal $y$ with $f_y\le f_x$ and the same
accumulation endpoints as $\lambda$. The costs vanish on a middle
interval, preserving integral slope congruence there, and vanish
near any gapped endpoint, preserving its prescribed slope. The
changes preserve Blaschke summability.

We specify exactly the directions appended to $y$ in the three cases
of \cite[proof of Theorem~6.3]{newbj:II}. If $x$ accumulates at both
ends, its Case~1 appends infinitely many copies of just the endpoints
missed by $y$, exactly as required by the second test. If $x$
accumulates only at zero, append $k=\sum x_i-\sum y_i$ ones, and
append zero zeros for infinite $y$ or infinitely many zeros for
finite $y$. These counts equal the upper slope and the required
lower allocation. Trace equality and $f_y\le f_x$ give ordinary
majorization of the positive target by the augmented list, exactly
the transfer in Case~2 of that proof. Its finite-positive-rank case
also follows from \cite[Theorem~1.8]{LW}. Reflection treats sole
accumulation at one. Finally, for finite $x$, $y$ is finite. Append
the two finite boundary slopes; their sum is $|x|-|y|$, and the
upper slope is the trace difference. The hinge inequality gives
finite majorization, so Schur--Horn transfers to $x$. Restoring the
visible endpoint vectors proves sufficiency.

For a non-Blaschke target the residual entries $x_i$ are strictly interior
and $\sum_i\min(x_i,1-x_i)=\infty$. If both endpoints remain essential,
the M\"uller--Tomilov theorem in \cite[Theorem~8.1]{newbj:II} realizes
this very list on $E'$. If zero is the sole essential point, $E'$ is compact.
The reflected statement holds when one is the sole essential point. These
give the stated tests for the specified allocation. When only existence is
asked for, retain infinitely many residual directions in each originally
infinite endpoint eigenspace, even if the visible count there is infinite.
Every finite endpoint eigenspace has an interior spectral tail at that
endpoint by essentiality. Thus both essential points persist. The visible
count bounds are necessary since a row with mean zero or one is an
endpoint eigenvector.
\end{proof}

\begin{corollary}[Pointwise margins at finite multiplicity]
Under the interior and count conditions of Theorem~\ref{r62bj:positive},
suppose the target is Blaschke and both endpoint eigenspaces are finite.\label{newbj:thm}\label{r60bj:proposed}
The eventual bounds $g_j(t)\ge r_jt$ for $j=0,1$, proposed after
\cite[Theorem~12.1]{newbj:II}, suffice for realization. Indeed they
make the two running deficits vanish near zero. The example above
shows that the corresponding lower-limit bounds alone do not suffice.
\end{corollary}

For a case separating the integral and pointwise tests at both ends,
put $\ell_j=2^{-j-2}$, $m_j=\ell_j(1-1/(8j^2))$ and
$a=\sum_j(\ell_j-m_j)$. Use the spectral block
$(0,\ell_1,\ell_2,\ldots)$ and its reflection at one. On the
lower block take the proposed diagonal $(a,m_1,m_2,\ldots)$,
and reflect this list on the upper block. The interlacing ratios have
positive product because $\sum_j1/(8j^2)<\infty$, so
\eqref{schur:eq:loss40} realizes both blocks with one hidden
direction at each end. At every sufficiently small node $t=\ell_n<a$
the lower slack equals $t-\sum_{j\ge n}(\ell_j-m_j)<t$.
The reflected slack has the same deficit. Thus the uniform
pointwise margins in Theorem~\ref{newbj:thm} fail at both ends of
this realized diagonal.

\subsection{The one-sided two-endpoint kernel problem}
This is the second two-point class singled out in \cite[Section~12]{newbj:II}, after reflection when necessary. Both the operator and the multiplicities of its endpoint eigenspaces are prescribed. The Jensen identity below determines the tails on which the residual endpoint dimensions must be realized.
Let $E\le I$ have essential spectrum $\{0,1\}$ and eigenvalue list
$\lambda$. For a real list $x$ use the interior hinge
$f_x(t)=\sum_{0<x_i<1}(\min(x_i,t)-t x_i)$ for $0<t<1$, and set
$f_x(t)=\sum_i(t-x_i)_+$ for $t<0$. Write $d'=(d_i:d_i<1)$,
$q=\#\{i:d_i=1\}$ and $N=\dim\ker(I-E)$. Require $d\le1$,
$q\le N$, and $X=f_\lambda-f_d\ge0$ on $(-\infty,0)$.
Write $L=\liminf_{t\uparrow0}X(t)$. In the case of finite $L$ and
finite interior Blaschke sums, take the largest $\delta\in[0,L]$
such that
\[
 G(t)=f_d(t)-f_\lambda(t)+(1-t)\delta
\]
has integer slope at a common differentiability point. Require
$\delta=L$ if the negative spectrum is summable. An empty set of
choices fails the test. The congruence fixes $\delta$ modulo one,
so the maximum exists whenever the set is nonempty. The strict case
means $\delta>0$ and $G(t)>0$ for all $0<t<1$.

\begin{theorem}[Signed strict branch with prescribed upper residual]\label{h50bj:signed}\label{r62bj:signed}
Assume $0<L<\infty$, finite Blaschke sums, strictness as above, and that $d'$ accumulates at zero. Put $g(s)=G(1-s)$ and prescribe a cardinal $r$ with $q+r=N$. If $d'$ accumulates at one, realizability with this actual residual $r$ is equivalent to
\[
 \begin{cases}
 \displaystyle\int_0^{1/2}\sup_{0<s\le t}[rs-g(s)]_+\,\frac{dt}{t^2}<\infty,
     &r<\infty,\ 1\text{ is an interior spectral limit},\\
 \displaystyle g(s)/s\longrightarrow\infty\quad(s\downarrow0),
     &r=\infty,\ 1\text{ is an interior spectral limit},\\
 r=\infty,&1\text{ is separated from the interior spectrum}.
 \end{cases}
\]
If $d'$ stays away from one, a realizing residual must have finite $r$ and no interior spectral accumulation at one. In this case its realizability is exactly the compact diagonal test for $E$ with precisely $r$ unit eigenvectors retained.
\end{theorem}
\begin{proof}
Reserve the visible ones. For a realizing residual basis $e_i$ put $S=I-E$ on the residual, $y_i=\langle Se_i,e_i\rangle>0$. The parameter $\delta_*=\sum_i(\langle E_+e_i,e_i\rangle-(d_i)_+)$ belongs to $[0,L]$ with the required congruence by \cite[Theorems 4.1, 11.1 and (4.11)]{newbj:II}, hence $\delta_*\le\delta$. For $0<s<1$, the identity
\[
 \min(z,s)-s\min(z,1)=f_{1-z}(1-s),\qquad z\ge0,
\]
where the right denotes the single-entry hinge, gives by summing nonnegative Jensen gaps
\begin{equation}\label{h50bj:jensen}
 g_*(s):=f_d(1-s)-f_\lambda(1-s)+s\delta_*
 =\sum_i\left(\min(y_i,s)-\langle\min(S,s)e_i,e_i\rangle\right).
\end{equation}
For $C=\min(S,I)$ its diagonal minimum is at most $\min(y_i,s)$ and $\min(C,s)=\min(S,s)$. Thus its Jensen sum is bounded by $g_*(s)\le g(s)$. The spectrum of $C$ is Blaschke. For finite $r>0$ the rounding estimate from the necessity proof of Theorem~\ref{r47bj:exact} applies to the diagonal $x_i=\langle C e_i,e_i\rangle>0$ with its own Jensen gap $G_C(s)\le g(s)$ and proves integrability of the running supremum of $[rs-G_C(s)]_+$ against $dt/t^2$. That supremum dominates the required one. For infinite $r$, rows with $y_i\ge s$ contribute at least $s\|P_{\ker S}e_i\|^2$ in \eqref{h50bj:jensen}. Their weights increase to the infinite trace of the kernel projection as $s\downarrow0$. If the upper spectrum has a gap and $r$ is finite, the residual $E$ is compact, precluding target accumulation at one. Conversely a separated target forces a finite $r$ and a spectral gap: for small $s$, $g_*(s)=\tr((sI-S)_+)<\infty$, which rules out infinite kernel or infinitely many positive eigenvalues of $S$ tending to zero.

Suppose the target accumulates at one. If the upper spectrum accumulates, apply Lemma~\ref{r62bj:local} on a reducing reflected upper tail together with its $r$ zero directions. On each spectral gap $g$ is concave, and its slopes are integers by the congruence. This absorbs the residual ones, preserves strictness and leaves the lower spectral part untouched. The new list still accumulates at one from below. For finite $r$ this follows from the interlacing tail, and for infinite $r$ from the split pairs.

It remains to create an upper tail when the upper spectrum has a gap and $r=\infty$. Here $g(s)/s\to\infty$ by target accumulation and boundedness of $f_\lambda(1-s)/s$. Choose distinct anchors $a_j\to0$ from the spectral list, leaving an infinite reserve near zero; choose either all $a_j\ge0$ or all $a_j<0$. This is possible by essentiality at zero. Pair each residual one with an anchor and change the pair to $(1-\epsilon_j,a_j+\epsilon_j)$, where $0<\epsilon_j<(1-a_j)/2$. For negative anchors require $\epsilon_j<|a_j|/2$. The new compensated slack loses at most $\min(s,\epsilon_j)$ per pair: for nonnegative anchors this follows by the minimum formula, while for a negative anchor the hinge change is $\min(s,\epsilon_j)-s\epsilon_j$ and reducing the compensating parameter by $\epsilon_j$ contributes $s\epsilon_j$. The exterior hinge loss is bounded by $\epsilon_j$, vanishes below $a_j$, and tends to $\epsilon_j$ at zero. Choose the negative anchors in an interval $(t_0,0)$ where $X\ge L/2$, and $\sum\epsilon_j<\min(L/4,\delta/2)$. Exterior majorization persists, and dominated convergence gives $L_{\rm new}=L-A$, $\delta_{\rm new}=\delta-A$, where $A$ is the total of the increments on negative anchors. Since $g(s)/s\to\infty$ near zero and $g(1-)=\delta>0$, choose the increments so that $\min(s,\epsilon_j)\le2^{-j-2}g(s)$ for every $s\in(0,1)$. This creates an upper interior tail, preserves strictness, and removes every residual one. Negative summability status is unchanged.

If the provisional list still lacks an interior lower tail, its untouched reserve contains infinitely many nonpositive entries tending to zero. Choose distinct upper entries $b_j\uparrow1$, and pair them with nonpositive $a_j\to0$. This is a rotation on the coordinate plane of the provisional basis (it need not be an eigenbasis). Choose $x_j>0$ and $\eta_j=|a_j|+x_j$ with
\[
 \eta_j\le 2^{-j-2}\inf_{0<t\le b_j}G(t),\qquad b_j-\eta_j>x_j.
\]
Here $G$ is the current slack before this family of changes. The infimum is positive by strictness and $G(0+)=\delta>0$. Replace the pair by $(x_j,b_j-\eta_j)$; the trace on the coordinate plane is unchanged and both values lie between its old diagonal values, so the rotation exists. Choose negative anchors where $X\ge L/2$ and $\sum|a_j|<\min(L/4,\delta/2)$. Decrease $L,\delta$ by this sum. The compensated cost of the pair at $t$ equals
\[
 \min(x_j,t)+\min(b_j-\eta_j,t)-\min(b_j,t)+|a_j|.
\]
It vanishes for $t\ge b_j$ and lies in $[0,\eta_j]$ otherwise. The exterior loss is supported near zero, bounded by $\sum|a_j|$, and tends to that sum at zero. Thus the compensated slack remains positive. For both countable constructions all but finitely many cost functions are constant on each compact interior interval, hence the integral slope congruence persists. The negative summability alternative and the relation $0<\delta\le L$ are preserved, including equality if required.

The provisional list is a diagonal of the actual residual operator; it and $d'$ accumulate only at $0,1$, have finite Blaschke sums, and have no entries equal to one. For this pair of lists (7.1) is exterior majorization and (7.2) is the compensated interior inequality. The derivative congruence gives (7.3). The upper exterior excess and parameter are both zero. Thus (7.5) asks $L_{\rm new}-\delta_{\rm new}\ge0$, and (7.4) asks the reverse inequality if the negative entries of the provisional list are summable. Both hold because this difference has been preserved and equals zero in the summable case. Theorem 7.4(ii) of \cite{newbj:II} now transfers the provisional list to $d'$. This is a transfer on the same operator, so the prescribed multiplicities are not altered.

For a separated target the preceding necessity leaves a compact residual with exactly $r$ unit eigenvectors. The compact fixed-operator criterion of \cite[Sections~10--11]{BJ}, using Theorem~\ref{schur:thm:finite40} in its positive finite-kernel branch, tests this restriction. A realizing residual basis adjoined to the reserved unit vectors gives the desired basis.
\end{proof}

\needspace{5\baselineskip}
\begin{corollary}[The one-sided fixed-operator test]\label{h50bj:fixed}\label{r62bj:signedall}
Let $E,d,X,L,q,N$ be as above, with a bounded real list of the dimension of the Hilbert space, $d\le1$, $q\le N$ and $X\ge0$. For $0<L<\infty$ with a Blaschke target require that the spectrum is Blaschke, that $0,1$ are accumulation points of the full target, and that the parameter $\delta$ exists and $G\ge0$. Under these conditions the following alternatives decide realization.
\begin{enumerate}
\item If $L=\infty$, require $\sum_{d_i<1}(1-d_i)=\infty$. If $0<L<\infty$ and $f_d(1/2)=\infty$, the target passes.
\item If $L=0$, cut at zero. In the finite positive-$L$ Blaschke branch cut at zero if $\delta=0$, or at any interior $c$ with $G(c)=0$. At the selected cut require an allocation of the actual cut eigenspace between the visible cut entries and the two sides such that both restricted fixed-operator tests pass.
\item In the remaining strict branch require some $r$ with $q+r=N$ for which Theorem~\ref{h50bj:signed} passes.
\end{enumerate}
\end{corollary}
\begin{proof}
Exterior majorization and the visible dimension bound are necessary. In the finite positive-$L$ Blaschke case \cite[Theorems~4.1, 9.1 and 11.1(iii)]{newbj:II} give the stated spectral and target prerequisites and an actual admissible parameter. Increasing that parameter to the admissible maximum preserves the interior inequality. When $L=\infty$ the sum in item~1 equals $\operatorname{Tr}(I-E)=\infty$ in any realizing basis because zero is essential for $E$. Reserve the visible ones, keeping infinitely many residual ones whenever $N=\infty$ (also when $q=\infty$). If $N<\infty$ the upper interior spectral tail is infinite. Thus the residual retains both essential endpoints. All residual target entries are strictly below one and the exterior excesses are $L$ and zero. Sufficiency follows from \cite[Theorem~8.3(c$'$), Lemma~8.2]{newbj:II} in the two respective cases of item~1.

At $L=0$ the necessity of the zero cut is \cite[Proposition~2.10]{newbj:II}. At maximal $\delta=0$ the actual positive-part Jensen gap $\delta_*$ vanishes. Equality row by row forces the negative and positive spectral parts to decouple. At $G(c)=0$ one has $0\le g_*(1-c)\le G(c)=0$. Rowwise equality in \eqref{h50bj:jensen} forces the same decoupling at $c$. In each case rows at the cut are cut eigenvectors. On each side a realizing basis spans the open spectral side together with a subspace of the cut eigenspace, by \cite[Theorem~2.22]{newbj:II}. Conversely adjoining such two bases and the visible eigenspace gives a realizing basis of $E$. At an interior cut each nonempty side is finite dimensional or a compact translate. At zero the lower side is compact, and the upper side is a positive contraction with two essential endpoints, a compact translate, or finite dimensional. Theorem~\ref{r62bj:positive}, the compact test above and finite Schur--Horn decide these restrictions. The full target accumulating at zero implies that $d'$ accumulates there, so item~3 is exactly Theorem~\ref{h50bj:signed}. Reflection interchanges the two one-sided orientations.
\end{proof}

\begin{corollary}[All diagonals of a fixed bounded self-adjoint operator]
\label{r60bj:global}
Let $T$ be a bounded self-adjoint operator on a separable Hilbert
space. The spectral reductions of \cite[Section~11]{BJ} and
\cite[Sections~10--11]{newbj:II}, with the following terminal tests, characterize
the diagonals of $T$ in complete orthonormal bases. At a positive
compact operator use Theorem~\ref{schur:thm:finite40} for a finite
nonzero kernel and infinite positive rank, and the known zero-kernel,
infinite-kernel and finite-rank tests otherwise. At two essential
points normalize them to $0,1$ and use Theorem~\ref{r62bj:positive}
for the positive-contraction class or Corollary~\ref{r62bj:signedall}
for the one-sided class, applying
reflection at the lower endpoint when appropriate. Every cut uses the actual eigenspace
multiplicity of the given operator.

For a bounded self-adjoint $T$ on an arbitrary Hilbert space, fix any
orthogonal decomposition into separable nonzero reducing subspaces
$H_a$. A family $(d_i)_{i\in I}$ is a diagonal precisely when the
indices $a$ and $i$ admit paired partitions $(A_b,I_b)$ into countable
sets such that the preceding test passes on
$T|_{\bigoplus_{a\in A_b}H_a}$ with diagonal $(d_i)_{i\in I_b}$.
\end{corollary}
\begin{proof}
Finite dimension uses Schur--Horn. A single essential spectral point
gives a compact self-adjoint operator after translation. The compact
reduction in \cite[Section~11]{BJ} leaves only the positive compact
kernel test, decided above and by \cite{KW,LW}. For two or more
essential spectral points, the algorithms in
\cite[Sections~10--11, Figures~1--3]{newbj:II} reduce every
inconclusive case to precisely three classes: the positive compact
class and the two classes in \cite[Theorems~11.1 and~11.3]{newbj:II}.
The criterion requires at least one consistent choice of the cut
eigenspace allocations for which all terminal tests, including the
residual test at limit stages, pass. Indeed a realizing basis splits
as prescribed by \cite[Theorem~2.22]{newbj:II}, and any successful
allocation yields a basis by orthogonal sum. In the published
countable pruning procedure each removed piece is a reducing subspace;
at a limit the remaining subspace is the intersection and the removed
pieces form its orthogonal complement. Thus the cited reductions and
the compact test and the two endpoint tests exhaust the fixed-operator problem, including
the countably transfinite branch.

For the final assertion the cyclic subspace of a vector under a
bounded self-adjoint operator is separable and reducing. A maximal
orthogonal family of such subspaces exhausts the space. For any fixed
separable decomposition and any realizing basis $(e_i)$, join $i$ to
$a$ if the projection of $e_i$ onto $H_a$ is nonzero. Each $i$ meets
only countably many summands by square summability. Each $a$ meets
only countably many $i$, by expanding a countable basis of $H_a$
in $(e_i)$. Every connected component is therefore countable; within
it the two closed spans coincide and the displayed test passes.
Conversely, bases on successful components have orthogonal union a
basis of the original space.
\end{proof}

\begin{corollary}[Frame norms and the possible excesses]\label{r63bj:frame}
Let $S$ be a bounded positive invertible operator on $H$ and
$(s_i)_{i\in I}$ a bounded family of nonnegative numbers. For every
cardinal $k$ with $\dim H+k=|I|$, apply the spectral criterion of
Corollary~\ref{r60bj:global} to $S\oplus 0_k$ and the target
$(s_i^2)_{i\in I}$. The passing cardinals are exactly the possible
dimensions of the kernel of the synthesis operator of a frame on $H$
with frame operator $S$ and norms $s_i$. In particular the set of
passing cardinals is nonempty exactly for the norm families in
\cite[Problem~1]{r63bj:frames}.
\end{corollary}
\begin{proof}
Use the Gram-operator reduction of Antezana, Massey, Ruiz and Stojanoff
\cite[Proposition~4.5]{r63bj:AMRS}, also given in
\cite[Theorem~1]{r63bj:frames}. In detail, a frame with
synthesis operator $T:\ell^2(I)\to H$ has $TT^*=S$. Thus
$V=T^*S^{-1/2}$ is an isometry, and $T^*T=VSV^*$ is unitarily
equivalent to $S\oplus0_{\dim\ker T}$. Its standard diagonal is
$(s_i^2)$. Conversely a basis realizing $(s_i^2)$ for $S\oplus0_k$
gives a unitary $U:\ell^2(I)\to H\oplus H_k$, $\dim H_k=k$. Set
$T=S^{1/2}P_H U$, with $P_H$ the projection onto $H$. Then $TT^*=S$,
$\dim\ker T=k$, and $\|Te_i\|^2=s_i^2$. Positivity and invertibility
of $S$ give the frame bounds for $(Te_i)$.
\end{proof}

For example, suppose $H$ is separable, $S=I-A$ with
$0\le A\le cI$, $c<1$, and $A$ trace class. For
$x_i=1-s_i^2\ge0$ with $\sum_i x_i<\infty$, the only possible excess is
$k=\sum_i x_i-\operatorname{Tr}A$. Indeed, in a realizing basis
$(x_i)$ is the diagonal of $A\oplus I_k$, so this identity follows
by taking its trace. If $k$ is a nonnegative integer of compatible
dimension, apply the positive compact diagonal test to $A\oplus I_k$.
When $k=0$ and $A$ has infinite rank and nonzero finite kernel,
Theorem~\ref{schur:thm:finite40} supplies its endpoint integral.
Thus the unrestricted frame-norm problem already contains the
compact finite-kernel obstruction.

\subsection{Convexity at two essential spectral points}
\begin{theorem}[Convexity at the extreme essential points]\label{r68bj:exteriorconvex}
Let $E=E^*$ be bounded on a separable Hilbert space, with $0,1$ the
extreme points of $\sigma_{\rm ess}(E)$. Put
$H(E)=\operatorname{Tr}\bigl((\min(E,I-E))_+\bigr)$.
If $H(E)<\infty$, then $\mathcal D(E)$ is nonconvex.
If the interval restriction $\mathbf1_{[0,1]}(E)H$ retains both
essential endpoints, the following are equivalent: $H(E)=\infty$,
convexity of $\mathcal D(E)$, and closure of $\mathcal D(E)$ under
all coordinatewise probability means.
\end{theorem}
\begin{proof}
Apply Lemma~\ref{r69bj:cut} to $e_0(s)=(-s)_+$ and $e_1(s)=(s-1)_+$.
Their shifted traces are finite since the exterior spectrum is discrete.
Write $L_j,J_{j,i}$ for the corresponding excesses and row gaps, and
$m(s)=\min(s,1-s)$.

Suppose $H(E)=\infty$. If the essential spectrum has exactly two
points, $E$ is diagonalizable, and its eigenvalue list has divergent
interior Blaschke sum. Thus \cite[Theorem~10.2]{newbj:II} applies
both in this case and when there are more essential points.
Its necessary conditions show that $L_0(D)<\infty$ implies
$\sum_i (D_i)_+=\infty$, and when both excesses are finite
$\sum_i m(D_i)_+=\infty$. These implications pass to $v$.
For the first use $(s)_+=s+e_0(s)$ and
$\sum_i J_{0,i}<\infty$ when $L_0(v)<\infty$. At the other endpoint
reflect about $1/2$. For the second use concavity of $m$ and
$m_+=m+e_0+e_1$ to obtain
$m(v_i)_+\ge\mathbb E m(D_i)_+-J_{0,i}-J_{1,i}$.
Together these give precisely alternatives (i)--(iv) in that theorem.
At each zero excess, the closed-side assertion of
Lemma~\ref{r69bj:cut} transfers the count bounds (10.3)--(10.4). Theorem~10.2 now realizes $v$ whenever
each endpoint with zero excess is essential from inside $[0,1]$,
which is exactly its condition (10.5).

Finally $H(E)<\infty$ rules out any other essential point. The rest
of the spectrum is discrete, so the spectrum is countable and $E$
has an eigenbasis. Fix its eigenvalue list $\lambda$ and choose distinct basis vectors
indexed by $\mathbb Z$ whose values $u_n$ approach zero on the left
and one on the right. Choose them so fast that the sum of the distances
from their respective limiting endpoints is below $1/4$. Shift that
sublist by $u_n\mapsto u_{n+1}$ to obtain another eigenbasis diagonal
$y$. Their midpoint $v=(\lambda+y)/2$ differs from $\lambda$ in
$\ell^1$, with total difference $1/2$ by telescoping.
Lipschitz continuity of the shifted exterior hinges and dominated
convergence give
$L_j(v)=\sum_i(e_j(\lambda_i)-e_j(v_i))\in[0,1/4)$,
since only the selected rows change. Both interior lists are Blaschke.
Write $c(s)=\min(1,\max(0,s))=s+e_0(s)-e_1(s)$.
For $0<t<1$ the interior hinge is
$f_d(t)=\sum_i(\min(c(d_i),t)-t c(d_i))$. At any common
differentiability point its single-entry derivative is
$\mathbf1_{\{s>t\}}-c(s)$.
Only finitely many threshold indicators change, and hence
\[
 f'_\lambda(t)-f'_v(t)\equiv \tfrac12-L_0(v)+L_1(v)\pmod1.
\]
For an actual diagonal \cite[Theorem~4.1]{newbj:II} requires this to
equal $\delta_1-\delta_0$ modulo one with $0\le\delta_j\le L_j(v)$.
But $(L_0(v)-\delta_0)-(L_1(v)-\delta_1)$ has absolute value below
$1/4$, so cannot be congruent to $1/2$. Thus the midpoint is missing.
\end{proof}

When $P=\mathbf1_{[1/2,\infty)}(E)$ satisfies $E-P\in\mathcal S_1$,
the entire open segment in the proof is missing. Write $q$ for the
Boolean list obtained by rounding $\lambda$ at $1/2$. In a basis
realizing $d$ with $d-\lambda\in\ell^1$, the diagonal $p$ of $P$
satisfies $p-q\in\ell^1$ and
$\sum_i(d_i-\lambda_i)=\sum_i(p_i-q_i)$ by trace invariance of $E-P$.
Kadison's projection theorem makes this an integer \cite{r68bj:Kadison}.

\begin{corollary}[All operators with two essential spectral points]
\label{r68bj:alltwo}
Suppose $\sigma_{\rm ess}(E)=\{0,1\}$. Finite $H(E)$ makes
$\mathcal D(E)$ nonconvex. If
$H(E)=\infty$, the interval restriction $\mathbf1_{[0,1]}(E)H$ retains
at least one essential point. Retaining both gives closure under every
coordinatewise probability mean. If only zero is retained, convexity
holds exactly when the strict negative spectral subspace is finite
dimensional or $\dim\ker E$ is zero or infinite; closure under all
probability means holds exactly when $\dim\ker E$ is zero or infinite.
If only the upper endpoint is retained, apply both tests to $I-E$.
\end{corollary}
\begin{proof}
The finite $H(E)$ case follows from Theorem~\ref{r68bj:exteriorconvex}.
If the interval restriction loses both of them it is finite dimensional,
so $H(E)<\infty$. Consider $H(E)=\infty$ with only zero retained.
Then $S=E|_{\mathbf1_{(-\infty,1)}(E)H}$ is compact and $S_+$ has
infinite trace. Above one lies $I+R$ with $R$ positive, compact,
injective and of infinite rank, and the unit eigenspace has finite
dimension $N$.

For a coordinatewise mean $v=\mathbb ED$ of exact diagonals,
the exterior conditions, alternatives (i)--(iv) and count bounds
hold as in Theorem~\ref{r68bj:exteriorconvex}. If $L_1(v)>0$,
condition (10.5) of \cite[Theorem~10.2]{newbj:II} holds since zero is
essential from inside the interval, so $v$ is realized.

If $L_1(v)=0$, use the common sets of indices where $v$ is below,
at, or above one. Lemma~\ref{r69bj:cut} gives $L_1(D)=0$ and closed-side
support almost surely. By \cite[Proposition~2.10]{newbj:II} the spectral
rows of the realizing bases respect the cut as well.
On the upper rows of $v$ the deviations from one are strictly positive
and majorized by the eigenvalues of $R$, with equal extended sums:
each original restriction there is a diagonal of $R$ with possibly
some extra zero directions after subtracting $I$. The upper list of
$v$ is thus realized on $I+R$ with no added unit eigenvectors by
\cite{KW}. After reserving the
$q$ cut rows, place all $K=N-q$ remaining unit eigenvectors on the
lower side. In each such basis the lower list $w$ is a diagonal
of $S\oplus I_k$ with $k\le K$. It remains a diagonal of $S\oplus I_K$.
For $k<K$ adding $K-k$ unit eigenvalues increases
the positive excess by exactly $K-k$, since the discarded last $K-k$
terms in the leading spectral sum tend to zero. The negative excess
and the nonpositive spectral counts stay fixed. Thus all the necessary
conditions of \cite[Theorem~10.1]{BJ} persist. If $w_+\in\ell^1$,
the original positive excess was already infinite since
$\operatorname{Tr}S_+=\infty$, and necessity forces its negative excess
to be infinite as well. The inequality required when $w_-\in\ell^1$
can only improve.
The total excess is now positive, giving realization by that theorem.
Theorem~\ref{r67bj:convex} now realizes the lower mean in this fixed
space under its stated conditions for two-point laws or for arbitrary
laws, respectively: adjoining $I_K$ changes neither the kernel nor the
negative rank, and the positive rank was already infinite.

Conversely, extend each diagonal of $S$ by the same eigenbasis list of
$I+R$ and all $N$ visible unit coordinates. Any mean of these lists has
zero upper excess and reserves all unit eigenvectors. Every realizing
basis splits at one, hence its lower restriction must realize the
original mean on $S$. The obstructions in Theorem~\ref{r67bj:convex}
therefore transfer to $E$.
\end{proof}

\paragraph{The same basis for several compact observables.}
A mean of joint diagonals retains the requirement that all observables
use the same basis. For a nonempty norm-precompact family
$\mathcal K\subset K(H)_{\rm sa}$ on a separable infinite-dimensional
complex Hilbert space, fix a basis $(e_n)$ and put
$d_V(n,A)=\langle A V e_n,V e_n\rangle$. Let $\mathcal D(\mathcal K)$
consist of the arrays $d_U$ for unitary $U$. Uniform convergence refers
to the supremum over all $n$ and $A\in\mathcal K$.

\begin{theorem}[Joint diagonals and common contractive compressions]
\label{r64bj:joint}
The pointwise and uniform closures of $\mathcal D(\mathcal K)$ both equal
$\{d_V:\|V\|\le1\}$. If the family commutes, this closed set is convex
if and only if $\dim_{\mathbb R}\operatorname{span}\mathcal K\le1$.
This is also exactly when every coordinatewise random joint diagonal
has its mean in the uniform closure. When the dimension exceeds one,
a midpoint of two common eigenbasis diagonals differing by a three-cycle
has positive distance on finitely many rows of two observables from
every joint diagonal.
\end{theorem}
\begin{proof}
The ball of contractions is weak-operator compact. Compactness of each
$A$ gives continuity of $V\mapsto d_V(n,A)$ there, so its image is
pointwise compact. For the reverse inclusion, $Ve_n\rightharpoonup0$
and a finite norm net gives $\sup_A|d_V(n,A)|\to0$. Fix $\eta>0$ and
choose $k$ beyond which this bound is less than $\eta$. Write
$X=(Ve_1,\ldots,Ve_k)$ as a column operator. Take a finite-rank
projection $P$ containing its range with $\sup_A\|A-PAP\|<\eta$.
For an isometry $J:\mathbb C^k\to(PH)^\perp$, the columns of
$X+J(I-X^*X)^{1/2}$ are orthonormal and their observations approximate
those of $X$ within $2\eta$. On their orthogonal complement take a
common finite-rank approximation for the family of compressed
observables with error $\eta/2$ and support of rank $l$. Enlarge that
support to an $m$-plane and take its Fourier basis. Each vector puts
mass $l/m$ on the support. For sufficiently large $m$ all its observations
are below $\eta$ in absolute value, as are those on the remaining
orthogonal complement. This completes the unitary approximation.

For a single observable the projection of the exact set onto its first
$k$ rows is convex. Compress to the finite span of two realizing
$k$-frames, extend each to a basis of the span and use finite Schur--Horn.
Thus the pointwise closure is convex, also for real multiples of a single
observable. Every probability mean belongs to a pointwise compact convex
set containing its values, by taking all finite-coordinate projections.

Now take independent commuting $A,B\in\mathcal K$. Their joint spectral
points, together with zero, form a compact noncollinear set
$S\subset\mathbb R^2$ with every nonzero point an isolated eigenvalue of
finite multiplicity. Take $v\in S$ of largest norm. Its linear functional
$z\mapsto v\cdot z$ strictly exposes $v$, with a uniform gap on the
remaining points. The directions from $v$ to them form a compact subset
of an open semicircle, with two distinct extreme rays. These rays determine
supporting faces $F_1,F_2$ meeting only at $v$, with $0\notin F_1$ after
relabelling. In each face let $w_j$ be the nearest spectral point
different from $v$. A probability measure on $S$ with mean
$b_j=(v+w_j)/2$ is supported on $F_j$ and assigns mass at least $1/2$
to $v$, by linear distance along that ray.

If the multiplicity of $v$ is $l$, the finite row pattern
$v$ repeated $l-1$ times followed by $b_1,b_2$ is impossible for a
contraction. Indeed, a column $z$ has each joint eigenspace mass
$\|P_s z\|^2$ and the remaining mass $1-\|z\|^2$ at zero.
The $v$-columns are unit vectors in $P_vH$. Unit columns of a contraction
are orthogonal to all other columns by positivity of $I-V^*V$.
The $b_1$-column is also unit because $0\notin F_1$. Its inner product
with the $b_2$-column is the inner product of their $P_v$ projections,
both nonzero in the same remaining line, since the faces meet only at $v$.
This contradicts orthogonality. The finite projection of the contraction
image is compact, so the pattern has positive distance from it.
Fix the $l-1$ eigenvectors and cycle eigenvectors of $v,w_1,w_2$ in
orders $(v,w_2,w_1)$ and $(w_1,v,w_2)$. If $w_2=0$ is not an eigenvalue,
use a nonzero joint eigenvalue tending to zero; the positive distance
persists. The common eigenbasis for the whole commuting compact family
is obtained by simultaneous diagonalization on the nonzero finite
eigenspaces and completion on the common null space. This proves the
midpoint obstruction and the converse statements.
\end{proof}

\begin{corollary}[Exact diagonals of a compact normal operator]\label{r67bj:normalconvex}
Let $N$ be compact and normal on a separable infinite-dimensional
complex Hilbert space. Its complete-basis diagonal set is convex if
and only if $N=e^{i\phi}T$ for some compact self-adjoint $T$ for which
at least one of $\operatorname{rank}T_+,\operatorname{rank}T_-$ is finite
or $\dim\ker T$ is zero or infinite.
All coordinatewise probability means are exact if and only if the spectrum
lies on a line through zero and $\dim\ker N$ is zero or infinite.
\end{corollary}
\begin{proof}
If the spectrum is not contained in a line through zero, the three-cycle
in Theorem~\ref{r64bj:joint} gives a midpoint of two exact diagonals
outside their uniform closure. If it is contained in such a line,
$N=e^{i\phi}T$ by the normal spectral theorem and
$\mathcal D(N)=e^{i\phi}\mathcal D(T)$. Apply Theorem~\ref{r67bj:convex}.
\end{proof}

Applied to $(\operatorname{Re}N,\operatorname{Im}N)$, the theorem says
that the uniform closure of the full diagonal set of a compact normal
operator $N$ is convex precisely when its spectrum lies on a line
through zero. Au--Yeung and Sing proved the finite-matrix collinearity
criterion for exact diagonals \cite{r64bj:AuSing}; Chuan studied
contractive limits of compact-operator compressions \cite{r64bj:Chuan}.
For the two moments $(T,T^2)$ of a compact self-adjoint $T$, convexity
therefore holds exactly when $T=cP$ for a real scalar $c$ and a
finite-rank projection $P$, including $T=0$.

Commutation is essential for the converse. Let $N=e_1e_2^*$ be a
rank-one nilpotent. Its real and imaginary parts are independent and
do not commute. In complex coordinates their joint closure is the
convex set
\[
 \left\{d\in\ell^1:\ \sum_n|d_n|+
                    \left|\sum_n d_n\right|\le1\right\}.
\]
Indeed $a=V^*e_1$ and $b=V^*e_2$ have Gram matrix at most $I_2$ and
$d_n=a_n\overline{b_n}$. Its top eigenvalue is at least
$(\|a\|^2+\|b\|^2)/2+|\sum_n d_n|$, which dominates the displayed
left side. Conversely take $|a_n|=|b_n|=\sqrt{|d_n|}$ and choose
their relative phases to give $d_n$. The largest Gram eigenvalue
is then exactly the displayed quantity. Sending $e_1,e_2$ to $a,b$
and their orthogonal complement to zero defines the required $V^*$.

\begin{theorem}[Joint observation on the compact unit ball]\label{r65bj:allcompact}
Let $H$ be separable and infinite dimensional. For any family
$\mathcal F\subset K(H)_{\rm sa}$, its pointwise closed joint diagonal
set consists exactly of the restrictions of the arrays
\[
 (n,A)\longmapsto\operatorname{Tr}(A\rho_n),\qquad
 \rho_n\ge0,\quad \operatorname{rank}\rho_n\le1,\quad
 \sum_n\rho_n\le I .
\]
The sum condition means that every finite partial sum is at most $I$.
An exact joint diagonal is equivalent to the existence of such a
representing family of rank-one orthogonal projections with sum $I$.

When $\mathcal F$ is the entire self-adjoint compact unit ball, the
representing operators are unique. The uniform closure consists exactly
of sequences of mutually orthogonal rank-one projections, with arbitrary
closed span. Both inclusions from the exact set to its uniform closure
and from the uniform to its pointwise closure are strict.
A coordinatewise probability mean of exact arrays belongs to the
pointwise closure exactly when the random array is almost surely constant.
\end{theorem}
\begin{proof}
In the weak-operator compact ball of contractions, $V\mapsto d_V(n,A)$
is continuous for each compact $A$. Its compact image contains every
pointwise limit. Conversely, on each finite set of observations
Theorem~\ref{r64bj:joint} approximates $d_V$ by a unitary array using
only the finitely many specified operators. No bound on the full
family is needed. The columns $v_n=Ve_n$ give $\rho_n=v_nv_n^*$
and $\sum_n\rho_n=VV^*\le I$. Conversely the partial-sum condition
makes the synthesis operator with columns $v_n$ a contraction.
Unitary columns give the assertion about exact arrays.

On the full unit ball the norm of the difference of two row
functionals is $\|\rho_n-\sigma_n\|_1$ by trace duality.
Thus uniform limits preserve trace one in every row. In the pointwise
description this makes each $\rho_n$ a rank-one projection, and
$\rho_i+\rho_j\le I$ forces orthogonality. These projections are
therefore the columns of an isometry $V$, in the sense that
$\rho_n=(Ve_n)(Ve_n)^*$. To approximate them uniformly, partition
the basis into blocks of size $m$ and let $G$ be the span, with closure,
of the normalized constant vector on each block. Define a unitary $U$
by $U=V$ on $G^\perp$ and by any unitary
$G\to VG\oplus(VH)^\perp$ on $G$. Then
$\|(U-V)e_n\|\le2\|P_Ge_n\|=2/\sqrt m$ for all $n$, so the rows
differ in trace norm by at most $4/\sqrt m$. A unilateral shift gives
an incomplete orthogonal family in the uniform closure, and $V=0$
gives a pointwise limit outside that closure.

Finally write $P_n$ for the random rank-one projection representing row
$n$. Its finite matrix compressions are measurable by polarization of
the coordinate observations, and converge in trace norm to $P_n$.
Thus $\rho_n=\mathbb E P_n$ has trace one. If the mean is in the
pointwise closure, uniqueness makes $\rho_n$ a rank-one projection and
$\mathbb E\operatorname{Tr}((I-\rho_n)P_n)=0$ forces $P_n=\rho_n$
almost surely. Intersect these full-measure events over the countable
rows. The converse follows by taking the constant exact array.
\end{proof}

For example, randomly exchange the first two basis vectors with
probability $1/2$. Each of the first two mean rows has density
$(e_1e_1^*+e_2e_2^*)/2$. This completely positive average preserves
the sum of the row densities and their individual traces. Its
rank-two rows exclude it even from the pointwise closure. Its uniform
distance to the exact set equals one: for a rank-one projection $P$,
the mean row $R$ satisfies $\|R-P\|_1\ge2(1-\operatorname{Tr}(RP))\ge1$,
and either of the original bases attains this bound. The four
real matrix-coordinate observations on $\operatorname{span}(e_1,e_2)$
already detect the obstruction, since a positive rank-one compression
has zero determinant.

The same rigidity of probability means holds whenever the real linear
span of $\mathcal F$ is norm dense in $K(H)_{\rm sa}$: equality of the
bounded row functionals on $\mathcal F$ determines them on all compacts.
This includes a norm-compact family $\{0\}\cup\{2^{-j}A_j:j\ge1\}$
built from a countable dense sequence $(A_j)$ in the self-adjoint compact
unit ball. Its two closures coincide by Theorem~\ref{r64bj:joint}.

\begin{corollary}[Observables closed under squares]\label{r65bj:jordan}
Suppose the norm-closed real linear span $\mathcal J$ of
$\mathcal F\subset K(H)_{\rm sa}$ satisfies $T^2\in\mathcal J$ for
every $T\in\mathcal J$. Its pointwise closed joint diagonal set is
convex exactly when $\dim_{\mathbb R}\mathcal J\le1$. For norm-bounded
$\mathcal F$ the same criterion decides convexity of the uniform closure.
It applies to the self-adjoint part of every compact-operator
$C^*$-subalgebra, including noncommutative algebras.
\end{corollary}
\begin{proof}
If $B,C\in\mathcal J$ are independent, the pairs
$(\langle Bx,x\rangle,\langle Cx,x\rangle)$ for unit $x$ span
$\mathbb R^2$ by polarization. A real linear combination $T$ of $B,C$
therefore has both positive and negative quadratic form values.
Consequently $T,T^2$ are independent commuting observables.
Their three-cycle midpoint from Theorem~\ref{r64bj:joint} is separated
on finitely many rows; use the same two bases for the full family.
Approximate $T,T^2$ by finite real combinations from $\mathcal F$.
Every unit-vector observation and its midpoint changes by at most
the operator norm error, so suitably close approximations still
separate the midpoint on finitely many observations from $\mathcal F$.
In dimension at most one the scalar finite-frame argument of
Theorem~\ref{r64bj:joint} gives convexity in the pointwise topology;
a bounded family in this span is norm precompact.
\end{proof}

\subsection{Convexity of the exact \(C\)-numerical range}
\label{r67:exactnumerical}
For a trace-class operator \(C\) and a bounded operator \(T\) on a
separable infinite-dimensional complex Hilbert space, put
\[
 W_C(T)=\{\operatorname{Tr}(UCU^*T):U\text{ unitary}\}.
\]
For diagonal $C$ the trace sums the observed diagonal entries with
the eigenvalues of $C$ as weights. This summed observation leads to a
different convexity question. Dirr and vom Ende asked whether the closure can be removed from their
convexity theorem under either of its two hypotheses
\cite[Theorem~3.8 and Open Problems~(b)]{r67:DvE}.
Loreaux and Patnaik proved convexity for the trace-norm closed orbit
of a self-adjoint \(C\), and for the exact orbit when \(T\) is
diagonalizable and \(C\ge0\) has zero or infinite-dimensional kernel
\cite[Theorem~4.5 and Corollary~7.3]{r67:LP}.

\begin{theorem}[The two convexity hypotheses]\label{r67:exactC}
If \(C\) is normal with collinear spectrum, \(W_C(T)\) is convex
for every bounded \(T\). If \(T\) is a complex affine function of
a compact self-adjoint operator, \(W_C(T)\) is convex for every
trace-class \(C\). For general bounded self-adjoint \(T\), convexity can
fail even for a normal trace-class \(C\) and an orthogonal projection \(T\).
\end{theorem}
\begin{proof}
The spectrum of \(C\) contains zero, so rotation reduces the first
assertion to \(C=C^*\). Fix an eigenbasis for \(C\). Unitaries acting
on its first \(n\) vectors and fixing their orthogonal complement give
a translate of a finite-dimensional \(C_n\)-numerical range, convex by
Westwick's theorem \cite{r67:Westwick}. These sets are nested and their
union is dense in \(W_C(T)\): every unitary is a strong limit of such
unitaries, by approximating any finite orthonormal frame inside large
coordinate subspaces. Finite-rank approximation of \(C\) then gives
trace-norm convergence of its conjugates. Thus the relative interior
of \(\overline{W_C(T)}\) belongs to \(W_C(T)\).

Fix a supporting line, write \(H=\operatorname{Re}(e^{-i\theta}T)\)
for its outward direction, and let \(X,Y\) be unitary conjugates of
\(C\) attaining the support value of \(\operatorname{Tr}(HX)\).
First variation under \(X\mapsto e^{itK}Xe^{-itK}\), with arbitrary
finite-rank \(K=K^*\), gives \([H,X]=0\). Interchanging two unit
vectors in eigenspaces of \(X\) of eigenvalues \(\lambda>\mu\) shows
that the \(H\)-expectation on the first vector is at least that on the
second. For \(c>0\), the finite-dimensional space
\(F=\mathbf1_{[c,\infty)}(X)\mathcal H\) therefore reduces \(H\), and
every eigenvalue of \(H|_F\) is at least \(\sup\sigma(H|_{F^\perp})\):
the eigenspaces of \(X\), including the kernel, reduce \(H\) and span
\(\mathcal H\). For \(k=\dim F>0\), the eigenvalues on \(F\) are thus
the first \(k\) upper variational eigenvalues of \(H\), counted with
multiplicity. Since \(k\) depends only on \(C\), these multisets agree
for \(X\) and \(Y\). Taking successive differences gives the same
\(H\)-eigenvalue multiset on each positive eigenspace of \(X\) and
the corresponding eigenspace of \(Y\). The same argument with
\((-H,-X)\) handles negative eigenspaces.

For each \(c>0\), match joint eigenbases on the eigenspaces
\(|\lambda|\ge c\). Within each affected eigenspace of \(H\) this
matches two finite orthonormal frames of equal size, so it extends to
a unitary \(V_c\) commuting with \(H\). Consequently
\[
 \|V_cXV_c^*-Y\|_1
 \le 2\operatorname{Tr}\!\left(|C|\mathbf1_{(0,c)}(|C|)\right)
 \longrightarrow0.
\]
The unitaries in the commutant of \(H\) are path connected, by bounded
self-adjoint logarithms. Their conjugates of \(X\) have numerical
values in a connected subset of the supporting line containing
\(\operatorname{Tr}(XT)\) and approaching \(\operatorname{Tr}(YT)\).
It therefore contains the open segment between them. Together with
the relative interior this proves convexity.

For compact self-adjoint \(T\), use its eigenbasis and vary
\(X=UTU^*\) in \(\operatorname{Tr}(CX)\). The finite-coordinate ranges
are convex by cyclicity and Westwick's theorem. Strong approximation
of a unitary gives operator-norm convergence of the conjugates of the
compact \(T\), which suffices against the trace-class \(C\).
At a supporting line put \(H=\operatorname{Re}(e^{-i\theta}C)\).
First variation against finite-rank \(K=K^*\) again gives \([H,X]=0\),
and the same finite-eigenspace matching gives
\(\|V_cXV_c^*-Y\|\le2c\). Testing against \(C\) costs at most
\(2c\|C\|_1\), so connectedness fills the face as before.
Finally \(W_C(\alpha T+\beta I)=\alpha W_C(T)+\beta\operatorname{Tr}C\).

For the counterexample take \(\mathcal H=\ell^2\oplus\ell^2\oplus
\mathbb C\), \(D=\operatorname{diag}(2^{-j})_{j\ge1}\),
\(G=D\oplus(-D)\oplus0\), and \(C=G+iE_0\), where \(E_0\) projects
onto the last summand. Let \(T\) be a projection of infinite rank and
corank. Its unitary orbit consists of all such projections \(Q\).
Then \(\operatorname{Re}\operatorname{Tr}(CQ)\le\operatorname{Tr}D=1\).
Equality requires \(Q=I\) on the first summand and \(Q=0\) on the second,
by the strictly positive diagonal entries of \(D\). The remaining
one-dimensional summand has \(Q=0\) or \(I\), and both possibilities
have infinite rank and corank. The support line meets \(W_C(T)\)
in exactly \(\{1,1+i\}\).
\end{proof}

\subsection{Projection rank and closedness of joint numerical ranges}
For a self-adjoint tuple $\mathbf A=(A_j)_j$ write
$W_p(\mathbf A)=\{(\operatorname{tr}(A_jP))_j:P=P^*=P^2,\
\operatorname{rank}P=p\}$. Chan, Li and Poon proved normalized inclusion
after convexification and downward propagation of closedness for the
convex hulls. For at most three self-adjoint entries the latter
also holds for the ranges themselves
\cite[Proposition~2.7 and Theorem~4.6]{r83:CLP}.
The first three questions of \cite[\S8]{r83:CLP} ask for normalized
inclusion of the ranges themselves, downward propagation of their
closedness, and an example of a nonclosed range with closed convex
hull at rank greater than one.

\begin{theorem}[A prescribed change of closedness with rank]\label{r83:clp}
Given integers $1\le k<L$, there is a finite tuple of compact
self-adjoint operators on a separable infinite-dimensional Hilbert
space with the following properties. The convex hulls of $W_p/p$
are the same compact set for $1\le p\le L$. The origin belongs to
$\overline{W_p}\setminus W_p$ for $1\le p\le k$, and $W_p$ equals its
convex hull for $k<p\le L$. For $p>L$ even $\operatorname{conv}W_p$
is nonclosed. Thus $W_p$ is closed precisely for $k<p\le L$, and
its convex hull is closed precisely for $p\le L$. In particular
$W_{k+1}/(k+1)\not\subseteq W_k/k$.
\end{theorem}
\begin{proof}
Set $r=L$, $d=kr+1$ and $F=\mathbb C^d\otimes\mathbb C^r$.
On $F\oplus\ell^2$ take the operators $(B\otimes I_r)\oplus0$, with
$B$ running over a Hilbert--Schmidt orthonormal real basis of the
traceless Hermitian $d\times d$ matrices, together with $0\oplus T$.
Here $T$ is positive, injective and compact, has norm one, and its
one-eigenspace has dimension $r$. For a rank-$p$ projection $P$,
write $\rho$ for the partial trace over $\mathbb C^r$ of its compression
to $F$, and $s=\operatorname{tr}\rho$. The observations in these coordinates are
$X=\rho-sI_d/d$ and $t=\operatorname{tr}((0\oplus T)P)$.
Since $0\le t\le p-s$, they belong to the compact convex set
\[
 D_p=\{(X,t): X=X^*,\ \operatorname{tr}X=0,\ 0\le t\le p,\
                  X+(p-t)I_d/d\ge0\}=pD_1.
\]

For $p\le r$, the normalized partial traces of rank-$p$ projections
on $F$ are exactly the density matrices of rank at most $pr$.
The upper bound follows by expanding each of $p$ vectors in the
second tensor factor. For the converse, let
$D=\sum_{a=1}^h\lambda_a h_a h_a^*$ be a density matrix with $h\le pr$.
There exist $pr$ pairwise Hilbert--Schmidt orthogonal isometries
$\mathbb C^p\to\mathbb C^r$. On the standard basis $e_j$, $0\le j<p$,
they are $e_j\mapsto e^{2\pi i b j/p} e_{(j+c)\bmod r}$ for
$0\le b<p$, $0\le c<r$. Choose $h$ of them, $V_1,\ldots,V_h$.
Then the $p$ vectors $u_j=\sum_a\sqrt{\lambda_a}\,h_a\otimes V_a e_j$
are orthonormal, and summing their vector projections gives partial
trace $pD$ because $\operatorname{tr}(V_a^*V_b)=p\delta_{ab}$.

For $0\le t<p$, every point of $D_p$ has the form
$((p-t)(D-I_d/d),t)$ for a density matrix $D$.
When $\operatorname{rank}D\le pr$, choose the preceding $u_j$ and
orthonormal vectors $v_j$ from the one-eigenspace of $T$. The
orthonormal vectors $\sqrt{1-t/p}\,u_j\oplus\sqrt{t/p}\,v_j$
realize this point. At $t=p$ use $0\oplus v_j$.
All pure density matrices are thus attained at each fixed $t$.
Their convex hull contains every density matrix, proving
$\operatorname{conv}W_p=D_p$ for $p\le L$. When $k<p\le L$, we have
$d\le pr$, so every point is attained.

When $p\le k$, the observations $(0,0)$ would give
$\operatorname{ran}P\subset F$ by positivity and injectivity of $T$,
then $\rho=pI_d/d$ by traceless observation. Its rank is $d>pr$.
On the other hand, $p$ orthonormal tail eigenvectors with eigenvalues
tending to zero give observations tending to $(0,0)$.

Finally suppose $p>r$. For a unit $h\in\mathbb C^d$ put
$X_0=r(hh^*-I_d/d)$. The projection onto $h\otimes\mathbb C^r$,
padded with $p-r$ tail eigenvectors whose eigenvalues tend to zero,
gives $(X_0,0)\in\overline{W_p}$. At any point of $W_p$ with $t=0$
the projection is supported on $F$, so $\rho\le rI_d$ and
$h^*Xh\le r-p/d$. Because $t$ is nonnegative, the same inequality
holds on the $t=0$ face of $\operatorname{conv}W_p$. It excludes
$(X_0,0)$ since $h^*X_0h=r-r/d>r-p/d$.
\end{proof}

\subsection{The essential numerical range of a tridiagonal Toeplitz operator}
M\"uller and Tomilov obtained the essential numerical range for equal
off-diagonal moduli by alternating two coordinate sizes. They left the
exact formula for unequal moduli open in \cite[Remark~4.8]{r86:MTToeplitz}.
Two geometric tails, with independent ratios on the two sides of their
common maximum, determine the support function.

\begin{theorem}[Two geometric tails for the essential numerical range]
\label{r86:toeplitz}
On the complex space $\ell_p(\mathbb N_0)$, $1<p<\infty$, let
$T=a_0 I+a_+S_R+a_-S_L$, where $(S_Rx)_j=x_{j-1}$ with $x_{-1}=0$
and $(S_Lx)_j=x_{j+1}$. For $\theta\in\mathbb R$ put
\[
 g_\theta(r)=
 \big|e^{-i\theta}a_+r^{1-1/p}
       + e^{i\theta}\overline{a_-}r^{1/p}\big|,\qquad r\ge0,
 \qquad
 H_\theta=\sup_{0<u<1<v}
 \frac{(v-1)g_\theta(u)+(1-u)g_\theta(v)}{v-u}.
\]
Then the essential numerical range and the closure of the spatial
numerical range both equal
\[
 \bigcap_{0\le\theta<2\pi}
 \{z\in\mathbb C:\ \operatorname{Re}(e^{-i\theta}(z-a_0))\le H_\theta\}.
\]
Equivalently, $H_\theta$ is the value at $1$ of the least concave
majorant of $g_\theta$ on $[0,\infty)$. The spatial numerical range
itself is the relative interior of this closed convex set.
\end{theorem}

\begin{proof}
For a finitely supported unit vector $x$ extend its coordinates by zero
to $\mathbb Z$ and put $w_j=|x_j|^p$. At each $w_j>0$ write
$r_j=w_{j+1}/w_j$. In the spatial numerical value
$\sum_j (Tx)_j |x_j|^{p-2}\overline{x_j}$, with the duality map taken
as zero at a zero coordinate, the edge from $j$ to $j+1$ contributes
\[
 w_j\big(a_+r_j^{1-1/p}e^{-i\delta_j}
          +a_-r_j^{1/p}e^{i\delta_j}\big),
 \qquad \delta_j=\arg x_{j+1}-\arg x_j.
\]
An edge incident with a zero coordinate contributes zero. The largest
rotated real part of an edge, over its phase increment, is
$w_jg_\theta(r_j)$.

Write $G$ for the least concave majorant of $g_\theta$. It is finite
since $g_\theta(r)\le (|a_+|+|a_-|)(1+r)$.
It is nondecreasing because a negative secant slope would force the
nonnegative concave function $G$ to become negative further to the right.
The value $G(1)$ is the supremum of the averages of $g_\theta$ under
finitely supported probability laws of mean $1$: this construction, at
every mean, itself gives the least concave majorant. For any fixed
finite support the average is maximized at a law on at most two points,
since the total mass and the mean impose two linear equations.
Boundary choices $u=0$ or a point mass at $1$ follow by continuity;
thus $G(1)=H_\theta$. Since $\sum_jw_j=1$ and
$\sum_{w_j>0} w_jr_j\le 1$, Jensen's inequality bounds the rotated
numerical value of $T-a_0 I$ by $H_\theta$. Density and continuity of
the duality map extend the bound to all unit vectors.

Conversely fix $0<u<1<v$ and on $\mathbb Z$ take masses proportional to
$v^j$ for $j<0$ and to $u^j$ for $j\ge0$. Normalize their sum to one.
The edges with ratio $v$ carry total initial mass $(1-u)/(v-u)$, and
the edges with ratio $u$ carry $(v-1)/(v-u)$.
Choose the phase increments consecutively so that every edge attains
its largest rotated real part. The resulting bilateral unit vector
attains the displayed chord value. Truncate both tails, normalize,
and translate beyond any prescribed initial segment of $\mathbb N_0$.
The numerical values converge to the bilateral value by continuity.
Taking the supremum recovers $H_\theta$ in every direction.
For these Toeplitz operators the essential numerical range is convex
and equals the closure of the spatial numerical range
\cite[Proposition~4.1]{r86:MTToeplitz}; its support function therefore
determines the asserted set.

For any unilateral unit vector the first nonzero mass $w_k$ gives
$\sum_{w_j>0}w_jr_j\le1-w_k<1$. In each direction where $g_\theta$
is not identically zero it is unbounded at infinity, so its concave
majorant is strictly increasing. Jensen's inequality for the countable
mass distribution now makes the support inequality strict. In the
remaining directions the rotated real part vanishes on the whole
centered range, so these lines define no proper relative face. Thus
the spatial range lies in the relative interior. Conversely,
finitely supported numerical values form a dense convex subset of
the closed range: two finite vectors can be translated to separated
blocks and multiplied by $t^{1/p}$ and $(1-t)^{1/p}$. This subset
contains the relative interior.
\end{proof}

For $p=2$, $g_\theta$ is a constant multiple of $\sqrt r$, so the
formula gives the classical filled ellipse. For $a_-=0$ it gives the
disk of radius $|a_+|$, since $r^{1-1/p}$ is concave.

\subsection{Joint constant diagonals of bounded operators}
For a finite tuple \(\mathbf T\), a joint constant diagonal is a vector
\(w\) such that some orthonormal basis \((u_j)\) satisfies
\(\langle T_k u_j,u_j\rangle=w_k\) for every \(j,k\).
M\"uller and Tomilov proved convexity for a single bounded operator
\cite[Theorem~3.9]{r69:MT}. At the beginning of Section~4 they discuss
the question for $n\ge2$, including commuting tuples. The following
noncommuting triple answers the general convexity question.

\begin{theorem}[Nonconvex joint constant diagonals]\label{r69:constant}
There are six bounded self-adjoint operators on a separable
infinite-dimensional Hilbert space with two joint constant diagonals
whose entire open segment contains no joint constant diagonal.
Equivalently, three bounded complex operators suffice.
\end{theorem}
\begin{proof}
On \(K=\mathbb C e_0\oplus\mathbb C f_0\oplus
\bigoplus_{n\ge1}\mathbb C^4\) choose two diagonal operators. Their
entries on the first two coordinates and on block \(n\) are
\[
 A:\quad 1,\ -\tfrac12,\ -2^{-n-3} I_4,\qquad
 B:\quad 0,\ 0,\ \operatorname{diag}(1,i,-1,-i).
\]
In particular \(A\) is trace class with trace zero and only one
positive eigenvalue. Write \(R=\operatorname{Re}B\),
\(J=\operatorname{Im}B\). There are four bases of \(K\) on which
the joint diagonal of \((A,B)\) is constant \((0,b)\), for
\(b=c,-c,ic,-ic\), where \(c=1/4\).

To obtain \(b=c\), split \(K=M_+\oplus M_-\). Put
\((e_0\pm f_0)/\sqrt2\) in \(M_\pm\), and put
\((e_{n,i}\pm e_{n,-i})/\sqrt2\) in \(M_\pm\), where the second
subscript labels the eigenvalue of \(B\).
Give each space one of \(e_{n,1},e_{n,-1}\) per block, interchanging
their assignment at each successive block. The compression of \(J\)
to either space is zero. The compression of \(A\) has trace
\(1/4-2\sum_{n\ge1}2^{-n-3}=0\). On the displayed basis of each
space the compression of \(R\) has infinitely many entries of each
of \(1,-1,0\). Reorder them in groups with two \(1\)'s, one \(-1\)
and one zero. At the end of each group the partial diagonal sum
of \(R-cI\) is zero, while the partial sums of \(A\) tend to zero by
trace invariance. Fan's single-operator criterion
\cite[Theorem~3.4]{r69:MT} gives a zero-diagonal basis for the
compression of \(A+i(R-cI)\). Joining the two bases gives \((0,c)\).
Cyclic permutation in every four-dimensional block commutes with
\(A\) and rotates \(B\) by a quarter turn, giving the other three bases.

Let \(X,Y,Z\) be the Pauli matrices on \(\mathbb C^2\), and on
\(\mathbb C^2\otimes K\) take the tuple
\[
 X\otimes A,\quad Y\otimes A,\quad Z\otimes A,\qquad
 X\otimes R,\quad X\otimes J,\quad Z\otimes R .
\]
On any joint constant-diagonal basis the first three entries vanish
by compactness. For a basis vector \(u\), write \(x_j\in\mathbb C^2\)
for its components along the displayed coordinate basis of \(K\).
If \(a_j\) are the diagonal entries of \(A\), vanishing of all three
Pauli expectations gives
\(\sum_j a_j x_jx_j^*=t_u I_2\), since \(X,Y,Z\) span the traceless
Hermitian matrices. Testing in a direction perpendicular to \(x_{e_0}\)
gives \(t_u\le0\). Over the complete basis,
\(\sum_u 2t_u=\operatorname{Tr}(I_2\otimes A)=0\); hence \(t_u=0\)
for every \(u\). All \(a_j\) except \(a_{e_0}=1\) are strictly
negative. The resulting positive-matrix identity forces every \(x_j\)
to lie on the line of \(x_{e_0}\), so \(u\) is a product vector.

For a unit product vector the last three observations are
\((xr,xs,zr)\), where \(x,z\) are the \(X,Z\) expectations of its
first factor and \(r,s\) are the \(R,J\) expectations of the second.
The constant triple \((0,c,0)\) is realized by pairing the \(+1,-1\)
eigenvectors of \(X\) with the two bases \(b=ic,-ic\), respectively.
Pairing the eigenvectors of \(Z\) with the bases \(b=c,-c\) gives
\((0,0,c)\). The compact entries in both bases are zero.
An interior point on their segment would require \(xr=0\) and both
\(xs,zr\ne0\), which is impossible. Pairing the six operators as
real and imaginary parts gives three complex operators with the
same obstruction.
\end{proof}

\section*{Use of generative AI}
The author completed this paper with the assistance of OpenAI’s GPT models. The author
takes full responsibility for the correctness of the paper.

\begingroup
\renewcommand{\bibliofont}{\small}

\endgroup

\end{document}